\documentclass[12pt]{amsart}
\usepackage{amscd,amssymb,amsthm,amsmath,enumerate,textcomp}
\usepackage[matrix,arrow]{xy}
\usepackage{url}

\theoremstyle{definition}
\newtheorem{example}[equation]{Example}
\newtheorem*{example*}{Example}

\newtheorem{lemma}[equation]{Lemma}
\newtheorem{corollary}[equation]{Corollary}
\newtheorem{proposition}[equation]{Proposition}

\newtheorem*{conjecture*}{Conjecture}
\newtheorem*{maintheorem*}{Main Theorem}
\newtheorem*{corollary*}{Corollary}

\newtheorem*{question*}{Question}

\newtheorem*{problem*}{Problem}
\newtheorem*{theorem*}{Theorem}

\theoremstyle{remark}
\newtheorem{remark}[equation]{Remark}
\newtheorem*{remark*}{Remark}

\makeatletter\@addtoreset{equation}{subsection} \makeatother

\author{Ivan Cheltsov and Victor Przyjalkowski}

\title[Katzarkov--Kontsevich--Pantev Conjecture for Fano threefolds]{Katzarkov--Kontsevich--Pantev Conjecture\\ for Fano threefolds}

\address{\emph{Ivan Cheltsov}
\newline
\textnormal{School of Mathematics, The University of Edinburgh,  Edinburgh, UK.}
\newline
\textnormal{Laboratory of Algebraic Geometry, NRU HSE, 6 Usacheva street, Moscow, Russia, 119048.}
\newline
\textnormal{\texttt{I.Cheltsov@ed.ac.uk}}}

\address{\emph{Victor Przyjalkowski}
\newline
\textnormal{Steklov Mathematical Institute of Russian Academy of Sciences, 8 Gubkina street, Moscow, Russia.}
\newline
\textnormal{Laboratory of Mirror Symmetry, NRU HSE, 6 Usacheva street, Moscow, Russia, 119048.}
\newline
\textnormal{\texttt{victorprz@mi.ras.ru}}}

\begin{document}

\begin{abstract}
We verify Katzarkov--Kontsevich--Pantev conjecture for Landau--Ginzburg models of smooth Fano threefolds.
\end{abstract}

\maketitle

\section*{Introduction}
\label{section:intro}

For a smooth Fano variety $X$, its Landau--Ginzburg model is a smooth quasiprojective variety $Y$
equipped with a regular function $\mathsf{w}\colon Y\to\mathbb{C}$.
Homological Mirror Symmetry Conjecture predicts the equivalences between the derived category of coherent sheaves on~$X$
(the derived category of singularities of $(Y,\mathsf{w})$, respectively)
and the Fukaya--Seidel category of the pair $(Y,\mathsf{w})$ (the Fukaya category of $X$, respectively).

In~\cite{KKP}, Katzarkov, Kontsevich, and Pantev considered \emph{tame compactification} of the Landau--Ginzburg model (see~\cite[Definition 2.4]{KKP}), that is
a commutative diagram
$$
\xymatrix{
Y\ar@{^{(}->}[rr]\ar@{->}[d]_{\mathsf{w}}&&Z\ar@{->}[d]^{\mathsf{f}}\\
\mathbb{C}\ar@{^{(}->}[rr]&&\mathbb{P}^1}
$$
such that $Z$ is a smooth compact variety that satisfies certain natural geometric conditions,
and $\mathsf{f}$ is a morphism such that $\mathsf{f}^{-1}(\infty)=-K_Z$.
If exists, the compactification $Z$ is unique up to flops in the fibers of the morphism $\mathsf{f}$.
The pair $(Z,\mathsf{f})$ is usually called the compactified Landau--Ginzburg model of the Fano variety $X$.

Katzarkov, Kontsevich and Pantev also defined the Hodge-type numbers $f^{p,q}(Y,\mathsf{w})$
of the Landau--Ginzburg model $(Y,\mathsf{w})$ that come from the sheaf cohomology of certain logarithmic forms.
They posed the following conjecture.

\begin{conjecture*}[Katzarkov--Kontsevich--Pantev]
Let $(Y,\mathsf{w})$ be a Landau--Ginzburg model of the smooth Fano variety $X$ with $\dim (X)=\dim (Y)=d$.
Suppose that it admits a tame compactification.
Then
$$
h^{p,q}\big(X\big)=f^{q,d-p}\big(Y,\mathsf{w}\big).
$$
\end{conjecture*}

In~\cite{LP16}, this conjecture was proved  for del Pezzo surfaces and their Landau--Ginzburg models
constructed by Auroux, Katzarkov, and Orlov in \cite{AKO06}.
In this paper, we verify Katzarkov--Kontsevich--Pantev Conjecture for smooth Fano threefolds
and their \emph{toric Landau--Ginzburg models} constructed in~\cite{P07,P13,ACGK,CCGK},
which satisfy all hypotheses of Katzarkov--Kontsevich--Pantev Conjecture by \cite[Theorem~1]{P16}.

From now on and until the end of this paper, we assume that $X$ is a smooth Fano threefold.
Its compactified Landau--Ginzburg model is given by the following commutative diagram
\begin{equation}
\label{equation:CCGK-compactification}\tag{$\maltese$}
\xymatrix{
(\mathbb C^*)^3\ar@{^{(}->}[rr]\ar@{->}[d]_{\mathsf{p}}&&Y\ar@{->}[d]^{\mathsf{w}}\ar@{^{(}->}[rr]&&Z\ar@{->}[d]^{\mathsf{f}}\\
\mathbb{C}\ar@{=}[rr]&&\mathbb{C}\ar@{^{(}->}[rr]&&\mathbb{P}^1}
\end{equation}
where $\mathsf{p}$ is a surjective morphism that is given by one of the Laurent polynomials explicitly described in \cite{ACGK,P16,CCGK},
the variety $Y$ is a smooth threefold with $K_Y\sim 0$,
and $Z$ is a smooth compact threefold such that
$$
-K_Z\sim \mathsf{f}^{-1}\big(\infty\big).
$$
Moreover, in every case, one also has $h^{1,2}(Z)=0$.

In \cite{Harder}, Harder showed how to compute the numbers $f^{p,q}(Y,\mathsf{w})$ using the global geometry of the compactification $Z$.
He showed that under some natural conditions one has $f^{3,0}(Y,\mathsf{w})=f^{0,3}(Y,\mathsf{w})=1$ and
\begin{equation}
\label{equation:Harder-1}\tag{$\clubsuit$}
f^{1,1}\big(Y,\mathsf{w}\big)=f^{2,2}\big(Y,\mathsf{w}\big)=\sum_{P\in\mathbb{C}^1}\big(\rho_P-1\big),
\end{equation}
where $\rho_P$ is the number of irreducible components of the fiber $\mathsf{w}^{-1}(P)$.
Moreover, he proved that
\begin{equation}
\label{equation:Harder-2}\tag{$\spadesuit$}
f^{1,2}\big(Y,\mathsf{w}\big)=f^{2,1}\big(Y,\mathsf{w}\big)=\dim\Bigg(\mathrm{coker}\Big(H^2\big(Z,\mathbb{R}\big)\to H^2\big(F,\mathbb{R}\big)\Big)\Bigg)-2+h^{1,2}\big(Z\big),
\end{equation}
where $F$ is a general fiber of the morphism $\mathsf{w}$.
Finally, he proved that the remaining $f^{p,q}$ numbers of the Landau--Ginzburg model $(Y,\mathsf{w})$ vanish.

Thus, to prove the Katzarkov--Kontsevich--Pantev Conjecture for smooth Fano threefolds,
one needs to compute the right-hand sides in~\eqref{equation:Harder-1} and~\eqref{equation:Harder-2}
and compare them with the well-known Hodge numbers of smooth Fano threefolds.
For smooth Fano threefolds of Picard rank one, this has been done in \cite{P13,ILP}.
The goal of this paper is to do the same for smooth Fano threefolds whose Picard rank is larger than one.

To be precise, we prove the following result.

\begin{maintheorem*}
\label{theorem:main}
Let $X$ be a smooth Fano threefold,
and let $\mathsf{f}\colon Z\to\mathbb{P}^1$ be its compactified Landau--Ginzburg model given by \eqref{equation:CCGK-compactification},
where $\mathsf{p}$ is a surjective morphism that is given by one of the Laurent polynomials described in \cite{ACGK,CCGK}.
Then
\begin{equation}
\label{equation:main-1}\tag{$\heartsuit$}
h^{1,2}\big(X\big)=\sum_{P\in\mathbb{C}^1}\big(\rho_P-1\big),
\end{equation}
where $\rho_P$ is the number of irreducible components of the fiber $\mathsf{w}^{-1}(P)$.
Moreover, one has
\begin{equation}
\label{equation:main-2}\tag{$\diamondsuit$}
\mathrm{rk}\,\mathrm{Pic}(X)=\dim\Bigg(\mathrm{coker}\Big(H^2\big(Z,\mathbb{R}\big)\to H^2\big(F,\mathbb{R}\big)\Big)\Bigg)-2,
\end{equation}
where $F$ is a general fiber of the morphism $\mathsf{f}$.
\end{maintheorem*}

Using \eqref{equation:Harder-1} and~\eqref{equation:Harder-2}, we obtain the following corollary.

\begin{corollary*}
\label{corollary:main}
Let $X$ be a smooth Fano threefold.
Then Katzarkov--Kontsevich--Pantev Conjecture holds for its compactified Landau--Ginzburg model \eqref{equation:CCGK-compactification},
where $\mathsf{p}$ a morphism that is given by one of the Laurent polynomials described in \cite{ACGK,P16,CCGK}.
\end{corollary*}

The proof of Main Theorem gives an explicit description of the fiber $\mathsf{f}^{-1}(\infty)$ in \eqref{equation:CCGK-compactification},
which show that the conditions of Harder's result are satisfied.
This has been already verified in~\cite[Corollary 35]{P16} for smooth Fano threefolds with very ample anticanonical divisor.
The~proof of Main Theorem also gives an explicit description of (isolated and non-isolated) singularities
of the fibers of the morphism $\mathsf{w}$ in \eqref{equation:CCGK-compactification}
in the case when $\mathsf{p}$ is given by one of the Laurent polynomials from \cite{ACGK,P16,CCGK}.
It~seems possible to use this description to check that the Jacobian rings of these Landau--Ginzburg models
are isomorphic to the quantum cohomology rings of the corresponding smooth Fano threefolds,
which reflects Homological Mirror Symmetry on the Hochschild cohomology level.
Perhaps, one can also use the proof of Main Theorem to compute the derived categories of singularities of our compactified Landau--Ginzburg models.

This paper is organized as follows. In Section~\ref{section:scheme} we give a detailed scheme of the proof of our Main Theorem.
We illustrate each step of the scheme by an appropriate example,
see Examples~\ref{example:r-3-n-27}, \ref{example:r-3-n-2}, \ref{example:r-8-n1},  \ref{example:r-3-n-11}, and \ref{example:r-2-n-34}.
In~Sections~\ref{section:rank-2}, \ref{section:rank-3},
\ref{section:rank-4}, \ref{section:rank-5}, \ref{section:rank-6}, \ref{section:rank-7}, \ref{section:rank-8},
\ref{section:rank-9}, \ref{section:rank-10}, we prove Main Theorem for smooth Fano threefolds of Picard rank
$2$, $3$, $4$, $5$, $6$, $7$, $8$, $9$, $10$, respectively.
These sections are split by subsections whose numbers matche the
numbers of families of smooth Fano threefolds given in \cite{IP}.
For instance, in Subsection~\ref{section:r-3-n-20}, we prove Main Theorem
in the case when $X$ is a blow up of a smooth quadric threefold in a disjoint union of two lines.
This is family \textnumero $3.20$.
Likewise, in Subsection~\ref{section:r-2-n-24}, we prove Main Theorem for family \textnumero $2.24$,
which consists of divisors of bidegree $(1,2)$ in $\mathbb{P}^2\times\mathbb{P}^2$.
Finally, in Appendix~\ref{section:intersection}, we review basic intersection theory for smooth curves on surfaces with du Val singularities,
which is probably well known to experts.

\medskip

\textbf{Notation and conventions.}
We assume that all varieties are defined over the field of complex numbers $\mathbb{C}$ unless it is specially mentioned.
For a (non necessary reduced) variety $V$, we denote the number of its irreducible components by $[V]$.
To denote Laurent polynomials from the database~\cite{fanosearch}, we use the notation
$P.N$, where $P$ is the number of the Newton polytope of the polynomial, and $N$ is the number
of polynomial for the polytope. If the polynomial for given polytope is unique, we
just say that it is the polynomial number $P$.

\medskip

\textbf{Acknowledgements.}
We would like to thank Andrew Harder for useful comments.
Ivan Cheltsov was supported by Royal Society grant No. 	IES\textbackslash R1\textbackslash 180205,
and by Russian Academic Excellence Project~5-100.
Victor Przyjalkowski was partially supported by Laboratory of Mirror Symmetry NRU
HSE, RF government grant, ag. No. 14.641.31.0001.
He is a Young Russian Mathematics award winner and would like to thank its sponsors and jury.

\section{The proof}
\label{section:scheme}

To prove Main Theorem, we fix a smooth Fano threefold $X$.
Then $X$ is contained in one of $105$ deformation families described in~Iskovskikh and Prokhorov's book~\cite{IP}.
We~add the variety found in~\cite{MM04} to the end of the list of Picard rank 4 threefolds.
We always assume that $X$ is a general threefold in its deformation family.

For each family, we have the commutative diagram \eqref{equation:CCGK-compactification},
where $\mathsf{p}$ is given by a Laurent polynomial, which we identify with $\mathsf{p}$.
Then we proceed as follows.

\subsection{Mirror partners.}
\label{subsection:scheme-step-1}
The polynomial $\mathsf{p}$ is not uniquely determined by $X$.
However, Akhtar, Coates, Galkin, and Kasprzyk proved in \cite{ACGK} that
all of them are related by birational transformations, called mutations.
Mutations preserve the right hand sides of \eqref{equation:main-1} and \eqref{equation:main-2} in Main Theorem.
Thus, to prove Main Theorem, we may choose any Laurent polynomial $\mathsf{p}$ from~\cite{fanosearch} among mirror partners for $X$.

\subsection{Rank of Picard group.}
\label{subsection:scheme-step-2}
If $X$ is a smooth Fano threefold such that $\mathrm{rk}\,\mathrm{Pic}(X)=1$,
then~\eqref{equation:main-1} in Main Theorem is already established in~\cite{P13,P18},
and \eqref{equation:main-2} in Main Theorem follows from the proof of \cite[Theorem~4.1]{ILP}.
Thus, we will always assume that $\mathrm{rk}\,\mathrm{Pic}(X)\geqslant 2$.
This leaves us with $88$ deformation families described in~\cite{IP}.

\subsection{Minkowski polynomials.}
\label{subsection:scheme-step-3}

If $-K_{X}$ is very ample, then $X$ admits a Gorenstein toric degeneration.
In this case, the Newton polytope of the Laurent polynomial $\mathsf{p}$ is a reflexive lattice polytope
which is a \emph{fan polytope} of the toric degeneration,
and the coefficients of  $\mathsf{p}$ correspond to expansions of its facets to Minkowski sums of elementary polygons.
Because of this, the Laurent polynomial $\mathsf{p}$ is usually called \emph{Minkowski polynomial} (see \cite{ACGK}).

The divisor $-K_{X}$ is very ample except for  $5$ special families.
These are the deformation families \textnumero $2.1$, $2.2$, $2.3$, $9.1$, $10.1$ in~\cite{IP}.
To prove Main Theorem, we deal with these cases separately.
Thus, in the remaining part of this section, we assume that $-K_X$ is very ample,
and $\mathsf{p}$ is one of the corresponding Minkowski polynomials.

\subsection{Pencil of quartic surfaces.}
\label{subsection:scheme-step-4}

For every smooth Fano threefold $X$ such that its anticanonical $-K_X$ is very ample,
we can always choose the corresponding Minkowski polynomial $\mathsf{p}$ in \cite{ACGK} such that
there is a pencil $\mathcal{S}$ of quartic surfaces on $\mathbb{P}^3$ given~by
\begin{equation}
\label{equation:quartic}
f_4(x,y,z,t)=\lambda xyzt
\end{equation}
that expands \eqref{equation:CCGK-compactification} to the following commutative diagram:
\begin{equation}
\label{equation:diagram}
\xymatrix{
(\mathbb{C}^*)^3\ar@{^{(}->}[rr]\ar@{->}[d]_{\mathsf{p}}&&Y\ar@{->}[d]^{\mathsf{w}}\ar@{^{(}->}[rr]&&Z\ar@{->}[d]^{\mathsf{f}}&& V\ar@{-->}[ll]_{\chi}\ar@{->}[d]^{\mathsf{g}}\ar@{->}[rr]^{\pi}&&\mathbb{P}^3\ar@{-->}[lld]^{\phi}\\
\mathbb{C}\ar@{=}[rr]&&\mathbb{C}\ar@{^{(}->}[rr]&&\mathbb{P}^1\ar@{=}[rr]&&\mathbb{P}^1&&}
\end{equation}
where $\phi$ is a rational map given by the pencil $\mathcal{S}$,
the map $\pi$ is a birational morphism to be explicitly constructed later in this section,
the threefold $V$ is smooth, and $\chi$ is a composition of flops.
Here $f_4(x,y,z,t)$ is a quartic homogeneous polynomial and \mbox{$\lambda\in\mathbb{C}\cup\{\infty\}$},
where $\lambda=\infty$ corresponds to the fiber $\mathsf{f}^{-1}(\infty)$.

\subsection{Fibers of the Landau--Ginzburg model.}
\label{subsection:scheme-step-5}

By \cite[Corollary~35]{P16}, we have
$$
\big[\mathsf{f}^{-1}(\infty)\big]=\frac{4-K_X^3}{2}.
$$
To verify \eqref{equation:main-1} in Main Theorem, we must find $[\mathsf{f}^{-1}(\lambda)]$ for every $\lambda\ne\infty$.
This can be done by checking basic properties of the pencil $\mathcal{S}$. Let us show how to do this in easy cases.

Let $S_\lambda$ be the quartic surface given by \eqref{equation:quartic},
let $\widetilde{S}_\lambda$ be its proper transform on the threefold $V$,
and let $E_1,\ldots,E_n$ be the $\pi$-exceptional divisors.
Then
$$
K_{V}+\widetilde{S}_\lambda+\sum_{i=1}^{n}\mathbf{a}_i^\lambda E_i\sim\pi^*\big(K_{\mathbb{P}^3}+S_\lambda\big)\sim 0
$$
for some non-negative integers $\mathbf{a}_1^\lambda,\ldots,\mathbf{a}_n^\lambda$.
Hence, since $-K_V\sim\mathsf{g}^{-1}(\infty)$, we conclude that
\begin{equation}
\label{equation:fiber}
\mathsf{g}^{-1}(\lambda)=\widetilde{S}_\lambda+\sum_{i=1}^{n}\mathbf{a}_i^\lambda E_i.
\end{equation}

Since $\chi$ in \eqref{equation:diagram} is a composition of  flops,
it follows from \eqref{equation:fiber} that
\begin{equation}
\label{equation:number-of-irredubicle-components-simple}
\big[\mathsf{f}^{-1}(\lambda)\big]=\big[S_\lambda\big]+\boxed{\text{the number of indices $i\in\{1,\ldots,n\}$ such that $\mathbf{a}_i^\lambda>0$}}\,.
\end{equation}

The number $[S_{\lambda}]$ is easy to compute.
How to determine the {correction} term in \eqref{equation:number-of-irredubicle-components-simple}?
One way to do this is to explicitly describe the birational morphism $\pi$ in \eqref{equation:diagram} and then compute the numbers $\mathbf{a}_1^\lambda,\ldots,\mathbf{a}_n^\lambda$.
However, this method is usually very time consuming.
Our main goal is to show how to do the same with less efforts.
We start with the following.

\begin{lemma}
\label{lemma:irreducible-fibers}
Let $P$ be a point in the base locus of the pencil $\mathcal{S}$.
Suppose that the quartic surface $S_\lambda$ has at most du Val singularity at~$P$.
If $P\in\pi(E_i)$, then $\mathbf{a}_i^\lambda=0$.
\end{lemma}

\begin{proof}
By \cite[Theorem~7.9]{Kollar}, the log pair $(\mathbb{P}^3,S_\lambda)$ has canonical singularities at $P$,
so that $\mathbf{a}_i^\lambda=0$ for every $E_i$ such that $P\in\pi(E_i)$.
\end{proof}

\begin{corollary}
\label{corollary:irreducible-fibers}
Suppose that $S_\lambda$ has du Val singularities in every point of the base locus of the pencil $\mathcal{S}$.
Then $\mathsf{f}^{-1}(\lambda)$ is irreducible.
\end{corollary}

\begin{proof}
The surface $S_\lambda$ is irreducible,
because $S_\lambda$ has du Val singularities in every point of the base locus of the pencil $\mathcal{S}$.
This follows from the fact that irreducible components of the surface $S_\lambda$ are hypersurfaces in $\mathbb P^3$.
By Lemma~\ref{lemma:irreducible-fibers}, we have
$$
\mathbf{a}_1^\lambda=\mathbf{a}_2^\lambda=\cdots=\mathbf{a}_n^\lambda=0,
$$
so that the fiber $\mathsf{f}^{-1}(\lambda)$ is irreducible by \eqref{equation:number-of-irredubicle-components-simple}.
\end{proof}

Let us show how to apply this result to prove \eqref{equation:main-1} in Main Theorem in one simple case.
Before doing this, let us fix handy notation that will be used throughout the whole paper.

\subsection{Handy notation}
\label{subsection:notations}

We will use $[x:y:z:t]$ as homogeneous coordinates on $\mathbb{P}^3$.
For distinct non-empty subsets $I$, $J$, and $K$ in $\{x,y,z,t\}$,
we will write $H_I$ for the plane in~$\mathbb{P}^3$ that is defined by setting the sum of coordinates in $I$ equal to zero.
For instance, we denote by $H_{\{ x\}} $ the plane in $\mathbb{P}^3$  that is given by  $x=0$.
Similarly, we denote by $H_{\{y,t\}}$ the plane in $\mathbb{P}^3$ that is given by
$$
y+t=0.
$$

We also write $L_{I,J}=H_I\cap H_J$. Likewise, we write $P_{I,J,K}=H_I \cap H_J \cap H_K $.
For instance, the symbol $L_{\{x\},\{y,z,t\}}$ denotes the line in $\mathbb{P}^3$ that is given by
$$
\left\{\aligned
&x=0,\\
&y+z+t=0.\\
\endaligned
\right.
$$
Similarly, we have $P_{\{x\},\{y\},\{z\}}=[0:0:0:1]$ and $P_{\{x\},\{y\},\{z,t\}}=[0:0:1:-1]$.

If the quartic surface $S_\lambda$ has du Val singularities, we always denote by $H_\lambda$ its general hyperplane section
or its class in $\mathrm{Pic}(S_\lambda)$.
We will use this often to compute the intersection form of some curves on $S_\lambda$ in the proof of \eqref{equation:main-2} in Main Theorem.

\subsection{Ap\'ery--Fermi pencil}
\label{subsection:two-examples}

Let us use Corollary~\ref{corollary:irreducible-fibers} in the case when $X=\mathbb{P}^1\times\mathbb{P}^1\times\mathbb{P}^1$.
In this case, the pencil $\mathcal{S}$ has been studied by Peters and Stienstra in \cite{PS89}.

\begin{example}
\label{example:r-3-n-27}
Suppose that $X=\mathbb{P}^1\times\mathbb{P}^1\times\mathbb{P}^1$.
This is family \textnumero $3.27$.
One its mirror partner is given by the Laurent polynomial
$$
x+y+z+\frac{1}{x}+\frac{1}{y}+\frac{1}{z}.
$$
This is the Minkowski polynomial \textnumero $30$.
The corresponding pencil $\mathcal{S}$ is given by
$$
x^2yz+y^2xz+z^2xy+t^2xy+t^2xz+t^2yz=\lambda xyzt,
$$
Its base locus consists of the lines $L_{\{x\},\{y\}}$, $L_{\{x\},\{z\}}$, $L_{\{x\},\{t\}}$,
$L_{\{y\},\{z\}}$, $L_{\{y\},\{t\}}$, $L_{\{z\},\{t\}}$, and $L_{\{t\},\{x,y,z\}}$.
If $\lambda\ne\infty$, then the singular points of the surface $S_\lambda$ contained in one of these lines
are the points $P_{\{y\},\{z\},\{t\}}$, $P_{\{x\},\{z\},\{t\}}$, and $P_{\{x\},\{y\},\{t\}}$,
which are du Val singular points of type $\mathbb{A}_3$,
and the points $P_{\{x\},\{y\},\{z\}}$, $P_{\{x\},\{t\},\{y,z\}}$, $P_{\{y\},\{t\},\{x,z\}}$,
and $P_{\{z\},\{t\},\{x,y\}}$, which are isolated ordinary double points.
Then $[\mathsf{f}^{-1}(\lambda)]=1$ for every $\lambda\ne\infty$ by Corollary~\ref{corollary:irreducible-fibers}.
Since $h^{1,2}(X)=0$, this proves \eqref{equation:main-1} in Main Theorem in this case.
\end{example}

This approach works for $55$ deformation families of smooth Fano threefolds.

\subsection{Base points and base curves.}
\label{subsection:scheme-step-6}

In many cases, we cannot apply Corollary~\ref{corollary:irreducible-fibers} to prove \eqref{equation:main-1} in Main Theorem,
simply because the pencil $\mathcal{S}$ contains surfaces that have non-du Val singularities in its base locus.
In fact, quite often, the pencil  $\mathcal{S}$ contains reducible surfaces, so that they have non-isolated singularities.
To deal with these cases, we have to refine the formula \eqref{equation:number-of-irredubicle-components-simple}.
Let us do first step in this direction.

Let $C_1,\ldots,C_r$ be irreducible curves contained in the base locus of the pencil $\mathcal{S}$.
With very few exceptions (see Subsections~\ref{section:r-3-n-8}, \ref{section:r-3-n-22}, \ref{section:r-3-n-24}, \ref{section:r-3-n-29}, \ref{section:r-7-n-1} and \ref{section:r-8-n-1}),
these curves are either lines or conics.
For every base curve $C_j$, we let
\begin{equation}
\label{equation:C-P3}
\mathbf{C}_j^\lambda=\boxed{\text{the number of indices $i\in\{1,\ldots,n\}$ such that $\mathbf{a}_i^\lambda>0$ and $\pi(E_i)=C_j$}}\,.
\end{equation}
Let $\Sigma$ be the (finite) subset of the base locus of the pencil $\mathcal{S}$ such that for every $P\in\Sigma$ there is an index $i\in\{1,\ldots,n\}$ such that $\pi(E_i)=P$.
For every $P\in\Sigma$, we let
\begin{equation}
\label{equation:D}
\mathbf{D}_P^\lambda=\boxed{\text{the number of indices $i\in\{1,\ldots,n\}$ such that $\mathbf{a}_i^\lambda>0$ and $\pi(E_i)=P$}}\,.
\end{equation}
We say that $\mathbf{D}_P^\lambda$ is the \emph{defect} of the \emph{fixed} singular point $P$.

Using \eqref{equation:number-of-irredubicle-components-simple}, we see that
\begin{equation}
\label{equation:equation:number-of-irredubicle-components-refined}
\big[\mathsf{f}^{-1}(\lambda)\big]=\big[S_{\lambda}\big]+\sum_{i=1}^{r}\mathbf{C}_j^\lambda+\sum_{P\in\Sigma}\mathbf{D}_P^\lambda.
\end{equation}
If $P$ is a point in $\Sigma$ such that the quartic surface $S_\lambda$ has du Val singularity at $P$,
then its defect vanishes by Lemma~\ref{lemma:irreducible-fibers}.
However, the defect may also vanish if $S_\lambda$ has worse than du Val singularity at the point $P$.

\begin{remark}
\label{remark:fixed-singular-points}
For a general $\lambda\in\mathbb{C}$, the singular points of the surface $S_\lambda$ are all du Val.
Moreover, they are of two kinds: those whose coordinates depend on $\lambda$,
and those whose coordinates do not depend on $\lambda$.
We call the latter ones \emph{fixed} singular points, and we call the former ones \emph{floating} singular points.
The set $\Sigma$ consists of {fixed} singular points.
\end{remark}

For every point $P\in\Sigma$, the number $\mathbf{D}_P^\lambda$ can be computed locally near $P$.
We will show how to do this later, see formula~\eqref{equation:D-A-B} below.
Now let us show how to compute the number $\mathbf{C}_i^\lambda$ defined in \eqref{equation:C-P3}.
For every $\lambda\in\mathbb{C}\cup\{\infty\}$ and every $i\in\{1,\ldots,r\}$, we let
$$
\mathbf{M}_i^\lambda=\mathrm{mult}_{C_i}\big(S_\lambda\big).
$$
For any two distinct quartic surfaces $S_{\lambda_1}$ and $S_{\lambda_2}$ in the pencil $\mathcal{S}$,
we have
$$
S_{\lambda_1}\cdot S_{\lambda_2}=\sum_{i=1}^{r}\textbf{m}_iC_i
$$
for some positive numbers $\textbf{m}_1,\ldots,\textbf{m}_r$.
Then $\textbf{m}_i\geqslant\mathbf{M}_i^\lambda$ for every $\lambda\in\mathbb{C}\cup\{\infty\}$.

\begin{lemma}
\label{lemma:main}
Fix $\lambda\in\mathbb{C}\cup\{\infty\}$ and $a\in\{1,\ldots,r\}$. Then
$$
\mathbf{C}_a^\lambda=\left\{\aligned
&0\ \text{if}\ \mathbf{M}_a^\lambda=1,\\
&\textbf{m}_a-1\ \text{if}\ \mathbf{M}_a^\lambda\geqslant 2.\\
\endaligned
\right.
$$
\end{lemma}

\begin{proof}
The required assertion can be checked in a general point of the curve $C_a$.
Because of this, we may assume that $C_a$ is smooth.
To resolve the base locus of the pencil $\mathcal{S}$ at general point of the curve $C_a$,
we observe that general surfaces in this pencil are smooth at general point of the curve $C_a$.
This implies that there exists a composition of $\mathbf{m}_a\geqslant 1$ blow ups of smooth curves
$$
\xymatrix{V_{\mathbf{m}_a}\ar@{->}[rr]^{\gamma_{\mathbf{m}_a}}&&V_{\mathbf{m}_a-1}\ar@{->}[rr]^{\gamma_{\mathbf{m}_a-1}}&&\cdots\ar@{->}[rr]^{\gamma_{2}}&&V_{1}\ar@{->}[rr]^{\gamma_{1}}&&\mathbb{P}^3}
$$
such that $\gamma_1$ is the blow up of the curve $C_a$,
for $i>1$ the morphism $\gamma_i$ is a blow up of a smooth curve $C_a^{i-1}\subset V_{i-1}$ such that
$$
\gamma_{i-1}\big(C_a^{i-1}\big)=C_a^{i-2}\subset V_{i-2}
$$
and the curve $C_a^{i-1}$ is contained in the proper transform of general surface in $\mathcal{S}$ on the threefold $V_{i-1}$.
Here, we have $V_0=\mathbb{P}^3$ and $C_a^0=C_a$.

For every index $i\in\{1,\ldots,\mathbf{m}_a\}$, let $F_i$ be the exceptional surface of the morphism $\gamma_i$.
Then $C_a^{i}\subset F_i$, and the curve $C_a^{i}$ is a section of the $\mathbb{P}^1$-bundle $G_i\to C_a^{i-1}$ induced by $\gamma_i$.
Note that $C_a^{i}$ is not contained in the proper transform of the surface $F_{i-1}$.

For every $i\in\{0,1,\ldots,\mathbf{m}_a\}$, denote by $S_\lambda^i$ the proper transform of the surface $S_\lambda$ on the threefold $V_i$.
Then
$$
\sum_{i=0}^{\mathbf{m}_a-1}\mathrm{mult}_{C_a^{i}}\big(S_\lambda^i\big)=\mathbf{m}_a.
$$
Moreover, for every $b\in\{1,\ldots,n\}$ such that
$\beta(E_b)=C_a$ there is $j\in\{1,\ldots,\mathbf{m}_a-1\}$ such that $E_b$ is the proper transform of the divisor $F_j$ on the threefold $V$
in diagram~\eqref{equation:diagram}.
Vice versa, for every $j\in\{1,\ldots,\mathbf{m}_a-1\}$, there is $b\in\{1,\ldots,n\}$ such that
$\beta(E_b)=C_a$, and $E_b$ is the proper transform of the divisor $F_j$ on the threefold $V$,
which implies that
$$
\mathbf{a}_b^\lambda=\sum_{i=0}^{j-1}\Big(\mathrm{mult}_{C_a^{i}}\big(S_\lambda^i\big)-1\Big).
$$
On the other hand, we also have
$$
\mathbf{M}_a^\lambda=\mathrm{mult}_{C_a}\big(S_\lambda\big)\geqslant\mathrm{mult}_{C_a^{1}}\big(S_\lambda^1\big)\geqslant\mathrm{mult}_{C_a^{2}}\big(S_\lambda^2\big)\geqslant\cdots\geqslant\mathrm{mult}_{C_a^{j-1}}\big(S_\lambda^{j-1}\big)\geqslant 0.
$$
Using this, we obtain a dichotomy:
\begin{itemize}
\item either $\mathbf{M}_a^\lambda=1$ and $\mathbf{a}_b^\lambda=0$ for every $b\in\{1,\ldots,n\}$ such that $\beta(E_b)=C_a$,

\item or $\mathbf{M}_a^\lambda\geqslant 2$ and $\mathbf{a}_b^\lambda>0$ for every $b\in\{1,\ldots,n\}$ such that $\beta(E_b)=C_a$ with a single exception:
when $E_b$ is a proper transform of the divisor $F_{\mathbf{m}_a}$ on the threefold $V$.
\end{itemize}
This immediately implies the required assertion.
\end{proof}

Let us show how to apply Lemma~\ref{lemma:main} to prove \eqref{equation:main-1} in Main Theorem in one case.

\begin{example}
\label{example:r-3-n-2}
Suppose that $X$ is a smooth Fano threefold in the family \textnumero $3.2$.
Then one its mirror partner is given by the Laurent polynomial
$$
\frac{z^{2}}{xy}+z+\frac{3z}{y}+\frac{3z}{x}+x+y+\frac{z}{xy}+\frac{3x}{y}+\frac{3y}{x}+\frac{1}{y}+\frac{1}{x}+\frac{x^{2}}{yz}+\frac{3x}{z}+\frac{3y}{z}+\frac{y^{2}}{xz}.
$$
This is the Minkowski polynomial \textnumero $2569$.
The pencil $\mathcal{S}$ is given by
\begin{multline*}
z^3t+xyz^2+3z^2xt+3z^2yt+x^2yz+y^2xz+z^2t^2+3x^2tz+3y^2tz+\\
+t^2xz+t^2yz+x^3t+3x^2yt+3y^2xt+y^3t=\lambda xyzt.
\end{multline*}
Suppose that $\lambda\ne\infty$.
Let $\mathcal{C}_1$ and $\mathcal{C}_2$ be conics that are given by $x=y^2+2yz+z^2+tz=0$ and
$y=x^2+2xz+z^2+tz=0$, respectively.
Then
$$
S_{\infty}\cdot S_{\lambda}=2L_{\{x\},\{t\}}+2L_{\{y\},\{t\}}+2L_{\{z\},\{t\}}+L_{\{x\},\{y,z\}}+L_{\{y\},\{x,z\}}+3L_{\{z\},\{x,z\}}+L_{\{t\},\{x,y,z\}}+\mathcal{C}_1+\mathcal{C}_2.
$$
Thus, we have $r=9$, and may assume that
$C_1=\mathcal{C}_1$,  $C_2=\mathcal{C}_2$,
$C_3=L_{\{x\},\{t\}}$, $C_4=L_{\{y\},\{t\}}$, $C_5=2L_{\{z\},\{t\}}$,
$C_6=L_{\{x\},\{y,z\}}$, \mbox{$C_7=L_{\{y\},\{x,z\}}$}, $C_8=L_{\{z\},\{x,z\}}$, and $C_9=L_{\{t\},\{x,y,z\}}$.
Then~$\textbf{m}_1=\textbf{m}_2=\textbf{m}_6=\textbf{m}_7=\textbf{m}_9=1$,
$\textbf{m}_3=\textbf{m}_4=\textbf{m}_5=2$, and $\textbf{m}_8=3$.
We have
$$
\Sigma=\Big\{P_{\{x\},\{y\},\{z\}},P_{\{x\},\{t\},\{y,z\}},P_{\{y\},\{t\},\{x,z\}},P_{\{z\},\{t\},\{x,y\}}\Big\}.
$$
If~$\lambda\ne -6$, then $S_\lambda$ is irreducible and has isolated singularities.
In this case, the surface~$S_\lambda$ has du Val singularities at
$P_{\{x\},\{y\},\{z\}}$, $P_{\{x\},\{t\},\{y,z\}}$, $P_{\{y\},\{t\},\{x,z\}}$, $P_{\{z\},\{t\},\{x,y\}}$,
and it does not have other singular points in the base locus of the pencil $\mathcal{S}$.
Then $[\mathsf{f}^{-1}(\lambda)]=1$ for every $\lambda\ne -6$ by Corollary~\ref{corollary:irreducible-fibers}.
On the other hand, we have
$$
S_{-6}=H_{\{x,y,z\}}+\mathbf{S},
$$
where $\mathbf{S}$ is a cubic surface that is given by $zt^2+x^2t+xyt+2xzt+y^2t+2yzt+z^2t+xyz=0$.
We have
$\mathbf{M}_1^{-6}=\mathbf{M}_2^{-6}=\mathbf{M}_3^{-6}=\mathbf{M}_4^{-6}=\mathbf{M}_5^{-6}=\mathbf{M}_6^{-6}=\mathbf{M}_7^{-6}=\mathbf{M}_9^{-6}=1$
and $\mathbf{M}_8^{-6}=2$. Thus, it follows from Lemma~\ref{lemma:main} that $\mathbf{C}_8^{-6}=2$ and
$$
\mathbf{C}_1^{-6}=\mathbf{C}_2^{-6}=\mathbf{C}_3^{-6}=\mathbf{C}_4^{-6}=\mathbf{C}_5^{-6}=\mathbf{C}_6^{-6}=\mathbf{C}_7^{-6}=\mathbf{C}_9^{-6}=0.
$$
Note that $S_{-6}$ has du Val singularities of type $\mathbb{A}$ or non-isolated ordinary double singularities
at the points of the set~$\Sigma$.
We will see in Lemma~\ref{lemma:normal-crossing} that this gives $\mathbf{D}_P^{-6}=0$ for each \mbox{$P\in\Sigma$}.
Then $[\mathsf{f}^{-1}(-6)]=4$ by \eqref{equation:equation:number-of-irredubicle-components-refined},
which gives \eqref{equation:main-1} in Main~Theorem.
\end{example}

Unlike what we just saw in Example~\ref{example:r-3-n-2},
the numbers $\mathbf{D}_P^\lambda$ in \eqref{equation:equation:number-of-irredubicle-components-refined} do not always vanish for every $P\in\Sigma$.
Thus, we have to provide an algorithm how to compute them.
To do this, we should choose a suitable birational morphism $\pi$ in \eqref{equation:diagram}.

\subsection{Blowing up fixed singular points.}
\label{subsection:scheme-step-7}

We can (partially) resolve all {fixed} singular points of the surfaces in the pencil $\mathcal{S}$
by consecutive blow ups of $\mathbb{P}^3$ in finitely many points.
This gives a birational map $\alpha\colon U\to\mathbb{P}^3$
such that the proper transform of the pencil $\mathcal{S}$ on the threefold $U$ does not have {fixed} singular points.
Let $\widehat{\mathcal{S}}$ be the proper transform of the pencil $\mathcal{S}$ on the threefold $U$.
Then we can (uniquely) choose $\alpha$ such that $\widehat{\mathcal{S}}\sim -K_{U}$.

\begin{remark}
\label{remark:U-main-property}
By construction, for every point $P$ in the base locus of the pencil $\widehat{\mathcal{S}}$, there exists
a surface in $\widehat{\mathcal{S}}$ that is smooth at $P$.
Note that a general surface in $\widehat{\mathcal{S}}$ is not necessarily smooth. However in most of the cases it is smooth.
In the remaining cases, it has du Val singular points of type $\mathbb{A}$ by \cite[Theorem~4.4]{Kollar}.
\end{remark}

Denote by $\widehat{C}_1,\ldots,\widehat{C}_r$ proper transforms of the curves $C_1,\ldots,C_r$ on the threefold~$U$, respectively.
Then these curves are contained in the base locus of the pencil $\widehat{\mathcal{S}}$.
However, the pencil $\widehat{\mathcal{S}}$ always has other base curves.
Denote them by $\widehat{C}_{r+1},\ldots,\widehat{C}_s$, where $s>r$.
A~posteriori, all base curves of the pencil $\widehat{\mathcal{S}}$ are smooth rational curves.

For any two distinct surfaces $\widehat{S}_{\lambda_1}$ and $\widehat{S}_{\lambda_2}$ in the pencil $\widehat{\mathcal{S}}$, we have
\begin{equation}
\label{equation:multiplicities}
\widehat{S}_{\lambda_1}\cdot\widehat{S}_{\lambda_2}=\sum_{i=1}^{s}\textbf{m}_i\widehat{C}_i
\end{equation}
for some positive numbers $\textbf{m}_1,\ldots,\textbf{m}_s$.
Since general surfaces in $\widehat{\mathcal{S}}$ are smooth at general points of the curves $\widehat{C}_1,\ldots,\widehat{C}_s$,
we can resolve the base locus of the pencil $\widehat{\mathcal{S}}$ by
$$
\textbf{m}_{1}+\textbf{m}_2+\textbf{m}_3+\cdots+\textbf{m}_{s}
$$
consecutive blow ups of smooth rational curves (cf. Remarks~\ref{remark:r-2-n-1-singular-curve} and \ref{remark:r-10-n-1-singular-curve}).
This gives a birational morphism
$\beta\colon V^\prime\to U$ such that there exists a commutative diagram
$$
\xymatrix{
U\ar@{->}[d]_{\alpha}&&V^\prime\ar@{->}[d]^{\mathsf{g}^\prime}\ar@{->}[ll]_{\beta} \\
\mathbb{P}^3\ar@{-->}[rr]_{\phi}&&\mathbb{P}^1}
$$
where $\mathsf{g}^\prime$ is a morphism whose general fibers are smooth $K3$ surfaces.

By construction, the threefold $V^\prime$ is smooth, and the anticanonical divisor $-K_{V^\prime}$ is rationally equivalent to a scheme fiber of the fibration $\mathsf{g}^\prime$.
This immediately implies that there exists a composition of flops $\eta\colon V\dasharrow V^\prime$ that makes the following diagram commutative:
$$
\xymatrix{
&V\ar@{-->}[rrrr]^{\eta}\ar@{->}[drr]_{\mathsf{g}}\ar@{->}[dd]_{\pi}&&&& V^\prime\ar@{->}[dll]^{\mathsf{g}^\prime}\ar@{->}[dd]^{\beta}\\
&&&\mathbb{P}^1&&&&\\
&\mathbb{P}^3\ar@{-->}[rru]_{\phi}&&&&U\ar@{->}[llll]^{\alpha}}
$$
Hence, in the following, we will always assume that $V=V^\prime$, $\pi=\alpha\circ\beta$, $\eta=\mathrm{Id}$ and $\mathsf{g}^\prime=\mathsf{g}$.
This gives us the commutative diagram
\begin{equation}
\label{equation:main-diagram}
\xymatrix{
&&U\ar@{->}[dd]_{\alpha}&&V\ar@{-->}[rr]^{\chi}\ar@{->}[dd]_{\mathsf{g}}\ar@{->}[ll]_{\beta}\ar@{->}[ddll]^{\pi}&&Z\ar@{->}[dd]^{\mathsf{f}}\\
&&&&&&&&\\
&&\mathbb{P}^3\ar@{-->}[rr]_{\phi}&&\mathbb{P}^1\ar@{=}[rr]&&\mathbb{P}^1.&&}
\end{equation}
Let $k=\mathrm{rk}\,\mathrm{Pic}(U)-1$.
For simplicity, we assume that $\beta(E_1),\ldots,\beta(E_k)$ are exceptional surfaces of the morphism $\alpha$,
while the surfaces $E_{k+1},\ldots,E_n$ are contracted by $\beta$.

\subsection{Counting multiplicities.}
\label{subsection:scheme-step-8}

Let us show how to explicitly compute $\mathbf{D}_P^\lambda$ in \eqref{equation:equation:number-of-irredubicle-components-refined} for every point $P\in\Sigma$.
To do this, we denote by $\widehat{E}_1,\ldots,\widehat{E}_k$
the proper transforms of the surfaces  $E_1,\ldots,E_k$ on the threefold $U$, respectively.
For every $\lambda\in\mathbb{C}\cup\{\infty\}$, we let
\begin{equation}
\label{equation:log-pull-back}
\widehat{D}_\lambda=\widehat{S}_\lambda+\sum_{i=1}^{k}\mathbf{a}_i^\lambda\widehat{E}_i.
\end{equation}
Then $\widehat{D}_\lambda\sim -K_{U}$, and the numbers $\mathbf{a}_1^\lambda,\ldots,\mathbf{a}_k^\lambda$ are uniquely determined by this rational equivalence.
Furthermore, we have $\widehat{D}_\lambda\in\widehat{\mathcal{S}}$ by construction.

\begin{lemma}
\label{lemma:log-pull-back}
Let $P$ be a point in the set $\Sigma$. If $\mathrm{mult}_{P}(S_\lambda)=2$, then
$$
\mathbf{a}_i^\lambda=0
$$
for every $i\in\{1,\ldots,k\}$ such that $\alpha(\widehat{E}_i)=P$.
\end{lemma}

\begin{proof}
Straightforward.
\end{proof}

For every {fixed} singular point $P\in\Sigma$, we let
\begin{equation}
\label{equation:A}
\mathbf{A}_{P}^\lambda=\boxed{\text{the number of indices $i\in\{1,\ldots,k\}$ such that $\mathbf{a}_i^\lambda>0$ and $\alpha(\widehat{E}_i)=P$}}\,.
\end{equation}
Then the assertion of Lemma~\ref{lemma:log-pull-back} can be restates as follows.

\begin{corollary}
\label{corollary:log-pull-back}
If $\mathrm{mult}_{P}(S_\lambda)=2$ for $P\in\Sigma$, then $\mathbf{A}_{P}^\lambda=0$.
\end{corollary}

For every $\lambda\in\mathbb{C}\cup\{\infty\}$ and every $a\in\{r+1,\ldots,s\}$, we let
\begin{equation}
\label{equation:C}
\mathbf{C}_{a}^\lambda=\boxed{\text{the number of indices $i\in\{1,\ldots,n\}$ such that $\mathbf{a}_i^\lambda>0$ and $\beta(E_i)=\widehat{C}_a$}}\,.
\end{equation}
For every $\lambda$ and every $a\in\{1,\ldots,s\}$, we let
\begin{equation}
\label{equation:M}
\mathbf{M}_a^\lambda=\mathrm{mult}_{\widehat{C}_a}\big(\widehat{D}_\lambda\big).
\end{equation}

\begin{lemma}
\label{lemma:main-2}
Fix $\lambda\in\mathbb{C}\cup\{\infty\}$ and $a\in\{1,\ldots,s\}$. Then
$$
\mathbf{C}_a^\lambda=\left\{\aligned
&0\ \text{if}\ \mathbf{M}_a^\lambda=1,\\
&\textbf{m}_a-1\ \text{if}\ \mathbf{M}_a^\lambda\geqslant 2.\\
\endaligned
\right.
$$
\end{lemma}

\begin{proof}
See the proof of Lemma~\ref{lemma:main}.
\end{proof}

On the other hand, it follows from \eqref{equation:number-of-irredubicle-components-simple} that
\begin{equation}
\label{equation:equation:number-of-irredubicle-components-refined-2}
\big[\mathsf{f}^{-1}(\lambda)\big]=\big[\widehat{D}_\lambda\big]+\sum_{i=1}^{s}\mathbf{C}_i^\lambda=\big[S_{\lambda}\big]+\sum_{P\in\Sigma}\mathbf{A}_P^\lambda+\sum_{i=1}^{s}\mathbf{C}_i^\lambda.
\end{equation}

Comparing the formulas \eqref{equation:equation:number-of-irredubicle-components-refined}
and \eqref{equation:equation:number-of-irredubicle-components-refined-2}, we obtain the formula for the \emph{defect}
\begin{equation}
\label{equation:D-A-B}
\mathbf{D}_{P}^\lambda=\mathbf{A}_{P}^\lambda+\sum_{\substack{i=r+1\\\alpha(\widehat{C}_i)=P}}^s\mathbf{C}_i^\lambda
\end{equation}
for every point $P\in\Sigma$.
Similarly, using \eqref{equation:equation:number-of-irredubicle-components-refined-2} and Lemma~\ref{lemma:main-2}, we get

\begin{corollary}
\label{corollary:main-2}
If $\mathbf{M}_i^\lambda=1$ for every $i\in\{1,\ldots,s\}$, then $[\mathsf{f}^{-1}(\lambda)]=[\widehat{D}_\lambda]$.
\end{corollary}

Let us show how to apply this handy result.

\begin{example}
\label{example:r-8-n1}
Suppose that $X=\mathbb{P}^1\times\mathbf{S}_3$, where $\mathbf{S}_3$ is a smooth cubic surface in~$\mathbb{P}^3$.
This is the family \textnumero $8.1$ in \cite{IP}.
One of its mirror partner is given by the Minkowski polynomial \textnumero $768$,
which is the Laurent polynomial
$$
\frac{1}{yz}+\frac{3}{y}+\frac {3z}{y}+x+\frac{{z}^{2}}{y}+\frac{3}{z}+3z+\frac{1}{x}+\frac {3y}{z}+3y+\frac{{y}^{2}}{z}.
$$
Then the corresponding quartic pencil $\mathcal{S}$ is given by
$$
t^3x+3t^2xz+3z^2xt+x^2zy+z^3x+3t^2xy+3z^2xy+t^2zy+3y^2xt+3y^2xz+y^3x=\lambda xyzt,
$$
and it has $6$ base curves: $C_1=L_{\{x\},\{y\}}$, $C_2=L_{\{x\},\{z\}}$, $C_3=L_{\{x\},\{t\}}$, $C_4=L_{\{y\},\{t,z\}}$, $C_5=L_{\{z\},\{t,y\}}$, and $C_6$
is the singular cubic curve $t=xyz+y^3+3y^2z+3yz^2+z^3=0$.
Suppose that $\lambda\ne\infty$. Then
$$
S_\lambda\cdot S_\infty=2C_1+2C_2+3C_3+3C_4+3C_5+C_6,
$$
and the surface $S_\lambda$ is irreducible.
Moreover, if $\lambda\ne -4$ and $\lambda\ne -8$, then the singularities of the surface $S_\lambda$ are du Val,
so that $[\mathsf{f}^{-1}(\lambda)]=1$ by Corollary~\ref{corollary:irreducible-fibers}.
However, the singular locus of the surface $S_{-4}$ consists of the point $P_{\{x\},\{y\},\{z\}}$ and the line $x-t=y+z+t=0$.
Similarly, the singular locus of the surface $S_{-8}$ consists of the point $P_{\{x\},\{y\},\{z\}}$ and the line $x+t=y+z+t=0$.
Thus, we cannot apply Corollary~\ref{corollary:irreducible-fibers} when $\lambda=-4$ or $\lambda=-8$.
Nevertheless, we have $[\mathsf{f}^{-1}(-4)]=1$ and $[\mathsf{f}^{-1}(-8)]=1$.
To show this, observe that
$$
\Sigma=\Big\{P_{\{y\},\{z\},\{t\}},P_{\{x\},\{t\},\{y,z\}}\Big\}.
$$
Moreover, if $\lambda\ne -4$ and $\lambda\ne-8$, then $P_{\{y\},\{z\},\{t\}}$ is a singular point of the surface $S_\lambda$ of type $\mathbb{A}_2$,
and the point $P_{\{x\},\{t\},\{y,z\}}$ is a singular point of the surface $S_\lambda$ of type $\mathbb{A}_5$.
The~birational morphism $\alpha\colon U\to\mathbb{P}^3$ can be decomposed as follows:
$$
\xymatrix{
&&U_2\ar@{->}[lld]_{\alpha_2}&&U_3\ar@{->}[ll]_{\alpha_3}&&\\
U_1\ar@{->}[drrr]_{\alpha_1}&&&&&&U\ar@{->}[dlll]^\alpha\ar@{->}[llu]_{\alpha_4}.\\
&&&\mathbb{P}^3&&& }
$$
Here $\alpha_1$ is the blow up of the point $P_{\{y\},\{z\},\{t\}}$,
the morphism $\alpha_2$ is the blow up of the preimage of the point $P_{\{x\},\{t\},\{y,z\}}$,
the morphism $\alpha_3$ is the blow up of a point in $\alpha_2$-exceptional surface,
and $\alpha_4$ is the blow up of a point in $\alpha_3$-exceptional surface.
We may assume that $\widehat{E}_4$ is $\alpha_4$-exceptional surface.
Likewise, we may assume that $\widehat{E}_1$, $\widehat{E}_2$, and $\widehat{E}_3$ are proper transforms on $U$ of the exceptional surfaces of the morphisms $\alpha_1$, $\alpha_2$, and $\alpha_3$, respectively.
Then
$$
\widehat{D}_\infty=\widehat{S}_{\infty}+\widehat{E}_1\sim \widehat{S}_{\lambda}\sim -K_U.
$$
One can show that  $\widehat{E}_2$, $\widehat{E}_3$, and $\widehat{E}_4$ do not contain base curves of the pencil $\widehat{\mathcal{S}}$,
and the surface  $\widehat{E}_1$ contains two base curves of the pencil $\widehat{\mathcal{S}}$.
They are cut out on $\widehat{E}_1$ by the proper transforms on $U$ of the planes $H_{\{y\}}$ and $H_{\{z\}}$.
Let us denote them by $\widehat{C}_7$ and $\widehat{C}_8$, respectively.
Then $\widehat{S}_{\lambda}$ and $\widehat{S}_{\infty}+\widehat{E}_1$ generate the pencil $\widehat{\mathcal{S}}$ and
$$
\widehat{S}_{\lambda}\cdot\Big(\widehat{S}_{\infty}+\widehat{E}_1\Big)=2\widehat{C}_1+2\widehat{C}_2+3\widehat{C}_3+3\widehat{C}_4+3\widehat{C}_5+\widehat{C}_6+2\widehat{C}_7+2\widehat{C}_8.
$$
Note that
$\mathbf{M}_1^\lambda=\mathbf{M}_2^\lambda=\mathbf{M}_3^\lambda=\mathbf{M}_4^\lambda=\mathbf{M}_5^\lambda=\mathbf{M}_6^\lambda=\mathbf{M}_7^\lambda=\mathbf{M}_8^\lambda=1$ for every $\lambda\in\mathbb{C}$.
Therefore, using Corollary~\ref{corollary:main-2}, we conclude that $[\mathsf{f}^{-1}(\lambda)]=1$ for every $\lambda\in\mathbb{C}$.
Thus, we see that \eqref{equation:main-1} in Main Theorem holds in this case, since $h^{1,2}(X)=0$.
\end{example}

\subsection{Extra notation}
\label{subsection:scheme-extra-notations}

In Example~\ref{example:r-8-n1}, we explicitly decomposed the birational morphism $\alpha$ in \eqref{equation:main-diagram} as a composition of blow ups.
To verify \eqref{equation:main-1} in Main Theorem, we have to do the same many times.
To save space, let us introduce common notations that will be used in all these decompositions.

Recall that $\alpha$ is a composition of $k\geqslant 1$ blow ups of points.
Suppose that we have the following commutative diagram:
$$
\xymatrix{
&U_2\ar@{->}[d]_{\alpha_2}&&U_3\ar@{->}[ll]_{\alpha_3}&&\cdots\ar@{->}[ll]_{\alpha_4}&&U_a\ar@{->}[ll]_{\alpha_a}&&\\
&U_1\ar@{->}[drrr]_{\alpha_1}&&&&&&\\
&&&&\mathbb{P}^3&&U\ar@{->}[ll]^\alpha\ar@{->}[uur]_{\gamma}& }
$$
where $a\leqslant k$, each $\alpha_i$ is a blow up of a point, and $\gamma$ is a (possibly biregular) birational morphism.
Then we denote the exceptional divisor of $\alpha_i$ by $\mathbf{E}_i$.
Moreover, for every $j\geqslant i$, we denote by $\mathbf{E}_i^j$ the proper transform of the divisor $\mathbf{E}_i$ on $U_j$.
Furthermore, we will always assume that the proper transform of the surface $\mathbf{E}_i$ on $U$ is the divisor $\widehat{E}_i$.

For every $\lambda\in\mathbb{C}\cup\{\infty\}$ and $i\leqslant a$, we denote by $S_\lambda^i$ the proper transform
of the quartic surface $S_\lambda$ on the threefold $U_i$.
Similarly, we denote by $\mathcal{S}^i$ the proper transform on $U_i$ of the pencil $\mathcal{S}$,
and we denote by $D_\lambda^i$ the divisor in the pencil $\mathcal{S}^i$ that contains the surface~$S_\lambda^i$.
Then $D_\lambda^i$ is just the image of the divisor $\widehat{D}_\lambda$ on the threefold $U_i$.

We denote by $C_1^i,\ldots,C_r^i$ the proper transforms on $U_i$ of the curves $C_1,\ldots,C_r$, respectively.
Similarly, if the surface $\mathbf{E}_i$ contains a base curve of the pencil $\mathcal{S}^i$,
then we denote this curve by $C_j^i$ for an appropriate $j>r$.
We will always assume that its proper transform on the threefold $U$ is the base curve $\widehat{C}_j$, which we introduced earlier.

\subsection{Good double points.}
\label{subsection:scheme-step-9}

As we already saw in Example~\ref{example:r-3-n-2},
in some cases all defects $\mathbf{D}_P^\lambda$ in \eqref{equation:equation:number-of-irredubicle-components-refined} vanish,
so that we do not need to blow up $\mathbb{P}^3$ to compute $[\mathsf{f}^{-1}(\lambda)]$.
A handy observation is that
$$
\mathbf{D}_P^\lambda=0
$$
for $P\in\Sigma$ if the rank of the quadratic form of the (local) defining
equation of the quartic surface $S_\lambda$ at the point $P$ is at least $2$.
We will call such points \emph{good} double points.
This unifies du Val singular points of type $\mathbb{A}$ and non-isolated ordinary double points.

\begin{lemma}
\label{lemma:normal-crossing}
Let $P$ be a {fixed} singular point in $\Sigma$.
Suppose that $P$ is a \emph{good} double point of the surface $S_\lambda$.
Then $\mathbf{D}_{P}^\lambda=0$.
\end{lemma}

\begin{proof}
By Corollary~\ref{corollary:log-pull-back}, we have $\mathbf{A}_{P}=0$.
Therefore, it follows from \eqref{equation:D-A-B} that we have to show that $\mathbf{C}_j^\lambda=0$
for every $j>r$ such that $\alpha(\widehat{C}_j)=P$.
Let $\widehat{E}_i$ be $\alpha$-exceptional surface such that $\alpha(\widehat{E}_i)=P$,
and let $\widehat{C}_j$ be a base curve of the pencil $\widehat{\mathcal{S}}$ that is contained in $\widehat{E}_i$.
By Lemma~\ref{lemma:main-2}, it is enough to show that
$$
\mathbf{M}_j^\lambda=\mathrm{mult}_{\widehat{C}_j}\big(\widehat{D}_\lambda\big)=1.
$$
To do this, we may assume that $\alpha\colon U\to\mathbb{P}^3$ is the blow up of the point $P$,
and $\widehat{E}_i$ is the exceptional divisor of this blow up.
Then the restriction $\widehat{D}_\lambda\vert_{\widehat{E}_i}$ is a union of two distinct lines in $\widehat{E}_i\cong\mathbb{P}^2$.
In particular, the surface $\widehat{D}_\lambda=\widehat{S}_\lambda$ is smooth at general points of any of these lines and the assertion follows.
\end{proof}

\begin{corollary}
\label{corollary:normal-crossing-simple}
Suppose that every fixed singular point of the pencil $\mathcal{S}$ is a {good} double point of the surface $S_\lambda$.
Then
$$
\big[\mathsf{f}^{-1}(\lambda)\big]=\big[S_{\lambda}\big]+\sum_{i=1}^{r}\mathbf{C}_i^\lambda.
$$
\end{corollary}

Let us show how to apply this corollary in one simple example.

\begin{example}
\label{example:r-3-n-11}
Suppose that $X$ is contained in the family \textnumero $3.11$ in \cite{IP}.
Then its mirror partner is given by the Minkowski polynomial \textnumero $1518$,
which is the Laurent polynomial
$$
x+y+z+{\frac {z}{x}}+{\frac {z}{y}}+{\frac {y}{x}}+{\frac {z}{xy}}+{\frac {y}{z}}+\frac{1}{x}+\frac{1}{y}+{\frac {y}{xz}}+{\frac {1}{xy}}.
$$
Thus, the pencil $\mathcal{S}$ is given by the equation
$$
xyz^2+x^2yz+xy^2z+xz^2t+yz^2t+y^2zt+z^2t^2+xy^2t+xzt^2+yzt^2+y^2t^2+zt^3=\lambda xyzt,
$$
and its base locus consists of the lines $L_{\{x\},\{t\}}$, $L_{\{y\},\{z\}}$,
$L_{\{y\},\{t\}}$, $L_{\{z\},\{t\}}$, $L_{\{x\},\{z,t\}}$, $L_{\{y\},\{x,t\}}$, $L_{\{y\},\{z,t\}}$,
$L_{\{z\},\{x,t\}}$, $L_{\{t\},\{x,y,z\}}$, and the conic $\{x=y^2+yz+zt=0\}$.
If $\lambda\ne-2$ and $\lambda\ne\infty$, then the surface $S_\lambda$ has at most du Val singularities,
so that $[\mathsf{f}^{-1}(\lambda)]=1$ by Corollary~\ref{corollary:irreducible-fibers}.
On the other hand, we have $S_{-2}=H_{\{x,t\}}+\mathbf{S}$,
where $\mathbf{S}$ is an irreducible cubic surface that is given by $xyz+yz^2+z^2t+zt^2+y^2z+y^2t+yzt=0$.
Note also that $S_{-2}$ is smooth at general point of every base curve of the pencil $\mathcal{S}$.
Thus, it follows from  \eqref{equation:equation:number-of-irredubicle-components-refined} that
$$
\big[\mathsf{f}^{-1}(-2)\big]=2+\sum_{P\in\Sigma}\mathbf{D}_P^{-2}.
$$
Furthermore, the set $\Sigma$ consists of the points
$P_{\{y\},\{z\},\{t\}}$, $P_{\{x\},\{z\},\{t\}}$, $P_{\{x\},\{y\},\{t\}}$,  $P_{\{x\},\{t\},\{y,z\}}$, and $P_{\{y\},\{z\},\{x,t\}}$,
and the quadratic terms of the Taylor expansions of the surface $S_{-2}$ at these points can be described as follows:
\begin{itemize}\setlength{\itemindent}{3cm}
\item[$P_{\{y\},\{z\},\{t\}}$:] quadratic term $yz$;

\item[$P_{\{x\},\{z\},\{t\}}$:] quadratic term $(x+t)(z+t)$;

\item[$P_{\{x\},\{y\},\{t\}}$:] quadratic term $(x+t)(y+t)$;

\item[$P_{\{x\},\{t\},\{y,z\}}$:] quadratic term $z(x+t)$;

\item[$P_{\{y\},\{z\},\{x,t\}}$:] quadratic term $z(x+t)$.
\end{itemize}
By Corollary~\ref{corollary:normal-crossing-simple}, we have $\mathbf{D}_P^{-2}=0$ for every $P\in\Sigma$,
so that $[\mathsf{f}^{-1}(-2)]=2$.
Thus, we see that \eqref{equation:main-1} in Main Theorem holds in this case, since $h^{1,2}(X)=1$.
\end{example}

\subsection{Curves on singular quartic surfaces.}
\label{subsection:scheme-step-10}

We will prove \eqref{equation:main-2} in Main Theorem by computing the intersections  form of the curves $C_1,\ldots,C_r$ on a general surface in~the~pencil~$\mathcal{S}$.
To~do this, let $\Bbbk=\mathbb{C}(\lambda)$, let $S_{\Bbbk}$ be the quartic surface in  $\mathbb{P}^3_{\Bbbk}$
that is given by~\eqref{equation:quartic}, and let $\nu\colon\widetilde{S}_{\Bbbk}\to S_{\Bbbk}$
be the minimal resolution of singularities of the surface~$S_{\Bbbk}$.

\begin{lemma}
\label{lemma:cokernel}
Suppose that $\lambda$ is a general element of $\mathbb{C}$.
Then the surface $S_\lambda$ is singular, and it has du Val singularities.
Let $M$ be the $r\times r$ matrix with entries $M_{ij}\in\mathbb{Q}$ that are given by $M_{ij}=C_i\cdot C_j$,
where $C_i\cdot C_j$ is the intersection of the curves $C_i$ and $C_j$
on the surface $S_\lambda$.
Then the right hand side of \eqref{equation:main-2} is equal to
$$
22-\mathrm{rk}\,\mathrm{Pic}\big(\widetilde{S}_{\Bbbk}\big)+\mathrm{rk}\,\mathrm{Pic}\big(S_{\Bbbk}\big)-\mathrm{rk}(M).
$$
\end{lemma}

\begin{proof}
Let $F$ be a general fiber of the morphism $\mathsf{f}$.
Then $H^2(F,\mathbb{R})\cong\mathbb{Z}^{22}$, since $F$ is a smooth $K3$ surface.
This easily implies the required assertion.
\end{proof}

Thus, to verify \eqref{equation:main-2} in Main Theorem, it is enough to show that
\begin{equation}
\label{equation:main-2-simple}\tag{$\bigstar$}
\mathrm{rk}\,\mathrm{Pic}(X)+\mathrm{rk}(M)+\mathrm{rk}\,\mathrm{Pic}\big(\widetilde{S}_{\Bbbk}\big)-\mathrm{rk}\,\mathrm{Pic}\big(S_{\Bbbk}\big)=20,
\end{equation}
where $M$ is the intersection matrix defined in Lemma~\ref{lemma:cokernel}.
For basic properties of the intersection of curves on surfaces with du Val singularities,
see Appendix~\ref{section:intersection}.

Let us show how to check \eqref{equation:main-2-simple} in one case.

\begin{example}
\label{example:r-2-n-34}
Suppose that $X=\mathbb{P}^1\times\mathbb{P}^2$. This is the family \textnumero $2.34$ in \cite{IP}.
One of its mirror partners is given by the Minkowski polynomial \textnumero $4$,
which is the Laurent polynomial $x+y+z+\frac{1}{x}+{\frac {1}{yz}}$.
Then the pencil $\mathcal{S}$ is given by
$$
x^2yz+y^2xz+z^2xy+t^2yz+t^3x=\lambda xyzt,
$$
and its base locus consists of the curves $L_{\{x\},\{y\}}$, $L_{\{x\},\{z\}}$,
$L_{\{x\},\{t\}}$, $L_{\{y\},\{t\}}$, $L_{\{z\},\{t\}}$, and $L_{\{t\},\{x,y,z\}}$.
Suppose that $\lambda\ne\infty$.
Then the singular points of the surface $S_\lambda$ contained in one of these lines
are $P_{\{x\},\{y\},\{t\}}$ and $P_{\{x\},\{z\},\{t\}}$, which are singular points of type $\mathbb{A}_4$,
the points $P_{\{y\},\{z\},\{t\}}$, $P_{\{y\},\{t\},\{x+z\}}$, and $P_{\{z\},\{t\},\{x+y\}}$,
which are singular points  of type~$\mathbb{A}_2$,
and the point $P_{\{x\},\{t\},\{y+z\}}$, which is an isolated ordinary double point of the surface $S_\lambda$.
In particular, we see that \eqref{equation:main-1} in Main Theorem holds by Corollary~\ref{corollary:irreducible-fibers}.
Resolving the singularities of the quartic surface $S_{\Bbbk}$, we also see that
$$
\mathrm{rk}\,\mathrm{Pic}\big(\widetilde{S}_{\Bbbk}\big)=\mathrm{rk}\,\mathrm{Pic}\big(S_{\Bbbk}\big)+15.
$$
Thus, to verify \eqref{equation:main-2-simple}, we have to compute the rank of the intersection matrix of the lines
$L_{\{x\},\{y\}}$, $L_{\{x\},\{z\}}$, $L_{\{x\},\{t\}}$, $L_{\{y\},\{t\}}$, $L_{\{z\},\{t\}}$, and $L_{\{t\},\{x,y,z\}}$ on the surface $S_\lambda$.
This matrix has the same rank as the  intersection matrix of the curves $L_{\{x\},\{y\}}$, $L_{\{x\},\{z\}}$, and $H_\lambda$, since
\begin{multline*}
L_{\{x\},\{y\}}+L_{\{x\},\{z\}}+2L_{\{x\},\{t\}}\sim L_{\{x\},\{y\}}+3L_{\{y\},\{t\}}\sim\\
\sim L_{\{x\},\{z\}}+3L_{\{z\},\{t\}}\sim L_{\{x\},\{t\}}+L_{\{y\},\{t\}}+L_{\{z\},\{t\}}+L_{\{t\},\{x,y,z\}}\sim H_\lambdaþ
\end{multline*}
These rational equivalences follow from
\begin{equation*}
\begin{split}
H_{\{x\}}\cdot S_\lambda&=L_{\{x\},\{y\}}+L_{\{x\},\{z\}}+2L_{\{x\},\{t\}},\\
H_{\{y\}}\cdot S_\lambda&=L_{\{x\},\{y\}}+3L_{\{y\},\{t\}},\\
H_{\{z\}}\cdot S_\lambda&=L_{\{x\},\{z\}}+3L_{\{z\},\{t\}},\\
H_{\{t\}}\cdot S_\lambda&=L_{\{x\},\{t\}}+L_{\{y\},\{t\}}+L_{\{z\},\{t\}}+L_{\{t\},\{x,y,z\}}.
\end{split}
\end{equation*}
On the other hand,
using Propositions~\ref{proposition:du-Val-intersection} and~\ref{proposition:du-Val-self-intersection},
we see that the intersection form of the curves $L_{\{x\},\{y\}}$, $L_{\{x\},\{z\}}$ and $H_\lambda$
on the surface $S_\lambda$ is given by
\begin{center}
\renewcommand\arraystretch{1.42}
\begin{tabular}{|c||c|c|c|}
\hline
 $\bullet$  & $L_{\{x\},\{y\}}$ & $L_{\{x\},\{z\}}$ & $H_{\lambda}$ \\
\hline\hline
 $L_{\{x\},\{y\}}$ &  $-\frac{4}{5}$ & $1$ & $1$ \\
\hline
 $L_{\{x\},\{z\}}$ &  $1$ & $-\frac{4}{5}$ & $1$ \\
\hline
 $H_{\lambda}$  & $1$ & $1$ & $4$ \\
\hline
\end{tabular}
\end{center}
This matrix has rank~$3$, so that \eqref{equation:main-2-simple} holds in this case.
\end{example}

\subsection{Scheme of the proof.}
\label{subsection:scheme-step-11}

In the remaining part of the paper, we prove \eqref{equation:main-1} and \eqref{equation:main-2} in Main Theorem
for every deformation family of smooth Fano threefolds
similar to what we did in Examples~\ref{example:r-3-n-27}, \ref{example:r-3-n-2}, \ref{example:r-8-n1},  \ref{example:r-3-n-11}, and \ref{example:r-2-n-34}.
We will do this case by case reserving one subsection per deformation family.
For convenience, we align the number of the family in~\cite{IP} with the corresponding subsection's number,
and we group families with the same Picard rank in one section.
For example, Subsection~\ref{section:r-4-n-1} contains the proof of Main Theorem for the family \textnumero $4.1$ in \cite{IP},
which consists of smooth divisors of multidegree $(1,1,1,1)$ on $\mathbb{P}^1\times\mathbb{P}^1\times\mathbb{P}^1\times\mathbb{P}^1$.

In every case when $-K_X$ is very ample, we proceed as follows.
First, we choose an appropriate toric Landau--Ginzburg model for the threefold $X$ such that
\eqref{equation:diagram} exists for some pencil $\mathcal{S}$, which is given~by the equation \eqref{equation:quartic}.
Second, we describe the base locus of this pencil.
Third, we describe the singularities of every surface $S_\lambda$ in the pencil $\mathcal{S}$ that are contained in the base locus of this pencil.
This also gives us explicit construction of the birational map $\alpha$ in \eqref{equation:main-diagram},
which can be used to describe the minimal resolution of singularities $\nu\colon\widetilde{S}_{\Bbbk}\to S_{\Bbbk}$.
Using it, we compute $\mathrm{rk}\,\mathrm{Pic}(\widetilde{S}_{\Bbbk})-\mathrm{rk}\,\mathrm{Pic}(S_{\Bbbk})$,
and verify \eqref{equation:main-2-simple} using intersection theory on $S_\lambda$ for general $\lambda\in\mathbb{C}$.
To do this more efficiently, we use basic results about intersection of curves on singular surfaces,
which we present in Appendix~\ref{section:intersection}.

If singular points of the surface $S_\lambda$ contained in the base locus of the pencil $\mathcal{S}$ are all du Val for every $\lambda\ne\infty$,
then we apply Corollary~\ref{corollary:irreducible-fibers} to deduce \eqref{equation:main-1} in Main Theorem.
Similarly, if every {fixed} singular point is a {good} double point of every non-du Val surface~$S_\lambda$ in the pencil $\mathcal{S}$,
then we can apply Corollary~\ref{corollary:normal-crossing-simple} together
with Lemma~\ref{lemma:main} to compute the right hand side of \eqref{equation:main-1} in Main Theorem.

If the pencil $\mathcal{S}$ contains a non-du Val quartic surface $S_\lambda$ that has \emph{bad} singularity at some {fixed} singular point~$P\in\Sigma$,
then we can compute the number of irreducible components of the fiber $\mathsf{f}^{-1}(\lambda)$ using \eqref{equation:equation:number-of-irredubicle-components-refined}.
This gives us
$$
\big[\mathsf{f}^{-1}(\lambda)\big]=\big[S_{\lambda}\big]+\sum_{i=1}^{r}\mathbf{C}_j^\lambda+\sum_{P\in\Sigma}\mathbf{D}_P^\lambda.
$$
Here, the term $[S_{\lambda}]$ is easy to compute.
Likewise, the second term in this formula can be computed using Lemma~\ref{lemma:main}.
Therefore, for every {fixed} singular point $P\in\Sigma$ that is neither du Val nor a {good} double point of the surface $S_{\lambda}$,
we must compute its {defect}~$\mathbf{D}_P^\lambda$.

To compute the defect $\mathbf{D}_P^\lambda$,
we describe the birational morphism $\alpha\colon U\to\mathbb{P}^3$ in \eqref{equation:main-diagram}.
This can be done locally in a neighborhood of the point $P$.
Then we describe the divisor
$$
\widehat{D}_\lambda=\widehat{S}_\lambda+\sum_{i=1}^{k}\mathbf{a}_i^\lambda\widehat{E}_i
$$
in \eqref{equation:log-pull-back}.
In many cases, we can use Lemma~\ref{lemma:log-pull-back} to show that some (or all) of the numbers $\mathbf{a}_1^\lambda,\ldots,\mathbf{a}_k^\lambda$ vanish.
But it is not hard to compute them in general.

Then we describe the base curves of the pencil $\widehat{\mathcal{S}}$,
and compute the intersection multiplicities $\mathbf{m}_1,\ldots,\mathbf{m}_s$ in \eqref{equation:multiplicities},
and the multiplicities $\mathbf{M}_1^\lambda,\ldots,\mathbf{M}_s^\lambda$ in \eqref{equation:M}.
For the proper transforms of the base curves of the pencil~$\mathcal{S}$,
these computations should have been already done at the previous steps.
For the remaining base curves of the pencil $\widehat{\mathcal{S}}$,
we can compute these numbers locally near every point in $\Sigma$.
For each such point $P\in\Sigma$, we can compute its {defect} $\mathbf{D}_P^\lambda$ arguing as in Subsection~\ref{subsection:scheme-step-8}.
If the surface $S_\lambda$ has du Val singularity or non-isolated ordinary double singularity at $P$,
we can use Lemma~\ref{lemma:normal-crossing} to deduce that its {defect} $\mathbf{D}_P^\lambda$ vanishes.
This allows us to skip many local computations.

Finally, we use \eqref{equation:equation:number-of-irredubicle-components-refined}
to compute $[\mathsf{f}^{-1}(\lambda)]$ for every $\lambda\ne\infty$.
This gives \eqref{equation:main-1} in Main Theorem and completes the proof of Main Theorem in the case when $-K_X$ is very~ample.

\begin{example}
\label{example:r-3-n-6}
Suppose that the threefold $X$ is contained in the family \textnumero $3.6$ in \cite{IP}.
Then $X$ can be obtained by blowing up $\mathbb{P}^3$ at a disjoint union of a line and a smooth elliptic curve of degree $4$, so that $h^{1,2}(X)=1$.
A~toric Landau--Ginzburg model of the threefold $X$ is given by the Minkowski polynomial \textnumero $1899$, which is
$$
x+z+{\frac {x}{z}}+{\frac {1}{xy}}+{\frac {z}{x}}+\frac{1}{y}+\frac{1}{z}+\frac{2}{y}+\frac{3}{x}+{\frac {yz}{x}}+{\frac {y}{z}}+\frac{3y}{x}+{\frac {{y}^{2}}{x}}.
$$
Then the corresponding pencil $\mathcal{S}$ is given by
$$
{x}^{2}yz+xz{t}^{2}+xy{z}^{2}+{x}^{2}yt+z{t}^{3}+y{z}^{2}t+xy{t}^{2}+2x{y}^{2}z+3yz{t}^{2}+{y}^{2}{z}^{2}+x{y}^{2}t+3{y}^{2}zt+{y}^{3}z=\lambda xyzt.
$$
Suppose that $\lambda\ne\infty$. Let $\mathcal{C}$ be the conic $x=yz+(y+t)^2=0$.
Then
\begin{equation}
\label{equation:r-3-n-6}
\begin{split}
H_{\{x\}}\cdot S_\lambda&=L_{\{x\},\{z\}}+L_{\{x\},\{y,t\}}+\mathcal{C},\\
H_{\{y\}}\cdot S_\lambda&=L_{\{y\},\{z\}}+2L_{\{y\},\{t\}}+L_{\{y\},\{x,t\}},\\
H_{\{z\}}\cdot S_\lambda&=L_{\{x\},\{z\}}+L_{\{y\},\{z\}}+L_{\{z\},\{t\}}+L_{\{z\},\{x,y,t\}},\\
H_{\{t\}}\cdot S_\lambda&=L_{\{y\},\{t\}}+L_{\{z\},\{t\}}+L_{\{t\},\{x,y\}}+L_{\{t\},\{x,y,z\}}.
\end{split}
\end{equation}
Let $\mathbf{S}$ be an irreducible cubic surface that given by $zt^2+2yzt+xyt+yz^2+xyz+y^2z=0$.
Then $S_{-3}=H_{\{x+y+t\}}+\mathbf{S}$.
If $\lambda\ne-3$, then $S_\lambda$ is irreducible, and its singularities contained in the base locus of the pencil $\mathcal{S}$
can be described as follows:
\begin{itemize}\setlength{\itemindent}{3cm}
\item[$P_{\{y\},\{z\},\{t\}}$:] type $\mathbb{A}_2$ with quadratic term $y(z+t)$;

\item[$P_{\{x\},\{y\},\{t\}}$:] type $\mathbb{A}_3$ with quadratic term $y(x+y+t)$;

\item[$P_{\{x\},\{z\},\{y,t\}}$:] type $\mathbb{A}_2$ with quadratic term $x(x+y+t-3z-\lambda z)$;

\item[$P_{\{y\},\{z\},\{x,t\}}$:] type $\mathbb{A}_1$ with quadratic term $4zy-(x+t)(y+z)-y^{2}+\lambda yz$;

\item[$P_{\{y\},\{t\},\{x,z\}}$:] type $\mathbb{A}_1$ with quadratic term $2yt-t^2-(x+z)y-y^{2}+\lambda ty$;

\item[$P_{\{z\},\{t\},\{x,y\}}$:] type $\mathbb{A}_2$ with quadratic term
$$
t(x+y+t-3z-\lambda z)
$$
for $\lambda\neq -4$, and type $\mathbb{A}_3$ for $\lambda=-4$.
\end{itemize}
These are the {fixed} singular points of the pencil $\mathcal{S}$.
All of them are {good} double points of the surface $S_{-3}$.
Now using Corollaries~\ref{corollary:irreducible-fibers} and \ref{corollary:normal-crossing-simple},
we obtain \eqref{equation:main-1} in Main Theorem.
To verify \eqref{equation:main-2-simple}, we observe that $\mathrm{rk}\,\mathrm{Pic}(\widetilde{S}_{\Bbbk})=\mathrm{rk}\,\mathrm{Pic}(S_{\Bbbk})+11$.
Now we must compute the rank of the intersection matrix $M$ in Lemma~\ref{lemma:cokernel}.
We may assume that $\lambda\not\in\{-4,-3\}$.
Using \eqref{equation:r-3-n-6}, we see that  $M$ has the same rank as the intersection matrix of the curves $L_{\{x\},\{z\}}$, $L_{\{x\},\{y,t\}}$, $L_{\{y\},\{z\}}$, $L_{\{y\},\{x,t\}}$, $L_{\{z\},\{x,y,t\}}$, $L_{\{t\},\{x,y,z\}}$ and $H_\lambda$,
which is given by
\begin{center}\renewcommand\arraystretch{1.42}
\begin{tabular}{|c||c|c|c|c|c|c|c|}
\hline
 $\bullet$  & $L_{\{x\},\{z\}}$ & $L_{\{x\},\{y,t\}}$ & $L_{\{y\},\{z\}}$ & $L_{\{y\},\{x,t\}}$ & $L_{\{z\},\{x,y,t\}}$ & $L_{\{t\},\{x,y,z\}}$ &  $H_{\lambda}$ \\
\hline\hline
$L_{\{x\},\{z\}}$ & $-\frac{4}{3}$ & $\frac{2}{3}$ & $1$ & $0$ & $\frac{1}{3}$ & $0$ & $1$ \\
\hline
$L_{\{x\},\{y,t\}}$ & $\frac{2}{3}$ & $-\frac{7}{12}$ & $0$ & $\frac{1}{2}$ & $\frac{1}{3}$ & $0$ & $1$ \\
\hline
$L_{\{y\},\{z\}}$ & $1$ & $0$ & $-\frac{5}{6}$ & $\frac{1}{2}$ & $\frac{1}{2}$ & $0$ & $1$ \\
\hline
$L_{\{y\},\{x,t\}}$ & $0$ & $\frac{1}{2}$ & $\frac{1}{2}$ & $-\frac{1}{2}$ & $\frac{1}{2}$ & $0$ & $1$ \\
\hline
$L_{\{z\},\{x,y,t\}}$ & $\frac{1}{3}$ & $\frac{1}{3}$ & $\frac{1}{2}$ & $\frac{1}{2}$ & $-\frac{1}{6}$ & $\frac{1}{3}$ & $1$ \\
\hline
$L_{\{t\},\{x,y,z\}}$ & $0$ & $0$ & $0$ & $0$ & $\frac{1}{3}$ & $-\frac{5}{6}$ & $1$ \\
\hline
 $H_{\lambda}$  & $1$ & $1$ & $1$ & $1$ & $1$ & $1$ & $4$ \\
\hline
\end{tabular}
\end{center}
It has rank $6$, so that \eqref{equation:main-2-simple} holds,
which gives \eqref{equation:main-2} in Main Theorem by Lemma~\ref{lemma:cokernel}.
\end{example}

In the remaining part of this paper, we will always use notations of this section except for $5$ families of smooth Fano threefolds
whose anticanonical divisors are not very ample.
These are the families \textnumero $2.1$, $2.2$, $2.3$, $9.1$, and $10.1$ in \cite{IP}.
We will deal with them in Subsections~\ref{section:r-2-n-1}, \ref{section:r-2-n-2}, \ref{section:r-2-n-3}, \ref{section:r-9-n-1}, and \ref{section:r-10-n-1},
respectively.
The proof of Main Theorem in these cases is  similar to the case when $-K_X$ is very ample.
For instance, if $X=\mathbb{P}^1\times\mathsf{S}_1$, where~$\mathsf{S}_1$ is a smooth del Pezzo surface of degree $1$,
the commutative diagram \eqref{equation:quartic} also exists.
But now by~\cite[Proposition 29]{P16} the pencil $\mathcal{S}$ is given by
$$
x^3y=(\lambda yz-y^2-z^2)(xt-xz-t^2),
$$
where $\lambda\in\mathbb{C}\cup\{\infty\}$.
In this case, which is the family \textnumero $9.1$,
we still can apply all steps described above to prove Main Theorem.

\section{Fano threefolds of Picard rank $2$}
\label{section:rank-2}

\subsection{Family \textnumero $2.1$}
\label{section:r-2-n-1}

In this case, the threefold $X$ can be obtained as a blow up of a smooth sextic hypersurface in $\mathbb{P}(1,1,1,2,3)$
along a smooth elliptic curve.
This implies that $h^{1,2}(X)=22$.
Note that $-K_X$ is not very ample.
Because of this, there exists no Laurent polynomial with reflexive Newton polytope that gives
the toric Landau--Ginzburg model of this deformation family.
However, there are Laurent polynomials with non-reflexive Newton polytopes that give the commutative diagram \eqref{equation:CCGK-compactification}.
One of them is
$$
\frac{(r+s+1)^6(t+1)^6}{rs^2}+\frac{1}{t},
$$
which we also denote by $\mathsf{p}$.

Let $\gamma\colon\mathbb{C}^3\dasharrow\mathbb{C}^\ast\times\mathbb{C}^\ast\times\mathbb{C}^\ast$
be a birational transformation
that is given by the change of coordinates
$$
\left\{\aligned
&r=\frac{1}{b}-\frac{1}{b^2c}-1,\\
&s=\frac{1}{b^2c},\\
&t=-\frac{1}{y}-1.\\
\endaligned
\right.
$$
Arguing as in Subsection~\ref{subsection:scheme-step-7}, we can expand \eqref{equation:CCGK-compactification} to the commutative diagram
\begin{equation}
\label{equation:diagram-2-1}
\xymatrix{
&&\mathbb{P}^2\times\mathbb{P}^1\ar@{-->}[dd]^{\phi}&\mathbb{C}^3\ar@{_{(}->}[l]\ar@{->}[dd]^{\mathsf{q}}\ar@{-->}[rr]^{\gamma}&&\mathbb{C}^\ast\times\mathbb{C}^\ast\times\mathbb{C}^\ast\ar@{^{(}->}[r]\ar@{->}[dd]_{\mathsf{p}}&Y\ar@{->}[dd]^{\mathsf{w}}\ar@{^{(}->}[r]&Z\ar@{->}[dd]^{\mathsf{f}}\\
V\ar@{->}[drr]_{\mathsf{g}}\ar@{->}[urr]^{\pi}&&&&&&&\\
&&\mathbb{P}^1&\mathbb{C}^1\ar@{_{(}->}[l]\ar@{=}[rr]&&\mathbb{C}^1\ar@{=}[r]&\mathbb{C}^1\ar@{^{(}->}[r]&\mathbb{P}^1}
\end{equation}
where $\mathsf{q}$ is a surjective morphism, $\pi$ is a birational morphism, the threefold $V$ is smooth,
the map $\mathsf{g}$ is a surjective morphism such that
$-K_{V}\sim\mathsf{g}^{-1}(\infty)$,
and $\phi$ is a rational map that is given by the pencil
\begin{equation}
\label{equation:2-1-pencil}
x(x+y)c^3=y\Big((x+y)\lambda+y\Big)\Big(abc-b^2c-a^3\Big),
\end{equation}
where $([x:y],[a:b:c])$ is a point in $\mathbb{P}^1\times\mathbb{P}^2$, and $\lambda\in\mathbb{C}\cup\{\infty\}$.

The commutative diagram \eqref{equation:diagram-2-1} is similar to the commutative diagram \eqref{equation:diagram} presented in Subsection~\ref{subsection:scheme-step-4}.
Like in \eqref{equation:diagram}, there exists a composition of flops $\chi\colon V\dasharrow Z$ that makes the following diagram commuting:
$$
\xymatrix{
V\ar@{->}[d]_{\mathsf{g}}\ar@{-->}[rr]^{\chi}&&Z\ar@{->}[d]^{\mathsf{f}}\\
\mathbb{P}^1&&\mathbb{P}^1\ar@{=}[ll]}
$$
So, to prove Main Theorem in this case, we will follow the scheme described in~Section~\ref{section:scheme}.
Moreover, we will use the same assumptions and notation as in the case when $-K_X$ is very ample.
The only difference is that $\mathbb{P}^3$ is now replaced by $\mathbb{P}^1\times\mathbb{P}^2$.
For instance, we denote by $\mathcal{S}$ the pencil \eqref{equation:2-1-pencil},
and we denote by $S_\lambda$ the surface in $\mathcal{S}$ given by \eqref{equation:2-1-pencil},
where $\lambda\in\mathbb{C}\cup\{\infty\}$.
Likewise, we extend handy notation in Subsection~\ref{subsection:notations}
to bilinear sections of $\mathbb{P}^1\times\mathbb{P}^2$.
Note that the curve $H_\lambda$ is not defined in this case.

Let $\mathsf{S}$ be the surface in $\mathbb{P}^1\times\mathbb{P}^2$ given by $abc-b^2c-a^3=0$.
Then $\mathsf{S}$ is irreducible and
$$
S_\infty=H_{\{y\}}+H_{\{x,y\}}+\mathsf{S}.
$$
Let $\mathbf{S}$ be the surface in $\mathbb{P}^1\times\mathbb{P}^2$ that is given by the equation $xc^3+yc^3-yabc+yb^2c+ya^3=0$.
Then $\mathbf{S}$ is irreducible and $S_{-1}=H_{\{x\}}+\mathbf{S}$.
These are all reducible surfaces in $\mathcal{S}$.

To describe the base locus of the pencil $\mathcal{S}$, we observe that
\begin{equation}
\label{equation:2-1}
\begin{split}
H_{\{x,y\}}\cdot S_{-1}&=\mathcal{C}_1,\\
H_{\{y\}}\cdot S_{-1}&=3L_{\{y\},\{c\}},\\
\mathsf{S}\cdot S_{-1}&=\mathcal{C}_1+\mathcal{C}_2+9L_{\{a\},\{c\}},
\end{split}
\end{equation}
where $\mathcal{C}_1$ is the curve in $\mathbb{P}^1\times\mathbb{P}^2$ that is given by $x+y=abc-b^2c-a^3=0$,
and $\mathcal{C}_2$ is the curve in $\mathbb{P}^1\times\mathbb{P}^2$ that is given by $x=abc-b^2c-a^3=0$.
Thus, we have
$$
S_{-1}\cdot S_{\infty}=2\mathcal{C}_1+\mathcal{C}_2+3L_{\{y\},\{c\}}+9L_{\{a\},\{c\}},
$$
so that the base locus of the pencil $\mathcal{S}$ consists of
the curves $\mathcal{C}_1$, $\mathcal{C}_2$, $L_{\{y\},\{c\}}$, and $L_{\{a\},\{c\}}$.

To match the notation used in Subsection~\ref{subsection:scheme-step-6},
we let $C_1=\mathcal{C}_1$, $C_2=\mathcal{C}_2$, $C_3=L_{\{y\},\{c\}}$, and $C_4=L_{\{a\},\{c\}}$.
Then $\mathbf{m}_{1}=2$, $\mathbf{m}_{2}=2$, $\mathbf{m}_{3}=3$, and $\mathbf{m}_{4}=9$.

Observe that $S_{0}$ is singular along the curve $L_{\{y\},\{c\}}$.
Moreover, if $\lambda\not\in\{0,-1,\infty\}$, then the surface $S_\lambda$ has isolated singularities.
In this case the singular points of the surface $S_\lambda$
contained in the base locus of the pencil $\mathcal{S}$ are du Val and can be described as follows:
\begin{itemize}\setlength{\itemindent}{6cm}
\item[$P_{\{y\},\{a\},\{c\}}$:] type $\mathbb{A}_8$;
\item[{$[\lambda+1:-\lambda]\times [0:1:0]$:}] type $\mathbb{A}_8$.
\end{itemize}

Applying Corollary~\ref{corollary:irreducible-fibers}, we obtain the following.

\begin{corollary}
\label{corollary:r-2-n-1-irreducible}
The fiber $\mathsf{f}^{-1}(\lambda)$ is irreducible for every $\lambda\not\in\{0,-1,\infty\}$.
\end{corollary}

Observe that the point $P_{\{y\},\{a\},\{c\}}$ is the only {fixed} singular point of the pencil $\mathcal{S}$.

\begin{remark}
\label{remark:r-2-n-1-singular-curve}
The base curve $C_1$ is singular at the point $P_{\{x,y\},\{a\},\{b\}}$.
Similarly, the base curve $C_2$ is singular at the point $P_{\{x\},\{a\},\{b\}}$.
Thus, in the notation of Subsection~\ref{subsection:scheme-step-7}, both curves $\widehat{C}_1$ and $\widehat{C}_2$ are singular.
This implies that the threefold $V$ in \eqref{equation:diagram-2-1} is singular: it has isolated ordinary double points.
But this is not important for the proof of Main Theorem in this case,
because these singular points are contained in the fiber $\mathbf{g}^{-1}(\infty)$.
Note that we can resolve them by composing the birational morphism $\pi$ in \eqref{equation:diagram-2-1} with small resolution of these double points.
However, the resulting smooth threefold would not be projective (cf. the proof of \cite[Proposition~29]{P16}).
\end{remark}

First, let us prove \eqref{equation:main-2} in Main Theorem.
By Lemma~\ref{lemma:cokernel}, it follows from

\begin{lemma}
\label{lemma:r-2-n-1-intersection}
The equality \eqref{equation:main-2-simple} holds.
\end{lemma}

\begin{proof}
Suppose that $\lambda\not\in\{0,-1,\infty\}$.
Let $H_\lambda$ be the intersection of the surface $S_\lambda$ with a general surface in $\mathbb{P}^1\times\mathbb{P}^2$ of bi-degree $(0,1)$.
Then it follows from \eqref{equation:2-1} that
$$
C_1+C_2+9C_4\sim H_\lambda
$$
and $C_1\sim C_2\sim 3C_3$ on the surface $S_\lambda$.
Thus, the intersection matrix of the curves $C_1$, $C_2$, $C_3$, $C_4$ on the surface $S_\lambda$
has the same rank as the intersection matrix
$$
\left(
    \begin{array}{cc}
      C_1^2 & H_{\lambda}\cdot C_1 \\
      H_{\lambda}\cdot C_1 & H_{\lambda}^2 \\
    \end{array}
  \right)=
\left(
    \begin{array}{cc}
      0 & 1 \\
      1 & 2 \\
    \end{array}
  \right).
$$
One the other hand, we have $\mathrm{rk}\,\mathrm{Pic}(\widetilde{S}_{\Bbbk})=\mathrm{rk}\,\mathrm{Pic}(S_{\Bbbk})+16$.
This shows that \eqref{equation:main-2-simple} holds.
\end{proof}

In the remaining part of this subsection, we will show that \eqref{equation:main-1} in Main Theorem also holds in this case.
To do this, we have to compute $[\mathsf{f}^{-1}(-1)]$ and $[\mathsf{f}^{-1}(0)]$.
We start with

\begin{lemma}
\label{lemma:r-2-n-1-irreducible-1}
One has $[\mathsf{f}^{-1}(-1)]=2$.
\end{lemma}

\begin{proof}
As we already mentioned, the point $P_{\{y\},\{a\},\{c\}}$ is the only {fixed} singular point of the pencil $\mathcal{S}$.
The surface $S_{-1}$ has a du Val singularity of type $\mathbb{A}_8$ at it.
Since
$$
\mathbf{M}_{1}^{-1}=\mathbf{M}_{2}^{-1}=\mathbf{M}_{3}^{-1}=\mathbf{M}_{4}^{-1}=1,
$$
we use Corollary~\ref{corollary:normal-crossing-simple} to deduce that $[\mathsf{f}^{-1}(-1)]=[S_{-1}]=2$.
\end{proof}

To compute $[\mathsf{f}^{-1}(0)]$, observe that $\mathbf{M}_{1}^0=1$, $\mathbf{M}_{2}^0=1$, $\mathbf{M}_{3}^0=2$, and $\mathbf{M}_{4}^0=1$.
Thus, it follows from \eqref{equation:equation:number-of-irredubicle-components-refined} and Lemma~\ref{lemma:main} that
\begin{equation}
\label{equation:r-2-n-1-D}
[\mathsf{f}^{-1}(0)]=3+\mathbf{D}_{P_{\{y\},\{a\},\{c\}}}^0,
\end{equation}
where $\mathbf{D}_{P_{\{y\},\{a\},\{c\}}}^0$ is the {defect} of the singular point $P_{\{y\},\{a\},\{c\}}$ defined in Subsection~\ref{subsection:scheme-step-6}.
The defect $\mathbf{D}_{P_{\{y\},\{a\},\{c\}}}^0$ can be computed locally near the point $P_{\{y\},\{a\},\{c\}}$.
The recipe how to compute it is given in Subsection~\ref{subsection:scheme-step-8}.
Let us use it.

Suppose that $\lambda\ne\infty$.
Consider a local chart $x=b=1$. Then the surface $S_\lambda$ in this chart is given by
$$
-\lambda yc+c(c^2+\lambda ya-\lambda y^2-y^2)+y(c^3-\lambda a^3+\lambda yac+yac)-(\lambda+1)(y^2a^3)=0.
$$
Let $\alpha_1\colon U_1\to\mathbb{P}^1\times\mathbb{P}^2$ be the blow up of the point $P_{\{y\},\{a\},\{c\}}$.
A chart of the blow up $\alpha_1$ is given by the coordinate change $a_1=a$, $y_1=\frac{y}{a}$, and $c_1=\frac{c}{a}$.
In this chart, the surface~$D^1_\lambda$ is given by the equation
$$
-\lambda y_1c_1+\lambda y_1a_1(c_1-a_1)+a_1c_1(c_1^2-\lambda y_1^2-y_1^2)+y_1^2 a_1^2 (\lambda+1)(c_1-a_1)+a_1^2c_1^3y_1=0.
$$
where $a_1=0$ defines the exceptional surface $\mathbf{E}_1$.
Then $\mathbf{E}_1$ contains two base curves of the pencil $\mathcal{S}^1$.
One of them given by  $a_1=y_1=0$, and another one is given by $a_1=c_1=0$.
Denote the former curve by $C_5^1$, and denote the latter curve by $C_6^1$.

If $\lambda\ne 0$, then the point $(a_1,y_1,c_1)=(0,0,0)$ is the only singular point of the surface~$D^1_\lambda$ that is contained in $\mathbf{E}_1$.
Let $\alpha_2\colon U_2\to U_1$ be the blow up of this point.
A chart of the blow up $\alpha_2$ is given by the coordinate change $a_2=a_1$, $y_2=\frac{y_1}{a_1}$, $c_2=\frac{c_1}{a_1}$.
Let $\hat{y}_2=y_2$, $\hat{a}_2=a_2$, and $\hat{c}_2=a_2+c_2$.
Then $D_\lambda^2$ is given by
\begin{multline*}
-\lambda\hat{y}_2\hat{c}_2+\lambda\hat{y}_2\hat{a}_2\big(\hat{c}_2-\hat{a}_2\big)+\\
+\hat{a}_2^2\Big(\hat{c}_2^3-\hat{a}_2^3+3\hat{a}_2^2\hat{c}_2-3\hat{c}_2^2\hat{a}_2-\lambda\hat{y}_2^2\hat{c}_2-\hat{y}_2^2\hat{c}_2\Big)+\\
+(\lambda+1)\hat{y}_2^2\hat{a}_2^3\big(\hat{c}_2-\hat{a}_2\big)+\hat{a}_2^4 \hat{y}_2\big(\hat{c}_2-\hat{a}_2\big)^3=0,
\end{multline*}
and $\mathbf{E}_2$ is given by $\hat{a}_2=0$.
Then $\mathbf{E}_2$ contains two base curves of the pencil $\mathcal{S}^2$.
One of them given by  $\hat{a}_2=\hat{y}_2=0$, and another one is given by $\hat{a}_2=\hat{c}_2=0$.
Denote the former curve by $C_7^2$, and denote the latter curve by $C_8^2$.

If $\lambda\ne 0$, then $(\hat{a}_2,\hat{y}_2,\hat{c}_2)=(0,0,0)$ is the only singular point of the surface~$D^2_\lambda$ that is contained in $\mathbf{E}_2$.
Let $\alpha_3\colon U_3\to U_2$ be the blow up of this point.
A chart of this blow up is given by the coordinate change
$\hat{a}_3=\hat{a}_2$, $\hat{y}_3=\frac{\hat{y}_2}{\hat{a}_2}$, $\hat{c}_3=\frac{\hat{c}_2}{\hat{a}_2}$
Let $\bar{y}_3=\hat{y}_3$, $\bar{a}_3=\hat{a}_3$, $\bar{c}_3=\hat{a}_3+\hat{c}_3$.
Denote by $\mathbf{E}_2$ the exceptional surface of the blow up $\alpha_2$.
Then  $D_\lambda^3$ is given~by
\begin{multline*}
-\lambda \bar{y}_3\bar{c}_3+\lambda\bar{a}_3\bar{y}_3\bar{c}_3-\bar{a}_3^3-\lambda \bar{y}_3\bar{a}_3^2+3\bar{a}_3^3\big(\bar{c}_3-\bar{a}_3\big)-3\bar{a}_3^3\big(\bar{c}_3-\bar{a}_3\big)^2+\\
+\bar{a}_3^3\Big(\bar{c}_3^3-\bar{a}_3^3+3\bar{a}_3^2\bar{c}_3-3\bar{a}_3\bar{c}_3^2-\lambda \bar{y}_3^2\bar{c}_3-\bar{y}_3^2\bar{c}_3\Big)+\bar{y}_3\bar{a}_3^4\Big(\lambda \bar{y}_3\bar{c}_3+\bar{y}_3\bar{c}_3-\bar{a}_3^2-\lambda \bar{y}_3\bar{a}_3-\bar{y}_3\bar{a}_3\Big)+\\
+3\bar{y}_3\bar{a}_3^6(\bar{c}_3-\bar{a}_3)-3\bar{y}_3\bar{a}_3^6(\bar{c}_3-\bar{a}_3)^2+\bar{y}_3\bar{a}_3^6\big(\bar{c}_3-\bar{a}_3\big)^3=0,
\end{multline*}
and $\mathbf{E}_3$ is given by $\bar{a}_3=0$.
Then $\mathbf{E}_3$ contains two base curves of the pencil $\mathcal{S}^3$.
One of them given by  $\bar{a}_3=\bar{y}_3=0$, and another one is given by $\bar{a}_3=\bar{c}_3=0$.
Denote the former curve by $C_9^3$, and denote the latter curve by $C_{10}^3$.

There exists a commutative diagram
$$
\xymatrix{
&U_2\ar@{->}[ld]_{\alpha_2}&&U_3\ar@{->}[ll]_{\alpha_3}\\
U_1\ar@{->}[drr]_{\alpha_1}&&&&U\ar@{->}[dll]^\alpha\ar@{->}[lu]_{\alpha_4}\\
&&\mathbb{P}^3&& }
$$
where $\alpha_4$ be the blow up of the point $(\bar{a}_3,\bar{y}_3,\bar{c}_3)=(0,0,0)$.
Note that $\widehat{E}_4$ contains two base curves of the pencil $\widehat{\mathcal{S}}$.
Denote them by $\widehat{C}_{11}$ and $\widehat{C}_{12}$.
Then $\widehat{C}_1$, $\widehat{C}_2$, $\widehat{C}_3$, $\widehat{C}_4$, $\widehat{C}_5$, $\widehat{C}_6$, $\widehat{C}_7$, $\widehat{C}_8$,
$\widehat{C}_9$, $\widehat{C}_{10}$, $\widehat{C}_{11}$, and $\widehat{C}_{12}$ are all base curves of the pencil $\widehat{\mathcal{S}}$,
because
$$
\widehat{S}_{\lambda_1}\cdot\widehat{S}_{\lambda_2}=2\widehat{C}_1+\widehat{C}_2+3\widehat{C}_3+9\widehat{C}_4+\widehat{C}_5+7\widehat{C}_6+2\widehat{C}_7+5\widehat{C}_8+3\widehat{C}_9+3\widehat{C}_{10}+\widehat{C}_{11}+\widehat{C}_{12}
$$
for two general $\lambda_1$ and $\lambda_2$ in $\mathbb{C}$.
This also shows that $\mathbf{m}_{5}=1$, $\mathbf{m}_{6}=7$, $\mathbf{m}_{7}=2$, $\mathbf{m}_{8}=5$, $\mathbf{m}_{9}=3$, $\mathbf{m}_{10}=3$, $\mathbf{m}_{11}=1$, and $\mathbf{m}_{12}=1$.

Let us compute the term $\mathbf{A}_{P_{\{y\},\{a\},\{c\}}}$ in \eqref{equation:D-A-B}.
We have
$\widehat{D}_{0}=\widehat{S}_{0}+\widehat{E}_1+2\widehat{E}_2+3\widehat{E}_3+\widehat{E}_4$.
This gives $\mathbf{A}_{P_{\{y\},\{a\},\{c\}}}^0=4$.
Note also that $\mathbf{M}_{5}^0=1$, $\mathbf{M}_{6}^0=2$, $\mathbf{M}_{7}^0=2$, $\mathbf{M}_{8}^0=3$, $\mathbf{M}_{9}^0=3$, $\mathbf{M}_{10}^0=3$, $\mathbf{M}_{11}^0=1$, $\mathbf{M}_{12}^0=1$.
Thus, it follows from \eqref{equation:D-A-B} that
$$
\mathbf{D}_{P}^0=4+\sum_{i=1}^{12}\mathbf{C}_i^0,
$$
where $\mathbf{C}_i^0$ is the number defined in \eqref{equation:C}. By Lemma~\ref{lemma:main-2}, we have
$$
\mathbf{C}_i^0=\left\{\aligned
&0\ \text{if}\ \mathbf{M}_i^0=1,\\
&\textbf{m}_i-1\ \text{if}\ \mathbf{M}_i^0\geqslant 2.\\
\endaligned
\right.
$$
Therefore, we have $\mathbf{D}_{P}^0=19$. Now using \eqref{equation:r-2-n-1-D}, we deduce that $[\mathsf{f}^{-1}(0)]=22$.
Keeping in mind that $h^{1,2}(X)=22$ and $[\mathsf{f}^{-1}(-1)]=2$,
we see that \eqref{equation:main-1} in Main Theorem holds.

\subsection{Family \textnumero $2.2$}
\label{section:r-2-n-2}

In this case, the threefold $X$ is a double cover of $\mathbb{P}^1\times\mathbb{P}^2$ ramified in a surface of bidegree $(2,4)$.
This implies that $h^{1,2}(X)=20$.
As in the previous case, the divisor $-K_X$ is not very ample,
and there are no toric Landau--Ginzburg models with reflexive Newton polytope in this case.
However, we can find a Laurent polynomial $\mathsf{p}$ with non-reflexive Newton polytope that gives the commutative diagram \eqref{equation:CCGK-compactification}.
For instance, we can choose $\mathsf{p}$ to be the Laurent polynomial
$$
\frac{(a+b+c+1)^2}{a}+\frac{(a+b+c+1)^4}{bc}.
$$

Let $\gamma\colon\mathbb{C}^3\dasharrow\mathbb{C}^\ast\times\mathbb{C}^\ast\times\mathbb{C}^\ast$
be a birational transformation
that is given by the change of coordinates
$$
\left\{\aligned
&a=xy,\\
&b=yz,\\
&c=z-xy-yz-1.\\
\endaligned
\right.
$$
By \cite[Proposition 16]{P16}, we can expand \eqref{equation:CCGK-compactification} to the commutative diagram
\begin{equation}
\label{equation:diagram-2-2}
\xymatrix{
&&\mathbb{P}^3\ar@{-->}[dd]^{\phi}&\mathbb{C}^3\ar@{_{(}->}[l]\ar@{->}[dd]^{\mathsf{q}}\ar@{-->}[rr]^{\gamma}&&\mathbb{C}^\ast\times\mathbb{C}^\ast\times\mathbb{C}^\ast\ar@{^{(}->}[r]\ar@{->}[dd]_{\mathsf{p}}&Y\ar@{->}[dd]^{\mathsf{w}}\ar@{^{(}->}[r]&Z\ar@{->}[dd]^{\mathsf{f}}\\
V\ar@{->}[drr]_{\mathsf{g}}\ar@{->}[urr]^{\pi}&&&&&&&\\
&&\mathbb{P}^1&\mathbb{C}^1\ar@{_{(}->}[l]\ar@{=}[rr]&&\mathbb{C}^1\ar@{=}[r]&\mathbb{C}^1\ar@{^{(}->}[r]&\mathbb{P}^1}
\end{equation}
where $\mathsf{q}$ is a surjective morphism, $\pi$ is a birational morphism, the threefold $V$ is smooth,
the map $\mathsf{g}$ is a surjective morphism such that $-K_{V}\sim\mathsf{g}^{-1}(\infty)$,
and $\phi$ is a rational map that is given by a pencil of quartic surfaces $\mathcal{S}$ given by
\begin{equation}
\label{equation:2-2-pencil}
xz^3=(zt-xy-yz-t^2)(\lambda xy-z^2),
\end{equation}
where $\lambda\in\mathbb{C}\cup\{\infty\}$.
Note that a general fiber of the morphism $\mathsf{g}$ is a smooth $K3$ surface.
Thus, a general surface in the pencil \eqref{equation:2-1-pencil} has at most du Val singularities.

The diagram \eqref{equation:diagram-2-2} is very similar to the diagram \eqref{equation:diagram} presented in Subsection~\ref{subsection:scheme-step-4}.
The only difference is that the pencil $\mathcal{S}$ is now given by the equation~\eqref{equation:2-1-pencil}.
Because of this, we will follow the scheme described in~Section~\ref{section:scheme},
and we will use the assumptions and the notation introduced in this section.

As in Section~\ref{section:scheme}, we denote by $S_\lambda$ the surface in $\mathcal{S}$ given by \eqref{equation:2-2-pencil}.
Then
$$
S_\infty=H_{\{x\}}+H_{\{y\}}+\mathbf{Q},
$$
where $\mathbf{Q}$  is the quadric in $\mathbb{P}^3$ given by  $zt-xy-yz-t^2=0$.
If $\lambda\ne\infty$, then
\begin{equation}
\label{equation:2-2}
\begin{split}
H_{\{x\}}\cdot S_\lambda&=2L_{\{x\},\{z\}}+\mathcal{C}_1,\\
H_{\{y\}}\cdot S_\lambda&=2L_{\{y\},\{z\}}+\mathcal{C}_2,\\
\mathbf{Q}\cdot S_\lambda&=\mathcal{C}_1+3\mathcal{C}_3,
\end{split}
\end{equation}
where $\mathcal{C}_1$, $\mathcal{C}_2$, and $\mathcal{C}_3$ are conics in $\mathbb{P}^3$ that are given by
the equations $x=zt-yz-t^2=0$, $y=xz-zt+t^2=0$, and $z=xy-t^2=0$, respectively.
It follows from \eqref{equation:2-2} that
the base locus of the pencil $\mathcal{S}$ consists of the curves
$L_{\{x\},\{z\}}$, $L_{\{y\},\{z\}}$,  $\mathcal{C}_1$, $\mathcal{C}_2$, and $\mathcal{C}_3$.

We already know that the surface $S_\infty$ is reducible.
The surface $S_0$ is also reducible. In fact, it is not reduced. Indeed, we have
$S_{0}=2H_{\{z\}}+\mathsf{Q}$,
where $\mathsf{Q}$ is a quadric surface that is given by $xz-yz+zt-t^2-xy=0$.
On the other hand, if $\lambda\ne\infty$ and $\lambda\ne 0$, then the surface $S_\lambda$ has isolated singularities,
which implies that it is irreducible.

If $\lambda\ne\infty$ and $\lambda\ne 0$, then the singular points of the surface $S_\lambda$ contained in the base locus of the pencil $\mathcal{S}$ can be described as follows:
\begin{itemize}\setlength{\itemindent}{3cm}
\item[$P_{\{x\},\{y\},\{z\}}$:] type $\mathbb{A}_1$ with quadratic term $\lambda xy-z^2$;

\item[$P_{\{x\},\{z\},\{t\}}$:] type $\mathbb{A}_9$ (see the proof of Lemma~\ref{lemma:2-2-P-x-z-t});

\item[$P_{\{y\},\{z\},\{t\}}$:] type $\mathbb{E}_6$ (see the proof of Lemma~\ref{lemma:2-2-P-y-z-t}).
\end{itemize}

If $\lambda\ne 0$ and $\lambda\ne\infty$, then the intersection
matrix of the curves $L_{\{x\},\{z\}}$, $L_{\{y\},\{z\}}$,  $\mathcal{C}_1$, $\mathcal{C}_2$, and $\mathcal{C}_3$
on the surface $S_{\lambda}$ has the same rank as the intersection
matrix of the curves $L_{\{x\},\{z\}}$, $L_{\{y\},\{z\}}$, and $H_{\lambda}$, because
$$
H_\lambda\sim 2L_{\{x\},\{z\}}+\mathcal{C}_1\sim 2L_{\{y\},\{z\}}+\mathcal{C}_2\sim_{\mathbb{Q}} \frac{1}{2}\mathcal{C}_1+\frac{3}{2}\mathcal{C}_3.
$$
on the surface $S_\lambda$. This follows from \eqref{equation:2-2}. On the other hand, we have

\begin{lemma}
\label{lemma:r2-n2-intersection}
Suppose that $\lambda\ne 0$ and $\lambda\ne\infty$.
Then the intersection matrix of the curves $L_{\{x\},\{z\}}$, $L_{\{y\},\{z\}}$, and $H_{\lambda}$
on the surface $S_\lambda$ is given by
\begin{center}
\renewcommand\arraystretch{1.42}
\begin{tabular}{|c||c|c|c|}
\hline
 $\bullet$  & $L_{\{x\},\{z\}}$ & $L_{\{y\},\{z\}}$ & $H_{\lambda}$ \\
\hline\hline
$L_{\{x\},\{z\}}$ & $\frac{1}{10}$ & $\frac{1}{2}$ & $1$ \\
\hline
$L_{\{y\},\{z\}}$ & $\frac{1}{2}$ & $-\frac{1}{6}$ & $1$ \\
\hline
 $H_{\lambda}$  & $1$ & $1$ & $4$ \\
\hline
\end{tabular}
\end{center}

\end{lemma}

\begin{proof}
The equalities $H_\lambda^2=4$ and $H_\lambda\cdot L_{\{x\},\{z\}}=H_\lambda\cdot L_{\{y\},\{z\}}=1$ are obvious.
Note that
$$
H_{\{z\}}\cdot S_\lambda=L_{\{x\},\{z\}}+L_{\{y\},\{z\}}+\mathcal{C}_3.
$$
Thus, on the surface $S_\lambda$, we have
\begin{multline*}
H_\lambda\sim L_{\{x\},\{z\}}+L_{\{y\},\{z\}}+\mathcal{C}_3\sim_{\mathbb{Q}} L_{\{x\},\{z\}}+L_{\{y\},\{z\}}+\frac{1}{3}\Big(2H_\lambda-\mathcal{C}_1\Big)\sim_{\mathbb{Q}}\\
\sim_{\mathbb{Q}} L_{\{x\},\{z\}}+L_{\{y\},\{z\}}+\frac{1}{3}\Big(H_\lambda+2L_{\{x\},\{z\}}\Big)\sim_{\mathbb{Q}}\frac{5}{3}L_{\{x\},\{z\}}+L_{\{y\},\{z\}}+\frac{1}{3}H_\lambda,
\end{multline*}
so that $L_{\{x\},\{z\}}\sim_{\mathbb{Q}}\frac{2}{5}H_\lambda-\frac{3}{5}L_{\{y\},\{z\}}$.
Therefore, to complete the proof of the lemma, it is enough to compute the numbers $L_{\{y\},\{z\}}^2$ and $L_{\{y\},\{z\}}\cdot L_{\{x\},\{z\}}$.

Observe that $P_{\{y\},\{z\},\{t\}}$ and $P_{\{x\},\{y\},\{z\}}$ are the only singular points of the surface $S_\lambda$ contained in the line $L_{\{y\},\{z\}}$.
So, using Proposition~\ref{proposition:du-Val-self-intersection}, we get
\mbox{$L_{\{y\},\{z\}}^2=-2+\frac{4}{3}+\frac{1}{2}=-\frac{1}{6}$.}

Since $L_{\{y\},\{z\}}\cap L_{\{x\},\{z\}}=P_{\{x\},\{y\},\{z\}}$, Proposition~\ref{proposition:du-Val-intersection} gives $L_{\{y\},\{z\}}\cdot L_{\{x\},\{z\}}=\frac{1}{2}$.
\end{proof}

The matrix in Lemma~\ref{lemma:r2-n2-intersection} has rank $2$.
Moreover, it follows from the proofs of Lemmas~\ref{lemma:2-2-P-x-z-t} and \ref{lemma:2-2-P-y-z-t} below that
$\mathrm{rk}\,\mathrm{Pic}(\widetilde{S}_{\Bbbk})=\mathrm{rk}\,\mathrm{Pic}(S_{\Bbbk})+16$.
Thus, we see that \eqref{equation:main-2-simple} holds.
Therefore, by Lemma~\ref{lemma:cokernel}, we see that \eqref{equation:main-2} in Main Theorem also holds.

To prove \eqref{equation:main-1} in Main Theorem, we observe that
$[\mathsf{f}^{-1}(\lambda)]=1$ for every $\lambda\not\in\{0,\infty\}$.
This follows from Lemma~\ref{corollary:irreducible-fibers}.
Therefore, to verify \eqref{equation:main-1} in Main Theorem, we have to show that $[\mathsf{f}^{-1}(0)]=21$.
We will do this in the remaining part of this subsection.

To match the notation introduced in Subsection~\ref{subsection:scheme-step-6}, we let
$C_1=\mathcal{C}_1$, $C_2=\mathcal{C}_2$, $C_3=\mathcal{C}_3$, $C_4=L_{\{x\},\{z\}}$, and $C_5=L_{\{y\},\{z\}}$.
Then \eqref{equation:2-2} gives
$$
S_0\cdot S_\infty=2C_1+C_2+3C_3+2C_4+2C_5,
$$
so that $\mathbf{m}_{1}=2$, $\mathbf{m}_{2}=1$, $\mathbf{m}_{3}=3$, and $\mathbf{m}_{4}=\mathbf{m}_{5}=2$.
Moreover, one has $\mathbf{M}_{1}^{0}=\mathbf{M}_{2}^{0}=1$ and $\mathbf{M}_{3}^{0}=\mathbf{M}_{4}^{0}=\mathbf{M}_{5}^{0}=2$.
Then $\mathbf{C}_1^0=\mathbf{C}_1^0=0$, $\mathbf{C}_3^0=2$, and $\mathbf{C}_4^0=\mathbf{C}_5^0=1$ by Lemma~\ref{lemma:main}.
Thus, using $[S_0]=2$ and \eqref{equation:equation:number-of-irredubicle-components-refined}, we see that
\begin{equation}
\label{equation:2-2-S-0}
[\mathsf{f}^{-1}(0)]=6+\mathbf{D}_{P_{\{x\},\{y\},\{z\}}}^0+\mathbf{D}_{P_{\{x\},\{z\},\{t\}}}^0+\mathbf{D}_{P_{\{y\},\{z\},\{t\}}}^0,
\end{equation}
where $\mathbf{D}_{P_{\{x\},\{y\},\{z\}}}^0$,  $\mathbf{D}_{P_{\{x\},\{z\},\{t\}}}^0$, and $\mathbf{D}_{P_{\{y\},\{z\},\{t\}}}^0$
are defects of the singular points $P_{\{x\},\{y\},\{z\}}$, $P_{\{x\},\{z\},\{t\}}$, and $P_{\{y\},\{z\},\{t\}}$, respectively.
For precise definition of defects, see \eqref{equation:D}.

\begin{lemma}
\label{lemma:2-2-P-x-y-z}
One has $\mathbf{D}_{P_{\{x\},\{y\},\{z\}}}^0=0$.
\end{lemma}

\begin{proof}
The required assertion follows from \eqref{equation:D-A-B}, because $P_{\{x\},\{y\},\{z\}}$ is a double point of the surface $S_0$,
and the quadratic term of the surface $S_\lambda$ at this point is $\lambda xy-z^2$.
\end{proof}

\begin{lemma}
\label{lemma:2-2-P-x-z-t}
One has $\mathbf{D}_{P_{\{x\},\{z\},\{t\}}}^0=10$.
\end{lemma}

\begin{proof}
In the chart $y=1$, the surface $S_\lambda$ is given by the equation
$$
\lambda x(x+z)-\big(xz^2+z^3+\lambda xzt-\lambda xt^2\big)+z^2(xz+zt-t^2)=0,
$$
where $P_{\{x\},\{z\},\{t\}}=(0,0,0)$.
We can rewrite this equation as
$$
\lambda\hat{x}\hat{z}+\Big(\lambda\hat{x}\hat{t}^2-\lambda\hat{x}\hat{z}\hat{t}+\lambda\hat{x}^2\hat{t}+2\hat{x}\hat{z}^2-\hat{x}^2\hat{z}-\hat{z}^3\Big)+\big(\hat{x}-\hat{z}\big)^2\Big(\hat{x}\hat{z}+\hat{z}\hat{t}-\hat{x}^2-\hat{x}\hat{t}-\hat{t}^2\Big)=0,
$$
where $\hat{x}=x$, $\hat{z}=x+z$, and $\hat{t}=t$.

Let $\alpha_1\colon U_1\to\mathbb{P}^3$ be the blow up of the point  $P_{\{x\},\{z\},\{t\}}$.
A chart of the blow up $\alpha_1$ is given by the coordinate change
$\hat{x}_1=\frac{\hat{x}}{\hat{t}}$, $\hat{z}_1=\frac{\hat{z}}{\hat{t}}$, $\hat{t}_1=\hat{t}$.
Let $\bar{x}_1=\hat{x}_1$, $\bar{z}_1=\hat{z}_1+\hat{t}_1$, and $\bar{t}_1=\hat{t}_1$.
Then $S^1_\lambda$ is given by the equation
\begin{multline*}
\lambda\bar{x}_1\bar{z}_1+\lambda\bar{x}_1\bar{t}_1(\bar{x}_1-\bar{z}_1+\bar{t}_1)-\bar{z}_1\bar{t}_1(\bar{x}_1-\bar{z}_1+\bar{t}_1)^2-\bar{t}_1^2(\bar{x}_1-\bar{z}_1+\bar{t}_1)^3-\bar{x}_1\bar{t}_1^2(\bar{x}_1-\bar{z}_1+\bar{t}_1)^3=0.
\end{multline*}
for every $\lambda\ne 0$. If $\lambda=0$, this equation defines $D_0^1=S^1_0+\mathbf{E}_1$.
By~\eqref{equation:A} and \eqref{equation:D-A-B}, this contributes~$\textbf{\textcircled{1}}$ to the defect $\mathbf{D}_{P_{\{x\},\{z\},\{t\}}}^0$.
Here and below we circle each contribution for reader's convenience.

Note that $\mathbf{E}_1$ is given by $\bar{t}_1=0$.
This shows that $\mathbf{E}_1$ contains two base curves of the pencil $\mathcal{S}^1$.
One of them is given by $\bar{x}_1=\bar{t}_1=0$, and another one is given by $\bar{z}_1=\bar{t}_1=0$.
We denote the former curve by $C_6^1$, and we denote the latter curve by $C_{7}^1$.
Then $S^1_0+\mathbf{E}_1$ is smooth at general point of the curve $C_6$,
so that this base curve does not give an extra addition to the defect by Lemma~\ref{lemma:main-2} and \eqref{equation:D-A-B}.
On the other hand, we have
$$
\mathrm{mult}_{C_{7}^1}\big(S^1_0+\mathbf{E}_1\big)=\mathbf{M}_7^0=\mathbf{m}_7=\mathrm{mult}_{C_{7}^1}\Big(\big(S^1_0+\mathbf{E}_1\big)\cdot S_\lambda^1\Big)=2,
$$
where $\lambda\ne 0$. By Lemma~\ref{lemma:main-2} and \eqref{equation:D-A-B}, the curve $C_{7}^1$ contributes $\textbf{\textcircled{1}}$ to the defect.

Let $\alpha_2\colon U_2\to U_1$ be the blow up of the point $C_{6}^1\cap C_{7}^1$.
Then $D_0^2=S^2_0+\mathbf{E}_1^2+2\mathbf{E}_2$.
By~\eqref{equation:A} and \eqref{equation:D-A-B}, this contributes $\textbf{\textcircled{1}}$ to $\mathbf{D}_{P_{\{x\},\{z\},\{t\}}}^0$.

A chart of the blow up $\alpha_2$ is given by the coordinate change
$\bar{x}_2=\frac{\bar{x}_1}{\bar{t}_1}$, $\bar{z}_2=\frac{\bar{z}_1}{\bar{t}_1}$, $\bar{t}_2=\bar{t}_1$.
Let $\check{x}_2=\bar{x}_2$, $\check{z}_2=\bar{z}_2+\bar{t}_2$, and $\check{t}_2=\bar{t}_2$.
Then $\mathbf{E}_2$ is given by $\bar{t}_2=0$,
and $D^2_\lambda$ is given by
\begin{multline*}
\lambda \check{x}_2\check{z}_2+\check{t}_2(\lambda \check{x}_2\check{t}_2-\check{z}_2\check{t}_2+\lambda \check{x}_2^2-\lambda \check{x}_2\check{z}_2)-\check{t}_2^2(\check{t}_2+2\check{z}_2)(\check{x}_2-\check{z}_2+\check{t}_2)-\\
-\check{t}_2^2(\check{x}_2^2\check{z}_2-3\check{z}_2\check{t}_2^2+2\check{t}_2^3-2\check{x}_2\check{z}_2^2+\check{z}_2^3-2\check{x}_2\check{z}_2\check{t}_2+2\check{x}_2^2\check{t}_2+5\check{x}_2\check{t}_2^2)-\\
-\check{t}_2^3(\check{x}_2-\check{z}_2+\check{t}_2)(\check{z}_2^2-2\check{x}_2\check{z}_2-2\check{z}_2\check{t}_2+\check{t}_2^2+5\check{x}_2\check{t}_2+\check{x}_2^2)-\\
-3\check{x}_2\check{t}_2^4(\check{x}_2-\check{z}_2+\check{t}_2)^2-\check{x}_2\check{t}_2^4(\check{x}_2-\check{z}_2+\check{t}_2)^3=0.
\end{multline*}

The pencil $\mathcal{S}^2$ has two base curves contained in the surface $\mathbf{E}_2$.
One of them is given by the equation $\bar{x}_2=\bar{t}_2=0$,
and another one is given by the equation $\bar{t}_2+\bar{z}_2=\bar{t}_2=0$.
Denote the former curve by $C_{8}^2$, and denote the latter curve by $C_{9}^2$.
Then
$$
\mathrm{mult}_{C_{8}^2}\big(S^2_0+\mathbf{E}_1^2+2\mathbf{E}_2\big)=\mathbf{M}_8^0=\mathbf{m}_8=\mathrm{mult}_{C_{8}^2}\Big(\big(S^2_0+\mathbf{E}_1^2+2\mathbf{E}_2\big)\cdot S_\lambda^2\Big)=2,
$$
where $\lambda\ne 0$. Thus, this curve contributes $\textbf{\textcircled{1}}$ to the defect by Lemma~\ref{lemma:main-2} and \eqref{equation:D-A-B}.
On the other hand, we have $\mathbf{M}_9^0=3$ and $\mathbf{m}_9=4$,
because $S^2_0+\mathbf{E}_1^2+2\mathbf{E}_2$ is given by
$$
\check{t}_2^4\check{x}_2\Big(\check{x}_2+\check{t}_2-\check{z}_2+1\Big)^3+\check{t}_2^2\Big(\check{t}_2^2+\check{t}_2\check{x}_2-\check{t}_2\check{z}_2+\check{z}_2\Big)\Big(\check{x}_2+\check{t}_2-\check{z}_2+1\Big)^2=0,
$$
and $S^2_\infty$ is given by
$\check{x}_2(\check{t}_2^2+\check{t}_2\check{x}_2-\check{t}_2\check{z}_2+\check{z}_2)=0$.
Thus, by Lemma~\ref{lemma:main-2} and \eqref{equation:D-A-B}, the curve $C_{9}^2$ contributes $\textbf{\textcircled{3}}$ to the defect $\mathbf{D}_{P_{\{x\},\{z\},\{t\}}}^0$.

Let $\alpha_3\colon U_3\to U_2$ be the blow up of the point $C_{8}^2\cap C_{9}^2$.
Then $D_0^3=S^3_0+\mathbf{E}_1^3+2\mathbf{E}_2^3+\mathbf{E}_3$.
By~\eqref{equation:A} and \eqref{equation:D-A-B}, this contributes $\textbf{\textcircled{1}}$ to $\mathbf{D}_{P_{\{x\},\{z\},\{t\}}}^0$.

A chart of the blow up $\alpha_3$ is given by the coordinate change
$\check{x}_3=\frac{\check{x}_2}{\check{t}_2}$, $\check{z}_3=\frac{\check{z}_2}{\bar{t}_2}$, $\check{t}_3=\check{t}_2$.
In this chart, the surface $\mathbf{E}_3$ is given by $\check{t}_3=0$,
and the surface $S^3_\lambda$ is given by
\begin{multline*}
\big(\lambda\check{x}_3-\check{t}_3\big)\big(\check{t}_3+\check{z}_3\big)+\lambda\check{t}_3\check{x}_3^2-\lambda\check{t}_3\check{x}_3\check{z}_3-2\check{t}_3^3-\check{t}_3^2\check{x}_3-\check{t}_3^2\check{z}_3+3\check{t}_3^3\check{z}_3-\check{t}_3^4-5\check{t}_3^3\check{x}_3-\\
-2\check{t}_3^2\check{x}_3\check{z}_3+2\check{t}_3^2\check{z}_3^2+3\check{t}_3^4\check{z}_3-6\check{t}_3^4\check{x}_3-2\check{t}_3^3\check{x}_3^2+2\check{t}_3^3\check{x}_3\check{z}_3+9\check{t}_3^4\check{x}_3\check{z}_3-3\check{t}_3^5\check{x}_3-6\check{t}_3^4\check{x}_3^2-3\check{t}_3^4\check{z}_3^2-\check{t}_3^3\check{x}_3^2\check{z}_3+\\
+2\check{t}_3^3\check{x}_3\check{z}_3^2-\check{t}_3^3\check{z}_3^3+6\check{t}_3^5\check{x}_3\check{z}_3-\check{t}_3^6\check{x}_3-6\check{t}_3^5\check{x}_3^2-\check{t}_3^4\check{x}_3^3+3\check{t}_3^4\check{x}_3^2\check{z}_3-3\check{t}_3^4\check{x}_3\check{z}_3^2+\check{t}_3^4\check{z}_3^3+3\check{t}_3^6\check{x}_3\check{z}_3-3\check{t}_3^6\check{x}_3^2-\\
-3\check{t}_3^5\check{x}_3^3+6\check{t}_3^5\check{x}_3^2\check{z}_3-3\check{t}_3^5\check{x}_3\check{z}_3^2+6\check{t}_3^6\check{x}_3^2\check{z}_3-3\check{t}_3^6\check{x}_3^3-3\check{t}_3^6\check{x}_3\check{z}_3^2+3\check{t}_3^6\check{x}_3^3\check{z}_3-\check{t}_3^6\check{x}_3^4-3\check{t}_3^6\check{x}_3^2\check{z}_3^2+\check{t}_3^6\check{x}_3\check{z}_3^3=0.
\end{multline*}
for $\lambda\ne 0$. If $\lambda=0$, then this equation defines $D_0^3=S^3_0+\mathbf{E}_1^3+2\mathbf{E}_2^3+\mathbf{E}_3$.

The pencil $\mathcal{S}^3$ has two base curves contained in the surface $\mathbf{E}_3$.
One of them is given by the equation $\check{t}_3=\check{z}_3=0$,
and another one is given by the equation $\check{t}_3=\check{x}_3=0$.
Denote the former curve by $C_{10}^3$, and denote the latter curve by $C_{11}^2$.
Then $\mathbf{M}_{10}^0=2$. Similarly, we have $\mathbf{m}_{10}=3$,
because (in general point of the curve $C_{10}^3$) the surface $S^3_0+\mathbf{E}_1^3+2\mathbf{E}_3+\mathbf{E}_3$ is given by
$$
\check{x}_3\check{t}_3^3\big(\check{t}_3\check{x}_3-\check{t}_3\check{z}_3+\check{t}_3+1\big)+\check{t}_3\big(\check{t}_3\check{x}_3-\check{t}_3\check{z}_3+\check{t}_3+\check{z}_3\big)=0,
$$
and $S^3_\infty$ is given by $\check{t}_3\check{x}_3-\check{t}_3\check{z}_3+\check{t}_3+\check{z}_3=0$.
Thus, the curve $C_{10}^3$ contributes $\textbf{\textcircled{2}}$ to the defect by Lemma~\ref{lemma:main-2}~and~\eqref{equation:D-A-B}.
On the other hand, we have $\mathbf{M}_{11}^0=1$.
Thus, by Lemma~\ref{lemma:main-2} and \eqref{equation:D-A-B}, the curve $C_{11}^3$ does not contribute to the defect.

Let  $\alpha_4\colon U_4\to U_3$ be the blow up of the intersection point $C_{10}^3\cap C_{11}^3$.
Then the birational map $\alpha\colon U\to\mathbb{P}^3$ in \eqref{equation:main-diagram}
can be decomposed via the following commutative diagram:
$$
\xymatrix{
&U_2\ar@{->}[ld]_{\alpha_2}&&U_3\ar@{->}[ll]_{\alpha_3}\\
U_1\ar@{->}[dr]_{\alpha_1}&&&&U_4\ar@{->}[lu]_{\alpha_4}\\
&\mathbb{P}^3&&U\ar@{->}[ll]^\alpha\ar@{->}[ru]_{\gamma}& }
$$
where $\gamma$ is a birational morphism that is an isomorphism along the exceptional locus of the composition $\alpha_1\circ\alpha_2\circ\alpha_3\circ\alpha_4$.

The surface $\mathbf{E}_4$ contains one base curve of the pencil $\mathcal{S}^4$.
Denote this curve by $C_{12}^4$.
Simple computations imply that neither  $\mathbf{E}_4$  nor the curve $C_{12}^4$ contribute to the defect.
Thus, summarizing, we see that $\mathbf{D}_{P_{\{x\},\{z\},\{t\}}}^0=10$.
\end{proof}

\begin{lemma}
\label{lemma:2-2-P-y-z-t}
One has $\mathbf{D}_{P_{\{y\},\{z\},\{t\}}}^0=5$.
\end{lemma}

\begin{proof}
Let us use the notation of the proof of Lemma~\ref{lemma:2-2-P-x-z-t}.
In a neighborhood of the preimage of the point  $P_{\{y\},\{z\},\{t\}}$ on the threefold $U_4$,
we can identify the threefold $U_4$ with the chart of $\mathbb{P}^3$ that is given by $x=1$.
In this chart, the surface $S_\lambda^4$ is given by
$$
\lambda y^2+z^3+z^3t-yz^2-\lambda yzt+\lambda y^2z+\lambda yt^2-yz^3-z^2t^2=0,
$$
and $(0,0,0)$ is the preimage of the point  $P_{\{y\},\{z\},\{t\}}$.

Let $\alpha_5\colon U_5\to U_4$ be the blow up of the point  $(0,0,0)$.
Then $D_0^5=S^5_0+\mathbf{E}_5$.
By~\eqref{equation:A} and \eqref{equation:D-A-B}, this contributes~$\textbf{\textcircled{1}}$ to the defect $\mathbf{D}_{P_{\{y\},\{z\},\{t\}}}^0$.

A chart of the blow up $\alpha_5$ is given by the coordinate change $y_5=\frac{y}{t}$, $z_5=\frac{z}{t}$, $t_5=t$.
In~this chart, the surface $D^5_\lambda$ is given by the equation
$$
\lambda y_5(t_5+y_5)-\lambda t_5y_5z_5+\big(\lambda t_5y_5^2z_5-t_5^2z_5^2-t_5y_5z_5^2+t_5z_5^3\big)+t_5^2z_5^3-t_5^2y_5z_5^3=0.
$$
We can rewrite this equation as
\begin{multline*}
\lambda \hat{y}_5\hat{t}_5+\lambda \hat{y}_5\hat{z}_5(\hat{y}_5-\hat{t}_5)-\hat{z}_5(\hat{y}_5-\hat{t}_5)(\lambda \hat{y}_5^2+\hat{z}_5^2-\hat{z}_5\hat{t}_5)+\hat{z}_5^3(\hat{y}_5-\hat{t}_5)^2-\hat{y}_5\hat{z}_5^3(\hat{y}_5-\hat{t}_5)^2=0,
\end{multline*}
where $\hat{y}_5=y_5$, $\hat{z}_5=z_5$, and $\hat{t}_5=y_5+t_5$.
Then $\mathbf{E}_5$ is given by $\hat{y}_5=\hat{t}_5$.

The surface $\mathbf{E}_5$ contains one base curve of the pencil $\mathcal{S}^5$.
Denote it by $C_{13}^5$. Then $C_{13}^5$ is given by $\hat{y_5}=\hat{t}_5=0$.
One has $\mathbf{M}_{13}^0=1$.
By Lemma~\ref{lemma:main-2} and \eqref{equation:D-A-B},
the curve~$C_{13}^5$ does not contribute to the defect of the singular point $P_{\{y\},\{z\},\{t\}}$.

Let $\alpha_6\colon U_6\to U_5$ be the blow up of the point $(\hat{y}_5,\hat{z}_5,\hat{t}_5)=(0,0,0)$.
Then
$$
D_0^6=S^6_0+\mathbf{E}_5^6+2\mathbf{E}_6.
$$
Thus, by~\eqref{equation:A} and \eqref{equation:D-A-B}, this contributes~$\textbf{\textcircled{1}}$ to the defect $\mathbf{D}_{P_{\{y\},\{z\},\{t\}}}^0$.

One (local) chart of the blow up $\alpha_6$ is given by $\hat{y}_6=\frac{\hat{y}_5}{\hat{z}_5}$, $\hat{z}_6=\hat{z}_5$, $\hat{t}_6=\frac{\hat{t}_5}{\hat{z}_5}$.
Thus, if $\lambda\ne 0$, then $S^6_\lambda$ is given by the equation
\begin{multline*}
\lambda \hat{y}_6\hat{t}_6+\hat{z}_6(\hat{y}_6-\hat{t}_6)(\lambda \hat{y}_6-\hat{z}_6)+\hat{z}_6^2\hat{t}_6(\hat{y}_6-\hat{t}_6)-\hat{z}_6^2(\hat{y}_6-\hat{t}_6)(\lambda \hat{y}_6^2-\hat{y}_6\hat{z}_6+\hat{z}_6\hat{t}_6)-\hat{y}_6\hat{z}_6^4(\hat{y}_6-\hat{t}_6)^2=0.
\end{multline*}
The surface $\mathbf{E}_6$ is given by $\hat{z}_6=0$.
It contains two base curves of the pencil $\mathcal{S}^6$.
One of them is given by $\hat{y}_6=\hat{z}_6=0$,
and another is given by $\hat{t}_6=\hat{z}_6=0$.
Denote the former one by $C_{14}^6$, and denote the latter one by $C_{15}^6$.
If $\lambda\ne 0$, then
$$
\mathrm{mult}_{C_{14}^6}\big(S^6_0+\mathbf{E}_5^6+2\mathbf{E}_6\big)=\mathbf{M}_{14}^0=\mathbf{m}_{14}=\mathrm{mult}_{C_{14}^6}\Big(\big(S^6_0+\mathbf{E}_5^6+2\mathbf{E}_6\big)\cdot S_\lambda^6\Big)=2.
$$
Thus, the curve $C_{14}^6$ contributes $\textbf{\textcircled{1}}$ to the defect by Lemma~\ref{lemma:main-2} and \eqref{equation:D-A-B}.
Similarly, we see that the curve $C_{15}^6$ contributes $\textbf{\textcircled{1}}$ to the defect of the singular point $P_{\{y\},\{z\},\{t\}}$.

Let $\alpha_7\colon U_7\to U_6$ be the blow up of the point $C_{14}^6\cap C_{15}^6$.
Then $D_0^7=S^7_0+\mathbf{E}_5^7+2\mathbf{E}_6^7+\mathbf{E}_7$.
By~\eqref{equation:A} and \eqref{equation:D-A-B}, this contributes~$\textbf{\textcircled{1}}$ to the defect $\mathbf{D}_{P_{\{y\},\{z\},\{t\}}}^0$.

One (local) chart of the blow up $\alpha_7$ is given by
$\hat{y}_7=\frac{\hat{y}_6}{\hat{z}_6}$, $\hat{t}_7=\frac{\hat{t}_6}{\hat{z}_6}$, $\hat{z}_7=\hat{z}_6$
If $\lambda\ne 0$, then the surface $S^7_\lambda$ is given by the equation
\begin{multline*}
\hat{z}_7\hat{t}_7-\hat{y}_7\hat{z}_7+\lambda\hat{y}_7\hat{t}_7+\lambda \hat{y}_7\hat{z}_7(\hat{y}_7-\hat{t}_7)+\hat{z}_7^2\hat{t}_7 (\hat{y}_7-\hat{t}_7)+\\
+\hat{z}_7^3(\hat{y}_7-\hat{t}_7)^2-\lambda \hat{y}_7^2\hat{z}_7^3(\hat{y}_7-\hat{t}_7)-\hat{y}_7\hat{z}_7^5(\hat{y}_7-\hat{t}_7)^2=0.
\end{multline*}
The surface $\mathbf{E}_7$ is given by $\hat{z}_7=0$.
It contains two base curves of the pencil $\mathcal{S}^7$.
One of them is given by $\hat{y}_7=\hat{z}_7=0$,
and another is given by $\hat{t}_7=\hat{z}_7=0$.
Denote the former one by $C_{16}^7$, and denote the latter one by $C_{17}^7$.
Then $\mathbf{M}_{16}^0=\mathbf{M}_{17}^0=1$,
so that $C_{16}^7$ and $C_{17}^7$ do not contribute anything to $\mathbf{D}_{P_{\{y\},\{z\},\{t\}}}^0$
by Lemma~\ref{lemma:main-2} and \eqref{equation:D-A-B}.

Let  $\alpha_8\colon U_8\to U_8$ be the blow up of the intersection point $C_{16}^7\cap C_{17}^7$.
Then the birational map $\alpha\colon U\to\mathbb{P}^3$ in \eqref{equation:main-diagram}
can be decomposed via the following commutative diagram:
$$
\xymatrix{
&U_3\ar@{->}[dl]_{\alpha_3}&&U_4\ar@{->}[ll]_{\alpha_4}&&U_5\ar@{->}[ll]_{\alpha_5}&&U_6\ar@{->}[ll]_{\alpha_6}&&\\
U_2\ar@{->}[rd]_{\alpha_2}&&&&&&&&U_7\ar@{->}[lu]_{\alpha_7}\\
&U_1\ar@{->}[rr]_{\alpha_1}&&\mathbb{P}^3&&U\ar@{->}[ll]^\alpha\ar@{->}[rr]_{\delta}&&U_8\ar@{->}[ru]_{\alpha_8}}
$$
where $\delta$ is a birational morphism that is an isomorphism along the exceptional locus of the composition $\alpha_5\circ\alpha_6\circ\alpha_7\circ\alpha_8$.

Arguing as above, we see that $\mathbf{E}_8$
does not contribute anything to the computation of defect.
Moreover, the surface $\mathbf{E}_8$ does not contain
base curves of the pencil $\mathcal{S}^8$.
Thus, summarizing, we see that $\mathbf{D}_{P_{\{x\},\{z\},\{t\}}}^0=5$.
\end{proof}

Using \eqref{equation:2-2-S-0} and Lemmas~\ref{lemma:2-2-P-x-y-z}, \ref{lemma:2-2-P-x-z-t}, and \ref{lemma:2-2-P-y-z-t},
we conclude that $[\mathsf{f}^{-1}(0)]=21$, so that \eqref{equation:main-1} in Main Theorem holds in this case.

\subsection{Family \textnumero $2.3$}
\label{section:r-2-n-3}

In this case, the threefold $X$ can be obtained from a smooth quartic hypersurface in $\mathbb{P}(1,1,1,1,2)$ by blowing up a smooth elliptic curve.
In particular, we have $h^{1,2}(X)=11$.
Let $\mathsf{p}$ be the Laurent polynomial
$$
\frac{(a+b+1)^4(c+1)}{abc}+c+1.
$$
Then $\mathsf{p}$ gives the commutative diagram \eqref{equation:CCGK-compactification} by \cite[Proposition 16]{P16}.

Let $\gamma\colon\mathbb{C}^3\dasharrow\mathbb{C}^\ast\times\mathbb{C}^\ast\times\mathbb{C}^\ast$
be a birational transformation that is given by the change of coordinates
$$
\left\{\aligned
&a=-xz,\\
&b=x+xz-1,\\
&c=-\frac{y}{z}-1.\\
\endaligned
\right.
$$
Like in Subsection~\ref{section:r-2-n-2}, we can use $\gamma$ to expand
\eqref{equation:CCGK-compactification} to the commutative diagram \eqref{equation:diagram-2-2}.
The only difference is that now the pencil $\mathcal{S}$ is given by the equation
\begin{equation}
\label{equation:2-3-pencil}
x^3y+(\lambda z+y)(y+z)(xz+xt-t^2)=0,
\end{equation}
where $\lambda\in\mathbb{C}\cup\{\infty\}$.
As in Subsection~\ref{section:r-2-n-2}, we will follow the scheme described in~Section~\ref{section:scheme},
and we will use assumptions and notation introduced in this section.
But now $S_\lambda$ denotes the quartic surface in $\mathbb{P}^3$ that is given by \eqref{equation:2-3-pencil}.

Let $\mathbf{Q}$ be the quadric given by  $xz+xt-t^2=0$. Then $S_\infty=H_{\{z\}}+H_{\{y,z\}}+\mathbf{Q}$.
Similarly, let $\mathbf{S}$ be the cubic surface in $\mathbb{P}^3$ that is given by the equation
$$
x^3+xyz+xyt-yt^2+xz^2+xzt-zt^2=0.
$$
Then $S_{0}=H_{\{y\}}+\mathbf{S}$.
Thus, we see that both $S_\infty$ and $S_0$ are reducible.
In fact, these are the only reducible surfaces in $\mathcal{S}$.
Indeed, if $\lambda\ne\infty$, $\lambda\ne 0$, and $\lambda\ne 1$, then $S_\lambda$ has isolated singularities,
which implies that it is irreducible.
Moreover, the surface $S_1$ is also irreducible, but it is singular along the line $L_{\{x\},\{y,z\}}$.

If $\lambda\ne\infty$, then
\begin{equation}
\label{equation:2-3}
\begin{split}
H_{\{z\}}\cdot S_\lambda&=L_{\{y\},\{z\}}+\mathcal{C}_1,\\
H_{\{y,z\}}\cdot S_\lambda&=L_{\{y\},\{z\}}+3L_{\{x\},\{y,z\}},\\
\mathbf{Q}\cdot S_\lambda&=6L_{\{x\},\{t\}}+\mathcal{C}_2,
\end{split}
\end{equation}
where $\mathcal{C}_1$ and $\mathcal{C}_2$ are the curves in $\mathbb{P}^3$ that are given by
the equations $z=x^3+xyt-yt^2=0$ and $y=xz+xt-t^2=0$, respectively.
Thus, if $\lambda\ne\infty$, then
$$
S_\infty\cdot S_\lambda=6L_{\{x\},\{t\}}+2L_{\{y\},\{z\}}+3L_{\{x\},\{y,z\}}+\mathcal{C}_1+\mathcal{C}_2.
$$
Hence, the base curves of the pencil $\mathcal{S}$ are $L_{\{x\},\{t\}}$,
$L_{\{y\},\{z\}}$, $L_{\{x\},\{y,z\}}$,  $\mathcal{C}_1$, and $\mathcal{C}_2$.

If $\lambda\ne 0$ and $\lambda\ne 1$, then the singular points of the surface $S_\lambda$
contained in the base locus of the pencil $\mathcal{S}$ can be described as follows:
\begin{itemize}\setlength{\itemindent}{3cm}
\item[$P_{\{x\},\{z\},\{t\}}$:] type $\mathbb{A}_1$ with quadratic term $xz+xt-t^2$;

\item[$P_{\{x\},\{y\},\{z\}}$:] type $\mathbb{A}_5$ with quadratic term $(y+\lambda z)(y+z)$;

\item[$P_{\{x\},\{t\},\{y,z\}}$:] type $\mathbb{A}_5$ with quadratic term $(\lambda-1)x(y+z)$;

\item[{$[0:\lambda:-1:0]$}:] type $\mathbb{A}_5$ with quadratic term $(\lambda-1)x(y+\lambda z)$.
\end{itemize}

If $\lambda\not\in\{\infty,0,1\}$, then it follows from \eqref{equation:2-3} that
$$
H_\lambda\sim L_{\{y\},\{z\}}+\mathcal{C}_1\sim L_{\{y\},\{z\}}+3L_{\{x\},\{y,z\}}\sim_{\mathbb{Q}} 3L_{\{x\},\{t\}}+\frac{1}{2}\mathcal{C}_2
$$
on the (singular) quartic surface $S_\lambda$.
Therefore, if $\lambda\not\in\{\infty,0,1\}$, then the intersection matrix of the curves $L_{\{x\},\{t\}}$, $L_{\{y\},\{z\}}$, $L_{\{x\},\{y,z\}}$,  $\mathcal{C}_1$, and $\mathcal{C}_2$
on the surface $S_{\lambda}$ has the same rank as the intersection
matrix of the curves $L_{\{y\},\{z\}}$, $L_{\{x\},\{t\}}$, and $H_{\lambda}$.
In this case, we also have
$$
H_{\{y\}}\cdot S_\lambda=2L_{\{y\},\{z\}}+\mathcal{C}_2,
$$
so that $2L_{\{y\},\{z\}}+\mathcal{C}_2\sim H_\lambda$, which gives $2L_{\{y\},\{z\}}+H_\lambda\sim 6L_{\{x\},\{t\}}$.

If $\lambda\not\in\{\infty,0,1\}$, then the intersection matrix of the curves $L_{\{y\},\{z\}}$, $L_{\{x\},\{t\}}$, and $H_{\lambda}$
on the surface $S_\lambda$ is given by
\begin{center}
\renewcommand\arraystretch{1.42}
\begin{tabular}{|c||c|c|c|}
\hline
 $\bullet$  & $L_{\{y\},\{z\}}$ & $L_{\{x\},\{t\}}$ & $H_{\lambda}$ \\
\hline\hline
$L_{\{y\},\{z\}}$ & $-\frac{1}{2}$ & $0$ & $1$ \\
\hline
$L_{\{x\},\{t\}}$ & $0$ & $\frac{1}{6}$ & $1$ \\
\hline
 $H_{\lambda}$  & $1$ & $1$ & $4$ \\
\hline
\end{tabular}
\end{center}
Its rank is $2$.
On the other hand, the description of singular points of the surface~$S_\lambda$ easily implies that
$\mathrm{rk}\,\mathrm{Pic}(\widetilde{S}_{\Bbbk})=\mathrm{rk}\,\mathrm{Pic}(S_{\Bbbk})+16$,
so that \eqref{equation:main-2-simple} holds.
Thus, by Lemma~\ref{lemma:cokernel}, we see that \eqref{equation:main-2} in Main Theorem holds.

Let us prove \eqref{equation:main-1} in Main Theorem.
Observe that $[\mathsf{f}^{-1}(\lambda)]=1$ for every $\lambda\not\in\{\infty,0,1\}$.
This follows from Lemma~\ref{corollary:irreducible-fibers}.
Thus, to verify \eqref{equation:main-1} in Main Theorem, we have to show that
$[\mathsf{f}^{-1}(0)]+[\mathsf{f}^{-1}(1)]=13$. We start with

\begin{lemma}
\label{lemma:2-3-S-0}
One has $[\mathsf{f}^{-1}(0)]=2$.
\end{lemma}

\begin{proof}
Note that $[S_0]=2$, and $S_0$ is smooth at general points of the curves
$L_{\{x\},\{t\}}$, $L_{\{y\},\{z\}}$, $L_{\{x\},\{y,z\}}$,  $\mathcal{C}_1$, and $\mathcal{C}_2$.
Furthermore, the points $P_{\{x\},\{z\},\{t\}}$, $P_{\{x\},\{y\},\{z\}}$, and $P_{\{x\},\{t\},\{y,z\}}$
are good double points of the surface $S_0$.
Then $[\mathsf{f}^{-1}(0)]=2$ by Corollary~\ref{corollary:normal-crossing-simple}.
\end{proof}

Let us show that $[\mathsf{f}^{-1}(1)]=11$.
Let $C_1=\mathcal{C}_1$, $C_2=\mathcal{C}_2$, $C_3=L_{\{x\},\{t\}}$, $C_4=L_{\{y\},\{z\}}$, $C_5=L_{\{x\},\{y,z\}}$.
Then $\mathbf{m}_{1}=\mathbf{m}_{2}=1$, $\mathbf{m}_{3}=6$, $\mathbf{m}_{4}=2$, and $\mathbf{m}_{5}=3$.
Moreover, one has $\mathbf{M}_{1}^{1}=\mathbf{M}_{2}^{1}=\mathbf{M}_{3}^{1}=\mathbf{M}_{4}^{1}=1$ and $\mathbf{M}_{5}^{0}=2$.
Then $\mathbf{C}_1^1=\mathbf{C}_1^1=\mathbf{C}_3^1=\mathbf{C}_4^1=0$ and $\mathbf{C}_5^1=2$ by Lemma~\ref{lemma:main}.
Thus, using \eqref{equation:equation:number-of-irredubicle-components-refined}, we see that
\begin{equation}
\label{equation:2-3-S-1}
\big[\mathsf{f}^{-1}(1)\big]=3+\mathbf{D}_{P_{\{x\},\{z\},\{t\}}}^1+\mathbf{D}_{P_{\{x\},\{y\},\{z\}}}^1+\mathbf{D}_{P_{\{x\},\{t\},\{y,z\}}}^1.
\end{equation}

\begin{lemma}
\label{lemma:2-3-P-x-z-t}
One has $\mathbf{D}_{P_{\{x\},\{z\},\{t\}}}^1=0$.
\end{lemma}

\begin{proof}
Observe that $P_{\{x\},\{z\},\{t\}}$ is an isolated ordinary double point of the surface $S_1$.
Thus, we have $\mathbf{D}_{P_{\{x\},\{z\},\{t\}}}^1=0$ by Lemma~\ref{lemma:normal-crossing}.
\end{proof}

\begin{lemma}
\label{lemma:2-3-P-x-y-z}
One has $\mathbf{D}_{P_{\{x\},\{y\},\{z\}}}^1=1$.
\end{lemma}

\begin{proof}
In the chart $t=1$, one has $P_{\{x\},\{y\},\{z\}}=(0,0,0)$, and the surface $S_\lambda$ is given by
$$
(y+\lambda z)(y+z)-x(y+\lambda z)(y+z)-x(x^2y+\lambda yz^2+\lambda z^3+y^2z+yz^2)=0.
$$
Let $\alpha_1\colon U_1\to\mathbb{P}^3$ be the blow up of the point  $P_{\{x\},\{y\},\{z\}}$.
Then $S^1_\lambda\sim -K_{U_1}$ for every~$\lambda\in\mathbb{C}$.
A chart of the blow up $\alpha_1$ is given by the coordinate change $x_1=x$, $y_1=\frac{y}{x}$, $z_1=\frac{z}{x}$.
In~this chart, the surface $\mathbf{E}_1$ is given by $x_1=0$, and the surface $S^1_\lambda$ is given by
$$
(y_1+\lambda z_1)(y_1+z_1)-x_1(x_1y_1+(y_1+\lambda z_1)(y_1+z_1))-x_1^2z_1(y_1+z_1)(y_1+\lambda z_1)=0.
$$
This shows that  $\mathbf{E}_1$ contains one base curve of the pencil $\mathcal{S}^1$.
It is given by $x_1=y_1+z_1=0$.
Denote this curve by $C_6^1$.
Then $\mathbf{M}_{6}^1=\mathbf{m}_{6}=2$.
But surfaces in the pencil $\mathcal{S}^1$ do not have fixed singular points in $\mathbf{E}_1$.
Thus, keeping in mind the construction of the birational morphism~$\alpha$,
we see that $\mathbf{D}_{P_{\{x\},\{y\},\{z\}}}^1=1$ by \eqref{equation:D-A-B}, \eqref{equation:A}, and Lemma~\ref{lemma:main-2}.
\end{proof}

\begin{lemma}
\label{lemma:2-3-P-x-t-yz}
One has $\mathbf{D}_{P_{\{x\},\{t\},\{y,z\}}}^1=7$.
\end{lemma}

\begin{proof}
Let us use the notation of the proof of Lemma~\ref{lemma:2-3-P-x-y-z}.
In a neighborhood of the preimage of the point $P_{\{x\},\{t\},\{y,z\}}$,
we can identify $U_1$ with the chart of $\mathbb{P}^3$ that is given by $z=1$.
In this chart, the surface $S_\lambda^1$ is given by the equation
$$
(\lambda-1)\hat{x}\hat{y}+ \left((\lambda-1) (\hat{x}\hat{y}\hat{t}-\hat{y}\hat{t}^2)+\hat{x}\hat{y}^2-\hat{x}^3\right)+\hat{y}(\hat{x}^3+\hat{x}\hat{y}\hat{t}-\hat{y}\hat{t}^2)=0.
$$
where $\hat{x}=x$, $\hat{t}=t$, $\hat{y}=y+z$.
In these coordinates, the point $(0,0,0)$ is the preimage of the point $P_{\{x\},\{t\},\{y,z\}}$.

Let $\alpha_2\colon U_2\to U_1$ be the blow up of the point $(0,0,0)$.
Then $D_1^2=S^2_1+\mathbf{E}_2$.
Thus, by~\eqref{equation:A} and \eqref{equation:D-A-B}, the surface $\mathbf{E}_2$ contributes $\textbf{\textcircled{1}}$ to $\mathbf{D}_{P_{\{x\},\{t\},\{y,z\}}}^1$.

One chart of the blow up $\alpha_2$ is given by the coordinate change $\hat{x}_2=\frac{\hat{x}}{\hat{t}}$, $\hat{y}_2=\frac{\hat{y}}{\hat{t}}$, $\hat{t}_2=\hat{t}$.
In this chart, the surface $S_\lambda^2$ is given by
$$
(\lambda-1)\hat{y}_2(\hat{x}_2-\hat{t}_2)+\Big(\lambda\hat{t}_2\hat{x}_2\hat{y}_2-\hat{t}_2\hat{x}_2\hat{y}_2\Big)+\Big(\hat{t}_2\hat{x}_2\hat{y}_2^2-\hat{t}_2^2\hat{y}_2^2-\hat{t}_2\hat{x}_2^3\Big)+\hat{t}_2^2\hat{x}_2\hat{y}_2^2+\hat{t}_2^2\hat{x}_2^3\hat{y}_2=0
$$
for $\lambda\ne 1$.
Let $\bar{x}_2=\hat{x}_2-\hat{t}_2$, $\bar{y}_2=\hat{z}_2$ and $\bar{t}_2=\hat{t}_2$.
We can rewrite the latter equation as
$$
(\lambda-1)\Big(\bar{x}_2\bar{y}_2+\bar{y}_2\bar{t}_2(\bar{x}_2+\bar{t}_2)\Big)=\bar{x}_2^3\bar{t}_2+3\bar{x}_2^2\bar{t}_2^2+3\bar{x}_2\bar{t}_2^3+\bar{t}_2^4-\bar{x}_2\bar{y}_2^2\bar{t}_2-\bar{y}_2^2\bar{t}_2^2\big(\bar{x}_2+\bar{t}_2\big)-\bar{y}_2\bar{t}_2^2(\bar{x}_2+\bar{t}_2)^3.
$$
For $\lambda=1$, this equation defines $D_2^2=S^2_1+\mathbf{E}_2$.

The surface $\mathbf{E}_2$ is given by $\bar{t}_2=0$.
It contains two base curves of the pencil~$\mathcal{S}^2$.
One of them is given by $\bar{x}_2=\bar{t}_2=0$,
and another one is given by $\bar{z}_2=\bar{t}_2=0$.
Denote the former curve by $C_{7}^2$, and denote the latter curve by $C_{8}^2$.
Then $\mathbf{M}_7^1=2$. Note that $\mathbf{m}_7=4$, because $S_\infty$ is given by $\bar{y}_2(\bar{t}_2^2+t\bar{x}_2+\bar{x}_2)=0$,
and $S^2_1+\mathbf{E}_2$ is given by
$$
\Big(\bar{t}_2^4+3\bar{t}_2^3\bar{x}_2+3\bar{t}_2^2\bar{x}_2^2+\bar{t}_2\bar{x}_2^3\Big)\big(\bar{t}_2\bar{y}_2-1\big)+\bar{y}_2\big(\bar{t}_2^2+\bar{t}_2\bar{x}_2+\bar{x}_2\big)\big(\bar{t}_2\bar{y}_2-1\big)=0.
$$
Thus, the curve $C_{7}^2$ contributes $\textbf{\textcircled{3}}$ to the defect by Lemma~\ref{lemma:main-2} and \eqref{equation:D-A-B}.
On the other hand, one has $\mathbf{M}_8^1=1$, so that $C_{8}^2$ does not contribute to the defect.

Let $\alpha_3\colon U_3\to U_2$ be the blow up of the point $C_{7}^2\cap C_{8}^2$.
Then $D_0^3=S^3_0+\mathbf{E}_2^3+2\mathbf{E}_3$.
By~\eqref{equation:A} and \eqref{equation:D-A-B}, the surface $\mathbf{E}_3$ contributes $\textbf{\textcircled{1}}$ to the defect $\mathbf{D}_{P_{\{x\},\{t\},\{y,z\}}}^1$.

A chart of the blow up $\alpha_3$ is given by the coordinate change
$\bar{x}_3=\frac{\bar{x}_2}{\bar{t}_2}$, $\bar{y}_3=\frac{\bar{y}_2}{\bar{t}_2}$, $\bar{t}_3=\bar{t}_2$.
In this chart, the surface $\mathbf{E}_3$ is given by $\bar{t}_3$.
Similarly, if $\lambda\ne 1$, then $S_\lambda^3$ is given by
\begin{multline*}
(\lambda-1)\bar{y}_3(\bar{x}_3+\bar{t}_3)-t^2_2+\bar{x}_3\bar{t}_3\Big((\lambda-1)\bar{y}_3-3\bar{t}_3\Big)-3\bar{x}_3^2\bar{t}_3^2-\\
-\bar{t}_3^2\Big(\bar{x}_3^3-\bar{y}_3\bar{t}_3^2-\bar{x}_3\bar{y}_3^2-\bar{y}_3^2\bar{t}_3\Big)+\bar{x}_3\bar{y}_3t^3_2\big(3\bar{t}_3+\bar{y}_3\big)+3\bar{x}_3^2\bar{y}_3\bar{t}_3^4+\bar{x}_3^3\bar{y}_3\bar{t}_3^4=0.
\end{multline*}
Then $\mathbf{E}_3$ contains two base curves of the pencil $\mathcal{S}^3$.
One of them is given by $\bar{t}_3=\bar{x}_3=0$,
and another one is given by $\bar{t}_3=\bar{y}_3=0$.
Denote the former curve by $C_{9}^3$, and denote the latter curve by $C_{10}^2$.
Then $\mathbf{M}_9^1=\mathbf{M}_{10}^1=2$ and $\mathbf{m}_9=\mathbf{m}_{10}=2$.
Thus, by Lemma~\ref{lemma:main-2} and \eqref{equation:D-A-B}, the curves $C_{9}^3$ and $C_{10}^3$ contribute $\textbf{\textcircled{2}}$ to the defect
$\mathbf{D}_{P_{\{x\},\{t\},\{y,z\}}}^1$.

Summarizing, we see that $\mathbf{D}_{P_{\{x\},\{t\},\{y,z\}}}^1\geqslant 7$.
Looking at the defining equation of the surface $S_\lambda^3$, one can easily see that $\mathbf{D}_{P_{\{x\},\{t\},\{y,z\}}}^1=7$.
\end{proof}

Using \eqref{equation:2-3-S-1} and Lemmas~\ref{lemma:2-3-P-x-z-t}, \ref{lemma:2-3-P-x-y-z}, \ref{lemma:2-3-P-x-t-yz},
we see that \eqref{equation:main-1} in Main Theorem holds.

\subsection{Family \textnumero $2.4$}
\label{section:r-2-n-4}

In this case, the threefold $X$ is a blow up of $\mathbb{P}^3$ along the smooth complete intersection of two cubic surfaces,
which implies that $h^{1,2}(X)=10$.
A mirror partner of the threefold $X$ is given by Minkowski polynomial \textnumero $3963.1$,
which is
\begin{multline*}
\frac{z^{2}}{x}+\frac{3z}{x}+\frac{3}{x}+\frac{yz}{x}+\frac{z^{2}}{y}+\frac{1}{xz}+\frac{2y}{x}+\frac{2z}{y}+\frac{y}{xz}+\frac{1}{y}+\\
+4z+\frac{3}{z}+2y+2\frac{xz}{y}+\frac{2y}{z}+\frac{2x}{y}+4x+3\frac{x}{z}+\frac{xy}{z}+\frac{x^{2}}{y}+\frac{x^{2}}{z}.
\end{multline*}
The quartic pencil $\mathcal{S}$ is given by
\begin{multline*}
z^3y+3z^2ty+3t^2yz+y^2z^2+z^3x+t^3y+2y^2tz+2z^2xt+\\
+y^2t^2+t^2xz+4xyz^2+3t^2xy+2y^2xz+2x^2z^2+2y^2xt+\\
+2x^2tz+4x^2yz+3x^2ty+x^2y^2+x^3z+x^3y=\lambda xyzt.
\end{multline*}
This equation is invariant with respect to the swap $x\leftrightarrow z$.

Suppose that $\lambda\ne\infty$.
Let $\mathcal{C}$ be the conic $t=xy+xz+yz=0$. Then
\begin{itemize}
\item $H_{\{x\}}\cdot S_\infty=L_{\{x\},\{y\}}+2L_{\{x\},\{z,t\}}+L_{\{x\},\{y,z,t\}}$,
\item $H_{\{y\}}\cdot S_\infty=L_{\{x\},\{y\}}+L_{\{y\},\{z\}}+2L_{\{y\},\{x,z,t\}}$,
\item $H_{\{z\}}\cdot S_\infty=L_{\{y\},\{z\}}+2L_{\{z\},\{x,t\}}+L_{\{z\},\{x,y,t\}}$,
\item $H_{\{t\}}\cdot S_\infty=L_{\{t\},\{x,z\}}+L_{\{t\},\{x,y,z\}}+\mathcal{C}$.
\end{itemize}
This shows that
\begin{multline*}
S_\infty\cdot S_\lambda=2L_{\{x\},\{y\}}+2L_{\{y\},\{z\}}+2L_{\{x\},\{z,t\}}+2L_{\{z\},\{x,t\}}+L_{\{t\},\{x,z\}}+\\
+2L_{\{y\},\{x,z,t\}}+L_{\{x\},\{y,z,t\}}+L_{\{z\},\{x,y,t\}}+L_{\{t\},\{x,y,z\}}+\mathcal{C}.
\end{multline*}
Hence, the base locus of the pencil $\mathcal{S}$ consists of the curves
$L_{\{x\},\{y\}}$, $L_{\{y\},\{z\}}$, $L_{\{x\},\{z,t\}}$, $L_{\{z\},\{x,t\}}$, $L_{\{t\},\{x,z\}}$,
$L_{\{y\},\{x,z,t\}}$, $L_{\{x\},\{y,z,t\}}$, $L_{\{z\},\{x,y,t\}}$, $L_{\{t\},\{x,y,z\}}$, and $\mathcal{C}$.

Observe that $S_{-7}=H_{\{x,z,t\}}+H_{\{x,y,z,t\}}+\mathsf{Q}$,
where $\mathsf{Q}$ is an irreducible quadric surface that is given by $xy+xz+yz+yt=0$.
If $\lambda\ne -7$ and $\lambda\ne\infty$, then the surface $S_\lambda$ has isolated singularities,
which implies that it is irreducible.

The singular locus of the surface $S_{-7}$ contained in the base locus of the pencil $\mathcal{S}$
consists of the lines $L_{\{x\},\{z,t\}}$, $L_{\{z\},\{x,t\}}$, and $L_{\{y\},\{x,z,t\}}$.

\begin{lemma}
\label{lemma:r2-n4-singularities}
Suppose that $\lambda\ne -7$.
Then singular points of the surface $S_\lambda$ contained in the base locus of the pencil $\mathcal{S}$ can be describes as follows:
\begin{itemize}\setlength{\itemindent}{3cm}
\item[$P_{\{x\},\{y\},\{z,t\}}$:] type $\mathbb{A}_4$ with quadratic term $(\lambda+7)xy$;
\item[$P_{\{x\},\{z\},\{t\}}$:] type $\mathbb{D}_4$ with quadratic term $(x+z+t)^2$;
\item[$P_{\{y\},\{t\},\{x,z\}}$:] type $\mathbb{A}_1$ with quadratic term $(\lambda+6)ty-t^2-(x-z)(y+x-z+2t)$;
\item[$P_{\{y\},\{z\},\{x,t\}}$:] type $\mathbb{A}_4$ with quadratic term $(\lambda+7)yz$.
\end{itemize}
\end{lemma}

\begin{proof}
First let us describe the singularity of the surface $S_\lambda$ at the point $P_{\{y\},\{z\},\{x,t\}}$.
In the chart $t=1$, the surface $S_\lambda$ is given by
\begin{multline*}
(\lambda+7)\bar{z}\bar{y}-\bar{x}^2\bar{z}-(\lambda+8)\bar{x}\bar{y}\bar{z}-2\bar{z}^2\bar{x}-\bar{z}^2\bar{y}-\bar{z}^3+\bar{x}^3\bar{y}+\bar{x}^3\bar{z}+\\
+\bar{x}^2\bar{y}^2+4\bar{x}^2\bar{y}\bar{z}+2\bar{x}^2\bar{z}^2+2\bar{x}\bar{y}^2\bar{z}+4\bar{x}\bar{y}\bar{z}^2+\bar{z}^3\bar{x}+\bar{y}^2\bar{z}^2+\bar{y}\bar{z}^3=0,
\end{multline*}
where $\bar{x}=x+1$, $\bar{y}=y$, and $\bar{z}=z$.
Introducing new coordinates $\bar{x}_2=\bar{x}$, $\bar{y}_2=\frac{\bar{y}}{\bar{x}}$, and $\bar{z}_2=\frac{\bar{z}}{\bar{x}}$,
we rewrite this equation (after dividing by $\bar{x}_2^2$) as
\begin{multline*}
\bar{z}_2\big((\lambda+7)\bar{y}_2-\bar{x}_2\big)+\bar{x}_2^2\bar{y}_2+\bar{x}_2^2\bar{z}_2-(\lambda+8)\bar{x}_2\bar{y}_2\bar{z}_2-2\bar{z}_2^2\bar{x}_2+\bar{x}_2^2\bar{y}_2^2+4\bar{x}_2^2\bar{y}_2\bar{z}_2+\\
+2\bar{x}_2^2\bar{z}_2^2-\bar{x}_2\bar{y}_2\bar{z}_2^2-\bar{z}_2^3\bar{x}_2+2\bar{x}_2^2\bar{y}_2^2\bar{z}_2+4\bar{x}_2^2\bar{y}_2\bar{z}_2^2+\bar{x}_2^2\bar{z}_2^3+\bar{x}_2^2\bar{y}_2^2\bar{z}_2^2+\bar{x}_2^2\bar{y}_2\bar{z}_2^3=0.
\end{multline*}
This equation defines (a chart of) the blow up of the surface $S_\lambda$ at the point $P_{\{y\},\{z\},\{x,t\}}$.
The two exceptional curves of the blow up are given by the equations $\bar{x}_2=\bar{z}_2=0$ and $\bar{x}_2=\bar{y}_2=0$, respectively.
They intersect by the point $(0,0,0)$, which is singular point of the obtained surface.
Introducing new coordinates $\hat{x}_2=(\lambda+7)\bar{y}_2-\bar{x}_2$, $\hat{y}_2=\bar{y}_2$, and $\hat{z}_2=\bar{z}_2$,
we can rewrite the latter equations as
$$
\hat{x}_2\hat{z}_2+(\lambda+7)^2\hat{y}_2^3+\text{higher order terms}=0
$$
with respect to the weights $\mathrm{wt}(\hat{x}_2)=3$, $\mathrm{wt}(\hat{z}_2)=2$, and $\mathrm{wt}(\hat{z}_2)=3$.
This shows that the blown up surface has singularity of type $\mathbb{A}_2$ at the point $(0,0,0)$,
so that $P_{\{y\},\{z\},\{x,t\}}$ is a singular point of the surface $S_\lambda$ of type $\mathbb{A}_4$.

Since the equation of the surface $S_\lambda$ is invariant with respect to the swap $x\leftrightarrow z$,
we see that $P_{\{y\},\{x\},\{z,t\}}$ is a singular point of the surface $S_\lambda$ of type $\mathbb{A}_4$,
and the quadratic term of its defining equation is $(\lambda+7)xy$.

To show that $P_{\{y\},\{t\},\{x,z\}}$ is an ordinary double point of the surface $S_\lambda$,
we simply observe that the quadratic part of the Taylor expansion of the defining equation of the surface $S_\lambda$ at  the point $P_{\{y\},\{t\},\{x,z\}}$ in the chart $z=1$ is
$$
(\lambda+6)\acute{t}\acute{y}-t^2-2\acute{t}\acute{x}-\acute{x}^2-\acute{x}\acute{y}.
$$
where $\acute{x}=x-1$, $\acute{y}=y$, and $\acute{t}=t$.
This quadratic form has rank $3$, so that $P_{\{y\},\{t\},\{x,z\}}$ is an ordinary double point of the surface $S_\lambda$.

Finally, let us show that $P_{\{x\},\{z\},\{t\}}$ is a singular point of the surface $S_\lambda$ of type $\mathbb{D}_4$.
Let us consider the chart $y=1$ and introduce new coordinates $\tilde{x}=x$, $\tilde{z}=z$, and $\tilde{t}=t+x+z$.
Then $S_\lambda$ is given by
$$
\tilde{t}^2+(\lambda+7)\tilde{x}^2\tilde{z}+(\lambda+7)\tilde{z}^2\tilde{x}-(\lambda+6)\tilde{t}\tilde{x}\tilde{z}+\tilde{t}^3+\tilde{t}^2\tilde{x}\tilde{z}=0,
$$
where $P_{\{x\},\{z\},\{t\}}=(0,0,0)$.
Let us blow up $S_\lambda$ at this point.
Introducing new coordinates $\tilde{x}_6=\tilde{x}$, $\tilde{z}_6=\frac{\tilde{z}}{\tilde{x}}$, $\tilde{t}_6=\frac{\tilde{t}}{\tilde{x}}$,
we rewrite this equation (after dividing by $\tilde{x}_6^2$) as
$$
\tilde{t}_6^2+(\lambda+7)\tilde{x}_6\tilde{z}_6-(\lambda+6)\tilde{t}_6\tilde{x}_6\tilde{z}_6+(\lambda+7)\tilde{z}_6^2\tilde{x}_6+\tilde{t}_6^3\tilde{x}_6+\tilde{t}_6^2\tilde{x}_6^2\tilde{z}_6=0.
$$
This equation defines (a chart of) the blow up of the surface $S_\lambda$ at the point $P_{\{x\},\{z\},\{t\}}$.
The exceptional curve of this birational map is given by $\tilde{x}_6=\tilde{t}_6=0$.
The obtained surface has an ordinary double point at $(0,0,0)$, since its quadratic form $\tilde{t}_6^2+(\lambda+7)\tilde{x}_6\tilde{z}_6$ is of rank~$3$.
Note, however, that this surface is also singular at the point $(\tilde{x}_6,\tilde{z}_6,\tilde{t}_6)=(0,-1,0)$,
and is smooth along the curve $\tilde{x}_6=\tilde{t}_6=0$ away from these two points.
Introducing new coordinates $\check{x}_6=\tilde{x}_6$, $\check{z}_6=\tilde{z}_6+1$, and $\check{t}_6=\tilde{t}_6$,
we rewrite the latter equation as
$$
\check{t}_6^2-(\lambda+7)\check{x}_6\check{z}_6+(\lambda+6)\check{t}_6\check{x}_6-(\lambda+6)\check{t}_6\check{x}_6\check{z}_6+\check{t}_6^3\check{x}_6-\check{t}_6^2\check{x}_6^2+(\lambda+7)\check{z}_6^2\check{x}_6+\check{t}_6^2\check{x}_6^2\check{z}_6=0.
$$
Since $\lambda\ne -7$, the quadratic form $\check{t}_6^2-(\lambda+7)\check{x}_6\check{z}_6+(\lambda+6)\check{t}_6\check{x}_6$ has rank $3$,
so that the second singular point is also an ordinary double point of the obtained surface.

Now let us consider another chart of the blow up of the surface $S_\lambda$ at the point $P_{\{x\},\{z\},\{t\}}$.
To do this, we introduce coordinates $\tilde{x}_6^\prime=\frac{\tilde{x}}{\tilde{z}}$, $\tilde{z}_6^\prime=\tilde{z}$ and $\tilde{t}_6^\prime=\frac{\tilde{t}}{\tilde{z}}$.
After dividing by $(\tilde{x}_6^\prime)^2$, we obtain the equation
$$
(\tilde{t}_6^\prime)^2+(\lambda+7)\tilde{x}_6^\prime\tilde{z}_6^\prime-(\lambda+6)\tilde{t}_6^\prime\tilde{x}_6^\prime\tilde{z}_6^\prime+(\lambda+7)(\tilde{x}_6^\prime)^2\tilde{z}_6^\prime+(\tilde{t}_6^\prime)^3\tilde{z}_6^\prime+(\tilde{t}_6^\prime)^2\tilde{x}_6^\prime(\tilde{z}_6^\prime)^2=0.
$$
This surface is smooth along the curve $\tilde{x}_6^\prime=\tilde{t}_6^\prime=0$ except for two points:
the point $(\tilde{x}_6^\prime,\tilde{z}_6^\prime,\tilde{t}_6^\prime)=(0,0,0)$ and the point $(\tilde{x}_6^\prime,\tilde{z}_6^\prime,\tilde{t}_6^\prime)=(-1,0,0)$.
Both these points are ordinary double points of the obtained surface.
Note also that the point $(\tilde{x}_6^\prime,\tilde{z}_6^\prime,\tilde{t}_6^\prime)=(-1,0,0)$ is the point $(\tilde{x}_6,\tilde{z}_6,\tilde{t}_6)=(0,-1,0)$ in the first chart of the blow up.
This shows that $P_{\{x\},\{z\},\{t\}}$ is a singular point of the surface $S_\lambda$ of type $\mathbb{D}_4$.
\end{proof}

The surface $S_{\Bbbk}$ is singular at the points $P_{\{x\},\{y\},\{z,t\}}$, $P_{\{x\},\{z\},\{t\}}$, $P_{\{y\},\{t\},\{x,z\}}$,~$P_{\{y\},\{z\},\{x,t\}}$.
Their minimal resolutions are described in the proof of Lemma~\ref{lemma:r2-n4-singularities}.
This gives

\begin{corollary}
\label{corollary:r2-n4-singularities}
One has $\mathrm{rk}\,\mathrm{Pic}(\widetilde{S}_{\Bbbk})=\mathrm{rk}\,\mathrm{Pic}(S_{\Bbbk})+13$.
\end{corollary}

The base locus of the pencil $\mathcal{S}$ consists of the curves
$L_{\{x\},\{y\}}$, $L_{\{y\},\{z\}}$, $L_{\{x\},\{z,t\}}$, $L_{\{z\},\{x,t\}}$, $L_{\{t\},\{x,z\}}$,
$L_{\{y\},\{x,z,t\}}$, $L_{\{x\},\{y,z,t\}}$, $L_{\{z\},\{x,y,t\}}$, $L_{\{t\},\{x,y,z\}}$, and $\mathcal{C}$.
To describe the rank of their intersection matrix on the surface $S_\lambda$ for $\lambda\ne -7$ and $\lambda\ne\infty$,
it is enough to compute the rank of the intersection matrix of the curves
$L_{\{x\},\{y\}}$, $L_{\{y\},\{z\}}$, $L_{\{x\},\{y,z,t\}}$, $L_{\{z\},\{x,y,t\}}$, $L_{\{t\},\{x,y,z\}}$
$L_{\{t\},\{x,z\}}$, and $H_\lambda$, because
\begin{multline*}
H_\lambda\sim L_{\{x\},\{y\}}+2L_{\{x\},\{z,t\}}+L_{\{x\},\{y,z,t\}}\sim L_{\{x\},\{y\}}+L_{\{y\},\{z\}}+2L_{\{y\},\{x,z,t\}}\sim\\
\sim L_{\{y\},\{z\}}+2L_{\{z\},\{x,t\}}+L_{\{z\},\{x,y,t\}}\sim L_{\{t\},\{x,z\}}+L_{\{t\},\{x,y,z\}}+\mathcal{C}.
\end{multline*}
Moreover, if $\lambda\ne -7$, then
$$
H_{\{x,y,z,t\}}\cdot S_\lambda=L_{\{y\},\{x,z,t\}}+L_{\{x\},\{y,z,t\}}+L_{\{z\},\{x,y,t\}}+L_{\{t\},\{x,y,z\}},
$$
so that $H_\lambda\sim L_{\{y\},\{x,z,t\}}+L_{\{x\},\{y,z,t\}}+L_{\{z\},\{x,y,t\}}+L_{\{t\},\{x,y,z\}}$.
Thus, if $\lambda\ne -7$, then the rank of the intersection matrix
of the curves $L_{\{x\},\{y\}}$, $L_{\{y\},\{z\}}$, $L_{\{x\},\{z,t\}}$, $L_{\{z\},\{x,t\}}$, $L_{\{t\},\{x,z\}}$,
$L_{\{y\},\{x,z,t\}}$, $L_{\{x\},\{y,z,t\}}$, $L_{\{z\},\{x,y,t\}}$, $L_{\{t\},\{x,y,z\}}$, and $\mathcal{C}$
on the surface $S_\lambda$ is the same as the rank of the intersection matrix of the curves
$L_{\{x\},\{y\}}$, $L_{\{y\},\{z\}}$, $L_{\{x\},\{y,z,t\}}$, $L_{\{z\},\{x,y,t\}}$, $L_{\{t\},\{x,z\}}$, and $H_\lambda$.
Moreover, we have the following.

\begin{lemma}
\label{lemma:r2-n4-intersection}
Suppose that $\lambda\ne -7$.
Then the intersection matrix of the curves
$L_{\{x\},\{y\}}$, $L_{\{y\},\{z\}}$, $L_{\{x\},\{y,z,t\}}$, $L_{\{z\},\{x,y,t\}}$,
$L_{\{t\},\{x,z\}}$, and $H_\lambda$ on the surface $S_\lambda$ is given by
\begin{center}
\renewcommand\arraystretch{1.42}
\begin{tabular}{|c||c|c|c|c|c|c|}
\hline
 $\bullet$  & $L_{\{x\},\{y\}}$ & $L_{\{y\},\{z\}}$ & $L_{\{x\},\{y,z,t\}}$ & $L_{\{z\},\{x,y,t\}}$ & $L_{\{t\},\{x,z\}}$ & $H_\lambda$ \\
\hline\hline
 $L_{\{x\},\{y\}}$ &  $-\frac{4}{5}$ & $1$ & $\frac{3}{5}$ & $0$ & $0$ & $1$ \\
\hline
 $L_{\{y\},\{z\}}$ &  $1$ & $-\frac{4}{5}$ & $0$ & $\frac{3}{5}$ & $0$ & $1$ \\
\hline
 $L_{\{x\},\{y,z,t\}}$ &  $\frac{3}{5}$ & $0$ & $-\frac{6}{5}$ & $1$ & $0$ & $1$ \\
\hline
 $L_{\{z\},\{x,y,t\}}$ &  $0$ & $\frac{3}{5}$ & $1$ & $-\frac{6}{5}$ & $0$ & $1$ \\
\hline
 $L_{\{t\},\{x,z\}}$ &  $0$ & $0$ & $0$ & $0$ & $-\frac{1}{2}$ & $1$ \\
\hline
 $H_\lambda$  & $1$ & $1$ & $1$ & $1$ & $1$ & $4$ \\
\hline
\end{tabular}
\end{center}
\end{lemma}

\begin{proof}
By definition, we have $H_\lambda^2=4$ and
$$
H_\lambda\cdot L_{\{x\},\{y\}}=H_\lambda\cdot L_{\{y\},\{z\}}=H_\lambda\cdot L_{\{x\},\{y,z,t\}}=H_\lambda\cdot L_{\{z\},\{x,y,t\}}=H_\lambda\cdot L_{\{t\},\{x,z\}}=1.
$$

Let us compute $L_{\{x\},\{y\}}^2$.
The only singular point of the surface $S_\lambda$ contained in $L_{\{x\},\{y\}}$ is the point $P_{\{x\},\{y\},\{z,t\}}$.
Moreover, the surface $S_\lambda$ has du Val singularity of type $\mathbb{A}_4$ at this point by Lemma~\ref{lemma:r2-n4-singularities}.
Let us use the notation of Remark~\ref{remark:transversal} with $S=S_\lambda$, $O=P_{\{x\},\{y\},\{z,t\}}$, $n=4$, and $C=L_{\{x\},\{y\}}$.
Then $\overline{C}$ passes through the point $\overline{G}_1\cap\overline{G}_4$,
so that $\widetilde{C}\cap G_2\ne\varnothing$ or $\widetilde{C}\cap G_3\ne\varnothing$.
In both cases, we get $L_{\{x\},\{y\}}^2=-\frac{4}{5}$ by Proposition~\ref{proposition:du-Val-self-intersection}.

Likewise, using Remark~\ref{remark:transversal} with $S=S_\lambda$, $O=P_{\{x\},\{y\},\{z,t\}}$, $n=4$, and $C=L_{\{x\},\{y,z,t\}}$,
we see that $\overline{C}$ does not pass through the point $\overline{G}_1\cap\overline{G}_4$, so that
$L_{\{x\},\{y,z,t\}}^2=-\frac{6}{5}$ by Proposition~\ref{proposition:du-Val-self-intersection}.
Keeping  in mind the symmetry $x\leftrightarrow z$, we see that
$L_{\{y\},\{z\}}^2=L_{\{x\},\{y\}}^2=-\frac{4}{5}$, and $L_{\{z\},\{x,y,t\}}^2=L_{\{x\},\{y,z,t\}}^2=-\frac{6}{5}$.
Using Proposition~\ref{proposition:du-Val-self-intersection} again, we see that
$L_{\{t\},\{x,z\}}^2=-\frac{1}{2}$, because $P_{\{x\},\{z\},\{t\}}$ and $P_{\{y\},\{t\},\{x,z\}}$ are the only singular points of the surface $S_\lambda$
that are contained in the curve $L_{\{t\},\{x,z\}}$.

Observe that $L_{\{x\},\{y\}}\cap L_{\{y\},\{z\}}=P_{\{x\},\{y\},\{z\}}$, which is a smooth point of the surface $S_\lambda$.
This gives $L_{\{x\},\{y\}}\cdot L_{\{y\},\{z\}}=1$.
We also have
$$
L_{\{x\},\{y\}}\cdot L_{\{z\},\{x,y,t\}}=L_{\{x\},\{y\}}\cdot L_{\{t\},\{x,z\}}=0,
$$
because $L_{\{x\},\{y\}}\cap L_{\{z\},\{x,y,t\}}=L_{\{x\},\{y\}}\cap L_{\{t\},\{x,z\}}=\varnothing$.

To compute $L_{\{x\},\{y\}}\cdot L_{\{x\},\{y,z,t\}}$,
recall that $L_{\{x\},\{y\}}+2L_{\{x\},\{z,t\}}+L_{\{x\},\{y,z,t\}}\sim H_{\lambda}$.
Then
\begin{multline*}
L_{\{x\},\{y\}}\cdot L_{\{x\},\{y,z,t\}}+2L_{\{x\},\{z,t\}}\cdot L_{\{x\},\{y,z,t\}}-\frac{6}{5}=\\
=L_{\{x\},\{y\}}\cdot L_{\{x\},\{y,z,t\}}+2L_{\{x\},\{z,t\}}\cdot L_{\{x\},\{y,z,t\}}+L_{\{x\},\{y,z,t\}}^2=H_{\lambda}\cdot L_{\{x\},\{y,z,t\}}=1.
\end{multline*}
Using Remark~\ref{remark:transversal} with $S=S_\lambda$, $O=P_{\{x\},\{y\},\{z,t\}}$, $n=4$, $C=L_{\{x\},\{z,t\}}$, $Z=L_{\{x\},\{y,z,t\}}$,
we see that neither $\overline{C}$ nor $\overline{Z}$ contains the point $\overline{G}_1\cap\overline{G}_4$,
and either $\overline{C}\cap\overline{G}_1\ne\varnothing\ne\overline{Z}\cap\overline{G}_1$ or $\overline{C}\cap\overline{G}_4\ne\varnothing\ne\overline{Z}\cap\overline{G}_4$.
In both cases, we have $L_{\{x\},\{z,t\}}\cdot L_{\{x\},\{y,z,t\}}=\frac{4}{5}$ by Proposition~\ref{proposition:du-Val-self-intersection},
which implies that $L_{\{x\},\{y\}}\cdot L_{\{x\},\{y,z,t\}}=\frac{3}{5}$.

Using the symmetry $x\leftrightarrow z$, we see that
$L_{\{y\},\{z\}}\cdot L_{\{z\},\{x,y,t\}}=L_{\{x\},\{y\}}\cdot L_{\{x\},\{y,z,t\}}=\frac{3}{5}$.
Since $L_{\{y\},\{z\}}\cap L_{\{x\},\{y,z,t\}}=\varnothing$, and $L_{\{y\},\{z\}}\cap L_{\{t\},\{x,z\}}=\varnothing$,
we have $L_{\{y\},\{z\}}\cdot L_{\{x\},\{y,z,t\}}=0$ and $L_{\{y\},\{z\}}\cdot L_{\{t\},\{x,z\}}=0$, respectively.

Note that $L_{\{x\},\{y,z,t\}}\cap L_{\{z\},\{x,y,t\}}=P_{\{x\},\{z\},\{y,t\}}$,
and $P_{\{x\},\{z\},\{y,t\}}$ is a smooth point of the surface $S_\lambda$.
This shows that $L_{\{x\},\{y,z,t\}}\cdot L_{\{z\},\{x,y,t\}}=1$.
Since $L_{\{x\},\{y,z,t\}}\cap L_{\{t\},\{x,z\}}=\varnothing$, we have $L_{\{x\},\{y,z,t\}}\cdot L_{\{t\},\{x,z\}}=0$.
Likewise, we have $L_{\{z\},\{x,y,t\}}\cdot L_{\{t\},\{x,z\}}=0$.
\end{proof}

The rank of the matrix in Lemma~\ref{lemma:r2-n4-intersection} is $5$, so that \eqref{equation:main-2-simple} holds by Corollary~\ref{corollary:r2-n4-singularities}.
Thus, we see that \eqref{equation:main-2} in Main Theorem holds in this case.

Using Lemma~\ref{lemma:r2-n4-singularities} and Corollary~\ref{corollary:irreducible-fibers},
we see that $[\mathsf{f}^{-1}(\lambda)]=1$ for every $\lambda\not\in\{\infty,-7\}$.
Moreover, we have the following.

\begin{lemma}
\label{lemma:r2-n4-irreducible-special}
One has $[\mathsf{f}^{-1}(-7)]=11$.
\end{lemma}

\begin{proof}
Let $C_1=L_{\{x\},\{y\}}$, $C_2=L_{\{y\},\{z\}}$, $C_3=L_{\{x\},\{z,t\}}$, $C_4=L_{\{z\},\{x,t\}}$, $C_5=L_{\{t\},\{x,z\}}$,
$C_6=L_{\{y\},\{x,z,t\}}$, $C_7=L_{\{x\},\{y,z,t\}}$, $C_8=L_{\{z\},\{x,y,t\}}$, \mbox{$C_9=L_{\{t\},\{x,y,z\}}$}, and $C_{10}=\mathcal{C}$.
Then
$$
\mathbf{M}_1^{-7}=\mathbf{M}_2^{-7}=\mathbf{M}_5^{-7}=\mathbf{M}_7^{-7}=\mathbf{M}_8^{-7}=\mathbf{M}_9^{-7}=\mathbf{M}_{10}^{-7}=1
$$
and $\mathbf{M}_3^{-7}=\mathbf{M}_4^{-7}=\mathbf{M}_6^{-7}=2$.
On the other hand, we have
$$
\mathbf{m}_1^{-7}=\mathbf{m}_2^{-7}=\mathbf{m}_3^{-7}=\mathbf{m}_4^{-7}=\mathbf{m}_6^{-7}=2,
$$
and $\mathbf{m}_5^{-7}=\mathbf{m}_7^{-7}=\mathbf{m}_8^{-7}=\mathbf{m}_9^{-7}=\mathbf{m}_{10}^{-7}=1$.
Using Lemma~\ref{lemma:main} and \eqref{equation:equation:number-of-irredubicle-components-refined},
we see that
$$
\big[\mathsf{f}^{-1}(-7)\big]=6+\mathbf{D}_{P_{\{x\},\{y\},\{z,t\}}}^{-7}+\mathbf{D}_{P_{\{x\},\{z\},\{t\}}}^{-7}+\mathbf{D}_{P_{\{y\},\{t\},\{x,z\}}}^{-7}+\mathbf{D}_{P_{\{y\},\{z\},\{x,t\}}}^{-7}.
$$

It follows from the proof of Lemma~\ref{lemma:r2-n4-singularities} that
the surface $S_{-7}$ has an isolated ordinary double singularity at the point
$P_{\{y\},\{t\},\{x,z\}}$.
Thus, it follows from Lemma~\ref{lemma:normal-crossing} that its {defect} is zero, so that  $\mathbf{D}_{P_{\{y\},\{t\},\{x,z\}}}^{-7}=0$.
Hence, we conclude that
$$
\big[\mathsf{f}^{-1}(-7)\big]=6+\mathbf{D}_{P_{\{x\},\{y\},\{z,t\}}}^{-7}+\mathbf{D}_{P_{\{y\},\{z\},\{x,t\}}}^{-7}+\mathbf{D}_{P_{\{x\},\{z\},\{t\}}}^{-7}.
$$

The numbers $\mathbf{D}_{P_{\{x\},\{y\},\{z,t\}}}^{-7}$, $\mathbf{D}_{P_{\{y\},\{z\},\{x,t\}}}^{-7}$, and $\mathbf{D}_{P_{\{x\},\{z\},\{t\}}}$
can be computed using algorithm described in Section~\ref{subsection:scheme-step-8}.
To use it, we have to know the structure of the birational morphism $\alpha$ in \eqref{equation:main-diagram}.
Implicitly, it has been described in the proof of Lemma~\ref{lemma:r2-n4-singularities}.
To be precise, we proved that there exists a commutative diagram
$$
\xymatrix{
&&U_3\ar@{->}[lld]_{\alpha_3}&&U_4\ar@{->}[ll]_{\alpha_4}&&U_5\ar@{->}[ll]_{\alpha_5}&\\
U_2\ar@{->}[drr]_{\alpha_2}&&&&&&&&U_6\ar@{->}[llu]_{\alpha_6}\\
&&U_1\ar@{->}[rr]_{\alpha_1}&&\mathbb{P}^3&&U\ar@{->}[ll]^\alpha\ar@{->}[rru]_{\gamma}&& }
$$
Here $\alpha_1$ is the blow up of the point $P_{\{y\},\{t\},\{x,z\}}$,
the morphism $\alpha_2$ is the blow up of the preimage of the point $P_{\{y\},\{z\},\{x,t\}}$,
the morphism $\alpha_3$ is the blow up of a point in~$\mathbf{E}_2$,
the morphism $\alpha_4$ is the blow up of the preimage of the point $P_{\{x\},\{y\},\{z,t\}}$,
the morphism $\alpha_5$ is the blow up of a point in $\mathbf{E}_4$,
the morphism $\alpha_6$ is the blow up of the preimage of the point $P_{\{x\},\{z\},\{t\}}$,
and $\gamma$ is the blow ups of three distinct points in $\mathbf{E}_6$,
which are described in the very end of the proof of Lemma~\ref{lemma:r2-n4-singularities}.
In the notation used in the proof of Lemma~\ref{lemma:r2-n4-singularities},
these are the points $(\tilde{x}_6,\tilde{z}_6,\tilde{t}_6)=(0,0,0)$, $(\tilde{x}_6,\tilde{z}_6,\tilde{t}_6)=(0,-1,0)$
and $(\tilde{x}_6^\prime,\tilde{z}_6^\prime,\tilde{t}_6^\prime)=(0,0,0)$.

Using Lemma~\ref{lemma:main-2} and \eqref{equation:D-A-B}, we can find
$\mathbf{D}_{P_{\{x\},\{y\},\{z,t\}}}^{-7}$, $\mathbf{D}_{P_{\{y\},\{z\},\{x,t\}}}^{-7}$, and $\mathbf{D}_{P_{\{x\},\{z\},\{t\}}}^{-7}$
by analyzing the base curves of the pencil $\widehat{\mathcal{S}}$.
Implicitly, this has been already done in the proof of Lemma~\ref{lemma:r2-n4-singularities},
so that we will use the notation introduced in this proof.

Observe that $\widehat{E}_1$ does not contain base curves of the pencil $\widehat{\mathcal{S}}$.
To describe the base curves in the surface $\mathbf{E}_2$,
note that $\mathcal{S}^2\vert_{\mathbf{E}_2}$ consists of two lines in $\mathbf{E}_2\cong\mathbb{P}^2$.
These curves are given by  $\bar{x}_2=\bar{z}_2=0$ and $\bar{x}_2=\bar{y}_2=0$.
Denote them by $C_{11}^2$ and $C_{12}^2$, respectively.
Note that
$D_{-7}^2=S_{-7}^2+\mathbf{E}_2$,
the surface $S_{-7}^2$ contains $C_{11}^2$, and it does not contain $C_{12}^2$.

Similarly, the restriction $\mathcal{S}^3\vert_{\mathbf{E}_3}$ contains one base curve,
which is a line in $\mathbf{E}_3\cong\mathbb{P}^2$.
Denote this curve by $C_{13}^3$.
Then $C_{13}^3$ is contained in $S_{-7}^3$,
and it is not contained in $\mathbf{E}_2^3$.
Moreover, the surface $S_{-7}^3$ is smooth at general point of the curve $C_{13}^3$.

The restriction $\mathcal{S}^4\vert_{\mathbf{E}_4}$ consists of two lines in $\mathbf{E}_4\cong\mathbb{P}^2$,
which we denote by $C_{14}^4$ and~$C_{15}^4$.
One of them is contained in the surface $S_{-7}^4$.
We may assume that this curve is $C_{14}^4$.
Similarly, the restriction $\mathcal{S}^5\vert_{\mathbf{E}_5}$ contains one base curve,
which is a line in $\mathbf{E}_5\cong\mathbb{P}^2$.
Let~us~denote this curve by $C_{16}^5$. It is contained in $S_{-7}^5$,
and it is not contained in $\mathbf{E}_4^5$.
By~construction, we have $D_{-7}^5=S_{-7}^5+\mathbf{E}_2^5+\mathbf{E}_4^5$.

The restriction $\mathcal{S}^6\vert_{\mathbf{E}_6}$ consists of a single line in $\mathbf{E}_6\cong\mathbb{P}^2$ (taken with multiplicity $2$).
Denote this line by $C_{17}^6$.
Note that the surface $S_{-7}^6$ is singular along the curve $C_{17}^6$.

Finally, we observe that $\widehat{E}_7$, $\widehat{E}_8$
and $\widehat{E}_9$ does not contain base curves of the pencil $\widehat{\mathcal{S}}$.

Now we are ready to compute $\mathbf{D}_{P_{\{x\},\{y\},\{z,t\}}}^{-7}$, $\mathbf{D}_{P_{\{y\},\{z\},\{x,t\}}}^{-7}$, $\mathbf{D}_{P_{\{x\},\{z\},\{t\}}}^{-7}$.
First, we observe that
$\widehat{D}_{-7}=\widehat{S}_{-7}+\widehat{E}_{2}+\widehat{E}_4$,
so that $\mathbf{A}_{P_{\{x\},\{y\},\{z,t\}}}^{-7}=\mathbf{A}_{P_{\{y\},\{z\},\{x,t\}}}^{-7}=1$ and $\mathbf{A}_{P_{\{x\},\{z\},\{t\}}}^{-7}=0$.
Second, we observe that the curves $\widehat{C}_{11}$, $\widehat{C}_{12}$, $\widehat{C}_{13}$, $\widehat{C}_{14}$,
$\widehat{C}_{15}$, $\widehat{C}_{16}$, and $\widehat{C}_{17}$ are all base curves of the pencil $\widehat{\mathcal{S}}$
that are contained in $\alpha$-exceptional divisors.
The curves $\widehat{C}_{11}$, $\widehat{C}_{12}$, and $\widehat{C}_{13}$ are mapped to the point $P_{\{x\},\{y\},\{z,t\}}$,
the curves  $\widehat{C}_{14}$, $\widehat{C}_{15}$, and $\widehat{C}_{16}$ are mapped to the point $P_{\{y\},\{z\},\{x,t\}}$
and the curve $\widehat{C}_{17}$ is mapped to the point  $P_{\{x\},\{z\},\{t\}}$.
Thus, to find $\mathbf{D}_{P_{\{x\},\{y\},\{z,t\}}}^{-7}$,
$\mathbf{D}_{P_{\{y\},\{z\},\{x,t\}}}^{-7}$, and
$\mathbf{D}_{P_{\{x\},\{z\},\{t\}}}^{-7}$,
we have to compute the numbers $\mathbf{C}_{11}^{-7}$, $\mathbf{C}_{12}^{-7}$, $\mathbf{C}_{13}^{-7}$,
$\mathbf{C}_{14}^{-7}$, $\mathbf{C}_{15}^{-7}$, $\mathbf{C}_{16}^{-7}$, and $\mathbf{C}_{17}^{-7}$ defined in \eqref{equation:C}.
This can be done using Lemma~\ref{lemma:main-2}.

Observe that  $\mathbf{M}_{12}^{-7}=\mathbf{M}_{13}^{-7}=\mathbf{M}_{15}^{-7}=\mathbf{M}_{16}^{-7}=1$
and $\mathbf{M}_{11}^{-7}=\mathbf{M}_{14}^{-7}=\mathbf{M}_{17}^{-7}=2$.
Let us find the numbers
$\mathbf{m}_{11}$, $\mathbf{m}_{12}$, $\mathbf{m}_{13}$, $\mathbf{m}_{14}$, $\mathbf{m}_{15}$, $\mathbf{m}_{16}$, and $\mathbf{m}_{17}$.

Among base curves of the pencil $\mathcal{S}$, only $C_2$, $C_4$, $C_6$, $C_8$ contain the point $P_{\{y\},\{z\},\{x,t\}}$.
This shows that
\begin{multline*}
7=\mathrm{mult}_{P_{\{y\},\{z\},\{x,t\}}}\Big(2C_2+2C_4+2C_6+C_8\Big)=\mathrm{mult}_{P_{\{y\},\{z\},\{x,t\}}}\Big(S_{\lambda_1}\cdot S_{\lambda_2}\Big)=\\
=\mathrm{mult}_{P_{\{y\},\{z\},\{x,t\}}}\big(S_{\lambda_1}\big)\mathrm{mult}_{P_{\{y\},\{z\},\{x,t\}}}\big(S_{\lambda_2}\big)+\mathbf{m}_{11}+\mathbf{m}_{12}=4+\mathbf{m}_{11}+\mathbf{m}_{12}.
\end{multline*}
Moreover, we have $\mathbf{m}_{11}\geqslant 2$, because $\widehat{D}_{-7}$ is singular along $\widehat{C}_{11}$.
This shows that $\mathbf{m}_{11}=2$ and $\mathbf{m}_{12}=1$.
Similarly, we see that $\mathbf{m}_{13}=1$.
Using symmetry $x\leftrightarrow z$, we deduce that $\mathbf{m}_{14}=2$ and $\mathbf{m}_{15}=\mathbf{m}_{16}=1$.
To find $\mathbf{m}_{17}$, we observe that $C_3$, $C_4$, $C_5$, and $C_{10}$ are the only base curves of the pencil $\mathcal{S}$ that
contain the point $P_{\{x\},\{z\},\{t\}}$.
This shows that
$$
\mathrm{mult}_{P_{\{x\},\{z\},\{t\}}}\Big(S_{\lambda_1}\cdot S_{\lambda_2}\Big)=\mathrm{mult}_{P_{\{x\},\{z\},\{t\}}}\Big(2C_3+2C_4+C_5+C_{10}\Big)=6,
$$
which implies that $\mathbf{m}_{17}=2$.

Recall from \eqref{equation:D-A-B} that
$$
\mathbf{D}_{P_{\{x\},\{y\},\{z,t\}}}^{-7}=\mathbf{A}_{P_{\{x\},\{y\},\{z,t\}}}^{-7}+\mathbf{C}_{11}^{-7}+\mathbf{C}_{12}^{-7}+\mathbf{C}_{13}^{-7}=1+\mathbf{C}_{11}^{-7}+\mathbf{C}_{12}^{-7}+\mathbf{C}_{13}^{-7},
$$
where each term $\mathbf{C}_i^{-7}$ is defined in \eqref{equation:C} and can be found using Lemma~\ref{lemma:main-2}.
This gives $\mathbf{D}_{P_{\{x\},\{y\},\{z,t\}}}^{-7}=\mathbf{C}_{11}^{-7}=2$.
Similarly, we see that $\mathbf{D}_{P_{\{y\},\{z\},\{y,t\}}}^{-7}=\mathbf{C}_{14}^{-7}=2$,
Likewise, we have $\mathbf{D}_{P_{\{x\},\{z\},\{t\}}}^{-7}=\mathbf{C}_{17}^{-7}=1$.
Thus, we see that
$$
\big[\mathsf{f}^{-1}(-7)\big]=6+\mathbf{D}_{P_{\{x\},\{y\},\{z,t\}}}^{-7}+\mathbf{D}_{P_{\{y\},\{z\},\{y,t\}}}^{-7}+\mathbf{D}_{P_{\{x\},\{z\},\{t\}}}^{-7}=11,
$$
which completes the proof of the lemma.
\end{proof}

Since $h^{1,2}(X)=10$, we see that  \eqref{equation:main-1} in Main Theorem also holds in this case.

\subsection{Family \textnumero $2.5$}
\label{section:r-2-n-5}

In this case, the threefold $X$ is a blow up of a smooth cubic threefold in $\mathbb{P}^4$ along a smooth plane cubic curve.
Note that $h^{1,2}(X)=6$.
A~toric Landau--Ginzburg model is given by Minkowski polynomial \textnumero $3452$, which is
\begin{multline*}
x+y+z+x^2y^{-1}z^{-1}+3xz^{-1}+3yz^{-1}+x^{-1}y^2z^{-1}+3xy^{-1}+3x^{-1}y+3y^{-1}z+\\
+3x^{-1}z+x^{-1}y^{-1}z^2+xy^{-1}z^{-1}+2z^{-1}+x^{-1}yz^{-1}+2y^{-1}+2x^{-1}+x^{-1}y^{-1}z.
\end{multline*}
The corresponding quartic pencil $\mathcal{S}$ is given by the equation
\begin{multline*}
x^2yz+y^2zx+z^2yx+x^3t+3x^2ty+3y^2tx+y^3t+3x^2tz+3y^2tz+3z^2tx+\\
+3z^2ty+z^3t+x^2t^2+2t^2yx+t^2y^2+2t^2zx+2t^2yz+t^2z^2=\lambda xyzt.
\end{multline*}
Observe that this equation is invariant with respect to any permutations of the coordinates $x$, $y$, and $z$.
To describe the base locus of the pencil $\mathcal{S}$, we observe that
\begin{equation*}
\begin{split}
H_{\{x\}}\cdot S_0&=L_{\{x\},\{t\}}+2L_{\{x\},\{y,z\}}+L_{\{x\},\{y,z,t\}},\\
H_{\{y\}}\cdot S_0&=L_{\{y\},\{t\}}+2L_{\{y\},\{x,z\}}+L_{\{y\},\{x,z,t\}},\\
H_{\{z\}}\cdot S_0&=L_{\{z\},\{t\}}+2L_{\{z\},\{x,y\}}+L_{\{z\},\{x,y,t\}},\\
H_{\{t\}}\cdot S_0&=L_{\{x\},\{t\}}+L_{\{y\},\{t\}}+L_{\{z\},\{t\}}+L_{\{t\},\{x,y,z\}}.
\end{split}
\end{equation*}

For every $\lambda\not\in\{-6,-7,\infty\}$, the surface $S_\lambda$ has isolated singularities, so that it is irreducible.
On the other hand, we have
$S_{-6}=H_{\{x,y,z\}}+\mathsf{S}$,
where $\mathsf{S}$ is a  cubic surface that is given by $t^2x+t^2y+t^2z+tx^2+2txy+2txz+ty^2+2tyz+tz^2+xyz=0$.
Likewise, we have
$S_{-7}=H_{\{x,y,z,t\}}+\mathbf{S}$,
where $\mathbf{S}$ is a cubic surface  that is given by $t(x+y+z)^2+xyz=0$.

One can show that $\mathsf{S}$ is smooth.
On the other hand, the surface $\mathbf{S}$ has a unique singular point $P_{\{x\},\{y\},\{z\}}$. The surface $\mathbf{S}$ has du Val singularity of type $\mathbb{D}_4$ at this point.
Observe also that
$H_{\{x,y,z\}}\cdot\mathsf{S}=L_{\{x\},\{y,z\}}+L_{\{y\},\{x,z\}}+L_{\{z\},\{x,y\}}$,
so that $S_{-6}$ is singular along the lines $L_{\{x\},\{y,z\}}$, $L_{\{y\},\{x,z\}}$, $L_{\{z\},\{x,y\}}$.
Note also that the intersection $H_{\{x,y,z,t\}}\cap\mathbf{S}$ is a smooth cubic curve,
which is not contained in the base locus of the pencil $\mathcal{S}$.

\begin{lemma}
\label{lemma:r2-n5-singularities}
Suppose that $\lambda\not\in\{-6,-7,\infty\}$.
Then singular points of the surface $S_\lambda$ contained in the base locus of the pencil $\mathcal{S}$ can be describes as follows:
\begin{itemize}\setlength{\itemindent}{3cm}
\item[$P_{\{x\},\{y\},\{z\}}$:] type $\mathbb{D}_4$ with quadratic term $(x+z+t)^2$;
\item[$P_{\{x\},\{t\},\{y,z\}}$:] type $\mathbb{A}_3$ with quadratic term $x(x+z+y-(\lambda+6)t)$;
\item[$P_{\{y\},\{t\},\{x,z\}}$:] type $\mathbb{A}_3$ with quadratic term $y(x+z+y-(\lambda+6)t)$;
\item[$P_{\{z\},\{t\},\{x,y\}}$:] type $\mathbb{A}_3$ with quadratic term $z(x+z+y-(\lambda+6)t)$.
\end{itemize}
\end{lemma}

\begin{proof}
First let us describe the singularity of the surface $S_\lambda$ at the point $P_{\{x\},\{y\},\{z\}}$.
In the chart $t=1$, the surface $S_\lambda$ is given by
$$
\hat{z}^2+\Big((\lambda+6)\hat{x}^2\hat{y}+(\lambda+6)\hat{x}\hat{y}^2-(\lambda+6)\hat{x}\hat{y}\hat{z}+\hat{z}^3\Big)+\Big(\hat{z}^2\hat{y}\hat{x}-\hat{x}^2\hat{y}\hat{z}-\hat{y}^2\hat{z}\hat{x}\Big)=0,
$$
where $\hat{x}=x$, $\hat{y}=y$, $\hat{z}=x+y+z$.
Introducing coordinates $\hat{x}_4=\hat{x}$, $\hat{y}_4=\frac{\hat{y}}{\hat{x}}$, $\hat{z}_4=\frac{\hat{z}}{\hat{x}}$,
we can rewrite this equation (after dividing by $\hat{x}_4^2$) as
$$
(\lambda+6)\hat{x}_4\hat{y}_4+\hat{z}_4^2+\Big((\lambda+6)\hat{x}_4\hat{y}_4^2-(6+\lambda)\hat{x}_4\hat{y}_4\hat{z}_4\Big)+\Big(\hat{z}_4^3\hat{x}_4-\hat{x}_4^2\hat{y}_4\hat{z}_4\Big)+\Big(\hat{x}_4^2\hat{y}_4\hat{z}_4^2-\hat{x}_4^2\hat{y}_4^2\hat{z}_4\Big)=0.
$$
This equation defines (a chart of) the blow up of the surface $S_\lambda$ at the point $P_{\{x\},\{y\},\{z\}}$.
The exceptional curve of the blow up is given by the equations $\hat{x}_4=\hat{z}_4=0$.
Observe that the point $(\hat{x}_4,\hat{y}_4,\hat{z}_4)=(0,0,0)$ is an ordinary double point of the obtained surface, because $\lambda\ne -6$.
The obtained surface is also singular at the point $(\hat{x}_4,\hat{y}_4,\hat{z}_4)=(0,-1,0)$.
This point is also an ordinary double point of this surface.
These are all singular points of the obtained surface at this chart of the blow up.
Keeping in mind the symmetry $x\leftrightarrow y$, we see that the exceptional curve of the blow up of the surface $S_\lambda$ at the point $P_{\{x\},\{y\},\{z\}}$ contains three ordinary double points of this surface.
This shows that $P_{\{x\},\{y\},\{z\}}$ is a singular point of type $\mathbb{D}_4$ of the surface $S_\lambda$.

To complete the proof, it is enough to show that $P_{\{y\},\{t\},\{x,z\}}$ is a singular point of the surface $S_\lambda$ of type $\mathbb{A}_3$,
because $S_\lambda$ is invariant with respect to the permutations of the coordinates $x$, $y$, and $z$.
In the chart $z=1$, the surface $S_\lambda$ is given by
$$
\bar{y}(\bar{x}+\bar{y}-(6+\lambda)\bar{t})=\bar{x}^2\bar{y}+\bar{x}\bar{y}^2-(6+\lambda)\bar{t}\bar{x}\bar{y}+\bar{t}^2\bar{x}^2+2\bar{t}^2\bar{x}\bar{y}+\bar{t}^2\bar{y}^2+\bar{t}\bar{x}^3+3\bar{t}\bar{x}^2\bar{y}+3\bar{t}\bar{x}\bar{y}^2+\bar{t}\bar{y}^3,
$$
where $\bar{x}=x+1$, $\bar{y}=y$, and $\bar{t}=t$.
Introducing new coordinates $\check{x}=\bar{x}+\bar{y}-(6-\lambda)\bar{t}$, $\check{y}=\bar{y}$, and $\check{t}=\bar{t}$,
we can rewrite this equation as
$$
(\lambda+7)(\lambda+6)^2\check{t}^4=\check{x}\check{y}-(\lambda+6)\big(\check{t}\check{x}\check{y}+(3\lambda+20)\check{t}^3\check{x}\big)-\check{x}^2\check{y}+\check{x}\check{y}^2-(3\lambda+19)\check{t}^2\check{x}^2-\check{t}\check{x}^3,
$$
where we grouped together monomials of the same quasihomogeneous degree
with respect to the weights $\mathrm{wt}(\check{x})=2$, $\mathrm{wt}(\check{y})=2$, and $\mathrm{wt}(\check{t})=1$.
This shows that the surface $S_\lambda$ has singularity of type $\mathbb{A}_3$ at the point $P_{\{y\},\{t\},\{x,z\}}$.
This complete the proof of the lemma.
\end{proof}

The proof of Lemma~\ref{lemma:r2-n5-singularities} implies that $\mathrm{rk}\,\mathrm{Pic}(\widetilde{S}_{\Bbbk})=\mathrm{rk}\,\mathrm{Pic}(S_{\Bbbk})+13$.

\begin{lemma}
\label{lemma:r2-n5-intersection}
Suppose that $\lambda\not\in\{-6,-7,\infty\}$.
Then the intersection matrix of the lines $L_{\{x\},\{t\}}$, $L_{\{y\},\{t\}}$, $L_{\{z\},\{t\}}$, $L_{\{x\},\{y,z,t\}}$, $L_{\{y\},\{x,z,t\}}$, $L_{\{z\},\{x,y,t\}}$, and $L_{\{t\},\{x,y,z\}}$
on the surface~$S_\lambda$ is given by
\begin{center}\renewcommand\arraystretch{1.42}
\begin{tabular}{|c||c|c|c|c|c|c|c|}
\hline
$\bullet$ & $L_{\{x\},\{t\}}$ & $L_{\{y\},\{t\}}$ & $L_{\{z\},\{t\}}$ & $L_{\{x\},\{y,z,t\}}$ & $L_{\{y\},\{x,z,t\}}$ & $L_{\{z\},\{x,y,t\}}$ & $L_{\{t\},\{x,y,z\}}$\\
\hline\hline
$L_{\{x\},\{t\}}$ & $-\frac{5}{4}$ &  $1$ &  $1$ &  $\frac{3}{4}$ &  $0$ &  $0$ &  $\frac{1}{4}$\\
\hline
$L_{\{y\},\{t\}}$ & $1$ &  $-\frac{5}{4}$ &  $1$ &  $0$ &  $\frac{3}{4}$ &  $0$ &  $\frac{1}{4}$\\
\hline
$L_{\{z\},\{t\}}$ & $1$ &  $1$ &  $-\frac{5}{4}$ &  $0$ &  $0$ &  $\frac{3}{4}$ &  $\frac{1}{4}$\\
\hline
$L_{\{x\},\{y,z,t\}}$ & $\frac{3}{4}$ &  $0$ &  $0$ &  $-\frac{5}{4}$ &  $1$ &  $1$ &  $\frac{1}{4}$\\
\hline
$L_{\{y\},\{x,z,t\}}$ & $0$ &  $\frac{3}{4}$ &  $0$ &  $1$ &  $-\frac{5}{4}$ &  $1$ &  $\frac{1}{4}$\\
\hline
$L_{\{z\},\{x,y,t\}}$ & $0$ &  $0$ &  $\frac{3}{4}$ &  $1$ &  $1$ &  $-\frac{5}{4}$ &  $\frac{1}{4}$\\
\hline
$L_{\{t\},\{x,y,z\}}$ & $\frac{1}{4}$ &  $\frac{1}{4}$ &  $\frac{1}{4}$ &  $\frac{1}{4}$ &  $\frac{1}{4}$ &  $\frac{1}{4}$ &  $\frac{1}{4}$\\
\hline
\end{tabular}
\end{center}
\end{lemma}

\begin{proof}
Keeping in mind that the equation of surface $S_\lambda$ is invariant with respect to the permutations of the coordinates $x$, $y$, and $z$,
it is enough to compute
$L_{\{x\},\{t\}}^2$, $L_{\{x\},\{y,z,t\}}^2$, $L_{\{t\},\{x,y,z\}}^2$,
$L_{\{x\},\{t\}}\cdot L_{\{y\},\{t\}}$, $L_{\{x\},\{t\}}\cdot L_{\{x\},\{y,z,t\}}$,
$L_{\{x\},\{t\}}\cdot L_{\{y\},\{x,z,t\}}$, $L_{\{x\},\{t\}}\cdot L_{\{t\},\{x,y,z\}}$,
and  $L_{\{x\},\{y,z,t\}}\cdot L_{\{t\},\{x,y,z\}}$.

Using Lemma~\ref{lemma:r2-n5-singularities},  Proposition~\ref{proposition:du-Val-self-intersection} and Remark~\ref{remark:transversal},
we see that
$L_{\{x\},\{t\}}^2=-\frac{5}{4}$,
because $P_{\{x\},\{t\},\{y,z\}}$ is the only singular point of the surface $S_\lambda$ that is contained in the line $L_{\{x\},\{t\}}$.
Likewise, we see that $L_{\{x\},\{y,z,t\}}^2=-\frac{5}{4}$.
Similarly, we get $L_{\{t\},\{x,y,z\}}^2=\frac{1}{4}$,
because the line $L_{\{t\},\{x,y,z\}}$ contains the points $P_{\{x\},\{t\},\{y,z\}}$,
$P_{\{y\},\{t\},\{x,z\}}$, and $P_{\{z\},\{t\},\{x,y\}}$.
We have $L_{\{x\},\{t\}}\cdot L_{\{y\},\{t\}}=1$, because
$L_{\{x\},\{t\}}\cap L_{\{y\},\{t\}}=P_{\{x\},\{y\},\{t\}}$,
which is a smooth point of the surface $S_\lambda$ by Lemma~\ref{lemma:r2-n5-singularities}.

Now applying Remark~\ref{remark:transversal} with $S=S_\lambda$, $O=P_{\{x\},\{t\},\{y,z\}}$, $n=3$, $C=L_{\{x\},\{t\}}$
and $Z=L_{\{x\},\{y,z,t\}}$, we see that $L_{\{x\},\{t\}}\cdot L_{\{x\},\{y,z,t\}}=\frac{3}{4}$ by Proposition~\ref{proposition:du-Val-intersection}.
Likewise, we have
$L_{\{x\},\{y,z,t\}}\cdot L_{\{t\},\{x,y,z\}}=L_{\{x\},\{y,z,t\}}\cdot L_{\{t\},\{x,y,z\}}=\frac{1}{4}$.
Finally, we have $L_{\{x\},\{t\}}\cdot L_{\{y\},\{x,z,t\}}=0$,
because $L_{\{x\},\{t\}}\cap L_{\{y\},\{x,z,t\}}=\varnothing$.
\end{proof}

If $\lambda\not\in\{-6,-7,\infty\}$, then the intersection matrix of the lines
$L_{\{x\},\{t\}}$, $L_{\{y\},\{t\}}$, $L_{\{z\},\{t\}}$, $L_{\{x\},\{y,z\}}$, $L_{\{y\},\{x,z\}}$, $L_{\{z\},\{x,y\}}$,
$L_{\{x\},\{y,z,t\}}$, $L_{\{y\},\{x,z,t\}}$, $L_{\{z\},\{x,y,t\}}$, and $L_{\{t\},\{x,y,z\}}$
on the surface $S_\lambda$ has the same rank as the intersection matrix of the curves
$L_{\{x\},\{t\}}$, $L_{\{y\},\{t\}}$, $L_{\{z\},\{t\}}$, $L_{\{x\},\{y,z,t\}}$, $L_{\{y\},\{x,z,t\}}$, $L_{\{z\},\{x,y,t\}}$, and $L_{\{t\},\{x,y,z\}}$.
This follows from
\begin{multline*}
H_\lambda\sim L_{\{x\},\{t\}}+2L_{\{x\},\{y,z\}}+L_{\{x\},\{y,z,t\}}\sim L_{\{y\},\{t\}}+2L_{\{y\},\{x,z\}}+L_{\{y\},\{x,z,t\}}\sim\\
\sim L_{\{z\},\{t\}}+2L_{\{z\},\{x,y\}}+L_{\{z\},\{x,y,t\}}\sim L_{\{x\},\{t\}}+L_{\{y\},\{t\}}+L_{\{z\},\{t\}}+L_{\{t\},\{x,y,z\}}.
\end{multline*}
The rank of the intersection matrix in Lemma~\ref{lemma:r2-n5-intersection} is $5$.
Thus, we see that \eqref{equation:main-2-simple} holds.
This proves \eqref{equation:main-2} in Main Theorem holds in this case.

To verify \eqref{equation:main-1} in Main Theorem,
observe that $[\mathsf{f}^{-1}(\lambda)]=1$ for every $\lambda\not\in\{-6,-7,\infty\}$.
This follows from Lemma~\ref{lemma:r2-n5-singularities} and  Corollary~\ref{corollary:irreducible-fibers}.
Moreover, we have

\begin{lemma}
\label{lemma:r2-n5-irreducible-special}
One has $[\mathsf{f}^{-1}(-7)]=2$ and $[\mathsf{f}^{-1}(-6)]=6$.
\end{lemma}

\begin{proof}
Let $C_1=L_{\{x\},\{t\}}$, $C_2=L_{\{y\},\{t\}}$, $C_3=L_{\{z\},\{t\}}$,
$C_4=L_{\{x\},\{y,z\}}$, $C_5=L_{\{y\},\{x,z\}}$, $C_6=L_{\{z\},\{x,y\}}$,
$C_7=L_{\{x\},\{y,z,t\}}$, $C_8=L_{\{y\},\{x,z,t\}}$, \mbox{$C_9=L_{\{z\},\{x,y,t\}}$}, and $C_{10}=L_{\{t\},\{x,y,z\}}$.
Then $\mathbf{m}_1=\mathbf{m}_2=\mathbf{m}_3=\mathbf{m}_4=\mathbf{m}_5=\mathbf{m}_6=2$
and $\mathbf{m}_7=\mathbf{m}_8=\mathbf{m}_9=\mathbf{m}_{10}=1$.

Recall that $S_{\infty}$ is singular along the curves $C_1$, $C_2$, and $C_3$,
and the surface $S_{-6}$ is singular along the curves $C_4$, $C_5$, and $C_6$.
Thus, we have $\mathbf{M}_4^{-6}=\mathbf{M}_5^{-6}=\mathbf{M}_6^{-6}=2$,
$$
\mathbf{M}_1^{-6}=\mathbf{M}_2^{-6}=\mathbf{M}_3^{-6}=\mathbf{M}_7^{-6}=\mathbf{M}_8^{-6}=\mathbf{M}_9^{-6}=\mathbf{M}_{10}^{-6}=1,
$$
and $\mathbf{M}_1^{-7}=\mathbf{M}_2^{-7}=\mathbf{M}_3^{-7}=\mathbf{M}_4^{-7}=\mathbf{M}_5^{-7}=\mathbf{M}_6^{-7}=\mathbf{M}_7^{-7}=\mathbf{M}_8^{-7}=\mathbf{M}_9^{-7}=\mathbf{M}_{10}^{-7}=1$.

The birational morphism $\alpha$ in \eqref{equation:main-diagram} is described in the proof of Lemma~\ref{lemma:r2-n5-singularities}.
Namely, it is given by the commutative diagram
$$
\xymatrix{
U_2\ar@{->}[d]_{\alpha_2}&&U_3\ar@{->}[ll]_{\alpha_3}&&U_4\ar@{->}[ll]_{\alpha_4}\\
U_1\ar@{->}[rr]_{\alpha_1}&&\mathbb{P}^3&&U\ar@{->}[ll]^\alpha\ar@{->}[u]_{\gamma}}
$$
Here $\alpha_1$ is the blow up of the point $P_{\{x\},\{t\},\{y,z\}}$,
the morphism $\alpha_2$ is the blow up of the preimage of the point $P_{\{y\},\{t\},\{x,z\}}$,
the morphism $\alpha_3$ is the blow up of the preimage of the point $P_{\{z\},\{t\},\{x,y\}}$,
the morphism $\alpha_4$ is the blow up of the preimage of the point $P_{\{x\},\{y\},\{z\}}$,
and $\gamma$ is the blow ups of three distinct points in $\mathbf{E}_4$.

If $\lambda\ne\infty$, then $\widehat{D}_\lambda=\widehat{S}_\lambda$.
This follows from the proof of Lemma~\ref{lemma:r2-n5-singularities}.
It should be pointed out that the surface $\widehat{S}_\lambda$ is singular for every $\lambda\in\mathbb{C}$.

The curves $\widehat{C}_1$, $\widehat{C}_2$, $\widehat{C}_3$, $\widehat{C}_4$, $\widehat{C}_5$,
$\widehat{C}_6$, $\widehat{C}_7$, $\widehat{C}_8$, $\widehat{C}_9$, and $\widehat{C}_{10}$ are base curves of the pencil $\widehat{\mathcal{S}}$.
Let us describe the remaining base curves of the pencil $\widehat{\mathcal{S}}$ using the data collected in the proof of Lemma~\ref{lemma:r2-n5-singularities}.

For every $\lambda\ne\infty$, the restriction $S_\lambda^2\vert_{\mathbf{E}_2}$ is given by
$$
\bar{y}(\bar{x}+\bar{y}-(6+\lambda)\bar{t})=0
$$
in the appropriate homogeneous coordinates $\bar{x}$, $\bar{y}$, and $\bar{t}$ on $\mathbf{E}_2\cong\mathbb{P}^2$.
This gives us the pencil of conics in $\mathbf{E}_2$ that has a unique base curve, which is given by $\bar{y}=0$.
Thus, the restriction $\mathcal{S}^2\vert_{\mathbf{E}_2}$ has one base curve.
This gives us the base curve of the pencil $\widehat{\mathcal{S}}$ that is contained in $\widehat{E}_2$.
Let us denote it by $\widehat{C}_{12}$.
Similarly, we see that one base curve of the pencil $\widehat{\mathcal{S}}$ is contained in the surface $\widehat{E}_1$,
and one base curve of the pencil $\widehat{\mathcal{S}}$ is contained in the surface $\widehat{E}_3$.
Let us denote them by $\widehat{C}_{11}$ and $\widehat{C}_{13}$, respectively.

The restriction $\mathcal{S}^4\vert_{\mathbf{E}_4}$ consists of one line (taken with multiplicity two).
This gives us one base curve of the pencil $\mathcal{S}^4$ that is contained in the surface $\mathbf{E}_4$.
Denote it by $C_{14}^4$.
Observe that the surface $S_{-6}^4$ is singular at general point of this curve.
Moreover, it follows from the proof of Lemma~\ref{lemma:r2-n4-irreducible-special} that
the curves $\widehat{C}_{1}$, $\widehat{C}_{2}$, $\widehat{C}_{3}$, $\widehat{C}_{4}$,
$\widehat{C}_{5}$, $\widehat{C}_{6}$, $\widehat{C}_{7}$,  $\widehat{C}_{8}$,  $\widehat{C}_{9}$,  $\widehat{C}_{10}$,
$\widehat{C}_{11}$, $\widehat{C}_{12}$, $\widehat{C}_{13}$, and $\widehat{C}_{14}$ are all base curves of the pencil $\widehat{\mathcal{S}}$.

Let us compute $\mathbf{m}_{11}$, $\mathbf{m}_{12}$, $\mathbf{m}_{13}$, and $\mathbf{m}_{14}$.
Among base curves of the pencil $\mathcal{S}$,
only the curves $C_2$, $C_5$, $C_8$, $C_{10}$ contain the point $P_{\{y\},\{t\},\{x,z\}}$,
This gives
\begin{multline*}
6=\mathrm{mult}_{P_{\{y\},\{t\},\{x,z\}}}\Big(2C_2+2C_5+C_8+C_{10}\Big)=\mathrm{mult}_{P_{\{y\},\{t\},\{x,z\}}}\Big(S_{\lambda_1}\cdot S_{\lambda_2}\Big)=\\
=\mathrm{mult}_{P_{\{y\},\{t\},\{x,z\}}}\big(S_{\lambda_1}\big)\mathrm{mult}_{P_{\{y\},\{t\},\{x,z\}}}\big(S_{\lambda_2}\big)+\mathbf{m}_{11}=4+\mathbf{m}_{11},
\end{multline*}
so that $\mathbf{m}_{11}=2$. Similarly, we get $\mathbf{m}_{12}=\mathbf{m}_{13}=\mathbf{m}_{14}=2$.

Observe that $\mathbf{M}_{11}^{-7}=\mathbf{M}_{12}^{-7}=\mathbf{M}_{13}^{-7}=\mathbf{M}_{14}^{-7}=1$ and
$\widehat{D}_{-7}=\widehat{S}_{-7}$ in \eqref{equation:log-pull-back}.
Thus, it follows from Corollary~\ref{corollary:main-2} that $[\mathsf{f}^{-1}(-7)]=2$.

Likewise, we see that $\mathbf{M}_{11}^{-6}=\mathbf{M}_{12}^{-6}=\mathbf{M}_{13}^{-6}=1$, $\mathbf{M}_{14}^{-6}=2$, and $\widehat{D}_{-6}=\widehat{S}_{-6}$.
Therefore, it follows from \eqref{equation:equation:number-of-irredubicle-components-refined-2} and Lemma~\ref{lemma:main-2} that
$[\mathsf{f}^{-1}(-6)]=6$.
\end{proof}

Since $h^{1,2}(X)=6$, we see that \eqref{equation:main-1} in Main Theorem holds in this case.

\subsection{Family \textnumero $2.6$}
\label{section:r-2-n-6}

In this case, the threefold $X$ is a divisor of bidegree $(2,2)$ in $\mathbb{P}^2\times\mathbb{P}^2$,
so that $h^{1,2}(X)=9$.
A toric Landau--Ginzburg model is given by
$$
x+y+{\frac {x}{z}}+{\frac {y}{z}}+{\frac {xz}{y}}+2z+{\frac {yz}{x}}+\frac{2x}{y}+{\frac {2y}{x}}+{\frac {x}{yz}}+\frac{2}{z}+{\frac {y}{xz}}+{\frac {{z}^{2}}{y}}+{\frac {{z}^{2}}{x}}+{\frac{3z}{y}}+{\frac {3z}{x}}+\frac{3}{y}+\frac{3}{x}+\frac{1}{yz}+\frac{1}{xz},
$$
which is Minkowski polynomial \textnumero $3873.2$.
The pencil $\mathcal{S}$ is given by
\begin{multline*}
x^2zy+y^2zx+x^2ty+y^2tx+x^2z^2+2z^2yx+y^2z^2+2x^2tz+2y^2tz+x^2t^2+\\
+2t^2yx+t^2y^2+z^3x+z^3y+3z^2tx+3z^2ty+3t^2zx+3t^2zy+t^3x+t^3y=\lambda xyzt.
\end{multline*}
This equation is invariant with respect to the swaps $x\leftrightarrow y$ and $z\leftrightarrow t$.

To describe the base locus of the pencil $\mathcal{S}$, we observe that
\begin{itemize}
\item $H_{\{x\}}\cdot S_0=L_{\{x\},\{y\}}+2L_{\{x\},\{z,t\}}+L_{\{x\},\{y,z,t\}}$,
\item $H_{\{y\}}\cdot S_0=L_{\{x\},\{y\}}+2L_{\{y\},\{z,t\}}+L_{\{y\},\{x,z,t\}}$,
\item $H_{\{z\}}\cdot S_0=L_{\{z\},\{t\}}+L_{\{z\},\{x,y\}}+L_{\{z\},\{y,t\}}+L_{\{z\},\{x,t\}}$,
\item $H_{\{t\}}\cdot S_0=L_{\{z\},\{t\}}++L_{\{t\},\{x,y\}}+L_{\{t\},\{y,z\}}+L_{\{t\},\{x,z\}}$.
\end{itemize}

We let $C_1=L_{\{x\},\{y\}}$,~\mbox{$C_2=L_{\{z\},\{t\}}$},~\mbox{$C_3=L_{\{x\},\{z,t\}}$},~\mbox{$C_4=L_{\{y\},\{z,t\}}$},~\mbox{$C_5=L_{\{x\},\{y,z,t\}}$},
\mbox{$C_6=L_{\{y\},\{x,z,t\}}$},~\mbox{$C_7=L_{\{z\},\{x,y\}}$},~\mbox{$C_8=L_{\{z\},\{y,t\}}$},~\mbox{$C_{9}=L_{\{z\},\{x,t\}}$},~\mbox{$C_{10}=L_{\{t\},\{x,y\}}$},~\mbox{$C_{11}=L_{\{t\},\{y,z\}}$}, and $C_{12}=L_{\{t\},\{x,z\}}$.
Then
$\mathbf{m}_{5}=\mathbf{m}_{6}=\mathbf{m}_{7}=\mathbf{m}_{8}=\mathbf{m}_{9}=\mathbf{m}_{10}=\mathbf{m}_{11}=\mathbf{m}_{12}=1$
and $\mathbf{m}_{1}=\mathbf{m}_{2}=\mathbf{m}_{3}=\mathbf{m}_{4}=2$.
Likewise, we have
$$
\mathbf{M}_{1}^{-4}=\mathbf{M}_{2}^{-4}=\mathbf{M}_{5}^{-4}=\mathbf{M}_{6}^{-4}=\mathbf{M}_{7}^{-4}=\mathbf{M}_{8}^{-4}=\mathbf{M}_{9}^{-4}=\mathbf{M}_{10}^{-4}=\mathbf{M}_{11}^{-4}=\mathbf{M}_{12}^{-4}=1
$$
and $\mathbf{M}_{3}^{-4}=\mathbf{M}_{4}^{-4}=2$, so that $S_{-4}$ is singular along the lines $L_{\{x\},\{z,t\}}$ and $L_{\{y\},\{z,t\}}$.

For every $\lambda\not\in\{-4,\infty\}$, the surface $S_\lambda$ has isolated singularities,
which implies, in particular, that $S_\lambda$ is irreducible.
One the other hand, the surface $S_{-4}$ is reducible:
$$
S_{-4}=H_{\{x,y\}}+H_{\{z,t\}}+H_{\{y,z,t\}}+H_{\{x,z,t\}},
$$

If $\lambda\not\in\{-4,\infty\}$, then
the singular points of the surface $S_\lambda$ contained in the base locus of the pencil $\mathcal{S}$
are the points $P_{\{x\},\{z\},\{t\}}$, $P_{\{y\},\{z\},\{t\}}$, $P_{\{x\},\{y\},\{z,t\}}$, and $P_{\{z\},\{t\},\{x,y\}}$.
These are the {fixed} singular points of the surfaces in $\mathcal{S}$.
Lets us describe their singularity types and explicitly construct the birational morphism $\alpha$ in \eqref{equation:main-diagram}.
We start with $P_{\{x\},\{z\},\{t\}}$.

In the chart $y=1$, the surface $S_\lambda$ is given by
\begin{multline*}
(z+t)(x+z+t)+\Big(t^3+2t^2x+3t^2z+x^2t+3z^2t+x^2z+2xz^2+z^3-\lambda txz\Big)+\\
+\Big(xt^3+x^2t^2+3t^2xz+2tx^2z+3txz^2+x^2z^2+z^3x\Big)=0.
\end{multline*}
For convenience, we rewrite the defining equation of the surface $S_\lambda$ as
$$
\hat{x}\hat{t}+\Big(4\hat{t}^2\hat{z}+\hat{t}\hat{x}^2-4\hat{t}\hat{x}\hat{z}-4\hat{t}\hat{z}^2+4\hat{x}\hat{z}^2+\lambda\hat{t}^2\hat{z}-\lambda\hat{t}\hat{x}\hat{z}-\lambda\hat{t}\hat{z}^2+\lambda\hat{x}\hat{z}^2\Big)+\Big(\hat{t}^2\hat{x}^2-\hat{t}^3\hat{x}\Big)=0,
$$
where $\hat{x}=x+z+t$, $\hat{z}=z$, and $\hat{t}=z+t$.

Let $\alpha_1\colon U_1\to\mathbb{P}^3$ be the blow up of the point $P_{\{x\},\{z\},\{t\}}$.
A chart of this blow up is given by the coordinate change $\hat{x}_1=\frac{\hat{x}}{\hat{z}}$, $\hat{z}_1=\hat{z}$, $\hat{t}_1=\frac{\hat{t}}{\hat{z}}$.
In this chart, the surface $D^1_\lambda$~is~given~by
$$
\hat{t}_1\hat{x}_1-(\lambda+4)\hat{t}_1\hat{z}_1+(\lambda+4)\hat{z}_1\hat{x}_1+(\lambda+4)\Big(\hat{t}_1^2\hat{z}_1-\hat{t}_1\hat{x}_1\hat{z}_1\Big)+\hat{t}_1\hat{x}_1^2\hat{z}_1+\Big(\hat{t}_1^2\hat{x}_1^2\hat{z}_1^2-\hat{t}_1^3\hat{x}_1\hat{z}_1^2\Big)=0,
$$
where $\hat{z}_1=0$ defines the surface $\mathbf{E}_1$.
Then $(\hat{x}_1,\hat{z}_1,\hat{t}_1)=(0,0,0)$ is the only singular point of the surface $S^1_\lambda$ that is contained in $\mathbf{E}_1$.
If $\lambda\not\in\{-4,\infty\}$, then this point is an ordinary double point of the surface $S_\lambda$.
Hence, if $\lambda\not\in\{-4,\infty\}$, then $P_{\{x\},\{z\},\{t\}}$ is a du Val singular point of the surface $S_\lambda$ of type $\mathbb{A}_3$.

Notice also that the pencil $\mathcal{S}^1$ has exactly two base curves contained in the surface $\mathbf{E}_1$.
Indeed, the restriction $\mathcal{S}^1\vert_{\mathbf{E}_1}$ consists of the curves $\{\hat{z}_1=\hat{x}_1=0\}$ and $\{\hat{z}_1=\hat{t}_1=0\}$.
Let~us denote these curves by $C_{13}^1$ and $C_{14}^1$, respectively.

Let $\alpha_2\colon U_2\to U_1$ be the blow up of the point $(\hat{x}_1,\hat{z}_1,\hat{t}_1)=(0,0,0)$.
Then
$D_\lambda^2=S^2_\lambda$ for every $\lambda\in\mathbb{C}$.
Moreover, the restriction $\mathcal{S}^2\vert_{\mathbf{E}_2}$ is a pencil of conics in $\mathbf{E}_2\cong\mathbb{P}^2$
that is given by the equation
$$
\hat{t}_1\hat{x}_1-(\lambda+4)\hat{t}_1\hat{z}_1+(\lambda+4)\hat{z}_1\hat{x}_1=0,
$$
where we consider $\hat{x}_1$, $\hat{z}_1$, $\hat{t}_1$ as projective coordinates on $\mathbf{E}_2$.
This pencil does not have base curves, which implies that $\mathcal{S}^2$ does not have base curves in $\mathbf{E}_2$ either.

Since the defining equation of the surface $S_\lambda$ is
invariant with respect to the swap $x\leftrightarrow y$,
the point $P_{\{y\},\{z\},\{t\}}$ is also a du Val singular point of the surface $S_\lambda$ of type $\mathbb{A}_3$
provided that $\lambda\not\in\{-4,\infty\}$.
Let $\alpha_{3}\colon U_3\to U_2$ be the blow up of the preimage of this point.
Then $\mathbf{E}_3$ contains two base curves of the pencil $\mathcal{S}^3$.
Denote them by $C_{15}^3$ and~$C_{16}^3$.
Let~$\alpha_4\colon U_4\to U_3$ be the blow up of the point $C_{15}^3\cap C_{16}^3$.
Then $D_\lambda^4=S^4_\lambda$ for every $\lambda\in\mathbb{C}$.
Moreover,  the surface $\mathbf{E}_4$ does not contain base curves of the pencil $\mathcal{S}^4$.

Now let us describe the singularity of the surface $S_\lambda$ at the point $P_{\{x\},\{y\},\{z,t\}}$.
In the chart $t=1$, the surface $S_\lambda$ is given by
$$
(\lambda+4)\bar{x}\bar{y}-(\lambda+4)\bar{x}\bar{y}\bar{z}+\Big(\bar{x}^2\bar{z}\bar{y}+\bar{x}^2\bar{z}^2+\bar{y}^2\bar{z}\bar{x}+2\bar{z}^2\bar{y}\bar{x}+\bar{z}^3\bar{x}+\bar{y}^2\bar{z}^2+\bar{y}\bar{z}^3\Big)=0,
$$
where $\bar{x}=x$, $\bar{y}=y$, and $\bar{z}=z+1$.
Let $\alpha_{5}\colon U_5\to U_4$ be the blow up of the preimage of the point $P_{\{x\},\{y\},\{z,t\}}$.
In a neighborhood of the point $P_{\{x\},\{y\},\{z,t\}}$,
one chart of this blow up is given by the coordinate change
$\bar{x}_5=\frac{\bar{x}}{\bar{z}}$, $\bar{y}_5=\frac{\bar{y}}{\bar{z}}$, and $\bar{z}_5=\bar{z}$.
Then $D_\lambda^5$ is given by
$$
(\lambda+4)\bar{x}_5\bar{y}_5+\Big(\bar{x}_5\bar{z}_5^2+\bar{z}_5^2\bar{y}_5-(\lambda+4)\bar{x}_5\bar{y}_5\bar{z}_5\Big)+\Big(\bar{x}_5^2\bar{z}_5^2+2\bar{z}_5^2\bar{y}_5\bar{x}_5+\bar{y}_5^2\bar{z}_5^2\Big)+\Big(\bar{x}_5^2\bar{y}_5\bar{z}_5^2+\bar{x}_5\bar{y}_5^2\bar{z}_5^2\Big)=0,
$$
and $\mathbf{E}_5$ is given by $\bar{z}_5=0$.
Note that $\mathbf{E}_5$ contains one singular point of this surface: the point $(\bar{x}_5,\bar{y}_5,\bar{z}_5)=(0,0,0)$.
Note also that
$D_{-4}^5=S^5_{-4}+2\mathbf{E}_5$, and $\mathcal{S}^5\vert_{\mathbf{E}_5}$ is a union of the curves $\{\bar{z}_5=\bar{x}_5=0\}$ and $\{\bar{z}_5=\bar{y}_5=0\}$.
Denote them by $C_{17}^5$ and $C_{18}^5$, respectively.

Let $\alpha_{6}\colon U_6\to U_5$ be the blow up of the point $C_{17}^5\cap C_{18}^5$.
Locally, one chart of this blow up is given by the coordinate change
$\bar{x}_6=\frac{\bar{x}_5}{\bar{z}_5}$, $\bar{y}_6=\frac{\bar{y}_5}{\bar{z}_5}$, and $\bar{z}_6=\bar{z}_5$.
Moreover, if $\lambda\ne -4$, then $S^6_\lambda$ in this chart is given by
$$
(\lambda+4)\bar{y}_6\bar{x}_6+\bar{z}_6\bar{x}_6+\bar{z}_6\bar{y}_6-(\lambda+4)\bar{x}_6\bar{y}_6\bar{z}_6+
\Big(\bar{x}_6^2\bar{z}_6^2+2\bar{z}_6^2\bar{y}_6\bar{x}_6+\bar{y}_6^2\bar{z}_6^2\Big)+\Big(\bar{x}_6^2\bar{y}_6\bar{z}_6^3+\bar{x}_6\bar{y}_6^2\bar{z}_6^3\Big)=0.
$$
Here, the surface $\mathbf{E}_6$ is given by $\bar{z}_6=0$.
If $\lambda\ne -4$, then $S^6_\lambda$ has ordinary double singularity at the point $(\bar{x}_6,\bar{y}_6,\bar{z}_6)=(0,0,0)$.
Therefore, if $\lambda\ne -4$, then $P_{\{x\},\{y\},\{z,t\}}$ is a du Val singular point of the surface $S_\lambda$ of type $\mathbb{A}_5$.

By construction, we have $D_\lambda^6=S^6_\lambda\sim -K_{U^6}$ for every $\lambda$ such that $\lambda\ne -4$ and $\lambda\ne\infty$.
One the other hand, we have
$D^6_{-4}=S^6_{-4}+2\mathbf{E}_5^6$.
This follows from the fact that $S^5_{-4}$ contains the point $C_{17}^5\cap C_{18}^5$ and is smooth at it.

\begin{remark}
\label{remark:r2-n6-Lxy}
Our computations implies that the proper transform of the line $L_{\{x\},\{y\}}$ on the threefold $U_6$
passes through the point $(\bar{x}_6,\bar{y}_6,\bar{z}_6)=(0,0,0)$.
\end{remark}

The restriction $\mathcal{S}^6\vert_{\mathbf{E}_6}$ consists of the curves $\{\bar{z}_6=\bar{x}_6=0\}$ and $\{\bar{z}_6=\bar{y}_6=0\}$.
Let us denote these curves by $C_{19}^6$ and $C_{20}^6$, respectively.
Let $\alpha_7\colon U_7\to U_6$ be the blow up of the point $(\bar{x}_6,\bar{y}_6,\bar{z}_6)=(0,0,0)$.
Then
$D^7_{-4}=S^7_{-4}+2\mathbf{E}_5^7+\mathbf{E}_6^{7}$.
Moreover, the restriction $\mathcal{S}^7\vert_{\mathbf{E}_7}$ is a pencil of conics in $\mathbf{E}_7\cong\mathbb{P}^2$
that is given by
$$
(\lambda+4)\bar{y}_6\bar{x}_6+\bar{z}_6\bar{x}_6+\bar{z}_6\bar{y}_6=0,
$$
where we consider $\bar{x}_6$, $\bar{y}_6$, $\bar{z}_6$ as projective coordinates on $\mathbf{E}_7$.
This pencil does not have base curves, so that $\mathcal{S}^7$ also does not have base curves contained in the surface $\mathbf{E}_7$.

If $\lambda\ne-4$, then $P_{\{z\},\{t\},\{x,y\}}$ is an ordinary double point of the surface $S_\lambda$.
Indeed, in the chart $y=1$, the surface $S_\lambda$ is given by
$$
\tilde{x}(\tilde{z}+\tilde{t})-(\lambda+4)\tilde{z}\tilde{t}=\tilde{x}^2\tilde{t}-(\lambda+4)\tilde{t}\tilde{x}\tilde{z}+\tilde{x}^2\tilde{z}+\tilde{t}^3\tilde{x}+\tilde{x}^2\tilde{t}^2+3\tilde{t}^2\tilde{x}\tilde{z}+2\tilde{t}\tilde{x}^2\tilde{z}+3\tilde{t}\tilde{x}\tilde{z}^2+\tilde{x}^2\tilde{z}^2+\tilde{z}^3\tilde{x},
$$
where $\tilde{x}=x-1$, $\tilde{z}=z$, $\tilde{t}=t$.
The quadratic form $\tilde{x}(\tilde{z}+\tilde{t})-(\lambda+4)\tilde{z}\tilde{t}$
is not degenerate for $\lambda\ne-4$, so that $P_{\{z\},\{t\},\{x,y\}}$ is an ordinary double point of the surface $S_\lambda$.

\begin{corollary}
\label{corollary:r2-n6-singularities}
Suppose that $\lambda\not\in\{-4,\infty\}$.
Then singular points of the surface $S_\lambda$ contained in the base locus of the pencil $\mathcal{S}$ can be describes as follows:
\begin{itemize}\setlength{\itemindent}{3cm}
\item[$P_{\{x\},\{z\},\{t\}}$:] type $\mathbb{A}_3$ with quadratic term $(z+t)(x+z+t)$;
\item[$P_{\{y\},\{z\},\{t\}}$:] type $\mathbb{A}_3$ with quadratic term $(z+t)(y+z+t)$;
\item[$P_{\{x\},\{y\},\{z,t\}}$:] type $\mathbb{A}_5$ with quadratic term $(\lambda+4)xy$;
\item[$P_{\{z\},\{t\},\{x,y\}}$:] type $\mathbb{A}_1$ with quadratic term $(\lambda+4)zt-(x+y)(z+t)$.
\end{itemize}
\end{corollary}

Let us finish the description of the birational morphism $\alpha$.
It is given by the following commutative diagram
$$
\xymatrix{
&&U_3\ar@{->}[dll]_{\alpha_3}&&U_4\ar@{->}[ll]_{\alpha_4}&&U_5\ar@{->}[ll]_{\alpha_5}&&U_6\ar@{->}[ll]_{\alpha_6}&&\\
U_2\ar@{->}[rrd]_{\alpha_2}&&&&&&&&&&U_7\ar@{->}[llu]_{\alpha_7}\\
&&U_1\ar@{->}[rr]_{\alpha_1}&&\mathbb{P}^3&&&&U\ar@{->}[llll]^\alpha\ar@{->}[urr]_{\alpha_8} }
$$
Here $\alpha_8$ is the blow up of the preimage of the point $P_{\{z\},\{t\},\{x,y\}}$.

If $\lambda\ne -4$, then $\widehat{D}_\lambda=\widehat{S}_\lambda\sim -K_{U}$.
On the other  hand, we have
$\widehat{D}_{-4}=\widehat{S}_{-4}+2\widehat{E}_5+\widehat{E}_6$.
Moreover, the curves $\widehat{C}_1$, $\widehat{C}_2$, $\widehat{C}_3$, $\widehat{C}_4$, $\widehat{C}_5$,
$\widehat{C}_6$, $\widehat{C}_7$, $\widehat{C}_8$, $\widehat{C}_9$, $\widehat{C}_{10}$, $\widehat{C}_{11}$, $\widehat{C}_{12}$,
$\widehat{C}_{13}$, $\widehat{C}_{14}$, $\widehat{C}_{15}$, $\widehat{C}_{16}$, $\widehat{C}_{17}$,
$\widehat{C}_{18}$, $\widehat{C}_{19}$, and $\widehat{C}_{20}$ are all base curves of the pencil $\widehat{\mathcal{S}}$.

\begin{lemma}
\label{lemma:r2-n6-m13-m14-m15-m16-m17-m18-m19-m20}
One has $\mathbf{m}_{13}=\mathbf{m}_{14}=\mathbf{m}_{15}=\mathbf{m}_{16}=\mathbf{m}_{19}=\mathbf{m}_{20}=1$ and $\mathbf{m}_{17}=\mathbf{m}_{18}=2$.
\end{lemma}

\begin{proof}
To find $\mathbf{m}_{13}$ and $\mathbf{m}_{14}$, we use
\begin{multline*}
6=\mathrm{mult}_{P_{\{x\},\{z\},\{t\}}}\Big(2C_2+2C_3+C_9+C_{12}\Big)=\mathrm{mult}_{P_{\{x\},\{z\},\{t\}}}\Big(S_{0}\cdot S_{1}\Big)=\\
=\mathrm{mult}_{P_{\{x\},\{z\},\{t\}}}\big(S_{0}\big)\mathrm{mult}_{P_{\{x\},\{z\},\{t\}}}\big(S_{1}\big)+\mathbf{m}_{13}+\mathbf{m}_{14}=4+\mathbf{m}_{13}+\mathbf{m}_{14},
\end{multline*}
so that $\mathbf{m}_{13}=\mathbf{m}_{14}=1$. Similarly, we see that  $\mathbf{m}_{15}=\mathbf{m}_{16}=1$.

Recall that $\widehat{D}_{-4}=\widehat{S}_{-4}+2\widehat{E}_5+\widehat{E}_6$,
so that $\mathbf{m}_{17}\geqslant 2$ and $\mathbf{m}_{18}\geqslant 2$.
But
\begin{multline*}
8=\mathrm{mult}_{P_{\{x\},\{y\},\{z,t\}}}\Big(2C_1+2C_3+2C_4+C_5+C_{6}\Big)=\mathrm{mult}_{P_{\{x\},\{z\},\{t\}}}\Big(S_{0}\cdot S_{1}\Big)=\\
=\mathrm{mult}_{P_{\{x\},\{y\},\{z,t\}}}\big(S_{0}\big)\mathrm{mult}_{P_{\{x\},\{y\},\{z,t\}}}\big(S_{1}\big)+\mathbf{m}_{17}+\mathbf{m}_{18}=4+\mathbf{m}_{17}+\mathbf{m}_{18},
\end{multline*}
which implies that $\mathbf{m}_{17}=2$ and $\mathbf{m}_{18}=2$.

To find $\mathbf{m}_{19}$ and $\mathbf{m}_{20}$, recall that $\alpha_{6}\colon U_6\to U_5$ is the blow up of the point $C_{17}^5\cap C_{18}^5$.
Let $P=C_{17}^5\cap C_{18}^5$. Then
$$
6=\mathrm{mult}_{P}\Big(2C_{17}^5+2C_{18}^5+2C^5_1\Big)=\mathrm{mult}_{P}\Big(S_{0}^5\cdot S_{1}^5\Big)=4+\mathbf{m}_{19}+\mathbf{m}_{20},
$$
which gives us $\mathbf{m}_{19}=1$ and $\mathbf{m}_{20}=1$.
\end{proof}

For every $\lambda\not\in\{-4,\infty\}$, we have $[\mathsf{f}^{-1}(\lambda)]=1$ by Corollaries~\ref{corollary:irreducible-fibers} and \ref{corollary:r2-n6-singularities}.

\begin{lemma}
\label{lemma:r2-n6-irreducible-special}
One has $[\mathsf{f}^{-1}(-4)]=10$.
\end{lemma}

\begin{proof}
Recall that $[S_{-4}]=4$ and $[\widehat{D}_{-4}]=6$.
Thus, it follows from \eqref{equation:equation:number-of-irredubicle-components-refined-2} that
$$
\big[\mathsf{f}^{-1}(-4)\big]=6+\sum_{i=1}^{18}\mathbf{C}_i^{-4}.
$$
On the other hand, we have $\mathbf{M}_{3}^{-4}=\mathbf{M}_{4}^{-4}=\mathbf{M}_{17}^{-4}=\mathbf{M}_{18}^{-4}=2$ and
\begin{multline*}
\mathbf{M}_{1}^{-4}=\mathbf{M}_{2}^{-4}=\mathbf{M}_{5}^{-4}=\mathbf{M}_{6}^{-4}=\mathbf{M}_{7}^{-4}=\mathbf{M}_{8}^{-4}=\mathbf{M}_{9}^{-4}=\\
=\mathbf{M}_{10}^{-4}=\mathbf{M}_{11}^{-4}=\mathbf{M}_{12}^{-4}=\mathbf{M}_{13}^{-4}=\mathbf{M}_{14}^{-4}=\mathbf{M}_{15}^{-4}=\mathbf{M}_{16}^{-4}=1.
\end{multline*}
But $\mathbf{m}_{3}=\mathbf{m}_{4}=2$, and $\mathbf{m}_{17}=\mathbf{m}_{18}=2$ by Lemma~\ref{lemma:r2-n6-m13-m14-m15-m16-m17-m18-m19-m20}.
This shows that
$$
\big[\mathsf{f}^{-1}(\lambda)\big]=6+\sum_{i=1}^{18}\mathbf{C}_i^{-4}=6+\mathbf{C}_3^{-4}+\mathbf{C}_4^{-4}+\mathbf{C}_{17}^{-4}+\mathbf{C}_{18}^{-4}=10,
$$
since $\mathbf{C}_3^{-4}=\mathbf{C}_4^{-4}=\mathbf{C}_{17}^{-4}=\mathbf{C}_{18}^{-4}=1$ by Lemma~\ref{lemma:main-2}.
\end{proof}

Thus, we see that \eqref{equation:main-1} in Main Theorem holds in this case.

To prove \eqref{equation:main-2} in Main Theorem, we have to check \eqref{equation:main-2-simple}.
To do this, note that
$$
\mathrm{rk}\,\mathrm{Pic}\big(\widetilde{S}_{\Bbbk}v)=\mathrm{rk}\,\mathrm{Pic}\big(S_{\Bbbk}\big)+12.
$$
This follows from the proof of Corollary~\ref{corollary:r2-n6-singularities}
Moreover, if $\lambda\not\in\{-4,\infty\}$, then the intersection matrix
of the base curves of the pencil $\mathcal{S}$ on the surface $S_\lambda$ has the same rank as
the intersection matrix of the curves
$L_{\{x\},\{y\}}$, $L_{\{z\},\{t\}}$, $L_{\{x\},\{z,t\}}$,
$L_{\{y\},\{z,t\}}$, $L_{\{z\},\{y,t\}}$, $L_{\{z\},\{x,t\}}$,
$L_{\{t\},\{y,z\}}$, $L_{\{t\},\{x,z\}}$, and $H_\lambda$, because
\begin{multline*}
H_\lambda\sim L_{\{x\},\{y\}}+2L_{\{x\},\{z,t\}}+L_{\{x\},\{y,z,t\}}\sim H_{\{y\}}\cdot S_0=L_{\{x\},\{y\}}+2L_{\{y\},\{z,t\}}+L_{\{y\},\{x,z,t\}}\sim\\
\sim L_{\{z\},\{t\}}+L_{\{z\},\{x,y\}}+L_{\{z\},\{y,t\}}+L_{\{z\},\{x,t\}}\sim L_{\{z\},\{t\}}+L_{\{t\},\{x,y\}}+L_{\{t\},\{y,z\}}+L_{\{t\},\{x,z\}}.
\end{multline*}
This implies \eqref{equation:main-2-simple},
because the rank of the intersection matrix in the following lemma is~$6$.

\begin{lemma}
\label{lemma:r2-n6-intersection}
Suppose that $\lambda\not\in\{-4,\infty\}$.
Then the intersection form of the curves
$L_{\{x\},\{y\}}$, $L_{\{z\},\{t\}}$, $L_{\{x\},\{z,t\}}$,
$L_{\{y\},\{z,t\}}$, $L_{\{z\},\{y,t\}}$, $L_{\{z\},\{x,t\}}$,
$L_{\{t\},\{y,z\}}$, $L_{\{t\},\{x,z\}}$, and $H_\lambda$ on the surface $S_\lambda$ is given by
\begin{center}\renewcommand\arraystretch{1.42}
\begin{tabular}{|c||c|c|c|c|c|c|c|c|c|}
\hline
$\bullet$ & $L_{\{x\},\{y\}}$ & $L_{\{z\},\{t\}}$ &  $L_{\{x\},\{z,t\}}$ & $L_{\{y\},\{z,t\}}$ &  $L_{\{z\},\{y,t\}}$ &  $L_{\{z\},\{x,t\}}$ & $L_{\{t\},\{y,z\}}$ &  $L_{\{t\},\{x,z\}}$ & $H_\lambda$\\
\hline\hline
$L_{\{x\},\{y\}}$& $-\frac{1}{2}$ & $0$ & $\frac{1}{2}$ & $\frac{1}{2}$ & $0$ & $0$ & $0$ & $0$ & $1$\\
\hline
$L_{\{z\},\{t\}}$& $0$ & $0$ & $\frac{1}{2}$ & $\frac{1}{2}$ & $\frac{1}{4}$ & $\frac{1}{4}$ & $\frac{1}{4}$ & $\frac{1}{4}$ & $1$\\
\hline
$L_{\{x\},\{z,t\}}$& $\frac{1}{2}$ & $\frac{1}{2}$ & $-\frac{1}{6}$ & $\frac{1}{6}$ & $0$ & $\frac{1}{2}$ & $0$ & $\frac{1}{2}$ & $1$\\
\hline
$L_{\{y\},\{z,t\}}$& $\frac{1}{2}$ & $\frac{1}{2}$ & $\frac{1}{6}$ & $-\frac{1}{6}$ & $\frac{1}{2}$ & $0$ & $\frac{1}{2}$ & $0$ & $1$\\
\hline
$L_{\{z\},\{y,t\}}$& $0$ & $\frac{1}{4}$ & $0$ & $\frac{1}{2}$ & $-\frac{5}{4}$ & $1$ & $\frac{3}{4}$ & $0$ & $1$\\
\hline
$L_{\{z\},\{x,t\}}$& $0$ & $\frac{1}{4}$ & $\frac{1}{2}$ & $0$ & $1$ & $-\frac{5}{4}$ & $0$ & $\frac{3}{4}$ & $1$\\
\hline
$L_{\{t\},\{y,z\}}$& $0$ & $\frac{1}{4}$ & $0$ & $\frac{1}{2}$ & $\frac{3}{4}$ & $0$ & $-\frac{5}{4}$ & $1$ & $1$\\
\hline
$L_{\{t\},\{x,z\}}$& $0$ & $\frac{1}{4}$ & $\frac{1}{2}$ & $0$ & $0$ & $\frac{3}{4}$ & $1$ & $-\frac{5}{4}$ & $1$\\
\hline
$H_\lambda$& $1$ & $1$ & $1$ & $1$ & $1$ & $1$ & $1$ & $1$ & $4$\\
\hline
\end{tabular}
\end{center}
\end{lemma}

\begin{proof}
The entries in last raw of the intersection matrix are obvious.
Let us compute its diagonal.
Using Proposition~\ref{proposition:du-Val-self-intersection} and Remark~\ref{remark:r2-n6-Lxy}, we obtain
$L_{\{x\},\{y\}}^2=-\frac{1}{2}$,
because $P_{\{x\},\{y\},\{z,t\}}$ is the only singular point of the surface $S_\lambda$ that is contained in $L_{\{x\},\{y\}}$.
Likewise, it follows from Proposition~\ref{proposition:du-Val-self-intersection} and Remark~\ref{remark:transversal} that
$$
L_{\{z\},\{t\}}^2=-2+\frac{3}{4}+\frac{3}{4}+\frac{1}{2}=0
$$
because the line $L_{\{z\},\{t\}}$ contains the points $P_{\{x\},\{z\},\{t\}}$, $P_{\{y\},\{z\},\{t\}}$, and $P_{\{z\},\{t\},\{x,y\}}$.

To compute $L_{\{x\},\{z,t\}}^2$, observe that the line $L_{\{x\},\{z,t\}}$ contains the points $P_{\{x\},\{z\},\{t\}}$ and $P_{\{x\},\{y\},\{z,t\}}$.
Applying Remark~\ref{remark:transversal} with $S=S_\lambda$, $O=P_{\{x\},\{z\},\{t\}}$, $n=3$, and $C=L_{\{x\},\{z,t\}}$,
we see that $\overline{C}$ contains the point $\overline{G}_1\cap\overline{G}_2$.
Similarly, applying Remark~\ref{remark:transversal} with $S=S_\lambda$, $O=P_{\{x\},\{y\},\{z,t\}}$, $n=5$, and $C=L_{\{x\},\{z,t\}}$,
we see that $\overline{C}$ does not contain the point $\overline{G}_1\cap\overline{G}_5$.
Thus, applying Proposition~\ref{proposition:du-Val-self-intersection}, we get
$L_{\{x\},\{z,t\}}^2=-2+1+\frac{5}{6}=-\frac{1}{6}$.

To compute $L_{\{z\},\{y,t\}}^2$, notice that $P_{\{y\},\{z\},\{t\}}$ is the only singular point of the surface $S_\lambda$
that is contained in $L_{\{z\},\{y,t\}}$.
Applying Proposition~\ref{proposition:du-Val-self-intersection}, we see that $L_{\{z\},\{y,t\}}^2=-\frac{5}{4}$.

Using the symmetry $x\leftrightarrow y$, we get $L_{\{x\},\{z,t\}}^2=-\frac{1}{6}$ and $L_{\{z\},\{x,t\}}^2=-\frac{5}{4}$.
Similarly, using the symmetry $z\leftrightarrow t$, we see that $L_{\{t\},\{y,z\}}^2=L_{\{t\},\{x,z\}}=-\frac{5}{4}$.

Now let us fill in the remaining entries in the first raw of the  table.
Clearly, we have
$L_{\{x\},\{y\}}\cdot L_{\{z\},\{t\}}=0$, $L_{\{x\},\{y\}}\cdot L_{\{z\},\{y,t\}}=0$,
$L_{\{x\},\{y\}}\cdot L_{\{z\},\{x,t\}}=0$, $L_{\{x\},\{y\}}\cdot L_{\{t\},\{y,z\}}=0$, and $L_{\{x\},\{y\}}\cdot L_{\{t\},\{x,z\}}=0$,
because $L_{\{x\},\{y\}}$ does not intersect the lines
$L_{\{z\},\{t\}}$,  $L_{\{z\},\{y,t\}}$, $L_{\{z\},\{x,t\}}$, $L_{\{t\},\{y,z\}}$, and $L_{\{t\},\{x,z\}}$.
Using symmetry $x\leftrightarrow y$, we see that
$$
L_{\{x\},\{y\}}\cdot L_{\{x\},\{z,t\}}=L_{\{x\},\{y\}}\cdot L_{\{y\},\{z,t\}}.
$$
To find $L_{\{x\},\{y\}}\cdot L_{\{x\},\{z,t\}}$, we observe that $L_{\{x\},\{y\}}\cap L_{\{x\},\{z,t\}}=P_{\{x\},\{y\},\{z,t\}}$.
Applying Proposition~\ref{proposition:du-Val-intersection} and Remark~\ref{remark:r2-n6-Lxy}, we see that $L_{\{x\},\{y\}}\cdot L_{\{x\},\{z,t\}}=\frac{1}{2}$.

Let us compute the remaining entries in the second raw of the intersection matrix.
Since $L_{\{z\},\{t\}}\cap L_{\{x\},\{z,t\}}=P_{\{x\},\{z\},\{t\}}$, we have
$L_{\{z\},\{t\}}\cdot L_{\{x\},\{z,t\}}=\frac{1}{2}$ by Proposition~\ref{proposition:du-Val-intersection} and Remark~\ref{remark:transversal}.
Using symmetry $x\leftrightarrow y$, we get $L_{\{z\},\{t\}}\cdot L_{\{y\},\{z,t\}}=\frac{1}{2}$.

Observe that $L_{\{z\},\{t\}}\cap L_{\{z\},\{x,t\}}=P_{\{x\},\{z\},\{t\}}$.
Applying Remark~\ref{remark:transversal} with $S=S_\lambda$, $O=P_{\{x\},\{z\},\{t\}}$, $n=3$, $C=L_{\{z\},\{t\}}$
and $Z=L_{\{z\},\{x,t\}}$, we see that $\overline{C}$ and $\overline{Z}$ intersect
different curves among $\overline{G}_1$ and $\overline{G}_3$.
This implies $L_{\{z\},\{t\}}\cdot L_{\{z\},\{x,t\}}=\frac{1}{4}$ by Proposition~\ref{proposition:du-Val-intersection}.
Using symmetry $x\leftrightarrow y$, we get $L_{\{z\},\{t\}}\cdot L_{\{z\},\{y,t\}}=\frac{1}{4}$.
Using symmetry $z\leftrightarrow t$, we get
$$
L_{\{z\},\{t\}}\cdot L_{\{t\},\{y,z\}}=L_{\{z\},\{t\}}\cdot L_{\{t\},\{x,z\}}=\frac{1}{4}.
$$
This gives us all entries in the second raw of the intersection matrix.

Let us compute the third raw.
Observe that $L_{\{x\},\{z,t\}}\cap L_{\{y\},\{z,t\}}=P_{\{x\},\{y\},\{z,t\}}$.
Applying Remark~\ref{remark:transversal} with $S=S_\lambda$, $O=P_{\{x\},\{y\},\{z,t\}}$, $n=5$, $C=L_{\{x\},\{z,t\}}$
and $Z=L_{\{y\},\{z,t\}}$, we see that $\overline{C}$ and $\overline{Z}$ intersect different curves among $\overline{G}_1$ and $\overline{G}_5$.
Then $L_{\{x\},\{z,t\}}\cdot L_{\{y\},\{z,t\}}=\frac{1}{6}$ by Proposition~\ref{proposition:du-Val-intersection}.

Since $L_{\{x\},\{z,t\}}\cap L_{\{z\},\{y,t\}}=\varnothing$, we have $L_{\{x\},\{z,t\}}\cdot L_{\{z\},\{y,t\}}=0$.
Using symmetry~$z\leftrightarrow t$, we get $L_{\{x\},\{z,t\}}\cap L_{\{t\},\{y,z\}}=0$.
Since $L_{\{x\},\{z,t\}}\cap L_{\{z\},\{x,t\}}=P_{\{x\},\{z\},\{t\}}$,
we get
$$
L_{\{x\},\{z,t\}}\cdot L_{\{z\},\{x,t\}}=\frac{1}{2}
$$
by Proposition~\ref{proposition:du-Val-intersection}.
Using symmetry $z\leftrightarrow t$, we get $L_{\{x\},\{z,t\}}\cdot L_{\{t\},\{x,z\}}=\frac{1}{2}$.

Let us compute the remaining four entries in the fourth raw of the intersection matrix.
Using symmetries $x\leftrightarrow y$ and $z\leftrightarrow t$, we get
$$
L_{\{y\},\{z,t\}}\cdot L_{\{z\},\{y,t\}}=L_{\{x\},\{z,t\}}\cdot L_{\{z\},\{x,t\}}=L_{\{y\},\{z,t\}}\cdot L_{\{t\},\{y,z\}}=L_{\{y\},\{z,t\}}\cdot L_{\{z\},\{y,t\}}=\frac{1}{2},
$$
and $L_{\{y\},\{z,t\}}\cdot L_{\{z\},\{x,t\}}=L_{\{x\},\{z,t\}}\cdot L_{\{z\},\{y,t\}}=L_{\{y\},\{z,t\}}\cdot L_{\{t\},\{x,z\}}=L_{\{y\},\{z,t\}}\cdot L_{\{z\},\{x,t\}}=0$.

Let us compute the remaining three entries in the fifth raw of the intersection matrix.
First, we have $L_{\{z\},\{y,t\}}\cdot L_{\{t\},\{x,z\}}=0$, because $L_{\{z\},\{y,t\}}\cap L_{\{t\},\{x,z\}}=\varnothing$.
Second, we have $L_{\{z\},\{y,t\}}\cdot L_{\{z\},\{x,t\}}=1$, because
$L_{\{z\},\{y,t\}}\cap L_{\{z\},\{x,t\}}$ is a smooth point of the surface $S_\lambda$.
Third, we compute $L_{\{z\},\{y,t\}}\cdot L_{\{t\},\{y,z\}}$.
Observe that $L_{\{z\},\{y,t\}}\cap L_{\{t\},\{y,z\}}=P_{\{y\},\{z\},\{t\}}$.
Applying Remark~\ref{remark:transversal} with $S=S_\lambda$, $O=P_{\{y\},\{z\},\{t\}}$, $n=3$, $C=L_{\{z\},\{y,t\}}$
and $Z=L_{\{t\},\{y,z\}}$, we see that $\overline{C}$ and $\overline{Z}$ intersect the same curve among $\overline{G}_1$ and $\overline{G}_3$,
and none of them contains the point $\overline{G}_1\cap\overline{G}_3$.
Thus, we have $L_{\{z\},\{y,t\}}\cdot L_{\{t\},\{y,z\}}=\frac{3}{4}$ by Proposition~\ref{proposition:du-Val-intersection}.

Let us compute the remaining three entries of the  matrix.
Using symmetry~$x\leftrightarrow y$, we get $L_{\{z\},\{x,t\}}\cdot L_{\{t\},\{y,z\}}=L_{\{z\},\{y,t\}}\cdot L_{\{t\},\{x,z\}}=0$.
Likewise, we have
$$
L_{\{z\},\{x,t\}}\cdot L_{\{t\},\{x,z\}}=L_{\{z\},\{y,t\}}\cdot L_{\{t\},\{y,z\}}=\frac{3}{4}.
$$
Finally, using symmetry $z\leftrightarrow t$, we get
$L_{\{t\},\{y,z\}}\cdot L_{\{t\},\{x,z\}}=L_{\{z\},\{y,t\}}\cdot L_{\{z\},\{x,t\}}=1$.
\end{proof}

\subsection{Family \textnumero $2.7$}
\label{section:r-2-n-7}

In this case, the threefold $X$ can be obtained by blowing up a smooth quadric threefold $\mathcal{Q}$ in $\mathbb{P}^4$
along a smooth curve of genus $5$. This implies that $h^{1,2}(X)=5$.
A~toric Landau--Ginzburg model of this family is given by Minkowski polynomial \textnumero $3238$.
It is
$$
x+{\frac {x}{y}}+z+{\frac {z}{y}}+{\frac {xy}{z}}+{\frac {2x}{z}}+{\frac {x}{yz}}+2y+\frac{2}{y}+{\frac {yz}{x}}+{\frac{2z}{x}}+{\frac {z}{xy}}+\frac{2y}{z}+\frac{2}{z}+{\frac {2y}{x}}+\frac{2}{x}+{\frac {y}{xz}}.
$$
The corresponding pencil of quartic surfaces $\mathcal{S}$ is given by
\begin{multline*}
x^2yz+x^2tz+z^2yx+z^2tx+y^2x^2+2x^2ty+x^2t^2+2y^2zx+2t^2zx+y^2z^2+\\
+2z^2ty+t^2z^2+2y^2tx+2t^2yx+2y^2tz+2t^2yz+y^2t^2=\lambda xyzt.
\end{multline*}
This equation is invariant with respect to the permutations $x\leftrightarrow z$ and $y\leftrightarrow t$.

Since the goal is to prove \eqref{equation:main-1} and \eqref{equation:main-2} in Main Theorem, we may assume that $\lambda\ne\infty$.
Let $\mathcal{C}_1$ be the conic $\{x=yz+ty+tz=0\}$,
let $\mathcal{C}_2$ be the conic $\{y=xz+tx+tz=0\}$,
let~$\mathcal{C}_3$ be the conic $\{z=xy+tx+ty=0\}$,
and let $\mathcal{C}_4$ be the conic $\{t=xy+xz+yz=0\}$.
Then
\begin{equation}
\label{equation:r2-n7-base-locus}
\begin{split}
H_{\{x\}}\cdot S_\lambda&=2\mathcal{C}_1,\\
H_{\{y\}}\cdot S_\lambda&=L_{\{y\},\{t\}}+L_{\{y\},\{x,z\}}+\mathcal{C}_2,\\
H_{\{z\}}\cdot S_\lambda&=2\mathcal{C}_3,\\
H_{\{t\}}\cdot S_\lambda&=L_{\{y\},\{t\}}+L_{\{t\},\{x,z\}}+\mathcal{C}_4,
\end{split}
\end{equation}
Thus, the base locus of the pencil $\mathcal{S}$ consists of $7$ smooth rational curves.
We let $C_1=\mathcal{C}_1$, $C_2=\mathcal{C}_2$, $C_3=\mathcal{C}_3$, $C_4=\mathcal{C}_4$, $C_5=L_{\{y\},\{t\}}$, $C_6=L_{\{y\},\{x,z\}}$, $C_7=L_{\{t\},\{x,z\}}$.

If $\lambda\ne -5$, then $S_\lambda$ has isolated singularities, so that it is irreducible.
On the other hand, one has $S_{-5}=\mathsf{Q}+\mathbf{Q}$,
where $\mathsf{Q}$ is a quadric surface given by $tx+ty+tz+xy+yz=0$,
and $\mathbf{Q}$ is a quadric surface given by $tx+ty+tz+xy+xz+yz=0$.
Both these quadric surfaces are irreducible.
The surface $\mathsf{Q}$ is singular at  $P_{\{y\},\{t\},\{x,z\}}$,
and the surface  $\mathbf{Q}$ is smooth.
One has $\mathsf{Q}\cap\mathbf{Q}=\mathcal{C}_1\cup \mathcal{C}_3$,
so that  $S_{-5}$ is singular along the conics $\mathcal{C}_1$ and $\mathcal{C}_3$.

If $\lambda\ne -5$, then the singular points of the surface $S_\lambda$ contained in the base locus of the pencil $\mathcal{S}$
are the points $P_{\{y\},\{z\},\{t\}}$,  $P_{\{x\},\{z\},\{t\}}$, $P_{\{x\},\{y\},\{t\}}$, $P_{\{x\},\{y\},\{z\}}$, and $P_{\{y\},\{t\},\{x,z\}}$.
They are all {fixed} singular points of the surfaces in $\mathcal{S}$.

\begin{lemma}
\label{lemma:r2-n7-singularities}
Suppose that $\lambda\ne-5$.
Then the singular points of the surface $S_\lambda$ contained in the base locus of the pencil $\mathcal{S}$ can be describes as follows:
\begin{itemize}\setlength{\itemindent}{3cm}
\item[$P_{\{y\},\{z\},\{t\}}$:] type $\mathbb{A}_3$ with quadratic term $(y+t)(y+z+t)$;
\item[$P_{\{x\},\{z\},\{t\}}$:] type $\mathbb{D}_4$ with quadratic term $(x+t+z)^2$;
\item[$P_{\{x\},\{y\},\{t\}}$:] type $\mathbb{A}_3$ with quadratic term $(y+t)(x+y+t)$;
\item[$P_{\{x\},\{y\},\{z\}}$:] type $\mathbb{D}_4$ with quadratic term $(x+y+z)^2$;
\item[$P_{\{y\},\{t\},\{x,z\}}$:] type $\mathbb{A}_1$ with quadratic term
$$
(x+z)(y+x)-(\lambda+4)ty
$$
for $\lambda\ne -4$, type $\mathbb{A}_3$ if $\lambda=-4$.
\end{itemize}
\end{lemma}

Let us prove this lemma and explicitly construct the birational morphism $\alpha$ in \eqref{equation:main-diagram}.
To start with, let us resolve the singularity of the surface $S_\lambda$ at the point $P_{\{y\},\{z\},\{t\}}$.
In~the~chart~$x=1$, the surface $S_\lambda$ is given by
\begin{multline*}
\hat{y}\hat{z}+\Big((\lambda+4)\hat{t}^2\hat{z}+(\lambda+6)\hat{t}\hat{y}^2-(\lambda+4)\hat{t}\hat{y}\hat{z}-(\lambda+6)\hat{t}^2\hat{y}-\hat{y}^3+\hat{y}\hat{z}^2\Big)+\\
+\Big(\hat{t}^4-2\hat{t}^3\hat{y}+3\hat{y}^2\hat{t}^2-2\hat{t}^2\hat{y}\hat{z}+\hat{y}^4+2\hat{t}\hat{y}^2\hat{z}-2\hat{t}\hat{y}^3-2\hat{y}^3\hat{z}+\hat{y}^2\hat{z}^2\Big)=0
\end{multline*}
where $\hat{y}=y+t$, $\hat{z}=t+z+y$, $\hat{t}=t$.
Let $\alpha_1\colon U_1\to\mathbb{P}^3$ be the blow up of the point~$P_{\{y\},\{z\},\{t\}}$.
A chart of the blow up $\alpha_1$ is given by the coordinate change $\hat{y}_1=\frac{\hat{y}}{\hat{t}}$, $\hat{z}_1=\frac{\hat{z}}{\hat{t}}$, and $\hat{t}_1=\hat{t}$.
In this chart, the surface $S^1_\lambda$ is given by the equation
\begin{multline*}
\Big(\hat{t}_1^2-(\lambda+6)\hat{t}_1\hat{y}_1+(\lambda+4)\hat{t}_1\hat{z}_1+\hat{y}_1\hat{z}_1\Big)-\Big(2\hat{t}_1^2\hat{y}_1-(\lambda+6)\hat{t}_1\hat{y}_1^2+(\lambda+4)\hat{t}_1\hat{y}_1\hat{z}_1\Big)+\\
+\Big(3\hat{y}_1^2\hat{t}_1^2-2\hat{t}_1^2\hat{y}_1\hat{z}_1-\hat{t}_1\hat{y}_1^3+\hat{t}_1\hat{y}_1\hat{z}_1^2\Big)+\Big(2\hat{t}_1^2\hat{y}_1^2\hat{z}_1-2\hat{t}_1^2\hat{y}_1^3\Big)+\Big(\hat{t}_1^2\hat{y}_1^4-2\hat{t}_1^2\hat{y}_1^3\hat{z}_1+\hat{t}_1^2\hat{y}_1^2\hat{z}_1^2\Big)=0.
\end{multline*}
where $\hat{t}_1=0$ defines the surface $\mathbf{E}_1$.
The only singular point of the surface $S^1_\lambda$ in $\mathbf{E}_1$ is the point $(\hat{y}_1,\hat{z}_1,\hat{t}_1)=(0,0,0)$.
If $\lambda\ne -5$, then this point is an ordinary double point of the surface $S_\lambda^1$,
so that $P_{\{x\},\{z\},\{t\}}$ is a singular point of the surface $S_\lambda$ of type $\mathbb{A}_3$.

The surface $\mathbf{E}_1$ contains two base curves of the pencil $\mathcal{S}^1$.
One of them is $\hat{t}_1=\hat{y}_1=0$, and another one is $\hat{t}_1=\hat{z}_1=0$.
Let us denote these curves by $C_{8}^1$ and $C_{9}^1$, respectively.
Then the proper transform of the line $L_{\{y\},\{t\}}$ on the threefold $U_1$ does not pass through the point $C_{8}^1\cap C_{9}^1$.

Let $\alpha_2\colon U_2\to U_1$ be the blow up of the point $(\hat{x}_1,\hat{z}_1,\hat{t}_1)=(0,0,0)$.
Then $D_\lambda^2\sim S^2_\lambda$ for every $\lambda\in\mathbb{C}$.
Moreover, the restriction $\mathcal{S}^2\vert_{\mathbf{E}_2}$ is a pencil of conics in $\mathbf{E}_2\cong\mathbb{P}^2$ that does not have base curves.
This shows that $\mathbf{E}_2$ contains no base curves of the pencil~$\mathcal{S}^2$.

Recall that the defining equation of the surface $S_\lambda$ is invariant with respect to the permutation $x\leftrightarrow z$.
Thus, if $\lambda\ne-5$, then $P_{\{x\},\{y\},\{t\}}$ is a du Val singular point of the surface $S_\lambda$ of type $\mathbb{A}_3$.
Moreover, the surface $S_{-5}$ has non-isolated ordinary double point at the point $P_{\{x\},\{y\},\{t\}}$,
so that $P_{\{x\},\{y\},\{t\}}$ is a {good} double point of the surface $S_{-5}$.

Let $\alpha_{3}\colon U_3\to U_2$ be the blow up of the preimage of the point $P_{\{x\},\{y\},\{t\}}$.
Then the pencil $\mathcal{S}^3$ has exactly two base curves contained in the surface $\mathbf{E}_3$.
Let us denote these curves by $C_{10}^3$ and $C_{11}^3$.
Let $\alpha_4\colon U_4\to U_3$ be the blow up of the point $C_{10}^3\cap C_{11}^3$.
Then~$\mathbf{E}_4$  does not contain base curves of the pencil $\mathcal{S}^4$,
and the proper transform of the line $L_{\{y\},\{t\}}$ on the threefold $U_3$ does not pass through the point $C_{10}^3\cap C_{11}^3$.

Now let us describe the singularity of the surface $S_\lambda$ at the point $P_{\{y\},\{t\},\{x,z\}}$.
In the chart $z=1$, the surface $S_\lambda$ is given by
$$
\bar{x}(\bar{t}+\bar{y})+(\lambda+4)\bar{t}\bar{y}-\Big(\bar{t}\bar{x}^2-(\lambda+4)\bar{t}\bar{x}\bar{y}+\bar{x}^2\bar{y}\Big)-\Big(\bar{x}^2\bar{t}^2+2\bar{t}^2\bar{x}\bar{y}+\bar{y}^2\bar{t}^2+2\bar{t}\bar{x}^2\bar{y}+2\bar{t}\bar{x}\bar{y}^2+\bar{x}^2\bar{y}^2\Big)=0,
$$
where $\bar{x}=x+1$, $\bar{y}=y$, and $\bar{t}=t$.
Thus, if $\lambda\ne -4$, then $P_{\{y\},\{t\},\{x,z\}}$ is an isolated ordinary double point of the surface $S_\lambda$.
If $\lambda=-4$, then the latter equation can be rewritten as
$$
\check{x}\check{t}-\check{y}^4+2\check{t}\check{y}^3-\check{y}^2\check{t}^2+2\check{t}\check{x}\check{y}^2-\check{t}\check{x}^2-2\check{t}^2\check{x}\check{y}v^7-\check{x}^2\check{t}^2=0,
$$
where $\check{x}=\bar{x}$, $\check{y}=\bar{y}$, and $\check{t}=\bar{t}+\bar{y}$.
Here, the term $\check{x}\check{t}-\check{y}^4$ has the smallest degree
with respect to the weights $\mathrm{wt}(\check{x})=2$, $\mathrm{wt}(\check{y})=1$, and $\mathrm{wt}(\check{t})=2$.
This shows that $P_{\{y\},\{t\},\{x,z\}}$ is a singular point of type $\mathbb{A}_3$ of the surface $S_{-4}$.

Let $\alpha_{5}\colon U_5\to U_4$ be the blow up of the preimage of the point $P_{\{y\},\{t\},\{x,z\}}$.
Then the restriction $\mathcal{S}^5\vert_{\mathbf{E}_5}$ is a pencil of conics that is given by
$$
\bar{x}(\bar{t}+\bar{y})+(\lambda+4)\bar{t}\bar{y}=0,
$$
where we consider $\bar{x}$, $\bar{y}$, and $\bar{t}$ as homogeneous coordinates on $\mathbf{E}_5\cong\mathbb{P}^2$.
This pencil does not have base curves, so that $\mathbf{E}_5$ does not contain base curves of the pencil $\mathcal{S}^5$ either.

Let us show that $P_{\{x\},\{z\},\{t\}}$ is a du Val singular point of the surface $S_\lambda$ of type $\mathbb{D}_4$.
In~the~ chart $y=1$, the surface $S_\lambda$ is given by
\begin{multline*}
\tilde{t}^2+\Big(2\tilde{t}^2\tilde{x}+2\tilde{t}^2\tilde{z}-2\tilde{t}\tilde{x}^2-(\lambda+8)\tilde{t}\tilde{x}\tilde{z}-2\tilde{t}\tilde{z}^2+(\lambda+5)\tilde{x}^2\tilde{z}+(\lambda+5)\tilde{x}\tilde{z}^2\Big)+\\
+\Big(\tilde{x}^2\tilde{t}^2+2\tilde{t}^2\tilde{x}\tilde{z}+\tilde{t}^2\tilde{z}^2-2\tilde{x}^3\tilde{t}-5\tilde{t}\tilde{x}^2\tilde{z}-5\tilde{t}\tilde{x}\tilde{z}^2-2\tilde{t}\tilde{z}^3+\tilde{x}^4+3\tilde{z}\tilde{x}^3+4\tilde{z}^2\tilde{x}^2+3\tilde{z}^3\tilde{x}+\tilde{z}^4\Big)=0,
\end{multline*}
where $\tilde{x}=x$, $\tilde{z}=z$, and $\tilde{t}=x+t+z$.
Let $\alpha_{6}\colon U_6\to U_5$ be the blow up of the preimage of the point $P_{\{x\},\{z\},\{t\}}$.
A chart of this blow up is given by the coordinate change $\tilde{x}_6=\frac{\tilde{x}}{\tilde{z}}$, $\tilde{z}_6=\tilde{z}$, and $\tilde{t}_6=\frac{\tilde{t}}{\tilde{z}}$.
Then $\tilde{z}_6=\tilde{t}_6=0$ define the exceptional curve of the induced birational morphism $S_{\lambda}^6\to S_{\lambda}^5$.
Moreover, if $\lambda\ne -5$, then the quadratic term of the surface $S_{\lambda}^6$ at the point $(\tilde{x}_6,\tilde{z}_6,\tilde{z})=(0,0,0)$ is
$$
\tilde{t}_6^2-2\tilde{t}_6\tilde{z}_6+(\lambda+5)\tilde{x}_6\tilde{z}_6+\tilde{z}_6^2.
$$
It is not degenerate. Thus, this point is an isolated ordinary double point of the surface~$S_{\lambda}^6$.
In this case, the chart of the surface $S_{\lambda}^6$ also has an isolated ordinary double singularity at the point $(\tilde{x}_6,\tilde{z}_6,\tilde{t}_6)=(-1,0,0)$,
and $S_{\lambda}^6$ is smooth along the curve $\tilde{z}_6=\tilde{t}_6=0$ away from these two points.

Now let us consider another chart of the blow up $\alpha_6$.
To do this, we introduce coordinates $\tilde{x}_6^\prime=\tilde{x}$, $\tilde{z}_6^\prime=\frac{\tilde{z}}{\tilde{x}}$, and $\tilde{t}_6^\prime=\frac{\tilde{t}}{\tilde{x}}$.
In this chart, the surface $S_{\lambda}^6$  is given by
$$
(\tilde{t}_6^\prime)^2-2\tilde{x}_6^\prime\tilde{t}_6^\prime+(\tilde{x}_6^\prime)^2+(\lambda+5)\tilde{x}_6^\prime\tilde{z}_6^\prime+\text{higher order terms}=0,
$$
so that $S_{\lambda}^6$ has an isolated ordinary double singularity at the point $(\tilde{x}_6^\prime,\tilde{z}_6^\prime,\tilde{t}_6^\prime)=(0,0,0)$
provided that $\lambda\ne -5$.
Therefore, we proved that if $\lambda\ne -5$, then $P_{\{x\},\{z\},\{t\}}$ is a singular point of the surface $S_\lambda$ of type $\mathbb{D}_4$.

The surface $\mathbf{E}_6$ contains one base curve of the pencil $\mathcal{S}^6$.
This is the curve $\{\tilde{z}_6=\tilde{t}_6=0\}$ in the first chart of our blow up.
Denote it by $C_{12}^6$.
Then $\mathbf{M}_{12}^{-5}=2$,
and $C_{12}^6$ contains three base points of the pencil $\mathcal{S}^6$,
which are {fixed} singular points of this pencil.
They are isolated ordinary double points of the surface $S_{\lambda}^6$ for $\lambda\ne -5$.

Recall that the defining equation of the surface $S_\lambda$ is invariant with respect to the permutation $y\leftrightarrow t$.
Thus, if $\lambda\ne -5$, then $P_{\{x\},\{y\},\{z\}}$ is a singular point of the surface $S_\lambda$ of type $\mathbb{D}_4$.
Using symmetry, we see that $(x+y+z)^2$ is the quadratic form of the Taylor expansion of the defining equation of the surface $S_\lambda$ at the point $P_{\{x\},\{y\},\{z\}}$.

Let $\alpha_{7}\colon U_7\to U_6$ be the blow up of the preimage of the point $P_{\{x\},\{y\},\{z\}}$.
Then $\mathbf{E}_7$ contains one base curve of the pencil $\mathcal{S}^7$.
Denote it by~$C_{13}^7$.
This curve is the exceptional curve of the induced birational morphism $S_{\lambda}^7\to S_{\lambda}^6$.
If $\lambda\ne-5$, then $S_{\lambda}^7$ has three isolated ordinary double points at $C_{13}^7$.
But  $S_{-5}^7$ is singular along the curve $C_{13}^7$.

For a general choice of $\lambda\in\mathbb{C}$, the surface $S_{\lambda}^7$ has six singular points.
All of them are {fixed} singular points of the pencil $\mathcal{S}^7$.
They are isolated ordinary double points on every surface $S_\lambda^7$ provided that $\lambda\ne -5$.
This proves the assertion of Lemma~\ref{lemma:r2-n7-singularities} and shows the existence
of the following commutative diagram:
$$
\xymatrix{
&&U_3\ar@{->}[dll]_{\alpha_3}&&U_4\ar@{->}[ll]_{\alpha_4}&&U_5\ar@{->}[ll]_{\alpha_5}&&U_6\ar@{->}[ll]_{\alpha_6}\\
U_2\ar@{->}[rrd]_{\alpha_2}&&&&&&&&&&U_7\ar@{->}[llu]_{\alpha_7}\\
&&U_1\ar@{->}[rr]_{\alpha_1}&&\mathbb{P}^3&&&&U\ar@{->}[llll]^\alpha\ar@{->}[urr]_{\gamma} }
$$
where $\gamma$ is the blow up of the six {fixed} singular points of surfaces in the pencil $\mathcal{S}^7$.

Using Lemma~\ref{lemma:r2-n7-singularities} and Corollary~\ref{corollary:irreducible-fibers},
we see that $[\mathsf{f}^{-1}(\lambda)]=1$ for every $\lambda\ne -5$.
To compute $[\mathsf{f}^{-1}(-5)]$, observe that $\widehat{D}_{-5}=\widehat{S}_{-5}$,
so that $[\widehat{D}_{-5}]=[\widehat{S}_{-5}]=2$.
Observe also that the base locus of the pencil $\widehat{\mathcal{S}}$ consists of the curves
$\widehat{C}_1$, $\widehat{C}_2$, $\widehat{C}_3$, $\widehat{C}_4$, $\widehat{C}_5$,
$\widehat{C}_6$, $\widehat{C}_7$, $\widehat{C}_8$, $\widehat{C}_9$, $\widehat{C}_{10}$, $\widehat{C}_{11}$, $\widehat{C}_{12}$, and
$\widehat{C}_{13}$.
Moreover, we have $\mathbf{M}_{1}^{-5}=\mathbf{M}_{2}^{-5}=\mathbf{M}_{12}^{-5}=\mathbf{M}_{13}^{-5}=2$ and
$$
\mathbf{M}_{3}^{-5}=\mathbf{M}_{4}^{-5}=\mathbf{M}_{5}^{-5}=\mathbf{M}_{6}^{-5}=\mathbf{M}_{7}^{-5}=\mathbf{M}_{8}^{-5}=\mathbf{M}_{9}^{-5}=\mathbf{M}_{10}^{-5}=\mathbf{M}_{11}^{-5}=1.
$$
Arguing as in the proof of Lemma~\ref{lemma:r2-n6-m13-m14-m15-m16-m17-m18-m19-m20},
we see that $\mathbf{m}_{8}=\mathbf{m}_{9}=\mathbf{m}_{10}=\mathbf{m}_{11}=1$ and $\mathbf{m}_{12}=\mathbf{m}_{13}=2$.
Now, using \eqref{equation:equation:number-of-irredubicle-components-refined-2} and Lemma~\ref{lemma:main-2}, we conclude that
$[\mathsf{f}^{-1}(-5)]=6$.
Thus, we see that \eqref{equation:main-1} in Main Theorem holds in this case, because $h^{1,2}(X)=5$.

To prove \eqref{equation:main-2} in Main Theorem, we have to check \eqref{equation:main-2-simple}.
To do this, recall that the base locus of the pencil $\mathcal{S}$ consists of the curves
$\mathcal{C}_1$, $\mathcal{C}_2$, $\mathcal{C}_3$, $\mathcal{C}_4$, $L_{\{y\},\{t\}}$, $L_{\{y\},\{x,z\}}$, and $L_{\{t\},\{x,z\}}$.
If $\lambda\ne -5$, then it follows from \eqref{equation:r2-n7-base-locus} that the intersection matrix of these curves on the surface $S_\lambda$ has the same rank as
the intersection matrix of the curves $L_{\{y\},\{t\}}$, $L_{\{y\},\{x,z\}}$, $L_{\{t\},\{x,z\}}$, and $H_\lambda$, which is given by
\begin{center}\renewcommand\arraystretch{1.42}
\begin{tabular}{|c||c|c|c|c|}
\hline
$\bullet$ & $L_{\{y\},\{t\}}$ &  $L_{\{y\},\{x,z\}}$ & $L_{\{t\},\{x,z\}}$ & $H_\lambda$\\
\hline\hline
$L_{\{y\},\{t\}}$& $0$ & $\frac{1}{2}$ & $\frac{1}{2}$ &   $1$\\
\hline
$L_{\{y\},\{x,z\}}$& $\frac{1}{2}$ & $-\frac{1}{2}$ & $\frac{1}{2}$ & $1$\\
\hline
$L_{\{t\},\{x,z\}}$& $\frac{1}{2}$ & $\frac{1}{2}$ & $-\frac{1}{2}$ & $1$\\
\hline
$H_\lambda$& $1$ & $1$ & $1$ & $4$\\
\hline
\end{tabular}
\end{center}
The rank of this intersection matrix is $3$.
Moreover, we have $\mathrm{rk}\,\mathrm{Pic}(\widetilde{S}_{\Bbbk})=\mathrm{rk}\,\mathrm{Pic}(S_{\Bbbk})+15$.
Thus, we see that \eqref{equation:main-2-simple} holds,
so that \eqref{equation:main-2} in Main Theorem holds in this case.

\subsection{Family \textnumero $2.8$}
\label{section:r-2-n-8}

One has $h^{1,2}(X)=9$.
In this case, the threefold $X$ is a double cover of the toric Fano threefold obtained by blowing up $\mathbb{P}^3$ at one point.
The ramification surface of this double cover is contained in the anticanonical linear system of this toric Fano threefold.
A~toric Landau--Ginzburg model of the threefold $X$ is given by Minkowski polynomial \textnumero $1968$.
It is
$$
\frac{xy}{z}+2x+\frac{xz}{y}+\frac{2x}{z}+\frac{2x}{y}+\frac{x}{yz}+2y+2z+\frac{2}{z}+\frac{2}{y}+\frac{yz}{x}+\frac{2}{yz}+\frac{2}{x}+\frac{1}{xyz}.
$$
The pencil of quartic surfaces $\mathcal{S}$ is given by
\begin{multline*}
x^2y^2+2x^2zy+x^2z^2+2x^2ty+2x^2tz+x^2t^2+2y^2zx+2z^2yx+\\
+2t^2yx+2t^2zx+y^2z^2+2t^3x+2t^2zy+t^4=\lambda xyzt.
\end{multline*}
This equation is invariant with respect to the permutation $y\leftrightarrow z$.

We may assume that $\lambda\ne\infty$.
Then
\begin{equation}
\label{equation:r2-n8-base-locus}
\begin{split}
H_{\{x\}}\cdot S_\lambda&=2\mathcal{C}_1,\\
H_{\{y\}}\cdot S_\lambda&=2\mathcal{C}_2,\\
H_{\{z\}}\cdot S_\lambda&=2\mathcal{C}_3,\\
H_{\{t\}}\cdot S_\lambda&=2\mathcal{C}_4,
\end{split}
\end{equation}
where $\mathcal{C}_1$ is a smooth conic given by $x=yz+t^2=0$,
the curve $\mathcal{C}_2$ is a smooth conic given by $y=xz+tx+t^2=0$,
the curve $\mathcal{C}_3$ is a smooth conic given by $z=xy+tx+t^2=0$,
and the curve $\mathcal{C}_4$ is a smooth conic given by $t=xy+xz+yz=0$.
Thus, we see that
$$
S_{\lambda}\cdot S_{\infty}=2\mathcal{C}_1+2\mathcal{C}_2+2\mathcal{C}_3+2\mathcal{C}_4,
$$
so that the base locus of the pencil $\mathcal{S}$ consists of $4$ smooth rational curves.
To match the notation introduced in Section~\ref{section:scheme}, we let
$C_1=\mathcal{C}_1$, $C_2=\mathcal{C}_2$, $C_3=\mathcal{C}_3$, $C_4=\mathcal{C}_4$.

If $\lambda\ne -2$, then $S_\lambda$ has isolated singularities, so that it is irreducible.
On the other hand, the surface $S_{-2}$ is not reduced. Indeed, one has
$S_{-2}=2\mathbf{Q}$, where $\mathbf{Q}$ is an irreducible quadric surface in $\mathbb{P}^3$ given by $t^2+tx+xy+xz+yz=0$.
One can check that $\mathbf{Q}$ is smooth.

If $\lambda\ne -2$, then the singular points of the surface $S_\lambda$ contained in the base locus of the pencil $\mathcal{S}$
are the points $P_{\{x\},\{y\},\{t\}}$,  $P_{\{x\},\{z\},\{t\}}$, $P_{\{y\},\{z\},\{t\}}$, and $P_{\{y\},\{z\},\{x,t\}}$.
In this case, the surface $S_\lambda$ has du Val singularities at these points.
In fact, we can say more.

\begin{lemma}
\label{lemma:r2-n8-singularities}
If $\lambda\ne-2$, then the singular points of the surface $S_\lambda$ contained in the base locus of the pencil $\mathcal{S}$ can be describes as follows:
\begin{itemize}\setlength{\itemindent}{3cm}
\item[$P_{\{x\},\{y\},\{t\}}$:] type $\mathbb{D}_6$ with quadratic term $(x+y)^2$;
\item[$P_{\{x\},\{z\},\{t\}}$:] type $\mathbb{D}_6$ with quadratic term $(x+z)^2$;
\item[$P_{\{y\},\{z\},\{t\}}$:] type $\mathbb{D}_4$ with quadratic term $(y+z+t)^2$;
\item[$P_{\{y\},\{z\},\{x,t\}}$:] type $\mathbb{A}_1$ with quadratic term $(x+t-y-z)^2+(\lambda+2)yz$.
\end{itemize}
\end{lemma}

\begin{proof}
We skip the computations of the quadratic terms, because they are straightforward.
If $\lambda\ne -2$, then $P_{\{y\},\{z\},\{x,t\}}$ is an isolated ordinary double point of the surface~$S_\lambda$.
Note that the expressions for quadratic terms are also valid for $\lambda=-2$.

Let us describe the singularity type of the point $P_{\{y\},\{z\},\{t\}}$.
In the chart $x=1$, the surface $S_\lambda$ is given by
\begin{multline*}
\hat{t}^2+2\hat{t}^3-4\hat{t}^2\hat{y}-4\hat{t}^2\hat{z}+2\hat{y}^2\hat{t}+(4-\lambda)\hat{t}\hat{y}\hat{z}+2\hat{z}^2\hat{t}+(2+\lambda)\hat{y}^2\hat{z}+\\
+(2+\lambda)\hat{z}^2\hat{y}+\hat{t}^4-4\hat{y}\hat{t}^3-4\hat{z}\hat{t}^3+6\hat{t}^2\hat{y}^2+14\hat{t}^2\hat{y}\hat{z}+6\hat{t}^2\hat{z}^2-4\hat{t}\hat{y}^3-\\
-16\hat{t}\hat{y}^2\hat{z}-16\hat{t}\hat{y}\hat{z}^2-4\hat{t}\hat{z}^3+\hat{y}^4+6\hat{y}^3\hat{z}+11\hat{y}^2\hat{z}^2+6\hat{y}\hat{z}^3+\hat{z}^4=0
\end{multline*}
for $\hat{y}=y$, $\hat{z}=z$, and $\hat{t}=y+z+t$.
Let $\alpha_1\colon U_1\to\mathbb{P}^3$ be the blow up of the point~$P_{\{y\},\{z\},\{t\}}$.
One chart of the blow up $\alpha_1$ is given by the coordinate change $\hat{y}_1=\hat{y}$, $\hat{z}_1=\frac{\hat{z}}{\hat{y}}$, and $\hat{t}_1=\frac{\hat{t}}{\hat{y}}$.
In this chart, the surface $S^1_\lambda$ is given by
\begin{multline*}
\hat{t}_1^2+2\hat{t}_1\hat{y}_1+\hat{y}_1^2+(2+\lambda)\hat{y}_1\hat{z}_1+\Big(6\hat{y}_1^2\hat{z}_1-4\hat{t}_1^2\hat{y}_1-4\hat{y}_1^2\hat{t}_1+(4-\lambda)\hat{t}_1\hat{y}_1\hat{z}_1+(2+\lambda)\hat{z}_1^2\hat{y}_1\Big)+\\
+\Big(2\hat{y}_1\hat{t}_1^3+6\hat{t}_1^2\hat{y}_1^2-4\hat{t}_1^2\hat{y}_1\hat{z}_1-16\hat{t}_1\hat{y}_1^2\hat{z}_1+2\hat{t}_1\hat{y}_1\hat{z}_1^2+11\hat{y}_1^2\hat{z}_1^2\Big)+\\
+(14\hat{t}_1^2\hat{y}_1^2\hat{z}_1-4\hat{t}_1^3\hat{y}_1^2-16\hat{t}_1\hat{y}_1^2\hat{z}_1^2+6\hat{y}_1^2\hat{z}_1^3\Big)+\Big(\hat{t}_1^4\hat{y}_1^2-4\hat{t}_1^3\hat{y}_1^2\hat{z}_1+6\hat{t}_1^2\hat{y}_1^2\hat{z}_1^2-4\hat{t}_1\hat{y}_1^2\hat{z}_1^3+\hat{y}_1^2\hat{z}_1^4\Big)=0,
\end{multline*}
and $\mathbf{E}_1$ is given by $\hat{y}_1=0$.
Let $C_5^1=S^1_{\lambda}\cap\mathbf{E}_1$.
Then $C_5^1$ is the line in $\mathbf{E}_1\cong\mathbb{P}^2$ that is given by $\hat{y}_1=\hat{t}_1=0$.
Note that $\mathbf{M}_5^{-2}=2$.
If $\lambda\ne -2$, then the only singular points of the surface $S^1_\lambda$ contained in $C_5^1$ are
the points $(\hat{y}_1,\hat{z}_1,\hat{t}_1)=(0,0,0)$ and $(\hat{y}_1,\hat{z}_1,\hat{t}_1)=(0,0,-1)$.
Both of them are isolated ordinary double points of the surface $S_{\lambda}$ in this case.

If $\lambda\ne -2$, then $S^1_\lambda$ has three isolated ordinary double points in $C_5^1$.
Two of them we have already described. The third one can be seen in another chart of the blow up~$\alpha_1$.
Thus, if $\lambda\ne -2$, then $P_{\{y\},\{z\},\{t\}}$ is a singular point of type $\mathbb{D}_4$ of the surface $S_\lambda$.

Now let us describe the singularity of the surface $S_\lambda$ at the point $P_{\{x\},\{z\},\{t\}}$.
In the chart $y=1$, the surface $S_\lambda$ is given by
\begin{multline*}
\bar{z}^2+\Big(2\bar{t}^2\bar{z}+(2+\lambda)\bar{x}^2\bar{t}-\lambda\bar{t}\bar{x}\bar{z}-2\bar{x}^2\bar{z}+2\bar{x}\bar{z}^2\Big)+\\
+\Big(\bar{t}^4+2\bar{t}^3\bar{x}-\bar{t}^2\bar{x}^2+2\bar{t}^2\bar{x}\bar{z}-2\bar{x}^3\bar{t}+2\bar{t}\bar{x}^2\bar{z}+\bar{x}^4-2\bar{x}^3\bar{z}+\bar{x}^2\bar{z}^2\Big)=0,
\end{multline*}
where $\bar{x}=x$, $\bar{z}=x+z$, and $\bar{t}=t$.
Let $\alpha_{2}\colon U_2\to U_1$ be the blow up of the preimage of the point $P_{\{x\},\{z\},\{t\}}$.
A chart of this blow up is given by the coordinate change $\bar{x}_2=\frac{\bar{x}}{\bar{t}}$, $\bar{z}_2=\frac{\bar{z}}{\bar{t}}$, and $\bar{t}_2=\bar{t}$.
In this chart, the surface $S_{\lambda}^2$ is given by
\begin{multline*}
\big(\bar{t}_2+\bar{z}_2\big)^2+\Big(2\bar{t}_2^2\bar{x}_2+(2+\lambda)\bar{x}_2^2\bar{t}_2-\lambda\bar{t}_2\bar{x}_2\bar{z}_2\Big)+(2\bar{t}_2^2\bar{x}_2\bar{z}_2-\bar{t}_2^2\bar{x}_2^2-2\bar{t}_2\bar{x}_2^2\bar{z}_2+2\bar{t}_2\bar{x}_2\bar{z}_2^2\Big)+\\
+\Big(2\bar{t}_2^2\bar{x}_2^2\bar{z}_2-2\bar{t}_2^2\bar{x}_2^3\Big)+\Big(\bar{t}_2^2\bar{x}_2^4-2\bar{t}_2^2\bar{x}_2^3\bar{z}_2+\bar{t}_2^2\bar{x}_2^2\bar{z}_2^2\Big)=0
\end{multline*}
and the surface $\mathbf{E}_2$ is given by $\bar{t}_2=0$.
Let $C_6^2=S^2_{\lambda}\cap\mathbf{E}_2$.
Then $C_6^2$ is the line in $\mathbf{E}_2\cong\mathbb{P}^2$ that is given by $\bar{t}_2=\bar{z}_2=0$.
Observe that $\mathbf{M}_6^{-2}=2$.
On the other hand, if $\lambda\ne -2$, then the point $(\bar{x}_2,\bar{z}_2,\bar{t}_2)=(0,0,0)$
is the only singular point of the surface $S^2_\lambda$ that is contained in the curve $C_6^2$ in this chart.
Note that~$C_6^2$ contains another singular point of  the surface $S^2_\lambda$
that can be seen in another chart of the blow up $\alpha_2$.
This point is an isolated ordinary double singularity of the  surface $S^2_\lambda$.

To determine the type of the singular point $(\bar{x}_2,\bar{z}_2,\bar{t}_2)=(0,0,0)$ on the surface $S^2_\lambda$ for every $\lambda\ne -2$,
we let  $\tilde{x}_{2}=\bar{x}_2$, $\tilde{z}_{2}=\bar{z}_2$, and $\tilde{t}_{2}=\bar{t}_2+\bar{z}_2$.
Then we can rewrite the defining equation of the surface $S_\lambda$ as
\begin{multline*}
\tilde{t}_2^2+\Big(2\tilde{t}_2^2\tilde{x}_2-(2+\lambda)\tilde{x}_2^2\tilde{z}_2+(2+\lambda)\tilde{x}_2^2\tilde{t}_2+(2+\lambda)\tilde{x}_2\tilde{z}_2^2-(\lambda+4)\tilde{t}_2\tilde{x}_2\tilde{z}_2\Big)+\\
+\Big(2\tilde{t}_2^2\tilde{x}_2\tilde{z}_2-\tilde{t}_2^2\tilde{x}_2^2-2\tilde{t}_2\tilde{x}_2\tilde{z}_2^2+\tilde{x}_2^2\tilde{z}_2^2\Big)+\Big(2\tilde{t}_2^2\tilde{x}_2^2\tilde{z}_2-2\tilde{t}_2^2\tilde{x}_2^3+4\tilde{t}_2\tilde{x}_2^3\tilde{z}_2-4\tilde{t}_2\tilde{x}_2^2\tilde{z}_2^2-2\tilde{x}_2^3\tilde{z}_2^2+2\tilde{x}_2^2\tilde{z}_2^3\Big)+\\
+\Big(\tilde{t}_2^2\tilde{x}_2^4-2\tilde{t}_2^2\tilde{x}_2^3\tilde{z}_2+\tilde{t}_2^2\tilde{x}_2^2\tilde{z}_2^2-2\tilde{t}_2\tilde{x}_2^4\tilde{z}_2+4\tilde{t}_2\tilde{x}_2^3\tilde{z}_2^2-2\tilde{t}_2\tilde{x}_2^2\tilde{z}_2^3+\tilde{x}_2^4\tilde{z}_2^2-2\tilde{x}_2^3\tilde{z}_2^3+\tilde{x}_2^2\tilde{z}_2^4\Big)=0.
\end{multline*}
Let $\alpha_{3}\colon U_3\to U_2$ be the blow up of the point $(\tilde{x}_{2},\tilde{z}_{2},\tilde{t}_{2})=(0,0,0)$.
A chart of this blow up is given by the coordinate change $\tilde{x}_3=\tilde{x}_2$, $\tilde{z}_3=\frac{\tilde{z}_2}{\tilde{t}_2}$, and $\tilde{t}_3=\frac{\tilde{t}_2}{\tilde{x}_2}$.
In this chart, the surface $S_{\lambda}^3$ is given by
\begin{multline*}
(2+\lambda)\tilde{z}_3\tilde{x}_3-(2+\lambda)\tilde{t}_3\tilde{x}_3-\tilde{t}_3^2=2\tilde{t}_3^2\tilde{x}_3+(2+\lambda)\tilde{x}_3\tilde{z}_3^2-(\lambda+4)\tilde{t}_3\tilde{x}_3\tilde{z}_3+\tilde{x}_3^2\tilde{z}_3^2-\tilde{t}_3^2\tilde{x}_3^2+\\
+2\tilde{t}_3^2\tilde{x}_3^2\tilde{z}_3-2\tilde{t}_3^2\tilde{x}_3^3+4\tilde{t}_3\tilde{x}_3^3\tilde{z}_3-2\tilde{t}_3\tilde{x}_3^2\tilde{z}_3^2-2\tilde{x}_3^3\tilde{z}_3^2+\tilde{t}_3^2\tilde{x}_3^4+2\tilde{t}_3^2\tilde{x}_3^3\tilde{z}_3-2\tilde{t}_3\tilde{x}_3^4\tilde{z}_3-\\
-4\tilde{t}_3\tilde{x}_3^3\tilde{z}_3^2+\tilde{x}_3^4\tilde{z}_3^2+2\tilde{x}_3^3\tilde{z}_3^3+4\tilde{t}_3\tilde{x}_3^4\tilde{z}_3^2-2\tilde{t}_3^2\tilde{x}_3^4\tilde{z}_3-2\tilde{x}_3^4\tilde{z}_3^3+\tilde{t}_3^2\tilde{x}_3^4\tilde{z}_3^2-2\tilde{t}_3\tilde{x}_3^4\tilde{z}_3^3+\tilde{x}_3^4\tilde{z}_3^4,
\end{multline*}
and the surface $\mathbf{E}_3$ is given by $\tilde{x}_3=0$.
Let $C_7^3=S^3_{\lambda}\cap\mathbf{E}_3$.
Then $C_7^3$ is the line in $\mathbf{E}_3\cong\mathbb{P}^2$ that is given by $\tilde{x}_3=\tilde{t}_3=0$ in our chart of the blow up $\alpha_3$.
Observe that $\mathbf{M}_{7}^{-2}=2$.

If $\lambda\ne -2$, then the point $(\tilde{x}_3,\tilde{z}_3,\tilde{t}_3)=(0,0,0)$ is an isolated ordinary double point of the surface $S^2_\lambda$.
This point is contained in the curve~$C_7^3$.
Moreover, this curve contains two more singular points of  the surface $S^2_\lambda$.
One of them is the point $(\tilde{x}_3,\tilde{z}_3,\tilde{t}_3)=(0,-1,0)$,
and the other one lies in another chart of the blow up~$\alpha_3$.
If~$\lambda\ne -2$, both these points are isolated ordinary double points of the surface~$S^3_\lambda$.
This means that the surface $S^2_\lambda$ has du Val singularity of type $\mathbb{D}_4$ at the point $(\bar{x}_2,\bar{z}_2,\bar{t}_2)=(0,0,0)$,
so that $S_\lambda$ has du Val singularity of type $\mathbb{D}_6$ at the point $P_{\{x\},\{z\},\{t\}}$ for every $\lambda\ne -2$.

Keeping in mind that the defining equation of the surface $S_\lambda$ is symmetric with respect to permutation $y\leftrightarrow z$,
we see that the surface $S_\lambda$ has du Val singularity of type $\mathbb{D}_6$ at the point $P_{\{x\},\{y\},\{t\}}$ for every $\lambda\ne -2$.
This complete the proof of the lemma.
\end{proof}

The proof of this lemma also gives
$\mathrm{rk}\,\mathrm{Pic}(\widetilde{S}_{\Bbbk})=\mathrm{rk}\,\mathrm{Pic}(S_{\Bbbk})+17$, which implies \eqref{equation:main-2-simple}.
Indeed, if $\lambda\ne -2$, then $2\mathcal{C}_1\sim 2\mathcal{C}_2\sim 2\mathcal{C}_3\sim 2\mathcal{C}_4\sim H_{\lambda}$ on the surface~$S_\lambda$ by \eqref{equation:r2-n8-base-locus},
so that the intersection matrix of the conics $\mathcal{C}_1$, $\mathcal{C}_2$, $\mathcal{C}_3$, and $\mathcal{C}_4$
on the surface $S_\lambda$ has rank $1$.
Thus, we see that \eqref{equation:main-2-simple} holds,
so that \eqref{equation:main-2} in Main Theorem holds in this case.

\begin{lemma}
\label{lemma:r2-n8-main-1}
If $\lambda\ne -2$, then $[\mathsf{f}^{-1}(\lambda)]=1$.
One also has $[\mathsf{f}^{-1}(-2)]=10$.
\end{lemma}

\begin{proof}
Using Lemma~\ref{lemma:r2-n8-singularities} and Corollary~\ref{corollary:irreducible-fibers},
we see that $[\mathsf{f}^{-1}(\lambda)]=1$ for $\lambda\ne -2$.
To~show that $[\mathsf{f}^{-1}(-2)]=10$,
let us describe the birational morphism $\alpha$ in \eqref{equation:main-diagram}.
Implicitly, this was already done in the proof of Lemma~\ref{lemma:r2-n8-singularities}.
Because of this, we will use the notations introduced in that proof.

Let $\alpha_4\colon U_4\to U_3$ be the blow up of the preimage of the point $P_{\{x\},\{y\},\{t\}}$.
If~$\lambda\ne -2$, then the surface $S_{\lambda}^4$ has a unique singular point (of type $\mathbb{D}_4$) that is contained in~$\mathbf{E}_4$.
Let~$\alpha_5\colon U_5\to U_4$ be the blow up of this singular point.
Then $S_{\lambda}^5$ has $12$ singular points for general $\lambda\in\mathbb{C}$.
One of them is $P_{\{y\},\{z\},\{x,t\}}$,
another three are contained in the surface~$\mathbf{E}_5$,
another one is contained in the surface~$\mathbf{E}_4^5$,
and the remaining seven were explicitly described in the proof of Lemma~\ref{lemma:r2-n8-singularities}.
All these $12$ points are isolated ordinary double points of the surface $S_{\lambda}^5$ provided that $\lambda\ne -2$.
Thus, there exists a commutative diagram
$$
\xymatrix{
&&U_3\ar@{->}[dll]_{\alpha_3}&&U_4\ar@{->}[ll]_{\alpha_4}&&U_5\ar@{->}[ll]_{\alpha_5}\\
U_2\ar@{->}[rrd]_{\alpha_2}&&&&&&&&U\ar@{->}[dll]^\alpha\ar@{->}[ull]_{\gamma}\\
&&U_1\ar@{->}[rrrr]_{\alpha_1}&&&&\mathbb{P}^3&&}
$$
where $\gamma\colon U\to U_5$ is the blow up of these $12$ points.
This gives $\widehat{D}_{\lambda}=\widehat{S}_{\lambda}$ for every $\lambda\in\mathbb{C}$.

Let us describe the base curves of the pencil~$\widehat{\mathcal{S}}$.
Four of them are $\widehat{C}_1$, $\widehat{C}_2$, $\widehat{C}_3$, and~$\widehat{C}_4$.
The~next three are the curves $\widehat{C}_{5}$, $\widehat{C}_{6}$, $\widehat{C}_{7}$,
which are described in the proof of Lemma~\ref{lemma:r2-n8-singularities}.
The pencil $\widehat{\mathcal{S}}$ contains two more base curves, whose construction is similar to the construction of the curves $\widehat{C}_6$ and $\widehat{C}_{7}$.
One of them is contained in the surface $\widehat{E}_4$,
and another one is contained in the surface $\widehat{E}_5$.
Denote the former curve by $\widehat{C}_{8}$, and denote the latter curve by $\widehat{C}_{9}$.
Then $\widehat{C}_1$, $\widehat{C}_2$, $\widehat{C}_3$, $\widehat{C}_4$,
$\widehat{C}_5$, $\widehat{C}_6$, $\widehat{C}_7$, $\widehat{C}_8$, $\widehat{C}_{9}$
are all base curves of the pencil $\widehat{\mathcal{S}}$.

Note that $\mathbf{m}_1=\mathbf{m}_2=\mathbf{m}_3=\mathbf{m}_4=\mathbf{m}_5=\mathbf{m}_6=\mathbf{m}_7=\mathbf{m}_8=\mathbf{m}_9=2$ and
$$
\mathbf{M}^{-2}_1=\mathbf{M}^{-2}_2=\mathbf{M}^{-2}_3=\mathbf{M}^{-2}_4=\mathbf{M}^{-2}_5=\mathbf{M}^{-2}_6=\mathbf{M}^{-2}_7=\mathbf{M}^{-2}_8=\mathbf{M}^{-2}_9=2.
$$
Thus, using \eqref{equation:equation:number-of-irredubicle-components-refined-2} and Lemma~\ref{lemma:main-2}, we conclude that $[\mathsf{f}^{-1}(-2)]=10$.
\end{proof}

Using Lemma~\ref{lemma:r2-n8-main-1}, we see that \eqref{equation:main-1} in Main Theorem holds in this case.

\subsection{Family \textnumero $2.9$}
\label{section:r-2-n-9}

In this case, the threefold $X$ is a blow up of $\mathbb{P}^3$ at a smooth curve of degree $7$ and genus $5$.
Thus, we have $h^{1,2}(X)=5$.
A~toric Landau--Ginzburg model of the threefold $X$ is given by Minkowski polynomial \textnumero $3013$, which is
$$
x+y+z+\frac{x}{z}+\frac{y}{z}+\frac{x}{y}+\frac{y}{x}+2\frac{z}{y}+2\frac{z}{x}+\frac{z^2}{xy}+\frac{x}{yz}+\frac{2}{z}+\frac{y}{xz}+\frac{2}{y}+\frac{2}{x}+\frac{z}{xy}.
$$
The corresponding pencil $\mathcal{S}$ is given by
\begin{multline*}
t^2x^2+2t^2xy+2t^2xz+t^2y^2+2t^2yz+t^2z^2+tx^2y+tx^2z+txy^2+\\
+2txz^2+ty^2z+2tyz^2+tz^3+x^2yz+xy^2z+xyz^2=\lambda xyzt.
\end{multline*}
Observe that this equation is invariant with respect to the permutation $x\leftrightarrow y$.

We may assume that $\lambda\ne\infty$.
To describe the base curves of the pencil $\mathcal{S}$,
let $\mathcal{C}$ be a smooth conic that is given by $z=xy+tx+ty=0$.
Then
\begin{equation}
\label{equation:r2-n9-base-locus}
\begin{split}
H_{\{x\}}\cdot S_\lambda&=L_{\{x\},\{t\}}+2L_{\{x\},\{y,z\}}+L_{\{x\},\{z,t\}},\\
H_{\{y\}}\cdot S_\lambda&=L_{\{y\},\{t\}}+2L_{\{y\},\{x,z\}}+L_{\{y\},\{z,t\}},\\
H_{\{z\}}\cdot S_\lambda&=L_{\{z\},\{t\}}+L_{\{z\},\{x,y\}}+\mathcal{C},\\
H_{\{t\}}\cdot S_\lambda&=L_{\{x\},\{t\}}+L_{\{y\},\{t\}}+L_{\{z\},\{t\}}+L_{\{t\},\{x,y,z\}}.
\end{split}
\end{equation}
Thus, we let $C_1=L_{\{x\},\{t\}}$, \mbox{$C_2=L_{\{y\},\{t\}}$}, \mbox{$C_3=L_{\{z\},\{t\}}$},~\mbox{$C_4=L_{\{x\},\{y,z\}}$}, $C_5=L_{\{y\},\{x,z\}}$, \mbox{$C_6=L_{\{x\},\{z,t\}}$}, $C_7=L_{\{y\},\{z,t\}}$,
\mbox{$C_8=L_{\{z\},\{x,y\}}$}, \mbox{$C_9=L_{\{t\},\{x,y,z\}}$}, and $C_{10}=\mathcal{C}$.
These are all base curves of the pencil~$\mathcal{S}$.

If $\lambda\ne -3$, then $S_\lambda$ has isolated singularities, so that it is irreducible.
On the other hand, we have $S_{-3}=H_{\{z,t\}}+H_{\{x,y,z\}}+\mathbf{Q}$,
where $\mathbf{Q}$ is a smooth quadric surface that is given by $xy+t(x+y+z)=0$.
Note that $S_{-3}$ is singular along $L_{\{x\},\{y,z\}}$ and $L_{\{y\},\{x,z\}}$,
and it is smooth at general points of the remaining base curves of the pencil $\mathcal{S}$.

If $\lambda\ne -3$, then the singular points of the surface $S_\lambda$ contained in the base locus of the pencil $\mathcal{S}$
are $P_{\{x\},\{y\},\{z\}}$,  $P_{\{x\},\{z\},\{t\}}$, $P_{\{y\},\{z\},\{t\}}$, $P_{\{x\},\{t\},\{y,z\}}$, $P_{\{y\},\{t\},\{x,z\}}$, and $P_{\{z\},\{t\},\{x,y\}}$.
In this case, all of them are du Val singular points of the surface $S_\lambda$ by the following.

\begin{lemma}
\label{lemma:r2-n9-singularities}
If $\lambda\ne-3$, then the singular points of the surface $S_\lambda$ contained in the base locus of the pencil $\mathcal{S}$ can be describes as follows:
\begin{itemize}\setlength{\itemindent}{3cm}
\item[$P_{\{x\},\{y\},\{z\}}$:] type $\mathbb{D}_4$ with quadratic term $(x+y+z)^2$;
\item[$P_{\{x\},\{z\},\{t\}}$:] type $\mathbb{A}_2$ with quadratic term $(x+t)(z+t)$;
\item[$P_{\{y\},\{z\},\{t\}}$:] type $\mathbb{A}_2$ with quadratic term $(y+t)(z+t)$;
\item[$P_{\{x\},\{t\},\{y,z\}}$:] type $\mathbb{A}_2$ with quadratic term $x(x+y+z-(\lambda+3)t)$;
\item[$P_{\{y\},\{t\},\{x,z\}}$:] type $\mathbb{A}_2$ with quadratic term $y(x+y+z-(\lambda+3)t)$;
\item[$P_{\{z\},\{t\},\{x,y\}}$:] type $\mathbb{A}_1$ with quadratic term $(x+y)(t+z)+z^2+(\lambda+2)tz$.
\end{itemize}
\end{lemma}

\begin{proof}
The proof is similar to the proof of Lemma~\ref{lemma:r2-n8-singularities}.
Because of this, we will only prove that $S_\lambda$ has du Val singularity of type $\mathbb{D}_4$
at the point $P_{\{x\},\{y\},\{z\}}$ for every $\lambda\ne -3$.
To do this, we rewrite the defining equation of the surface $S_\lambda$ in the chart $t=1$ as
$$
\bar{z}^2+(3+\lambda)\bar{x}^2\bar{y}+(3+\lambda)\bar{x}\bar{y}^2-(\lambda+2)\bar{x}\bar{y}\bar{z}-\bar{x}\bar{z}^2-\bar{y}\bar{z}^2+\bar{z}^3-\bar{y}\bar{x}^2\bar{z}-\bar{y}^2\bar{x}\bar{z}+\bar{y}\bar{x}\bar{z}^2=0,
$$
where $\bar{x}=x$, $\bar{y}=y$, and $\bar{z}=x+y+z$.

Let $\alpha_1\colon U_1\to\mathbb{P}^3$ be the blow up of the point $P_{\{x\},\{y\},\{z\}}$.
A chart of this is given by the coordinate change $\bar{x}_1=\bar{x}$, $\bar{y}_1=\frac{\bar{y}}{\bar{x}}$, $\bar{z}_1=\frac{\bar{z}}{\bar{z}}$.
In this chart, the surface $S^1_\lambda$ is given by
$$
(3+\lambda)\bar{x}_1\bar{y}_1+\bar{z}_1^2=\bar{x}_1\bar{z}_1^2+(3+\lambda)\bar{x}_1\bar{y}_1^2+(\lambda+2)\bar{x}_1\bar{y}_1\bar{z}_1-\bar{x}_1\bar{z}_1^3+\bar{y}_1\bar{x}_1^2\bar{z}_1+\bar{y}_1\bar{x}_1\bar{z}_1^2-\bar{y}_1\bar{x}_1^2\bar{z}_1^2+\bar{x}_1^2\bar{y}_1^2\bar{z}_1,
$$
and the surface $\mathbf{E}_1$ is given by $\bar{x}_1=0$.

Let $C_{11}^1$ be the line in $\mathbf{E}_1\cong\mathbb{P}^2$  given by $\bar{x}_1=\bar{z}_1=0$.
Then $S^1_\lambda\cdot\mathbf{E}_1=2C_{11}^1$ and $\mathbf{M}_{11}^{-3}=2$.
If $\lambda\ne -3$, then the curve $C_{11}^1$ contains three singular points of the surface $S^1_\lambda$.
One of them is the point $(\bar{x}_1,\bar{y}_1,\bar{z}_1)=(0,0,0)$.
Another one is the point $(\bar{x}_1,\bar{y}_1,\bar{z}_1)=(0,-1,0)$.
The third singular point can be described in another chart of the blow up $\alpha_1$.
All these points are isolated ordinary double points of the surface $S^1_\lambda$ in the case when $\lambda\ne -3$.
Thus, if $\lambda\ne -3$, then $P_{\{x\},\{y\},\{z\}}$ is a singular point of type $\mathbb{D}_4$ of the surface $S_\lambda$.
\end{proof}

The proof of Lemma~\ref{lemma:r2-n9-singularities} implies that $\mathrm{rk}\,\mathrm{Pic}(\widetilde{S}_{\Bbbk})=\mathrm{rk}\,\mathrm{Pic}(S_{\Bbbk})+13$.

If $\lambda\ne -3$, then the intersection form of the curves
$L_{\{x\},\{t\}}$, $L_{\{y\},\{t\}}$, $L_{\{z\},\{t\}}$, $L_{\{x\},\{z,t\}}$, and $L_{\{y\},\{z,t\}}$ on the surface $S_\lambda$ is given by
\begin{center}\renewcommand\arraystretch{1.42}
\begin{tabular}{|c||c|c|c|c|c|}
\hline
$\bullet$ & $L_{\{x\},\{t\}}$ & $L_{\{y\},\{t\}}$ & $L_{\{z\},\{t\}}$ & $L_{\{x\},\{z,t\}}$ & $L_{\{y\},\{z,t\}}$\\
\hline\hline
$L_{\{x\},\{t\}}$& $-\frac{2}{3}$ & $1$ & $\frac{1}{3}$ &   $\frac{1}{3}$ & $0$\\
\hline
$L_{\{y\},\{t\}}$& $1$ & $-\frac{2}{3}$ & $\frac{1}{3}$ & $0$ & $\frac{1}{3}$\\
\hline
$L_{\{z\},\{t\}}$ & $\frac{1}{3}$ & $\frac{1}{3}$ & $-\frac{1}{6}$ & $\frac{2}{3}$ & $\frac{2}{3}$\\
\hline
$L_{\{x\},\{z,t\}}$& $\frac{1}{3}$ & $0$ & $\frac{2}{3}$ & $-\frac{4}{3}$ & $1$ \\
\hline
$L_{\{y\},\{z,t\}}$& $0$ & $\frac{1}{3}$ & $\frac{2}{3}$ & $1$  & $-\frac{4}{3}$ \\
\hline
\end{tabular}
\end{center}
The determinant of this matrix is $\frac{34}{81}$.
This easily gives \eqref{equation:main-2-simple}.
Indeed, the base locus of the pencil $\mathcal{S}$ consists of the lines
$L_{\{x\},\{t\}}$, $L_{\{y\},\{t\}}$, $L_{\{z\},\{t\}}$,
$L_{\{x\},\{y,z\}}$, $L_{\{y\},\{x,z\}}$, $L_{\{x\},\{z,t\}}$, $L_{\{y\},\{z,t\}}$,
$L_{\{z\},\{x,y\}}$, $L_{\{t\},\{x,y,z\}}$, and the conic $\mathcal{C}$.
On the other hand, it follows from \eqref{equation:r2-n9-base-locus} that
\begin{multline*}
H_{\lambda}\sim L_{\{x\},\{t\}}+2L_{\{x\},\{y,z\}}+L_{\{x\},\{z,t\}}\sim L_{\{y\},\{t\}}+2L_{\{y\},\{x,z\}}+L_{\{y\},\{z,t\}}\sim\\
\sim L_{\{z\},\{t\}}+L_{\{z\},\{x,y\}}+\mathcal{C}\sim L_{\{x\},\{t\}}+L_{\{y\},\{t\}}+L_{\{z\},\{t\}}+L_{\{t\},\{x,y,z\}}
\end{multline*}
on the surface $S_\lambda$ provided that $\lambda\ne -3$.
In this case, we also have
$$
H_{\lambda}\sim L_{\{x\},\{y,z\}}+L_{\{y\},\{x,z\}}+L_{\{z\},\{x,y\}}+L_{\{y\},\{x,y,z\}},
$$
because $H_{\{x,y,z\}}\cdot S_\lambda=L_{\{x\},\{y,z\}}+L_{\{y\},\{x,z\}}+L_{\{z\},\{x,y\}}+L_{\{y\},\{x,y,z\}}$.
Likewise, we have
$$
H_{\lambda}\sim L_{\{x\},\{t,z\}}+L_{\{y\},\{t,z\}}+2L_{\{z\},\{t\}},
$$
because $H_{\{z,t\}}\cdot S_\lambda=L_{\{x\},\{t,z\}}+L_{\{y\},\{t,z\}}+2L_{\{z\},\{t\}}$.
Thus, one can express the classes of the curves
$L_{\{x\},\{z,t\}}$, $L_{\{y\},\{z,t\}}$, $L_{\{z\},\{x,y\}}$, $L_{\{t\},\{x,y,z\}}$, and $\mathcal{C}$
in $\mathrm{Pic}(S_\lambda)\otimes\mathbb{Q}$ as linear combinations of the classes of the lines
$L_{\{x\},\{t\}}$, $L_{\{y\},\{t\}}$, $L_{\{z\},\{t\}}$, $L_{\{x\},\{z,t\}}$, and $L_{\{y\},\{z,t\}}$.
For instance, we have
$$
L_{\{t\},\{x,y,z\}}\sim L_{\{x\},\{z,t\}}+L_{\{y\},\{z,t\}}+L_{\{z\},\{t\}}-L_{\{x\},\{t\}}-L_{\{y\},\{t\}}
$$
and $L_{\{z\},\{x,y\}}\sim L_{\{z\},\{t\}}+L_{\{x\},\{t\}}+L_{\{y\},\{t\}}-L_{\{x\},\{z,t\}}-L_{\{y\},\{z,t\}}$.
This shows that the intersection matrix $M$ in Lemma~\ref{lemma:cokernel} has rank $5$,
so that \eqref{equation:main-2-simple} holds in this case.
Thus, we see that \eqref{equation:main-2} in Main Theorem holds in this case.

Therefore, to complete the proof of Main Theorem in this case, we have to prove \eqref{equation:main-1}.
Since $h^{1,2}(X)=5$, the proof is given by the following.

\begin{lemma}
\label{lemma:r2-n9-main-1}
One has $[\mathsf{f}^{-1}(\lambda)]=1$ for every $\lambda\ne -3$.
One also has $[\mathsf{f}^{-1}(-3)]=6$.
\end{lemma}

\begin{proof}
If $\lambda\ne-3$, then we have $[\mathsf{f}^{-1}(\lambda)]=1$ by Lemma~\ref{lemma:r2-n9-singularities} and Corollary~\ref{corollary:irreducible-fibers}.
To show that $[\mathsf{f}^{-1}(-3)]=6$, observe that $\mathbf{M}_4^{-3}=\mathbf{M}_5^{-3}=2$ and
$$
\mathbf{M}_1^{-3}=\mathbf{M}_2^{-3}=\mathbf{M}_3^{-3}=\mathbf{M}_6^{-3}=\mathbf{M}_7^{-3}=\mathbf{M}_8^{-3}=\mathbf{M}_9^{-3}=1.
$$
We also have
$\mathbf{m}_1=\mathbf{m}_2=\mathbf{m}_3=\mathbf{m}_4=\mathbf{m}_5=2$ and $\mathbf{m}_6=\mathbf{m}_7=\mathbf{m}_8=\mathbf{m}_9=1$.

Observe that $[S_{-3}]=3$, and the set $\Sigma$ consists of the points $P_{\{x\},\{y\},\{z\}}$,  $P_{\{x\},\{z\},\{t\}}$, $P_{\{y\},\{z\},\{t\}}$, $P_{\{x\},\{t\},\{y,z\}}$, $P_{\{y\},\{t\},\{x,z\}}$, and $P_{\{z\},\{t\},\{x,y\}}$.
Thus, it follows from \eqref{equation:equation:number-of-irredubicle-components-refined} and Lemma~\ref{lemma:main}
that
$$
\big[\mathsf{f}^{-1}(-3)\big]=5+\sum_{P\in\Sigma}\mathbf{D}_P^{-3}.
$$
Moreover, it follows from the proof of Lemma~\ref{lemma:r2-n9-singularities}
that $P_{\{x\},\{z\},\{t\}}$, $P_{\{y\},\{z\},\{t\}}$, $P_{\{x\},\{t\},\{y,z\}}$, $P_{\{y\},\{t\},\{x,z\}}$, and $P_{\{z\},\{t\},\{x,y\}}$
are {good} double points of the surface $S_{-3}$.
Thus, their defects vanish by Lemma~\ref{lemma:normal-crossing}.
Therefore, we conclude that
$$
\big[\mathsf{f}^{-1}(-3)\big]=5+\mathbf{D}_{P_{\{x\},\{y\},\{z\}}}^{-3}.
$$

Let us show that $\mathbf{D}_{P_{\{x\},\{y\},\{z\}}}^{-3}=1$.
To do this, we use the notation introduced in the proof of Lemma~\ref{lemma:r2-n9-singularities}.
Then there exists a commutative diagram
$$
\xymatrix{
&&&U\ar@{->}[drr]^\alpha\ar@{->}[dll]_{\gamma}\\
&U_1\ar@{->}[rrrr]_{\alpha_1}&&&&\mathbb{P}^3&&}
$$
for some birational morphism $\gamma$.
On the other hand, the curve $\widehat{C}_{11}$ is the only base of the pencil $\widehat{\mathcal{S}}$ that is mapped to $P_{\{x\},\{y\},\{z\}}$ by the birational morphism $\alpha$.
This follows from the proof of Lemma~\ref{lemma:r2-n9-singularities}.
Using Corollary~\ref{corollary:log-pull-back} and \eqref{equation:D-A-B}, we see that $\mathbf{D}_{P_{\{x\},\{y\},\{z\}}}^{-3}=\mathbf{C}_{11}^{-3}$.
By Lemma~\ref{lemma:main-2}, we have $\mathbf{C}_{11}^{-3}=1$, so that $\mathbf{D}_{P_{\{x\},\{y\},\{z\}}}^{-3}=1$ and $[\mathsf{f}^{-1}(-3)]=6$.
\end{proof}

\subsection{Family \textnumero $2.10$}
\label{section:r-2-n-10}

In this case, the threefold $X$ is a blow up of a complete intersection of two quadrics in $\mathbb{P}^3$
at a smooth elliptic curve of degree $4$. This implies that $h^{1,2}(X)=3$.
A~toric Landau--Ginzburg model of the threefold $X$ is given by Minkowski polynomial \textnumero $3018$, which is
$$
x+y+\frac{x}{z}+2z+\frac{yz}{x}+\frac{x}{y}+\frac{y}{x}+\frac{x}{yz}+\frac{2}{z}+\frac{{z}^{2}}{x}+\frac{z}{y}+\frac{3z}{x}+\frac{2}{y}+\frac{3}{x}+\frac{1}{yz}+\frac{1}{xz}.
$$
The quartic pencil $\mathcal{S}$ is given by
\begin{multline*}
x^2zy+y^2zx+x^2ty+2z^2xy+y^2z^2+x^2tz+y^2tz+x^2t^2+2t^2xy+\\
+z^3y+z^2tx+3z^2ty+2t^2zx+3t^2zy+t^3x+t^3y=\lambda xyzt.
\end{multline*}

We may assume that $\lambda\ne\infty$.
Let $\mathcal{C}_{1}$ be a smooth conic given by $x=yz+z^2+2zt+t^2=0$,
and let $\mathcal{C}_{2}$ be a smooth conic given by $z=xy+xt+yt=0$.
Then
\begin{equation}
\label{equation:r2-n10-base-locus}
\begin{split}
H_{\{x\}}\cdot S_\lambda&=L_{\{x\},\{y\}}+L_{\{x\},\{z,t\}}+\mathcal{C}_1,\\
H_{\{y\}}\cdot S_\lambda&=L_{\{x\},\{y\}}+L_{\{y\},\{t\}}+L_{\{y\},\{z,t\}}+L_{\{y\},\{x,z,t\}},\\
H_{\{z\}}\cdot S_\lambda&=L_{\{z\},\{t\}}+L_{\{z\},\{x,t\}}+\mathcal{C}_{2},\\
H_{\{t\}}\cdot S_\lambda&=L_{\{y\},\{t\}}+L_{\{z\},\{t\}}+L_{\{t\},\{x,z\}}+L_{\{t\},\{x,y,z\}}.
\end{split}
\end{equation}
We let $C_1=\mathcal{C}_1$, $C_2=\mathcal{C}_{2}$,
$C_3=L_{\{x\},\{y\}}$, $C_4=L_{\{y\},\{t\}}$, $C_5=L_{\{z\},\{t\}}$, $C_6=L_{\{x\},\{z,t\}}$, $C_7=L_{\{y\},\{z,t\}}$,
$C_8=L_{\{z\},\{x,t\}}$, $C_9=L_{\{t\},\{x,z\}}$, $C_{10}=L_{\{y\},\{x,z,t\}}$, and $C_{11}=L_{\{t\},\{x,y,z\}}$.
These are all base curves of the pencil $\mathcal{S}$.

If $\lambda\ne -4$ and $\lambda\ne -5$, then $S_\lambda$ has isolated singularities, so that it is irreducible.
On the other hand, both surfaces $S_{-4}$ and $S_{-5}$ are reducible. Indeed, one has
$S_{-4}=H_{\{x,z,t\}}+\mathsf{S}$,
where $\mathsf{S}$ is a cubic surface that is given by $t^2x+t^2y+txy+txz+2tyz+xyz+y^2z+yz^2=0$.
Likewise, we have
$S_{-5}=\mathsf{Q}+\mathbf{Q}$
where $\mathbf{Q}$ and $\mathsf{Q}$ are quadric surfaces that are given by the equations
$t^2+tx+2tz+xz+yz+z^2=0$ and $tx+ty+xy+yz=0$, respectively.

Both quadric surfaces $\mathbf{Q}$ and $\mathsf{Q}$ are smooth.
On the other hand, the surface $\mathsf{S}$ has two singular points: the points $P_{\{x\},\{y\},\{z,t\}}$ and $P_{\{y\},\{z\},\{t\}}$.
One can show that $\mathsf{S}$ has an ordinary double singularity at $P_{\{x\},\{y\},\{z,t\}}$,
and it has a singularity of type $\mathbb{A}_2$ at $P_{\{y\},\{z\},\{t\}}$.

If $\lambda\ne -4$ and $\lambda\ne -5$, then the singular points of the surface $S_\lambda$
contained in the base locus of the pencil $\mathcal{S}$ are the points
$P_{\{x\},\{z\},\{t\}}$, $P_{\{y\},\{z\},\{t\}}$, $P_{\{x\},\{y\},\{z,t\}}$
and $P_{\{y\},\{t\},\{x,z\}}$.
In~this case, all of them are du Val singular points of the surface $S_\lambda$ by

\begin{lemma}
\label{lemma:r2-n10-singularities}
If $\lambda\ne -4$ and $\lambda\ne -5$, then the singular points of the surface $S_\lambda$ contained in the base locus of the pencil $\mathcal{S}$ can be describes as follows:
\begin{itemize}\setlength{\itemindent}{3cm}
\item[$P_{\{x\},\{z\},\{t\}}$:] type $\mathbb{A}_4$ with quadratic term $z(x+z+t)$;
\item[$P_{\{y\},\{z\},\{t\}}$:] type $\mathbb{A}_2$ with quadratic term $(y+t)(z+t)$;
\item[$P_{\{x\},\{y\},\{z,t\}}$:] type $\mathbb{A}_4$ with quadratic term $(s+4)xy$;
\item[$P_{\{y\},\{t\},\{x,z\}}$:] type $\mathbb{A}_2$ with quadratic term $t(x+z+t-(\lambda+4)t)$.
\end{itemize}
\end{lemma}

\begin{proof}
We will only describe the singularity of the surface $S_\lambda$ at the point $P_{\{x\},\{y\},\{z,t\}}$.
To do this, we rewrite the defining equation of the surface $S_\lambda$ in the chart $t=1$ as
$$
(\lambda+4)\bar{x}\bar{y}+\Big(\bar{x}^2\bar{z}-\bar{x}\bar{y}^2-(\lambda+4)\bar{x}\bar{y}\bar{z}+\bar{z}^2\bar{x}-\bar{y}^2\bar{z}\Big)+\Big(\bar{x}^2\bar{z}\bar{y}+\bar{y}^2\bar{z}\bar{x}+2\bar{z}^2\bar{x}\bar{y}+\bar{y}^2\bar{z}^2+\bar{z}^3\bar{y}\Big)=0,
$$
where $\bar{x}=x$, $\bar{y}=y$, and $\bar{z}=z+t$.

Let $\alpha_1\colon U_1\to\mathbb{P}^3$ be the blow up of the point $P_{\{x\},\{y\},\{z,t\}}$.
One chart of this blow up is given by the coordinate change
$\bar{x}_1=\frac{\bar{x}}{\bar{z}}$, $\bar{y}_1=\frac{\bar{y}}{\bar{z}}$, and $\bar{z}_1=\bar{z}$.
In this chart, the surface $\mathbf{E}_1$ is given by $\bar{z}_1=0$.
If $\lambda\ne -4$, then the surface $S^1_\lambda$ is given by
$$
\bar{x}_{1}\big(\bar{z}_{1}+(\lambda+4)\bar{y}_{1}\big)=(\lambda+4)\bar{x}_{1}\bar{y}_{1}\bar{z}_{1}-\bar{x}_{1}^2\bar{z}_{1}+\bar{y}_{1}^2\bar{z}_{1}-\bar{z}_{1}^2\bar{y}_{1}-2\bar{x}_{1}\bar{y}_{1}\bar{z}_{1}^2-\bar{y}_{1}^2\bar{z}_{1}^2+\bar{x}_{1}\bar{y}_{1}^2\bar{z}_{1}-\bar{x}_{1}^2\bar{y}_{1}\bar{z}_{1}^2-\bar{x}_{1}\bar{y}_{1}^2\bar{z}_{1}^2.
$$
If $\lambda=-4$, then this equation defines $D_{-4}^1=S^1_{-4}+\mathbf{E}_1$.

Let $C_{12}^1$ and $C_{13}^1$ be the lines in $\mathbf{E}_1\cong\mathbb{P}^2$
that are is given by $\bar{z}_1=\bar{x}_1=0$ and $\bar{z}_1=\bar{y}_1=0$, respectively.
Then $S^1_{-4}$ does not contain them.

Let $\alpha_2\colon U_2\to U_{1}$ be the blow up of the point $C_{12}^1\cap C_{13}^1$.
If $\lambda\ne -4$ and $\lambda\ne -5$, then the surface $S^2_\lambda$ is smooth along $\mathbf{E}_2$.
Hence, in this case, the surface $S_\lambda$ has a singular point  of type $\mathbb{A}_4$ at $P_{\{x\},\{y\},\{z,t\}}$.
\end{proof}

The base locus of the pencil $\mathcal{S}$ consists of the curves
$L_{\{x\},\{y\}}$, $L_{\{y\},\{t\}}$, $L_{\{z\},\{t\}}$, $L_{\{x\},\{z,t\}}$, $L_{\{y\},\{z,t\}}$,
$L_{\{z\},\{x,t\}}$, $L_{\{t\},\{x,z\}}$, $L_{\{y\},\{x,z,t\}}$, $L_{\{t\},\{x,y,z\}}$, $\mathcal{C}_1$, and $\mathcal{C}_{2}$.
If $\lambda\in\{-4,-5\}$, then
\begin{multline*}
H_{\lambda}\sim L_{\{x\},\{y\}}+L_{\{x\},\{z,t\}}+\mathcal{C}_1\sim L_{\{x\},\{y\}}+L_{\{y\},\{t\}}+L_{\{y\},\{z,t\}}+L_{\{y\},\{x,z,t\}}\sim \\
\sim L_{\{z\},\{t\}}+L_{\{z\},\{x,t\}}+\mathcal{C}_{2}\sim L_{\{y\},\{t\}}+L_{\{z\},\{t\}}+L_{\{t\},\{x,z\}}+L_{\{t\},\{x,y,z\}}
\end{multline*}
on the surface $S_\lambda$. This follows from \eqref{equation:r2-n10-base-locus}.
Moreover, in this case, we also have
$$
H_{\lambda}\sim L_{\{x\},\{z,t\}}+L_{\{z\},\{x,t\}}+L_{\{t\},\{x,z\}}+L_{y\},\{x,z,t\}},
$$
because $H_{\{x,z,t\}}\cdot S_\lambda=L_{\{x\},\{z,t\}}+L_{\{z\},\{x,t\}}+L_{\{t\},\{x,z\}}+L_{y\},\{x,z,t\}}$.
This shows that the intersection matrix $M$ in Lemma~\ref{lemma:cokernel} has the same rank
as the intersection matrix of the curves $L_{\{x\},\{y\}}$, $L_{\{y\},\{t\}}$, $L_{\{z\},\{t\}}$, $L_{\{x\},\{z,t\}}$, $L_{\{y\},\{z,t\}}$,
$L_{\{z\},\{x,t\}}$, and $L_{\{t\},\{x,z\}}$ on the surface $S_\lambda$.
If $\lambda\ne -4$ and $\lambda\ne -5$, then the latter matrix is given by
\begin{center}\renewcommand\arraystretch{1.42}
\begin{tabular}{|c||c|c|c|c|c|c|c|}
\hline
$\bullet$ & $L_{\{x\},\{y\}}$ & $L_{\{y\},\{t\}}$ & $L_{\{z\},\{t\}}$ & $L_{\{x\},\{z,t\}}$ & $L_{\{y\},\{z,t\}}$ & $L_{\{z\},\{x,t\}}$ & $L_{\{t\},\{x,z\}}$ \\
\hline\hline
$L_{\{x\},\{y\}}$& $-\frac{4}{5}$ & $1$ & $0$ & $\frac{3}{5}$ & $\frac{2}{5}$ & $0$ & $0$ \\
\hline
$L_{\{y\},\{t\}}$& $1$ & $-\frac{2}{3}$ & $\frac{1}{3}$ & $0$ & $\frac{1}{3}$ & $0$ & $\frac{2}{3}$\\
\hline
$L_{\{z\},\{t\}}$& $0$ & $\frac{1}{3}$ & $-\frac{8}{5}$ & $\frac{1}{5}$ & $\frac{2}{3}$ & $\frac{3}{5}$ & $\frac{1}{5}$\\
\hline
$L_{\{x\},\{z,t\}}$& $\frac{3}{5}$ & $0$ & $\frac{1}{5}$ & $-\frac{2}{5}$ & $\frac{1}{5}$ & $\frac{2}{5}$ & $\frac{4}{5}$\\
\hline
$L_{\{y\},\{z,t\}}$& $\frac{2}{5}$ & $\frac{1}{3}$ & $\frac{2}{3}$ & $\frac{1}{5}$ & $-\frac{8}{5}$ & $0$ & $0$ \\
\hline
$L_{\{z\},\{x,t\}}$ & $0$ & $0$ & $\frac{3}{5}$ & $\frac{2}{5}$ & $0$& $-\frac{4}{5}$ & $\frac{2}{5}$ \\
\hline
$L_{\{t\},\{x,z\}}$& $0$ & $\frac{2}{3}$ & $\frac{1}{5}$ & $\frac{4}{5}$ & $0$& $\frac{2}{5}$ & $-\frac{8}{5}$\\
\hline
\end{tabular}
\end{center}
The rank of this matrix is $6$.
We see that \eqref{equation:main-2-simple} holds, because
$\mathrm{rk}\,\mathrm{Pic}(\widetilde{S}_{\Bbbk})=\mathrm{rk}\,\mathrm{Pic}(S_{\Bbbk})+12$.
Thus, we conclude that \eqref{equation:main-2} in Main Theorem also holds in this case.

\begin{lemma}
\label{lemma:r2-n10-main-1}
One has $[\mathsf{f}^{-1}(\lambda)]=1$ for $\lambda\not\in\{-4,-5\}$,
$[\mathsf{f}^{-1}(-4)]=3$, and $[\mathsf{f}^{-1}(-5)]=2$.
\end{lemma}

\begin{proof}
If $\lambda\not\in-4$ and $\lambda\not\in-5$,
then $[\mathsf{f}^{-1}(\lambda)]=1$ by Lemma~\ref{lemma:r2-n10-singularities} and Corollary~\ref{corollary:irreducible-fibers}.
Moreover, it follows from Corollary~\ref{corollary:normal-crossing-simple} that
$[\mathsf{f}^{-1}(-5)]=2$,
because
$P_{\{x\},\{z\},\{t\}}$, $P_{\{y\},\{z\},\{t\}}$, $P_{\{x\},\{y\},\{z,t\}}$, and $P_{\{y\},\{t\},\{x,z\}}$
are {good} double points of the surface $S_{-5}$.

To complete the proof, we have to show that $[\mathsf{f}^{-1}(-4)]=3$.
Using \eqref{equation:equation:number-of-irredubicle-components-refined}, we see that
$$
\big[\mathsf{f}^{-1}(-4)\big]=2+\mathbf{D}_{P_{\{x\},\{y\},\{z,t\}}}^{-4}.
$$
Here, we also used Lemmas~\ref{lemma:main} and \ref{lemma:normal-crossing}.

To compute the {defect} $\mathbf{D}_{P_{\{x\},\{y\},\{z,t\}}}^{-4}$, let us use the proof of Lemma~\ref{lemma:r2-n10-singularities} and the notation used in this proof.
First, we have $D_{-4}^2=S^2_{-4}+\mathbf{E}_1^2$, so that $\mathbf{A}_{P_{\{x\},\{y\},\{z,t\}}}^{-4}=1$,
where $\mathbf{A}_{P_{\{x\},\{y\},\{z,t\}}}^{-4}$ is the number defined in \eqref{equation:A}.

The curves $C_{12}^2$ and $C_{13}^2$ are the base curves of the pencil $\mathcal{S}^2$.
Aside of these curves, this pencil has one more base curve contained in $\mathbf{E}_2\cup\mathbf{E}_1^2$.
However, the divisor $D_{-4}^2$ is smooth at general points of these three curves.
Now using Lemma~\ref{lemma:main-2} and \eqref{equation:D-A-B}, we conclude that
$\mathbf{D}_{P_{\{x\},\{y\},\{z,t\}}}^{-4}=\mathbf{A}_{P_{\{x\},\{y\},\{z,t\}}}^{-4}=1$,
so that $[\mathsf{f}^{-1}(-4)]=3$.
\end{proof}

Note that Lemma~\ref{lemma:r2-n10-main-1} implies \eqref{equation:main-1} in Main Theorem, because $h^{1,2}(X)=3$.

\subsection{Family \textnumero $2.11$}
\label{section:r-2-n-11}

In this case, the threefold $X$ is a blow up of a smooth cubic threefold at a line, so that $h^{1,2}(X)=5$.
A~toric Landau--Ginzburg model of the threefold $X$ is given by Minkowski polynomial \textnumero $1700$, which is
$$
y+\frac{x}{z}+\frac{y}{z}+z+\frac{yz}{x}+\frac{2y}{x}+\frac{2z^2}{x}+\frac{x}{yz}+\frac{2}{z}+\frac{y}{xz}+\frac{2z}{y}+\frac{2z}{x}+\frac{z^3}{xy}.
$$
The pencil of quartic surfaces $\mathcal{S}$ is given by the equation
$$
y^2zx+x^2ty+y^2tx+z^2xy+y^2z^2+2y^2tz+2z^3y+x^2t^2+2t^2xy+t^2y^2+2z^2tx+2z^2ty+z^4=\lambda xyzt.
$$
In the remaining part of this subsection, we will assume that $\lambda\ne\infty$.

Let $\mathcal{C}_{1}$ be the smooth conic given by $x=ty+yz+z^2=0$,
let $\mathcal{C}_{2}$ be the a smooth conic given by $y=tx+z^2=0$,
let $\mathcal{C}_{3}$ be the smooth conic given by $z=tx+ty+xy=0$,
and let $\mathcal{C}_{4}$ be the smooth conic given by $t=tx+ty+xy=0$.
Then
\begin{equation}
\label{equation:r2-n11-base-locus}
\begin{split}
H_{\{x\}}\cdot S_\lambda&=2\mathcal{C}_1,\\
H_{\{y\}}\cdot S_\lambda&=2\mathcal{C}_2,\\
H_{\{z\}}\cdot S_\lambda&=L_{\{z\},\{t\}}+L_{\{z\},\{x,y\}}+\mathcal{C}_{3},\\
H_{\{t\}}\cdot S_\lambda&=L_{\{z\},\{t\}}+L_{\{t\},\{y,z\}}+\mathcal{C}_{4},
\end{split}
\end{equation}
so that $L_{\{z\},\{t\}}$, $L_{\{z\},\{x,y\}}$, $L_{\{t\},\{y,z\}}$, $\mathcal{C}_1$, $\mathcal{C}_2$, $\mathcal{C}_{3}$, and $\mathcal{C}_{4}$
are all base curves of the pencil $\mathcal{S}$.
To~match the notation used in Subsection~\ref{subsection:scheme-step-6}, we let
$C_1=\mathcal{C}_1$, $C_2=\mathcal{C}_2$, $C_3=\mathcal{C}_3$, $C_4=\mathcal{C}_4$,
$C_5=L_{\{z\},\{t\}}$, $C_6=L_{\{z\},\{x,y\}}$, $C_7=L_{\{t\},\{y,z\}}$.

If $\lambda\ne -2$, then the surface $S_\lambda$ has isolated singularities, so that $S_\lambda$ is irreducible.
On the other hand, we have $S_{-2}=\mathbf{Q}+\mathsf{Q}$,
where $\mathbf{Q}$ and $\mathsf{Q}$ are irreducible quadric surfaces that are given by the equations
$xy+yz+z^2+xt+yt=0$ and $xt+ty+yz+z^2=0$, respectively. Both these quadric surfaces are smooth.
Note that $\mathbf{Q}\cap\mathsf{Q}=\mathcal{C}_1\cup\mathcal{C}_2$.

If $\lambda\ne -2$, then the singular points of the surface $S_\lambda$
contained in the base locus of the pencil $\mathcal{S}$ are the points
$P_{\{x\},\{z\},\{t\}}$, $P_{\{x\},\{y\},\{z\}}$, $P_{\{y\},\{z\},\{t\}}$, and $P_{\{x\},\{t\},\{y,z\}}$,
which are du Val singular points of the surface $S_\lambda$. In fact, we can say more:

\begin{lemma}
\label{lemma:r2-n11-singularities}
If $\lambda\ne -2$, then the singular points of the surface $S_\lambda$ contained in the base locus of the pencil $\mathcal{S}$ can be describes as follows:
\begin{itemize}\setlength{\itemindent}{3cm}
\item[$P_{\{x\},\{y\},\{z\}}$:] type $\mathbb{D}_6$ with quadratic term $(x+y)^2$;
\item[$P_{\{x\},\{z\},\{t\}}$:] type $\mathbb{A}_3$ with quadratic term $(z+t)(x+z+t)$;
\item[$P_{\{y\},\{z\},\{t\}}$:] type $\mathbb{A}_5$ with quadratic term $t(y+t)$;
\item[$P_{\{x\},\{t\},\{y,z\}}$:] type $\mathbb{A}_1$ with quadratic term $(\lambda+2)xt+(y+z-t)(y+z-x-t)$.
\end{itemize}
\end{lemma}

\begin{proof}
We will only describe the singularity of the surface $S_\lambda$ at the point $P_{\{x\},\{y\},\{z\}}$.
To do this, we rewrite the defining equation of the surface $S_\lambda$ in the chart $t=1$ as
$$
\bar{x}^2+\Big(\bar{x}^2\bar{y}-\lambda\bar{x}\bar{y}\bar{z}+(\lambda+2)\bar{y}^2\bar{z}-\bar{x}\bar{y}^2+2\bar{x}\bar{z}^2\Big)+\Big(\bar{y}^2\bar{z}\bar{x}+\bar{z}^2\bar{x}\bar{y}-\bar{y}^3\bar{z}+2\bar{z}^3\bar{y}+\bar{z}^4\Big)=0,
$$
where $\bar{x}=x+y$, $\bar{y}=y$, and $\bar{z}=z$.
Let $\alpha_1\colon U_1\to\mathbb{P}^3$ be the blow up of the point $P_{\{x\},\{y\},\{z,t\}}$.
One chart of this blow up is given by the coordinate change
$\bar{x}_1=\frac{\bar{x}}{\bar{z}}$, $\bar{y}_1=\frac{\bar{y}}{\bar{z}}$, and $\bar{z}_1=\bar{z}$.
In this chart, the surface $\mathbf{E}_1$ is given by $\bar{z}_1=0$.
Then $S^1_\lambda$ is given by
\begin{multline*}
\big(\bar{x}_1+\bar{z}_1\big)^2+\Big(2\bar{y}_1\bar{z}_1^2-\lambda\bar{x}_1\bar{y}_1\bar{z}_1+(\lambda+2)\bar{y}_1^2\bar{z}_1\Big)+\\
+\Big(\bar{x}_1^2\bar{y}_1\bar{z}_1-\bar{y}_1^2\bar{z}_1\bar{x}_1+\bar{z}_1^2\bar{x}_1\bar{y}_1\Big)+\Big(\bar{x}_1\bar{y}_1^2\bar{z}_1^2-\bar{y}_1^3\bar{z}_1^2\Big)=0.
\end{multline*}

Denote by $C_{8}^1$ the line in $\mathbf{E}_1\cong\mathbb{P}^2$ that is given by $\bar{z}_1=\bar{x}_1=0$.
Then $S^1_{-2}$ is singular along this line.
If $\lambda\ne -2$, then  $S^1_\lambda$ has two singular points in $\mathbf{E}_1$.
One of them is the point $(\bar{x}_1,\bar{y}_1,\bar{z}_1)=(0,0,0)$.
The second singular point lies in another chart of the blow up  $\alpha_1$.
If $\lambda\ne -2$, then this point is an isolated ordinary double point of the surface $S^1_\lambda$.

Let $\hat{x}_1=\bar{x}_1+\bar{z}_1$, $\hat{y}_1=\bar{y}_1$, and $\hat{z}_1=\bar{z}_{1}$.
Then we can rewrite the (local) defining equation of the surface $S^1_\lambda$ as
\begin{multline*}
\hat{x}_1^2+\Big((\lambda+2)\hat{y}_1\hat{z}_1^2-\lambda\hat{x}_1\hat{y}_1\hat{z}_1+(\lambda+2)\hat{y}_1^2\hat{z}_1\Big)+\\
+\Big(\hat{x}_1^2\hat{y}_1\hat{z}_1-\hat{y}_1^2\hat{z}_1\hat{x}_1-\hat{z}_1^2\hat{x}_1\hat{y}_1+\hat{y}_1^2\hat{z}_1^2\Big)+\Big(\hat{x}_1\hat{y}_1^2\hat{z}_1^2-\hat{y}_1^3\hat{z}_1^2-\hat{y}_1^2\hat{z}_1^3\Big)=0.
\end{multline*}

Let $\alpha_2\colon U_2\to U_{1}$ be the blow up of the point $(\hat{x}_1,\hat{y}_1,\hat{z}_1)=(0,0,0)$.
One chart of this blow up  is given by the coordinate change
$\hat{x}_2=\frac{\hat{x}_1}{\hat{z}_1}$, $\hat{y}_2=\frac{\hat{y}_1}{\hat{z}_1}$, and $\hat{z}_2=\hat{z}_1$.
In this chart, the surface $S^2_\lambda$ is given by
\begin{multline*}
\hat{x}_2^2+(\lambda+2)\hat{y}_2\hat{z}_2+\Big((\lambda+2)\hat{y}_2^2\hat{z}_2-\lambda\hat{x}_2\hat{y}_2\hat{z}_2\Big)+\Big(\hat{y}_2^2\hat{z}_2^2-\hat{z}_2^2\hat{x}_2\hat{y}_2\Big)+\\
+\Big(\hat{x}_2^2\hat{y}_2\hat{z}_2^2-\hat{x}_2\hat{y}_2^2\hat{z}_2^2-\hat{y}_2^2\hat{z}_2^3\Big)+\Big(\hat{x}_2\hat{y}_2^2\hat{z}_2^3-\hat{y}_2^3\hat{z}_2^3\Big)=0,
\end{multline*}
and the surface $\mathbf{E}_2$ is given by $\hat{z}_2=0$.
Thus, if $\lambda\ne -2$, then $S^2_\lambda$ has isolated ordinary double singularity at the point $(\hat{x}_2,\hat{y}_2,\hat{z}_2)=(0,0,0)$.

Denote by $C_{9}^2$ the line in $\mathbf{E}_2\cong\mathbb{P}^2$ that is given by $\hat{z}_2=\hat{x}_2=0$.
Then $S^2_{-2}$ is singular along this line.
If $\lambda\ne -2$, then $\mathbf{E}_2$ contains three singular points of the surface  $S^2_\lambda$.
One of them is the point $(\hat{x}_2,\hat{y}_2,\hat{z}_2)=(0,0,0)$.
The second one is $(\hat{x}_2,\hat{y}_2,\hat{z}_2)=(0,-1,0)$.
The  third point is contained in another chart of the blow up $\alpha_2$.
All of them are isolated ordinary double singularities of the surface $S^2_\lambda$.
Hence, if $\lambda\ne -2$, then $S_\lambda$ has a singular point of type $\mathbb{D}_6$ at the point $P_{\{x\},\{y\},\{z\}}$.
\end{proof}

The proof of Lemma~\ref{lemma:r2-n11-singularities} implies
\begin{equation}
\label{equation:r2-n11-singularities-Pic}
\mathrm{rk}\,\mathrm{Pic}(\widetilde{S}_{\Bbbk})=\mathrm{rk}\,\mathrm{Pic}(S_{\Bbbk})+15.
\end{equation}

By Lemma~\ref{lemma:cokernel}, to verify \eqref{equation:main-2} in Main Theorem,
we have to compute the rank of the intersection matrix of the curves
$\mathcal{C}_1$, $\mathcal{C}_2$, $\mathcal{C}_3$, $\mathcal{C}_4$,
$L_{\{z\},\{t\}}$, $L_{\{z\},\{t\}}$, and $L_{\{t\},\{y,z\}}$ on a general surface in the pencil $\mathcal{S}$.
On the other hand, if $\lambda\ne -2$, then it follows from \eqref{equation:r2-n11-base-locus} that
$$
H_{\lambda}\sim 2\mathcal{C}_1\sim 2\mathcal{C}_2\sim L_{\{z\},\{t\}}+L_{\{z\},\{x,y\}}+\mathcal{C}_{3}\sim L_{\{z\},\{t\}}+L_{\{t\},\{y,z\}}+\mathcal{C}_{4}.
$$
on the surface $S_\lambda$.
We have
$\mathcal{C}_1+\mathcal{C}_2+\mathcal{C}_3+\mathcal{C}_4\sim 2H_{\lambda}$,
because $\mathbf{Q}\cdot S_\lambda=\mathcal{C}_1+\mathcal{C}_2+\mathcal{C}_3+\mathcal{C}_4$.
Likewise, we also have $\mathsf{Q}\cdot S_\lambda=2L_{\{z\},\{t\}}+L_{\{z\},\{x,y\}}+L_{\{t\},\{y,z\}}+\mathcal{C}_{1}+\mathcal{C}_{2}$, so that
$$
2L_{\{z\},\{t\}}+L_{\{z\},\{x,y\}}+L_{\{t\},\{y,z\}}+\mathcal{C}_{1}+\mathcal{C}_{2}\sim 2H_{\lambda},
$$
which implies that $2L_{\{z\},\{t\}}+L_{\{z\},\{x,y\}}+L_{\{t\},\{y,z\}}\sim_{\mathbb{Q}} H_{\lambda}$.
Thus, if $\lambda\ne -2$,
then the rank of the intersection matrix of the curves
$\mathcal{C}_1$, $\mathcal{C}_2$, $\mathcal{C}_3$, $\mathcal{C}_4$,
$L_{\{z\},\{t\}}$, $L_{\{z\},\{t\}}$, and $L_{\{t\},\{y,z\}}$ on the surface $S_\lambda$
is the same as the rank of the intersection matrix of the curves $L_{\{z\},\{t\}}$, $L_{\{z\},\{x,y\}}$
and $H_{\lambda}$, which is very easy to compute.

\begin{lemma}
\label{lemma:r2-n11-intersection}
Suppose that $\lambda\ne -2$.
Then the intersection form of the curves
$L_{\{z\},\{t\}}$, $L_{\{z\},\{x,y\}}$
and $H_{\lambda}$ on the surface $S_\lambda$ is given by
\begin{center}\renewcommand\arraystretch{1.42}
\begin{tabular}{|c||c|c|c|}
\hline
$\bullet$ & $L_{\{z\},\{t\}}$ & $L_{\{z\},\{x,y\}}$ & $H_{\lambda}$  \\
\hline\hline
$L_{\{z\},\{t\}}$& $-\frac{5}{12}$ & $1$ & $1$  \\
\hline
$L_{\{z\},\{x,y\}}$& $1$ & $-1$ & $1$ \\
\hline
$H_{\lambda}$& $1$ & $1$ & $4$\\
\hline
\end{tabular}
\end{center}
\end{lemma}

\begin{proof}
Since $L_{\{z\},\{t\}}\cap L_{\{z\},\{x,y\}}=P_{\{z\},\{t\},\{x,y\}}$ and $S_\lambda$ is smooth at this point, we conclude that $L_{\{z\},\{t\}}\cdot L_{\{z\},\{x,y\}}=1$.
So, to complete the proof, we have to find $L_{\{z\},\{t\}}^2$ and $L_{\{z\},\{x,y\}}^2$.

Observe that $P_{\{x\},\{z\},\{t\}}$ and $P_{\{y\},\{z\},\{t\}}$ are the only singular points
of the surface $S_\lambda$ that are contained in the line $L_{\{z\},\{t\}}$.
Applying Remark~\ref{remark:transversal} with $S=S_\lambda$, $O=P_{\{x\},\{z\},\{t\}}$, $n=3$, and $C=L_{\{z\},\{t\}}$,
we see that $\overline{C}$ does not contain the point $\overline{G}_1\cap\overline{G}_3$.
Similarly, applying Remark~\ref{remark:transversal} with $S=S_\lambda$, $O=P_{\{y\},\{z\},\{t\}}$, $n=5$, and $C=L_{\{z\},\{t\}}$,
we see that $\overline{C}$ does not contain the point $\overline{G}_1\cap\overline{G}_5$.
Thus, it follows from Proposition~\ref{proposition:du-Val-self-intersection} that
$$
L_{\{z\},\{t\}}^2=-2+\frac{3}{4}+\frac{5}{6}=-\frac{5}{12}.
$$

Note that $P_{\{x\},\{y\},\{z\}}$ is the only singular point of the surface $S_\lambda$ that is contained in the line $L_{\{z\},\{x,y\}}^2$.
To find $L_{\{z\},\{x,y\}}^2$, let us use the notations of Lemma~\ref{lemma:Dn}
with $S=S_\lambda$, $O=P_{\{x\},\{y\},\{z\}}$, $n=6$, and $C=L_{\{z\},\{x,y\}}$.
Let us also use the notation of the proof of Lemma~\ref{lemma:r2-n11-singularities}.
It follows from this proof that the proper transform of the line $L_{\{z\},\{x,y\}}$ on the surface $S_\lambda^1$ does not contain the point $(\bar{x}_1,\bar{y}_1,\bar{z}_1)=(0,0,0)$.
Thus, in the notation of Lemma~\ref{lemma:Dn}, we have $\widetilde{C}\cdot G_6=1$,
which implies that $L_{\{z\},\{x,y\}}^2=-1$ by Lemma~\ref{lemma:Dn}.
\end{proof}

The determinant of the intersection matrix in Lemma~\ref{lemma:r2-n11-intersection} is $\frac{13}{12}$.
Using \eqref{equation:r2-n11-singularities-Pic}, we get \eqref{equation:main-2-simple},
so that \eqref{equation:main-2} in Main Theorem holds in this case.
Moreover, since $h^{1,2}(X)=5$, the assertion \eqref{equation:main-1} in  Main Theorem is given by

\begin{lemma}
\label{lemma:r2-n11-main-1}
If $\lambda\ne -2$, then $[\mathsf{f}^{-1}(\lambda)]=1$.
One also has $[\mathsf{f}^{-1}(-2)]=6$.
\end{lemma}

\begin{proof}
If $\lambda\not\in-2$,
then $[\mathsf{f}^{-1}(\lambda)]=1$ by Lemma~\ref{lemma:r2-n11-singularities} and Corollary~\ref{corollary:irreducible-fibers}.
Hence, to complete the proof, we must show that $[\mathsf{f}^{-1}(-2)]=6$.
Observe that
\mbox{$\mathbf{m}_{3}=\mathbf{m}_{4}=\mathbf{m}_{6}=\mathbf{m}_{7}=1$}
and $\mathbf{m}_{1}=\mathbf{m}_{2}=\mathbf{m}_{5}=2$.
Note that
$\mathbf{M}_{3}^{-2}=\mathbf{M}_{4}^{-2}=\mathbf{M}_{5}^{-2}=\mathbf{M}_{6}^{-2}=\mathbf{M}_{7}^{-2}=1$
and $\mathbf{M}_{1}^{-2}=\mathbf{M}_{2}^{-2}=2$.
Thus, applying Lemmas~\ref{lemma:main} and \ref{lemma:normal-crossing},
and using \eqref{equation:equation:number-of-irredubicle-components-refined},
we see that
$$
\big[\mathsf{f}^{-1}(-2)\big]=4+\mathbf{D}_{P_{\{x\},\{y\},\{z\}}}^{-2},
$$
where $\mathbf{D}_{P_{\{x\},\{y\},\{z\}}}^{-2}$ is the {defect} of the singular point $P_{\{x\},\{y\},\{z\}}$.

To compute $\mathbf{D}_{P_{\{x\},\{y\},\{z\}}}^{-2}$, we will use local computations done in the proof of Lemma~\ref{lemma:r2-n11-singularities}.
They give $\mathbf{M}_8^{-2}=\mathbf{M}_9^{-2}=2$ and $D_{-2}^2=S^2_{-2}$.
Now using Lemma~\ref{lemma:main-2} and \eqref{equation:D-A-B}, we conclude that
$\mathbf{D}_{P_{\{x\},\{y\},\{z\}}}^{-2}\geqslant 2$.
In fact, the proof of Lemma~\ref{lemma:r2-n11-singularities} implies that $\mathbf{D}_{P_{\{x\},\{y\},\{z\}}}^{-2}=2$, so that $[\mathsf{f}^{-1}(-2)]=6$.
\end{proof}

\subsection{Family \textnumero $2.12$}
\label{section:r-2-n-12}

In this case, the threefold $X$ is a blow up of $\mathbb{P}^3$ along a smooth curve of genus $3$ and degree $6$, so that $h^{1,2}(X)=3$.
Here, we chose its toric Landau--Ginzburg model to be given by Minkowski polynomial \textnumero $1193$, which is
$$
x+\frac{xy}{z}+z+y+\frac{2x}{z}+\frac{2y}{z}+\frac{x}{yz}+\frac{2}{y}+\frac{2}{z}+\frac{z}{xy}+\frac{2}{x}+\frac{y}{xz}.
$$
The quartic pencil $\mathcal{S}$ is given by
\begin{multline*}
x^2zy+x^2y^2+z^2xy+y^2zx+2x^2ty+2y^2tx+x^2t^2+\\
+2t^2zx+2t^2xy+t^2z^2+2t^2zy+t^2y^2=\lambda xyzt.
\end{multline*}
This equation is symmetric with respect to the swapping $x\leftrightarrow y$.

Let $\mathcal{C}$ be a  conic that is given by $z=xy+xt+yt=0$. If $\lambda\ne\infty$, then
\begin{equation}
\label{equation:r2-n12-base-locus}
\begin{split}
H_{\{x\}}\cdot S_\lambda&=2L_{\{x\},\{t\}}+2L_{\{x\},\{y,z\}},\\
H_{\{y\}}\cdot S_\lambda&=2L_{\{y\},\{t\}}+2L_{\{y\},\{x,z\}},\\
H_{\{z\}}\cdot S_\lambda&=2\mathcal{C},\\
H_{\{t\}}\cdot S_\lambda&=L_{\{x\},\{t\}}+L_{\{y\},\{t\}}+L_{\{t\},\{x,z\}}+L_{\{t\},\{y,z\}}.
\end{split}
\end{equation}
Thus, we may assume that
$C_1=L_{\{x\},\{t\}}$, $C_2=L_{\{y\},\{t\}}$, \mbox{$C_3=L_{\{x\},\{y,z\}}$}, $C_4=L_{\{y\},\{x,z\}}$, $C_5=L_{\{t\},\{x,z\}}$,
$C_6=L_{\{t\},\{y,z\}}$, and $C_7=\mathcal{C}$.
These are all base curves of the pencil $\mathcal{S}$.

If $\lambda\ne\infty$ and $\lambda\ne -2$, then the surface $S_\lambda$ has isolated singularities, so that $S_\lambda$ is irreducible.
On the other hand, the surface $S_{-2}$ is singular along  $L_{\{x\},\{y,z\}}$ and $L_{\{y\},\{x,z\}}$.

\begin{lemma}
\label{equation:r2-n12-S-2-irreducible}
The surface $S_{-2}$ is irreducible.
\end{lemma}

\begin{proof}
Let $\Pi$ be a plane in $\mathbb{P}^3$ that is given by $z=t$.
Then the intersection $S_{-2}\cap\Pi$ is a plane quartic curve that is singular
at the points $\Pi\cap L_{\{x\},\{y,z\}}$ and $\Pi\cap L_{\{y\},\{x,z\}}$.
Moreover, this curve is smooth away from these points.
Furthermore, both of these points are isolated ordinary double points of the curve $S_{-2}\cap\Pi$.
This implies that this curve is irreducible, so that
the surface $S_{-2}$ is also irreducible.
\end{proof}

The {fixed} singular points of the surfaces in the pencil $\mathcal{S}$
are $P_{\{x\},\{y\},\{z\}}$, $P_{\{x\},\{y\},\{t\}}$, $P_{\{x\},\{z\},\{t\}}$,
$P_{\{y\},\{z\},\{t\}}$, $P_{\{x\},\{t\},\{y,z\}}$, and $P_{\{y\},\{t\},\{x,z\}}$.
If $\lambda\ne\infty$ and $\lambda\ne -2$,
then the singularities of the surface $S_\lambda$
at these points can be describes as follows:
\begin{itemize}\setlength{\itemindent}{3cm}
\item[$P_{\{x\},\{y\},\{z\}}$:] type $\mathbb{D}_4$ with quadratic term $(x+y+z)^2$;
\item[$P_{\{x\},\{y\},\{t\}}$:] type $\mathbb{A}_1$ with quadratic term $xy+t^2$;
\item[$P_{\{x\},\{z\},\{t\}}$:] type $\mathbb{A}_1$ with quadratic term $x^2+xz+2xt+t^2$;
\item[$P_{\{y\},\{z\},\{t\}}$:] type $\mathbb{A}_1$ with quadratic term $y^2+yz+2yt+t^2$;
\item[$P_{\{x\},\{t\},\{y,z\}}$:] type $\mathbb{A}_3$ with quadratic term $x(y+z-(\lambda+2)t)$;
\item[$P_{\{y\},\{t\},\{x,z\}}$:] type $\mathbb{A}_3$ with quadratic term $y(x+z-(\lambda+2)t)$.
\end{itemize}

The surfaces in $\mathcal{S}$ also have {floating} singular points.
They are contained in the conic~$\mathcal{C}$.
To describe them nicely, we introduce a new parameter $\mu\in\mathbb{C}\cup\{\infty\}$ such that
$$
\lambda=\frac{2\mu^2-2\mu-1}{\mu(1-\mu)}.
$$
Then $S_\lambda$ is singular at the points $[1-\mu:\mu:0:\mu(\mu-1)]$ and $[\mu:1-\mu:0:\mu(\mu-1)]$.
Denote these two points by $P_\mu$ and $P_{1-\mu}$, respectively.
Then $P_\mu\ne P_{1-\mu}$ $\iff$ $\mu\not\in\{\infty,\frac{1}{2}\}$.
If $\mu=\frac{1}{2}$, then $P_\mu=P_{1-\mu}=[-2:-2:0:1]$.
If $\mu=\infty$, then $P_\mu=P_{1-\mu}=P_{\{x\},\{y\},\{z\}}$.

\begin{lemma}
\label{lemma:r2-n12-singularities-floating}
Suppose that $\lambda\ne\infty$ and $\lambda\ne -2$.
If $\mu\ne\frac{1}{2}$, then $S_\lambda$ has isolated ordinary double singularities at the points $P_\mu$ and $P_{1-\mu}$.
If $\mu=\frac{1}{2}$, then $\lambda=-6$, and $S_{-6}$ has a du Val singularity of type $\mathbb{A}_3$ at  the point $P_{\frac{1}{2}}$.
\end{lemma}

\begin{proof}
Due to symmetry $x\leftrightarrow y$, it is enough to describe the singularity of the surface $S_\lambda$ at the point $P_\mu$.
Moreover, we may assume that $\mu\ne 0$ and $\mu\ne 1$, since $P_0=P_{\{y\},\{z\},\{t\}}$ and $P_{1}=P_{\{x\},\{z\},\{t\}}$.
Then $P_\mu=[\frac{1}{\mu}:\frac{1}{\mu-1}:0:1]$.
In the chart $t=1$, the surface $S_\lambda$ is given by
\begin{multline*}
(\mu-1)^4\bar{y}^2+\mu(\mu-1)(\mu^2-\mu-1)\bar{z}^2+\mu(\mu-1)(2\mu^2-1)\bar{z}\bar{x}+2\mu^2(\mu-1)^2\bar{x}\bar{y}+\\
+\mu(\mu-1)(2\mu^2-4\mu+1)\bar{y}\bar{z}+\mu^4\bar{x}^2+\text{higher order terms}=0.
\end{multline*}
where $\bar{x}=x-\frac{1}{\mu}$, $\bar{y}=y-\frac{1}{\mu-1}$, and $\bar{z}=z$.
If $\mu\ne\frac{1}{2}$, this quadratic form is non-degenerate, so that $S_\lambda$ has an isolated ordinary double singularity at $P_\mu$.

To complete the proof, we may assume that $\mu=\frac{1}{2}$.
Then $P_{\frac{1}{2}}=[-2:-2:0:1]$. Note that $\lambda=-6$ in this case.
In the chart $t=1$, the surface $S_{-6}$ is given by
\begin{multline*}
\hat{x}^2+2\hat{z}\hat{x}+5\hat{z}^2+\Big(2\hat{x}\hat{y}^2-2\hat{x}^2\hat{y}-2\hat{x}^2\hat{z}+2\hat{x}\hat{y}\hat{z}-2\hat{x}\hat{z}^2-2\hat{y}^2\hat{z}\Big)+\\
\Big(\hat{x}^2\hat{y}^2+\hat{z}\hat{y}\hat{x}^2-2\hat{y}^3\hat{x}-\hat{z}\hat{y}^2\hat{x}+\hat{z}^2\hat{y}\hat{x}+\hat{y}^4-\hat{y}^2\hat{z}^2\Big)=0,
\end{multline*}
where $\hat{x}=x+y+4$, $\hat{y}=y+2$, and $\hat{z}=z$.
Introducing new coordinates $\hat{x}_1=\frac{\hat{x}}{\hat{y}}$, $\hat{y}_1=\hat{y}$, and $\hat{z}_1=\frac{\hat{z}}{\hat{y}}$,
we can rewrite this equation (after dividing by $\hat{y}_1^2$) as
\begin{multline*}
\hat{x}_1^2+2\hat{x}_1\hat{y}_1+2\hat{z}_1\hat{x}_1+\hat{y}_1^2-2\hat{y}_1\hat{z}_1+5\hat{z}_1^2+\Big(2\hat{x}_1\hat{y}_1\hat{z}_1-2\hat{x}_1^2\hat{y}_1-2\hat{x}_1\hat{y}_1^2\Big)+\\
+\Big(\hat{x}_1^2\hat{y}_1^2-2\hat{z}_1\hat{y}_1\hat{x}_1^2-\hat{z}_1\hat{y}_1^2\hat{x}_1-2\hat{z}_1^2\hat{y}_1\hat{x}_1-\hat{y}_1^2\hat{z}_1^2\Big)+\hat{z}_1\hat{y}_1^2\hat{x}_1^2+\hat{x}_1\hat{y}_1^2\hat{z}_1^2.
\end{multline*}
This equation defines a chart of a blow up of the surface $S_{-6}$ at the point $P_{\frac{1}{2}}$.
Its quadratic part is not degenerate, which shows that $P_{\frac{1}{2}}$ is a du Val singular point of type $\mathbb{A}_3$ of the surface $S_{-6}$.
This completes the proof of the lemma.
\end{proof}

If $\lambda\ne\infty$ and $\lambda\ne-2$, then the singular points of the surface $S_\lambda$
contained in the base locus of the pencil $\mathcal{S}$ can be describes as follows:
\begin{itemize}
\item $P_{\{x\},\{y\},\{z\}}$, $P_{\{x\},\{y\},\{t\}}$, $P_{\{x\},\{z\},\{t\}}$, $P_{\{y\},\{z\},\{t\}}$, $P_{\{x\},\{t\},\{y,z\}}$, and $P_{\{y\},\{t\},\{x,z\}}$;

\item $P_\mu$ and $P_{1-\mu}$, where $\mu\in\mathbb{C}\cup\{\infty\}$ such that $\lambda=\frac{2\mu^2-2\mu-1}{\mu(1-\mu)}$;

\item $P_{\{t\},\{x,z\},\{y,z\}}$, which is an isolated ordinary double point of the surface $S_{-4}$.
\end{itemize}
If $\lambda\ne -4$, then $S_\lambda$ is smooth at the point $P_{\{t\},\{x,z\},\{y,z\}}$.

Note that {fixed} singular points of the quartic surfaces in the pencil $\mathcal{S}$ can be considered as singular points of the surface $S_{\Bbbk}$.
In our case, all exceptional curves of the minimal resolution of the surface $S_{\Bbbk}$ at these singular points are geometrically irreducible.
Likewise, we can consider the union $P_\mu\cup P_{1-\mu}$ as a (geometrically reducible) singular point of the surface $S_{\Bbbk}$.
This gives $\mathrm{rk}\,\mathrm{Pic}(\widetilde{S}_{\Bbbk})=\mathrm{rk}\,\mathrm{Pic}(S_{\Bbbk})+14$.
Thus, to prove \eqref{equation:main-2-simple}, we have to show that
the intersection matrix of the curves $L_{\{x\},\{t\}}$, $L_{\{y\},\{t\}}$, $L_{\{x\},\{y,z\}}$,
$L_{\{y\},\{x,z\}}$, $L_{\{t\},\{x,z\}}$, $L_{\{t\},\{y,z\}}$, and $\mathcal{C}$ on a general surface in $\mathcal{S}$ is of rank~$4$.
If $\lambda\ne\infty$ and $\lambda\ne -2$,
then it follows from \eqref{equation:r2-n12-base-locus} that
\begin{multline*}
H_{\lambda}\sim 2L_{\{x\},\{t\}}+2L_{\{x\},\{y,z\}}\sim 2L_{\{y\},\{t\}}+2L_{\{y\},\{x,z\}}\sim \\
\sim 2\mathcal{C}\sim L_{\{x\},\{t\}}+L_{\{y\},\{t\}}+L_{\{t\},\{x,z\}}+L_{\{t\},\{y,z\}}
\end{multline*}
on the surface $S_\lambda$.
Therefore, in this case,
the intersection matrix of the curves
$L_{\{x\},\{t\}}$, $L_{\{y\},\{t\}}$, $L_{\{x\},\{y,z\}}$,
$L_{\{y\},\{x,z\}}$, $L_{\{t\},\{x,z\}}$, $L_{\{t\},\{y,z\}}$, and $\mathcal{C}$ on the surface $S_\lambda$
has the same rank as the intersection matrix of the four lines
$L_{\{x\},\{t\}}$, $L_{\{y\},\{t\}}$, $L_{\{t\},\{x,z\}}$, and $L_{\{t\},\{y,z\}}$.
If~$\lambda\not\in\{\infty,-2,-4,-6\}$,
then the latter matrix is given by
\begin{center}\renewcommand\arraystretch{1.42}
\begin{tabular}{|c||c|c|c|c|}
\hline
$\bullet$ & $L_{\{x\},\{t\}}$ & $L_{\{y\},\{t\}}$ &  $L_{\{t\},\{x,z\}}$ & $L_{\{t\},\{y,z\}}$  \\
\hline\hline
$L_{\{x\},\{t\}}$& $-\frac{1}{4}$ & $\frac{1}{2}$ & $\frac{1}{2}$ & $\frac{1}{4}$ \\
\hline
$L_{\{y\},\{t\}}$& $\frac{1}{2}$ & $-\frac{1}{4}$ & $\frac{1}{4}$ & $\frac{1}{2}$\\
\hline
$L_{\{t\},\{x,z\}}$& $\frac{1}{2}$ & $\frac{1}{4}$ & $-\frac{3}{4}$ & $1$\\
\hline
$L_{\{t\},\{y,z\}}$& $\frac{1}{4}$ & $\frac{1}{2}$ & $1$ & $-\frac{3}{4}$\\
\hline
\end{tabular}
\end{center}
Its determinant is $-\frac{5}{8}$,
so that \eqref{equation:main-2} in Main Theorem  holds by Lemma~\ref{lemma:cokernel}.

\begin{lemma}
\label{lemma:r2-n12-main-1}
If $\lambda\ne -2$, then $[\mathsf{f}^{-1}(\lambda)]=1$. One also has $[\mathsf{f}^{-1}(-2)]=4$.
\end{lemma}

\begin{proof}
If $\lambda\ne -2$, then $S_\lambda$ has du Val singularities at the base locus of the pencil $\mathcal{S}$,
so that $[\mathsf{f}^{-1}(\lambda)]=1$ by Corollary~\ref{corollary:irreducible-fibers}.
Hence, to complete the proof, we have to show that $[\mathsf{f}^{-1}(-2)]=4$.
To do this, we observe that
$\mathbf{m}_{1}=\mathbf{m}_{2}=\mathbf{m}_{3}=\mathbf{m}_{4}=\mathbf{m}_{7}=2$ and $\mathbf{m}_{5}=\mathbf{m}_{6}=1$.
Similarly, we have
$\mathbf{M}_{1}^{-2}=\mathbf{M}_{2}^{-2}=\mathbf{M}_{5}^{-2}=\mathbf{M}_{6}^{-2}=\mathbf{M}_{7}^{-2}=1$
and $\mathbf{M}_{3}^{-2}=\mathbf{M}_{4}^{-2}=2$.
Thus, using \eqref{equation:equation:number-of-irredubicle-components-refined} and Lemma~\ref{lemma:main},
we see that
$$
\big[\mathsf{f}^{-1}(-2)\big]=3+\sum_{P\in\Sigma}\mathbf{D}_{P}^{-2},
$$
where $\Sigma$ is the set consisting of the  points
$P_{\{x\},\{y\},\{z\}}$, $P_{\{x\},\{y\},\{t\}}$, $P_{\{x\},\{z\},\{t\}}$, $P_{\{y\},\{z\},\{t\}}$, $P_{\{x\},\{t\},\{y,z\}}$, and $P_{\{y\},\{t\},\{x,z\}}$.
Using Lemma~\ref{lemma:normal-crossing}, we see that
$$
\mathbf{D}_{P_{\{x\},\{y\},\{t\}}}^{-2}=\mathbf{D}_{P_{\{x\},\{z\},\{t\}}}^{-2}=\mathbf{D}_{P_{\{y\},\{z\},\{t\}}}^{-2}=\mathbf{D}_{P_{\{x\},\{t\},\{y,z\}}}^{-2}=\mathbf{D}_{P_{\{y\},\{t\},\{x,z\}}}^{-2}=0.
$$
Thus, we conclude that $[\mathsf{f}^{-1}(-2)]=3+\mathbf{D}_{P_{\{x\},\{y\},\{z\}}}^{-2}$.
Let us show that $\mathbf{D}_{P_{\{x\},\{y\},\{z\}}}^{-2}=1$.

Let $\alpha_1\colon U_1\to\mathbb{P}^3$ be a blow up of the point $P_{\{x\},\{y\},\{z\}}$,
Then $D_{\lambda}^1=S_{\lambda}^1$ for every $\lambda\ne\infty$.
Moreover, the surface $\mathbf{E}_1$ contains a unique base curve of the pencil $\mathcal{S}^1$.
Denote it by $C_{8}^1$. Then $\mathbf{m}_{8}=2$ and $\mathbf{M}_8^{-2}=2$.
Thus, using \eqref{equation:D-A-B}, we see that  $\mathbf{D}_{P_{\{x\},\{y\},\{z\}}}^{-2}\geqslant 1$.

To show that $\mathbf{D}_{P_{\{x\},\{y\},\{z\}}}^{-2}=1$, observe that there exists a commutative diagram
$$
\xymatrix{
&&U\ar@{->}[drr]^\alpha\ar@{->}[dll]_{\gamma}\\
U_1\ar@{->}[rrrr]_{\alpha_1}&&&&\mathbb{P}^3}
$$
for some birational morphism $\gamma$.
Then $\widehat{C}_{8}$ is the unique base curve of the pencil $\widehat{\mathcal{S}}$ that is mapped to $P_{\{x\},\{y\},\{z\}}$ by the morphism $\alpha$.
Thus, using Lemma~\ref{lemma:main-2} and \eqref{equation:D-A-B}, we conclude that $\mathbf{D}_{P_{\{x\},\{y\},\{z\}}}^{-2}=1$.
\end{proof}

Recall that $h^{1,2}(X)=3$. Then \eqref{equation:main-1} in  Main Theorem follows from Lemma~\ref{lemma:r2-n12-main-1}.

\subsection{Family \textnumero $2.13$}
\label{section:r-2-n-13}

In this case, the threefold $X$ is a blow up of a smooth quadric threefold in $\mathbb{P}^4$
along a smooth curve of genus $2$ and degree $6$, which gives $h^{1,2}(X)=2$.
Its~toric Landau--Ginzburg model is given by Minkowski polynomial \textnumero $1392$.
Replacing $x$ by $\frac{x}{y}$, we rewrite it as
$$
x+y+\frac{xz}{y}+\frac{x}{yz}+\frac{z}{y}+\frac{yz}{x}+\frac{2}{z}+\frac{2}{y}+\frac{2y}{x}+\frac{1}{yz}+\frac{y}{xz}.
$$
The quartic pencil $\mathcal{S}$ is given by
\begin{multline*}
xt^{3}+x^{2}t^{2}+2xyt^{2}+2xzt^{2}+y^{2}t^{2}+x^{2}tz+xz^{2}t+2y^{2}tz+\\
+x^{2}yz+xy^{2}z+xy{z}^{2}+y^{2}z^{2}=\lambda xyzt.
\end{multline*}

To prove Main Theorem in this case, we may assume that $\lambda\ne\infty$.
Let $\mathcal{C}$ be a smooth conic that is given by $z=x^2+2xy+y^2+xt=0$.
Then
\begin{equation}
\label{equation:r2-n13-base-locus}
\begin{split}
H_{\{x\}}\cdot S_\lambda&=2L_{\{x\},\{y\}}+2L_{\{x\},\{z,t\}},\\
H_{\{y\}}\cdot S_\lambda&=L_{\{x\},\{y\}}+L_{\{y\},\{t\}}+L_{\{y\},\{z,t\}}+L_{\{y\},\{x,z,t\}},\\
H_{\{z\}}\cdot S_\lambda&=2L_{\{z\},\{t\}}+\mathcal{C},\\
H_{\{t\}}\cdot S_\lambda&=L_{\{y\},\{t\}}+L_{\{z\},\{t\}}+L_{\{t\},\{x,z\}}+L_{\{t\},\{x,y\}}.
\end{split}
\end{equation}
Thus, we may assume that
$C_1=L_{\{x\},\{y\}}$, $C_2=L_{\{y\},\{t\}}$, $C_3=L_{\{z\},\{t\}}$, $C_4=L_{\{x\},\{z,t\}}$, $C_5=L_{\{y\},\{z,t\}}$,
$C_6=L_{\{t\},\{x,y\}}$, $C_7=L_{\{t\},\{x,z\}}$, $C_8=L_{\{y\},\{x,z,t\}}$, and $C_9=\mathcal{C}$.
These are all base curves of the pencil $\mathcal{S}$.

For every $\lambda\in\mathbb{C}$, the surface $S_\lambda$ has isolated singularities, so that $S_\lambda$ is irreducible.

If $\lambda\ne-3$, then the singularities of the surface $S_\lambda$
that are contained in the base locus of the pencil $\mathcal{S}$ are all du Val and can be described as follows:
\begin{itemize}\setlength{\itemindent}{3cm}
\item[$P_{\{x\},\{y\},\{t\}}$:] type $\mathbb{A}_1$ with quadratic term $xy+y^2+xt$;
\item[$P_{\{x\},\{z\},\{t\}}$:] type $\mathbb{A}_1$ with quadratic term $xz+z^2+2tz+t^2$;
\item[$P_{\{y\},\{z\},\{t\}}$:] type $\mathbb{A}_1$ with quadratic term $yz+zt+t^2$;
\item[$P_{\{x\},\{y\},\{z,t\}}$:] type $\mathbb{A}_5$ with quadratic term $(\lambda+3)xy$;
\item[$P_{\{y\},\{t\},\{x,z\}}$:] type $\mathbb{A}_1$ with quadratic term
$$
(x+z+t)(y+t)-(\lambda+1)yt
$$
for $\lambda\ne-1$, type $\mathbb{A}_2$ for $\lambda=-1$;
\item[$P_{\{z\},\{t\},\{x,y\}}$:] type $\mathbb{A}_2$ with quadratic term $z(x+y+(\lambda+3)t)$;
\item[{$[0:\lambda+3:-1:1]$}:] type $\mathbb{A}_1$;
\item[$P_{\{t\},\{x,y\},\{x,z\}}$:] smooth for $\lambda\ne -2$, type $\mathbb{A}_2$ for $\lambda=-2$.
\end{itemize}
Therefore, the points $P_{\{x\},\{z\},\{t\}}$, $P_{\{x\},\{y\},\{t\}}$,
$P_{\{y\},\{z\},\{t\}}$, $P_{\{x\},\{y\},\{z,t\}}$, $P_{\{y\},\{t\},\{x,z\}}$, and $P_{\{z\},\{t\},\{x,y\}}$
are the {fixed} singular points of the surfaces in the pencil $\mathcal{S}$.

\begin{lemma}
\label{lemma:r2-n13-main-1}
If $\lambda\ne -3$, then $[\mathsf{f}^{-1}(\lambda)]=1$.
One also has $[\mathsf{f}^{-1}(-3)]=3$.
\end{lemma}

\begin{proof}
If $\lambda\ne-3$, then $S_\lambda$ has du Val singularities at the base locus of the pencil $\mathcal{S}$,
so that $[\mathsf{f}^{-1}(\lambda)]=1$ by Corollary~\ref{corollary:irreducible-fibers}.
Hence, we must show that $[\mathsf{f}^{-1}(-3)]=3$.

Recall that $S_{-3}$ has isolated singularities.
Moreover, the points $P_{\{x\},\{z\},\{t\}}$, $P_{\{x\},\{y\},\{t\}}$,
$P_{\{y\},\{z\},\{t\}}$, $P_{\{y\},\{t\},\{x,z\}}$, and $P_{\{z\},\{t\},\{x,y\}}$
are {good} double points of this surface (see Section~\ref{subsection:scheme-step-9}).
Thus, it follows from \eqref{equation:equation:number-of-irredubicle-components-refined} and Lemmas~\ref{lemma:main} and \ref{lemma:normal-crossing} that
$$
[\mathsf{f}^{-1}(-3)]=1+\mathbf{D}_{P_{\{x\},\{y\},\{z,t\}}}^{-3},
$$
where $\mathbf{D}_{\mathbf{D}_{P_{\{x\},\{y\},\{z,t\}}}}^{-3}$ is the {defect} of the singular point $P_{\{x\},\{y\},\{z,t\}}$.

Let $\alpha_1\colon U_1\to\mathbb{P}^3$ be a blow up of the point $P_{\{x\},\{y\},\{z,t\}}$.
Then $D_{-3}^1=S_{-3}^1+\mathbf{E}_1$.
In the chart $t=1$, the surface $S_{\lambda}$ is given by
$$
(\lambda+3)\bar{x}\bar{y}+\Big(\bar{x}^2\bar{z}-\bar{x}^2\bar{y}-\bar{x}\bar{y}^2-(\lambda+2)\bar{x}\bar{y}\bar{z}+\bar{x}\bar{z}^2\big)+\Big(\bar{x}^2\bar{y}\bar{z}+\bar{x}\bar{y}^2\bar{z}+\bar{x}\bar{y}\bar{z}^2+\bar{y}^2\bar{z}^2\Big)=0,
$$
where $\bar{x}=x$, $\bar{y}=y$, and $\bar{z}=z+1$.
Then a chart of the blow up $\alpha_1$ is given by the coordinate change
$\bar{x}_1=\frac{\bar{x}}{\bar{z}}$, $\bar{y}_1=\frac{\bar{y}}{\bar{z}}$, and $\bar{z}_1=\bar{x}$.
Then $D_{\lambda}^1$ is given by
\begin{multline*}
\bar{x}_1\big(\bar{z}_1+(\lambda+3)\bar{y}_1\big)+\Big(\bar{x}_1^2\bar{z}_1-(\lambda+2)\bar{x}_1\bar{y}_1\bar{z}_1\Big)+\\
+\Big(\bar{x}_1\bar{y}_1\bar{z}_1^2-\bar{x}_1^2\bar{y}_1\bar{z}_1-\bar{x}_1\bar{y}_1^2\bar{z}_1+\bar{y}_1^2\bar{z}_1^2\Big)+\Big(\bar{x}_1^2\bar{y}_1\bar{z}_1^2+\bar{x}_1\bar{y}_1^2\bar{z}_1^2\Big)=0
\end{multline*}
and the surface $\mathbf{E}_1$ is given by $\bar{z}_1=0$.

The surface $\mathbf{E}_1$ contains two base curves of the pencil $\mathcal{S}^1$.
They are given by $\bar{z}_1=\bar{x}_1=0$ and $\bar{z}_1=\bar{y}_1=0$.
Denote them by $C_{9}^1$ and $C_{10}^1$, respectively.
Then $\mathbf{M}_{9}^{-3}=2$, $\mathbf{M}_{10}^{-3}=1$, and $\mathbf{m}_9=2$.
Thus, using \eqref{equation:D-A-B} and Lemma~\ref{lemma:main-2}, we see that  $\mathbf{D}_{P_{\{x\},\{y\},\{z,t\}}}\geqslant 2$.

To show that $\mathbf{D}_{P_{\{x\},\{y\},\{z,t\}}}=2$, we have to blow up $U_2$ at the point $(\bar{x}_1,\bar{x}_2,\bar{z}_1)=(0,0,0)$.
Namely, let $\alpha_2\colon U_2\to U_1$ be this blow up.
Then $D_{-3}^2=S_{-3}^2+\mathbf{E}_1^2$, and $\mathbf{E}_2$ contains a unique base curve of the pencil $\mathcal{S}^2$.
Denote it by $C_{11}^2$. Then $\mathbf{M}_{11}^{-3}=1$.
Now, using \eqref{equation:D-A-B} and Lemma~\ref{lemma:main-2} again, we obtain $\mathbf{D}_{P_{\{x\},\{y\},\{z,t\}}}=2$.
This gives $[\mathsf{f}^{-1}(-3)]=3$.
\end{proof}

Note that Lemma~\ref{lemma:r2-n13-main-1} implies \eqref{equation:main-1} in  Main Theorem, since $h^{1,2}(X)=2$.

To verify \eqref{equation:main-1} in  Main Theorem,
recall that the base curves of the pencil $\mathcal{S}$ are
$L_{\{x\},\{y\}}$, $L_{\{y\},\{t\}}$, $L_{\{z\},\{t\}}$, $L_{\{x\},\{z,t\}}$, $L_{\{y\},\{z,t\}}$,
$L_{\{t\},\{x,y\}}$, $L_{\{t\},\{x,z\}}$, $L_{\{y\},\{x,z,t\}}$, and $\mathcal{C}$.
On a general quartic surface in this pencil, the intersection matrix of these curves
has the same rank as the intersection matrix of the curves
$L_{\{x\},\{y\}}$, $L_{\{y\},\{t\}}$, $L_{\{z\},\{t\}}$, $L_{\{y\},\{z,t\}}$, $L_{\{t\},\{x,y\}}$, and $H_{\lambda}$.
This follows from \eqref{equation:r2-n13-base-locus}.
On the other hand, if $\lambda\not\in\{-1,-2,-3\}$, then the intersection form of the curves
$L_{\{x\},\{y\}}$, $L_{\{y\},\{t\}}$, $L_{\{z\},\{t\}}$, $L_{\{y\},\{z,t\}}$, $L_{\{t\},\{x,y\}}$, and $H_{\lambda}$
on the surface $S_\lambda$ is given by
\begin{center}\renewcommand\arraystretch{1.42}
\begin{tabular}{|c||c|c|c|c|c|c|}
\hline
$\bullet$ & $L_{\{x\},\{y\}}$ & $L_{\{y\},\{t\}}$ & $L_{\{z\},\{t\}}$ & $L_{\{y\},\{z,t\}}$ & $L_{\{t\},\{x,y\}}$ & $H_{\lambda}$\\
\hline\hline
 $L_{\{x\},\{y\}}$& $-\frac{1}{6}$ & $\frac{1}{2}$ & $0$ & $\frac{1}{3}$ & $\frac{1}{2}$ & $1$\\
\hline
 $L_{\{y\},\{t\}}$& $\frac{1}{2}$ & $-1$ & $\frac{1}{2}$ & $\frac{1}{2}$ & $\frac{1}{2}$ & $1$\\
\hline
$L_{\{z\},\{t\}}$& $0$ & $\frac{1}{2}$ & $-\frac{1}{3}$ & $\frac{1}{2}$ & $\frac{1}{3}$ & $1$\\
\hline
 $L_{\{y\},\{z,t\}}$& $\frac{1}{3}$ & $\frac{1}{2}$ & $\frac{1}{2}$ & $-\frac{2}{3}$ & $0$& $1$\\
\hline
 $L_{\{t\},\{x,y\}}$& $\frac{1}{2}$ & $\frac{1}{2}$ & $\frac{1}{3}$ & $0$ & $-\frac{5}{6}$& $1$\\
\hline
$H_{\lambda}$ & $1$  & $1$  & $1$  & $1$  & $1$  & $4$ \\
\hline
\end{tabular}
\end{center}
Since the determinant of this matrix is $-\frac{5}{12}$, we see that its rank is $6$.
On the other hand, we have
$\mathrm{rk}\,\mathrm{Pic}(\widetilde{S}_{\Bbbk})=\mathrm{rk}\,\mathrm{Pic}(S_{\Bbbk})+12$.
Hence, we see that \eqref{equation:main-2-simple} holds,
so that \eqref{equation:main-2} in Main Theorem also holds by Lemma~\ref{lemma:cokernel}.

\subsection{Family \textnumero $2.14$}
\label{section:r-2-n-14}

Let $V_5$ be a smooth threefold such that
$-K_{V_5}\sim 2H$ and $H^3=5$, where $H$ is an ample Cartier divisor.
Then $V_5$ is determined by these properties uniquely up to isomorphism.
A general surface in $|H|$ is a smooth del Pezzo surface of degree $5$.
This linear system is base point free and gives an embedding $V_5\hookrightarrow\mathbb{P}^6$.

In our case, the threefold $X$ is the blow up of the threefold $V_5$ along an elliptic curve
that is a complete intersection of two general surfaces in the linear system $|H|$.
Its toric Landau--Ginzburg model is given by Minkowski polynomial \textnumero $1658$, which is
$$
x+\frac{xy}{z}+z+\frac{2y}{z}+\frac{z^2}{xy}+\frac{z}{x}+\frac{2}{z}+\frac{3z}{xy}+\frac{3}{x}+\frac{y}{xz}+\frac{3}{xy}+\frac{2}{xz}+\frac{1}{xyz}.
$$
The quartic pencil $\mathcal{S}$ is given by
\begin{multline*}
x^2zy+y^2x^2+z^2yx+2y^2tx+z^3t+z^2ty+2t^2yx+3t^2z^2+\\
+3t^2zy+t^2y^2+3t^3z+2t^3y+t^4=\lambda xyzt.
\end{multline*}

Suppose that $\lambda\ne\infty$.
Let $\mathcal{C}_1$ be a conic that is given by $x=t^2+ty+2tz+z^2=0$,
let~$\mathcal{C}_2$ be a conic that is given by $z=xy+ty+t^2=0$,
and let $\mathcal{C}_3$ be a conic that is given by $t=xy+xz+z^2=0$.
Then
\begin{equation}
\label{equation:r2-n14-base-locus}
\begin{split}
H_{\{x\}}\cdot S_\lambda&=L_{\{x\},\{t\}}+L_{\{x\},\{y,z,t\}}+\mathcal{C}_1,\\
H_{\{y\}}\cdot S_\lambda&=L_{\{y\},\{t\}}+3L_{\{y\},\{z,t\}},\\
H_{\{z\}}\cdot S_\lambda&=2\mathcal{C}_2,\\
H_{\{t\}}\cdot S_\lambda&=L_{\{x\},\{t\}}+L_{\{y\},\{t\}}+\mathcal{C}_3.
\end{split}
\end{equation}
We let $C_1=\mathcal{C}_1$, $C_2=\mathcal{C}_2$, $C_3=\mathcal{C}_3$, $C_4=L_{\{x\},\{t\}}$,
$C_5=L_{\{y\},\{t\}}$, $C_6=L_{\{x\},\{y,z,t\}}$, and $C_7=L_{\{y\},\{z,t\}}$.
These are all base curves of the pencil~$\mathcal{S}$.

For every $\lambda\in\mathbb{C}$, the surface $S_\lambda$ has isolated singularities, so that $S_\lambda$ is irreducible.

If $\lambda\ne-4$, then the singularities of the surface $S_\lambda$
that are contained in the base locus of the pencil $\mathcal{S}$ are all du Val and can be described as follows:
\begin{itemize}\setlength{\itemindent}{5cm}
\item[$P_{\{y\},\{z\},\{t\}}$:] type $\mathbb{A}_3$ with quadratic term $y(z+y)$;
\item[$P_{\{x\},\{z\},\{t\}}$:] type $\mathbb{D}_5$ with quadratic term $(x+t)^2$;
\item[$P_{\{x\},\{z\},\{y,t\}}$:] type $\mathbb{A}_1$ with quadratic term
$$
(3+\lambda)xz+(x-y-2z)(x-y-z)
$$
for $\lambda\ne-3$, type $\mathbb{A}_3$ for $\lambda=-3$;
\item[$P_{\{x\},\{y\},\{z,t\}}$:] type $\mathbb{A}_3$ with quadratic term
$$
y((\lambda+3)x+y+z+t)
$$
for $\lambda\ne-3$, type $\mathbb{A}_5$ for $\lambda=-3$;
\item[{$[\lambda+3:0:-1:1]$}:] type $\mathbb{A}_2$ for $\lambda\ne -3$;

\item[{$[(\lambda+4)(\lambda+3):-1:0:\lambda+4]$}:] type $\mathbb{A}_1$ for $\lambda\ne -3$.
\end{itemize}
Therefore, if $\lambda\ne-4$, then $S_\lambda$ has du Val singularities at the base locus of the pencil~$\mathcal{S}$,
so that the fiber $\mathsf{f}^{-1}(\lambda)$ is irreducible by Corollary~\ref{corollary:irreducible-fibers}.
On the other hand, we have

\begin{lemma}
\label{lemma:r2-n14-main-1}
One has $[\mathsf{f}^{-1}(-4)]=2$.
\end{lemma}

\begin{proof}
The points $P_{\{y\},\{z\},\{t\}}$, $P_{\{x\},\{z\},\{y,t\}}$, and $P_{\{x\},\{y\},\{z,t\}}$
are {good} double points of the surface $S_{-4}$.
Thus, it follows from \eqref{equation:equation:number-of-irredubicle-components-refined} and Lemmas~\ref{lemma:main} and \ref{lemma:normal-crossing} that
$$
\big[\mathsf{f}^{-1}(-4)\big]=1+\mathbf{D}_{P_{\{x\},\{z\},\{t\}}}^{-4}.
$$
Here, the number $\mathbf{D}_{P_{\{x\},\{z\},\{t\}}}^{-4}$ is the {defect} of the point $P_{\{x\},\{z\},\{t\}}$.
To compute it, we have to (partially) resolve
the singularity of the surface $S_{-4}$ at the point $P_{\{x\},\{z\},\{t\}}$.

Let $\alpha_1\colon U_1\to\mathbb{P}^3$ be a blow up of the point $P_{\{x\},\{z\},\{t\}}$.
In the chart $y=1$, the surface~$S_{\lambda}$ is given by
$$
\bar{x}^2+\Big(2\bar{t}^2\bar{x}+(\lambda+4)\bar{t}^2\bar{z}-(\lambda+2)\bar{t}\bar{x}\bar{z}+\bar{x}^2\bar{z}+\bar{x}\bar{z}^2\Big)+\Big(\bar{t}^4+3\bar{z}\bar{t}^3+3\bar{t}^2\bar{z}^2+\bar{t}\bar{z}^3\Big)=0,
$$
where $\bar{x}=x+t$, $\bar{z}=z$, and $\bar{t}=t$.
Then a chart of the blow up $\alpha_1$ is given by the coordinate change
$\bar{x}_1=\frac{\bar{x}}{\bar{z}}$, $\bar{z}_1=\bar{x}$, and $\bar{t}_1=\frac{\bar{t}}{\bar{z}}$.
Let $\hat{x}_1=\bar{x}_1$, $\hat{z}_1=\bar{x}_1+\bar{z}_1$ and $\hat{t}_1=\bar{t}_1$.
In these coordinates, the surface~$S_{\lambda}^1$ is given by
\begin{multline*}
\hat{x}_1\hat{z}_1+\Big(\hat{t}_1\hat{z}_1^2-\hat{x}_1^3+\hat{x}_1^2\hat{z}_1-(\lambda+4)\hat{t}_1^2\hat{x}_1+(\lambda+4)\hat{t}_1^2\hat{z}_1+(3+\lambda)\hat{x}_1^2\hat{t}_1-(\lambda+4)\hat{t}_1\hat{x}_1\hat{z}_1\Big)+\\
+\hat{t}_1^2\hat{x}_1^2-4\hat{t}_1^2\hat{x}_1\hat{z}_1+3\hat{t}_1^2\hat{z}_1^2+3\hat{x}_1^2\hat{t}_1^3-6\hat{x}_1\hat{z}_1\hat{t}_1^3+3\hat{t}_1^3\hat{z}_1^2+\hat{t}_1^4\hat{x}_1^2-2\hat{t}_1^4\hat{x}_1\hat{z}_1+\hat{t}_1^4\hat{z}_1^2=0.
\end{multline*}

If $\lambda\ne -4$, the surface $S_{\lambda}^1$ has isolated singularity at $(\hat{x}_1,\hat{z}_1,\hat{t}_1)=(0,0,0)$.
In this case, the surface $\mathbf{E}_1$ contains another singular point of the surface  $S_{\lambda}^1$,
which lies in another chart of the blow up $\alpha_1$.
If $\lambda\ne -4$, then this point is an isolated ordinary double point of the surface  $S_{\lambda}^1$.
On the other hand, the surface $S_{-4}^1$ is singular along the curve  $\bar{z}_1=\bar{x}_1=0$.
This explains why the singularity of the surface $S_{-4}$ at the point $P_{\{x\},\{z\},\{t\}}$ is not du Val.

Let $\alpha_2\colon U_2\to U_1$ be a blow up of the point $(\hat{x}_1,\hat{z}_1,\hat{t}_1)=(0,0,0)$.
A chart of this blow up is given by the coordinate change
$\hat{x}_2=\frac{\hat{x}_1}{\hat{t}_1}$, $\hat{y}_2=\frac{\hat{y}_1}{\hat{t}_1}$, and $\hat{t}_2=\hat{t}_1$.
In these coordinates, the surface $S_{\lambda}^2$ is given by
\begin{multline*}
\hat{x}_2\hat{z}_2-(\lambda+4)\hat{x}_2\hat{t}_2+(\lambda+4)\hat{z}_2\hat{t}_2+\Big(\hat{t}_2\hat{z}_2^2+(3+\lambda)\hat{x}_2^2\hat{t}_2-(\lambda+4)\hat{t}_2\hat{x}_2\hat{z}_2\Big)+\\
+\hat{t}_2^2\hat{x}_2^2-4\hat{t}_2^2\hat{x}_2\hat{z}_2+3\hat{t}_2^2\hat{z}_2^2-\hat{t}_2\hat{x}_2^3+\hat{t}_2\hat{x}_2^2\hat{z}_2+3\hat{x}_2^2\hat{t}_2^3-6\hat{x}_2\hat{z}_2\hat{t}_2^3+3\hat{t}_2^3\hat{z}_2^2+\hat{t}_2^4\hat{x}_2^2-2\hat{t}_2^4\hat{x}_2\hat{z}_2+\hat{t}_2^4\hat{z}_2^2=0,
\end{multline*}
and the surface $\mathbf{E}_2$ is given by $\hat{t}_2=0$.
Note that $D_{\lambda}^2=S_{\lambda}^2\sim -K_{U_2}$ for every $\lambda\in\mathbb{C}$.

If $\lambda\ne -4$, then the quadric form $\hat{x}_2\hat{z}_2-(\lambda+4)\hat{x}_2\hat{t}_2+(\lambda+4)\hat{z}_2\hat{t}_2$ is not degenerate,
so that $S_{\lambda}^2$ has an isolated ordinary double singularity at $(\hat{x}_2,\hat{z}_2,\hat{t}_2)=(0,0,0)$.
Thus, in this case, the surface $S_{\lambda}^1$ has a du Val singularity of type $\mathbb{A}_3$ at the point $(\hat{x}_1,\hat{z}_1,\hat{t}_1)=(0,0,0)$.
Therefore, if $\lambda\ne -4$, then $S_\lambda$ has a du Val singularity of type $\mathbb{D}_5$ at the point $P_{\{x\},\{z\},\{t\}}$.

Now we are ready to compute $\mathbf{D}_{P_{\{x\},\{z\},\{t\}}}$ using the algorithm described in Section~\ref{subsection:scheme-step-8}.
Observe that $\mathbf{E}_1$ contains one base curve of the pencil $\mathcal{S}^1$.
It is given by $\bar{z}_1=\bar{x}_1=0$. Denote this curve by $C_{8}^1$.
Then $\mathbf{m}_8=2$ and $\mathbf{M}_8^{-4}=2$.
Similarly, the surface $\mathbf{E}_2$ contains two  base curves of the pencil $\mathcal{S}^2$.
They are given by $\hat{t}_2=\hat{x}_2=0$ and $\hat{t}_2=\hat{z}_2=0$.
Denote them by $C_{9}^2$ and $C_{10}^2$, respectively.
Then $\mathbf{M}_9^{-4}=\mathbf{M}_{10}^{-4}=1$.
Now, using  \eqref{equation:D-A-B} and Lemma~\ref{lemma:main-2}, we deduce that $\mathbf{D}_{P_{\{x\},\{z\},\{t\}}}=1$,
so that $[\mathsf{f}^{-1}(-4)]=2$.
\end{proof}

Note that Lemma~\ref{lemma:r2-n14-main-1} implies \eqref{equation:main-1} in  Main Theorem, since $h^{1,2}(X)=2$.

\begin{lemma}
\label{lemma:r2-n14-intersection}
Suppose that $\lambda\ne -4$ and $\lambda\ne -3$.
Then the intersection form of the curves
$L_{\{x\},\{t\}}$, $L_{\{y\},\{t\}}$, $L_{\{x\},\{y,z,t\}}$, and $H_{\lambda}$
on the surface $S_\lambda$ is given by
\begin{center}\renewcommand\arraystretch{1.42}
\begin{tabular}{|c||c|c|c|c|}
\hline
$\bullet$ & $L_{\{x\},\{t\}}$ & $L_{\{y\},\{t\}}$ & $L_{\{x\},\{y,z,t\}}$ & $H_{\lambda}$\\
\hline\hline
 $L_{\{x\},\{t\}}$& $-\frac{3}{4}$ & $1$ & $1$ & $1$\\
\hline
 $L_{\{y\},\{t\}}$& $1$ & $\frac{5}{4}$ & $0$ & $1$\\
\hline
 $L_{\{x\},\{y,z,t\}}$& $1$ & $0$ & $-\frac{5}{6}$ & $1$\\
\hline
$H_{\lambda}$ & $1$  & $1$  & $1$ & $4$ \\
\hline
\end{tabular}
\end{center}
\end{lemma}

\begin{proof}
To compute $L_{\{x\},\{t\}}^2$, let us use the notation of Lemma~\ref{lemma:Dn} with $S=S_{\lambda}$, $n=5$, $O=P_{\{x\},\{z\},\{t\}}$, and $C=L_{\{x\},\{t\}}$.
Then $\overline{C}$ contains the point $\alpha(G_1)=\alpha(G_2)=\alpha(G_3)$,
and either $\widetilde{C}\cdot G_2=1$ or $\widetilde{C}\cdot G_3=1$.
This follows from the proof of Lemma~\ref{lemma:r2-n14-main-1}.
Thus, we have $L_{\{x\},\{t\}}^2=-\frac{3}{4}$ by Lemma~\ref{lemma:Dn}.

To find $L_{\{y\},\{t\}}^2$, we observe that $P_{\{y\},\{z\},\{t\}}$ is the only singular point of the surface $S_\lambda$ that is contained in  $L_{\{y\},\{t\}}$.
Thus, it follows from Proposition~\ref{proposition:du-Val-intersection} that $L_{\{y\},\{t\}}^2=-\frac{5}{4}$.

To find $L_{\{x\},\{y,z,t\}}^2$, we observe that $P_{\{x\},\{z\},\{y,t\}}$ and $P_{\{x\},\{y\},\{z,t\}}$
are the only singular points of the surface $S_\lambda$ that are contained in the line $L_{\{x\},\{y,z,t\}}$.
Thus, it follows from Proposition~\ref{proposition:du-Val-intersection} that $L_{\{y\},\{t\}}^2=-\frac{5}{6}$.

Since $L_{\{x\},\{t\}}\cap L_{\{y\},\{t\}}=P_{\{x\},\{y\},\{t\}}$ and $L_{\{x\},\{t\}}\cap L_{\{x\},\{y,z,t\}}=P_{\{x\},\{t\},\{y,z\}}$, we obtain
$$
L_{\{x\},\{t\}}\cdot L_{\{y\},\{t\}}=L_{\{x\},\{t\}}\cdot L_{\{x\},\{y,z,t\}}=1,
$$
because $S_\lambda$ is smooth at the points $P_{\{x\},\{y\},\{t\}}$ and $P_{\{x\},\{t\},\{y,z\}}$.

Finally, we have $L_{\{y\},\{t\}}\cdot L_{\{x\},\{y,z,t\}}=0$, since
$L_{\{y\},\{t\}}\cap L_{\{x\},\{y,z,t\}}=\varnothing$.
\end{proof}

Recall that the base curves of the pencil $\mathcal{S}$ are
the curves $L_{\{x\},\{t\}}$,
$L_{\{y\},\{t\}}$, $L_{\{x\},\{y,z,t\}}$, $L_{\{y\},\{z,t\}}$, $\mathcal{C}_1$, $\mathcal{C}_2$, and $\mathcal{C}_3$.
If follows from \eqref{equation:r2-n14-base-locus} that
the intersection matrix of these curves on $S_\lambda$
has the same rank as the intersection matrix of the curves $L_{\{x\},\{t\}}$, $L_{\{y\},\{t\}}$, $L_{\{x\},\{y,z,t\}}$, and $H_{\lambda}$.
On the other hand, the determinant of the intersection matrix in Lemma~\ref{lemma:r2-n14-intersection} is $\frac{25}{16}$.
Thus, if $\lambda\ne -4$ and $\lambda\ne -3$, then the intersection matrix of the curves $L_{\{x\},\{t\}}$,
$L_{\{y\},\{t\}}$, $L_{\{x\},\{y,z,t\}}$, $L_{\{y\},\{z,t\}}$, $\mathcal{C}_1$, $\mathcal{C}_2$, and $\mathcal{C}_3$
on the surface $S_\lambda$ has rank $4$.
On the other hand, we have
$\mathrm{rk}\,\mathrm{Pic}(\widetilde{S}_{\Bbbk})=\mathrm{rk}\,\mathrm{Pic}(S_{\Bbbk})+14$.
Hence, we see that \eqref{equation:main-2-simple} holds,
so that \eqref{equation:main-2} in Main Theorem also holds by Lemma~\ref{lemma:cokernel}.

\subsection{Family \textnumero $2.15$}
\label{section:r-2-n-15}

In this case, the Fano threefold $X$ is a blow up of $\mathbb{P}^3$ at a smooth curve of degree $6$ and genus $4$.
Thus, we have $h^{1,2}(X)=4$. A~toric Landau--Ginzburg model of this family is given by Minkowski polynomial \textnumero $910$,
which is
$$
x+y+z+\frac{x}{z}+\frac{y}{z}+\frac{x}{yz}+\frac{2}{z}+\frac{y}{xz}+\frac{2}{y}+\frac{2}{x}+\frac{z}{xy}.
$$
The pencil $\mathcal{S}$ is given by
$$
x^2zy+y^2zx+z^2yx+x^2ty+y^2tx+x^2t^2+2t^2yx+t^2y^2+2t^2zx+2t^2zy+t^2z^2=\lambda xyzt.
$$
Observe that this equation is symmetric with respect to the permutation $x\leftrightarrow y$.

We may assume that $\lambda\ne\infty$. Let $\mathcal{C}$ be the conic $\{z=xy+xt+yt=0\}$. Then
\begin{equation}
\label{equation:r2-n15-base-locus}
\begin{split}
H_{\{x\}}\cdot S_\lambda&=2L_{\{x\},\{t\}}+2L_{\{x\},\{y,z\}},\\
H_{\{y\}}\cdot S_\lambda&=2L_{\{y\},\{t\}}+2L_{\{y\},\{x,z\}},\\
H_{\{z\}}\cdot S_\lambda&=L_{\{z\},\{t\}}+L_{\{z\},\{x,y\}}+\mathcal{C},\\
H_{\{t\}}\cdot S_\lambda&=L_{\{x\},\{t\}}+L_{\{y\},\{t\}}+L_{\{z\},\{t\}}+L_{\{t\},\{x,y,z\}}.
\end{split}
\end{equation}
Thus, we may assume that $C_1=L_{\{x\},\{t\}}$, $C_2=L_{\{y\},\{t\}}$, $C_3=L_{\{z\},\{t\}}$,
$C_4=L_{\{x\},\{y,z\}}$, $C_5=L_{\{y\},\{x,z\}}$, $C_6=L_{\{z\},\{x,y\}}$, $C_7=L_{\{t\},\{x,y,z\}}$, and $C_8=\mathcal{C}$.
These are all base curves of the pencil~$\mathcal{S}$.

If $\lambda\ne -1$, then the surface $S_\lambda$ has isolated singularities, so that $S_\lambda$ is irreducible.
One the other hand, we have $S_{-1}=H_{\{x,y,z\}}+\mathbf{S}$,
where $\mathbf{S}$ is an irreducible cubic surface that is given by $xyz+xyt+xt^2+yt^2+zt^2=0$.
The surface $\mathbf{S}$ is singular at $P_{\{y\},\{z\},\{t\}}$, $P_{\{x\},\{z\},\{t\}}$, and $P_{\{x\},\{y\},\{t\}}$.
These are isolated ordinary double points of this surface.
Note also that
$$
H_{\{x,y,z\}}\cap\mathbf{S}=L_{\{x\},\{y,z\}}+L_{\{y\},\{x,z\}}+\ell,
$$
where $\ell$ is the line $\{y+x-t=z+t=0\}$.

If $\lambda\ne-1$, then the singularities of the surface $S_\lambda$
that are contained in the base locus of the pencil $\mathcal{S}$ are all du Val and can be described as follows:
\begin{itemize}\setlength{\itemindent}{3cm}
\item[$P_{\{x\},\{y\},\{z\}}$:] type $\mathbb{D}_4$ with quadratic term $(x+y+z)^2$;
\item[$P_{\{x\},\{y\},\{t\}}$:] type $\mathbb{A}_1$ with quadratic term $xy+t^2$;
\item[$P_{\{x\},\{z\},\{t\}}$:] type $\mathbb{A}_1$ with quadratic term $xz+xt+t^2$;
\item[$P_{\{y\},\{z\},\{t\}}$:] type $\mathbb{A}_1$ with quadratic term $yz+yt+t^2$;
\item[$P_{\{x\},\{t\},\{y,z\}}$:] type $\mathbb{A}_3$ with quadratic term $x(x+y+z+(\lambda+1)t)$;
\item[$P_{\{y\},\{t\},\{x,z\}}$:] type $\mathbb{A}_3$ with quadratic term $y(x+y+z+(\lambda+1)t)$;
\item[$P_{\{z\},\{t\},\{x,y\}}$:] type $\mathbb{A}_1$ with quadratic term $(x+y)(z+t)+z^2-\lambda zt$.
\end{itemize}
If $\lambda\ne-1$, then  $[\mathsf{f}^{-1}(\lambda)]=1$ by Corollary~\ref{corollary:irreducible-fibers},
so that \eqref{equation:main-1} in  Main Theorem follows~from

\begin{lemma}
\label{lemma:r2-n15-main-1}
One has $[\mathsf{f}^{-1}(-1)]=5$.
\end{lemma}

\begin{proof}
It follows from \eqref{equation:equation:number-of-irredubicle-components-refined} and Lemmas~\ref{lemma:main} and \ref{lemma:normal-crossing} that
$$
\big[\mathsf{f}^{-1}(-1)\big]=4+\mathbf{D}_{P_{\{x\},\{y\},\{z\}}}^{-1}.
$$
Observe that $\mathbf{S}$ is smooth at $P_{\{x\},\{y\},\{z\}}$,
and $H_{\{x,y,z\}}$ is tangent to $\mathbf{S}$ at this point.
Thus, the proper transforms of these surfaces on the blow up of $\mathbb{P}^3$ at the point $P_{\{x\},\{y\},\{z\}}$
both pass through the base curve of the proper transform of the pencil $\mathcal{S}$ that is contained in the exceptional divisor.
Using \eqref{equation:D-A-B}, we conclude that $\mathbf{D}_{P_{\{x\},\{y\},\{z\}}}^{-1}\geqslant 1$.
Arguing as in the proof of Lemma~\ref{lemma:r2-n5-irreducible-special}, we see that $\mathbf{D}_{P_{\{x\},\{y\},\{z\}}}^{-1}=1$, so that $[\mathsf{f}^{-1}(-1)]=5$.
\end{proof}

If $\lambda\ne -1$, then the intersection form of the curves
$L_{\{x\},\{t\}}$, $L_{\{y\},\{t\}}$, $L_{\{z\},\{t\}}$, $L_{\{z\},\{x,y\}}$, and $H_{\lambda}$
on the surface $S_\lambda$ is given by
\begin{center}\renewcommand\arraystretch{1.42}
\begin{tabular}{|c||c|c|c|c|c|}
\hline
$\bullet$ & $L_{\{x\},\{t\}}$ & $L_{\{y\},\{t\}}$ & $L_{\{z\},\{t\}}$ & $L_{\{z\},\{x,y\}}$ & $H_{\lambda}$\\
\hline\hline
$L_{\{x\},\{t\}}$  & $-\frac{1}{4}$  & $\frac{1}{2}$  & $\frac{1}{2}$ & $0$ & $1$ \\
\hline
$L_{\{y\},\{t\}}$  & $\frac{1}{2}$  & $-\frac{1}{4}$  & $\frac{1}{2}$ & $0$ & $1$ \\
\hline
$L_{\{z\},\{t\}}$  & $\frac{1}{2}$  & $\frac{1}{2}$  & $-\frac{1}{2}$ & $\frac{1}{2}$ & $1$ \\
\hline
 $L_{\{z\},\{x,y\}}$  & $0$  & $0$  & $\frac{1}{2}$ & $-\frac{1}{2}$ & $1$ \\
\hline
$H_{\lambda}$ & $1$  & $1$  & $1$ & $1$ & $4$ \\
\hline
\end{tabular}
\end{center}
The rank of this matrix is $4$.
Thus, if $\lambda\ne -1$, then it follows from \eqref{equation:r2-n15-base-locus} that
the intersection matrix of the base curves of the pencil $\mathcal{S}$ on the surface $S_\lambda$ also has rank $4$.
On the other hand, we have
$\mathrm{rk}\,\mathrm{Pic}(\widetilde{S}_{\Bbbk})=\mathrm{rk}\,\mathrm{Pic}(S_{\Bbbk})+14$.
Hence, we see that \eqref{equation:main-2-simple} holds, so that \eqref{equation:main-2} in Main Theorem also holds by Lemma~\ref{lemma:cokernel}.

\subsection{Family \textnumero $2.16$}
\label{section:r-2-n-16}

In this case, the Fano threefold $X$ is a blow up of a smooth complete intersection of two quadrics in a conic.
We have $h^{1,2}(X)=2$.
A~toric Landau--Ginzburg model of this family is given by Minkowski polynomial \textnumero $1939$, which is
$$
x+z+{\frac {y}{z}}+{\frac {y}{x}}+{\frac {x}{z}}+{\frac {z}{x}}+{
\frac {x}{y}}+\frac{1}{z}+{\frac {z}{y}}+\frac{1}{x}+{\frac {x}{yz}}+\frac{2}{y}+{\frac {z}{xy}}.
$$
The quartic pencil $\mathcal{S}$ is given by
$$
x^2yz+xy^2t+xyz^2+y^2zt+x^2yt+yz^2t+x^2zt+xyt^2+xz^2t+yzt^2+x^2t^2+2xzt^2+z^2t^2=\lambda xyzt.
$$

As usual, we suppose that  $\lambda\ne\infty$.
If $\lambda\ne-2$, then the surface $S_\lambda$ has isolated singularities, so that it is irreducible.
One also has $S_{-2}=H_{\{x,z\}}+H_{\{y,t\}}+\mathbf{Q}$,
where $\mathbf{Q}$ is an irreducible quadric surface given by $xz+xt+yt+zt=0$.

Let $\mathcal{C}$ be the conic in $\mathbb{P}^3$ that is given by $y=xz+xt+zt=0$.
Then
\begin{itemize}
\item $H_{\{x\}}\cdot S_\lambda=L_{\{x\},\{z\}}+L_{\{x\},\{t\}}+L_{\{x\},\{y,z\}}+L_{\{x\},\{y,t\}}$;
\item $H_{\{y\}}\cdot S_\lambda=L_{\{y\},\{t\}}+L_{\{y\},\{x,z\}}+\mathcal{C}$;
\item $H_{\{z\}}\cdot S_\lambda=L_{\{x\},\{z\}}+L_{\{z\},\{t\}}+L_{\{z\},\{x,y\}}+L_{\{z\},\{y,t\}}$;
\item $H_{\{t\}}\cdot S_\lambda=L_{\{x\},\{t\}}+L_{\{y\},\{t\}}+L_{\{z\},\{t\}}+L_{\{t\},\{x,z\}}$,
\end{itemize}
so that the base locus of the pencil~$\mathcal{S}$ consists of the curves
$L_{\{x\},\{z\}}$, $L_{\{x\},\{t\}}$, \mbox{$L_{\{y\},\{t\}}$},
$L_{\{z\},\{t\}}$, $L_{\{x\},\{y,z\}}$, $L_{\{x\},\{y,t\}}$, $L_{\{y\},\{x,z\}}$,
$L_{\{z\},\{x,y\}}$, \mbox{$L_{\{z\},\{y,t\}}$}, $L_{\{t\},\{x,z\}}$, and $\mathcal{C}$.

If $\lambda\ne-2$, then singular points of the surface $S_\lambda$ contained in the base locus of the pencil $\mathcal{S}$
are all du Val and can be described as follows:
\begin{itemize}\setlength{\itemindent}{3cm}
\item[$P_{\{x\},\{y\},\{z\}}$:] type $\mathbb{A}_3$ with quadratic term
$$
(x+z)(x+y+z)
$$
for $\lambda\neq -3$, type $\mathbb{A}_5$ for $\lambda=-3$;

\item[$P_{\{x\},\{y\},\{t\}}$:] type $\mathbb{A}_2$ with quadratic term $(x+t)(y+t)$;

\item[$P_{\{x\},\{z\},\{t\}}$:] type $\mathbb{A}_3$ with quadratic term $t(x+z)$;

\item[$P_{\{y\},\{z\},\{t\}}$:] type $\mathbb{A}_2$ with quadratic term $(y+t)(z+t)$;

\item[$P_{\{x\},\{z\},\{y,t\}}$:] type $\mathbb{A}_1$ with quadratic term $xy+yz+zt-(\lambda+2)xz$;

\item[$P_{\{y\},\{t\},\{x,z\}}$:] type $\mathbb{A}_1$ with quadratic term $xy+xt+yz-(\lambda+2)yt$.
\end{itemize}
Moreover, the singularities of the surface $S_{-2}$ at these points are non-isolated ordinary double points.
Thus, if $\lambda\ne-2$, then the fiber $\mathsf{f}^{-1}(\lambda)$ is irreducible by Corollary~\ref{corollary:irreducible-fibers}.
Similarly, it follows from \eqref{equation:equation:number-of-irredubicle-components-refined} and Lemma~\ref{lemma:normal-crossing}
that $[\mathsf{f}^{-1}(-2)]=3$.
This confirms \eqref{equation:main-1} in Main Theorem, since  $h^{1,2}(X)=2$.

If $\lambda\ne-2$ and $\lambda\ne-3$,
then the intersection matrix of the curves
$L_{\{y\},\{x,t\}}$, $L_{\{t\},\{x,z\}}$, $L_{\{x\},\{y,t\}}$, $L_{\{z\},\{y,t\}}$, $L_{\{x\},\{y,t\}}$, and $H_{\lambda}$
on the surface $S_\lambda$ is given by
\begin{center}\renewcommand\arraystretch{1.42}
\begin{tabular}{|c||c|c|c|c|c|c|}
\hline
 $\bullet$  & $L_{\{y\},\{x,t\}}$ & $L_{\{t\},\{x,z\}}$ & $L_{\{x\},\{y,t\}}$ & $L_{\{z\},\{y,t\}}$ & $L_{\{x\},\{y,t\}}$ & $H_{\lambda}$ \\
\hline\hline
$L_{\{y\},\{x,t\}}$ & $-\frac{4}{3}$ & $0$ & $\frac{1}{3}$ & $0$ & $0$ & $1$ \\
\hline
$L_{\{t\},\{x,z\}}$ & $0$ & $-\frac{1}{2}$ & $0$ &  $0$ & $0$ & $1$ \\
\hline
$L_{\{x\},\{y,t\}}$ & $\frac{1}{3}$ & $0$ & $-\frac{5}{6}$ &  $\frac{1}{2}$ & $1$ & $1$ \\
\hline
$L_{\{z\},\{y,t\}}$ & $0$ & $0$ & $\frac{1}{2}$ &  $-\frac{5}{6}$ & $0$ & $1$ \\
\hline
$L_{\{x\},\{y,t\}}$ & $0$ & $0$ & $1$ &  $0$ & $-\frac{5}{4}$ & $1$ \\
\hline
 $H_{\lambda}$   & $1$ & $1$ & $1$ & $1$ & $1$ & $4$ \\
\hline
\end{tabular}
\end{center}
This matrix has rank $6$.
One the other hand, if $\lambda\ne -2$, then
$$
H_\lambda\sim 2L_{\{x\},\{z\}}+L_{\{y\},\{x,z\}}+L_{\{t\},\{x,z\}}\sim L_{\{x\},\{y,t\}}+2L_{\{y\},\{t\}}+L_{\{z\},\{y,t\}}
$$
on the surface $S_{\lambda}$, because $H_{\{x,z\}}\cdot S_\lambda=2L_{\{x\},\{z\}}+L_{\{y\},\{x,z\}}+L_{\{t\},\{x,z\}}$ and
$$
H_{\{y,t\}}\cdot S_\lambda=L_{\{x\},\{y,t\}}+2L_{\{y\},\{t\}}+L_{\{z\},\{y,t\}}.
$$
Thus, if $\lambda\ne -2$ and $\lambda\ne -2$, then the intersection matrix of the curves
$L_{\{x\},\{z\}}$, $L_{\{x\},\{t\}}$, $L_{\{y\},\{t\}}$,  $L_{\{z\},\{t\}}$, $L_{\{x\},\{y,z\}}$, $L_{\{x\},\{y,t\}}$,
$L_{\{y\},\{x,z\}}$,  $L_{\{z\},\{x,y\}}$, $L_{\{z\},\{y,t\}}$,
$L_{\{t\},\{x,z\}}$, and $\mathcal{C}$ also has rank $6$.
Hence, we see that \eqref{equation:main-2-simple} holds,
because $\mathrm{rk}\,\mathrm{Pic}(\widetilde{S}_{\Bbbk})=\mathrm{rk}\,\mathrm{Pic}(S_{\Bbbk})+12$.
Then \eqref{equation:main-2} in Main Theorem holds by Lemma~\ref{lemma:cokernel}.

\subsection{Family \textnumero $2.17$}
\label{section:r-2-n-17}

The threefold $X$ is a blow up of a smooth quadric threefold along a smooth elliptic curve of degree $5$, so that $h^{1,2}(X)=1$.
A~toric Landau--Ginzburg model of this family is given by Minkowski polynomial \textnumero $1926$, which is
$$
x+y+z+\frac{x}{y}+\frac{y}{x}+\frac{z}{y}+\frac{z}{x}+\frac{1}{z}+\frac{2}{y}+\frac{1}{x}+\frac{z}{xy}+\frac{1}{xz}+\frac{1}{xy}.
$$
The quartic pencil $\mathcal{S}$ is given by
$$
x^2yz+xy^2z+xyz^2+x^2zt+y^2zt+xz^2t+yz^2t+xyt^2+2xzt^2+yzt^2+z^2t^2+yt^3+zt^3=\lambda xyzt.
$$

As usual, we suppose that  $\lambda\ne\infty$. Let $\mathcal{C}$ be the conic $\{x=yz+tz+t^2=0\}$.
Then
\begin{equation}
\label{equation:2-17}
\begin{split}
H_{\{x\}}\cdot S_\lambda&=L_{\{x\},\{t\}}+L_{\{x\},\{y,z\}}+\mathcal{C},\\
H_{\{y\}}\cdot S_\lambda&=L_{\{y\},\{z\}}+L_{\{y\},\{t\}}+L_{\{y\},\{x,t\}}+L_{\{y\},\{x,z,t\}},\\
H_{\{z\}}\cdot S_\lambda&=L_{\{y\},\{z\}}+2L_{\{z\},\{t\}}+L_{\{z\},\{x,t\}},\\
H_{\{t\}}\cdot S_\lambda&=L_{\{x\},\{t\}}+L_{\{y\},\{t\}}+L_{\{z\},\{t\}}+L_{\{t\},\{x,y,z\}}.\\
\end{split}
\end{equation}
Thus, we may assume that $C_1=L_{\{x\},\{t\}}$, \mbox{$C_2=L_{\{y\},\{z\}}$}, \mbox{$C_3=L_{\{y\},\{t\}}$}, $C_4=L_{\{z\},\{t\}}$,
\mbox{$C_5=L_{\{x\},\{y,z\}}$}, \mbox{$C_6=L_{\{y\},\{x,t\}}$}, $C_7=L_{\{z\},\{x,t\}}$, \mbox{$C_8=L_{\{y\},\{x,z,t\}}$}, \mbox{$C_9=L_{\{t\},\{x,y,z\}}$}, \mbox{$C_{10}=\mathcal{C}$}.
These are all base curves of the pencil~$\mathcal{S}$.
Note that $\textbf{m}_1=2$, $\textbf{m}_2=2$, $\textbf{m}_3=2$, $\textbf{m}_3=2$, $\textbf{m}_4=3$, $\textbf{m}_5=1$, $\textbf{m}_6=1$, $\textbf{m}_7=1$, $\textbf{m}_8=1$, $\textbf{m}_9=1$, and $\textbf{m}_{10}=1$.

If $\lambda\ne-2$, then the surface $S_\lambda\in\mathcal{S}$ has isolated singularities, so that it is irreducible.
On the other hand, one also has $S_{-2}=H_{\{x,t\}}+\mathbf{S}$,
where $\mathbf{S}$ is an irreducible cubic surface that is given by the equation
$xyz+xzt+y^2z+yz^2+yzt+yt^2+z^2t+zt^2=0$.

If $\lambda\ne-2$, then the singular points of the surface $S_\lambda$ contained in the base locus
of the pencil $\mathcal{S}$ can be described as follows:
\begin{itemize}\setlength{\itemindent}{3cm}
\item[$P_{\{x\},\{y\},\{t\}}$:] type $\mathbb{A}_2$ with quadratic term $(x+t)(y+t)$;

\item[$P_{\{x\},\{z\},\{t\}}$:] type $\mathbb{A}_3$ with quadratic term $z(x+t)$;

\item[$P_{\{y\},\{z\},\{t\}}$:] type $\mathbb{A}_2$ with quadratic term $z(y+t)$;

\item[$P_{\{x\},\{t\},\{y,z\}}$:] type $\mathbb{A}_1$ with quadratic term $(x+y+z)(x+t)-(\lambda+2)xt$;

\item[$P_{\{y\},\{t\},\{x,z\}}$:] type $\mathbb{A}_1$ with quadratic term
$$
(y+t)(x+y+z+t)-(\lambda+3)yt
$$
for $\lambda\neq -3$, type $\mathbb{A}_2$ for $\lambda=-3$;

\item[$P_{\{y\},\{z\},\{x,t\}}$:] type $\mathbb{A}_2$ with quadratic term $y(x+t+(\lambda+2)z)$;

\item[$P_{\{z\},\{t\},\{x,y\}}$:] type $\mathbb{A}_1$ with quadratic term $xz+yz-z^2-(\lambda+2)zt+t^2$;

\item[{$[0:-2,2:-1\pm\sqrt{5}]$}:] smooth point for $\lambda\ne\frac{-1\mp\sqrt{5}}{2}$, type $\mathbb{A}_1$ for $\lambda=\frac{-1\mp\sqrt{5}}{2}$.
\end{itemize}
Thus, if $\lambda\ne-2$, then $[\mathsf{f}^{-1}(\lambda)]=1$ by Corollary~\ref{corollary:irreducible-fibers}.
To find $[\mathsf{f}^{-1}(-2)]$, observe that the set $\Sigma$ consists of the points
$P_{\{y\},\{z\},\{t\}}$, $P_{\{x\},\{z\},\{t\}}$, $P_{\{x\},\{y\},\{t\}}$, $P_{\{x\},\{t\},\{y,z\}}$,
$P_{\{y\},\{z\},\{x,t\}}$, $P_{\{y\},\{t\},\{x,z\}}$, and $P_{\{z\},\{t\},\{x,y\}}$.
Thus, it follows from \eqref{equation:equation:number-of-irredubicle-components-refined} and Lemma~\ref{lemma:main} that
$$
\big[\mathsf{f}^{-1}(-2)\big]=2+\sum_{P\in\Sigma}\mathbf{D}^{-2}_{P}.
$$
Moreover, the quadratic terms of the surface $S_{\lambda}$ at the singular points
$P_{\{y\},\{z\},\{t\}}$, $P_{\{x\},\{z\},\{t\}}$, $P_{\{x\},\{y\},\{t\}}$, $P_{\{x\},\{t\},\{y,z\}}$, $P_{\{y\},\{z\},\{x,t\}}$,
$P_{\{y\},\{t\},\{x,z\}}$, and $P_{\{z\},\{t\},\{x,y\}}$ given above are also valid for $\lambda=-2$.
This shows that all these points are good double points of the surface $S_{-2}$, so that their defects vanish by Lemma~\ref{lemma:normal-crossing}.
Hence, we have $[\mathsf{f}^{-1}(-2)]=2$.
This confirms \eqref{equation:main-1} in Main Theorem, since  $h^{1,2}(X)=1$.

If $\lambda\ne -2$, then the intersection matrix of base curves of the pencil $\mathcal{S}$
on the surface $S_{\lambda}$ has the same rank as the intersection matrix of the curves
$L_{\{x\},\{y,z\}}$, $L_{\{y\},\{x,t\}}$, $L_{\{z\},\{x,t\}}$, $L_{\{y\},\{x,z,t\}}$, $L_{\{y\},\{z\}}$, and $H_{\lambda}$,
because
\begin{multline*}
L_{\{x\},\{t\}}+L_{\{x\},\{y,z\}}+\mathcal{C}\sim L_{\{y\},\{z\}}+L_{\{y\},\{t\}}+L_{\{y\},\{x,t\}}+L_{\{y\},\{x,z,t\}}\sim\\
\sim L_{\{y\},\{z\}}+2L_{\{z\},\{t\}}+L_{\{z\},\{x,t\}}\sim L_{\{x\},\{t\}}+L_{\{y\},\{t\}}+L_{\{z\},\{t\}}+L_{\{t\},\{x,y,z\}}\sim H_\lambda
\end{multline*}
on the surface $S_\lambda$ by \eqref{equation:2-17}, and
$$
2L_{\{x\},\{t\}}+L_{\{y\},\{x,t\}}+L_{\{z\},\{x,t\}}\sim H_\lambda,
$$
since $H_{\{x,t\}}\cdot S_\lambda=2L_{\{x\},\{t\}}+L_{\{y\},\{x,t\}}+L_{\{z\},\{x,t\}}$.
Moreover, if $\lambda\not\in\{-2,-3,\frac{-1\pm\sqrt{5}}{2}\}$,
then the intersection matrix of the curves $L_{\{x\},\{y,z\}}$, $L_{\{y\},\{x,t\}}$, $L_{\{z\},\{x,t\}}$, $L_{\{y\},\{x,z,t\}}$, $L_{\{y\},\{z\}}$, and $H_{\lambda}$
on the surface $S_\lambda$ is given by
\begin{center}\renewcommand\arraystretch{1.42}
\begin{tabular}{|c||c|c|c|c|c|c|}
\hline
 $\bullet$  & $L_{\{x\},\{y,z\}}$ & $L_{\{y\},\{x,t\}}$ & $L_{\{z\},\{x,t\}}$ & $L_{\{y\},\{x,z,t\}}$ & $L_{\{y\},\{z\}}$ & $H_{\lambda}$ \\
\hline\hline
$L_{\{x\},\{y,z\}}$ & $-\frac{4}{3}$ & $0$ & $0$ & $0$ & $1$ & $1$ \\
\hline
$L_{\{y\},\{x,t\}}$ & $0$ & $-\frac{2}{3}$ & $\frac{1}{3}$ &  $\frac{2}{3}$ & $\frac{2}{3}$ & $1$ \\
\hline
$L_{\{z\},\{x,t\}}$ & $0$ & $\frac{1}{3}$ & $-\frac{1}{3}$ &  $\frac{1}{3}$ & $\frac{1}{3}$ & $1$ \\
\hline
$L_{\{y\},\{x,z,t\}}$ & $0$ & $\frac{2}{3}$ & $\frac{1}{3}$ &  $-\frac{5}{6}$ & $\frac{2}{3}$ & $1$ \\
\hline
$L_{\{y\},\{z\}}$ & $1$ & $\frac{2}{3}$ & $\frac{1}{3}$ &  $\frac{2}{3}$ & $-\frac{2}{3}$ & $1$ \\
\hline
 $H_{\lambda}$   & $1$ & $1$ & $1$ & $1$ & $1$ & $4$ \\
\hline
\end{tabular}
\end{center}
The determinant of this matrix is $-\frac{7}{18}$.
But $\mathrm{rk}\,\mathrm{Pic}(\widetilde{S}_{\Bbbk})=\mathrm{rk}\,\mathrm{Pic}(S_{\Bbbk})+12$.
Thus, we see that \eqref{equation:main-2-simple} holds,
so that \eqref{equation:main-2} in Main Theorem also holds by Lemma~\ref{lemma:cokernel}.

\subsection{Family \textnumero $2.18$}
\label{section:r-2-n-18}
In this case, the threefold $X$ is  a double cover of $\mathbb{P}^1\times\mathbb{P}^2$ ramified in a divisor of bidegree $(2,2)$.
In this case, we have $h^{1,2}(X)=2$.
A~toric Landau--Ginzburg model of this family is given by Minkowski polynomial \textnumero $1922$, which is
$$
x+y+z+\frac{y}{x}+\frac{z}{x}+\frac{x}{yz}+\frac{1}{z}+\frac{1}{y}+\frac{1}{x}+\frac{2}{yz}+\frac{1}{xz}+\frac{1}{xy}+\frac{1}{xyz}.
$$
The quartic pencil $\mathcal{S}$ is given by
$$
x^2yz+xy^2z+xyz^2+y^2zt+yz^2t+x^2t^2+xyt^2+xzt^2+yzt^2+2xt^3+yt^3+zt^3+t^4=\lambda xyzt.
$$

Suppose that $\lambda\ne\infty$. Let $\mathcal{C}$ be the conic which is given by $x=yz+t^2=0$. Then
\begin{equation}
\label{equation:2-18}
\begin{split}
H_{\{x\}}\cdot S_\lambda&=L_{\{x\},\{t\}}+L_{\{x\},\{y,z,t\}}+\mathcal{C},\\
H_{\{y\}}\cdot S_\lambda&=2L_{\{y\},\{t\}}+L_{\{y\},\{x,t\}}+L_{\{y\},\{x,z,t\}},\\
H_{\{z\}}\cdot S_\lambda&=2L_{\{z\},\{t\}}+L_{\{z\},\{x,t\}}+L_{\{z\},\{x,y,t\}},\\
H_{\{t\}}\cdot S_\lambda&=L_{\{x\},\{t\}}+L_{\{y\},\{t\}}+L_{\{z\},\{t\}}+L_{\{t\},\{x,y,z\}}.
\end{split}
\end{equation}
Thus, we may assume that
$C_1=L_{\{x\},\{t\}}$, $C_2=L_{\{y\},\{t\}}$, \mbox{$C_3=L_{\{z\},\{t\}}$}, $C_4=L_{\{y\},\{x,t\}}$,
$C_5=L_{\{z\},\{x,t\}}$, $C_6=L_{\{x\},\{y,z,t\}}$, $C_7=L_{\{y\},\{x,z,t\}}$, \mbox{$C_8=L_{\{z\},\{x,y,t\}}$},
\mbox{$C_9=L_{\{t\},\{x,y,z\}}$}, and $C_{10}=\mathcal{C}$.
These are all base curves of the pencil~$\mathcal{S}$.

Note that $S_{-2}=H_{\{x,t\}}+H_{\{x,y,z,t\}}+\mathbf{Q}$,
where $\mathbf{Q}$ is an irreducible quadric cone in $\mathbb{P}^3$ that is given by $yz+t^2=0$.
On the other hand, if $\lambda\ne-2$, then $S_\lambda$ has isolated singularities.
Furthermore, if $\lambda\ne-2$, then the singular points of the surface $S_\lambda$ contained in the base locus
of the pencil $\mathcal{S}$ can be described as follows:
\begin{itemize}\setlength{\itemindent}{3cm}
\item[$P_{\{y\},\{z\},\{t\}}$:] type $\mathbb{A}_1$ with quadratic term $yz+t^2$;

\item[$P_{\{x\},\{z\},\{t\}}$:] type $\mathbb{A}_3$ with quadratic term $z(x+t)$;

\item[$P_{\{x\},\{y\},\{t\}}$:] type $\mathbb{A}_3$ with quadratic term $y(x+t)$;

\item[$P_{\{x\},\{t\},\{y,z\}}$:] type $\mathbb{A}_1$ with quadratic term $(x+t)(x+y+z+t)-(\lambda+2)xt$;

\item[$P_{\{y\},\{z\},\{x,t\}}$:] type $\mathbb{A}_1$ with quadratic term $x(x+y+z+t)-(\lambda+2)yz$;

\item[$P_{\{y\},\{t\},\{x,z\}}$:] type $\mathbb{A}_2$ with quadratic term $y(x+y+z-t-\lambda t)$;

\item[$P_{\{z\},\{t\},\{x,y\}}$:] type $\mathbb{A}_2$ with quadratic term $z(x+y+z-t-\lambda t)$.
\end{itemize}
Thus, the set $\Sigma$ consists of the points
$P_{\{y\},\{z\},\{t\}}$, $P_{\{x\},\{z\},\{t\}}$, $P_{\{x\},\{y\},\{t\}}$,
$P_{\{x\},\{t\},\{y,z\}}$, $P_{\{y\},\{z\},\{x,t\}}$,
$P_{\{y\},\{t\},\{x,z\}}$, and $P_{\{z\},\{t\},\{x,y\}}$.

If $\lambda\ne-2$, then the fiber $\mathsf{f}^{-1}(\lambda)$ is irreducible by Corollary~\ref{corollary:irreducible-fibers}.
Similarly, it follows from \eqref{equation:equation:number-of-irredubicle-components-refined},
Lemma~\ref{lemma:main} and Lemma~\ref{lemma:normal-crossing} that  $[\mathsf{f}^{-1}(-2)]=[S_{-2}]=3$,
because
$$
\mathbf{M}_{1}^{-2}=\mathbf{M}_{2}^{-2}=\mathbf{M}_{3}^{-2}=\mathbf{M}_{4}^{-2}=
\mathbf{M}_{5}^{-2}=\mathbf{M}_{6}^{-2}=\mathbf{M}_{7}^{-2}=\mathbf{M}_{8}^{-26}=\mathbf{M}_{9}^{-2}=\mathbf{M}_{10}^{-26}=1,
$$
and every point of the set $\Sigma$ is a good double point of the surface $S_{-2}$.
This confirms \eqref{equation:main-1} in Main Theorem, since $h^{1,2}(X)=2$.

To verify \eqref{equation:main-2} in Main Theorem, observe that
$$
H_{\{x,t\}}\cdot S_\lambda=2L_{\{x\},\{t\}}+L_{\{y\},\{x,t\}}+L_{\{z\},\{x,t\}}
$$
for $\lambda\ne -2$. Thus, if $\lambda\ne -2$, then $2L_{\{x\},\{t\}}+L_{\{y\},\{x,t\}}+L_{\{z\},\{x,t\}}\sim H_\lambda$ on the surface~$S_\lambda$.
Likewise, if $\lambda\ne -2$, then
$$
2L_{\{x\},\{y,z,t\}}+L_{\{y\},\{x,z,t\}}+L_{\{z\},\{x,y,t\}}+L_{\{t\},\{x,y,z\}}\sim H_\lambda,
$$
since $H_{\{x,y,z,t\}}\cdot S_\lambda=2L_{\{x\},\{y,z,t\}}+L_{\{y\},\{x,z,t\}}+L_{\{z\},\{x,y,t\}}+L_{\{t\},\{x,y,z\}}$.
If $\lambda\ne -2$, then
\begin{multline*}
H_\lambda\sim L_{\{x\},\{t\}}+L_{\{x\},\{y,z,t\}}+\mathcal{C}\sim 2L_{\{y\},\{t\}}+L_{\{y\},\{x,t\}}+L_{\{y\},\{x,z,t\}}\sim\\
\sim 2L_{\{z\},\{t\}}+L_{\{z\},\{x,t\}}+L_{\{z\},\{x,y,t\}}\sim L_{\{x\},\{t\}}+L_{\{y\},\{t\}}+L_{\{z\},\{t\}}+L_{\{t\},\{x,y,z\}}
\end{multline*}
on the surface $S_\lambda$. This follows from \eqref{equation:2-18}.
So, if $\lambda\ne -2$, then the rank of the intersection matrix of the curves
$L_{\{x\},\{t\}}$, $L_{\{y\},\{t\}}$, $L_{\{z\},\{t\}}$, $L_{\{y\},\{x,t\}}$,
$L_{\{z\},\{x,t\}}$, $L_{\{x\},\{y,z,t\}}$, $L_{\{y\},\{x,z,t\}}$, $L_{\{z\},\{x,y,t\}}$,
$L_{\{t\},\{x,y,z\}}$, and $\mathcal{C}$ on the surface $S_{\lambda}$ is $5$, since
the intersection matrix of the curves $L_{\{x\},\{t\}}$, $L_{\{x\},\{y,z,t\}}$, $L_{\{y\},\{x,z,t\}}$, $L_{\{y\},\{x,t\}}$, and $H_{\lambda}$
on the surface $S_\lambda$ is given~by
\begin{center}\renewcommand\arraystretch{1.42}
\begin{tabular}{|c||c|c|c|c|c|}
\hline
 $\bullet$  & $L_{\{x\},\{t\}}$ & $L_{\{x\},\{y,z,t\}}$ & $L_{\{y\},\{x,z,t\}}$ & $L_{\{y\},\{x,t\}}$ & $H_{\lambda}$ \\
\hline\hline
$L_{\{x\},\{t\}}$ & $0$ & $\frac{1}{2}$ & $0$ & $\frac{3}{4}$ & $1$ \\
\hline
$L_{\{x\},\{y,z,t\}}$ & $\frac{1}{2}$ & $-\frac{3}{2}$ & $1$ &  $0$ & $1$ \\
\hline
$L_{\{y\},\{x,z,t\}}$ & $0$ & $1$ & $-\frac{5}{6}$ &  $\frac{1}{2}$ & $1$ \\
\hline
$L_{\{y\},\{x,t\}}$ & $\frac{3}{4}$ & $0$ & $\frac{1}{2}$ &  $-\frac{1}{2}$ & $1$ \\
\hline
 $H_{\lambda}$   & $1$ & $1$ & $1$ & $1$ & $4$ \\
\hline
\end{tabular}
\end{center}
One the other hand, we have
$\mathrm{rk}\,\mathrm{Pic}(\widetilde{S}_{\Bbbk})=\mathrm{rk}\,\mathrm{Pic}(S_{\Bbbk})+13$,
so that \eqref{equation:main-2-simple} holds in this case.
Then \eqref{equation:main-2} in Main Theorem holds by Lemma~\ref{lemma:cokernel}.

\subsection{Family \textnumero $2.19$}
\label{section:r-2-n-19}

In this case, the threefold $X$ can be obtained by blowing up a smooth complete intersection of two quadrics in $\mathbb{P}^5$ along a line,
so that $h^{1,2}(X)=2$.
A~toric Landau--Ginzburg model of this family is given by
$$
x+y+z+\frac{x}{y}+\frac{z}{y}+\frac{yz}{x}+\frac{x}{z}+\frac{x}{yz}+\frac{y}{z}+\frac{1}{y}+\frac{y}{x},
$$
which is Minkowski polynomial \textnumero $1108$.
Then the pencil $\mathcal{S}$ is given by
$$
x^2yz+xyz^2+x^2zt+xy^2z+xz^2t+y^2z^2+x^2yt+x^2t^2+xy^2t+xzt^2+y^2zt=\lambda xyzt.
$$

For simplicity, we assume that $\lambda\ne\infty$.
If $\lambda\ne-1$, then the surface $S_\lambda$ has isolated singularities, so that it is irreducible.
On the other hand, we have $S_{-1}=H_{\{x,z\}}+H_{\{x,t\}}+\mathbf{Q}$,
where $\mathbf{Q}$ is a smooth quadric in $\mathbb{P}^3$ that is given by $xy+xt+y^2=0$.

Let $\mathcal{C}$ be the conic in $\mathbb{P}^3$ that is given by $z=xy+xt+y^2=0$.
Then
\begin{itemize}
\item $H_{\{x\}}\cdot S_\lambda=2L_{\{x\},\{y\}}+L_{\{x\},\{z\}}+L_{\{x\},\{z,t\}}$,
\item $H_{\{y\}}\cdot S_\lambda=L_{\{x\},\{y\}}+L_{\{y\},\{t\}}+L_{\{y\},\{x,z\}}+L_{\{y\},\{z,t\}}$,
\item $H_{\{z\}}\cdot S_\lambda=L_{\{x\},\{z\}}+L_{\{z\},\{t\}}+\mathcal{C}$,
\item $H_{\{t\}}\cdot S_\lambda=L_{\{y\},\{t\}}+L_{\{z\},\{t\}}+L_{\{t\},\{x,y\}}+L_{\{t\},\{x,z\}}$.
\end{itemize}
This shows that the base locus of the pencil $\mathcal{S}$ consists of the curves
$L_{\{x\},\{y\}}$, $L_{\{x\},\{z\}}$, $L_{\{y\},\{t\}}$, $L_{\{z\},\{t\}}$, $L_{\{x\},\{z,t\}}$,
 $L_{\{y\},\{x,z\}}$, $L_{\{y\},\{z,t\}}$, $L_{\{t\},\{x,y\}}$, $L_{\{t\},\{x,z\}}$, and $\mathcal{C}$.
This also gives
\begin{multline*}
2L_{\{x\},\{y\}}+L_{\{x\},\{z\}}+L_{\{x\},\{z,t\}}\sim L_{\{x\},\{y\}}+L_{\{y\},\{t\}}+L_{\{y\},\{x,z\}}+L_{\{y\},\{z,t\}}\sim\\
\sim L_{\{x\},\{z\}}+L_{\{z\},\{t\}}+\mathcal{C}\sim L_{\{y\},\{t\}}+L_{\{z\},\{t\}}+L_{\{t\},\{x,y\}}+L_{\{t\},\{x,z\}}\sim H_\lambda
\end{multline*}
on the surface $S_\lambda$ with $\lambda\ne -1$.

If $\lambda\ne -1$, then the singularities of the surface $S_\lambda$ contained in the base locus of the pencil $\mathcal{S}$
can be described as follows:
\begin{itemize}\setlength{\itemindent}{3cm}
\item[$P_{\{y\},\{z\},\{t\}}$:] type $\mathbb{A}_2$ with quadratic term $(y+t)(z+t)$;

\item[$P_{\{x\},\{z\},\{t\}}$:] type $\mathbb{A}_2$ with quadratic term $(x+z)(z+t)$;

\item[$P_{\{x\},\{y\},\{t\}}$:] type $\mathbb{A}_1$ with quadratic term $xy+xt+y^2$;

\item[$P_{\{x\},\{y\},\{z\}}$:] type $\mathbb{A}_4$ with quadratic term $x(x+z)$;

\item[$P_{\{x\},\{y\},\{z,t\}}$:] type $\mathbb{A}_2$ with quadratic term $x(y+\lambda y-z-t)$;

\item[$P_{\{y\},\{t\},\{x,z\}}$:] type $\mathbb{A}_1$ with quadratic term $(x+z)(y+t)-(\lambda+1)yt$;

\item[$P_{\{z\},\{t\},\{x,y\}}$:] type $\mathbb{A}_1$ with quadratic term $(z+t)(x+y-t)-(\lambda+1)zt$.
\end{itemize}
These quadratic terms remain valid also for $\lambda=-1$.
Thus, using \eqref{equation:equation:number-of-irredubicle-components-refined}
and applying Lemmas~\ref{lemma:main} and \ref{lemma:normal-crossing},
we see that $[\mathsf{f}^{-1}(-1)]=[S_{-1}]=3$, because $S_{-1}$ is smooth at general points of the curves
$L_{\{x\},\{y\}}$, $L_{\{x\},\{z\}}$, $L_{\{y\},\{t\}}$, $L_{\{z\},\{t\}}$, $L_{\{x\},\{z,t\}}$,
$L_{\{y\},\{x,z\}}$, $L_{\{y\},\{z,t\}}$, $L_{\{t\},\{x,y\}}$, $L_{\{t\},\{x,z\}}$, and $\mathcal{C}$.
This confirms \eqref{equation:main-1} in Main Theorem, since $h^{1,2}(X)=2$.

To verify \eqref{equation:main-2} in Main Theorem, observe that
$\mathrm{rk}\,\mathrm{Pic}(\widetilde{S}_{\Bbbk})=\mathrm{rk}\,\mathrm{Pic}(S_{\Bbbk})+13$.
Moreover, if $\lambda\ne -1$, then the rank of the intersection matrix of the curves
$L_{\{x\},\{y\}}$, $L_{\{x\},\{z\}}$, $L_{\{y\},\{t\}}$, $L_{\{z\},\{t\}}$, $L_{\{x\},\{z,t\}}$,
$L_{\{y\},\{x,z\}}$, $L_{\{y\},\{z,t\}}$, $L_{\{t\},\{x,y\}}$, $L_{\{t\},\{x,z\}}$, and $\mathcal{C}$
on the surface $S_{\lambda}$ is the same as  the rank of the intersection matrix of the curves
$L_{\{x\},\{z,t\}}$, $L_{\{y\},\{t\}}$, $L_{\{y\},\{x,z\}}$, $L_{\{y\},\{z,t\}}$, $L_{\{z\},\{t\}}$, $L_{\{t\},\{x,y\}}$, and $H_{\lambda}$.
Thus, using Lemma~\ref{lemma:cokernel}, we see that \eqref{equation:main-2} in Main Theorem holds in this case,
because the matrix in the following lemma has rank $5$.

\begin{lemma}
\label{lemma:r2-n19-intersection}
Suppose that $\lambda\ne-1$.
Then the intersection matrix of the curves $L_{\{x\},\{z,t\}}$, $L_{\{y\},\{t\}}$, $L_{\{y\},\{x,z\}}$, $L_{\{y\},\{z,t\}}$, $L_{\{z\},\{t\}}$, $L_{\{t\},\{x,y\}}$, and $H_{\lambda}$
on the surface $S_\lambda$ is given~by
\begin{center}\renewcommand\arraystretch{1.42}
\begin{tabular}{|c||c|c|c|c|c|c|c|}
\hline
 $\bullet$  & $L_{\{x\},\{z,t\}}$ & $L_{\{y\},\{t\}}$ & $L_{\{y\},\{x,z\}}$ & $L_{\{y\},\{z,t\}}$ & $L_{\{z\},\{t\}}$ & $L_{\{t\},\{x,y\}}$ & $H_{\lambda}$ \\
\hline\hline
$L_{\{x\},\{z,t\}}$ & $-\frac{2}{3}$ & $0$ & $0$ & $\frac{1}{3}$ & $\frac{2}{3}$ & $0$ & $1$ \\
\hline
$L_{\{y\},\{t\}}$ & $0$ & $-\frac{1}{3}$ & $\frac{1}{2}$ &  $\frac{1}{3}$ & $\frac{1}{3}$ & $\frac{1}{2}$ & $1$ \\
\hline
$L_{\{y\},\{x,z\}}$ & $0$ & $\frac{1}{2}$ & $-\frac{7}{10}$ &  $1$ & $0$ & $0$ & $1$ \\
\hline
$L_{\{y\},\{z,t\}}$ & $\frac{1}{3}$ & $\frac{1}{3}$ & $1$ &  $-\frac{2}{3}$ & $\frac{2}{3}$ & $0$ & $1$ \\
\hline
$L_{\{z\},\{t\}}$ & $\frac{2}{3}$ & $\frac{1}{3}$ & $0$ &  $\frac{2}{3}$ & $-\frac{1}{6}$ & $\frac{1}{2}$ & $1$ \\
\hline
$L_{\{t\},\{x,y\}}$ & $\frac{1}{2}$ & $0$ & $0$ &  $0$ & $0$ & $-\frac{1}{2}$ & $1$ \\
\hline
 $H_{\lambda}$  & $1$ & $1$ & $1$ & $1$ & $1$ & $1$ & $4$ \\
\hline
\end{tabular}
\end{center}
\end{lemma}

\begin{proof}
The last column and the last raw in this matrix are obvious.
To find its diagonal entries, we use  Proposition~\ref{proposition:du-Val-self-intersection}.
For instance, the line $L_{\{x\},\{z,t\}}$ contains two singular points of the surface $S_\lambda$.
These are the points $P_{\{x\},\{z\},\{t\}}$ and $P_{\{x\},\{y\},\{z,t\}}$.
Both of them are singular points of type $\mathbb{A}_2$.
Thus, by Proposition~\ref{proposition:du-Val-self-intersection},
we have
$$
L_{\{x\},\{z\}}^2=-2+\frac{2}{3}+\frac{2}{3}=-\frac{2}{3}.
$$
Likewise, we obtain the remaining diagonal entries.

To find the remaining entries of the intersection matrix, observe
that the line $L_{\{x\},\{z,t\}}$ does not intersect the lines $L_{\{y\},\{t\}}$, $L_{\{y\},\{x,z\}}$, and $L_{\{t\},\{x,y\}}$,
so that
$$
L_{\{x\},\{z,t\}}\cdot L_{\{y\},\{t\}}=L_{\{x\},\{z,t\}}\cdot L_{\{y\},\{x,z\}}=L_{\{x\},\{z,t\}}\cdot L_{\{t\},\{x,y\}}=0.
$$

Now observe that $L_{\{x\},\{z,t\}}\cap L_{\{y\},\{z,t\}}=P_{\{x\},\{y\},\{z,t\}}$,
which is a singular point of the surface $S_\lambda$ of type $\mathbb{A}_2$.
Moreover, the strict transforms of the lines $L_{\{x\},\{z,t\}}$ and $L_{\{y\},\{z,t\}}$
on the minimal resolution of singularities of the surface $S_\lambda$ at the point $P_{\{x\},\{y\},\{z,t\}}$
intersect different exceptional curves.
This implies that $L_{\{x\},\{z,t\}}\cdot L_{\{y\},\{z,t\}}=\frac{1}{3}$ by Proposition~\ref{proposition:du-Val-self-intersection}.
Similarly, we see that $L_{\{x\},\{z,t\}}\cdot L_{\{z\},\{t\}}=\frac{2}{3}$,
$L_{\{y\},\{t\}}\cdot L_{\{y\},\{x,z\}}=\frac{1}{2}$, $L_{\{y\},\{t\}}\cdot L_{\{y\},\{z,t\}}=\frac{1}{3}$,
$L_{\{y\},\{t\}}\cdot L_{\{z\},\{t\}}=\frac{1}{3}$, $L_{\{y\},\{t\}}\cdot L_{\{t\},\{x,y\}}=\frac{1}{2}$
and $L_{\{y\},\{z,t\}}\cdot L_{\{z\},\{t\}}=\frac{2}{3}$.

Observe that the line $L_{\{y\},\{x,z\}}$ does not intersect the lines $L_{\{z\},\{t\}}$ and $L_{\{t\},\{x,y\}}$,
and the line $L_{\{y\},\{z,t\}}$ does not intersect the line $L_{\{t\},\{x,y\}}$, so that
$$
L_{\{y\},\{x,z\}}\cdot L_{\{z\},\{t\}}=L_{\{y\},\{x,z\}}\cdot L_{\{t\},\{x,y\}}=L_{\{y\},\{z,t\}}\cdot L_{\{t\},\{x,y\}}=0.
$$
Moreover, the intersection $L_{\{y\},\{x,z\}}\cap L_{\{y\},\{z,t\}}$ consists of a smooth point of the surface~$S_\lambda$.
Thus, we have $L_{\{y\},\{x,z\}}\cdot L_{\{y\},\{z,t\}}=1$.

Finally, observe that $L_{\{z\},\{t\}}\cap L_{\{t\},\{x,y\}}=P_{\{z,t\},\{x,y\}}$,
which is a singular point of the surface $S_\lambda$ of type $\mathbb{A}_1$.
Thus, we have $L_{\{z\},\{t\}}\cdot L_{\{t\},\{x,y\}}=\frac{1}{2}$ by Proposition~\ref{proposition:du-Val-self-intersection}.
\end{proof}

\subsection{Family \textnumero $2.20$}
\label{section:r-2-n-20}
In this case, the threefold $X$ is  a blow up of the threefold $V_5$ along a twisted cubic (see Subsection~\ref{section:r-2-n-14}).
Thus, we have $h^{1,2}(X)=0$.
A~toric Landau--Ginzburg model of this family is given by Minkowski polynomial \textnumero $1109$, which is
$$
\frac{y}{z}+x+y+\frac{1}{z}+\frac{y}{xz}+\frac{y}{x}+\frac{1}{xz}+\frac{xz}{y}+z+\frac{1}{y}+\frac{1}{x}.
$$
The pencil $\mathcal{S}$ is given by the equation
$$
y^2tx+x^2zy+y^2zx+t^2xy+t^2y^2+y^2zt+t^3y+x^2z^2+z^2xy+t^2zx+t^2zy=\lambda xyzt.
$$
Suppose that $\lambda\ne\infty$. Then the surface $S_\lambda$ has isolated singularities.
In particular, we see that $S_\lambda$ is irreducible.

Let $\mathcal{C}$ be a conic in $\mathbb{P}^3$ that is given by $y=xz+t^2=0$. Then
\begin{equation}
\label{equation:2-20}
\begin{split}
H_{\{x\}}\cdot S_\lambda&=L_{\{x\},\{y\}}+L_{\{x\},\{t\}}+L_{\{x\},\{y,t\}}+L_{\{x\},\{z,t\}},\\
H_{\{y\}}\cdot S_\lambda&=L_{\{x\},\{y\}}+L_{\{y\},\{z\}}+\mathcal{C},\\
H_{\{z\}}\cdot S_\lambda&=L_{\{y\},\{z\}}+L_{\{z\},\{t\}}+L_{\{z\},\{x,t\}}+L_{\{z\},\{y,t\}},\\
H_{\{t\}}\cdot S_\lambda&=L_{\{x\},\{t\}}+L_{\{z\},\{t\}}+L_{\{t\},\{x,y\}}+L_{\{t\},\{y,z\}}.
\end{split}
\end{equation}
This shows that
$L_{\{x\},\{y\}}$, $L_{\{x\},\{t\}}$, $L_{\{y\},\{z\}}$, $L_{\{z\},\{t\}}$, $L_{\{x\},\{y,t\}}$, $L_{\{x\},\{z,t\}}$
$L_{\{z\},\{x,t\}}$, $L_{\{z\},\{y,t\}}$, $L_{\{t\},\{x,y\}}$, $L_{\{t\},\{y,z\}}$, and $\mathcal{C}$
are all base curves of the pencil $\mathcal{S}$.

If $\lambda\ne -2$ and $\lambda\ne -3$, then the singular points of the surface $S_\lambda$
contained in the base locus of the pencil $\mathcal{S}$ can be described as follows:
\begin{itemize}\setlength{\itemindent}{3cm}
\item[$P_{\{y\},\{z\},\{t\}}$:] type $\mathbb{A}_4$ with quadratic term $z(y+z)$;

\item[$P_{\{x\},\{z\},\{t\}}$:] type $\mathbb{A}_2$ with quadratic term $(x+t)(z+t)$;

\item[$P_{\{x\},\{y\},\{t\}}$:] type $\mathbb{A}_4$ with quadratic term $x(x+y)$.
\end{itemize}
The surface $S_{-3}$ has the same singularities at $P_{\{y\},\{z\},\{t\}}$, $P_{\{x\},\{z\},\{t\}}$, and $P_{\{x\},\{y\},\{t\}}$.
In~addition to them, it is also singular at the points $[0:1:1:-1]$ and $[1:1:0:-1]$, which are isolated ordinary double points of the surface $S_{-3}$.
Similarly, the singular points of the surface $S_{-2}$ are $P_{\{y\},\{z\},\{t\}}$, $P_{\{x\},\{z\},\{t\}}$, $P_{\{x\},\{y\},\{t\}}$, and $[1:-1:1:0]$.
They are singular points of the surface $S_{-2}$ of types $\mathbb{A}_6$, $\mathbb{A}_2$, $\mathbb{A}_6$, and $\mathbb{A}_1$, respectively.

We see that every surface $S_{\lambda}$ has du Val singularities in every base point of the pencil~$\mathcal{S}$.
Thus, by Lemma~\ref{corollary:irreducible-fibers}, every fiber $\mathsf{f}^{-1}(\lambda)$ is irreducible.
This confirms \eqref{equation:main-1} in Main Theorem, since $h^{1,2}(X)=0$.
To verify \eqref{equation:main-2} in Main Theorem, we need

\begin{lemma}
\label{lemma:r2-n20-intersection}
Suppose that $\lambda\ne-2$ and $\lambda\ne-3$.
Then the intersection matrix of the curves $L_{\{x\},\{y\}}$, $L_{\{x\},\{y,t\}}$, $L_{\{x\},\{z,t\}}$, $L_{\{y\},\{z\}}$, $L_{\{z\},\{x,t\}}$, $L_{\{z\},\{y,t\}}$, $L_{\{t\},\{y,z\}}$, and $H_{\lambda}$
on the surface $S_\lambda$ is given by
\begin{center}\renewcommand\arraystretch{1.42}
\begin{tabular}{|c||c|c|c|c|c|c|c|c|}
\hline
 $\bullet$  & $L_{\{x\},\{y\}}$ & $L_{\{x\},\{y,t\}}$ & $L_{\{x\},\{z,t\}}$ & $L_{\{y\},\{z\}}$ & $L_{\{z\},\{x,t\}}$ & $L_{\{z\},\{y,t\}}$ &
 $L_{\{t\},\{y,z\}}$ & $H_{\lambda}$ \\
\hline\hline
$L_{\{x\},\{y\}}$ & $-\frac{4}{5}$ & $\frac{2}{5}$ & $0$ & $1$ & $0$ & $0$ & $0$ & $1$ \\
\hline
$L_{\{x\},\{y,t\}}$ & $\frac{2}{5}$ & $-\frac{6}{5}$ & $1$ &  $0$ & $0$ & $1$ & $0$ & $1$ \\
\hline
$L_{\{x\},\{z,t\}}$ & $0$ & $1$ & $-\frac{4}{3}$ &  $0$ & $\frac{1}{3}$ & $0$ & $0$ & $1$ \\
\hline
$L_{\{y\},\{z\}}$ & $1$ & $0$ & $0$ &  $-\frac{4}{5}$ & $1$ & $\frac{2}{5}$ & $\frac{3}{5}$ & $1$ \\
\hline
$L_{\{z\},\{x,t\}}$ & $0$ & $0$ & $\frac{1}{3}$ &  $1$ & $-\frac{4}{3}$ & $1$ & $0$ & $1$ \\
\hline
$L_{\{z\},\{y,t\}}$ & $0$ & $1$ & $0$ & $\frac{2}{5}$ & $1$ & $-\frac{6}{5}$ & $\frac{1}{5}$ & $1$ \\
\hline
$L_{\{t\},\{y,z\}}$ & $0$ & $0$ & $0$ & $\frac{3}{5}$ &  $0$ & $\frac{1}{5}$ & $-\frac{6}{5}$ & $1$ \\
\hline
 $H_{\lambda}$  & $1$ & $1$ & $1$ & $1$ & $1$ & $1$ & $1$ & $4$ \\
\hline
\end{tabular}
\end{center}
\end{lemma}

\begin{proof}
All diagonal entries here can be found using Proposition~\ref{proposition:du-Val-self-intersection}.
For instance, the only singular point of the surface $S_\lambda$ that is contained in $L_{\{x\},\{y\}}$
is the point $P_{\{x\},\{y\},\{t\}}$, which is a singular point of type $\mathbb{A}_4$ of the surface $S_\lambda$.
Applying Remark~\ref{remark:transversal} with $S=S_\lambda$, $n=4$, $O=P_{\{x\},\{y\},\{t\}}$, and $C=L_{\{x\},\{y\}}$,
we see that $\overline{C}$ contains the point $\overline{G}_1\cap\overline{G}_4$,
because the quadratic term of the surface $S_\lambda$ at the point $P_{\{x\},\{y\},\{t\}}$ is $x(x+y)$.
This shows that $\widetilde{C}$ intersects either $G_2$ or $G_3$.
Then  $L_{\{x\},\{y\}}^2=-\frac{4}{5}$ by Proposition~\ref{proposition:du-Val-self-intersection}.

Applying Proposition~\ref{proposition:du-Val-intersection}, we can find the remaining entries of the intersection matrix.
For instance, observe that
$$
L_{\{x\},\{y\}}\cap L_{\{x\},\{t\}}=L_{\{x\},\{y\}}\cap L_{\{x\},\{y,t\}}=P_{\{x\},\{y\},\{t\}}.
$$
Thus, it follows from Proposition~\ref{proposition:du-Val-intersection} that
$L_{\{x\},\{y\}}\cdot L_{\{x\},\{t\}}$ and $L_{\{x\},\{y\}}\cdot L_{\{x\},\{y,t\}}$
are among $\frac{2}{5}$ and $\frac{3}{5}$.
But $L_{\{x\},\{y\}}+L_{\{x\},\{t\}}+L_{\{x\},\{y,t\}}+L_{\{x\},\{z,t\}}\sim H_{\lambda}$ by \eqref{equation:2-20}, so that
\begin{multline*}
1=\big( L_{\{x\},\{y\}}+L_{\{x\},\{t\}}+L_{\{x\},\{y,t\}}+L_{\{x\},\{z,t\}}\big)\cdot L_{\{x\},\{y\}}=\\=L_{\{x\},\{y\}}\cdot L_{\{x\},\{t\}}+L_{\{x\},\{y\}}\cdot L_{\{x\},\{y,t\}}-\frac{1}{5}.
\end{multline*}
Hence, we deduce that $L_{\{x\},\{y\}}\cdot L_{\{x\},\{t\}}=\frac{2}{5}$ and $L_{\{x\},\{y\}}\cdot L_{\{x\},\{y,t\}}=\frac{2}{5}$.
Similarly, we can find all remaining entries of the intersection matrix.
\end{proof}

The matrix in Lemma~\ref{lemma:r2-n20-intersection} has rank~$8$.
But $\mathrm{rk}\,\mathrm{Pic}(\widetilde{S}_{\Bbbk})=\mathrm{rk}\,\mathrm{Pic}(S_{\Bbbk})+10$,
so that \eqref{equation:main-2-simple} holds.
By Lemma~\ref{lemma:cokernel}, this shows that \eqref{equation:main-2} in Main Theorem also holds.

\subsection{Family \textnumero $2.21$}
\label{section:r-2-n-21}

In this case, the threefold $X$ can be obtained from a smooth quadric threefold in $\mathbb{P}^4$ by blowing up a smooth rational curve of degree $4$.
Then $h^{1,2}(X)=0$.
A~toric Landau--Ginzburg model of this family is given by
$$
\frac{x}{z}+x+\frac{y}{z}+\frac{x}{y}+\frac{1}{z}+y+z+\frac{y}{x}+\frac{z}{y}+\frac{1}{x},
$$
which is Minkowski polynomial \textnumero $730$. Then the pencil $\mathcal{S}$ is given by
$$
x^2ty+x^2zy+y^2tx+x^2zt+t^2yx+y^2zx+z^2yx+y^2zt+z^2tx+t^2zy=\lambda xyzt.
$$

As usual, we assume that $\lambda\ne\infty$. Then
\begin{itemize}
\item $H_{\{x\}}\cdot S_\lambda=L_{\{x\},\{y\}}+L_{\{x\},\{z\}}+L_{\{x\},\{t\}}+L_{\{x\},\{y,t\}}$,
\item $H_{\{y\}}\cdot S_\lambda=L_{\{x\},\{y\}}+L_{\{y\},\{z\}}+L_{\{y\},\{t\}}+L_{\{y\},\{x,z\}}$,
\item $H_{\{z\}}\cdot S_\lambda=L_{\{x\},\{z\}}+L_{\{y\},\{z\}}+L_{\{z\},\{t\}}+L_{\{z\},\{x,y,t\}}$,
\item $H_{\{t\}}\cdot S_\lambda=L_{\{x\},\{t\}}+L_{\{y\},\{t\}}+L_{\{z\},\{t\}}+L_{\{t\},\{x,y,z\}}$.
\end{itemize}
This shows that $L_{\{x\},\{y\}}$, $L_{\{x\},\{z\}}$, $L_{\{x\},\{t\}}$, $L_{\{y\},\{z\}}$, $L_{\{y\},\{t\}}$,
$L_{\{z\},\{t\}}$, $L_{\{x\},\{y,t\}}$, $L_{\{y\},\{x,z\}}$, $L_{\{z\},\{x,y,t\}}$, and $L_{\{t\},\{x,y,z\}}$
are all base curves of the pencil $\mathcal{S}$.

For every $\lambda\in\mathbb{C}$, the surface $S_\lambda$ is irreducible, it has isolated singularities,
and its singular points contained in the base locus of the pencil~$\mathcal{S}$
can be described as follows:
\begin{itemize}\setlength{\itemindent}{2cm}
\item[$P_{\{x\},\{y\},\{z\}}$:] type $\mathbb{A}_3$ with quadratic term $y(x+z)$;

\item[$P_{\{x\},\{y\},\{t\}}$:] type $\mathbb{A}_3$ with quadratic term $x(y+t)$ for $\lambda\ne -1$, type $\mathbb{A}_5$ for $\lambda=-1$;

\item[$P_{\{x\},\{z\},\{t\}}$:] type $\mathbb{A}_1$;

\item[$P_{\{y\},\{z\},\{t\}}$:] type $\mathbb{A}_1$;

\item[$P_{\{x\},\{z\},\{y,t\}}$:] type $\mathbb{A}_1$ for $\lambda\neq -2$, type $\mathbb{A}_2$ for $\lambda=-2$;

\item[$P_{\{y\},\{t\},\{x,z\}}$:] type $\mathbb{A}_1$ for $\lambda\neq -2$, type $\mathbb{A}_2$ for $\lambda=-2$;

\item[$P_{\{z\},\{t\},\{x,y\}}$:] type $\mathbb{A}_1$ for $\lambda\neq -4$, type $\mathbb{A}_2$ for $\lambda=-4$.
\end{itemize}
Then $[\mathsf{f}^{-1}(\lambda)]=1$ for every $\lambda\in\mathbb{C}$ by Lemma~\ref{corollary:irreducible-fibers}.
This confirms \eqref{equation:main-1} in Main Theorem.

The rank of the intersection matrix of the curves
$L_{\{x\},\{y\}}$, $L_{\{x\},\{z\}}$, $L_{\{x\},\{t\}}$, $L_{\{y\},\{z\}}$, $L_{\{y\},\{t\}}$,
$L_{\{z\},\{t\}}$, $L_{\{x\},\{y,t\}}$, $L_{\{y\},\{x,z\}}$, $L_{\{z\},\{x,y,t\}}$, and $L_{\{t\},\{x,y,z\}}$
on the surface $S_{\lambda}$ is the same as  the rank of the intersection matrix of the curves
$L_{\{x\},\{y\}}$, $L_{\{x\},\{t\}}$, $L_{\{y\},\{z\}}$, $L_{\{z\},\{t\}}$, $L_{\{z\},\{x,y,t\}}$, $L_{\{t\},\{x,y,z\}}$, and $H_{\lambda}$.
If $\lambda\not\in\{-1,-2,-4\}$, then the latter matrix is given by
\begin{center}\renewcommand\arraystretch{1.42}
\begin{tabular}{|c||c|c|c|c|c|c|c|}
\hline
 $\bullet$  & $L_{\{x\},\{y\}}$ & $L_{\{x\},\{t\}}$ & $L_{\{y\},\{z\}}$ & $L_{\{z\},\{t\}}$ & $L_{\{z\},\{x,y,t\}}$ & $L_{\{t\},\{x,y,z\}}$ & $H_{\lambda}$ \\
\hline\hline
$L_{\{x\},\{y\}}$ & $-\frac{1}{2}$ & $\frac{3}{4}$ & $\frac{3}{4}$ & $0$ & $0$ & $0$ & $1$ \\
\hline
$L_{\{x\},\{t\}}$ & $\frac{3}{4}$ & $-\frac{3}{4}$ & $0$ &  $\frac{1}{2}$ & $0$ & $1$ & $1$ \\
\hline
$L_{\{y\},\{z\}}$ & $\frac{3}{4}$ & $0$ & $-\frac{3}{4}$ &  $\frac{1}{2}$ & $0$ & $1$ & $1$ \\
\hline
$L_{\{z\},\{t\}}$ & $0$ & $\frac{1}{2}$ & $\frac{1}{2}$ &  $-1$ & $1$ & $1$ & $1$ \\
\hline
$L_{\{z\},\{x,y,t\}}$ & $0$ & $0$ & $1$ &  $1$ & $-\frac{3}{2}$ & $1$ & $1$ \\
\hline
$L_{\{t\},\{x,y,z\}}$ & $0$ & $1$ & $0$ &  $1$ & $1$ & $-1$ & $1$ \\
\hline
 $H_{\lambda}$  & $1$ & $1$ & $1$ & $1$ & $1$ & $1$ & $4$ \\
\hline
\end{tabular}
\end{center}
Thus, its determinant is $-\frac{45}{16}\ne 0$.
Moreover, we have $\mathrm{rk}\,\mathrm{Pic}(\widetilde{S}_{\Bbbk})=\mathrm{rk}\,\mathrm{Pic}(S_{\Bbbk})+11$.
Hence, we see that \eqref{equation:main-2-simple} holds, so that \eqref{equation:main-2} in Main Theorem also holds by Lemma~\ref{lemma:cokernel}.

\subsection{Family \textnumero $2.22$}
\label{section:r-2-n-22}

The threefold $X$ is a blow up of the threefold $V_5$ along a conic (see Subsection~\ref{section:r-2-n-14}),
so that $h^{1,2}(X)=0$.
A~toric Landau--Ginzburg model of this family is given by Minkowski polynomial \textnumero $413$, which is
$$
\frac{y}{x}+\frac{1}{x}+y+z+\frac{1}{xz}+\frac{1}{z}+\frac{1}{y}+x+\frac{xz}{y}.
$$
The quartic pencil $\mathcal{S}$ is given by
$$
y^2zt+t^2zy+y^2zx+z^2yx+t^3y+t^2yx+t^2zx+x^2zy+x^2z^2=\lambda xyzt.
$$

Suppose that $\lambda\ne\infty$.
Let $\mathcal{C}_1$ be the conic in $\mathbb{P}^3$ that is given by $x=yz+zt+t^2=0$,
and let $\mathcal{C}_2$ be the conic in $\mathbb{P}^3$ that is given by $y=xz+t^2=0$.
Then
\begin{equation}
\label{equation:2-22}
\begin{split}
H_{\{x\}}\cdot S_\lambda&=L_{\{x\},\{y\}}+L_{\{x\},\{t\}}+\mathcal{C}_1,\\
H_{\{y\}}\cdot S_\lambda&=L_{\{x\},\{y\}}+L_{\{y\},\{z\}}+\mathcal{C}_2,\\
H_{\{z\}}\cdot S_\lambda&=L_{\{y\},\{z\}}+2L_{\{z\},\{t\}}+L_{\{z\},\{x,t\}},\\
H_{\{t\}}\cdot S_\lambda&=L_{\{x\},\{t\}}+L_{\{z\},\{t\}}+L_{\{t\},\{x,y\}}+L_{\{t\},\{y,z\}}.
\end{split}
\end{equation}
Thus, the base locus of the pencil $\mathcal{S}$ consists of the curves
$L_{\{x\},\{y\}}$, $L_{\{x\},\{t\}}$, $L_{\{y\},\{z\}}$, $L_{\{z\},\{t\}}$, $L_{\{z\},\{x,t\}}$, $L_{\{t\},\{x,y\}}$, $L_{\{t\},\{y,z\}}$,
$\mathcal{C}_1$, and $\mathcal{C}_2$.

For every $\lambda\in\mathbb{C}$, the surface $S_\lambda$ is irreducible and has isolated singularities.
Moreover, the singular points of the surface $S_\lambda$ contained in the base locus of the pencil~$\mathcal{S}$ can be described as follows:
\begin{itemize}\setlength{\itemindent}{1.5cm}
\item[$P_{\{x\},\{y\},\{t\}}$:] type $\mathbb{A}_4$ with quadratic term $x(x+y)$ for $\lambda\neq -2$, type $\mathbb{A}_5$ for $\lambda=-2$;

\item[$P_{\{x\},\{z\},\{t\}}$:] type $\mathbb{A}_3$ with quadratic term $z(x+t)$ for $\lambda\neq -2$, type $\mathbb{A}_5$ for $\lambda=-2$;

\item[$P_{\{y\},\{z\},\{t\}}$:] type $\mathbb{A}_4$ with quadratic term $z(y+z)$ for $\lambda\neq -1$, type $\mathbb{A}_5$ for $\lambda=-1$;

\item[$P_{\{z\},\{t\},\{x,y\}}$:] type $\mathbb{A}_1$;

\item[{$[1:-1:1:0]$}:] smooth for $\lambda\ne-1$, type $\mathbb{A}_1$ for $\lambda=-1$.
\end{itemize}
Then $[\mathsf{f}^{-1}(\lambda)]=1$ for every $\lambda\in\mathbb{C}$ by Lemma~\ref{corollary:irreducible-fibers}.
This confirms \eqref{equation:main-1} in Main Theorem.

If $\lambda\ne-1$ and $\lambda\ne-2$, then the intersection matrix of the curves $L_{\{x\},\{y\}}$, $L_{\{z\},\{t\}}$, $L_{\{z\},\{x,t\}}$, $L_{\{t\},\{x,y\}}$, $L_{\{t\},\{y,z\}}$, and $H_{\lambda}$
on the surface $S_\lambda$ is given by
\begin{center}\renewcommand\arraystretch{1.42}
\begin{tabular}{|c||c|c|c|c|c|c|}
\hline
 $\bullet$  & $L_{\{x\},\{y\}}$ & $L_{\{z\},\{t\}}$ & $L_{\{z\},\{x,t\}}$ & $L_{\{t\},\{x,y\}}$ & $L_{\{t\},\{y,z\}}$ &  $H_{\lambda}$ \\
\hline\hline
$L_{\{x\},\{y\}}$ & $-\frac{4}{5}$ & $0$ & $0$ & $\frac{1}{5}$ & $0$ & $1$ \\
\hline
$L_{\{z\},\{t\}}$ & $0$ &  $\frac{1}{20}$ & $\frac{1}{2}$ & $\frac{1}{2}$ & $\frac{1}{5}$ & $1$ \\
\hline
$L_{\{z\},\{x,t\}}$ & $0$ &  $\frac{1}{2}$ & $-1$ & $0$ & $0$ & $1$ \\
\hline
$L_{\{t\},\{x,y\}}$ & $\frac{1}{5}$ &  $\frac{1}{2}$ & $0$ & $-\frac{7}{10}$ & $1$ & $1$ \\
\hline
$L_{\{t\},\{y,z\}}$ & $0$ &  $\frac{1}{5}$ & $0$ & $1$ & $-\frac{6}{5}$ & $1$ \\
\hline
 $H_{\lambda}$  & $1$ & $1$ & $1$ & $1$ & $1$ & $4$ \\
\hline
\end{tabular}
\end{center}
This matrix has rank~$6$.
Hence, using \eqref{equation:2-22}, we see that the rank of the intersection
matrix of the curves $L_{\{x\},\{y\}}$, $L_{\{x\},\{t\}}$, $L_{\{y\},\{z\}}$, $L_{\{z\},\{t\}}$, $L_{\{z\},\{x,t\}}$,
$L_{\{t\},\{x,y\}}$, $L_{\{t\},\{y,z\}}$, $\mathcal{C}_1$, and $\mathcal{C}_2$ is also $5$.
But $\mathrm{rk}\,\mathrm{Pic}(\widetilde{S}_{\Bbbk})=\mathrm{rk}\,\mathrm{Pic}(S_{\Bbbk})+12$,
so that we conclude that \eqref{equation:main-2-simple} holds.
Then~\eqref{equation:main-2} in Main Theorem holds by Lemma~\ref{lemma:cokernel}.

\subsection{Family \textnumero $2.23$}
\label{section:r-2-n-23}

The threefold $X$ is a blow up of a smooth quadric threefold in~$\mathbb{P}^4$ along a smooth elliptic curve of degree $4$.
Then $h^{1,2}(X)=1$.
A~toric Landau--Ginzburg model of this family is given by Minkowski polynomial \textnumero $410$, which is
$$
x+y+z+\frac{z}{x}+\frac{z}{y}+\frac{x}{z}+\frac{y}{z}+\frac{1}{x}+\frac{1}{y}.
$$
In this case, the pencil $\mathcal{S}$ is given by the equation
$$
xyz^2+x^2yz+xy^2z+xz^2t+yz^2t+x^2yt+xy^2t+xzt^2+yzt^2=\lambda xyzt.
$$
As usual, we assume that $\lambda\ne\infty$. Then
\begin{equation}
\label{equation:2-23}
\begin{split}
H_{\{x\}}\cdot S_\lambda&=L_{\{x\},\{y\}}+L_{\{x\},\{z\}}+L_{\{x\},\{t\}}+L_{\{x\},\{z,t\}},\\
H_{\{y\}}\cdot S_\lambda&=L_{\{x\},\{y\}}+L_{\{y\},\{z\}}+L_{\{y\},\{t\}}+L_{\{y\},\{z,t\}},\\
H_{\{z\}}\cdot S_\lambda&=L_{\{x\},\{z\}}+L_{\{y\},\{z\}}+L_{\{z\},\{t\}}+L_{\{z\},\{x,y\}},\\
H_{\{t\}}\cdot S_\lambda&=L_{\{x\},\{t\}}+L_{\{y\},\{t\}}+L_{\{z\},\{t\}}+L_{\{t\},\{x,y,z\}}.
\end{split}
\end{equation}
This shows that the base locus of the pencil $\mathcal{S}$ is a union of the curves
$L_{\{x\},\{y\}}$, $L_{\{x\},\{z\}}$, $L_{\{x\},\{t\}}$, $L_{\{y\},\{z\}}$, $L_{\{y\},\{t\}}$, $L_{\{z\},\{t\}}$, $L_{\{x\},\{z,t\}}$,
$L_{\{y\},\{z,t\}}$, $L_{\{z\},\{x,y\}}$, and $L_{\{t\},\{x,y,z\}}$.

Observe that $S_{-1}=H_{\{z,t\}}+\mathbf{S}$,
where $\mathbf{S}$ is an irreducible cubic surface in $\mathbb{P}^3$ that is given by $xzt+yzt+x^2y+xy^2+xyz=0$.
On the other hand, if $\lambda\ne-1$, then $S_\lambda$ has isolated singularities, so that it is irreducible.
Moreover, in this case, the singular points of the surface $S_\lambda$ contained in the base locus of the pencil $\mathcal{S}$
can be described as follows:
\begin{itemize}\setlength{\itemindent}{2cm}
\item[$P_{\{y\},\{z\},\{t\}}$:] type $\mathbb{A}_3$ with quadratic term $y(z+t)$;

\item[$P_{\{x\},\{z\},\{t\}}$:] type $\mathbb{A}_3$ with quadratic term $x(z+t)$;

\item[$P_{\{x\},\{y\},\{t\}}$:] type $\mathbb{A}_1$ with quadratic term $xy+xt+yt$;

\item[$P_{\{x\},\{y\},\{z\}}$:] type $\mathbb{A}_3$ with quadratic term $z(x+y)$ for $\lambda\neq 0$, type $\mathbb{A}_5$ for $\lambda=0$;

\item[$P_{\{z\},\{t\},\{x,y\}}$:] type $\mathbb{A}_1$ with quadratic term $(x+y+z)(z+t)-(\lambda+1)zt$;

\item[$P_{\{x\},\{y\},\{z,t\}}$:] type $\mathbb{A}_1$ with quadratic term $(x+y)(z+t)-(\lambda+1)xy$.
\end{itemize}
Thus, it follows from Lemma~\ref{corollary:irreducible-fibers} that the fiber $\mathsf{f}^{-1}(\lambda)$ is irreducible for every $\lambda\ne -1$.
Moreover, the surface $S_{-1}$ has good double points at $P_{\{y\},\{z\},\{t\}}$, $P_{\{x\},\{z\},\{t\}}$,
$P_{\{x\},\{y\},\{t\}}$, $P_{\{x\},\{y\},\{z\}}$, $P_{\{z\},\{t\},\{x,y\}}$, and $P_{\{x\},\{y\},\{z,t\}}$.
Furthermore, it is smooth at general points of the curves
$L_{\{x\},\{y\}}$, $L_{\{x\},\{z\}}$, $L_{\{x\},\{t\}}$, $L_{\{y\},\{z\}}$, $L_{\{y\},\{t\}}$, $L_{\{z\},\{t\}}$, $L_{\{x\},\{z,t\}}$,
$L_{\{y\},\{z,t\}}$, $L_{\{z\},\{x,y\}}$, and $L_{\{t\},\{x,y,z\}}$.
This gives $[\mathsf{f}^{-1}(-1)]=[S_{-1}]=2$ by \eqref{equation:equation:number-of-irredubicle-components-refined} and
Lemmas~\ref{lemma:main} and \ref{lemma:normal-crossing} and confirms \eqref{equation:main-1} in Main Theorem, since $h^{1,2}(X)=1$.

To verify \eqref{equation:main-2} in Main Theorem, we may assume that $\lambda\ne 0$ and $\lambda\ne-1$.
Then the intersection matrix of the curves $L_{\{x\},\{z\}}$, $L_{\{x\},\{t\}}$, $L_{\{x\},\{z,t\}}$, $L_{\{y\},\{t\}}$, $L_{\{y\},\{z,t\}}$, $L_{\{z\},\{x,y\}}$, and $H_{\lambda}$
on the surface $S_\lambda$ is given by
\begin{center}\renewcommand\arraystretch{1.42}
\begin{tabular}{|c||c|c|c|c|c|c|c|}
\hline
 $\bullet$  & $L_{\{x\},\{z\}}$ & $L_{\{x\},\{t\}}$ & $L_{\{x\},\{z,t\}}$ & $L_{\{y\},\{t\}}$ & $L_{\{y\},\{z,t\}}$ & $L_{\{z\},\{x,y\}}$ & $H_{\lambda}$ \\
\hline\hline
$L_{\{x\},\{z\}}$ & $-\frac{1}{2}$ & $\frac{3}{4}$ & $\frac{1}{2}$ & $0$ & $0$ & $\frac{1}{2}$ & $1$ \\
\hline
$L_{\{x\},\{t\}}$ & $\frac{3}{4}$ & $-\frac{3}{4}$ & $\frac{1}{2}$ &  $\frac{1}{2}$ & $0$ & $0$ & $1$ \\
\hline
$L_{\{x\},\{z,t\}}$ & $\frac{1}{2}$ & $\frac{1}{2}$ & $-\frac{1}{2}$ &  $0$ & $\frac{1}{2}$ & $0$ & $1$ \\
\hline
$L_{\{y\},\{t\}}$ & $0$ & $\frac{1}{2}$ & $0$ &  $-\frac{3}{4}$ & $\frac{1}{2}$ & $0$ & $1$ \\
\hline
$L_{\{y\},\{z,t\}}$ & $0$ & $0$ & $\frac{1}{2}$ &  $\frac{1}{2}$ & $-\frac{1}{2}$ & $0$ & $1$ \\
\hline
$L_{\{z\},\{x,y\}}$ & $\frac{1}{2}$ & $0$ & $0$ &  $0$ & $0$ & $-\frac{1}{2}$ & $1$ \\
\hline
 $H_{\lambda}$  & $1$ & $1$ & $1$ & $1$ & $1$ & $1$ & $4$ \\
\hline
\end{tabular}
\end{center}
The rank of this matrix is $6$.
Thus, using \eqref{equation:2-23}, we see that the intersection matrix of the curves
$L_{\{x\},\{y\}}$, $L_{\{x\},\{z\}}$, $L_{\{x\},\{t\}}$, $L_{\{y\},\{z\}}$, $L_{\{y\},\{t\}}$, $L_{\{z\},\{t\}}$, $L_{\{x\},\{z,t\}}$,
$L_{\{y\},\{z,t\}}$, $L_{\{z\},\{x,y\}}$, and $L_{\{t\},\{x,y,z\}}$ is also $6$.
But $\mathrm{rk}\,\mathrm{Pic}(\widetilde{S}_{\Bbbk})=\mathrm{rk}\,\mathrm{Pic}(S_{\Bbbk})+12$.
This shows that \eqref{equation:main-2-simple} holds, so that \eqref{equation:main-2} in Main Theorem also holds by Lemma~\ref{lemma:cokernel}.

\subsection{Family \textnumero $2.24$}
\label{section:r-2-n-24}

The threefold $X$ is a smooth divisor in $\mathbb{P}^2\times\mathbb{P}^2$ of bidegree $(1,2)$,
which implies that $h^{1,2}(X)=0$.
A~toric Landau--Ginzburg model of this family is given by Minkowski polynomial \textnumero $411$, which is
$$
\frac{xy}{z}+x+y+z+\frac{x}{z}+\frac{y}{x}+\frac{z}{y}+\frac{1}{y}+\frac{1}{x}.
$$
Then the pencil $\mathcal{S}$ is given by the equation
$$
x^2y^2+x^2yz+y^2xz+z^2xy+x^2yt+y^2tz+z^2xt+t^2xz+t^2yz=\lambda xyzt.
$$
Moreover, the base locus of this pencil consists of the lines
$L_{\{x\},\{y\}}$, $L_{\{x\},\{z\}}$, $L_{\{x\},\{t\}}$, $L_{\{y\},\{z\}}$, $L_{\{y\},\{t\}}$, $L_{\{x\},\{y,t\}}$,
$L_{\{y\},\{z,t\}}$, $L_{\{z\},\{y,t\}}$, $L_{\{t\},\{y,z\}}$, and $L_{\{t\},\{x,z\}}$, because
\begin{equation}
\label{equation:2-24}
\begin{split}
H_{\{x\}}\cdot S_\lambda&=L_{\{x\},\{y\}}+L_{\{x\},\{z\}}+L_{\{x\},\{t\}}+L_{\{x\},\{y,t\}},\\
H_{\{y\}}\cdot S_\lambda&=L_{\{x\},\{y\}}+L_{\{y\},\{z\}}+L_{\{y\},\{t\}}+L_{\{y\},\{z,t\}},\\
H_{\{z\}}\cdot S_\lambda&=2L_{\{x\},\{z\}}+L_{\{y\},\{z\}}+L_{\{z\},\{y,t\}},\\
H_{\{t\}}\cdot S_\lambda&=L_{\{x\},\{t\}}+L_{\{y\},\{t\}}+L_{\{t\},\{y,z\}}+L_{\{t\},\{x,z\}}.
\end{split}
\end{equation}
Here, as usual, we assume that $\lambda\ne\infty$.

For every $\lambda\in\mathbb{C}$, the surface $S_\lambda$ has isolated singularities, so that $S_\lambda$ is irreducible.
Its~singular points contained in the base locus of the pencil $\mathcal{S}$ can be described as follows:
\begin{itemize}\setlength{\itemindent}{1.5cm}
\item[$P_{\{x\},\{z\},\{t\}}$:] type $\mathbb{A}_1$;

\item[$P_{\{x\},\{y\},\{t\}}$:] type $\mathbb{A}_3$ with quadratic term $x(y+t)$ for $\lambda\neq -2$, type $\mathbb{A}_4$ for $\lambda=-2$;

\item[$P_{\{x\},\{y\},\{z\}}$:] type $\mathbb{A}_3$ with quadratic term $z(x+y)$ for $\lambda\neq -\frac{3}{2}$, type $\mathbb{A}_4$ for $\lambda=-\frac{3}{2}$;

\item[$P_{\{y\},\{z\},\{t\}}$:] type $\mathbb{A}_3$ with quadratic term $y(y+z+t)$ for $\lambda\neq -2$, type $\mathbb{A}_4$ for $\lambda=-2$;

\item[$P_{\{x\},\{z\},\{y,t\}}$:] type $\mathbb{A}_1$ for $\lambda\neq -\frac{3}{2}$, type $\mathbb{A}_2$ for $\lambda=-\frac{3}{2}$;

\item[{$[1:1:-1:0]$}:] smooth for $\lambda\neq -1$, type $\mathbb{A}_1$ for $\lambda=-1$.
\end{itemize}
Then $[\mathsf{f}^{-1}(\lambda)]=1$ for every $\lambda\in\mathbb{C}$ by Corollary~\ref{corollary:irreducible-fibers}.
This confirms \eqref{equation:main-1} in Main Theorem.

To verify \eqref{equation:main-2} in Main Theorem,
we may assume that $\lambda\not\in\{-1,-\frac{3}{2},-2\}$.
Then the intersection matrix of the curves $L_{\{x\},\{z\}}$, $L_{\{x\},\{t\}}$, $L_{\{x\},\{y,t\}}$, $L_{\{y\},\{t\}}$, $L_{\{y\},\{z,t\}}$, $L_{\{t\},\{x,z\}}$, and $H_{\lambda}$
on the surface $S_\lambda$ is given by
\begin{center}\renewcommand\arraystretch{1.42}
\begin{tabular}{|c||c|c|c|c|c|c|c|}
\hline
 $\bullet$  & $L_{\{x\},\{z\}}$ & $L_{\{x\},\{t\}}$ & $L_{\{x\},\{y,t\}}$ & $L_{\{y\},\{t\}}$ & $L_{\{y\},\{z,t\}}$ & $L_{\{t\},\{x,z\}}$ & $H_{\lambda}$ \\
\hline\hline
$L_{\{x\},\{z\}}$ & $-\frac{1}{4}$ & $\frac{1}{2}$ & $\frac{1}{2}$ & $0$ & $0$ & $\frac{1}{2}$ & $1$ \\
\hline
$L_{\{x\},\{t\}}$ & $\frac{1}{2}$ & $-\frac{3}{4}$ & $\frac{1}{2}$ &  $\frac{1}{4}$ & $0$ & $\frac{1}{2}$ & $1$ \\
\hline
$L_{\{x\},\{y,t\}}$ & $\frac{1}{2}$ & $\frac{1}{2}$ & $-\frac{1}{2}$ &  $\frac{1}{2}$ & $0$ & $0$ & $1$ \\
\hline
$L_{\{y\},\{t\}}$ & $0$ & $\frac{1}{4}$ & $\frac{1}{2}$ &  $-\frac{1}{2}$ & $\frac{1}{2}$ & $\frac{1}{2}$ & $1$ \\
\hline
$L_{\{y\},\{z,t\}}$ & $0$ & $0$ & $0$ &  $\frac{1}{2}$ & $-1$ & $0$ & $1$ \\
\hline
$L_{\{t\},\{x,z\}}$ & $\frac{1}{2}$ & $\frac{1}{2}$ & $0$ &  $\frac{1}{2}$ & $0$ & $-\frac{3}{2}$ & $1$ \\
\hline
 $H_{\lambda}$  & $1$ & $1$ & $1$ & $1$ & $1$ & $1$ & $4$ \\
\hline
\end{tabular}
\end{center}
The rank of this intersection matrix is $7$.
Thus, using \eqref{equation:2-24}, we see that
the rank of the intersection
matrix of the curves $L_{\{x\},\{y\}}$, $L_{\{x\},\{z\}}$, $L_{\{x\},\{t\}}$, $L_{\{y\},\{z\}}$, $L_{\{y\},\{t\}}$, $L_{\{x\},\{y,t\}}$,
$L_{\{y\},\{z,t\}}$, $L_{\{z\},\{y,t\}}$, $L_{\{t\},\{y,z\}}$, and $L_{\{t\},\{x,z\}}$ on the surface $S_{\lambda}$ is also $7$.
On the other hand, we have
$\mathrm{rk}\,\mathrm{Pic}(\widetilde{S}_{\Bbbk})=\mathrm{rk}\,\mathrm{Pic}(S_{\Bbbk})+11$.
Hence, we see that \eqref{equation:main-2-simple} holds.
By Lemma~\ref{lemma:cokernel}, we see that \eqref{equation:main-2} in Main Theorem also holds.

\subsection{Family \textnumero $2.25$}
\label{section:r-2-n-25}

In this case, the threefold $X$ is  a blow up of $\mathbb{P}^3$ along a smooth elliptic curve, which is an intersection of two quadrics.
This shows that $h^{1,2}(X)=1$.
A~toric Landau--Ginzburg model of this family is given by Minkowski polynomial \textnumero $198$, which is
$$
x+y+z+\frac{yz}{x} + \frac{x}{z} + \frac{1}{y} + \frac{1}{x} + \frac{1}{yz}.
$$
Thus, the pencil of quartic surfaces $\mathcal{S}$ is given by the equation
$$
x^2yz+y^2xz+xyz^2+y^2z^2+x^2yt+zxt^2+yzt^2+xt^3=\lambda xyzt.
$$
As usual, we assume that $\lambda\ne\infty$.

Let $\mathcal{C}_1$ be the conic in $\mathbb{P}^3$ that is given by $x=yz+t^2=0$,
and let $\mathcal{C}_2$ be the conic in $\mathbb{P}^3$ that is given by $z=xy+t^2=0$.
Then
\begin{equation}
\label{equation:2-25}
\begin{split}
H_{\{x\}}\cdot S_\lambda&=L_{\{x\},\{y\}}+L_{\{x\},\{z\}}+\mathcal{C}_1,\\
H_{\{y\}}\cdot S_\lambda&=L_{\{x\},\{y\}}+2L_{\{y\},\{t\}}+L_{\{y\},\{z,t\}},\\
H_{\{z\}}\cdot S_\lambda&=L_{\{x\},\{z\}}+L_{\{z\},\{t\}}+\mathcal{C}_2,\\
H_{\{t\}}\cdot S_\lambda&=L_{\{y\},\{t\}}+L_{\{z\},\{t\}}+L_{\{t\},\{x,y\}}+L_{\{t\},\{x,z\}}.
\end{split}
\end{equation}
This shows that the base locus of the pencil $\mathcal{S}$ consists of the curves
$L_{\{x\},\{y\}}$, $L_{\{x\},\{z\}}$, $L_{\{y\},\{t\}}$, $L_{\{z\},\{t\}}$,
$L_{\{y\},\{z,t\}}$, $L_{\{t\},\{x,y\}}$, $L_{\{t\},\{x,z\}}$, $\mathcal{C}_1$, and $\mathcal{C}_2$.

To describe the singularities of the surfaces in the pencil $\mathcal{S}$, observe that
$$
S_{-1}=\mathsf{Q}+\mathbf{Q}
$$
where $\mathsf{Q}$ is an irreducible quadric surface that is given by $yz+xt+xz=0$,
and $\mathbf{Q}$ is an irreducible quadric surface given by $yz+xy+t^2=0$.
Thus, the singularities of the surface $S_{-1}$ are not isolated.
On the other hand, if $\lambda\ne-1$, then the surface $S_\lambda$ has isolated singularities, so that it is irreducible.
Moreover, in this case, the singular points of the surface $S_\lambda$ contained in the base locus of the pencil $\mathcal{S}$ can be described as follows:
\begin{itemize}\setlength{\itemindent}{3cm}
\item[$P_{\{y\},\{z\},\{t\}}$:] type $\mathbb{A}_3$ with quadratic term $y(z+t)$;

\item[$P_{\{x\},\{z\},\{t\}}$:] type $\mathbb{A}_5$ with quadratic term $z(x+z)$;

\item[$P_{\{x\},\{y\},\{t\}}$:] type $\mathbb{A}_4$ with quadratic term $y(x+y)$;

\item[$P_{\{y\},\{t\},\{x,z\}}$:] type $\mathbb{A}_1$ with quadratic term $xy+yz-(\lambda+1)yt+t^2$.
\end{itemize}

By Lemma~\ref{corollary:irreducible-fibers}, we have $[\mathsf{f}^{-1}(\lambda)]=1$ for every $\lambda\ne -1$.
Moreover, the points $P_{\{y\},\{z\},\{t\}}$, $P_{\{x\},\{z\},\{t\}}$,
$P_{\{x\},\{y\},\{t\}}$, and $P_{\{y\},\{t\},\{x,z\}}$ are good double points of the surface~$S_{-1}$.
Furthermore, the surface $S_{-1}$ is smooth at general points of the curves $L_{\{x\},\{y\}}$, $L_{\{x\},\{z\}}$, $L_{\{y\},\{t\}}$, $L_{\{z\},\{t\}}$,
$L_{\{y\},\{z,t\}}$, $L_{\{t\},\{x,y\}}$, $L_{\{t\},\{x,z\}}$, $\mathcal{C}_1$, and $\mathcal{C}_2$.
Thus, it follows from
\eqref{equation:equation:number-of-irredubicle-components-refined},
Lemma~\ref{lemma:main} and Lemma~\ref{lemma:normal-crossing}
that $[\mathsf{f}^{-1}(-1)]=[S_{-1}]=2$.
This confirms \eqref{equation:main-1} in Main Theorem, since $h^{1,2}(X)=1$.

To verify \eqref{equation:main-2} in Main Theorem, we may assume that $\lambda\ne -1$.
Then, using \eqref{equation:2-25}, we see that the intersection
matrix of the curves $L_{\{x\},\{y\}}$, $L_{\{x\},\{z\}}$, $L_{\{y\},\{t\}}$, $L_{\{z\},\{t\}}$,
$L_{\{y\},\{z,t\}}$, $L_{\{t\},\{x,y\}}$, $L_{\{t\},\{x,z\}}$, $\mathcal{C}_1$, $\mathcal{C}_2$ on the surface $S_{\lambda}$
has the same rank as the intersection matrix of the curves
$L_{\{x\},\{y\}}$, $L_{\{x\},\{z\}}$, $L_{\{y\},\{z,t\}}$, $L_{\{t\},\{x,y\}}$, $L_{\{t\},\{x,z\}}$,~$H_{\lambda}$, which is given~by
\begin{center}\renewcommand\arraystretch{1.42}
\begin{tabular}{|c||c|c|c|c|c|c|}
\hline
 $\bullet$  & $L_{\{x\},\{y\}}$ & $L_{\{x\},\{z\}}$ & $L_{\{y\},\{z,t\}}$ & $L_{\{t\},\{x,y\}}$ & $L_{\{t\},\{x,z\}}$ &  $H_{\lambda}$ \\
\hline\hline
$L_{\{x\},\{y\}}$ & $-\frac{4}{5}$ & $1$ & $1$ & $\frac{3}{5}$ & $0$ & $1$ \\
\hline
$L_{\{x\},\{z\}}$ & $1$ &  $-\frac{2}{3}$ & $0$ & $0$ & $\frac{1}{3}$ & $1$ \\
\hline
$L_{\{y\},\{z,t\}}$ & $1$ &  $0$ & $-1$ & $0$ & $0$ & $1$ \\
\hline
$L_{\{t\},\{x,y\}}$ & $\frac{3}{5}$ &  $0$ & $0$ & $-\frac{6}{5}$ & $1$ & $1$ \\
\hline
$L_{\{t\},\{x,z\}}$ & $0$ &  $\frac{1}{3}$ & $0$ & $1$ & $-\frac{2}{3}$ & $1$ \\
\hline
 $H_{\lambda}$  & $1$ & $1$ & $1$ & $1$ & $1$ & $4$ \\
\hline
\end{tabular}
\end{center}
The rank of this matrix is $5$.
On the other hand, using the description of the singular points of the surface $S_\lambda$, we conclude that
$\mathrm{rk}\,\mathrm{Pic}(\widetilde{S}_{\Bbbk})=\mathrm{rk}\,\mathrm{Pic}(S_{\Bbbk})+13$.
Hence, we see that \eqref{equation:main-2-simple} holds.
By Lemma~\ref{lemma:cokernel}, we see that \eqref{equation:main-2} in Main Theorem also holds.

\subsection{Family \textnumero $2.26$}
\label{section:r-2-n-26}

In this case, the threefold $X$ is a blow up of the threefold $V_5$ along a line (see Subsection~\ref{section:r-2-n-14}).
Then $h^{1,2}(X)=0$.
A~toric Landau--Ginzburg model of this family is given by
$$
\frac{y}{x}+\frac{1}{x}+y+z+\frac{1}{z}+\frac{1}{y}+x+\frac{x}{yz},
$$
which is Minkowski polynomial \textnumero $201$.
The quartic pencil $\mathcal{S}$ is given by
$$
y^2zt+t^2yz+y^2xz+z^2xy+t^2xy+t^2xz+x^2yz+x^2t^2=\lambda xyzt.
$$

Suppose that $\lambda\ne\infty$. Then
\begin{equation}
\label{equation:2-26}
\begin{split}
H_{\{x\}}\cdot S_\lambda&=L_{\{x\},\{y\}}+L_{\{x\},\{z\}}+L_{\{x\},\{t\}}+L_{\{x\},\{y,t\}},\\
H_{\{y\}}\cdot S_\lambda&=L_{\{x\},\{y\}}+2L_{\{y\},\{t\}}+L_{\{y\},\{x,z\}},\\
H_{\{z\}}\cdot S_\lambda&=L_{\{x\},\{z\}}+2L_{\{z\},\{t\}}+L_{\{z\},\{x,y\}},\\
H_{\{t\}}\cdot S_\lambda&=L_{\{x\},\{t\}}+L_{\{y\},\{t\}}+L_{\{z\},\{t\}}+L_{\{t\},\{x,y,z\}}.
\end{split}
\end{equation}
Thus, the base locus of the pencil $\mathcal{S}$ consists of the lines
$L_{\{x\},\{y\}}$, $L_{\{x\},\{z\}}$, $L_{\{x\},\{t\}}$, $L_{\{x\},\{y,t\}}$,
$L_{\{y\},\{t\}}$, $L_{\{y\},\{x,z\}}$, $L_{\{z\},\{t\}}$, $L_{\{z\},\{x,y\}}$, and $L_{\{t\},\{x,y,z\}}$.

For every $\lambda\in\mathbb{C}$, the surface $S_\lambda$ is irreducible, it has isolated singularities,
and its singular points contained in the base locus of the pencil~$\mathcal{S}$ can be described as follows:
\begin{itemize}\setlength{\itemindent}{1.3cm}
\item[$P_{\{x\},\{y\},\{z\}}$:] type $\mathbb{A}_2$ with quadratic term $(x+y)(x+z)$ for $\lambda\neq -1$, type $\mathbb{A}_3$ for $\lambda=-1$;

\item[$P_{\{x\},\{y\},\{t\}}$:] type $\mathbb{A}_3$ with quadratic term $xy$;

\item[$P_{\{x\},\{z\},\{t\}}$:] type $\mathbb{A}_2$ with quadratic term $z(x+t)$;

\item[$P_{\{y\},\{z\},\{t\}}$:] type $\mathbb{A}_1$;

\item[$P_{\{y\},\{t\},\{x,z\}}$:] type $\mathbb{A}_2$ with quadratic term $y(x+y+z-\lambda t)$ for $\lambda\neq 0$, type $\mathbb{A}_3$ for $\lambda=0$;

\item[$P_{\{z\},\{t\},\{x,y\}}$:] type $\mathbb{A}_2$ with quadratic term $z(x+y+z-t-\lambda t)$.
\end{itemize}
Then $[\mathsf{f}^{-1}(\lambda)]=1$ for every $\lambda\in\mathbb{C}$ by Lemma~\ref{corollary:irreducible-fibers}.
This confirms \eqref{equation:main-1} in Main Theorem.

By Lemma~\ref{lemma:cokernel},
to verify  \eqref{equation:main-2} in Main Theorem,
we have to prove \eqref{equation:main-2-simple}.
Observe that
the intersection matrix of the lines $L_{\{x\},\{y\}}$, $L_{\{x\},\{z\}}$, $L_{\{x\},\{t\}}$, $L_{\{x\},\{y,t\}}$,
$L_{\{y\},\{t\}}$, $L_{\{y\},\{x,z\}}$, $L_{\{z\},\{t\}}$, $L_{\{z\},\{x,y\}}$, $L_{\{t\},\{x,y,z\}}$
on the surface $S_\lambda$ has the same rank as the intersection matrix of the curves $L_{\{x\},\{t\}}$, $L_{\{z\},\{t\}}$, $L_{\{x\},\{y,t\}}$, $L_{\{y\},\{x,z\}}$, $L_{\{z\},\{x,y\}}$, $L_{\{t\},\{x,y,z\}}$, $H_{\lambda}$,
since
\begin{multline*}
L_{\{x\},\{y\}}+L_{\{x\},\{z\}}+L_{\{x\},\{t\}}+L_{\{x\},\{y,t\}}\sim L_{\{x\},\{y\}}+2L_{\{y\},\{t\}}+L_{\{y\},\{x,z\}}\sim\\
\sim L_{\{x\},\{z\}}+2L_{\{z\},\{t\}}+L_{\{z\},\{x,y\}}\sim L_{\{x\},\{t\}}+L_{\{y\},\{t\}}+L_{\{z\},\{t\}}+L_{\{t\},\{x,y,z\}}\sim H_\lambda,
\end{multline*}
which follow from \eqref{equation:2-26}.
On the other hand, if $\lambda\ne 0$ and $\lambda\ne-1$,
then the intersection matrix of the curves $L_{\{x\},\{t\}}$, $L_{\{z\},\{t\}}$, $L_{\{x\},\{y,t\}}$, $L_{\{y\},\{x,z\}}$, $L_{\{z\},\{x,y\}}$, $L_{\{t\},\{x,y,z\}}$, and $H_{\lambda}$
on the surface $S_\lambda$ is given by the following table:
\begin{center}\renewcommand\arraystretch{1.42}
\begin{tabular}{|c||c|c|c|c|c|c|c|}
\hline
 $\bullet$  & $L_{\{x\},\{t\}}$ & $L_{\{z\},\{t\}}$ & $L_{\{x\},\{y,t\}}$ & $L_{\{y\},\{x,z\}}$ & $L_{\{z\},\{x,y\}}$ & $L_{\{t\},\{x,y,z\}}$ & $H_{\lambda}$ \\
\hline\hline
$L_{\{x\},\{t\}}$ & $-\frac{7}{12}$ & $\frac{1}{3}$ & $\frac{3}{4}$ & $0$ & $0$ & $1$ & $1$ \\
\hline
$L_{\{z\},\{t\}}$ & $\frac{1}{3}$ & $-\frac{1}{6}$ & $0$ &  $0$ & $\frac{2}{3}$ & $\frac{1}{3}$ & $1$ \\
\hline
$L_{\{x\},\{y,t\}}$ & $\frac{3}{4}$ & $0$ & $-\frac{5}{4}$ &  $0$ & $0$ & $0$ & $1$ \\
\hline
$L_{\{y\},\{x,z\}}$ & $0$ & $0$ & $0$ &  $-\frac{2}{3}$ & $\frac{1}{3}$ & $\frac{1}{3}$ & $1$ \\
\hline
$L_{\{z\},\{x,y\}}$ & $0$ & $\frac{2}{3}$ & $0$ &  $\frac{1}{3}$ & $-\frac{2}{3}$ & $\frac{1}{3}$ & $1$ \\
\hline
$L_{\{t\},\{x,y,z\}}$ & $1$ & $\frac{1}{3}$ & $0$ &  $\frac{1}{3}$ & $\frac{1}{3}$ & $-\frac{2}{3}$ & $1$ \\
\hline
 $H_{\lambda}$  & $1$ & $1$ & $1$ & $1$ & $1$ & $1$ & $4$ \\
\hline
\end{tabular}
\end{center}
The rank of this matrix is $6$.
But $\mathrm{rk}\,\mathrm{Pic}(\widetilde{S}_{\Bbbk})=\mathrm{rk}\,\mathrm{Pic}(S_{\Bbbk})+12$.
We conclude that \eqref{equation:main-2-simple} holds, so that \eqref{equation:main-2} in Main Theorem holds by Lemma~\ref{lemma:cokernel}.

\subsection{Family \textnumero $2.27$}
\label{section:r-2-n-27}

In this case, the threefold $X$ is  a blow up of $\mathbb{P}^3$ in a twisted cubic, so that $h^{1,2}(X)=0$.
A~toric Landau--Ginzburg model of this family is given by Minkowski polynomial \textnumero $70$, which is
$$
x+y+z+{\frac {x}{z}}+\frac{1}{x}+{\frac {1}{yz}}+{\frac {1}{xy}}.
$$
The quartic pencil $\mathcal{S}$ is given by the equation:
$$
x^2zy+y^2zx+z^2xy+x^2ty+t^2zy+t^3x+t^3z=\lambda xyzt.
$$

Suppose that $\lambda\ne\infty$. Let $\mathcal{C}$ be the conic in $\mathbb{P}^3$ that is given by $z=xy+t^2=0$. Then
\begin{equation}
\label{equation:2-27}
\begin{split}
H_{\{x\}}\cdot S_\lambda&=L_{\{x\},\{z\}}+2L_{\{x\},\{t\}}+L_{\{x\},\{y,t\}},\\
H_{\{y\}}\cdot S_\lambda&=3L_{\{y\},\{t\}}+L_{\{y\},\{x,z\}},\\
H_{\{z\}}\cdot S_\lambda&=L_{\{x\},\{z\}}+L_{\{z\},\{t\}}+\mathcal{C},\\
H_{\{t\}}\cdot S_\lambda&=L_{\{x\},\{t\}}+L_{\{y\},\{t\}}+L_{\{z\},\{t\}}+L_{\{t\},\{x,y,z\}}.
\end{split}
\end{equation}
Thus, the base locus of the pencil $\mathcal{S}$ consists of the curves
$L_{\{x\},\{z\}}$, $L_{\{x\},\{t\}}$, $L_{\{y\},\{t\}}$, $L_{\{z\},\{t\}}$,
$L_{\{x\},\{y,t\}}$, $L_{\{y\},\{x,z\}}$, $L_{\{t\},\{x,y,z\}}$, and $\mathcal{C}$.

For every $\lambda\in\mathbb{C}$, the surface $S_\lambda$ is irreducible, it has isolated singularities,
and its singular points contained in the base locus of the pencil $\mathcal{S}$ can be described as follows:
\begin{itemize}\setlength{\itemindent}{3cm}
\item[$P_{\{x\},\{y\},\{t\}}$:] type $\mathbb{A}_2$ with quadratic term $xy$;

\item[$P_{\{x\},\{z\},\{t\}}$:] type $\mathbb{A}_5$ with quadratic term $xz$ for $\lambda\neq -1$, type $\mathbb{A}_6$ for $\lambda=-1$;

\item[$P_{\{y\},\{z\},\{t\}}$:] type $\mathbb{A}_2$ with quadratic term $y(z+t)$;

\item[$P_{\{x\},\{z\},\{y,t\}}$:] type $\mathbb{A}_1$;

\item[$P_{\{y\},\{t\},\{x,z\}}$:] type $\mathbb{A}_3$ with quadratic term
$$
y(x+y+z-t-\lambda t)
$$
for $\lambda\neq -1$, type $\mathbb{A}_4$ for $\lambda=-1$.
\end{itemize}
By Lemma~\ref{corollary:irreducible-fibers}, each fiber $\mathsf{f}^{-1}(\lambda)$ is irreducible.
This confirms \eqref{equation:main-1} in Main Theorem.

\begin{lemma}
\label{lemma:r2-n27-intersection}
Suppose that $\lambda\ne-1$.
Then the intersection matrix of the curves $L_{\{x\},\{t\}}$, $L_{\{x\},\{y,t\}}$, $L_{\{y\},\{x,z\}}$, $L_{\{z\},\{t\}}$, and $H_{\lambda}$
on the surface $S_\lambda$ is given by
\begin{center}\renewcommand\arraystretch{1.42}
\begin{tabular}{|c||c|c|c|c|c|}
\hline
 $\bullet$  & $L_{\{x\},\{t\}}$ & $L_{\{x\},\{y,t\}}$ & $L_{\{y\},\{x,z\}}$ & $L_{\{z\},\{t\}}$ &  $H_{\lambda}$ \\
\hline\hline
$L_{\{x\},\{t\}}$ & $-\frac{1}{2}$ & $\frac{2}{3}$ & $0$ & $\frac{1}{6}$ & $1$ \\
\hline
$L_{\{x\},\{y,t\}}$ & $\frac{2}{3}$ &  $-\frac{3}{2}$ & $0$ & $0$ & $1$ \\
\hline
$L_{\{y\},\{x,z\}}$ & $0$ &  $0$ & $-\frac{5}{4}$ & $0$ & $1$ \\
\hline
$L_{\{z\},\{t\}}$ & $\frac{1}{6}$ &  $0$ & $0$ & $-\frac{1}{2}$ & $1$ \\
\hline
 $H_{\lambda}$  & $1$ & $1$ & $1$ & $1$ & $4$ \\
\hline
\end{tabular}
\end{center}
\end{lemma}

\begin{proof}
The entries of the last raw and the last column in the intersection matrix are obvious.
To find its diagonal entries, we use Proposition~\ref{proposition:du-Val-self-intersection}.
For instance, to compute $L_{\{x\},\{t\}}^2$, observe that the only singular points of the surface $S_\lambda$ contained
in the line $L_{\{x\},\{t\}}$ are the points $P_{\{x\},\{z\},\{t\}}$ and $P_{\{x\},\{y\},\{t\}}$.
Using Remark~\ref{remark:transversal} with $S=S_\lambda$, $n=5$, $O=P_{\{x\},\{z\},\{t\}}$, and $C=L_{\{x\},\{t\}}$,
we see that $\overline{C}$ does not contain the point $\overline{G}_1\cap\overline{G}_5$.
Thus, it follows  from Proposition~\ref{proposition:du-Val-self-intersection} that $L_{\{x\},\{y\}}^2=-\frac{1}{2}$.
Similarly, we see that $L_{\{x\},\{y,t\}}^2=-\frac{3}{2}$, $L_{\{y\},\{x,z\}}^2=-\frac{5}{4}$, and $L_{\{z\},\{t\}}^2=-\frac{1}{2}$.

Note that $L_{\{x\},\{y,t\}}\cap L_{\{y\},\{x,z\}}=L_{\{x\},\{y,t\}}\cap L_{\{z\},\{t\}}=L_{\{y\},\{x,z\}}\cap L_{\{z\},\{t\}}=\varnothing$,
so that
$$
L_{\{x\},\{y,t\}}\cdot L_{\{y\},\{x,z\}}=L_{\{x\},\{y,t\}}\cdot L_{\{z\},\{t\}}=L_{\{y\},\{x,z\}}\cdot L_{\{z\},\{t\}}=0.
$$
Similarly, we see that $L_{\{x\},\{t\}}\cdot L_{\{y\},\{x,z\}}=0$.

To find the remaining entries of the intersection matrix, we use Proposition~\ref{proposition:du-Val-intersection}.
To start with, let us compute $L_{\{x\},\{t\}}\cdot L_{\{z\},\{t\}}$.
Observe that $L_{\{x\},\{t\}}\cap L_{\{z\},\{t\}}=P_{\{x\},\{z\},\{t\}}$.
Using Remark~\ref{remark:transversal} with $S=S_\lambda$, $n=5$, $O=P_{\{x\},\{z\},\{t\}}$, $C=L_{\{x\},\{t\}}$, and $Z=L_{\{z\},\{t\}}$,
we see that both curves $\overline{C}$ and $\overline{Z}$ do not contain the point $\overline{G}_1\cap\overline{G}_5$.
Moreover, since the quadratic term of the surface $S_\lambda$ at the singular point $P_{\{x\},\{z\},\{t\}}$ is $xz$,
we see that either $\overline{C}\cdot\overline{G}_1=\overline{Z}\cdot\overline{G}_5=1$ or
$\overline{C}\cdot\overline{G}_5=\overline{Z}\cdot\overline{G}_1=1$.
Thus, using Proposition~\ref{proposition:du-Val-intersection}, we conclude that $L_{\{x\},\{t\}}\cdot L_{\{z\},\{t\}}=\frac{1}{6}$.

Finally, let us compute $L_{\{x\},\{t\}}\cdot L_{\{x\},\{y,t\}}$.
Observe that $L_{\{x\},\{t\}}\cap L_{\{x\},\{y,t\}}=P_{\{x\},\{y,\},\{t\}}$,
and $P_{\{x\},\{y,\},\{t\}}$ is a singular point of the surface $S_\lambda$ of type $\mathbb{A}_2$.
Let us use the notation of Appendix~\ref{subsection:A} with $S=S_\lambda$, $n=2$, $O=P_{\{x\},\{y\},\{t\}}$, $C=L_{\{x\},\{t\}}$, and $Z=L_{\{x\},\{y,t\}}$.
Then~$\pi$ is the blow up of the point $O$,
and either both curves $\widetilde{C}$ and $\widetilde{Z}$ intersect $G_1$,
or they intersect the curve $G_2$.
Thus, we have $L_{\{x\},\{t\}}\cdot L_{\{x\},\{y,t\}}=\frac{2}{3}$ Proposition~\ref{proposition:du-Val-intersection}.
\end{proof}

Using \eqref{equation:2-27}, we see that the
intersection matrix of the curves $L_{\{x\},\{z\}}$, $L_{\{x\},\{t\}}$, $L_{\{y\},\{t\}}$, $L_{\{z\},\{t\}}$,
$L_{\{x\},\{y,t\}}$, $L_{\{y\},\{x,z\}}$, $L_{\{t\},\{x,y,z\}}$, and $\mathcal{C}$ on the surface
$S_\lambda$ has the same rank as the intersection matrix of the
curves $L_{\{x\},\{t\}}$, $L_{\{x\},\{y,t\}}$, $L_{\{y\},\{x,z\}}$, $L_{\{z\},\{t\}}$, and $H_{\lambda}$.
But the matrix in Lemma~\ref{lemma:r2-n27-intersection} has rank~$5$.
Thus, since
$\mathrm{rk}\,\mathrm{Pic}(\widetilde{S}_{\Bbbk})=\mathrm{rk}\,\mathrm{Pic}(S_{\Bbbk})+12$,
we conclude that \eqref{equation:main-2-simple} holds.
Hence, it follows from Lemma~\ref{lemma:cokernel} that \eqref{equation:main-2} in Main Theorem holds.

\subsection{Family \textnumero $2.28$}
\label{section:r-2-n-28}
In this case, the threefold $X$ is  a blow up of $\mathbb{P}^3$ in a smooth plane cubic curve,
which implies that $h^{1,2}(X)=1$.
A~toric Landau--Ginzburg model of this family is given by Minkowski polynomial \textnumero $68$, which is
$$
x+\frac{x}{z}+\frac{x}{yz}+\frac{y}{z}+z+\frac{1}{y}+\frac{y}{x}.
$$
The quartic pencil $\mathcal{S}$ is given by
$$
x^2yz+x^2yt+x^2t^2+xy^2t+xyz^2+xzt^2+y^2zt=\lambda xyzt.
$$

Suppose that $\lambda\ne\infty$.
Let $\mathcal{C}$ be the conic that is given by $z=xy+xt+y^2=0$.
Then
\begin{equation}
\label{equation:2-28}
\begin{split}
H_{\{x\}}\cdot S_\lambda&=2L_{\{x\},\{y\}}+L_{\{x\},\{z\}}+L_{\{x\},\{t\}},\\
H_{\{y\}}\cdot S_\lambda&=L_{\{x\},\{y\}}+2L_{\{y\},\{t\}}+L_{\{y\},\{x,z\}},\\
H_{\{z\}}\cdot S_\lambda&=L_{\{x\},\{z\}}+L_{\{z\},\{t\}}+\mathcal{C},\\
H_{\{t\}}\cdot S_\lambda&=L_{\{x\},\{t\}}+L_{\{y\},\{t\}}+L_{\{z\},\{t\}}+L_{\{t\},\{x,z\}}.
\end{split}
\end{equation}
Thus, the base locus of the pencil $\mathcal{S}$ consists of the curves
$L_{\{x\},\{y\}}$, $L_{\{x\},\{z\}}$, $L_{\{x\},\{t\}}$, $L_{\{y\},\{t\}}$,
$L_{\{z\},\{t\}}$, $L_{\{y\},\{x,z\}}$, $L_{\{t\},\{x,z\}}$, and $\mathcal{C}$.

If $\lambda\ne-1$, then $S_\lambda$ is irreducible, it has isolated singularities,
and its singular points contained in the base locus of the pencil $\mathcal{S}$ can be described as follows:
\begin{itemize}\setlength{\itemindent}{3cm}
\item[$P_{\{x\},\{y\},\{z\}}$:] type $\mathbb{A}_4$ with quadratic term $x(x+z)$;

\item[$P_{\{x\},\{y\},\{t\}}$:] type $\mathbb{A}_4$ with quadratic term $xy$;

\item[$P_{\{x\},\{z\},\{t\}}$:] type $\mathbb{A}_3$ with quadratic term $t(x+z)$;

\item[$P_{\{y\},\{z\},\{t\}}$:] type $\mathbb{A}_1$ with quadratic term $yz+yt+t^2$;

\item[$P_{\{y\},\{t\},\{x,z\}}$:] type $\mathbb{A}_2$ with quadratic term $y(x+z-t-\lambda t)$.
\end{itemize}
Thus, if $\lambda\ne -1$, then the fiber $\mathsf{f}^{-1}(\lambda)$ is irreducible by Lemma~\ref{corollary:irreducible-fibers}.
On the other hand, we have $S_{-1}=H_{\{x,z\}}+\mathbf{S}$,
where $\mathbf{S}$ is an irreducible cubic surface in $\mathbb{P}^3$ that is given by $xyz+xyt+xt^2+y^2t=0$.
Nevertheless, the points
$P_{\{y\},\{z\},\{t\}}$, $P_{\{x\},\{z\},\{t\}}$, $P_{\{x\},\{y\},\{t\}}$, $P_{\{x\},\{y\},\{z\}}$, and $P_{\{y\},\{t\},\{x,z\}}$
are good double points of the surface $S_{-1}$.
Moreover, the surface $S_{-1}$ is smooth
at general points of the curves $L_{\{x\},\{y\}}$, $L_{\{x\},\{z\}}$, $L_{\{x\},\{t\}}$, $L_{\{y\},\{t\}}$,
$L_{\{z\},\{t\}}$, $L_{\{y\},\{x,z\}}$, $L_{\{t\},\{x,z\}}$, and $\mathcal{C}$.
Thus, using \eqref{equation:equation:number-of-irredubicle-components-refined} and
Lemmas~\ref{lemma:main} and \ref{lemma:normal-crossing}, we conclude that $[\mathsf{f}^{-1}(-1)]=[S_{-1}]=2$.
This confirms \eqref{equation:main-1} in Main Theorem.

Now let us verify \eqref{equation:main-2} in Main Theorem.
By Lemma~\ref{lemma:cokernel}, it is enough to show that the equality \eqref{equation:main-2-simple} holds.
If $\lambda\ne -1$, then it follows from \eqref{equation:2-28} that
the intersection matrix of the curves
$L_{\{x\},\{y\}}$, $L_{\{x\},\{z\}}$, $L_{\{x\},\{t\}}$, $L_{\{y\},\{t\}}$,
$L_{\{z\},\{t\}}$, $L_{\{y\},\{x,z\}}$, $L_{\{t\},\{x,z\}}$, and $\mathcal{C}$ on the surface $S_\lambda$
has the same rank as the intersection matrix of the curves $L_{\{x\},\{y\}}$, $L_{\{x\},\{t\}}$, $L_{\{y\},\{t\}}$, $L_{\{z\},\{t\}}$, and $H_{\lambda}$.
If $\lambda\ne-1$, the latter matrix is given by
\begin{center}\renewcommand\arraystretch{1.42}
\begin{tabular}{|c||c|c|c|c|c|}
\hline
 $\bullet$  & $L_{\{x\},\{y\}}$ & $L_{\{x\},\{t\}}$ & $L_{\{y\},\{t\}}$ & $L_{\{z\},\{t\}}$ &  $H_{\lambda}$ \\
\hline\hline
$L_{\{x\},\{y\}}$ & $0$ & $\frac{3}{5}$ & $\frac{2}{5}$ & $0$ & $1$ \\
\hline
$L_{\{x\},\{t\}}$ & $\frac{3}{5}$ &  $-\frac{9}{20}$ & $\frac{1}{5}$ & $\frac{3}{4}$ & $1$ \\
\hline
$L_{\{y\},\{t\}}$ & $\frac{2}{5}$ &  $\frac{1}{5}$ & $-\frac{1}{30}$ & $\frac{1}{2}$ & $1$ \\
\hline
$L_{\{z\},\{t\}}$ & $0$ &  $\frac{3}{4}$ & $\frac{1}{2}$ & $-\frac{3}{4}$ & $1$ \\
\hline
 $H_{\lambda}$  & $1$ & $1$ & $1$ & $1$ & $4$ \\
\hline
\end{tabular}
\end{center}
Its rank is $4$.
On the other hand, it follows from the description of the singular points of the surface $S_\lambda$ that
$\mathrm{rk}\,\mathrm{Pic}(\widetilde{S}_{\Bbbk})=\mathrm{rk}\,\mathrm{Pic}(S_{\Bbbk})+14$.
Thus, we can conclude that \eqref{equation:main-2-simple} holds,
so that \eqref{equation:main-2} in Main Theorem also holds.

\subsection{Family \textnumero $2.29$}
\label{section:r-2-n-29}
In this case, the threefold $X$ is  a blow up of a smooth quadric threefold in $\mathbb{P}^4$ along a conic.
This implies that $h^{1,2}(X)=0$.
A~toric Landau--Ginzburg model of this family is given by Minkowski polynomial \textnumero $71$, which is
$$
x+y+z+\frac{x}{z}+\frac{y}{z}+\frac{1}{y}+\frac{1}{x}.
$$
The quartic pencil $\mathcal{S}$ is given by
$$
x^2zy+y^2zx+z^2yx+x^2ty+y^2tx+t^2zx+t^2zy=\lambda xyzt.
$$
Its base locus consists of the lines
$L_{\{x\},\{y\}}$, $L_{\{x\},\{z\}}$, $L_{\{x\},\{t\}}$,
$L_{\{y\},\{z\}}$, $L_{\{y\},\{t\}}$, $L_{\{z\},\{t\}}$, $L_{\{z\},\{x,y\}}$, and $L_{\{t\},\{x,y,z\}}$, because
\begin{itemize}
\item $H_{\{x\}}\cdot S_\lambda=L_{\{x\},\{y\}}+L_{\{x\},\{z\}}+2L_{\{x\},\{t\}}$,
\item $H_{\{y\}}\cdot S_\lambda=L_{\{x\},\{y\}}+L_{\{y\},\{z\}}+2L_{\{y\},\{t\}}$,
\item $H_{\{z\}}\cdot S_\lambda=L_{\{x\},\{z\}}+L_{\{y\},\{z\}}+L_{\{z\},\{t\}}+L_{\{z\},\{x,y\}}$,
\item $H_{\{t\}}\cdot S_\lambda=L_{\{x\},\{t\}}+L_{\{y\},\{t\}}+L_{\{z\},\{t\}}+L_{\{t\},\{x,y,z\}}$.
\end{itemize}
Here, as usual, we assume that $\lambda\ne\infty$.

For every $\lambda\in\mathbb{C}$, the surface $S_\lambda$ has isolated singularities, so that it is irreducible.
Its singular points contained in the base locus of the pencil $\mathcal{S}$ can be described as follows:
\begin{itemize}\setlength{\itemindent}{2cm}
\item[$P_{\{x\},\{y\},\{z\}}$:] type $\mathbb{A}_3$ with quadratic term $z(x+y)$ for $\lambda\neq 0$, type $\mathbb{A}_5$ for $\lambda=0$;

\item[$P_{\{x\},\{y\},\{t\}}$:] type $\mathbb{A}_3$ with quadratic term $xy$;

\item[$P_{\{x\},\{z\},\{t\}}$:] type $\mathbb{A}_2$ with quadratic term $x(z+t)$;

\item[$P_{\{y\},\{z\},\{t\}}$:] type $\mathbb{A}_2$ with quadratic term $y(z+t)$;

\item[$P_{\{x\},\{t\},\{y,z\}}$:] type $\mathbb{A}_1$;

\item[$P_{\{y\},\{t\},\{x,z\}}$:] type $\mathbb{A}_1$;

\item[$P_{\{z\},\{t\},\{x,y\}}$:] type $\mathbb{A}_1$ for $\lambda\neq -1$, type $\mathbb{A}_3$ for $\lambda=-1$.
\end{itemize}
By Lemma~\ref{corollary:irreducible-fibers}, every fiber $\mathsf{f}^{-1}(\lambda)$ is irreducible.
This confirms \eqref{equation:main-1} in Main Theorem.

If $\lambda\ne 0$ and $\lambda\ne -1$, then the intersection matrix of the lines
$L_{\{x\},\{y\}}$, $L_{\{x\},\{z\}}$, $L_{\{x\},\{t\}}$,
$L_{\{y\},\{z\}}$, $L_{\{y\},\{t\}}$, $L_{\{z\},\{t\}}$, $L_{\{z\},\{x,y\}}$, and $L_{\{t\},\{x,y,z\}}$
on the surface $S_\lambda$ has the same rank as the following intersection matrix:
\begin{center}\renewcommand\arraystretch{1.42}
\begin{tabular}{|c||c|c|c|c|c|}
\hline
 $\bullet$  & $L_{\{x\},\{y\}}$ & $L_{\{x\},\{z\}}$ & $L_{\{y\},\{z\}}$ & $L_{\{z\},\{t\}}$ &  $H_{\lambda}$ \\
\hline\hline
$L_{\{x\},\{y\}}$ & $-\frac{1}{4}$ & $\frac{1}{4}$ & $\frac{1}{4}$ & $0$ & $1$ \\
\hline
$L_{\{x\},\{z\}}$ & $\frac{1}{4}$ &  $-\frac{7}{12}$ & $\frac{3}{4}$ & $\frac{1}{3}$ & $1$ \\
\hline
$L_{\{y\},\{z\}}$ & $\frac{1}{4}$ &  $\frac{3}{4}$ & $-\frac{7}{12}$ & $\frac{1}{3}$ & $1$ \\
\hline
$L_{\{z\},\{t\}}$ & $0$ &  $\frac{1}{3}$ & $\frac{1}{3}$ & $-\frac{1}{6}$ & $1$ \\
\hline
 $H_{\lambda}$  & $1$ & $1$ & $1$ & $1$ & $4$ \\
\hline
\end{tabular}
\end{center}
This matrix has rank $5$.
On the other hand, we have $\mathrm{rk}\,\mathrm{Pic}(\widetilde{S}_{\Bbbk})=\mathrm{rk}\,\mathrm{Pic}(S_{\Bbbk})+13$.
Thus, we conclude that \eqref{equation:main-2-simple} holds, so that \eqref{equation:main-2} in Main Theorem also holds.

\subsection{Family \textnumero $2.30$}
\label{section:r-2-n-30}

In this case, the threefold $X$ is  a blow up of $\mathbb{P}^3$ along a conic, so that $h^{1,2}(X)=0$.
A~toric Landau--Ginzburg model of this family is given by Minkowski polynomial \textnumero $22$, which is
$$
\frac{x}{yz}+x+\frac{1}{y}+\frac{1}{z}+\frac{z}{x}+\frac{y}{x}.
$$
The pencil $\mathcal{S}$ is given by the equation
$$
x^2t^2+x^2yz+t^2zx+t^2yx+z^2yt+y^2zt=\lambda xyzt.
$$

Suppose, for simplicity, that $\lambda\ne\infty$. Then
\begin{equation}
\label{equation:2-30}
\begin{split}
H_{\{x\}}\cdot S_\lambda&=L_{\{x\},\{y\}}+L_{\{x\},\{z\}}+L_{\{x\},\{t\}}+L_{\{x\},\{y,z\}},\\
H_{\{y\}}\cdot S_\lambda&=L_{\{x\},\{y\}}+2L_{\{y\},\{t\}}+L_{\{y\},\{x,z\}},\\
H_{\{z\}}\cdot S_\lambda&=L_{\{x\},\{z\}}+2L_{\{z\},\{t\}}+L_{\{z\},\{x,y\}},\\
H_{\{t\}}\cdot S_\lambda&=2L_{\{x\},\{t\}}+L_{\{y\},\{t\}}+L_{\{z\},\{t\}}.
\end{split}
\end{equation}
Thus, the base locus of the pencil $\mathcal{S}$ is a union of the lines
$L_{\{x\},\{y\}}$, $L_{\{x\},\{z\}}$, $L_{\{x\},\{t\}}$, $L_{\{y\},\{t\}}$, $L_{\{z\},\{t\}}$, $L_{\{x\},\{y,z\}}$,
$L_{\{y\},\{x,z\}}$, and $L_{\{z\},\{x,y\}}$,

For every $\lambda\in\mathbb{C}$, the surface $S_\lambda$ has isolated singularities, so that it is irreducible.
Its singular points contained in the base locus of the pencil~$\mathcal{S}$ can be described as follows:
\begin{itemize}\setlength{\itemindent}{1.5cm}
\item[$P_{\{y\},\{z\},\{t\}}$:] type $\mathbb{A}_1$;

\item[$P_{\{x\},\{z\},\{t\}}$:] type $\mathbb{A}_4$ with quadratic term $zt$;

\item[$P_{\{x\},\{y\},\{t\}}$:] type $\mathbb{A}_4$ with quadratic term $yt$;

\item[$P_{\{x\},\{y\},\{z\}}$:] type $\mathbb{A}_3$ with quadratic term $x(x+y+z)$ for $\lambda\neq -1$, type $\mathbb{A}_5$ for $\lambda=-1$;

\item[$P_{\{x\},\{t\},\{y,z\}}$:] type $\mathbb{A}_1$.
\end{itemize}
Then $[\mathsf{f}^{-1}(\lambda)]=1$ for every $\lambda\in\mathbb{C}$ by Lemma~\ref{corollary:irreducible-fibers}.
This confirms \eqref{equation:main-1} in Main Theorem.

Let us verify \eqref{equation:main-2} in Main Theorem.
If $\lambda\ne-1$, then the intersection matrix of the curves $L_{\{x\},\{y\}}$, $L_{\{x\},\{z\}}$, $L_{\{y\},\{x,z\}}$, $L_{\{z\},\{x,y\}}$, and $H_{\lambda}$
on the surface $S_\lambda$ is given by
\begin{center}\renewcommand\arraystretch{1.42}
\begin{tabular}{|c||c|c|c|c|c|}
\hline
 $\bullet$  & $L_{\{x\},\{y\}}$ & $L_{\{x\},\{z\}}$ & $L_{\{y\},\{x,z\}}$ & $L_{\{z\},\{x,y\}}$ &  $H_{\lambda}$ \\
\hline\hline
$L_{\{x\},\{y\}}$ & $-\frac{9}{20}$ & $\frac{3}{4}$ & $\frac{1}{4}$ & $\frac{1}{4}$ & $1$ \\
\hline
$L_{\{x\},\{z\}}$ & $\frac{3}{4}$ &  $-\frac{9}{20}$ & $\frac{1}{4}$ & $\frac{1}{4}$ & $1$ \\
\hline
$L_{\{y\},\{x,z\}}$ & $\frac{1}{4}$ &  $\frac{1}{4}$ & $-2$ & $\frac{1}{4}$ & $1$ \\
\hline
$L_{\{z\},\{x,y\}}$ & $\frac{1}{4}$ &  $\frac{1}{4}$ & $\frac{1}{4}$ & $-2$ & $1$ \\
\hline
 $H_{\lambda}$  & $1$ & $1$ & $1$ & $1$ & $4$ \\
\hline
\end{tabular}
\end{center}
 Observe that this matrix has rank~$5$.
Thus, if $\lambda\ne 1$, then the intersection matrix of the lines
$L_{\{x\},\{y\}}$, $L_{\{x\},\{z\}}$, $L_{\{x\},\{t\}}$, $L_{\{y\},\{t\}}$, $L_{\{z\},\{t\}}$, $L_{\{x\},\{y,z\}}$,
$L_{\{y\},\{x,z\}}$, and $L_{\{z\},\{x,y\}}$ on the surface $S_\lambda$ also has rank $5$, because
\begin{multline*}
L_{\{x\},\{y\}}+L_{\{x\},\{z\}}+L_{\{x\},\{t\}}+L_{\{x\},\{y,z\}}\sim L_{\{x\},\{y\}}+2L_{\{y\},\{t\}}+L_{\{y\},\{x,z\}}\sim \\
\sim L_{\{x\},\{z\}}+2L_{\{z\},\{t\}}+L_{\{z\},\{x,y\}}\sim 2L_{\{x\},\{t\}}+L_{\{y\},\{t\}}+L_{\{z\},\{t\}}\sim H_\lambda
\end{multline*}
on the surface $S_\lambda$ by \eqref{equation:2-30}.
On the other hand, we have
$\mathrm{rk}\,\mathrm{Pic}(\widetilde{S}_{\Bbbk})=\mathrm{rk}\,\mathrm{Pic}(S_{\Bbbk})+13$.
Hence, we see that \eqref{equation:main-2-simple} holds, so that \eqref{equation:main-2} in Main Theorem also holds by Lemma~\ref{lemma:cokernel}.

\subsection{Family \textnumero $2.31$}
\label{section:r-2-n-31}
In this case, the threefold $X$ is  a blow up of the smooth quadric threefold in $\mathbb{P}^4$ along a line.
This shows that $h^{1,2}(X)=0$.
A~toric Landau--Ginzburg model of this family is given by Minkowski polynomial \textnumero $20$, which is
$$
x+y+z+\frac {x}{y}+\frac{1}{x}+\frac {1}{yz}.
$$
The pencil $\mathcal{S}$ is given by the equation
$$
x^2yz+y^2xz+z^2yx+x^2tz+t^2yz+t^3x=\lambda xyzt.
$$

We suppose that $\lambda\ne\infty$.
Let $\mathcal{C}$ be the conic  $\{y=xz+t^2=0\}$.
Then
\begin{equation}
\label{equation:2-31}
\begin{split}
H_{\{x\}}\cdot S_\lambda&=L_{\{x\},\{y\}}+L_{\{x\},\{z\}}+2L_{\{x\},\{t\}},\\
H_{\{y\}}\cdot S_\lambda&=L_{\{x\},\{y\}}+L_{\{y\},\{t\}}+\mathcal{C}, \\
H_{\{z\}}\cdot S_\lambda&=L_{\{x\},\{z\}}+3L_{\{z\},\{t\}},\\
H_{\{t\}}\cdot S_\lambda&=L_{\{x\},\{t\}}+L_{\{y\},\{t\}}+L_{\{z\},\{t\}}+L_{\{t\},\{x,y,z\}}.
\end{split}
\end{equation}
Therefore, the base locus of the pencil $\mathcal{S}$ consists of the lines
$L_{\{x\},\{y\}}$, $L_{\{x\},\{z\}}$, $L_{\{x\},\{t\}}$,
$L_{\{y\},\{t\}}$ $L_{\{z\},\{t\}}$, $L_{\{t\},\{x,y,z\}}$, and the conic $\mathcal{C}$.

For every $\lambda\in\mathbb{C}$, the surface $S_\lambda$ has isolated singularities, so that it is irreducible.
Moreover, the singular points of the surface $S_\lambda$ contained in the base locus of the pencil~$\mathcal{S}$ can be described as follows:
\begin{itemize}\setlength{\itemindent}{2cm}
\item[$P_{\{x\},\{y\},\{t\}}$:] type $\mathbb{A}_5$ with quadratic term $xy$ for $\lambda\neq 0$, type $\mathbb{A}_6$ for $\lambda=0$;

\item[$P_{\{x\},\{z\},\{t\}}$:] type $\mathbb{A}_4$ with quadratic term $xz$;

\item[$P_{\{y\},\{z\},\{t\}}$:] type $\mathbb{A}_2$ with quadratic term $z(y+t)$;

\item[$P_{\{z\},\{t\},\{x,y\}}$:] type $\mathbb{A}_2$ with quadratic term $z(x+y+z-t-\lambda t)$;

\item[$P_{\{x\},\{t\},\{y,z\}}$:] type $\mathbb{A}_1$.
\end{itemize}
In particular, every fiber
$\mathsf{f}^{-1}(\lambda)$ is irreducible by Lemma~\ref{corollary:irreducible-fibers}.
This confirms \eqref{equation:main-1} in Main Theorem, since $h^{1,2}(X)=0$.

Now let us verify \eqref{equation:main-2} in Main Theorem.
If $\lambda\ne 0$, then the intersection matrix of the curves $L_{\{x\},\{y\}}$, $L_{\{y\},\{t\}}$, $L_{\{z\},\{t\}}$, and $H_{\lambda}$
on the surface $S_\lambda$ is given by
\begin{center}\renewcommand\arraystretch{1.42}
\begin{tabular}{|c||c|c|c|c|}
\hline
 $\bullet$  & $L_{\{x\},\{y\}}$ & $L_{\{y\},\{t\}}$ & $L_{\{z\},\{t\}}$ &  $H_{\lambda}$ \\
\hline\hline
$L_{\{x\},\{y\}}$ &  $-\frac{2}{3}$ & $\frac{3}{5}$ & 0 & $1$ \\
\hline
$L_{\{y\},\{t\}}$ &  $\frac{3}{5}$ & $-\frac{4}{3}$ & $\frac{1}{3}$ & $1$ \\
\hline
$L_{\{z\},\{t\}}$ &  $0$ & $\frac{1}{3}$ & $-\frac{8}{15}$ & $1$ \\
\hline
 $H_{\lambda}$  & $1$ & $1$ & $1$ & $4$ \\
\hline
\end{tabular}
\end{center}
This matrix has rank~$4$.
On the other hand, if $\lambda\ne 0$, then the intersection matrix of the curves
$L_{\{x\},\{y\}}$, $L_{\{x\},\{z\}}$, $L_{\{x\},\{t\}}$,
$L_{\{y\},\{t\}}$ $L_{\{z\},\{t\}}$, $L_{\{t\},\{x,y,z\}}$, and $\mathcal{C}$
on the surface $S_\lambda$ has the same rank as the intersection matrix
of the curves $L_{\{x\},\{y\}}$, $L_{\{y\},\{t\}}$, $L_{\{z\},\{t\}}$, and $H_{\lambda}$,
because
\begin{multline*}
L_{\{x\},\{y\}}+L_{\{x\},\{z\}}+2L_{\{x\},\{t\}}\sim L_{\{x\},\{y\}}+L_{\{y\},\{t\}}+\mathcal{C}\sim\\
\sim L_{\{x\},\{z\}}+3L_{\{z\},\{t\}}\sim L_{\{x\},\{t\}}+L_{\{y\},\{t\}}+L_{\{z\},\{t\}}+L_{\{t\},\{x,y,z\}}\sim H_\lambda
\end{multline*}
on the surface $S_\lambda$ by \eqref{equation:2-31}.
Moreover, it follows from the description of singularities of the surface $S_\lambda$ that
$\mathrm{rk}\,\mathrm{Pic}(\widetilde{S}_{\Bbbk})=\mathrm{rk}\,\mathrm{Pic}(S_{\Bbbk})+14$.
Therefore, we conclude that \eqref{equation:main-2-simple} holds, so that \eqref{equation:main-2} in Main Theorem also holds by Lemma~\ref{lemma:cokernel}.

\subsection{Family \textnumero $2.32$}
\label{section:r-2-n-32}

In this case, the threefold $X$ is a divisor of bidegree $(1,1)$ on $\mathbb{P}^2\times\mathbb{P}^2$, so that $h^{1,2}(X)=0$.
A~toric Landau--Ginzburg model of this family is given by Minkowski polynomial \textnumero $21$, which is
$$
x+y+z+\frac{1}{y}+\frac{1}{x}+\frac{1}{xyz}.
$$
The quartic pencil $\mathcal{S}$ is given by the equation
$$
x^2yz+y^2xz+z^2yx+t^2xz+t^2yz+t^4=\lambda xyzt.
$$
As usual, we suppose that $\lambda\ne\infty$.

Let $\mathcal{C}_1$ be the conic in $\mathbb{P}^3$ that is given by $x=yz+t^2=0$,
and let $\mathcal{C}_2$ be the conic in $\mathbb{P}^3$ that is given by $y=xz+t^2=0$.
Then
\begin{equation}
\label{equation:2-32}
\begin{split}
H_{\{x\}}\cdot S_\lambda&=2L_{\{x\},\{t\}}+\mathcal{C}_1,\\
H_{\{y\}}\cdot S_\lambda&=2L_{\{y\},\{t\}}+\mathcal{C}_2,\\
H_{\{z\}}\cdot S_\lambda&=4L_{\{z\},\{t\}},\\
H_{\{t\}}\cdot S_\lambda&=L_{\{x\},\{t\}}+L_{\{y\},\{t\}}+L_{\{z\},\{t\}}+L_{\{t\},\{x,y,z\}}.
\end{split}
\end{equation}
This shows that the base locus of the pencil $\mathcal{S}$ consists of the curves
$L_{\{x\},\{t\}}$, $L_{\{y\},\{t\}}$, $L_{\{z\},\{t\}}$, $L_{\{t\},\{x,y,z\}}$, $\mathcal{C}_1$, and $\mathcal{C}_2$.

For every $\lambda\in\mathbb{C}$, the surface $S_\lambda$ has isolated singularities, so that it is irreducible.
Its~singularities contained in the base locus of the pencil $\mathcal{S}$ can be described as follows:
\begin{itemize}\setlength{\itemindent}{2cm}
\item[$P_{\{x\},\{y\},\{t\}}$:] type $\mathbb{A}_4$ with quadratic term $xy$, for $\lambda\neq 0$, type $\mathbb{A}_5$ for $\lambda=0$;

\item[$P_{\{x\},\{z\},\{t\}}$:] type $\mathbb{A}_3$ with quadratic term $xz$;

\item[$P_{\{y\},\{z\},\{t\}}$:] type $\mathbb{A}_3$ with quadratic term $yz$;

\item[$P_{\{x\},\{t\},\{y,z\}}$:] type $\mathbb{A}_1$;

\item[$P_{\{y\},\{t\},\{x,z\}}$:] type $\mathbb{A}_1$;

\item[$P_{\{z\},\{t\},\{x,y\}}$:] type $\mathbb{A}_3$ with quadratic term $z(x+y+z-\lambda t)$.
\end{itemize}
In particular, every fiber
$\mathsf{f}^{-1}(\lambda)$ is irreducible by Lemma~\ref{corollary:irreducible-fibers}.
This confirms \eqref{equation:main-1} in Main Theorem, since $h^{1,2}(X)=0$.
To verify \eqref{equation:main-2} in Main Theorem, we need the following result:

\begin{lemma}
\label{lemma:r2-n32-intersection}
Suppose that $\lambda\ne 0$.
Then the intersection matrix of the curves $L_{\{x\},\{t\}}$, $L_{\{y\},\{t\}}$, and $H_{\lambda}$
on the surface $S_\lambda$ is given by
\begin{center}\renewcommand\arraystretch{1.42}
\begin{tabular}{|c||c|c|c|}
\hline
 $\bullet$  & $L_{\{x\},\{t\}}$ & $L_{\{y\},\{t\}}$ & $H_{\lambda}$ \\
\hline\hline
$L_{\{x\},\{t\}}$ &  $\frac{1}{20}$ & $\frac{1}{5}$ & $1$ \\
\hline
$L_{\{y\},\{t\}}$ &  $\frac{1}{5}$ & $\frac{1}{20}$ & $1$ \\
\hline
\hline
 $H_{\lambda}$  & $1$ & $1$ & $4$ \\
\hline
\end{tabular}
\end{center}
\end{lemma}

\begin{proof}
To find $L_{\{x\},\{t\}}^2$, observe that the singular points of the surface $S_\lambda$ contained in the line $L_{\{x\},\{t\}}$ are
the points $P_{\{x\},\{z\},\{t\}}$, $P_{\{x\},\{y\},\{t\}}$, and $P_{\{x\},\{t\},\{y,z\}}$.
These points are singular points of the surface $S_\lambda$ of types $\mathbb{A}_3$, $\mathbb{A}_4$, and $\mathbb{A}_1$, respectively.
Applying Proposition~\ref{proposition:du-Val-self-intersection},
we see that
$$
L_{\{x\},\{t\}}^2=-2+\frac{3}{4}+\frac{4}{5}+\frac{1}{2}=\frac{1}{20}.
$$
Similarly, we find $L_{\{y\},\{t\}}^2=\frac{1}{20}$.
Finally, observe that
$$
L_{\{x\},\{t\}}\cap L_{\{y\},\{t\}}=P_{\{x\},\{y\},\{t\}}.
$$
Using Remark~\ref{remark:transversal} with $S=S_\lambda$, $n=4$, $O=P_{\{x\},\{y\},\{t\}}$, $C=L_{\{x\},\{t\}}$, and $Z=L_{\{y\},\{t\}}$,
we see that both curves $\overline{C}$ and $\overline{Z}$ do not contain the point $\overline{G}_1\cap\overline{G}_4$.
Moreover, since the quadratic term of the surface $S_\lambda$ at the singular point $P_{\{x\},\{y\},\{t\}}$ is $xy$,
we see that either $\overline{C}\cdot\overline{G}_1=\overline{Z}\cdot\overline{G}_4=1$, or
$\overline{C}\cdot\overline{G}_4=\overline{Z}\cdot\overline{G}_1=1$.
Thus, using Proposition~\ref{proposition:du-Val-intersection}, we conclude that $L_{\{x\},\{t\}}\cdot L_{\{y\},\{t\}}=\frac{1}{5}$.
\end{proof}

If $\lambda\ne 0$, then the intersection matrix of the curves
$L_{\{x\},\{t\}}$, $L_{\{y\},\{t\}}$, $L_{\{z\},\{t\}}$, $L_{\{t\},\{x,y,z\}}$, $\mathcal{C}_1$, and $\mathcal{C}_2$
on the surface $S_\lambda$ has the same rank as the intersection matrix of the curves $L_{\{x\},\{t\}}$, $L_{\{y\},\{t\}}$, and $H_{\lambda}$,
because
$$
2L_{\{x\},\{t\}}+\mathcal{C}_1\sim 2L_{\{y\},\{t\}}+\mathcal{C}_2\sim 4L_{\{z\},\{t\}}\sim L_{\{x\},\{t\}}+L_{\{y\},\{t\}}+L_{\{z\},\{t\}}+L_{\{t\},\{x,y,z\}}\sim H_\lambda
$$
on the surface $S_\lambda$ by \eqref{equation:2-32}.
On the other hand, the matrix in Lemma~\ref{lemma:r2-n32-intersection} has rank~$3$.
Moreover, we have
$\mathrm{rk}\,\mathrm{Pic}(\widetilde{S}_{\Bbbk})=\mathrm{rk}\,\mathrm{Pic}(S_{\Bbbk})+15$.
Hence, we see that \eqref{equation:main-2-simple} holds, so that \eqref{equation:main-2} in Main Theorem also holds by Lemma~\ref{lemma:cokernel}.

\subsection{Family \textnumero $2.33$}
\label{section:r-2-n-33}

The threefold $X$ is a blow up of $\mathbb{P}^3$ in a line, so that $h^{1,2}(X)=0$.
A~toric Landau--Ginzburg model of this family is given by Minkowski polynomial \textnumero $6$, which is
$$
x+y+z+\frac{x}{z}+\frac{1}{xy}.
$$
The quartic pencil $\mathcal{S}$ is given by the equation
$$
x^2yz+y^2xz+z^2yx+x^2ty+t^3z=\lambda xyzt.
$$

Suppose that $\lambda\ne\infty$. Then
\begin{itemize}
\item $H_{\{x\}}\cdot S_\lambda=L_{\{x\},\{z\}}+3L_{\{x\},\{t\}}$,
\item $H_{\{y\}}\cdot S_\lambda=L_{\{y\},\{z\}}+3L_{\{y\},\{t\}}$,
\item $H_{\{z\}}\cdot S_\lambda=2L_{\{x\},\{z\}}+L_{\{y\},\{z\}}+L_{\{z\},\{t\}}$,
\item $H_{\{t\}}\cdot S_\lambda=L_{\{x\},\{t\}}+L_{\{y\},\{t\}}+L_{\{z\},\{t\}}+L_{\{t\},\{x,y,z\}}$.
\end{itemize}
Thus, the base locus of the pencil $\mathcal{S}$ consists of the lines
$L_{\{x\},\{z\}}$, $L_{\{x\},\{t\}}$, $L_{\{y\},\{z\}}$, $L_{\{y\},\{t\}}$, $L_{\{z\},\{t\}}$, and $L_{\{t\},\{x,y,z\}}$.

Observe that the surface $S_\lambda$ has isolated singularities for every $\lambda\in\mathbb{C}$, so that it is irreducible.
Moreover, the singular points of the surface $S_\lambda$ contained in the base locus of the pencil $\mathcal{S}$ can be described as follows:
\begin{itemize}\setlength{\itemindent}{3cm}
\item[$P_{\{x\},\{y\},\{t\}}$:] type $\mathbb{A}_2$ with quadratic term $xy$;

\item[$P_{\{x\},\{z\},\{t\}}$:] type $\mathbb{A}_6$ with quadratic term $xz$;

\item[$P_{\{y\},\{z\},\{t\}}$:] type $\mathbb{A}_3$ with quadratic term $y(z+t)$;

\item[$P_{\{x\},\{t\},\{y,z\}}$:] type $\mathbb{A}_2$ with quadratic term $x(x+y+z+\lambda t)$;

\item[$P_{\{y\},\{t\},\{x,z\}}$:] type $\mathbb{A}_2$ with quadratic term $y(x+y+z-t-\lambda t)$.
\end{itemize}
Then $[\mathsf{f}^{-1}(\lambda)]=1$ for every $\lambda\in\mathbb{C}$ by Lemma~\ref{corollary:irreducible-fibers}.
This confirms \eqref{equation:main-1} in Main Theorem.

\begin{lemma}
\label{lemma:r2-n33-intersection}
The intersection matrix of the curves $L_{\{x\},\{z\}}$, $L_{\{y\},\{z\}}$, and $H_{\lambda}$
on the surface $S_\lambda$ is given by
\begin{center}\renewcommand\arraystretch{1.42}
\begin{tabular}{|c||c|c|c|}
\hline
 $\bullet$  & $L_{\{x\},\{z\}}$ & $L_{\{y\},\{z\}}$ & $H_{\lambda}$ \\
\hline\hline
$L_{\{x\},\{z\}}$ &  $-\frac{2}{7}$ & $1$ & $1$ \\
\hline
$L_{\{y\},\{z\}}$ &  $1$ & $-\frac{5}{4}$ & $1$ \\
\hline
 $H_{\lambda}$  & $1$ & $1$ & $4$ \\
\hline
\end{tabular}
\end{center}
\end{lemma}

\begin{proof}
The only singular point of the surface $S_\lambda$ contained in $L_{\{x\},\{z\}}$ is the point $P_{\{x\},\{z\},\{t\}}$.
Let us use the notation of Appendix~\ref{subsection:A} with $S=S_\lambda$, $n=6$, $O=P_{\{x\},\{z\},\{t\}}$, and $C=L_{\{x\},\{z\}}$.
Then it follows from explicit computations that $\widetilde{C}$ intersects one of the curves $G_3$ or $G_4$.
Then $L_{\{x\},\{z\}}^2=-\frac{2}{7}$ by Proposition~\ref{proposition:du-Val-self-intersection}.

The only singular point of the surface $S_\lambda$ contained in $L_{\{y\},\{z\}}$ is the point $P_{\{y\},\{z\},\{t\}}$.
Using Remark~\ref{remark:transversal} with $S=S_\lambda$, $n=3$, $O=P_{\{y\},\{z\},\{t\}}$, and $C=L_{\{y\},\{z\}}$,
we see that the curve $\overline{C}$ does not contain the point $\overline{G}_1\cap\overline{G}_3$.
Then $L_{\{y\},\{z\}}^2=-\frac{5}{4}$ by Proposition~\ref{proposition:du-Val-self-intersection}.

Finally, observe that $L_{\{x\},\{z\}}\cap L_{\{y\},\{z\}}=P_{\{x\},\{y\},\{z\}}$
and $S_\lambda$ is smooth at $P_{\{x\},\{y\},\{z\}}$. Thus, we conclude that $L_{\{x\},\{z\}}\cdot L_{\{y\},\{z\}}=1$.
\end{proof}

The intersection matrix of the lines
$L_{\{x\},\{z\}}$, $L_{\{x\},\{t\}}$, $L_{\{y\},\{z\}}$, $L_{\{y\},\{t\}}$, $L_{\{z\},\{t\}}$, and $L_{\{t\},\{x,y,z\}}$
on the surface $S_\lambda$ has the same rank as the intersection matrix in Lemma~\ref{lemma:r2-n32-intersection},
because
\begin{multline*}
H_\lambda \sim L_{\{x\},\{z\}}+3L_{\{x\},\{t\}}\sim L_{\{y\},\{z\}}+3L_{\{y\},\{t\}}\sim\\
\sim 2L_{\{x\},\{z\}}+L_{\{y\},\{z\}}+L_{\{z\},\{t\}}\sim L_{\{x\},\{t\}}+L_{\{y\},\{t\}}+L_{\{z\},\{t\}}+L_{\{t\},\{x,y,z\}}.
\end{multline*}
On the other hand, the matrix in Lemma~\ref{lemma:r2-n33-intersection} has rank~$3$.
Moreover, it follows from the description of singularities of the surface $S_\lambda$ that
$\mathrm{rk}\,\mathrm{Pic}(\widetilde{S}_{\Bbbk})=\mathrm{rk}\,\mathrm{Pic}(S_{\Bbbk})+15$.
Hence, we see that \eqref{equation:main-2-simple} holds, so that \eqref{equation:main-2} in Main Theorem also holds by Lemma~\ref{lemma:cokernel}.

\subsection{Family \textnumero $2.34$}
\label{section:r-2-n-34}

One has $X\cong\mathbb{P}^1\times\mathbb{P}^2$.
We discussed this case in Example~\ref{example:r-2-n-34}, where we described the pencil $\mathcal{S}$ and its base locus.
In this example, we also verified \eqref{equation:main-2} in Main Theorem,
so that now we will only check \eqref{equation:main-1} in Main Theorem.

If $\lambda\ne\infty$, then $S_\lambda$ is irreducible, it has isolated singularities,
and its singular points contained in the base locus of the pencil $\mathcal{S}$ can be described as follows:
\begin{itemize}\setlength{\itemindent}{3cm}

\item[$P_{\{x\},\{y\},\{t\}}$:] type $\mathbb{A}_4$ with quadratic term $xy$;

\item[$P_{\{x\},\{z\},\{t\}}$:] type $\mathbb{A}_4$ with quadratic term $xz$;

\item[$P_{\{y\},\{z\},\{t\}}$:] type $\mathbb{A}_2$ with quadratic term $yz$;

\item[$P_{\{y\},\{t\},\{x+z\}}$:] type $\mathbb{A}_2$ with quadratic term $y(x+y+z-\lambda t)$;

\item[$P_{\{z\},\{t\},\{x+y\}}$:] type $\mathbb{A}_2$ with quadratic term $z(x+y+z-\lambda t)$;

\item[$P_{\{x\},\{t\},\{y+z\}}$:] type $\mathbb{A}_1$.
\end{itemize}
Then $[\mathsf{f}^{-1}(\lambda)]=1$ for every $\lambda\in\mathbb{C}$ by Lemma~\ref{corollary:irreducible-fibers}.
This confirms \eqref{equation:main-1} in Main Theorem.

\subsection{Family \textnumero $2.35$}
\label{section:r-2-n-35}

We have $X\cong\mathbb{P}(\mathcal O_{\mathbb{P}^2}\oplus \mathcal O_{\mathbb{P}^2}(1))$.
A~toric Landau--Ginzburg model of this family is given by Minkowski polynomial \textnumero $5$, which is
$$
x+y+z+\frac {x}{yz}+\frac{1}{x}.
$$
The quartic pencil $\mathcal{S}$ is given by
$$
x^2yz+y^2zx+z^2yx+x^2t^2+t^2yz=\lambda xyzt.
$$

Suppose that $\lambda\ne\infty$. Then
\begin{itemize}
\item $H_{\{x\}}\cdot S_\lambda=L_{\{x\},\{y\}}+L_{\{x\},\{z\}}+2L_{\{x\},\{t\}}$,
\item $H_{\{y\}}\cdot S_\lambda=2L_{\{x\},\{y\}}+2L_{\{y\},\{t\}}$,
\item $H_{\{z\}}\cdot S_\lambda=2L_{\{x\},\{z\}}+2L_{\{z\},\{t\}}$,
\item $H_{\{t\}}\cdot S_\lambda=L_{\{x\},\{t\}}+L_{\{y\},\{t\}}+L_{\{z\},\{t\}}+L_{\{t\},\{x,y,z\}}$.
\end{itemize}
Thus, the base locus of the pencil $\mathcal{S}$ consists of the lines
$L_{\{x\},\{y\}}$, $L_{\{x\},\{z\}}$, $L_{\{x\},\{t\}}$, $L_{\{y\},\{t\}}$, $L_{\{z\},\{t\}}$, and $L_{\{t\},\{x,y,z\}}$.

For every $\lambda\in\mathbb{C}$, the surface $S_\lambda\in\mathcal{S}$ has isolated singularities.
In particular, it is irreducible.
Moreover, its singular points contained in the base locus of the pencil $\mathcal{S}$ can be described as follows:
\begin{itemize}\setlength{\itemindent}{3cm}
\item[$P_{\{x\},\{y\},\{z\}}$:] type $\mathbb{A}_1$;

\item[$P_{\{x\},\{y\},\{t\}}$:] type $\mathbb{A}_5$ with quadratic term $xy$;

\item[$P_{\{x\},\{z\},\{t\}}$:] type $\mathbb{A}_5$ with quadratic term $xz$;

\item[$P_{\{y\},\{z\},\{t\}}$:] type $\mathbb{A}_1$;

\item[$P_{\{x\},\{t\},\{y,z\}}$:] type $\mathbb{A}_1$;

\item[$P_{\{y\},\{t\},\{x,z\}}$:] type $\mathbb{A}_1$;

\item[$P_{\{z\},\{t\},\{x,y\}}$:] type $\mathbb{A}_1$.
\end{itemize}
In particular, every fiber $\mathsf{f}^{-1}(\lambda)$ is irreducible by Lemma~\ref{corollary:irreducible-fibers}.
This confirms \eqref{equation:main-1} in Main Theorem, since $h^{1,2}(X)=0$.

To verify \eqref{equation:main-2} in Main Theorem, observe that
the intersection matrix of the curves $L_{\{x\},\{y\}}$, $L_{\{x\},\{z\}}$, $L_{\{x\},\{t\}}$, $L_{\{y\},\{t\}}$, $L_{\{z\},\{t\}}$, and $H_{\lambda}$
on the surface $S_\lambda$ is given by
\begin{center}\renewcommand\arraystretch{1.42}
\begin{tabular}{|c||c|c|c|c|c|c|}
\hline
 $\bullet$  & $L_{\{x\},\{y\}}$ & $L_{\{x\},\{z\}}$ & $L_{\{x\},\{t\}}$ & $L_{\{y\},\{t\}}$ & $L_{\{z\},\{t\}}$ & $H_{\lambda}$ \\
\hline\hline
 $L_{\{x\},\{y\}}$ &  $-\frac{1}{6}$ & $\frac{1}{2}$ & $\frac{1}{3}$ & $\frac{2}{3}$ & $0$ & $1$ \\
\hline
 $L_{\{x\},\{z\}}$ &  $\frac{1}{2}$ & $-\frac{1}{6}$ & $\frac{1}{3}$ & $0$ & $\frac{2}{3}$ & $1$ \\
\hline
 $L_{\{x\},\{t\}}$ &  $\frac{1}{3}$ & $\frac{1}{3}$ & $\frac{1}{6}$ & $\frac{1}{6}$ & $\frac{1}{6}$ & $1$ \\
\hline
 $L_{\{y\},\{t\}}$ &  $\frac{2}{3}$ & $0$ & $\frac{1}{6}$ & $-\frac{1}{6}$ & $\frac{1}{2}$ & $1$ \\
\hline
 $L_{\{z\},\{t\}}$ &  $0$ & $\frac{2}{3}$ & $\frac{1}{6}$ & $\frac{1}{2}$ & $-\frac{1}{6}$ & $1$ \\
\hline
 $H_{\lambda}$  & $1$ & $1$ & $1$ & $1$ & $1$ & $4$ \\
\hline
\end{tabular}
\end{center}
This matrix has rank~$3$.
On the other hand, the intersection matrix of the lines
$L_{\{x\},\{y\}}$, $L_{\{x\},\{z\}}$, $L_{\{x\},\{t\}}$, $L_{\{y\},\{t\}}$, $L_{\{z\},\{t\}}$, and $L_{\{t\},\{x,y,z\}}$
on the surface $S_\lambda$ has the same rank as the intersection matrix
of the curves $L_{\{x\},\{y\}}$, $L_{\{x\},\{z\}}$, $L_{\{x\},\{t\}}$, $L_{\{y\},\{t\}}$, $L_{\{z\},\{t\}}$, and $H_{\lambda}$,
because
\begin{multline*}
H_\lambda\sim L_{\{x\},\{y\}}+L_{\{x\},\{z\}}+2L_{\{x\},\{t\}}\sim 2L_{\{x\},\{y\}}+2L_{\{y\},\{t\}}\sim\\
\sim 2L_{\{x\},\{z\}}+2L_{\{z\},\{t\}}\sim L_{\{x\},\{t\}}+L_{\{y\},\{t\}}+L_{\{z\},\{t\}}+L_{\{t\},\{x,y,z\}}.
\end{multline*}
Moreover, we have $\mathrm{rk}\,\mathrm{Pic}(\widetilde{S}_{\Bbbk})=\mathrm{rk}\,\mathrm{Pic}(S_{\Bbbk})+15$.
Therefore, we conclude that \eqref{equation:main-2-simple} holds, so that \eqref{equation:main-2} in Main Theorem also holds by Lemma~\ref{lemma:cokernel}.

\subsection{Family \textnumero $2.36$}
\label{section:r-2-n-36}

In this case, we have $X\cong\mathbb{P}(\mathcal O_{\mathbb{P}^2}\oplus \mathcal O_{\mathbb{P}^2}(2))$, so that $h^{1,2}(X)=0$.
A~toric Landau--Ginzburg model of this family is given by Minkowski polynomial \textnumero $7$, which is
$$
x+y+z+\frac{{x}^2}{yz}+\frac{1}{x}.
$$
The pencil $\mathcal{S}$ is given by the equation
$$
x^2yz+y^2zx+z^2yx+x^3t+t^2yz=\lambda xyzt.
$$

Suppose that $\lambda\ne\infty$. Then
\begin{equation}
\label{equation:2-36}
\begin{split}
H_{\{x\}}\cdot S_\lambda&=L_{\{x\},\{y\}}+L_{\{x\},\{z\}}+2L_{\{x\},\{t\}},\\
H_{\{y\}}\cdot S_\lambda&=3L_{\{x\},\{y\}}+L_{\{y\},\{t\}},\\
H_{\{z\}}\cdot S_\lambda&=3L_{\{x\},\{z\}}+L_{\{z\},\{t\}},\\
H_{\{t\}}\cdot S_\lambda&=L_{\{x\},\{t\}}+L_{\{y\},\{t\}}+L_{\{z\},\{t\}}+L_{\{t\},\{x,y,z\}}.
\end{split}
\end{equation}
Thus, the base locus of the pencil $\mathcal{S}$ consists of the lines
$L_{\{x\},\{y\}}$, $L_{\{x\},\{z\}}$, $L_{\{x\},\{t\}}$,
$L_{\{y\},\{t\}}$, $L_{\{z\},\{t\}}$, and $L_{\{t\},\{x,y,z\}}$.

For every $\lambda\in\mathbb{C}$, the surface $S_\lambda$ has isolated singularities, so that it is irreducible.
Moreover, its singular points contained in the base locus of the pencil $\mathcal{S}$ can be described as follows:
\begin{itemize}\setlength{\itemindent}{3cm}
\item[$P_{\{x\},\{y\},\{t\}}$:] type $\mathbb{A}_6$ with quadratic term $xy$;

\item[$P_{\{x\},\{z\},\{t\}}$:] type $\mathbb{A}_6$ with quadratic term $xz$;

\item[$P_{\{x\},\{t\},\{y,z\}}$:] type $\mathbb{A}_1$;

\item[$P_{\{x\},\{y\},\{z\}}$:] type $\mathbb{A}_2$ with quadratic term $yz$.
\end{itemize}
In particular, every fiber $\mathsf{f}^{-1}(\lambda)$ is irreducible by Lemma~\ref{corollary:irreducible-fibers}.
This confirms \eqref{equation:main-1} in Main Theorem, since $h^{1,2}(X)=0$.

On the surface $S_\lambda$, the intersection matrix of the lines
$L_{\{x\},\{y\}}$, $L_{\{x\},\{z\}}$, $L_{\{x\},\{t\}}$, $L_{\{y\},\{t\}}$, $L_{\{z\},\{t\}}$, and $L_{\{t\},\{x,y,z\}}$
has the same rank as the intersection matrix
of the curves $L_{\{x\},\{y\}}$, $L_{\{x\},\{z\}}$, and $H_{\lambda}$,
because
\begin{multline*}
H_\lambda\sim L_{\{x\},\{y\}}+L_{\{x\},\{z\}}+2L_{\{x\},\{t\}}\sim 3L_{\{x\},\{y\}}+L_{\{y\},\{t\}}\sim\\
\sim 3L_{\{x\},\{z\}}+L_{\{z\},\{t\}}\sim L_{\{x\},\{t\}}+L_{\{y\},\{t\}}+L_{\{z\},\{t\}}+L_{\{t\},\{x,y,z\}}.
\end{multline*}
These rational equivalences follows from \eqref{equation:2-36}.

\begin{lemma}
\label{lemma:r2-n36-intersection}
The intersection matrix of the curves $L_{\{x\},\{y\}}$, $L_{\{x\},\{z\}}$, and $H_{\lambda}$ on the surface $S_\lambda$ is given by
\begin{center}\renewcommand\arraystretch{1.42}
\begin{tabular}{|c||c|c|c|}
\hline
 $\bullet$  & $L_{\{x\},\{y\}}$ & $L_{\{x\},\{z\}}$ & $H_{\lambda}$ \\
\hline\hline
$L_{\{x\},\{y\}}$ &  $\frac{2}{21}$ & $\frac{1}{3}$ & $1$ \\
\hline
$L_{\{x\},\{z\}}$ &  $\frac{1}{3}$ & $\frac{2}{21}$ & $1$ \\
\hline
 $H_{\lambda}$  & $1$ & $1$ & $4$ \\
\hline
\end{tabular}
\end{center}
\end{lemma}

\begin{proof}
By definition, we have $H_\lambda^2=4$ and $H_\lambda\cdot L_{\{x\},\{y\}}=H_\lambda\cdot L_{\{x\},\{z\}}=1$.
Note that
$$
L_{\{x\},\{y\}}\cap L_{\{x\},\{z\}}=P_{\{x\},\{y\},\{z\}}.
$$
Recall that $P_{\{x\},\{y\},\{z\}}$ is a singular point of the surface $S_\lambda$ of type $\mathbb{A}_2$.
Then one gets $L_{\{x\},\{y\}}\cdot L_{\{x\},\{z\}}=\frac{1}{3}$ by Proposition~\ref{proposition:du-Val-intersection}.

We claim that $L_{\{y\},\{t\}}^2=-\frac{8}{7}$.
Indeed, the point $P_{\{x\},\{y\},\{t\}}$ is the only singular point of the surface $S_\lambda$ that is contained in $L_{\{y\},\{t\}}$.
Using Remark~\ref{remark:transversal} with $S=S_\lambda$, $n=6$, $O=P_{\{x\},\{y\},\{t\}}$, and $C=L_{\{y\},\{t\}}$,
we see that $\overline{C}$ does not contain the point $\overline{G}_1\cap\overline{G}_6$, because
the quadratic term of the surface $S_\lambda$ at the point $P_{\{x\},\{y\},\{t\}}$ is $xy$.
Thus, we have $L_{\{y\},\{t\}}^2=-\frac{8}{7}$ by Proposition~\ref{proposition:du-Val-self-intersection}.

Since $L_{\{y\},\{t\}}^2=-\frac{8}{7}$, we get $L_{\{x\},\{y\}}\cdot L_{\{y\},\{t\}}=\frac{5}{7}$, because
$$
1=H_\lambda\cdot L_{\{y\},\{t\}}=\big(3L_{\{x\},\{y\}}+L_{\{y\},\{t\}}\big)\cdot L_{\{y\},\{t\}}=3 L_{\{x\},\{y\}}\cdot L_{\{y\},\{t\}}-\frac{8}{7}.
$$
Since $L_{\{x\},\{y\}}\cdot L_{\{y\},\{t\}}=\frac{5}{7}$, we get $L_{\{x\},\{y\}}^2=\frac{2}{21}$, because
$$
1=H_\lambda\cdot L_{\{y\},\{t\}}=\big(3L_{\{x\},\{y\}}+L_{\{y\},\{t\}}\big)\cdot L_{\{x\},\{y\}}=3 L_{\{x\},\{y\}}^2+\frac{5}{7}.
$$
Similarly, we see that $L_{\{z\},\{t\}}^2=\frac{2}{21}$.
\end{proof}

The matrix in Lemma~\ref{lemma:r2-n36-intersection} has rank~$3$.
Moreover, we have
$\mathrm{rk}\,\mathrm{Pic}(\widetilde{S}_{\Bbbk})=\mathrm{rk}\,\mathrm{Pic}(S_{\Bbbk})+15$.
Hence, we see that \eqref{equation:main-2-simple} holds, so that \eqref{equation:main-2} in Main Theorem also holds by Lemma~\ref{lemma:cokernel}.

\section{Fano threefolds of Picard rank $3$}
\label{section:rank-3}

\subsection{Family \textnumero $3.1$}
\label{section:r-3-n-1}

In this case, the threefold $X$ is a double cover of $\mathbb{P}^1\times\mathbb{P}^1\times\mathbb{P}^1$ branched over a
smooth divisor of tridegree $(2,2,2)$, which implies that $h^{1,2}(X)=8$.
The toric Landau--Ginzburg model is given by Minkowski polynomial \textnumero $3873.4$,
which is the Laurent polynomial
$$
x+y+\frac{x}{z}+\frac{y}{z}+\frac{xz}{y}+3z+\frac{yz}{x}+\frac{2x}{y}+\frac{2y}{x}+\frac{x}{yz}+\frac{3}{z}+\frac{y}{xz}+\frac{z^2}{y}+\frac{z^2}{x}+\frac{3z}{y}+\frac{3z}{x}+\frac{3}{y}+\frac{3}{x}+\frac{1}{yz}+\frac{1}{xz}.
$$
The quartic pencil $\mathcal{S}$ is given by
\begin{multline*}
x^2yz+y^2zx+x^2ty+y^2tx+x^2z^2+3z^2yx+y^2z^2+2x^2tz+2y^2tz+x^2t^2+3t^2yx+\\
+t^2y^2+z^3x+z^3y+3z^2tx+3z^2ty+3t^2zx+3t^2yz+t^3x+t^3y=\lambda xyzt.
\end{multline*}
This equation is symmetric with respect to permutations of variables $x\leftrightarrow y$ and $z\leftrightarrow t$.

To prove Main Theorem in this case, we may assume that $\lambda\ne\infty$. Then
\begin{equation}
\label{equation:r3-n1-base-locus}
\begin{split}
H_{\{x\}}\cdot S_\lambda&=L_{\{x\},\{y\}}+2L_{\{x\},\{z,t\}}+L_{\{x\},\{y,z,t\}},\\
H_{\{y\}}\cdot S_\lambda&=L_{\{x\},\{y\}}+2L_{\{y\},\{z,t\}}+L_{\{y\},\{x,z,t\}},\\
H_{\{z\}}\cdot S_\lambda&=L_{\{z\},\{t\}}+L_{\{z\},\{x,y,t\}}+\mathcal{C}_1,\\
H_{\{t\}}\cdot S_\lambda&=L_{\{z\},\{t\}}+L_{\{t\},\{x,y,z\}}+\mathcal{C}_2,
\end{split}
\end{equation}
where $\mathcal{C}_1$ is a smooth conic that is given by $z=xy+xt+yt=0$,
and $\mathcal{C}_2$ is a smooth conic that is given by $t=xy+xz+yz=0$.
Hence, since $\lambda\ne\infty$, we have
\begin{multline*}
S_{\lambda}\cdot S_{\infty}=2L_{\{x\},\{y\}}+2L_{\{z\},\{t\}}+2L_{\{x\},\{z,t\}}+2L_{\{y\},\{z,t\}}+\\
+L_{\{x\},\{y,z,t\}}+L_{\{y\},\{x,z,t\}}+L_{\{z\},\{x,y,t\}}+L_{\{t\},\{x,y,z\}}+\mathcal{C}_1+\mathcal{C}_2.
\end{multline*}
We let $C_1=\mathcal{C}_1$, $C_2=\mathcal{C}_2$, $C_3=L_{\{x\},\{y\}}$, $C_4=L_{\{z\},\{t\}}$, $C_5=L_{\{x\},\{z,t\}}$,
$C_6=L_{\{y\},\{z,t\}}$, $C_7=L_{\{x\},\{y,z,t\}}$, $C_8=L_{\{y\},\{x,z,t\}}$, $C_9=L_{\{z\},\{x,y,t\}}$, and $C_{10}=L_{\{t\},\{x,y,z\}}$.
These are all base curves of the pencil $\mathcal{S}$.

For every $\lambda\ne -6$, the surface $S_\lambda$ has isolated singularities, so that $S_\lambda$ is irreducible.
On the other hand, we have $S_{-6}=H_{\{z,t\}}+H_{\{x,y,z,t\}}+\mathbf{Q}$,
where $\mathbf{Q}$ is an irreducible quadric given by $xy+xz+yz+xt+yt=0$.
This quadric is singular at the point $P_{\{y\},\{z,t\}}$,
which is also contained in the planes $H_{\{z,t\}}$ and $H_{\{x,y,z,t\}}$.

If $\lambda\ne-6$, then the singularities of the surface $S_\lambda$
that are contained in the base locus of the pencil $\mathcal{S}$ are all du Val and can be described as follows:
\begin{itemize}\setlength{\itemindent}{3cm}
\item[$P_{\{y\},\{z\},\{t\}}$:] type $\mathbb{A}_3$ with quadratic term $(z+t)(y+z+t)$;
\item[$P_{\{x\},\{z\},\{t\}}$:] type $\mathbb{A}_3$ with quadratic term $(z+t)(x+z+t)$;
\item[$P_{\{x\},\{y\},\{z,t\}}$:] type $\mathbb{A}_5$ with quadratic term $(\lambda+6)xy$;
\item[$P_{\{z\},\{t\},\{x,y\}}$:] type $\mathbb{A}_1$ with quadratic term $(\lambda+6)zt-(z+t)(x+y+z+t)$.
\end{itemize}
In the notation of Subsection~\ref{subsection:scheme-step-6}, the points $P_{\{y\},\{z\},\{t\}}$,
$P_{\{x\},\{z\},\{t\}}$, $P_{\{x\},\{y\},\{z,t\}}$, and $P_{\{z\},\{t\},\{x,y\}}$
are the {fixed} singular points of the quartic surfaces in the pencil $\mathcal{S}$.

By Corollary~\ref{corollary:irreducible-fibers}, the fiber $\mathsf{f}^{-1}(\lambda)$ is irreducible for every $\lambda\ne -6$.
Therefore, the assertion \eqref{equation:main-1} in  Main Theorem follows from

\begin{lemma}
\label{lemma:r3-n1-main-1}
One has $[\mathsf{f}^{-1}(-6)]=9$.
\end{lemma}

\begin{proof}
Recall that $[S_{-6}]=3$.
Moreover, we have $\mathbf{M}_{5}^{-6}=\mathbf{M}_{6}^{-6}=2$ and
$$
\mathbf{M}_{1}^{-6}=\mathbf{M}_{2}^{-6}=\mathbf{M}_{3}^{-6}=\mathbf{M}_{4}^{-6}=\mathbf{M}_{7}^{-6}=\mathbf{M}_{8}^{-6}=\mathbf{M}_{9}^{-6}=\mathbf{M}_{10}^{-6}=1.
$$
But
$\mathbf{m}_{1}=\mathbf{m}_{2}=\mathbf{m}_{7}=\mathbf{m}_{8}=\mathbf{m}_{9}=\mathbf{m}_{10}=1$,
$\mathbf{m}_{3}=\mathbf{m}_{4}=\mathbf{m}_{5}=\mathbf{m}_{4}=6$,
and the points $P_{\{y\},\{z\},\{t\}}$, $P_{\{x\},\{z\},\{t\}}$, and $P_{\{z\},\{t\},\{x,y\}}$
are non-isolated ordinary double points of the surface $S_{-6}$.
Thus, it follows from \eqref{equation:equation:number-of-irredubicle-components-refined} and Lemmas~\ref{lemma:main} and \ref{lemma:normal-crossing} that
$$
\big[\mathsf{f}^{-1}(-6)\big]=5+\mathbf{D}_{P_{\{x\},\{y\},\{z,t\}}}^{-6}.
$$

Let $\alpha_1\colon U_1\to\mathbb{P}^3$ be a blow up of the point $P_{\{x\},\{y\},\{z,t\}}$.
Then $D_{-6}^1=S_{-6}^1+2\mathbf{E}_1$.
The surface $\mathbf{E}_1$ contains two base curves of the pencil $\mathcal{S}^1$.
Denote them by $C_{11}^1$ and $C_{12}^1$, respectively.
Then $\mathbf{m}_{11}=\mathbf{m}_{12}=2$ and $\mathbf{M}_{11}^{-6}=\mathbf{M}_{12}^{-6}=2$.

Let $\alpha_2\colon U_2\to U_1$ be the blow up of the point $C_{11}^1\cap C_{12}^1$. Then
$D_{-6}^2=S_{-6}^2+2\mathbf{E}_1^2+\mathbf{E}_2$.
The surface $\mathbf{E}_2$ contains two base curves of the pencil $\mathcal{S}^2$.
Denote them by $C_{13}^2$ and $C_{14}^2$, respectively. Then $\mathbf{M}_{13}^{-6}=\mathbf{M}_{14}^{-6}=2$.

Note that there exists a commutative diagram
$$
\xymatrix{
&&U_2\ar@{->}[dll]_{\alpha_2}&&\\
U_1\ar@{->}[rrd]_{\alpha_1}&&&&U\ar@{->}[lld]^{\alpha}\ar@{->}[llu]_{\gamma}\\
&&\mathbb{P}^3&&}
$$
for some birational morphism $\gamma$.
Moreover, the only base curves of the pencil $\widehat{\mathcal{S}}$ that
are mapped to the singular point $P_{\{x\},\{y\},\{z,t\}}$ are the curves $\widehat{C}_{11}$, $\widehat{C}_{12}$, $\widehat{C}_{13}$, and $\widehat{C}_{14}$.
Furthermore, our computations also give $\mathbf{A}_{P_{\{x\},\{y\},\{z,t\}}}=2$.
Thus, it follows from \eqref{equation:D-A-B} and Lemma~\ref{lemma:main-2} that
$\mathbf{D}_{P_{\{x\},\{y\},\{z,t\}}}=4$, so that $[\mathsf{f}^{-1}(-6)]=5$.
\end{proof}

If $\lambda\ne -6$, then the intersection matrix of the curves
$L_{\{x\},\{y\}}$, $L_{\{z\},\{t\}}$, $L_{\{x\},\{z,t\}}$,
$L_{\{y\},\{z,t\}}$, $L_{\{x\},\{y,z,t\}}$, $L_{\{y\},\{x,z,t\}}$, $L_{\{z\},\{x,y,t\}}$, $L_{\{t\},\{x,y,z\}}$, $\mathcal{C}_1$, and $\mathcal{C}_2$ on the surface $S_\lambda$
has the same  rank as the intersection matrix of the curves
$L_{\{x\},\{y\}}$, $L_{\{z\},\{t\}}$, $L_{\{x\},\{z,t\}}$, $L_{\{y\},\{z,t\}}$,
$L_{\{z\},\{x,y,t\}}$, $L_{\{t\},\{x,y,z\}}$, and $H_{\lambda}$.
This follows from \eqref{equation:r3-n1-base-locus}.
On the other hand, if $\lambda\ne -6$, then
$$
L_{\{x\},\{z,t\}}+L_{\{y\},\{z,t\}}+2L_{\{z\},\{t\}}\sim H_{\lambda}
$$
on the surface $S_\lambda$, because $H_{\{z,t\}}\cdot S_\lambda=L_{\{x\},\{z,t\}}+L_{\{y\},\{z,t\}}+2L_{\{z\},\{t\}}$. Similarly, if $\lambda\ne -6$, then
$$
L_{\{x\},\{y,z,t\}}+L_{\{y\},\{x,z,t\}}+L_{\{z\},\{x,y,t\}}+L_{\{t\},\{x,y,z\}}\sim H_{\lambda}.
$$
Using this, we can easily compute the intersection form of the curves
$L_{\{x\},\{y\}}$, $L_{\{z\},\{t\}}$, $L_{\{x\},\{z,t\}}$, $L_{\{y\},\{z,t\}}$,
$L_{\{z\},\{x,y,t\}}$, $L_{\{t\},\{x,y,z\}}$, and $H_{\lambda}$ on the surface $S_\lambda$.
If $\lambda\ne -6$, it is given by the following matrix:
\begin{center}\renewcommand\arraystretch{1.42}
\begin{tabular}{|c||c|c|c|c|c|c|c|}
\hline
$\bullet$ & $L_{\{x\},\{y\}}$ & $L_{\{z\},\{t\}}$ & $L_{\{x\},\{z,t\}}$ & $L_{\{y\},\{z,t\}}$ & $L_{\{z\},\{x,y,t\}}$ & $L_{\{t\},\{x,y,z\}}$ & $H_{\lambda}$\\
\hline\hline
$L_{\{x\},\{y\}}$ & $-\frac{1}{2}$ & $0$ & $\frac{1}{2}$ & $\frac{1}{2}$ & $0$& $0$ & $1$\\
\hline
$L_{\{z\},\{t\}}$ & $0$ & $0$ & $\frac{1}{2}$ & $\frac{1}{2}$ & $\frac{1}{2}$& $\frac{1}{2}$ & $1$\\
\hline
$L_{\{x\},\{z,t\}}$ & $\frac{1}{2}$ & $\frac{1}{2}$ & $-\frac{1}{6}$ & $\frac{1}{6}$ & $0$& $0$ & $1$\\
\hline
$L_{\{y\},\{z,t\}}$ & $\frac{1}{2}$ & $\frac{1}{2}$ & $\frac{1}{6}$ & $-\frac{1}{6}$ & $0$& $0$ & $1$\\
\hline
$L_{\{z\},\{x,y,t\}}$ & $0$ & $\frac{1}{2}$ & $0$ & $0$ & $-\frac{3}{2}$& $\frac{1}{2}$ & $1$\\
\hline
$L_{\{t\},\{x,y,z\}}$ & $0$ & $\frac{1}{2}$ & $0$ & $0$ & $\frac{1}{2}$& $-\frac{3}{2}$ & $1$\\
\hline
$H_{\lambda}$ & $1$  & $1$  & $1$  & $1$  & $1$ & $1$ & $4$ \\
\hline
\end{tabular}
\end{center}
The rank of this intersection matrix is $5$.
Moreover, we have
$\mathrm{rk}\,\mathrm{Pic}(\widetilde{S}_{\Bbbk})=\mathrm{rk}\,\mathrm{Pic}(S_{\Bbbk})+12$.
Hence, we see that \eqref{equation:main-2-simple} holds,
so that \eqref{equation:main-2} in Main Theorem also holds by Lemma~\ref{lemma:cokernel}.

\subsection{Family \textnumero $3.2$}
\label{section:r-3-n-2}

We already discussed this case in Example~\ref{example:r-3-n-2}.
Because of this, let us use the notation of this example.
Note that $h^{1,2}(X)=3$, and the defining equation of the surface $S_\lambda$
is symmetric with respect to the swaps $x\leftrightarrow y$ and $z\leftrightarrow t$.

To prove Main Theorem in this case, we may assume that $\lambda\ne\infty$. Then
\begin{equation}
\label{equation:r3-n2-base-locus}
\begin{split}
H_{\{x\}}\cdot S_\lambda&=L_{\{x\},\{t\}}+L_{\{x\},\{y,z\}}+\mathcal{C}_1,\\
H_{\{y\}}\cdot S_\lambda&=L_{\{y\},\{t\}}+L_{\{y\},\{x,z\}}+\mathcal{C}_2,\\
H_{\{z\}}\cdot S_\lambda&=L_{\{z\},\{t\}}+3L_{\{z\},\{x,y\}},\\
H_{\{t\}}\cdot S_\lambda&=L_{\{x\},\{t\}}+L_{\{y\},\{t\}}+L_{\{z\},\{t\}}+L_{\{t\},\{x,y,z\}},
\end{split}
\end{equation}

For every $\lambda\ne -6$, the surface $S_\lambda$ is irreducible, it has isolated singularities,
and its singularities contained in the base locus of the pencil $\mathcal{S}$ can be described as follows:
\begin{itemize}\setlength{\itemindent}{3cm}
\item[$P_{\{x\},\{y\},\{z\}}$:] type $\mathbb{A}_5$ with quadratic term $z(x+y+z)$;
\item[$P_{\{x\},\{t\},\{y,z\}}$:] type $\mathbb{A}_2$ with quadratic term $x(x+y+z+(\lambda+6)t)$;
\item[$P_{\{y\},\{t\},\{x,z\}}$:] type $\mathbb{A}_2$ with quadratic term $y(x+y+z+(\lambda+6)t)$;
\item[$P_{\{z\},\{t\},\{x,y\}}$:] type $\mathbb{A}_3$ with quadratic term $z(x+y+z+(\lambda+6)t)$.
\end{itemize}
By Corollary~\ref{corollary:irreducible-fibers}, one has $[\mathsf{f}^{-1}(\lambda)]=1$ for every $\lambda\ne -6$.

Recall that $S_{-6}=H_{\{x,y,z\}}+\mathbf{S}$,
where $\mathbf{S}$ is a cubic surface whose singular locus consists of the points $P_{\{z\},\{x,y\},\{x,t\}}$ and $P_{\{z\},\{x,y\},\{y,t\}}$.
Observe also that $H_{\{x,y,z\}}\cap\mathbf{S}$ consists of the line $L_{\{z\},\{x,y\}}$
and an irreducible conic $x+y+z=xy+t^2$.
Then $S_{-6}$ has good double points at $P_{\{x\},\{y\},\{z\}}$, $P_{\{x\},\{t\},\{y,z\}}$, $P_{\{y\},\{t\},\{x,z\}}$, and $P_{\{z\},\{t\},\{x,y\}}$.
Hence, using \eqref{equation:equation:number-of-irredubicle-components-refined} and Lemmas~\ref{lemma:main} and \ref{lemma:normal-crossing},
we get $[\mathsf{f}^{-1}(-6)]=4$.
This confirms \eqref{equation:main-1} in  Main Theorem.

Let us verify \eqref{equation:main-2} in Main Theorem.
We may assume that $\lambda\ne -6$. Then
$$
L_{\{x\},\{y,z\}}+L_{\{y\},\{x,z\}}+L_{\{z\},\{x,y\}}\sim L_{\{x\},\{t\}}+L_{\{y\},\{t\}}+L_{\{z\},\{t\}}
$$
on the surface $S_\lambda$. This follows from \eqref{equation:r3-n2-base-locus} and the fact that
$$
H_{\{x,y,z\}}\cdot S_\lambda=L_{\{x\},\{y,z\}}+L_{\{y\},\{x,z\}}+L_{\{z\},\{x,y\}}+L_{\{t\},\{x,y,z\}}.
$$
Using this, we can compute the intersection form of the curves
$L_{\{x\},\{t\}}$, $L_{\{x\},\{y,z\}}$, $L_{\{y\},\{t\}}$, $L_{\{y\},\{x,z\}}$,
$L_{\{z\},\{t\}}$, and $H_{\lambda}$ on the surface $S_\lambda$.
Namely, it is given by the following matrix:
\begin{center}\renewcommand\arraystretch{1.42}
\begin{tabular}{|c||c|c|c|c|c|c|}
\hline
$\bullet$ & $L_{\{x\},\{t\}}$ & $L_{\{x\},\{y,z\}}$ & $L_{\{y\},\{t\}}$ & $L_{\{y\},\{x,z\}}$ & $L_{\{z\},\{t\}}$ & $H_{\lambda}$\\
\hline
\hline
$L_{\{x\},\{t\}}$ &$-\frac{4}{3}$ & $\frac{2}{3}$ & $1$ & $0$ & $1$ & $1$\\
\hline
$L_{\{x\},\{y,z\}}$&$\frac{2}{3}$ & $-\frac{1}{2}$ & $0$ & $\frac{5}{6}$ & $0$ & $1$\\
\hline
$L_{\{y\},\{t\}}$&$1$ & $0$ & $-\frac{4}{3}$ & $\frac{2}{3}$ & $1$ & $1$\\
\hline
$L_{\{y\},\{x,z\}}$&$0$ & $\frac{5}{6}$ & $\frac{2}{3}$ & $-\frac{1}{2}$ & $0$ & $1$\\
\hline
$L_{\{z\},\{t\}}$&$1$ & $0$ & $1$ & $0$ & $-\frac{5}{4}$ & $1$\\
\hline
$H_{\lambda}$ & $1$  & $1$  & $1$  & $1$ & $1$ & $4$ \\
\hline
\end{tabular}
\end{center}
The rank of this matrix is $5$.
Thus, if $\lambda\ne -6$, then it follows from \eqref{equation:r3-n2-base-locus} that
the intersection matrix of the curves $L_{\{x\},\{t\}}$, $L_{\{y\},\{t\}}$, $L_{\{z\},\{t\}}$,
$L_{\{x\},\{y,z\}}$, $L_{\{y\},\{x,z\}}$, $L_{\{z\},\{x,y\}}$, $L_{\{t\},\{x,y,z\}}$, $\mathcal{C}_1$, and $\mathcal{C}_2$
on the surface $S_\lambda$ also has rank $5$.
But
$\mathrm{rk}\,\mathrm{Pic}(\widetilde{S}_{\Bbbk})=\mathrm{rk}\,\mathrm{Pic}(S_{\Bbbk})+12$.
Hence, we see that \eqref{equation:main-2-simple} holds,
so that \eqref{equation:main-2} in Main Theorem also holds by Lemma~\ref{lemma:cokernel}.

\subsection{Family \textnumero $3.3$}
\label{section:r-3-n-3}

The threefold $X$ is a divisor of tridegree $(1,1,2)$ on $\mathbb{P}^1\times\mathbb{P}^1\times\mathbb{P}^1$, which implies that $h^{1,2}(X)=3$.
Its toric Landau--Ginzburg model is given by Minkowski polynomial $1804$.
Using the coordinate change $x\mapsto\frac{y}{x}$ and  $z\mapsto\frac{z}{x}$, we can rewrite it as
$$
\frac{yz}{x}+\frac{y}{x}+y+\frac{z}{x}+z+\frac{2}{x}+2x+\frac{1}{z}+\frac{2x}{z}+\frac{x^2}{z}+\frac{1}{xy}+\frac{2}{y}+\frac{x}{y}.
$$
The quartic pencil $\mathcal{S}$ is given by the equation
\begin{multline*}
t^3z+t^2xy+2t^2xz+2t^2yz+2tx^2y+tx^2z+ty^2z+tyz^2+\\
+x^3y+2x^2yz+xy^2z+xyz^2+y^2z^2=\lambda xyzt.
\end{multline*}
This equation is symmetric with respect to the involution $[x:y:z:t]\leftrightarrow[t:z:y:x]$.

Suppose that $\lambda\ne\infty$.
Let $\mathcal{C}_1$ be a smooth conic that is given by $x=yz+yt+t^2=0$,
and let $\mathcal{C}_2$ is a smooth conic that is given by $t=x^2+xz+yz=0$. Then
\begin{equation}
\label{equation:r3-n3-base-locus}
\begin{split}
H_{\{x\}}\cdot S_\lambda&=L_{\{x\},\{z\}}+L_{\{x\},\{y,t\}}+\mathcal{C}_1,\\
H_{\{y\}}\cdot S_\lambda&=L_{\{y\},\{t\}}+L_{\{y\},\{z\}}+2L_{\{y\},\{x,t\}},\\
H_{\{z\}}\cdot S_\lambda&=L_{\{x\},\{z\}}+L_{\{y\},\{z\}}+2L_{\{z\},\{x,t\}},\\
H_{\{t\}}\cdot S_\lambda&=L_{\{y\},\{t\}}+L_{\{t\},\{x,z\}}+\mathcal{C}_2.
\end{split}
\end{equation}
Hence, we conclude that $L_{\{x\},\{z\}}$, $L_{\{y\},\{z\}}$, $L_{\{y\},\{t\}}$, $L_{\{x\},\{y,t\}}$, $L_{\{y\},\{x,t\}}$,
$L_{\{z\},\{x,t\}}$, $L_{\{t\},\{x,z\}}$, $\mathcal{C}_1$, and $\mathcal{C}_2$ are all base curves of the pencil $\mathcal{S}$.

The surface $S_{-4}$ is irreducible. However, its singularities are not isolated:
it is singular along the lines $L_{\{z\},\{x,t\}}$ and $L_{\{t\},\{x,z\}}$, and smooth away from them.

If $\lambda\ne -4$, then $S_\lambda$ has isolated singularities, so that $S_\lambda$ is irreducible.
In this case, the singularities of the surface $S_\lambda$
that are contained in the base locus of the pencil $\mathcal{S}$ are all du Val and can be described as follows:
\begin{itemize}\setlength{\itemindent}{3cm}
\item[$P_{\{x\},\{y\},\{t\}}$:] type $\mathbb{A}_4$ with quadratic term $y(y+z+t)$;
\item[$P_{\{x\},\{z\},\{t\}}$:] type $\mathbb{A}_4$ with quadratic term $z(x+z+t)$;
\item[$P_{\{y\},\{z\},\{x,t\}}$:] type $\mathbb{A}_3$ with quadratic term $(\lambda+4)yz$.
\end{itemize}
Thus, it follows from Corollary~\ref{corollary:irreducible-fibers} that $[\mathsf{f}^{-1}(\lambda)]=1$ for every $\lambda\ne -4$.

\begin{lemma}
\label{lemma:r3-n3-main-1}
One has $[\mathsf{f}^{-1}(-4)]=4$.
\end{lemma}

\begin{proof}
Let $C_1=\mathcal{C}_1$, $C_2=\mathcal{C}_2$,
$C_3=L_{\{x\},\{z\}}$, $C_4=L_{\{y\},\{z\}}$, $C_5=L_{\{y\},\{t\}}$, $C_6=L_{\{x\},\{y,t\}}$, $C_7=L_{\{y\},\{x,t\}}$,
$C_8=L_{\{z\},\{x,t\}}$, and $C_9=L_{\{t\},\{x,z\}}$.
Then
$\mathbf{m}_1=\mathbf{m}_2=\mathbf{m}_6=\mathbf{m}_9=1$
and $\mathbf{m}_3=\mathbf{m}_4=\mathbf{m}_5=\mathbf{m}_7=\mathbf{m}_8=2$.
Likewise, we have $\mathbf{M}^{-4}_7=\mathbf{M}^{-4}_8=2$ and
$$
\mathbf{M}^{-4}_1=\mathbf{M}^{-4}_2=\mathbf{M}^{-4}_3=\mathbf{M}^{-4}_4=\mathbf{M}^{-4}_5=\mathbf{M}^{-4}_4=\mathbf{M}^{-4}_9=1
$$
Thus, it follows from \eqref{equation:equation:number-of-irredubicle-components-refined} and Lemma~\ref{lemma:main} that
$$
\big[\mathsf{f}^{-1}(-4)\big]=3+\mathbf{D}_{P_{\{x\},\{z\},\{t\}}}^{-4}+\mathbf{D}_{P_{\{x\},\{y\},\{t\}}}^{-4}+\mathbf{D}_{P_{\{y\},\{z\},\{x,t\}}}^{-4}.
$$
The surface $S_{-4}$ has (non-isolated) ordinary double singularities at the points $P_{\{x\},\{z\},\{t\}}$ and $P_{\{x\},\{y\},\{t\}}$.
Thus, it follows from Lemma~\ref{lemma:normal-crossing} that $\mathbf{D}_{P_{\{x\},\{z\},\{t\}}}^{-4}=\mathbf{D}_{P_{\{x\},\{y\},\{t\}}}^{-4}=0$.

Let $\alpha_1\colon U_1\to\mathbb{P}^3$ be a blow up of the point $P_{\{y\},\{z\},\{x,t\}}$.
Then $D_{-4}^1=S_{-4}^1+\mathbf{E}_1$.
The~surface $\mathbf{E}_1$ contains two base curves of the pencil $\mathcal{S}^1$.
Denote them by $C_{10}^1$ and $C_{11}^1$.
Then $\mathbf{M}_{10}^{-4}=\mathbf{M}_{11}^{-4}=1$, $\mathbf{A}_{P_{\{y\},\{z\},\{x,t\}}}=1$,
and the only base curves of the pencil $\widehat{\mathcal{S}}$ that
are mapped to the singular point $P_{\{y\},\{z\},\{x,t\}}$ are the curves $\widehat{C}_{10}$ and $\widehat{C}_{11}$.
Then
$$
\mathbf{D}_{P_{\{y\},\{z\},\{x,t\}}}=1
$$
by \eqref{equation:D-A-B} and Lemma~\ref{lemma:main-2}. We conclude that $[\mathsf{f}^{-1}(-4)]=4$.
\end{proof}

Recall that $h^{1,2}(X)=3$. Since the fiber $\mathsf{f}^{-1}(\lambda)$ is irreducible for every $\lambda\ne -4$,
we see that \eqref{equation:main-1} in Main Theorem follows from Lemma~\ref{lemma:r3-n3-main-1}.

Now let us prove \eqref{equation:main-2} in Main Theorem.
We may assume that $\lambda\ne -4$. Then the intersection form of the curves
$L_{\{x\},\{z\}}$, $L_{\{y\},\{z\}}$, $L_{\{y\},\{t\}}$, $L_{\{x\},\{t,z\}}$,
$L_{\{t\},\{x,z\}}$, and $H_{\lambda}$ on the surface $S_\lambda$ is given by
\begin{center}\renewcommand\arraystretch{1.42}
\begin{tabular}{|c||c|c|c|c|c|c|}
\hline
$\bullet$ & $L_{\{x\},\{z\}}$ & $L_{\{y\},\{z\}}$ & $L_{\{y\},\{t\}}$ & $L_{\{x\},\{t,z\}}$ & $L_{\{t\},\{x,z\}}$ & $H_{\lambda}$\\
\hline\hline
$L_{\{x\},\{z\}}$ &$-\frac{6}{5}$ & $1$ & $0$ & $1$ & $\frac{1}{5}$ & $1$\\
\hline
$L_{\{y\},\{z\}}$& $1$ & $-1$ & $1$ & $0$ & $0$ & $1$\\
\hline
$L_{\{y\},\{t\}}$& $0$ & $1$ & $-\frac{6}{5}$ & $\frac{1}{5}$ & $1$ & $1$\\
\hline
$L_{\{x\},\{t,z\}}$& $1$ & $0$ & $\frac{1}{5}$ & $-\frac{6}{5}$ & $0$ & $1$\\
\hline
$L_{\{t\},\{x,z\}}$& $\frac{1}{5}$ & $0$ & $1$ & $0$ & $-\frac{6}{5}$ & $1$\\
\hline
$H_{\lambda}$ & $1$  & $1$  & $1$  & $1$ & $1$ & $4$ \\
\hline
\end{tabular}
\end{center}
The determinant of this matrix is $-\frac{112}{25}$.
Thus, it follows from \eqref{equation:r3-n3-base-locus}
that the intersection matrix of the curves $L_{\{x\},\{z\}}$, $L_{\{y\},\{z\}}$, $L_{\{y\},\{t\}}$, $L_{\{x\},\{y,t\}}$, $L_{\{y\},\{x,t\}}$,
$L_{\{z\},\{x,t\}}$, $L_{\{t\},\{x,z\}}$, $\mathcal{C}_1$, and $\mathcal{C}_2$ on the surface $S_\lambda$ also has rank $6$.
But $\mathrm{rk}\,\mathrm{Pic}(\widetilde{S}_{\Bbbk})=\mathrm{rk}\,\mathrm{Pic}(S_{\Bbbk})+11$.
Summarizing, we see that \eqref{equation:main-2-simple} holds,
so that \eqref{equation:main-2} in Main Theorem also holds by Lemma~\ref{lemma:cokernel}.

\subsection{Family \textnumero $3.4$}
\label{section:r-3-n-4}

The threefold $X$ is a blow up of a double cover of $\mathbb{P}^1\times\mathbb{P}^2$
in a divisor of bidegree $(2,2)$ along a smooth fiber of the projection to $\mathbb{P}^2$.
One has $h^{1,2}(X)=2$.
The toric Landau--Ginzburg model is given by the Minkowski polynomial \textnumero $1724$, which is
$$
y+\frac{yz}{x}+\frac{2y}{x}+x+\frac{y}{xz}+2z+\frac{z}{x}+\frac{2}{z}+\frac{xz}{y}+\frac{2}{x}+\frac{2x}{y}+\frac{1}{xz}+\frac{x}{yz}.
$$
The pencil $\mathcal{S}$ is given by
\begin{multline*}
y^2xz+y^2z^2+2y^2tz+x^2yz+t^2y^2+2xyz^2+z^2yt+2t^2xy+\\
+x^2z^2+2t^2yz+2x^2tz+t^3y+x^2t^2=\lambda xyzt.
\end{multline*}
As usual, we will assume that $\lambda\ne\infty$.

Let $\mathcal{C}_1$ be a conic that is given by $z=x^2+2xy+y^2+yt=0$,
and let $\mathcal{C}_2$ be a conic that is given by $t=xy+xz+yz=0$.
Then
\begin{equation}
\label{equation:r3-n4-base-locus}
\begin{split}
H_{\{x\}}\cdot S_\lambda&=L_{\{x\},\{y\}}+2L_{\{x\},\{z,t\}}+L_{\{x\},\{y,t\}},\\
H_{\{y\}}\cdot S_\lambda&=2L_{\{x\},\{y\}}+2L_{\{y\},\{z,t\}},\\
H_{\{z\}}\cdot S_\lambda&=2L_{\{z\},\{t\}}+\mathcal{C}_1,\\
H_{\{t\}}\cdot S_\lambda&=L_{\{z\},\{t\}}+L_{\{t\},\{x,y\}}+\mathcal{C}_2.
\end{split}
\end{equation}
This shows that $L_{\{x\},\{y\}}$, $L_{\{z\},\{t\}}$, $L_{\{x\},\{y,t\}}$, $L_{\{x\},\{z,t\}}$, $L_{\{y\},\{z,t\}}$, $L_{\{t\},\{x,y\}}$, $\mathcal{C}_1$, and $\mathcal{C}_2$
are all base curves of the pencil $\mathcal{S}$.

For every $\lambda\in\mathbb{C}$, the surface $S_\lambda$ has isolated singularities, so that $S_\lambda$ is irreducible.
Moreover, if $\lambda\ne-4$, then the singularities of the surface $S_\lambda$
that are contained in the base locus of the pencil $\mathcal{S}$ are all du Val and can be described as follows:
\begin{itemize}\setlength{\itemindent}{3cm}
\item[$P_{\{x\},\{y\},\{t\}}$:] type $\mathbb{A}_1$ with quadratic term $x^2+2xy+y^2+yt$;

\item[$P_{\{x\},\{z\},\{t\}}$:] type $\mathbb{A}_1$ with quadratic term $xz+z^2+2zt+t^2$;

\item[$P_{\{y\},\{z\},\{t\}}$:] type $\mathbb{A}_1$ with quadratic term $z^2+yz+2zt+t^2$;

\item[$P_{\{x\},\{y\},\{z,t\}}$:] type $\mathbb{A}_5$ with quadratic term $(\lambda+4)xy$;

\item[$P_{\{z\},\{t\},\{x,y\}}$:] type $\mathbb{A}_2$ with quadratic term $z(x+y-(\lambda+4)t)$;

\item[{$[\lambda+4:0:-1:1]$}:] type $\mathbb{A}_1$;

\item[{$[0:\lambda+4:-1:1]$}:] type $\mathbb{A}_1$ for $\lambda\ne -5$, type $\mathbb{A}_3$ for $\lambda=-5$.
\end{itemize}
Therefore, it follows from Corollary~\ref{corollary:irreducible-fibers} that $[\mathsf{f}^{-1}(\lambda)]=1$ for every $\lambda\ne -4$.
Thus, the assertion \eqref{equation:main-1} in Main Theorem follows from

\begin{lemma}
\label{lemma:r3-n4-main-1}
One has $[\mathsf{f}^{-1}(-4)]=3$.
\end{lemma}

\begin{proof}
The surface $S_{-4}$ has isolated ordinary double singularities at the points $P_{\{x\},\{y\},\{t\}}$, $P_{\{x\},\{z\},\{t\}}$, and $P_{\{y\},\{z\},\{t\}}$,
and it has du Val singularity of type $\mathbb{A}_2$ at the point $P_{\{z\},\{t\},\{x,y\}}$.
Thus, using \eqref{equation:equation:number-of-irredubicle-components-refined},
we see that
$$
\big[\mathsf{f}^{-1}(-4)\big]=1+\mathbf{D}_{P_{\{x\},\{y\},\{z,t\}}}^{-4}
$$
by Lemmas~\ref{lemma:main} and \ref{lemma:normal-crossing}.
To compute $\mathbf{D}_{P_{\{x\},\{y\},\{z,t\}}}^{-4}$, we have to (partially) describe the birational morphism
$\alpha$ in \eqref{equation:main-diagram}.

In the chart $t=1$, the surface $S_{-4}$ is given by
$$
\bar{y}\bar{z}^2-\bar{x}^2\bar{y}-\bar{x}\bar{y}^2+\Big(\bar{x}^2\bar{y}\bar{z}+\bar{x}^2\bar{z}^2+\bar{x}\bar{y}^2\bar{z}+2\bar{x}\bar{y}\bar{z}^2+\bar{y}^2\bar{z}^2\Big)=0,
$$
where $\bar{x}=x$, $\bar{y}=y$, and $\bar{z}=z+1$.
In particular, the singularity of the surface $S_{-4}$ at the point $P_{\{x\},\{y\},\{z,t\}}$ is not du Val.
Since $P_{\{x\},\{y\},\{z,t\}}$ is a singular point of the surface $S_{-4}$ of multiplicity $3$,
we can use \eqref{equation:equation:number-of-irredubicle-components-refined} to conclude that $\mathbf{D}_{P_{\{x\},\{y\},\{z,t\}}}^{-4}>0$.

Let $\alpha_1\colon U_1\to\mathbb{P}^3$ be the blow up of the point $P_{\{x\},\{y\},\{z,t\}}$.
A local chart of this blow up is given by the coordinate change $\bar{x}_1=\frac{\bar{x}}{\bar{z}}$, $\bar{y}_1=\frac{\bar{y}}{\bar{z}}$, and $\bar{z}_1=\bar{z}$.
In this chart, the surface~$\mathbf{E}_1$ is given by $\bar{z}_1=0$,
and $D_{\lambda}^1$ is given by
\begin{multline*}
(\lambda+4)\bar{x}_1\bar{y}_1+\bar{y}_1\bar{z}_1-(\lambda+4)\bar{x}_1\bar{y}_1\bar{z}_1v+\\
+\Big(-\bar{x}_1^2\bar{y}_1\bar{z}_1+\bar{x}_1^2\bar{z}_1^2-\bar{x}_1\bar{y}_1^2\bar{z}_1+2\bar{x}_1\bar{y}_1\bar{z}_1^2+\bar{y}_1^2\bar{z}_1^2\Big)+\Big(\bar{x}_1^2\bar{y}_1\bar{z}_1^2+\bar{y}_1^2\bar{x}_1\bar{z}_1^2\Big)=0.
\end{multline*}
Thus, we see that $D_{-4}^1=S_{-4}^1+\mathbf{E}_1$.

The surface $\mathbf{E}_1$ contains two base curves of the pencil $\mathcal{S}^1$.
One of them is given by $\bar{z}_1=\bar{x}_1=0$, and another one is given by $\bar{z}_1=\bar{y}_1=0$.
Denote the former curve by $C_{9}^1$, and denote the latter curve by $C_{10}^1$.
Then $\mathbf{M}_{9}^{-4}=1$ and $\mathbf{M}_{10}^{-4}=2$.
Hence, using \eqref{equation:D-A-B} and Lemma~\ref{lemma:main-2},
we conclude that $\mathbf{D}_{P_{\{x\},\{y\},\{z,t\}}}^{-4}\geqslant 2$.

Let $\alpha_2\colon U_2\to U_1$ be the blow up of the point $C_{9}^1\cap C_{10}^1$.
Then $D_{-4}^2=S_{-4}^2+\mathbf{E}_1^2$.
Moreover, the surface $\mathbf{E}_2$ contains two base curves of the pencil $\mathcal{S}^2$.
Denote them by $C_{11}^2$ and $C_{12}^2$, respectively.
Then $\mathbf{M}_{11}^{-4}=\mathbf{M}_{12}^{-4}=1$.
Moreover, one can show that the only base curves of the pencil $\widehat{\mathcal{S}}$ that
are mapped to $P_{\{x\},\{y\},\{z,t\}}$ are the curves $\widehat{C}_{9}$, $\widehat{C}_{10}$, $\widehat{C}_{11}$, and $\widehat{C}_{12}$.
Finally, local computations imply that $\mathbf{A}_{P_{\{x\},\{y\},\{z,t\}}}^{-4}=1$.
Thus, using \eqref{equation:D-A-B} and Lemma~\ref{lemma:main-2} we get $\mathbf{D}_{P_{\{x\},\{y\},\{z,t\}}}^{-4}=2$,
so that $[\mathsf{f}^{-1}(-4)]=3$.
\end{proof}

To prove \eqref{equation:main-2} in Main Theorem, we need the following result.

\begin{lemma}
\label{lemma:r3-n4-intersection}
Suppose that $\lambda\ne -4$ and $\lambda\ne -5$.
Then the intersection form of the curves
$L_{\{x\},\{y\}}$, $L_{\{x\},\{y,t\}}$, $L_{\{z\},\{t\}}$, $L_{\{t\},\{x,y\}}$, and $H_{\lambda}$
on the surface $S_\lambda$ is given by
\begin{center}\renewcommand\arraystretch{1.42}
\begin{tabular}{|c||c|c|c|c|c|}
\hline
$\bullet$ & $L_{\{x\},\{y\}}$ & $L_{\{x\},\{y,t\}}$ & $L_{\{z\},\{t\}}$ & $L_{\{t\},\{x,y\}}$ & $H_{\lambda}$\\
\hline\hline
$L_{\{x\},\{y\}}$ &$-\frac{1}{6}$ & $\frac{1}{2}$ & $0$ & $\frac{1}{2}$ & $1$\\
\hline
$L_{\{x\},\{y,t\}}$& $\frac{1}{2}$ & $-\frac{3}{2}$ & $0$ & $\frac{1}{2}$ & $1$\\
\hline
$L_{\{z\},\{t\}}$& $0$ & $0$ & $-\frac{5}{6}$ & $\frac{1}{3}$ & $1$\\
\hline
$L_{\{t\},\{x,y\}}$& $\frac{1}{2}$ & $\frac{1}{2}$ & $\frac{1}{3}$ & $-\frac{5}{6}$ & $1$\\
\hline
$H_{\lambda}$ & $1$  & $1$  & $1$ & $1$ & $4$ \\
\hline
\end{tabular}
\end{center}
\end{lemma}

\begin{proof}
First, let us compute non-diagonal entries.
Since $L_{\{x\},\{y\}}\cap L_{\{x\},\{y,t\}}=P_{\{x\},\{y\},\{t\}}$ and $P_{\{x\},\{y\},\{t\}}$
is an ordinary double point of the surface $S_\lambda$, we get
$L_{\{x\},\{y\}}\cdot L_{\{x\},\{y,t\}}=\frac{1}{2}$ by Proposition~\ref{proposition:du-Val-intersection}. Likewise, we have $L_{\{x\},\{y\}}\cdot L_{\{t\},\{x,y\}}=L_{\{x\},\{y,t\}}\cdot L_{\{t\},\{x,y\}}=\frac{1}{2}$.

Since $L_{\{x\},\{y\}}\cap L_{\{z\},\{t\}}=L_{\{x\},\{y,t\}}\cap L_{\{z\},\{t\}}=\varnothing$, we have
$$
L_{\{x\},\{y\}}\cdot L_{\{z\},\{t\}}=L_{\{x\},\{y,t\}}\cdot L_{\{z\},\{t\}}=0.
$$

To compute $L_{\{z\},\{t\}}\cdot L_{\{t\},\{x,y\}}$, observe that $L_{\{z\},\{t\}}\cap L_{\{t\},\{x,y\}}=P_{\{z\},\{t\},\{x,y\}}$.
Moreover, the surface $S_\lambda$ has du Val singularity of type $\mathbb{A}_2$ at the point $P_{\{z\},\{t\},\{x,y\}}$.
Furthermore, the quadratic term of its defining equation at this point is $z(x+y-(\lambda+4)t)$.
Thus, using Remark~\ref{remark:transversal} with $S=S_\lambda$, $O=P_{\{z\},\{t\},\{x,y\}}$, $n=3$, $C=L_{\{z\},\{t\}}$, and $Z=L_{\{t\},\{x,y\}}$,
we see that $\widetilde{C}$ and $\widetilde{Z}$ intersect different curves among $G_1$ and $G_2$.
Then $L_{\{z\},\{t\}}\cdot L_{\{t\},\{x,y\}}=\frac{1}{3}$ by Proposition~\ref{proposition:du-Val-intersection}.

Now let us compute the diagonal entries.
Since $P_{\{x\},\{y\},\{t\}}$ and $P_{\{z\},\{t\},\{x,y\}}$
are the only singular points of the surface $S_\lambda$ that are contained in the line $L_{\{t\},\{x,y\}}$,
we see that
$$
L_{\{z\},\{t\}}^2=-2+\frac{1}{2}+\frac{2}{3}=-\frac{5}{6}
$$
by Proposition~\ref{proposition:du-Val-self-intersection}. Likewise, we have $L_{\{t\},\{x,y\}}^2=-\frac{5}{6}$.
We also have \mbox{$L_{\{x\},\{y,t\}}^{2}=-\frac{3}{2}$}, because
$P_{\{x\},\{y\},\{t\}}$ is the only singular point of the surface $S_\lambda$ that is contained in $L_{\{x\},\{y,t\}}$.

To compute  $L_{\{x\},\{y\}}^{2}$, let us use the notation of the proof of Lemma~\ref{lemma:r3-n4-main-1}.
Note that the proper transform of the line $L_{\{x\},\{y\}}$ on the surface  $S_{\lambda}^1$ passes through the point $C_{9}^1\cap C_{10}^1$.
On the other hand, its proper transform on the surface  $S_{\lambda}^2$ does not pass through
the intersection $C_{11}^2$ and $C_{12}^2$.
Applying Remark~\ref{remark:transversal} with $S=S_\lambda$, $O=P_{\{x\},\{y\},\{z,t\}}$, $n=5$, and $C=L_{\{x\},\{y\}}$,
we see that $\widetilde{C}$ intersects either the curve $G_2$ or the curve $G_4$.
Thus, it follows from Proposition~\ref{proposition:du-Val-self-intersection} that
$L_{\{x\},\{y\}}^2=-\frac{1}{6}$,
because $P_{\{x\},\{y\},\{z,t\}}$ and $P_{\{x\},\{y\},\{t\}}$ are the only singular points of the surface $S_\lambda$ contained in the line $L_{\{x\},\{y\}}$.
\end{proof}

The determinant of the matrix in Lemma~\ref{lemma:r3-n4-intersection} is $-\frac{16}{9}$.
Thus, if $\lambda\ne -4$ and $\lambda\ne -5$, then it follows from \eqref{equation:r3-n4-base-locus} that the intersection matrix of the curves
$L_{\{x\},\{y\}}$, $L_{\{z\},\{t\}}$, $L_{\{x\},\{y,t\}}$, $L_{\{x\},\{z,t\}}$, $L_{\{y\},\{z,t\}}$, $L_{\{t\},\{x,y\}}$, $\mathcal{C}_1$, and $\mathcal{C}_2$
also has rank $5$.
On the other hand, one can easily see that $\mathrm{rk}\,\mathrm{Pic}(\widetilde{S}_{\Bbbk})=\mathrm{rk}\,\mathrm{Pic}(S_{\Bbbk})+12$.
Hence, we conclude that \eqref{equation:main-2-simple} holds,
so that \eqref{equation:main-2} in Main Theorem also holds by Lemma~\ref{lemma:cokernel}.

\subsection{Family \textnumero $3.5$}
\label{section:r-3-n-5}

The threefold $X$ can be obtained by blowing up $\mathbb{P}^1\times\mathbb{P}^2$  along a smooth rational curve of bidegree $(5,2)$.
Then $h^{1,2}(X)=0$.
A~toric Landau--Ginzburg model of this family is given by the Minkowski polynomial \textnumero $1819$, which is
$$
\frac{1}{x}+\frac{1}{y}+\frac{1}{z}+\frac{2y}{x}+\frac{2x}{y}+\frac{y}{z}+\frac{x}{z}+\frac{yz}{x}+z+\frac{y^2}{x}+3y+3x+\frac{x^2}{y}.
$$
The quartic pencil $\mathcal{S}$ is given by
\begin{multline*}
t^2xy+t^2xz+t^2yz+tx^2y+2tx^2z+txy^2+2ty^2z+x^3z+\\
+3x^2yz+3xy^2z+xyz^2+y^3z+y^2z^2=\lambda xyzt.
\end{multline*}

Suppose that $\lambda\ne\infty$. Then $S_\lambda$ has isolated singularities, so that it is irreducible.

Let $\mathcal{C}_1$ be the conic in $\mathbb{P}^3$ that is given by $x=(y+t)^2+yz=0$,
and let $\mathcal{C}_2$ be the conic that is given by $t=(x+y)^2+yz=0$.
Then
\begin{itemize}
\item $H_{\{x\}}\cdot S_\lambda=L_{\{x\},\{y\}}+L_{\{x\},\{z\}}+\mathcal{C}_1$;
\item $H_{\{y\}}\cdot S_\lambda=L_{\{x\},\{y\}}+L_{\{y\},\{z\}}+2L_{\{y\},\{x,t\}}$;
\item $H_{\{z\}}\cdot S_\lambda=L_{\{x\},\{z\}}+L_{\{y\},\{z\}}+L_{\{z\},\{t\}}+L_{\{z\},\{x,y,t\}}$;
\item $H_{\{t\}}\cdot S_\lambda=L_{\{z\},\{t\}}+L_{\{t\},\{x,y\}}+\mathcal{C}_2$.
\end{itemize}
Therefore, the base locus of the pencil $\mathcal{S}$ consists of the curves
$L_{\{x\},\{y\}}$, $L_{\{x\},\{z\}}$, $L_{\{y\},\{z\}}$, $L_{\{z\},\{t\}}$, $L_{\{y\},\{x,t\}}$,
$L_{\{t\},\{x,y\}}$, $L_{\{z\},\{x,y,t\}}$, $\mathcal{C}_1$, and $\mathcal{C}_2$.

For every $\lambda\in\mathbb{C}$, the singular points of the surface $S_\lambda$ contained in the base locus of the pencil $\mathcal{S}$ can be described as follows:
\begin{itemize}\setlength{\itemindent}{3cm}
\item[$P_{\{x\},\{y\},\{z\}}$:] type $\mathbb{A}_1$;

\item[$P_{\{x\},\{y\},\{t\}}$:] type $\mathbb{A}_4$ for $\lambda\ne -4$, $\mathbb{A}_5$ for $\lambda=-4$;

\item[{$P_{\{x\},\{z\},\{y,t\}}$}:] type $\mathbb{A}_1$;

\item[{$P_{\{y\},\{z\},\{x,t\}}$}:] type $\mathbb{A}_2$ for $\lambda\ne -4$, $\mathbb{A}_4$ for $\lambda=-4$;

\item[{$P_{\{z\},\{t\},\{x,y\}}$}:] type $\mathbb{A}_2$ for $\lambda\ne -4$, $\mathbb{A}_3$ for $\lambda=-4$;

\item[{$[1:0:\lambda+4:-1]$}:] type $\mathbb{A}_1$ for $\lambda\ne -4$.
\end{itemize}
In particular, it follows from Corollary~\ref{corollary:irreducible-fibers} that $\mathsf{f}^{-1}(\lambda)$ is irreducible for every $\lambda\ne\infty$.
Thus, since $h^{1,2}(X)=0$, we see that \eqref{equation:main-1} in Main Theorem holds in this case.

\begin{lemma}
\label{lemma:r3-n5-intersection}
Suppose that $\lambda\ne -4$.
Then the intersection matrix of the curves $L_{\{x\},\{y\}}$, $L_{\{x\},\{z\}}$, $L_{\{y\},\{z\}}$, $L_{\{z\},\{t\}}$, $L_{\{y\},\{x,t\}}$, $L_{\{t\},\{x,y\}}$, and $H_{\lambda}$
on the surface $S_\lambda$ is given by
\begin{center}\renewcommand\arraystretch{1.42}
\begin{tabular}{|c||c|c|c|c|c|c|c|}
\hline
 $\bullet$  & $L_{\{x\},\{y\}}$ & $L_{\{x\},\{z\}}$ & $L_{\{y\},\{z\}}$ & $L_{\{z\},\{t\}}$ & $L_{\{y\},\{x,t\}}$ & $L_{\{t\},\{x,y\}}$ & $H_{\lambda}$ \\
\hline\hline
 $L_{\{x\},\{y\}}$ &  $-\frac{3}{10}$ & $\frac{1}{2}$ & $\frac{1}{2}$ & $0$ & $\frac{2}{5}$ & $\frac{3}{5}$ & $1$ \\
\hline
 $L_{\{x\},\{z\}}$  & $\frac{1}{2}$& $-1$& $\frac{1}{2}$ & $1$ & $0$ & $0$ & $1$ \\
\hline
 $L_{\{y\},\{z\}}$  & $\frac{1}{2}$& $\frac{1}{2}$& $-\frac{5}{6}$& $1$& $\frac{2}{3}$& $0$ & $1$ \\
\hline
 $L_{\{z\},\{t\}}$  & $0$& $1$& $1$& $-\frac{4}{3}$ & $0$& $\frac{2}{3}$ & $1$ \\
\hline
 $L_{\{y\},\{x,t\}}$  & $\frac{2}{5}$ & $0$ & $\frac{2}{3}$ & $0$& $-\frac{1}{30}$ & $\frac{1}{5}$ & $1$ \\
\hline
 $L_{\{t\},\{x,y\}}$  & $\frac{3}{5}$ & $0$ & $0$ & $\frac{2}{3}$& $\frac{1}{5}$ & $-\frac{8}{15}$ & $1$ \\
\hline
 $H_{\lambda}$  & $1$ & $1$ & $1$ & $1$ & $1$ & $1$ & $4$ \\
\hline
\end{tabular}
\end{center}
\end{lemma}

\begin{proof}
Observe that $L_{\{x\},\{y\}}\cap L_{\{z\},\{t\}}=\varnothing$, so that $L_{\{x\},\{y\}}\cdot L_{\{z\},\{t\}}=0$.
Similarly, we see that $L_{\{x\},\{z\}}\cdot L_{\{z\},\{x,y,t\}}=0$, $L_{\{x\},\{z\}}\cdot L_{\{t\},\{x,y\}}=0$, and $L_{\{y\},\{z\}}\cdot L_{\{t\},\{x,y\}}=0$.
Since $L_{\{x\},\{z\}}\cap L_{\{z\},\{t\}}=P_{\{x\},\{z\},\{t\}}$ and $S_\lambda$ is smooth at $P_{\{x\},\{z\},\{t\}}$, we have $L_{\{x\},\{z\}}\cdot L_{\{z\},\{t\}}=1$.
Likewise, we have $L_{\{y\},\{z\}}\cdot L_{\{z\},\{t\}}=1$.

The points $P_{\{x\},\{y\},\{z\}}$ and $P_{\{x\},\{z\},\{y,t\}}$ are the only singular points of the surface $S_\lambda$ that are contained in $L_{\{x\},\{z\}}$.
Thus, we have $L_{\{x\},\{z\}}^2=-1$ by Proposition~\ref{proposition:du-Val-self-intersection}.
Similarly, we see that $L_{\{y\},\{z\}}^2=-\frac{5}{6}$,
because $P_{\{x\},\{y\},\{z\}}$ and $P_{\{y\},\{z\},\{x,t\}}$ are the only singular points of the surface $S_\lambda$ that are contained in $L_{\{y\},\{z\}}$.
Likewise, we have $L_{\{z\},\{t\}}^2=-\frac{4}{3}$,
because $P_{\{z\},\{t\},\{x,y\}}$ is the only singular point of the surface $S_\lambda$ contained in $L_{\{z\},\{t\}}$.

Since $L_{\{x\},\{y\}}\cap L_{\{x\},\{z\}}=P_{\{x\},\{y\},\{z\}}$, we have  $L_{\{x\},\{y\}}\cdot L_{\{x\},\{z\}}=\frac{1}{2}$ by Proposition~\ref{proposition:du-Val-intersection}.
Similarly, we have $L_{\{x\},\{y\}}\cdot L_{\{y\},\{z\}}=\frac{1}{2}$.

Let us show that $L_{\{y\},\{z\}}\cdot L_{\{y\},\{x,t\}}=\frac{2}{3}$.
To do this, let us use the notation of Appendix~\ref{section:intersection} with $S=S_\lambda$, $O=P_{\{y\},\{z\},\{x,t\}}$, $n=2$, $C=L_{\{y\},\{x,t\}}$, and $Z=L_{\{y\},\{z\}}$.
We may assume that $\widetilde{C}\cap E_1\ne\varnothing$.
If $\widetilde{Z}\cap E_1\ne\varnothing$, then $L_{\{y\},\{z\}}\cdot L_{\{y\},\{x,t\}}=\frac{2}{3}$ by Proposition~\ref{proposition:du-Val-intersection}.
Otherwise, we have $L_{\{y\},\{z\}}\cdot L_{\{y\},\{x,t\}}=\frac{1}{3}$.
In the  chart $t=1$, the surface $S_\lambda$ is given by
$$
\bar{y}\big(\bar{x}+\bar{y}-(\lambda+4)\bar{z}\big)+\text{higher order terms}=0,
$$
where $\bar{x}=x+1$, $\bar{y}=y$, and $\bar{z}=z$. Here $O=(0,0,0)$.
In these coordinates, the line $L_{\{y\},\{x,t\}}$ is given by $\bar{y}=\bar{x}=0$,
and the line $L_{\{y\},\{z\}}$ is given by $\bar{y}=\bar{z}=0$.
This shows that $\widetilde{Z}\cap E_1\ne\varnothing$, so that $L_{\{y\},\{z\}}\cdot L_{\{y\},\{x,t\}}=\frac{2}{3}$.

Let us compute $L_{\{y\},\{x,t\}}^2$, $L_{\{x\},\{y\}}^2$, and $L_{\{x\},\{y\}}\cdot L_{\{y\},\{x,t\}}$.
Using Remark~\ref{remark:transversal} with $S=S_\lambda$, $O=P_{\{x\},\{y\},\{t\}}$, $n=4$, $C=L_{\{y\},\{x,t\}}$, and $Z=L_{\{x\},\{y\}}$,
we see that $\overline{C}$ does not pass through the point $\overline{G}_1\cap\overline{G}_4$,
and $\overline{Z}$ passes through the point $\overline{G}_1\cap\overline{G}_4$.
Now, using Proposition~\ref{proposition:du-Val-self-intersection}, we obtain
$$
L_{\{y\},\{x,t\}}^2=-2+\frac{1}{2}+\frac{2}{3}+\frac{4}{5}=-\frac{1}{30}
$$
because $P_{\{x\},\{y\},\{t\}}$, $P_{\{y\},\{z\},\{x,t\}}$, and $[1:0:\lambda+4:-1]$
are the only singular points of the surface $S_\lambda$ that are contained in the line $L_{\{y\},\{x,t\}}$.
Similarly, we get
$$
L_{\{x\},\{y\}}^2=-2+\frac{1}{2}+\frac{6}{5}=-\frac{3}{10}.
$$
because $P_{\{x\},\{y\},\{z\}}$ and $P_{\{x\},\{y\},\{t\}}$
are the only singular points of the surface $S_\lambda$ that are contained in the line $L_{\{x\},\{y\}}$.
Moreover, using Proposition~\ref{proposition:du-Val-intersection},
we see that either $L_{\{x\},\{y\}}\cdot L_{\{y\},\{x,t\}}=\frac{2}{5}$ or $L_{\{x\},\{y\}}\cdot L_{\{y\},\{x,t\}}=\frac{3}{5}$.
In fact, we have $L_{\{x\},\{y\}}\cdot L_{\{y\},\{x,t\}}=\frac{2}{5}$, because
\begin{multline*}
1=H_{\lambda}\cdot L_{\{y\},\{x,t\}}=\Big(L_{\{x\},\{y\}}+L_{\{y\},\{z\}}+2L_{\{y\},\{x,t\}}\Big)\cdot L_{\{y\},\{x,t\}}=\\
=L_{\{x\},\{y\}}\cdot L_{\{y\},\{x,t\}}+L_{\{y\},\{z\}}\cdot L_{\{y\},\{x,t\}}+2L_{\{y\},\{x,t\}}^2=L_{\{x\},\{y\}}\cdot L_{\{y\},\{x,t\}}+\frac{3}{5},
\end{multline*}
since $H_\lambda\sim L_{\{x\},\{y\}}+L_{\{y\},\{z\}}+2L_{\{y\},\{x,t\}}$ on the surface $S_\lambda$.

To complete the proof of the lemma, we must find $L_{\{t\},\{x,y\}}\cdot L_{\{x\},\{y\}}$, $L_{\{t\},\{x,y\}}\cdot L_{\{z\},\{t\}}$, $L_{\{t\},\{x,y\}}\cdot L_{\{y\},\{x,t\}}$, and $L_{\{t\},\{x,y\}}^2$.
Observe that $P_{\{x\},\{y\},\{t\}}$ and $P_{\{z\},\{t\},\{x,y\}}$
are the only singular points of the surface $S_\lambda$ that are contained in the line $L_{\{t\},\{x,y\}}$.
Thus, since $L_{\{t\},\{x,y\}}\cap L_{\{z\},\{t\}}=P_{\{z\},\{t\},\{x,y\}}$,
we get $L_{\{t\},\{x,y\}}\cdot L_{\{z\},\{t\}}=\frac{2}{3}$ by Proposition~\ref{proposition:du-Val-intersection}.

To find the remaining entries of the intersection matrix, let us use Remark~\ref{remark:transversal} with $S=S_\lambda$, $O=P_{\{x\},\{y\},\{t\}}$, $n=4$, $C=L_{\{y\},\{x,t\}}$, and $Z=L_{\{t\},\{x,y\}}$.
As we already checked above, the curve $\overline{C}$ does not pass through the point $\overline{G}_1\cap\overline{G}_4$.
Likewise, the curve $\overline{Z}$ does not pass through this point,
so that we may assume that $\widetilde{C}\cap G_1\ne\varnothing$ and $\widetilde{Z}\cap G_4\ne\varnothing$.
Hence, we have $L_{\{t\},\{x,y\}}\cdot L_{\{y\},\{x,t\}}=\frac{1}{5}$ by Proposition~\ref{proposition:du-Val-intersection}.
Likewise, it follows from Proposition~\ref{proposition:du-Val-self-intersection} that $L_{\{t\},\{x,y\}}^2=-\frac{8}{15}$.
This gives $L_{\{x\},\{y\}}\cdot L_{\{t\},\{x,y\}}=\frac{3}{5}$, because
\begin{multline*}
1=H_\lambda\cdot L_{\{t\},\{x,y\}}=\Big(L_{\{x\},\{y\}}+L_{\{y\},\{z\}}+2L_{\{y\},\{x,t\}}\Big)\cdot L_{\{t\},\{x,y\}}=\\
=L_{\{x\},\{y\}}\cdot L_{\{t\},\{x,y\}}+L_{\{y\},\{z\}}\cdot L_{\{t\},\{x,y\}}+2L_{\{t\},\{x,y\}}^2=L_{\{x\},\{y\}}\cdot L_{\{t\},\{x,y\}}+\frac{2}{5},
\end{multline*}
since $L_{\{x\},\{y\}}+L_{\{y\},\{z\}}+2L_{\{y\},\{x,t\}}\sim H_\lambda$.
\end{proof}

The matrix in Lemma~\ref{lemma:r3-n5-intersection} has rank~$6$.
Moreover, we have
$\mathrm{rk}\,\mathrm{Pic}(\widetilde{S}_{\Bbbk})=\mathrm{rk}\,\mathrm{Pic}(S_{\Bbbk})+11$.
Hence, we see that \eqref{equation:main-2-simple} holds,
so that \eqref{equation:main-2} in Main Theorem also holds by Lemma~\ref{lemma:cokernel}.

\subsection{Family \textnumero $3.6$}
\label{section:r-3-n-6}

In this case, Main Theorem is proved in Example~\ref{example:r-3-n-6}.

\subsection{Family \textnumero $3.7$}
\label{section:r-3-n-7}

In this case, the threefold $X$ can be obtained by blowing up a hypersurface of bidegree $(1,1)$ in $\mathbb{P}^2\times\mathbb{P}^2$ along a smooth elliptic curve,
so that $h^{1,2}(X)=1$.
The toric Landau--Ginzburg model of the threefold $X$ is given by
\begin{multline*}
x+y+z+\frac{y}{z}+\frac{y}{x}+\frac{z}{y}+\frac{z}{x}+\frac{1}{z}+\frac{y}{xz}+\frac{1}{y}+\frac{2}{x}+\frac{z}{xy}+\frac{1}{xz}+\frac{1}{xy},
\end{multline*}
which is the Minkowski polynomial \textnumero $2354.2$.
The pencil $\mathcal{S}$ is given by
\begin{multline*}
x^2yz+xy^2z+xyz^2+xy^2t+y^2zt+xz^2t+yz^2t+xyt^2+y^2t^2+\\+xt^2z+2yzt^2+z^2t^2+yt^3+zt^3=\lambda xyzt.
\end{multline*}
As usual, we suppose that $\lambda\ne\infty$.

For every $\lambda\ne-3$, the surface $S_\lambda$ has isolated singularities, so that $S_\lambda$ is irreducible.
On the other hand, one has $S_{-3}=H_{\{x,t\}}+\mathsf{S}$,
where $\mathsf{S}$ is an irreducible cubic surface that is given by $xyz+yt^2+zt^2+y^2t+z^2t+2yzt+y^2z+yz^2=0$.

To describe the base locus of the pencil $\mathcal{S}$, we observe that
\begin{equation}
\label{equation:r3-n7-base-locus}
\begin{split}
H_{\{x\}}\cdot S_\lambda&=L_{\{x\},\{t\}}+L_{\{x\},\{y,z\}}+L_{\{x\},\{y,t\}}+L_{\{x\},\{z,t\}},\\
H_{\{y\}}\cdot S_\lambda&=L_{\{y\},\{z\}}+L_{\{y\},\{t\}}+L_{\{y\},\{x,t\}}+L_{\{y\},\{z,t\}},\\
H_{\{z\}}\cdot S_\lambda&=L_{\{y\},\{z\}}+L_{\{z\},\{t\}}+L_{\{z\},\{x,t\}}+L_{\{z\},\{y,t\}},\\
H_{\{t\}}\cdot S_\lambda&=L_{\{x\},\{t\}}+L_{\{y\},\{t\}}+L_{\{z\},\{t\}}+L_{\{t\},\{x,y,z\}}.
\end{split}
\end{equation}
Thus, the lines
$L_{\{x\},\{t\}}$, $L_{\{y\},\{z\}}$, $L_{\{y\},\{t\}}$, $L_{\{z\},\{t\}}$,
$L_{\{x\},\{y,z\}}$, $L_{\{x\},\{y,t\}}$, $L_{\{x\},\{z,t\}}$,
$L_{\{y\},\{x,t\}}$, $L_{\{y\},\{z,t\}}$,
$L_{\{z\},\{x,t\}}$, $L_{\{z\},\{y,t\}}$, and $L_{\{t\},\{x,y,z\}}$
are all base curves of the pencil $\mathcal{S}$.

If $\lambda\ne-2$ and $\lambda\ne-3$, then the singular points of the surface $S_\lambda$ contained in the base locus
of the pencil $\mathcal{S}$ are all du Val and can be described as follows:
\begin{itemize}\setlength{\itemindent}{3cm}
\item[$P_{\{y\},\{z\},\{t\}}$:] type $\mathbb{A}_3$ with quadratic term $yz$;
\item[$P_{\{x\},\{z\},\{t\}}$:] type $\mathbb{A}_2$ with quadratic term $(x+t)(z+t)$;
\item[$P_{\{x\},\{y\},\{t\}}$:] type $\mathbb{A}_2$ with quadratic term $(x+t)(y+t)$;
\item[$P_{\{x\},\{t\},\{y,z\}}$:] type $\mathbb{A}_1$ with quadratic term $(x+y+z)(x+t)-(\lambda+3)xt$;
\item[$P_{\{y\},\{z\},\{x,t\}}$:] type $\mathbb{A}_1$ with quadratic term $(x+t)(y+z)+(\lambda+3)yz$.
\end{itemize}
Thus, it follows from Corollary~\ref{corollary:irreducible-fibers} that $[\mathsf{f}^{-1}(\lambda)]=1$ for every $\lambda\ne -3$ and $\lambda\ne -2$.

The surface $S_{-2}$ has the same singularities at the points $P_{\{y\},\{z\},\{t\}}$, $P_{\{x\},\{z\},\{t\}}$, $P_{\{x\},\{y\},\{t\}}$, $P_{\{x\},\{t\},\{y,z\}}$, and $P_{\{y\},\{z\},\{x,t\}}$.
In addition to them, it also has isolated ordinary double singularities at the points $[0:-1:1:1]$, $[0:1:-1:1]$, and $[0:1:1:-1]$.
Thus, using Corollary~\ref{corollary:irreducible-fibers}, we conclude that $[\mathsf{f}^{-1}(-2)]=1$.

The surface $S_{-3}$ has {good} double points at $P_{\{y\},\{z\},\{t\}}$, $P_{\{x\},\{z\},\{t\}}$, $P_{\{x\},\{y\},\{t\}}$, $P_{\{x\},\{t\},\{y,z\}}$, $P_{\{y\},\{z\},\{x,t\}}$,
and it is smooth at general points of the lines $L_{\{x\},\{t\}}$, $L_{\{y\},\{z\}}$, $L_{\{y\},\{t\}}$, $L_{\{z\},\{t\}}$,
$L_{\{x\},\{y,z\}}$, $L_{\{x\},\{y,t\}}$, $L_{\{x\},\{z,t\}}$,
$L_{\{y\},\{x,t\}}$, $L_{\{y\},\{z,t\}}$,
$L_{\{z\},\{x,t\}}$, $L_{\{z\},\{y,t\}}$, $L_{\{t\},\{x,y,z\}}$.
Thus, it follows from \eqref{equation:equation:number-of-irredubicle-components-refined} and Lemmas~\ref{lemma:main} and \ref{lemma:normal-crossing} that
$[\mathsf{f}^{-1}(-3)]=[S_{-3}]=2$.
Hence, we see that \eqref{equation:main-1} in Main Theorem holds in this case, because $h^{1,2}(X)=1$.

To prove \eqref{equation:main-2} in Main Theorem, we may assume that $\lambda\ne-2$ and $\lambda\ne-3$.
Then
$$
H_{\{x,t\}}\cdot S_\lambda=2L_{\{x\},\{t\}}+L_{\{y\},\{x,t\}}+L_{\{z\},\{x,t\}},
$$
so that $2L_{\{x\},\{t\}}+L_{\{y\},\{x,t\}}+L_{\{z\},\{x,t\}}\sim H_\lambda$ on the surface $S_\lambda$.
It follows from \eqref{equation:r3-n7-base-locus} that
the intersection matrix of the lines $L_{\{x\},\{t\}}$, $L_{\{y\},\{z\}}$, $L_{\{y\},\{t\}}$, $L_{\{z\},\{t\}}$,
$L_{\{x\},\{y,z\}}$, $L_{\{x\},\{y,t\}}$, $L_{\{x\},\{z,t\}}$,
$L_{\{y\},\{x,t\}}$, $L_{\{y\},\{z,t\}}$,
$L_{\{z\},\{x,t\}}$, $L_{\{z\},\{y,t\}}$, and $L_{\{t\},\{x,y,z\}}$
on the surface $S_\lambda$ has the same rank as the intersection matrix of the curves
$L_{\{x\},\{y,z\}}$, $L_{\{x\},\{y,t\}}$, $L_{\{x\},\{z,t\}}$, $L_{\{y\},\{x,t\}}$, $L_{\{y\},\{z,t\}}$, $L_{\{z\},\{y,t\}}$, $L_{\{t\},\{x,y,z\}}$, and $H_{\lambda}$.
The latter matrix is given by
\begin{center}\renewcommand\arraystretch{1.42}
\begin{tabular}{|c||c|c|c|c|c|c|c|c|}
\hline
 $\bullet$  & $L_{\{x\},\{y,z\}}$ & $L_{\{x\},\{y,t\}}$ & $L_{\{x\},\{z,t\}}$ & $L_{\{y\},\{x,t\}}$ & $L_{\{y\},\{z,t\}}$ & $L_{\{z\},\{y,t\}}$ & $L_{\{t\},\{x,y,z\}}$ & $H_{\lambda}$ \\
\hline\hline
$L_{\{x\},\{y,z\}}$ & $-\frac{3}{2}$ & $1$ & $1$ & $0$ & $0$ & $0$ & $\frac{1}{2}$ & $1$ \\
\hline
$L_{\{x\},\{y,t\}}$ & $1$ & $-\frac{4}{3}$ & $1$ &  $\frac{1}{3}$ & $10$ & $1$ & $0$ & $1$ \\
\hline
$L_{\{x\},\{z,t\}}$ & $1$ & $1$ & $-\frac{4}{3}$ &  $0$ & $1$ & $0$ & $0$ & $1$ \\
\hline
$L_{\{y\},\{x,t\}}$ & $0$ & $\frac{1}{3}$ & $0$ &  $-\frac{5}{6}$ & $1$ & $0$ & $0$ & $1$ \\
\hline
$L_{\{y\},\{z,t\}}$ & $0$ & $0$ & $1$ &  $1$ & $-\frac{5}{4}$ & $\frac{1}{4}$ & $0$ & $1$ \\
\hline
$L_{\{z\},\{y,t\}}$ & $0$ & $1$ & $0$ &  $0$ & $\frac{1}{4}$ & $-\frac{5}{4}$ & $0$ & $1$ \\
\hline
$L_{\{t\},\{x,y,z\}}$ & $\frac{1}{2}$ & $0$ & $0$ &  $0$ & $0$ & $0$ & $-\frac{3}{2}$ & $1$ \\
\hline
 $H_{\lambda}$  & $1$ & $1$ & $1$ & $1$ & $1$ & $1$ & $1$ & $4$ \\
\hline
\end{tabular}
\end{center}
This matrix has rank~$8$, and $\mathrm{rk}\,\mathrm{Pic}(\widetilde{S}_{\Bbbk})=\mathrm{rk}\,\mathrm{Pic}(S_{\Bbbk})+9$.
Hence, we see that \eqref{equation:main-2-simple} holds, so that \eqref{equation:main-2} in Main Theorem also holds by Lemma~\ref{lemma:cokernel}.

\subsection{Family \textnumero $3.8$}
\label{section:r-3-n-8}

In this case, we have $h^{1,2}(X)=0$,
and a toric Landau--Ginzburg model of the threefold $X$ is given by
$$
x+y+z+\frac{xz}{y}+\frac{x}{y}+\frac{y}{x}+\frac{z}{y}+\frac{1}{z}+\frac{2}{y}+\frac{2}{x}+\frac{1}{xz}+\frac{1}{xt},
$$
which is the Minkowski polynomials \textnumero $1504$.
The pencil $\mathcal{S}$ is given by
$$
x^2yz+xy^2z+x^2z^2+xyz^2+x^2zt+y^2zt+xz^2t+xyt^2+2xzt^2+2yzt^2+yt^3+zt^3=\lambda xyzt.
$$

Suppose that $\lambda\ne\infty$. Then $S_\lambda$ has isolated singularities, so that it is irreducible.

Let $\mathcal{C}_1$ be a plane cubic curve that is given by $x=y^2z+2yzt+yt^2+zt^2=0$.
Then $\mathcal{C}_1$ is singular at $P_{\{x\},\{y\},\{t\}}$.
Let $\mathcal{C}_2$ be a  conic that is given by $y=xz+xt+t^2=0$.
Then
\begin{equation}
\label{equation:r3-n8-base-locus}
\begin{split}
H_{\{x\}}\cdot S_\lambda&=L_{\{x\},\{t\}}+\mathcal{C}_1,\\
H_{\{y\}}\cdot S_\lambda&=L_{\{y\},\{z\}}+L_{\{y\},\{x,t\}}+\mathcal{C}_2,\\
H_{\{z\}}\cdot S_\lambda&=L_{\{y\},\{z\}}+2L_{\{z\},\{t\}}+L_{\{z\},\{x,t\}},\\
H_{\{t\}}\cdot S_\lambda&=L_{\{x\},\{t\}}+L_{\{z\},\{t\}}+L_{\{t\},\{x,y\}}+L_{\{t\},\{y,z\}}.
\end{split}
\end{equation}
Thus, the base locus of the pencil $\mathcal{S}$ consists of the curves
$L_{\{x\},\{t\}}$, $L_{\{y\},\{z\}}$,
$L_{\{z\},\{t\}}$, $L_{\{y\},\{x,t\}}$, $L_{\{z\},\{x,t\}}$,
$L_{\{t\},\{x,y\}}$, $L_{\{t\},\{y,z\}}$,  $\mathcal{C}_1$, and $\mathcal{C}_2$.

For every $\lambda\in\mathbb{C}$, the singular points of the surface $S_\lambda$ contained in the base locus
of the pencil $\mathcal{S}$ are du Val and can be described as follows:
\begin{itemize}\setlength{\itemindent}{2cm}
\item[$P_{\{x\},\{y\},\{t\}}$:] type $\mathbb{A}_3$ with quadratic term $x(x+y+t)$;
\item[$P_{\{x\},\{z\},\{t\}}$:] type $\mathbb{A}_3$ with quadratic term $z(x+t)$, for $\lambda\neq -3$, type $\mathbb{A}_4$ for $\lambda=-3$;
\item[$P_{\{y\},\{z\},\{t\}}$:] type $\mathbb{A}_2$ with quadratic term $z(y+z+t)$;
\item[$P_{\{y\},\{z\},\{x,t\}}$:] type $\mathbb{A}_2$ with quadratic term
$$
y(x+3z+\lambda z+t)
$$
for $\lambda\neq -3$ and $\lambda\ne-4$, type $\mathbb{A}_3$ for $\lambda=-3$ or $\lambda=-4$;
\item[$P_{\{z\},\{t\},\{x,y\}}$:] type $\mathbb{A}_1$;
\item[$P_{\{t\},\{x,y\},\{y,z\}}$:] smooth if $\lambda\ne -3$, type $\mathbb{A}_1$ if $\lambda=3$.
\end{itemize}
Thus, it follows from Corollary~\ref{corollary:irreducible-fibers} that the fiber $\mathsf{f}^{-1}(\lambda)$ is irreducible for every $\lambda\ne\infty$.
Hence, we see that \eqref{equation:main-1} in Main Theorem holds in this case.

To prove \eqref{equation:main-2} in Main Theorem, we may assume that $\lambda\ne-3$ and $\lambda\ne-4$.
Then it follows from \eqref{equation:r3-n8-base-locus} that
the intersection matrix of the curves
$L_{\{x\},\{t\}}$,
$L_{\{y\},\{z\}}$, $L_{\{y\},\{x,t\}}$,
$L_{\{y\},\{z\}}$, $L_{\{z\},\{t\}}$, $L_{\{z\},\{x,t\}}$,
$L_{\{x\},\{t\}}$ $L_{\{z\},\{t\}}$, $L_{\{t\},\{x,y\}}$, $L_{\{t\},\{y,z\}}$
$\mathcal{C}_1$, and $\mathcal{C}_2$ on the surface~$S_\lambda$ has the same rank as the intersection matrix of the curves
$L_{\{x\},\{t\}}$, $L_{\{y\},\{z\}}$, $L_{\{y\},\{x,t\}}$, $L_{\{t\},\{x,y\}}$, $L_{\{t\},\{y,z\}}$, and $H_{\lambda}$.
The latter matrix is given by
\begin{center}\renewcommand\arraystretch{1.42}
\begin{tabular}{|c||c|c|c|c|c|c|}
\hline
 $\bullet$  & $L_{\{x\},\{t\}}$ & $L_{\{y\},\{z\}}$ & $L_{\{y\},\{x,t\}}$ & $L_{\{t\},\{x,y\}}$ & $L_{\{t\},\{y,z\}}$ &  $H_{\lambda}$ \\
\hline\hline
$L_{\{x\},\{t\}}$ & $-\frac{1}{2}$ & $0$ & $\frac{1}{4}$ & $\frac{1}{4}$ & $0$ & $1$ \\
\hline
$L_{\{y\},\{z\}}$ & $0$ & $-\frac{2}{3}$ & $\frac{2}{3}$ & $0$ & $\frac{1}{3}$ & $1$ \\
\hline
$L_{\{y\},\{x,t\}}$ & $\frac{1}{4}$ & $\frac{2}{3}$ & $-\frac{7}{12}$ & $\frac{3}{4}$ & $0$ & $1$ \\
\hline
$L_{\{t\},\{x,y\}}$ & $\frac{1}{4}$ & $0$ & $\frac{3}{4}$ & $-\frac{3}{4}$ & $1$ & $1$ \\
\hline
$L_{\{t\},\{y,z\}}$ & $0$ & $\frac{1}{3}$ & $0$ & $1$ & $-\frac{4}{3}$ & $1$ \\
\hline
 $H_{\lambda}$  & $1$ & $1$ & $1$ & $1$ & $1$ & $4$ \\
\hline
\end{tabular}
\end{center}
This matrix has rank~$6$, and $\mathrm{rk}\,\mathrm{Pic}(\widetilde{S}_{\Bbbk})=\mathrm{rk}\,\mathrm{Pic}(S_{\Bbbk})+11$.
Hence, we see that \eqref{equation:main-2-simple} holds,
so that \eqref{equation:main-2} in Main Theorem also holds by Lemma~\ref{lemma:cokernel}.

\subsection{Family \textnumero $3.9$}
\label{section:r-3-n-9}

In this case, the threefold $X$ is a blow up of a cone over a Veronese surface in $\mathbb{P}^5$
in a disjoint union of the vertex and a smooth curve of genus $3$.
Thus, we have $h^{1,2}(X)=3$.
A~toric Landau--Ginzburg model of this family is given by
$$
x+y+z+\frac{x^2}{yz}+\frac{y}{x}+\frac{z}{x}\frac{2x}{yz}+\frac{1}{x}+\frac{1}{yz},
$$
which is the polynomial \textnumero $373$.
The pencil $\mathcal{S}$ is given by
$$
x^2yz+xy^2z+xyz^2+x^3t+y^2zt+yz^2t+2x^2t^2+yzt^2+xt^3=\lambda xyzt.
$$
As usual, we assume that $\lambda\ne\infty$.

If $\lambda\ne-2$, then the surface $S_\lambda$ has isolated singularities, so that it is irreducible.
But
$$
S_{-2}=H_{\{x,t\}}+\mathbf{S},
$$
where $\mathbf{S}$ is an irreducible cubic surface that is given by $xt^2+x^2t+yzt+y^2z+yz^2+xyz=0$.
The surface has $\mathbf{S}$ isolated singularities, and $H_{\{x,t\}}\cdot\mathbf{S}=L_{\{y\},\{x,t\}}+L_{\{z\},\{x,t\}}+L_{\{x,t\},\{y,z\}}$.

To describe the base locus of the pencil $\mathcal{S}$, we observe that
\begin{itemize}
\item $H_{\{x\}}\cdot S_\lambda=L_{\{x\},\{y\}}+L_{\{x\},\{z\}}+L_{\{x\},\{t\}}+L_{\{x\},\{y,z,t\}}$,
\item $H_{\{y\}}\cdot S_\lambda=L_{\{x\},\{y\}}+L_{\{y\},\{t\}}+2L_{\{y\},\{x,t\}}$,
\item $H_{\{z\}}\cdot S_\lambda=L_{\{x\},\{z\}}+L_{\{z\},\{t\}}+2L_{\{z\},\{x,t\}}$,
\item $H_{\{t\}}\cdot S_\lambda=L_{\{x\},\{t\}}+L_{\{y\},\{t\}}+L_{\{z\},\{t\}}+L_{\{t\},\{x,y,z\}}$.
\end{itemize}
Therefore, the lines $L_{\{x\},\{z\}}$,
$L_{\{x\},\{y\}}$, $L_{\{y\},\{t\}}$, $L_{\{z\},\{t\}}$, $L_{\{z\},\{x,t\}}$, $L_{\{y\},\{x,t\}}$,,
$L_{\{x\},\{t\}}$, $L_{\{y\},\{t\}}$, $L_{\{t\},\{x,y,z\}}$ are all base curves of the pencil $\mathcal{S}$.

If $\lambda\ne -2$, then the singular points of the surface $S_\lambda$ contained in the base locus of the pencil $\mathcal{S}$
can be described as follows:
\begin{itemize}\setlength{\itemindent}{3cm}
\item[$P_{\{x\},\{z\},\{t\}}$:] type $\mathbb{A}_5$ with quadratic term $z(x+t)$;
\item[$P_{\{x\},\{y\},\{t\}}$:] type $\mathbb{A}_5$ with quadratic term $z(y+t)$;
\item[$P_{\{x\},\{t\},\{y,z\}}$:] type $\mathbb{A}_1$ with quadratic term $x^2+xy+xz+yt+zt+t^2+\lambda xt$;
\item[$P_{\{y\},\{z\},\{x,t\}}$:] type $\mathbb{A}_1$ with quadratic term $(\lambda+2)yz-(x+t)^2$.
\end{itemize}
Thus, it follows from Corollary~\ref{corollary:irreducible-fibers} that the fiber $\mathsf{f}^{-1}(\lambda)$ is irreducible for every $\lambda\ne -2$.

Note that the surface $S_{-2}$ consists of two irreducible components,
and it is singular along the lines $L_{\{y\},\{x,t\}}$ and $L_{\{z\},\{x,t\}}$.
Thus, it follows from \eqref{equation:equation:number-of-irredubicle-components-refined} and Lemma~\ref{lemma:main} that
$$
\big[\mathsf{f}^{-1}(-2)\big]=4+\mathbf{D}_{P_{\{x\},\{z\},\{t\}}}^{-2}+\mathbf{D}_{P_{\{x\},\{y\},\{t\}}}^{-2}+\mathbf{D}_{P_{\{x\},\{t\},\{y,z\}}}^{-2}+\mathbf{D}_{P_{\{y\},\{z\},\{x,t\}}}^{-2}.
$$
Moreover, the surface $S_{-2}$ has {good} double points at $P_{\{x\},\{z\},\{t\}}$, $P_{\{x\},\{y\},\{t\}}$, $P_{\{x\},\{t\},\{y,z\}}$, and $P_{\{y\},\{z\},\{x,t\}}$.
By Lemma~\ref{lemma:normal-crossing}, this implies
$$
\mathbf{D}_{P_{\{x\},\{z\},\{t\}}}^{-2}=\mathbf{D}_{P_{\{x\},\{y\},\{t\}}}^{-2}=\mathbf{D}_{P_{\{x\},\{t\},\{y,z\}}}^{-2}=\mathbf{D}_{P_{\{y\},\{z\},\{x,t\}}}^{-2}=0,
$$
so that $[\mathsf{f}^{-1}(-2)]=4$.
Hence, we see that \eqref{equation:main-1} in Main Theorem holds in this case.

If $\lambda\ne -2$, then the intersection matrix of the lines
$L_{\{x\},\{z\}}$, $L_{\{x\},\{y\}}$, $L_{\{y\},\{t\}}$, $L_{\{z\},\{t\}}$, $L_{\{z\},\{x,t\}}$, $L_{\{y\},\{x,t\}}$,
$L_{\{x\},\{t\}}$, $L_{\{y\},\{t\}}$, and $L_{\{t\},\{x,y,z\}}$
on the surface $S_\lambda$ has the same rank as the intersection matrix of the curves
$L_{\{x\},\{z\}}$, $L_{\{x\},\{y,z,t\}}$, $L_{\{y\},\{x,t\}}$, $L_{\{z\},\{x,t\}}$, $L_{\{t\},\{x,y,z\}}$, and $H_{\lambda}$.
The latter matrix is given by
\begin{center}\renewcommand\arraystretch{1.42}
\begin{tabular}{|c||c|c|c|c|c|c|}
\hline
 $\bullet$  & $L_{\{x\},\{z\}}$ & $L_{\{x\},\{y,z,t\}}$ & $L_{\{y\},\{x,t\}}$ & $L_{\{z\},\{x,t\}}$ & $L_{\{t\},\{x,y,z\}}$ &  $H_{\lambda}$ \\
\hline\hline
$L_{\{x\},\{z\}}$ & $-\frac{7}{6}$ & $1$ & $0$ & $\frac{2}{3}$ & $0$ & $1$ \\
\hline
$L_{\{x\},\{y,z,t\}}$ & $1$ & $-\frac{3}{2}$ & $0$ & $0$ & $\frac{1}{2}$ & $1$ \\
\hline
$L_{\{y\},\{x,t\}}$ & $0$ & $0$ & $-\frac{1}{6}$ & $\frac{1}{2}$ & $0$ & $1$ \\
\hline
$L_{\{z\},\{x,t\}}$ & $\frac{2}{3}$ & $0$ & $\frac{1}{2}$ & $-\frac{1}{6}$ & $0$ & $1$ \\
\hline
$L_{\{t\},\{x,y,z\}}$ & $0$ & $\frac{1}{2}$ & $0$ & $0$ & $-\frac{3}{2}$ & $1$ \\
\hline
 $H_{\lambda}$  & $1$ & $1$ & $1$ & $1$ & $1$ & $4$ \\
\hline
\end{tabular}
\end{center}
Its determinant vanishes.
The geometric reason for this is the following: if $\lambda\ne -2$, then
$$
H_{\{x,t\}}\cdot S_\lambda=2L_{\{x\},\{t\}}+L_{\{y\},\{x,t\}}+L_{\{y\},\{x,t\}}.
$$
which implies that $2L_{\{x\},\{t\}}+L_{\{y\},\{x,t\}}+L_{\{y\},\{x,t\}}\sim H_{\lambda}$ on the surface $S_\lambda$.
In fact, one can check that the rank of this matrix is $5$.
Moreover, we have $\mathrm{rk}\,\mathrm{Pic}(\widetilde{S}_{\Bbbk})=\mathrm{rk}\,\mathrm{Pic}(S_{\Bbbk})+12$.
Hence, we see that \eqref{equation:main-2-simple} holds, so that \eqref{equation:main-2} in Main Theorem also holds by Lemma~\ref{lemma:cokernel}.

\subsection{Family \textnumero $3.10$}
\label{section:r-3-n-10}

In this case, the threefold $X$ is a blow up of a smooth quadric hypersurface in $\mathbb{P}^4$ along a disjoint union of two irreducible conics.
Thus, we have $h^{1,2}(X)=0$.
A~toric Landau--Ginzburg model i given by the Laurent polynomial
$$
\frac{z}{y}+x+\frac{1}{y}+z+\frac{z}{xy}+\frac{x}{z}+\frac{z}{x}+\frac{xy}{z}+\frac{1}{z}+y+\frac{1}{x}.
$$
which is the Minkowski polynomial \textnumero $1112$.
The pencil $\mathcal{S}$ is given by
$$
z^2tx+x^2yz+t^2zx+z^2yx+t^2z^2+x^2yt+z^2yt+x^2y^2+t^2yx+y^2zx+t^2yz=\lambda xyzt.
$$

If $\lambda\ne\infty$, then $S_\lambda$ has isolated singularities, so that, in particular, it is irreducible.

To describe the base locus of the pencil $\mathcal{S}$, we observe that
\begin{equation}
\label{equation:r3-n10-base-locus}
\begin{split}
H_{\{x\}}\cdot S_\lambda&=L_{\{x\},\{z\}}+L_{\{x\},\{t\}}+\mathcal{C}_1,\\
H_{\{y\}}\cdot S_\lambda&=L_{\{y\},\{z\}}+L_{\{y\},\{t\}}+\mathcal{C}_2,\\
H_{\{z\}}\cdot S_\lambda&=L_{\{x\},\{z\}}+L_{\{y\},\{z\}}+\mathcal{C}_3,\\
H_{\{t\}}\cdot S_\lambda&=L_{\{x\},\{t\}}+L_{\{y\},\{t\}}+L_{\{t\},\{x,z\}}+L_{\{t\},\{y,z\}},
\end{split}
\end{equation}
where $\mathcal{C}_1$ is the conic $\{x=ty+tz+yz=0\}$,
the curve $\mathcal{C}_2$ is the conic $\{y=tx+tz+xz=0\}$, and $\mathcal{C}_3$ is the conic $\{z=t^2+tx+xy=0\}$.
Thus, the curves $L_{\{x\},\{z\}}$, $L_{\{x\},\{t\}}$, $L_{\{y\},\{z\}}$, $L_{\{y\},\{t\}}$,
$L_{\{t\},\{x,z\}}$, $L_{\{t\},\{y,z\}}$, $\mathcal{C}_1$, $\mathcal{C}_2$, and $\mathcal{C}_3$
are all base curves of the pencil $\mathcal{S}$.

\begin{lemma}
\label{lemma:r3-n10-singular-points}
Suppose that $\lambda\ne\infty$.
Then the singular points of the surface $S_\lambda$ contained in the base locus can be described as follows:
\begin{itemize}\setlength{\itemindent}{3cm}
\item[$P_{\{y\},\{z\},\{t\}}$:] type $\mathbb{A}_3$ for $\lambda\ne -4$, type $\mathbb{A}_4$ for $\lambda=-4$;

\item[$P_{\{x\},\{z\},\{t\}}$:] type $\mathbb{A}_4$ for $\lambda\ne -2$, type $\mathbb{A}_6$ for $\lambda=-2$;

\item[$P_{\{x\},\{y\},\{t\}}$:] type $\mathbb{A}_2$ for $\lambda\ne -4$, type $\mathbb{A}_3$ for $\lambda=-4$;

\item[$P_{\{x\},\{y\},\{z\}}$:] type $\mathbb{A}_2$ for $\lambda\ne -3$, type $\mathbb{A}_4$ for $\lambda=-3$;

\item[$P_{\{t\},\{x,z\},\{y,z\}}$:] smooth for $\lambda\ne -3$, type $\mathbb{A}_2$ for $\lambda=-3$.
\end{itemize}
\end{lemma}

\begin{proof}
Taking partial derivatives, we see that $P_{\{y\},\{z\},\{t\}}$,
$P_{\{x\},\{z\},\{t\}}$, $P_{\{x\},\{y\},\{t\}}$, and $P_{\{x\},\{y\},\{z\}}$ are the singular points of the surface $S_\lambda$.
Moreover, if $\lambda\ne -3$, then these points are the only singular points of the surface $S_\lambda$
that are contained in the base locus of the pencil $\mathcal{S}$.
If $\lambda=-3$, then $P_{\{t\},\{x,z\},\{y,z\}}$ is also a singular point of the surface $S_\lambda$.
In this case, the surface $S_\lambda$ does not have other singular points which are contained in the base locus of the pencil $\mathcal{S}$.

In the chart $x=1$, the surface $S_\lambda$ is given by the equation
$$
y(y+z+t)+y^2z+yz^2-\lambda tyz+t^2y+t^2z+tz^2+t^2yz+t^2z^2+tyz^2=0.
$$
Introducing coordinates $\bar{y}=y$, $\bar{z}=z$, and $\bar{t}=t+y+z$, we can rewrite this equation as
\begin{multline*}
\bar{t}\bar{y}+\bar{t}^2\bar{y}+\bar{t}^2\bar{z}-2\bar{t}\bar{y}^2-(\lambda+4)\bar{z}\bar{t}\bar{y}-\bar{z}^2\bar{t}+\bar{y}^3+(\lambda+4)\bar{z}\bar{y}^2+(\lambda+3)\bar{z}^2\bar{y}+\\
+\bar{t}^2\bar{y}\bar{z}+\bar{t}^2\bar{z}^2-2\bar{t}\bar{y}^2\bar{z}-3\bar{t}\bar{y}\bar{z}^2-2\bar{t}\bar{z}^3+\bar{y}^3\bar{z}+2\bar{y}^2\bar{z}^2+2\bar{y}\bar{z}^3+\bar{z}^4=0.
\end{multline*}
Here, we have $P_{\{y\},\{z\},\{t\}}=(0,0,0)$.
Let us blow up this point.

Let $\hat{z}=z$, $\hat{y}=\frac{y}{z}$, $\hat{t}=\frac{t}{z}$.
We can rewrite the latter equation (after dividing by $\hat{z}^2$) as
\begin{multline*}
\hat{t}\hat{y}-\hat{t}\hat{z}+(\lambda+3)\hat{y}\hat{z}+\hat{z}^2+\Big(\hat{t}^2\hat{z}-(\lambda+4)\hat{z}\hat{t}\hat{y}-2\hat{z}^2\hat{t}+(\lambda+4)\hat{z}\hat{y}^2+2\hat{z}^2\hat{y}\Big)+\\
+\Big(\hat{t}^2\hat{y}\hat{z}+\hat{t}^2\hat{z}^2-2\hat{t}\hat{y}^2\hat{z}-3\hat{t}\hat{y}\hat{z}^2+\hat{y}^3\hat{z}+2\hat{y}^2\hat{z}^2\Big)+\Big(\hat{t}^2\hat{y}\hat{z}^2-2\hat{t}\hat{y}^2\hat{z}^2+\hat{y}^3\hat{z}^2\Big)=0.
\end{multline*}
This equation defines (a chart of) the blow up of the surface $S_\lambda$ at $P_{\{y\},\{z\},\{t\}}$.
The two exceptional curves of the blow up are given by the equations $\hat{z}=\hat{t}=0$ and $\hat{z}=\hat{y}=0$.
They intersect at the point $(0,0,0)$, which is singular point of the obtained surface.

If $\lambda\ne 4$, then $\hat{t}\hat{y}-\hat{t}\hat{z}+(\lambda+3)\hat{y}\hat{z}+\hat{z}^2$ is non-degenerate,
so that $P_{\{y\},\{z\},\{t\}}$ is a singular point of the surface $S_\lambda$ of type $\mathbb{A}_3$.
If $\lambda=4$, then this form splits as $(\hat{y}-\hat{z})(\hat{t}-\hat{z})$.
In this case, introducing new coordinates $\tilde{y}=\hat{t}-\hat{z}$,
$\tilde{z}=\hat{y}-\hat{z}$, and $\tilde{t}=\hat{t}$, we rewrite the latter equation (with $\lambda=-4)$ as
$$
\tilde{y}\tilde{z}+\tilde{t}^3+\text{higher order terms}=0,
$$
where we order monomials with respect to weights $\mathrm{wt}(\tilde{y})=3$, $\mathrm{wt}(\tilde{z})=3$, and $\mathrm{wt}(\tilde{t})=2$.
We see that this point is a singular point of type $\mathbb{A}_2$.
Therefore, if $\lambda=-4$, then $P_{\{y\},\{z\},\{t\}}$ is a singular point of the surface $S_\lambda$ of type $\mathbb{A}_4$.

We leave the proofs of the remaining assertions of the lemma to the reader.
\end{proof}

Thus, it follows from Corollary~\ref{corollary:irreducible-fibers} that the fiber $\mathsf{f}^{-1}(\lambda)$ is irreducible for every $\lambda\ne\infty$.
This implies \eqref{equation:main-1} in Main Theorem.
To prove \eqref{equation:main-2} in Main Theorem, we need the following.

\begin{lemma}
\label{lemma:r3-n10-intersection}
Suppose that $\lambda\not\in\{-2,-3,-4,\infty\}$.
Then the intersection matrix of the curves  $L_{\{x\},\{z\}}$, $L_{\{x\},\{t\}}$, $L_{\{y\},\{z\}}$, $L_{\{y\},\{t\}}$, $L_{\{t\},\{x,z\}}$, and $H_{\lambda}$ on the surface $S_\lambda$ is given by
\begin{center}\renewcommand\arraystretch{1.42}
\begin{tabular}{|c||c|c|c|c|c|c|}
\hline
 $\bullet$  & $L_{\{x\},\{z\}}$  & $L_{\{x\},\{t\}}$  & $L_{\{y\},\{z\}}$  & $L_{\{y\},\{t\}}$  & $L_{\{t\},\{x,z\}}$  & $H_{\lambda}$\\
\hline\hline
$L_{\{x\},\{z\}}$ &  $-\frac{2}{15}$ & $\frac{2}{5}$ & $\frac{1}{3}$ & $0$ & $\frac{3}{5}$ & $1$ \\
\hline
$L_{\{x\},\{t\}}$ &  $\frac{2}{5}$ & $-\frac{8}{5}$ & $0$ & $\frac{1}{3}$ & $\frac{1}{5}$ & $1$ \\
\hline
$L_{\{y\},\{z\}}$ &  $\frac{1}{3}$ & $0$ & $-\frac{7}{2}$ & $\frac{3}{4}$ & $0$ & $1$ \\
\hline
$L_{\{y\},\{t\}}$ &  $0$ & $\frac{1}{3}$ & $\frac{3}{4}$ & $-\frac{7}{12}$ & $1$ & $1$ \\
\hline
$L_{\{t\},\{x,z\}}$ &  $\frac{3}{5}$ & $\frac{1}{5}$ & $0$ & $1$ & $-\frac{6}{5}$ & $1$ \\
\hline
 $H_{\lambda}$  & $1$ & $1$ & $1$ & $1$ & $1$ & $4$ \\
\hline
\end{tabular}
\end{center}
\end{lemma}

\begin{proof}
Lets us show how to compute the diagonal entries of the intersection table.
To start with, let us compute $L_{\{x\},\{z\}}^2$.
Observe that $P_{\{x\},\{z\},\{t\}}$ and $P_{\{x\},\{y\},\{z\}}$ are the only singular points of the surface $S_\lambda$ that are contained in $L_{\{x\},\{z\}}$.
Thus, by Proposition~\ref{proposition:du-Val-self-intersection}, one has
$L_{\{x\},\{z\}}^2=-2+\frac{2}{3}+\frac{k}{5}$,
where either $k=4$ or $k=6$. In fact, we have $k=6$ here.
Indeed, let us use the notation of Remark~\ref{remark:transversal} with $S=S_\lambda$, $O=P_{\{x\},\{z\},\{t\}}$, $n=4$, $C=L_{\{x\},\{z\}}$.
In the chart $y=1$, the surface~$S_\lambda$ is given by
$$
x(x+z)+\text{higher order terms}=0,
$$
and $L_{\{x\},\{z\}}$ is given by $x=z=0$.
This shows that $\overline{C}$ contains the point $\overline{G}_1\cap\overline{G}_4$.
Thus, either $\widetilde{C}\cap G_2\ne\varnothing$ or $\widetilde{C}\cap G_3\ne\varnothing$.
In both cases, we have $k=6$ by Proposition~\ref{proposition:du-Val-self-intersection}.
Thus, we have $L_{\{x\},\{z\}}^2=-\frac{2}{15}$.

Similarly, it follows from Proposition~\ref{proposition:du-Val-self-intersection} that $L_{\{x\},\{t\}}^2=-\frac{8}{15}$,
because $P_{\{x\},\{z\},\{t\}}$ and $P_{\{x\},\{y\},\{t\}}$ are the only singular points of the surface $S_\lambda$ contained in $L_{\{x\},\{t\}}$.
Likewise, we see that $L_{\{t\},\{x,z\}}^2=-\frac{6}{5}$,
because $P_{\{x\},\{z\},\{t\}}$ is the only singular point of the surface $S_\lambda$ that is contained in $L_{\{t\},\{x,z\}}$.
Using Proposition~\ref{proposition:du-Val-self-intersection} again, we get
$L_{\{y\},\{t\}}^2=L_{\{y\},\{z\}}^2=-\frac{7}{12}$.

Now let us compute the remaining entries of the first raw in the intersection table.
Since $L_{\{x\},\{z\}}\cap L_{\{y\},\{t\}}=\varnothing$, we have $L_{\{x\},\{z\}}\cdot L_{\{y\},\{t\}}=0$.
To compute $L_{\{x\},\{z\}}\cdot L_{\{y\},\{z\}}$, observe that $L_{\{x\},\{z\}}\cap L_{\{y\},\{z\}}=P_{\{x\},\{y\},\{t\}}$.
In the chart $t=1$, the surface $S_\lambda$ is given by
$$
(x+z)(z+y)+\text{higher order terms}=0.
$$
Thus, using Proposition~\ref{proposition:du-Val-intersection} and Remark~\ref{remark:transversal} with $S=S_\lambda$, $O=P_{\{y\},\{z\},\{t\}}$, $n=2$,
$C=L_{\{x\},\{z\}}$, and $Z=L_{\{y\},\{z\}}$, we see that $L_{\{x\},\{z\}}\cdot L_{\{y\},\{z\}}=\frac{1}{3}$.

To find $L_{\{x\},\{z\}}\cdot L_{\{x\},\{t\}}$ and $L_{\{x\},\{z\}}\cdot L_{\{t\},\{y,z\}}$,
we notice that
$$
L_{\{x\},\{z\}}\cap L_{\{x\},\{t\}}=L_{\{x\},\{z\}}\cap L_{\{t\},\{y,z\}}=P_{\{x\},\{z\},\{t\}}.
$$
Let us use the notation of Remark~\ref{remark:transversal} with $O=P_{\{x\},\{z\},\{t\}}$, $n=4$, $C=L_{\{x\},\{t\}}$, and $Z=L_{\{t\},\{y,z\}}$.
Keeping in mind the equation of the surface $S_\lambda$ in the chart $y=1$,
we see that neither $\overline{C}$ nor $\overline{Z}$ contains the point $\overline{G}_1\cap\overline{G}_4$.
By Proposition~\ref{proposition:du-Val-intersection}, this implies, in particular, that $L_{\{x\},\{t\}}\cdot L_{\{t\},\{y,z\}}=\frac{1}{5}$.
On the other hand, we already checked above that the proper transform of the line $L_{\{x\},\{z\}}$
on the surface $\overline{S}$ does contain the point $\overline{G}_1\cap\overline{G}_4$.
This implies that $L_{\{x\},\{z\}}\cdot L_{\{x\},\{t\}}$ and $L_{\{x\},\{z\}}\cdot L_{\{t\},\{y,z\}}$ are among $\frac{2}{5}$ and $\frac{3}{5}$.
Moreover, one has
\begin{multline*}
1=H_{\{t\}}\cdot L_{\{x\},\{z\}}=\Big(L_{\{x\},\{t\}}+L_{\{y\},\{t\}}+L_{\{t\},\{x,z\}}+L_{\{t\},\{y,z\}}\Big)\cdot L_{\{x\},\{z\}}=\\
=L_{\{x\},\{t\}}\cdot L_{\{x\},\{z\}}+L_{\{y\},\{t\}}\cdot L_{\{x\},\{z\}}+L_{\{t\},\{x,z\}}\cdot L_{\{x\},\{z\}}+L_{\{t\},\{y,z\}}\cdot L_{\{x\},\{z\}}=\\
=L_{\{x\},\{t\}}\cdot L_{\{x\},\{z\}}+L_{\{t\},\{y,z\}}\cdot L_{\{x\},\{z\}},
\end{multline*}
because $L_{\{y\},\{t\}}\cdot L_{\{x\},\{z\}}=0$ and $L_{\{t\},\{x,z\}}\cdot L_{\{x\},\{z\}}=0$.
Similarly, we have
\begin{multline*}
H_{\{x\}}\cdot L_{\{x\},\{t\}}=\Big(L_{\{x\},\{z\}}+L_{\{x\},\{t\}}+C_1\Big)\cdot L_{\{x\},\{t\}}=\\
=L_{\{x\},\{z\}}\cdot L_{\{x\},\{t\}}+L_{\{x\},\{t\}}^2+C_1\cdot L_{\{x\},\{t\}}=L_{\{x\},\{z\}}\cdot L_{\{x\},\{t\}}-\frac{8}{5}+C_1\cdot L_{\{x\},\{t\}}.
\end{multline*}
Moreover, we have $C_1\cap L_{\{x\},\{t\}}=P_{\{x\},\{z\},\{t\}}\cup P_{\{x\},\{y\},\{t\}}$.
Thus, applying Proposition~\ref{proposition:du-Val-intersection} and Remark~\ref{remark:transversal}, we see that
$C_1\cdot L_{\{x\},\{t\}}=\frac{1}{3}+\frac{4}{5}=\frac{17}{15}$,
so  that $L_{\{x\},\{z\}}\cdot L_{\{x\},\{t\}}=\frac{2}{5}$.
Thus, we have $L_{\{x\},\{z\}}\cdot L_{\{t\},\{y,z\}}=\frac{3}{5}$.

To compute the remaining entries of the second raw in the intersection table,
we have to find $L_{\{x\},\{t\}}\cdot L_{\{y\},\{z\}}$ and $L_{\{x\},\{t\}}\cdot L_{\{y\},\{t\}}$.
But $L_{\{x\},\{t\}}\cap L_{\{y\},\{z\}}=\varnothing$, so that $L_{\{x\},\{t\}}\cdot L_{\{y\},\{z\}}=0$.
Moreover, we have $L_{\{x\},\{t\}}\cap L_{\{y\},\{t\}}=P_{\{x\},\{y\},\{t\}}$,
so that $L_{\{x\},\{t\}}\cdot L_{\{y\},\{t\}}=\frac{1}{3}$ by Proposition~\ref{proposition:du-Val-intersection}.

To complete the proof of the lemma, we have to find $L_{\{y\},\{z\}}\cdot L_{\{y\},\{t\}}$,
$L_{\{y\},\{z\}}\cdot L_{\{t\},\{x,z\}}$, and $L_{\{y\},\{t\}}\cdot L_{\{t\},\{x,z\}}$.
Since $L_{\{y\},\{z\}}\cap L_{\{t\},\{x,z\}}=\varnothing$, we have $L_{\{y\},\{z\}}\cdot L_{\{t\},\{x,z\}}=0$.
Similarly, we have  $L_{\{y\},\{t\}}\cdot L_{\{t\},\{x,z\}}=1$, since
$L_{\{y\},\{t\}}\cap L_{\{t\},\{x,z\}}=P_{\{y\},\{z\},\{x,t\}}$ and the surface~$S_\lambda$ is smooth at the point $[1:0:-1:0]$.
Finally, observe that $L_{\{y\},\{z\}}\cdot L_{\{y\},\{t\}}=\frac{3}{4}$ by Proposition~\ref{proposition:du-Val-intersection},
since $L_{\{y\},\{z\}}\cap L_{\{y\},\{t\}}=P_{\{y\},\{z\},\{t\}}$.
\end{proof}

If $\lambda\not\in\{-2,-3,-4,\infty\}$, then it follows from \eqref{equation:r3-n10-base-locus} that
the intersection matrix of the curves $L_{\{x\},\{z\}}$, $L_{\{x\},\{t\}}$, $L_{\{y\},\{z\}}$, $L_{\{y\},\{t\}}$,
$L_{\{t\},\{x,z\}}$, $L_{\{t\},\{y,z\}}$, $\mathcal{C}_1$, $\mathcal{C}_2$, and $\mathcal{C}_3$
on the surface $S_\lambda$ has the same rank as the intersection matrix of the curves
$L_{\{x\},\{z\}}$, $L_{\{x\},\{t\}}$, $L_{\{y\},\{z\}}$, $L_{\{y\},\{t\}}$, $L_{\{t\},\{x,z\}}$, and $H_{\lambda}$.
On the other hand, the determinant of the matrix in Lemma~\ref{lemma:r3-n10-intersection} is $-\frac{2}{9}$,
and $\mathrm{rk}\,\mathrm{Pic}(\widetilde{S}_{\Bbbk})=\mathrm{rk}\,\mathrm{Pic}(S_{\Bbbk})+11$.
Hence, we see that \eqref{equation:main-2-simple} holds,
so that \eqref{equation:main-2} in Main Theorem also holds by Lemma~\ref{lemma:cokernel}.

\subsection{Family \textnumero $3.11$}
\label{section:r-3-n-11}

The threefold $X$ can be obtained from $\mathbb{P}^3$ by blowing up a disjoint union of a point and a smooth elliptic curve.
We discussed this case in Example~\ref{example:r-3-n-11},
where we described the pencil $\mathcal{S}$ and its base locus.
Let us use the notation introduced in this example.
As usual, we assume that $\lambda\ne\infty$. Observe that
\begin{equation}
\label{equation:r-3-n-11}
\begin{split}
H_{\{x\}}\cdot S_\lambda&=L_{\{x\},\{t\}}+L_{\{x\},\{z,t\}}+\mathcal{C},\\
H_{\{y\}}\cdot S_\lambda&=L_{\{y\},\{z\}}+L_{\{y\},\{t\}}+L_{\{y\},\{x,t\}}+L_{\{y\},\{z,t\}},\\
H_{\{z\}}\cdot S_\lambda&=2L_{\{y\},\{z\}}+L_{\{z\},\{t\}}+L_{\{z\},\{x,t\}},\\
H_{\{t\}}\cdot S_\lambda&=L_{\{x\},\{t\}}+L_{\{y\},\{t\}}+L_{\{z\},\{t\}}+L_{\{t\},\{x,y,z\}}.
\end{split}
\end{equation}

If $\lambda\ne -2$, then $S_\lambda$ is irreducible, it has isolated singularities,
and its singular points contained in the base locus of the pencil $\mathcal{S}$ can be described as follows:
\begin{itemize}\setlength{\itemindent}{3.5cm}
\item[$P_{\{y\},\{z\},\{t\}}$:] type $\mathbb{A}_4$ with quadratic term $yz$;
\item[$P_{\{x\},\{z\},\{t\}}$:] type $\mathbb{A}_2$ with quadratic term $(x+t)(z+t)$;
\item[$P_{\{x\},\{y\},\{t\}}$:] type $\mathbb{A}_2$ with quadratic term $(x+t)(y+t)$;
\item[$P_{\{x\},\{t\},\{y,z\}}$:] type $\mathbb{A}_1$ with quadratic term $(x+t)(x+y+z-t)-(\lambda+2)xt$;
\item[$P_{\{y\},\{z\},\{x,t\}}$:] type $\mathbb{A}_2$ with quadratic term $z(x+t+(\lambda+2)y)$;
\item[{$[0:1\mp\sqrt{5}:-2:+2]$}:] smooth if $\lambda\ne\frac{-1\pm\sqrt{5}}{2}$, type $\mathbb{A}_1$ if $\lambda=\frac{-1\pm\sqrt{5}}{2}$.
\end{itemize}
Then $[\mathsf{f}^{-1}(\lambda)]=1$ for every $\lambda\ne -2$ by Corollary~\ref{corollary:irreducible-fibers}.

Recall that $S_{-2}=H_{\{x,t\}}+\mathbf{S}$,
where $\mathbf{S}$ is an irreducible cubic surface that has {good} double points at
$P_{\{y\},\{z\},\{t\}}$, $P_{\{x\},\{z\},\{t\}}$, $P_{\{x\},\{y\},\{t\}}$, $P_{\{x\},\{t\},\{y,z\}}$, and $P_{\{y\},\{z\},\{x,t\}}$.
Moreover, the surface $S_{-2}$ is smooth at general points of the base curves
$L_{\{x\},\{t\}}$, $L_{\{y\},\{z\}}$, $L_{\{y\},\{t\}}$, $L_{\{z\},\{t\}}$, $L_{\{x\},\{z,t\}}$,
$L_{\{y\},\{x,t\}}$, $L_{\{y\},\{z,t\}}$,  $L_{\{z\},\{x,t\}}$, $L_{\{t\},\{x,y,z\}}$, and $\mathcal{C}$.
Thus, it follows from \eqref{equation:equation:number-of-irredubicle-components-refined} and Lemmas~\ref{lemma:main} and \ref{lemma:normal-crossing} that
$[\mathsf{f}^{-1}(-2)]=[S_{-3}]=2$.
Therefore, we conclude that \eqref{equation:main-1} in Main Theorem holds in this case.

To verify \eqref{equation:main-2} in Main Theorem, we may assume that $\lambda\ne-2$ and $\lambda\ne\frac{-1\pm\sqrt{5}}{2}$.
Then, using \eqref{equation:r-3-n-11}, we see that the intersection matrix of the curves
$L_{\{x\},\{t\}}$, $L_{\{y\},\{z\}}$, $L_{\{y\},\{t\}}$, $L_{\{z\},\{t\}}$, $L_{\{x\},\{z,t\}}$,
$L_{\{y\},\{x,t\}}$, $L_{\{y\},\{z,t\}}$,  $L_{\{z\},\{x,t\}}$, $L_{\{t\},\{x,y,z\}}$, and $\mathcal{C}$
on the surface $S_\lambda$ has the same rank as the intersection matrix of the curves
$L_{\{x\},\{t\}}$, $L_{\{x\},\{z,t\}}$, $L_{\{y\},\{x,t\}}$, $L_{\{y\},\{z,t\}}$, $L_{\{z\},\{x,t\}}$, $L_{\{t\},\{x,y,z\}}$, and $H_{\lambda}$.
On the other hand, the latter matrix is given by
\begin{center}\renewcommand\arraystretch{1.42}
\begin{tabular}{|c||c|c|c|c|c|c|c|}
\hline
 $\bullet$  & $L_{\{x\},\{t\}}$ & $L_{\{x\},\{z,t\}}$ & $L_{\{y\},\{x,t\}}$ & $L_{\{y\},\{z,t\}}$ & $L_{\{z\},\{x,t\}}$ & $L_{\{t\},\{x,y,z\}}$ &  $H_{\lambda}$ \\
\hline\hline
$L_{\{x\},\{t\}}$ & $-\frac{1}{6}$ & $\frac{1}{3}$ & $\frac{2}{3}$ & $0$ & $\frac{2}{3}$ & $\frac{1}{2}$ & $1$ \\
\hline
$L_{\{x\},\{z,t\}}$ & $\frac{1}{3}$ & $-\frac{4}{3}$ & $0$ & $1$ & $\frac{1}{3}$ & $0$ & $1$ \\
\hline
$L_{\{y\},\{x,t\}}$ & $\frac{2}{3}$ & $0$ & $-\frac{2}{3}$ & $1$ & $\frac{1}{3}$ & $0$ & $1$ \\
\hline
$L_{\{y\},\{z,t\}}$ & $0$ & $1$ & $1$ & $-\frac{6}{5}$ & $0$ & $0$ & $1$ \\
\hline
$L_{\{z\},\{x,t\}}$ & $\frac{2}{3}$ & $\frac{1}{3}$ & $\frac{1}{3}$ & $0$ & $-\frac{2}{3}$ & $0$ & $1$ \\
\hline
$L_{\{t\},\{x,y,z\}}$ & $\frac{1}{2}$ & $0$ & $0$ & $0$ & $0$ & $-\frac{3}{2}$ & $1$ \\
\hline
 $H_{\lambda}$  & $1$ & $1$ & $1$ & $1$ & $1$ & $1$ & $4$ \\
\hline
\end{tabular}
\end{center}
Its rank is $6$. Note also that $\mathrm{rk}\,\mathrm{Pic}(\widetilde{S}_{\Bbbk})=\mathrm{rk}\,\mathrm{Pic}(S_{\Bbbk})+11$.
Hence, we see that \eqref{equation:main-2-simple} holds,
so that \eqref{equation:main-2} in Main Theorem also holds by Lemma~\ref{lemma:cokernel}.

\subsection{Family \textnumero $3.12$}
\label{section:r-3-n-12}

In this case, the threefold $X$ can be obtained from $\mathbb{P}^3$ by blowing up a disjoint union of a line and a twisted cubic curve.
Its toric Landau--Ginzburg model is given by
$$
\frac {z}{x}+\frac{1}{x}+y+z+\frac{y}{z}+\frac{z}{y}+\frac{1}{z}+\frac{xy}{z}+\frac{1}{y}+x,
$$
which is the Minkowski polynomials \textnumero $737$.
The pencil $\mathcal{S}$ is given by
$$
z^2yt+t^2yz+y^2zx+z^2yx+y^2tx+z^2tx+t^2yx+x^2y^2+t^2zx+x^2yz=\lambda xyzt.
$$
As usual, we suppose that $\lambda\ne\infty$.

Let $\mathcal{C}$ be the conic $z=xy+yt+t^2=0$. Then
\begin{equation}
\label{equation:3-12}
\begin{split}
H_{\{x\}}\cdot S_\lambda&=L_{\{x\},\{y\}}+L_{\{x\},\{z\}}+L_{\{x\},\{t\}}+L_{\{x\},\{z,t\}},\\
H_{\{y\}}\cdot S_\lambda&=L_{\{x\},\{y\}}+L_{\{y\},\{z\}}+L_{\{y\},\{t\}}+L_{\{y\},\{z,t\}},\\
H_{\{z\}}\cdot S_\lambda&=L_{\{x\},\{z\}}+L_{\{y\},\{z\}}+\mathcal{C},\\
H_{\{t\}}\cdot S_\lambda&=L_{\{x\},\{t\}}+L_{\{y\},\{t\}}+L_{\{t\},\{x,z\}}+L_{\{t\},\{y,z\}}.
\end{split}
\end{equation}
Thus, the base locus of the pencil $\mathcal{S}$ consists of the curves
$L_{\{x\},\{y\}}$, $L_{\{x\},\{z\}}$, $L_{\{x\},\{t\}}$,
$L_{\{y\},\{z\}}$, $L_{\{y\},\{t\}}$, $L_{\{x\},\{z,t\}}$, $L_{\{y\},\{z,t\}}$,
$L_{\{t\},\{x,z\}}$, $L_{\{t\},\{y,z\}}$, and $\mathcal{C}$.

For every $\lambda\in\mathbb{C}$, the surface $S_\lambda$ has isolated singularities, so that it is irreducible.
Moreover, the singular points of the surface $S_\lambda$ contained in the base locus of the pencil~$\mathcal{S}$
are du Val and can be described as follows:
\begin{itemize}\setlength{\itemindent}{1.3cm}
\item[$P_{\{x\},\{y\},\{z\}}$:] type $\mathbb{A}_1$;
\item[$P_{\{x\},\{y\},\{t\}}$:] type $\mathbb{A}_1$;
\item[$P_{\{x\},\{z\},\{t\}}$:] type $\mathbb{A}_3$ with quadratic term $x(x+z+t)$ for $\lambda\neq -3$, type $\mathbb{A}_4$ for $\lambda=-3$;
\item[$P_{\{y\},\{z\},\{t\}}$:] type $\mathbb{A}_4$ with quadratic term $y(y+z)$ for $\lambda\neq -2$, type $\mathbb{A}_5$ for $\lambda=-2$;
\item[$P_{\{x\},\{y\},\{z,t\}}$:] type $\mathbb{A}_1$ for $\lambda\neq -2$, type $\mathbb{A}_2$ for $\lambda=-2$;
\item[{$P_{\{t\},\{x,z\},\{y,z\}}$}:] smooth if $\lambda\ne -3$, type $\mathbb{A}_1$ if $\lambda=-3$.
\end{itemize}
Thus, it follows from Corollary~\ref{corollary:irreducible-fibers} that the fiber $\mathsf{f}^{-1}(\lambda)$ is irreducible for every $\lambda\in\mathbb{C}$.
Since $h^{1,2}(X)=0$, we see that \eqref{equation:main-1} in Main Theorem holds in this case.

Now let us verify \eqref{equation:main-2} in Main Theorem. We may assume that  $\lambda\ne-2$ and $\lambda\ne-3$.
Then the intersection matrix of the curves $L_{\{x\},\{y\}}$, $L_{\{x\},\{z\}}$, $L_{\{x\},\{t\}}$, $L_{\{y\},\{t\}}$, $L_{\{y\},\{z,t\}}$, $L_{\{t\},\{x,z\}}$, and~$H_{\lambda}$
on the surface $S_\lambda$ is given by
\begin{center}\renewcommand\arraystretch{1.42}
\begin{tabular}{|c||c|c|c|c|c|c|c|}
\hline
 $\bullet$  & $L_{\{x\},\{y\}}$ & $L_{\{x\},\{z\}}$ & $L_{\{x\},\{t\}}$ & $L_{\{y\},\{t\}}$ & $L_{\{y\},\{z,t\}}$ & $L_{\{t\},\{x,z\}}$ &  $H_{\lambda}$ \\
\hline\hline
$L_{\{x\},\{y\}}$ & $-\frac{1}{2}$ & $\frac{1}{2}$ & $\frac{1}{2}$ & $\frac{1}{2}$ & $\frac{1}{2}$ &$0$ & $1$ \\
\hline
$L_{\{x\},\{z\}}$ & $\frac{1}{2}$ & $-\frac{3}{4}$ & $\frac{3}{4}$ & $0$ & $0$ & $\frac{1}{4}$ & $1$ \\
\hline
$L_{\{x\},\{t\}}$ & $\frac{1}{2}$ & $\frac{3}{4}$ & $-\frac{3}{4}$ & $\frac{1}{2}$ & $0$ & $\frac{1}{4}$ & $1$ \\
\hline
$L_{\{y\},\{t\}}$ & $\frac{1}{2}$ & $0$ & $\frac{1}{2}$ & $-\frac{7}{10}$ & $\frac{3}{5}$ & $1$ & $1$ \\
\hline
$L_{\{y\},\{z,t\}}$ & $\frac{1}{2}$ & $0$ & $0$ & $\frac{3}{5}$ & $-\frac{7}{10}$ & $0$ & $1$ \\
\hline
$L_{\{t\},\{x,z\}}$ & $0$ & $\frac{1}{4}$ & $\frac{1}{4}$ & $1$ & $0$ & $-\frac{5}{4}$ & $1$ \\
\hline
 $H_{\lambda}$  & $1$ & $1$ & $1$ & $1$ & $1$ & $1$ & $4$ \\
\hline
\end{tabular}
\end{center}
The rank of this matrix is $7$. On the other hand, it follows from \eqref{equation:3-12} that
\begin{multline*}
L_{\{x\},\{y\}}+L_{\{x\},\{z\}}+L_{\{x\},\{t\}}+L_{\{x\},\{z,t\}}\sim L_{\{x\},\{y\}}+L_{\{y\},\{z\}}+L_{\{y\},\{t\}}+L_{\{y\},\{z,t\}}\sim\\
\sim L_{\{x\},\{z\}}+L_{\{y\},\{z\}}+\mathcal{C}\sim L_{\{x\},\{t\}}+L_{\{y\},\{t\}}+L_{\{t\},\{x,z\}}+L_{\{t\},\{y,z\}}\sim H_\lambda.
\end{multline*}
This implies that the rank of the intersection matrix of the curves
$L_{\{x\},\{y\}}$, $L_{\{x\},\{z\}}$, $L_{\{x\},\{t\}}$,
$L_{\{y\},\{z\}}$, $L_{\{y\},\{t\}}$, $L_{\{x\},\{z,t\}}$, $L_{\{y\},\{z,t\}}$,
$L_{\{t\},\{x,z\}}$, $L_{\{t\},\{y,z\}}$, and $\mathcal{C}$ is also $7$.
As we have seen above, $\mathrm{rk}\,\mathrm{Pic}(\widetilde{S}_{\Bbbk})=\mathrm{rk}\,\mathrm{Pic}(S_{\Bbbk})+10$.
Hence, we see that \eqref{equation:main-2-simple} holds, so that \eqref{equation:main-2} in Main Theorem holds by~Lemma~\ref{lemma:cokernel}.

\subsection{Family \textnumero $3.13$}
\label{section:r-3-n-13}

The threefold $X$ is a blow up of a smooth hypersurface in $\mathbb{P}^2\times\mathbb{P}^2$ of bidegree $(1,1)$
in  a smooth rational curve of bidegree~$(2,2)$.
Thus, we have $h^{1,2}(X)=0$.
A~toric Landau--Ginzburg model of this family is given by Minkowski polynomial \textnumero $420$, which is
$$
\frac{x}{y}+x+\frac{1}{y}+z+\frac{z}{x}+\frac{1}{z}+y+\frac{1}{x}+\frac{y}{z}.
$$
The quartic pencil $\mathcal{S}$ is given by
$$
x^2zt+x^2yz+t^2zx+z^2yx+z^2yt+t^2yx+y^2zx+t^2yz+y^2tx=\lambda xyzt.
$$

As usual, we assume that $\lambda\ne\infty$.
Then
\begin{itemize}
\item $H_{\{x\}}\cdot S_\lambda=L_{\{x\},\{y\}}+L_{\{x\},\{z\}}+L_{\{x\},\{t\}}+L_{\{x\},\{z,t\}}$,
\item $H_{\{y\}}\cdot S_\lambda=L_{\{x\},\{y\}}+L_{\{y\},\{z\}}+L_{\{y\},\{t\}}+L_{\{y\},\{x,t\}}$,
\item $H_{\{z\}}\cdot S_\lambda=L_{\{x\},\{z\}}+L_{\{y\},\{z\}}+L_{\{z\},\{t\}}+L_{\{z\},\{y,t\}}$,
\item $H_{\{t\}}\cdot S_\lambda=L_{\{x\},\{t\}}+L_{\{y\},\{t\}}+L_{\{z\},\{t\}}+L_{\{t\},\{x,y,z\}}$.
\end{itemize}
Thus, the lines $L_{\{x\},\{y\}}$, $L_{\{x\},\{z\}}$, $L_{\{x\},\{t\}}$,
$L_{\{y\},\{z\}}$, $L_{\{y\},\{t\}}$, $L_{\{z\},\{t\}}$, $L_{\{x\},\{z,t\}}$, $L_{\{y\},\{x,t\}}$,
$L_{\{z\},\{y,t\}}$, and $L_{\{t\},\{x,y,z\}}$ are all base curves of the pencil $\mathcal{S}$.

Each surface $S_\lambda$  is irreducible, it has isolated singularities, and its singular points contained in the base locus of the pencil $\mathcal{S}$
can be described as follows:
\begin{itemize}\setlength{\itemindent}{1.5cm}
\item[$P_{\{x\},\{y\},\{z\}}$:] type $\mathbb{A}_1$;

\item[$P_{\{x\},\{y\},\{t\}}$:] type $\mathbb{A}_3$ with quadratic term $y(x+t)$ for $\lambda\neq -2$, type $\mathbb{A}_4$ for $\lambda=-2$;

\item[$P_{\{x\},\{z\},\{t\}}$:] type $\mathbb{A}_3$ with quadratic term $x(z+t)$ for $\lambda\neq -2$, type $\mathbb{A}_4$ for $\lambda=-2$;

\item[$P_{\{y\},\{z\},\{t\}}$:] type $\mathbb{A}_3$ with quadratic term $z(y+t)$ for $\lambda\neq -2$, type $\mathbb{A}_4$ for $\lambda=-2$.
\end{itemize}
Then $[\mathsf{f}^{-1}(\lambda)]=1$ for every $\lambda\in\mathbb{C}$ by Lemma~\ref{corollary:irreducible-fibers}.
This confirms \eqref{equation:main-1} in Main Theorem.

Now we suppose that $\lambda\ne-2$.
Then the rank of the intersection matrix of the lines $L_{\{x\},\{y\}}$, $L_{\{x\},\{z\}}$, $L_{\{x\},\{t\}}$,
$L_{\{y\},\{z\}}$, $L_{\{y\},\{t\}}$, $L_{\{z\},\{t\}}$, $L_{\{x\},\{z,t\}}$, $L_{\{y\},\{x,t\}}$,
$L_{\{z\},\{y,t\}}$, and $L_{\{t\},\{x,y,z\}}$ on the surface $S_{\lambda}$ has the same
as the rank of the following matrix:
\begin{center}\renewcommand\arraystretch{1.42}
\begin{tabular}{|c||c|c|c|c|c|c|c|}
\hline
 $\bullet$  & $L_{\{x\},\{z\}}$ & $L_{\{x\},\{t\}}$ & $L_{\{x\},\{z,t\}}$ & $L_{\{y\},\{x,t\}}$ & $L_{\{z\},\{y,t\}}$ & $L_{\{t\},\{x,y,z\}}$ &  $H_{\lambda}$ \\
\hline\hline
$L_{\{x\},\{z\}}$ & $-\frac{3}{4}$ & $\frac{3}{4}$ & $1$ & $0$ & $1$ &$0$ & $1$ \\
\hline
$L_{\{x\},\{t\}}$ & $\frac{3}{4}$ & $-\frac{1}{2}$ & $1$ & $1$ & $0$ & $1$ & $1$ \\
\hline
$L_{\{x\},\{z,t\}}$ & $1$ & $1$ & $-1$ & $0$ & $0$ & $0$ & $1$ \\
\hline
$L_{\{y\},\{x,t\}}$ & $0$ & $1$ & $0$ & $-1$ & $0$ & $0$ & $1$ \\
\hline
$L_{\{z\},\{y,t\}}$ & $1$ & $0$ & $0$ & $0$ & $-1$ & $0$ & $1$ \\
\hline
$L_{\{t\},\{x,y,z\}}$ & $0$ & $1$ & $0$ & $0$ & $0$ & $-2$ & $1$ \\
\hline
 $H_{\lambda}$  & $1$ & $1$ & $1$ & $1$ & $1$ & $1$ & $4$ \\
\hline
\end{tabular}
\end{center}
Its rank is $7$.
On the other hand, we have $\mathrm{rk}\,\mathrm{Pic}(\widetilde{S}_{\Bbbk})=\mathrm{rk}\,\mathrm{Pic}(S_{\Bbbk})+10$.
Hence, we see that \eqref{equation:main-2-simple} holds.
By Lemma~\ref{lemma:cokernel}, we see that \eqref{equation:main-2} in Main Theorem also holds.

\subsection{Family \textnumero $3.14$}
\label{section:r-3-n-14}

The threefold $X$ is $\mathbb{P}^3$ blown up in a union of a smooth plane cubic and a point that does not lie on the plane containing the cubic,
so that $h^{1,2}(X)=1$.
A~toric Landau--Ginzburg model of this family is given by
$$
x+y+z+\frac{x^2}{yz}+\frac{y}{x}+\frac{z}{x}+\frac{x}{yz}+\frac{1}{x},
$$
which is Minkowski polynomial \textnumero $202$. The quartic pencil $\mathcal{S}$ is given by
$$
x^2yz+xy^2z+xyz^2+x^3t+y^2zt+yz^2t+x^2t^2+yzt^2=\lambda xyzt.
$$

Suppose that $\lambda\ne\infty$. Then
\begin{equation}
\label{equation:3-14}
\begin{split}
H_{\{x\}}\cdot S_\lambda&=L_{\{x\},\{y\}}+L_{\{x\},\{z\}}+L_{\{x\},\{t\}}+L_{\{x\},\{y,z,t\}},\\
H_{\{y\}}\cdot S_\lambda&=2L_{\{x\},\{y\}}+L_{\{y\},\{t\}}+L_{\{y\},\{x,t\}},\\
H_{\{z\}}\cdot S_\lambda&=2L_{\{x\},\{z\}}+L_{\{z\},\{t\}}+L_{\{z\},\{x,t\}},\\
H_{\{t\}}\cdot S_\lambda&=L_{\{x\},\{t\}}+L_{\{y\},\{t\}}+L_{\{z\},\{t\}}+L_{\{t\},\{x,y,z\}}.
\end{split}
\end{equation}
Thus, the base locus of the pencil $\mathcal{S}$ consists of the lines
$L_{\{x\},\{y\}}$, $L_{\{x\},\{z\}}$, $L_{\{x\},\{t\}}$,
$L_{\{y\},\{t\}}$, $L_{\{z\},\{t\}}$, $L_{\{x\},\{y,z,t\}}$,
$L_{\{y\},\{x,t\}}$,  $L_{\{z\},\{x,t\}}$, $L_{\{t\},\{x,y,z\}}$.

Let $\mathbf{S}$ be the cubic surface in $\mathbb{P}^3$ that is given by
$$
xyz+x^2t+y^2z+yz^2+yzt=0.
$$
Then $\mathbf{S}$ is irreducible and $S_{-2}=H_{\{x,t\}}+\mathbf{S}$.
On the other hand, if $\lambda\ne-2$, then the surface $S_\lambda$ has isolated singularities, so that it is irreducible.
In this case, its singular points contained in the base locus of the pencil $\mathcal{S}$ can be described as follows:
\begin{itemize}\setlength{\itemindent}{3cm}
\item[$P_{\{x\},\{y\},\{z\}}$:] type $\mathbb{A}_1$ with quadratic term $x^2+yz$;

\item[$P_{\{x\},\{y\},\{t\}}$:] type $\mathbb{A}_4$ with quadratic term $y(x+t)$;

\item[$P_{\{x\},\{z\},\{t\}}$:] type $\mathbb{A}_4$ with quadratic term $z(x+t)$;

\item[$P_{\{x\},\{y\},\{z,t\}}$:] type $\mathbb{A}_1$ with quadratic term $x^2-y^2-yz-yt+(\lambda+1)xy$;

\item[$P_{\{x\},\{z\},\{y,t\}}$:] type $\mathbb{A}_1$ with quadratic term $x^2-z^2-yz-zt+(\lambda+1)xz$;

\item[$P_{\{x\},\{t\},\{y,z\}}$:] type $\mathbb{A}_1$ with quadratic term $(x+t)(x+y+z+t)-(\lambda+2)xt$.
\end{itemize}

By Lemma~\ref{corollary:irreducible-fibers}, we have $[\mathsf{f}^{-1}(\lambda)]=1$ for every $\lambda\ne -2$.
Moreover, the points
$P_{\{x\},\{z\},\{t\}}$, $P_{\{x\},\{y\},\{t\}}$,
$P_{\{x\},\{y\},\{z\}}$, $P_{\{x\},\{y\},\{z,t\}}$,
$P_{\{x\},\{z\},\{y,t\}}$, and $P_{\{x\},\{t\},\{y,z\}}$
are good double points of the surface $S_{-2}$.
Furthermore, one can check that the surface $S_{-2}$ is smooth
at general points of the lines
$L_{\{x\},\{y\}}$, $L_{\{x\},\{z\}}$, $L_{\{x\},\{t\}}$,
$L_{\{y\},\{t\}}$, $L_{\{z\},\{t\}}$, $L_{\{x\},\{y,z,t\}}$,
$L_{\{y\},\{x,t\}}$,  $L_{\{z\},\{x,t\}}$, and~$L_{\{t\},\{x,y,z\}}$.
Therefore, using \eqref{equation:equation:number-of-irredubicle-components-refined} and
applying Lemmas~\ref{lemma:main} and \ref{lemma:normal-crossing}, we conclude that
$[\mathsf{f}^{-1}(-2)]=[S_{-2}]=2$.
This confirms \eqref{equation:main-1} in Main Theorem.

To verify \eqref{equation:main-2} in Main Theorem, we suppose that $\lambda\ne-2$.
Then \eqref{equation:3-14} gives
\begin{multline*}
H_{\lambda}\sim L_{\{x\},\{y\}}+L_{\{x\},\{z\}}+L_{\{x\},\{t\}}+L_{\{x\},\{y,z,t\}}\sim 2L_{\{x\},\{y\}}+L_{\{y\},\{t\}}+L_{\{y\},\{x,t\}}\sim\\
\sim 2L_{\{x\},\{z\}}+L_{\{z\},\{t\}}+L_{\{z\},\{x,t\}}\sim L_{\{x\},\{t\}}+L_{\{y\},\{t\}}+L_{\{z\},\{t\}}+L_{\{t\},\{x,y,z\}}.
\end{multline*}
Therefore, the intersection matrix of the lines
$L_{\{x\},\{y\}}$, $L_{\{x\},\{z\}}$, $L_{\{x\},\{t\}}$,
$L_{\{y\},\{t\}}$, $L_{\{z\},\{t\}}$, $L_{\{x\},\{y,z,t\}}$,
$L_{\{y\},\{x,t\}}$,  $L_{\{z\},\{x,t\}}$, $L_{\{t\},\{x,y,z\}}$ on the surface $S_{\lambda}$
has the same rank as the intersection
matrix of the curves $L_{\{y\},\{x,t\}}$, $L_{\{x\},\{y,z,t\}}$, $L_{\{z\},\{x,t\}}$, $L_{\{t\},\{x,y,z\}}$, $L_{\{x\},\{y\}}$, and~$H_{\lambda}$.
The latter matrix is given by
\begin{center}\renewcommand\arraystretch{1.42}
\begin{tabular}{|c||c|c|c|c|c|c|}
\hline
 $\bullet$  & $L_{\{y\},\{x,t\}}$ & $L_{\{x\},\{y,z,t\}}$ & $L_{\{z\},\{x,t\}}$ & $L_{\{t\},\{x,y,z\}}$ & $L_{\{x\},\{y\}}$ &  $H_{\lambda}$ \\
\hline\hline
$L_{\{y\},\{x,t\}}$ & $-\frac{4}{5}$ & $0$ & $1$ & $0$ & $\frac{3}{5}$ & $1$ \\
\hline
$L_{\{x\},\{y,z,t\}}$ & $0$ & $-\frac{1}{2}$ & $0$ & $\frac{1}{2}$ & $\frac{1}{2}$ & $1$ \\
\hline
$L_{\{z\},\{x,t\}}$ & $1$ & $0$ & $-\frac{4}{5}$ & $0$ & $0$ & $1$ \\
\hline
$L_{\{t\},\{x,y,z\}}$ & $0$ & $\frac{1}{2}$ & $0$ & $-\frac{3}{2}$ & $0$ & $1$ \\
\hline
$L_{\{x\},\{y\}}$ & $\frac{3}{5}$ & $\frac{1}{2}$ & $0$ & $0$ & $-\frac{1}{5}$ & $1$ \\
\hline
 $H_{\lambda}$  & $1$ & $1$ & $1$ & $1$ & $1$ & $4$ \\
\hline
\end{tabular}
\end{center}
The rank of this matrix is $5$.
We can see that the determinant of this matrix is~$0$ without computing it.
Indeed, we have $H_\lambda\sim 2L_{\{x\},\{t\}}+L_{\{y\},\{x,t\}}+L_{\{z\},\{x,t\}}$ on the surface~$S_\lambda$,
because $H_{\{x,t\}}\cdot S_\lambda=2L_{\{x\},\{t\}}+L_{\{y\},\{x,t\}}+L_{\{z\},\{x,t\}}$.

Observe that $\mathrm{rk}\,\mathrm{Pic}(\widetilde{S}_{\Bbbk})=\mathrm{rk}\,\mathrm{Pic}(S_{\Bbbk})+12$.
Therefore, we conclude that \eqref{equation:main-2-simple} holds.
Using Lemma~\ref{lemma:cokernel}, we see that \eqref{equation:main-2} in Main Theorem also holds.

\subsection{Family \textnumero $3.15$}
\label{section:r-3-n-15}

In this case, the threefold $X$ is a blow up of a quadric in a disjoint union of a line and a conic, so that $h^{1,2}(X)=0$.
A~toric Landau--Ginzburg model of this family is given by Minkowski polynomial \textnumero $419$ is
$$
x+y+z+\frac{x}{z}+\frac{z}{y}+\frac{1}{z}+\frac{1}{y}+\frac{1}{x}+\frac{z}{xy}.
$$
The quartic pencil $\mathcal{S}$ is given by
$$
x^2zy+y^2zx+z^2yx+x^2ty+z^2tx+t^2yx+t^2zx+t^2zy+t^2z^2=\lambda xyzt.
$$

Suppose that $\lambda\ne\infty$.
Let $\mathcal{C}$ be a conic that is given by $y=xz+xt+zt=0$.
Then
\begin{equation}
\label{equation:3-15}
\begin{split}
H_{\{x\}}\cdot S_\lambda&=L_{\{x\},\{z\}}+2L_{\{x\},\{t\}}+L_{\{x\},\{y,z\}},\\
H_{\{y\}}\cdot S_\lambda&=L_{\{y\},\{z\}}+L_{\{y\},\{t\}}+\mathcal{C},\\
H_{\{z\}}\cdot S_\lambda&=L_{\{x\},\{z\}}+L_{\{y\},\{z\}}+L_{\{z\},\{t\}}+L_{\{z\},\{x,t\}},\\
H_{\{t\}}\cdot S_\lambda&=L_{\{x\},\{t\}}+L_{\{y\},\{t\}}+L_{\{z\},\{t\}}+L_{\{t\},\{x,y,z\}}.
\end{split}
\end{equation}
Thus, the base locus of the pencil $\mathcal{S}$ consists of
the curves $L_{\{x\},\{z\}}$, $L_{\{x\},\{t\}}$,  $L_{\{y\},\{z\}}$, $L_{\{y\},\{t\}}$,
$L_{\{x\},\{y,z\}}$, $L_{\{z\},\{t\}}$, $L_{\{z\},\{x,t\}}$, $L_{\{t\},\{x,y,z\}}$, and $\mathcal{C}$.

For every $\lambda\in\mathbb{C}$, the surface $S_\lambda$ is irreducible, it has isolated singularities,
and its singular points  contained in the base locus of the pencil~$\mathcal{S}$ can be described as follows:
\begin{itemize}\setlength{\itemindent}{1.5cm}
\item[$P_{\{x\},\{y\},\{z\}}$:] type $\mathbb{A}_2$ with quadratic term $(x+z)(y+z)$ for $\lambda\neq -2$, type $\mathbb{A}_3$ for $\lambda=-2$;

\item[$P_{\{x\},\{y\},\{t\}}$:] type $\mathbb{A}_1$;

\item[$P_{\{x\},\{z\},\{t\}}$:] type $\mathbb{A}_3$ with quadratic term $xz$;

\item[$P_{\{y\},\{z\},\{t\}}$:] type $\mathbb{A}_3$ with quadratic term $y(z+t)$ for $\lambda\neq -3$, type $\mathbb{A}_4$ for $\lambda=-3$;

\item[$P_{\{x\},\{t\},\{y,z\}}$:] type $\mathbb{A}_2$ with quadratic term $x(x+y+z-t-\lambda t)$.
\end{itemize}
So, by Lemma~\ref{corollary:irreducible-fibers}, each fiber $\mathsf{f}^{-1}(\lambda)$ is irreducible.
This confirms \eqref{equation:main-1} in Main Theorem.

If $\lambda\ne-2$ and $\lambda\ne-3$,
then the intersection matrix of the curves $L_{\{x\},\{z\}}$, $L_{\{x\},\{y,z\}}$, $L_{\{y\},\{z\}}$, $L_{\{z\},\{x,t\}}$, $L_{\{t\},\{x,y,z\}}$, and $H_{\lambda}$
on the surface $S_\lambda$ is given by the following table:
\begin{center}\renewcommand\arraystretch{1.42}
\begin{tabular}{|c||c|c|c|c|c|c|}
\hline
 $\bullet$  & $L_{\{x\},\{z\}}$ & $L_{\{x\},\{y,z\}}$ & $L_{\{y\},\{z\}}$ & $L_{\{z\},\{x,t\}}$ & $L_{\{t\},\{x,y,z\}}$ &  $H_{\lambda}$ \\
\hline\hline
$L_{\{x\},\{z\}}$ & $-\frac{1}{3}$ & $\frac{1}{3}$ & $\frac{1}{3}$ & $\frac{1}{2}$ & $0$ & $1$ \\
\hline
$L_{\{x\},\{y,z\}}$ & $\frac{1}{3}$ & $-\frac{2}{3}$ & $\frac{2}{3}$ & $0$ & $\frac{1}{3}$ & $1$ \\
\hline
$L_{\{y\},\{z\}}$ & $\frac{1}{3}$ & $\frac{2}{3}$ & $-\frac{7}{12}$ & $0$ & $0$ & $1$ \\
\hline
$L_{\{z\},\{x,t\}}$ & $\frac{1}{2}$ & $0$ & $0$ & $-\frac{5}{4}$ & $0$ & $1$ \\
\hline
$L_{\{t\},\{x,y,z\}}$ & $0$ & $\frac{1}{3}$ & $0$ & $0$ & $-\frac{4}{3}$ & $1$ \\
\hline
 $H_{\lambda}$  & $1$ & $1$ & $1$ & $1$ & $1$ & $4$ \\
\hline
\end{tabular}
\end{center}
This matrix has rank $6$.
On the other hand, using \eqref{equation:3-15}, we see that
\begin{multline*}
H_\lambda\sim L_{\{x\},\{z\}}+2L_{\{x\},\{t\}}+L_{\{x\},\{y,z\}}\sim L_{\{y\},\{z\}}+L_{\{y\},\{t\}}+\mathcal{C}\sim\\
\sim L_{\{x\},\{z\}}+L_{\{y\},\{z\}}+L_{\{z\},\{t\}}+L_{\{z\},\{x,t\}}\sim L_{\{x\},\{t\}}+L_{\{y\},\{t\}}+L_{\{z\},\{t\}}+L_{\{t\},\{x,y,z\}}
\end{multline*}
on the surface $S_\lambda$.
Thus, the rank of the intersection matrix of the curves $L_{\{x\},\{z\}}$, $L_{\{x\},\{t\}}$,  $L_{\{y\},\{z\}}$, $L_{\{y\},\{t\}}$,
$L_{\{x\},\{y,z\}}$, $L_{\{z\},\{t\}}$, $L_{\{z\},\{x,t\}}$, $L_{\{t\},\{x,y,z\}}$, and $\mathcal{C}$ is also $6$.
One the other hand, we have $\mathrm{rk}\,\mathrm{Pic}(\widetilde{S}_{\Bbbk})=\mathrm{rk}\,\mathrm{Pic}(S_{\Bbbk})+11$.
Thus, we see that \eqref{equation:main-2-simple} holds,
so that \eqref{equation:main-2} in Main Theorem holds by Lemma~\ref{lemma:cokernel}.

\subsection{Family \textnumero $3.16$}
\label{section:r-3-n-16}

In this case, the threefold $X$ is can be obtained from $\mathbb{P}^3$ blown up in a point by blowing up a proper transform of a twisted cubic curve passing through the point.
Thus, we see that $h^{1,2}(X)=0$.
A~toric Landau--Ginzburg model of this family is given by Minkowski polynomial \textnumero $212$, which is
$$
x+y+z+\frac{y}{z}+\frac{x}{y}+\frac{y}{xz}+\frac{1}{y}+\frac{1}{x}.
$$
The pencil $\mathcal{S}$ is given by the equation
$$
x^2zy+y^2zx+z^2yx+y^2tx+x^2tz+t^2y^2+t^2zx+t^2zy=\lambda xyzt.
$$

Suppose that $\lambda\ne\infty$. Then
\begin{equation}
\label{equation:3-16}
\begin{split}
H_{\{x\}}\cdot S_\lambda&=L_{\{x\},\{y\}}+2L_{\{x\},\{t\}}+L_{\{x\},\{y,z\}},\\
H_{\{y\}}\cdot S_\lambda&=L_{\{x\},\{y\}}+L_{\{y\},\{z\}}+L_{\{y\},\{t\}}+L_{\{y\},\{x,t\}},\\
H_{\{z\}}\cdot S_\lambda&=2L_{\{y\},\{z\}}+L_{\{z\},\{t\}}+L_{\{z\},\{x,t\}},\\
H_{\{t\}}\cdot S_\lambda&=L_{\{x\},\{t\}}+L_{\{y\},\{t\}}+L_{\{z\},\{t\}}+L_{\{t\},\{x,y,z\}}.
\end{split}
\end{equation}
Thus, the base locus of the pencil $\mathcal{S}$ consists of
the lines $L_{\{x\},\{y\}}$, $L_{\{x\},\{t\}}$, $L_{\{y\},\{z\}}$, $L_{\{y\},\{t\}}$, $L_{\{z\},\{t\}}$,
$L_{\{x\},\{y,z\}}$, $L_{\{y\},\{x,t\}}$, $L_{\{z\},\{x,t\}}$, and $L_{\{t\},\{x,y,z\}}$.

For every $\lambda\in\mathbb{C}$, the surface $S_\lambda\in\mathcal{S}$ has isolated singularities, so that it is irreducible.

The singular points of the surface $S_\lambda$ contained in the base locus of the pencil $\mathcal{S}$ can be described as follows:
\begin{itemize}\setlength{\itemindent}{3cm}
\item[$P_{\{x\},\{y\},\{z\}}$:] type $\mathbb{A}_1$;

\item[$P_{\{x\},\{z\},\{t\}}$:] type $\mathbb{A}_1$;

\item[$P_{\{x\},\{y\},\{t\}}$:] type $\mathbb{A}_3$ with quadratic term $xy$;

\item[$P_{\{x\},\{t\},\{y,z\}}$:] type $\mathbb{A}_2$ with quadratic term
$$
x(x+y+z-t-\lambda t)
$$
for $\lambda\neq -1$, type $\mathbb{A}_3$ for $\lambda=-1$;

\item[$P_{\{y\},\{z\},\{t\}}$:] type $\mathbb{A}_2$ with quadratic term $z(y+t)$;

\item[$P_{\{y\},\{z\},\{x,t\}}$:] type $\mathbb{A}_2$ with quadratic term
$$
z(x+t-2y-\lambda y)
$$
for $\lambda\neq -2$, type $\mathbb{A}_3$ for $\lambda=-2$.
\end{itemize}
Therefore,  every fiber $\mathsf{f}^{-1}(\lambda)$ is irreducible by Lemma~\ref{corollary:irreducible-fibers}.
This confirms \eqref{equation:main-1} in Main Theorem, because $h^{1,2}(X)=0$.

Now let us verify \eqref{equation:main-2} in Main Theorem. We may assume that $\lambda\ne-1$ and $\lambda\ne-2$.
Using~\eqref{equation:3-16}, we see that the intersection matrix of
the lines $L_{\{x\},\{y\}}$, $L_{\{x\},\{t\}}$, $L_{\{y\},\{z\}}$, $L_{\{y\},\{t\}}$, $L_{\{z\},\{t\}}$,
$L_{\{x\},\{y,z\}}$, $L_{\{y\},\{x,t\}}$, $L_{\{z\},\{x,t\}}$, and $L_{\{t\},\{x,y,z\}}$
on the surface $S_\lambda$ has the same rank
as the intersection matrix of the curves $L_{\{x\},\{t\}}$, $L_{\{x\},\{y,z\}}$, $L_{\{y\},\{z\}}$, $L_{\{z\},\{t\}}$, $L_{\{t\},\{x,y,z\}}$, and $H_{\lambda}$.
But the latter matrix is given by
\begin{center}\renewcommand\arraystretch{1.42}
\begin{tabular}{|c||c|c|c|c|c|c|}
\hline
 $\bullet$  & $L_{\{x\},\{t\}}$ & $L_{\{x\},\{y,z\}}$ & $L_{\{y\},\{z\}}$ & $L_{\{z\},\{t\}}$ & $L_{\{t\},\{x,y,z\}}$ &  $H_{\lambda}$ \\
\hline\hline
$L_{\{x\},\{t\}}$ & $-\frac{1}{12}$ & $\frac{2}{3}$ & $0$ & $\frac{1}{2}$ & $\frac{1}{3}$ & $1$ \\
\hline
$L_{\{x\},\{y,z\}}$ & $\frac{2}{3}$ & $-\frac{5}{6}$ & $\frac{1}{2}$ & $0$ & $\frac{1}{3}$ & $1$ \\
\hline
$L_{\{y\},\{z\}}$ & $0$ & $\frac{1}{2}$ & $-\frac{1}{6}$ & $\frac{2}{3}$ & $0$ & $1$ \\
\hline
$L_{\{z\},\{t\}}$ & $\frac{1}{2}$ & $0$ & $\frac{2}{3}$ & $-\frac{5}{6}$ & $1$ & $1$ \\
\hline
$L_{\{t\},\{x,y,z\}}$ & $\frac{1}{3}$ & $\frac{1}{3}$ & $0$ & $1$ & $-\frac{4}{3}$ & $1$ \\
\hline
 $H_{\lambda}$  & $1$ & $1$ & $1$ & $1$ & $1$ & $4$ \\
\hline
\end{tabular}
\end{center}
Its rank is $6$.
On the other hand, the description of the singular points of the surface $S_\lambda$ easily gives
$\mathrm{rk}\,\mathrm{Pic}(\widetilde{S}_{\Bbbk})=\mathrm{rk}\,\mathrm{Pic}(S_{\Bbbk})+11$.
Thus, we can conclude that \eqref{equation:main-2-simple} holds in this case, so that \eqref{equation:main-2} in Main Theorem also
holds by Lemma~\ref{lemma:cokernel}.

\subsection{Family \textnumero $3.17$}
\label{section:r-3-n-17}

The threefold $X$ is a divisor of tridegree $(1,1,1)$ in $\mathbb{P}^1\times\mathbb{P}^1\times\mathbb{P}^2$, so that $h^{1,2}(X)=0$.
A~toric Landau--Ginzburg model of this family is given by Minkowski polynomial \textnumero $208$, which is
$$
\frac{z}{y}+x+\frac{1}{y}+z+y+\frac{1}{x}+\frac{1}{xz}+\frac{y}{xz}.
$$
The pencil of quartic surfaces $\mathcal{S}$ is given by the equation
$$
z^2tx+x^2zy+t^2zx+z^2yx+y^2zx+t^2zy+t^3y+t^2y^2=\lambda xyzt.
$$

To describe the base locus of the pencil $\mathcal{S}$, we observe that
\begin{itemize}
\item $H_{\{x\}}\cdot S_\lambda=L_{\{x\},\{y\}}+2L_{\{x\},\{t\}}+L_{\{x\},\{y,z,t\}}$,
\item $H_{\{y\}}\cdot S_\lambda=L_{\{x\},\{y\}}+L_{\{y\},\{z\}}+L_{\{y\},\{t\}}+L_{\{y\},\{z,t\}}$,
\item $H_{\{z\}}\cdot S_\lambda=L_{\{y\},\{z\}}+2L_{\{z\},\{t\}}+L_{\{z\},\{y,t\}}$,
\item $H_{\{t\}}\cdot S_\lambda=L_{\{x\},\{t\}}+L_{\{y\},\{t\}}+L_{\{z\},\{t\}}+L_{\{t\},\{x,y,z\}}$.
\end{itemize}
Thus, the base locus of the pencil $\mathcal{S}$ consists of the lines
$L_{\{x\},\{y\}}$, $L_{\{x\},\{t\}}$, $L_{\{y\},\{z\}}$, $L_{\{y\},\{t\}}$, $L_{\{z\},\{t\}}$,
$L_{\{x\},\{y,z,t\}}$, $L_{\{y\},\{z,t\}}$, $L_{\{z\},\{y,t\}}$, and $L_{\{t\},\{x,y,z\}}$.

If $\lambda\ne\infty$, then $S_\lambda$ has isolated singularities, so that it is irreducible.
In this case, the singular points of the surface $S_\lambda$ contained in the base locus of the pencil $\mathcal{S}$ can be described as follows:
\begin{itemize}\setlength{\itemindent}{3cm}
\item[$P_{\{x\},\{y\},\{t\}}$:] type $\mathbb{A}_2$ with quadratic term $x(y+t)$;

\item[$P_{\{x\},\{z\},\{t\}}$:] type $\mathbb{A}_1$;

\item[$P_{\{y\},\{z\},\{t\}}$:] type $\mathbb{A}_4$ with quadratic term $yz$;

\item[$P_{\{x\},\{y\},\{z,t\}}$:] type $\mathbb{A}_1$ for $\lambda\neq -2$, type $\mathbb{A}_2$ for $\lambda=-2$;

\item[$P_{\{x\},\{t\},\{y,z\}}$:] type $\mathbb{A}_2$ with quadratic term $x(x+y+z-t-\lambda t)$;

\item[$P_{\{z\},\{t\},\{x,y\}}$:] type $\mathbb{A}_1$.
\end{itemize}
Thus, by Lemma~\ref{corollary:irreducible-fibers}, the fiber $\mathsf{f}^{-1}(\lambda)$ is irreducible for every $\lambda\ne\infty$.
This confirms \eqref{equation:main-1} in Main Theorem, since $h^{1,2}(X)=0$.

Let us check \eqref{equation:main-2} in Main Theorem.
To do this, we may assume that $\lambda\ne\infty$ and~$\lambda\ne-2$.
Then
the intersection matrix of the curves $L_{\{x\},\{y\}}$, $L_{\{x\},\{t\}}$, $L_{\{y\},\{t\}}$, $L_{\{y\},\{z,t\}}$, $L_{\{t\},\{x,y,z\}}$, and $H_{\lambda}$
on the surface $S_\lambda$ is given by the following table:
\begin{center}\renewcommand\arraystretch{1.42}
\begin{tabular}{|c||c|c|c|c|c|c|}
\hline
 $\bullet$  & $L_{\{x\},\{y\}}$ & $L_{\{x\},\{t\}}$ & $L_{\{y\},\{t\}}$ & $L_{\{y\},\{z,t\}}$ & $L_{\{t\},\{x,y,z\}}$ &  $H_{\lambda}$ \\
\hline\hline
$L_{\{x\},\{y\}}$ & $-\frac{5}{6}$ & $\frac{2}{3}$ & $\frac{1}{3}$ & $\frac{1}{2}$ & $0$ & $1$ \\
\hline
$L_{\{x\},\{t\}}$ & $\frac{2}{3}$ & $-\frac{1}{6}$ & $\frac{1}{3}$ & $0$ & $\frac{1}{3}$ & $1$ \\
\hline
$L_{\{y\},\{t\}}$ & $\frac{1}{3}$ & $\frac{1}{3}$ & $-\frac{8}{15}$ & $\frac{4}{5}$ & $1$ & $1$ \\
\hline
$L_{\{y\},\{z,t\}}$ & $\frac{1}{2}$ & $0$ & $\frac{4}{5}$ & $-\frac{7}{10}$ & $0$ & $1$ \\
\hline
$L_{\{t\},\{x,y,z\}}$ & $0$ & $\frac{1}{3}$ & $1$ & $0$ & $-\frac{5}{6}$ & $1$ \\
\hline
 $H_{\lambda}$  & $1$ & $1$ & $1$ & $1$ & $1$ & $4$ \\
\hline
\end{tabular}
\end{center}
This matrix has rank $6$.
Thus, the intersection matrix of the lines
$L_{\{x\},\{y\}}$, $L_{\{x\},\{t\}}$, $L_{\{y\},\{z\}}$, $L_{\{y\},\{t\}}$, $L_{\{z\},\{t\}}$,
$L_{\{x\},\{y,z,t\}}$, $L_{\{y\},\{z,t\}}$, $L_{\{z\},\{y,t\}}$, and $L_{\{t\},\{x,y,z\}}$ on the surface $S_\lambda$ also has rank $6$,
because
\begin{multline*}
H_{\lambda}\sim L_{\{x\},\{y\}}+2L_{\{x\},\{t\}}+L_{\{x\},\{y,z,t\}}\sim L_{\{x\},\{y\}}+L_{\{y\},\{z\}}+L_{\{y\},\{t\}}+L_{\{y\},\{z,t\}}\sim\\
\sim L_{\{y\},\{z\}}+2L_{\{z\},\{t\}}+L_{\{z\},\{y,t\}}\sim L_{\{x\},\{t\}}+L_{\{y\},\{t\}}+L_{\{z\},\{t\}}+L_{\{t\},\{x,y,z\}}.
\end{multline*}
On the other hand, one has
$\mathrm{rk}\,\mathrm{Pic}(\widetilde{S}_{\Bbbk})=\mathrm{rk}\,\mathrm{Pic}(S_{\Bbbk})+11$.
We conclude that \eqref{equation:main-2-simple} holds.
By Lemma~\ref{lemma:cokernel}, this implies that \eqref{equation:main-2} in Main Theorem also holds.

\subsection{Family \textnumero $3.18$}
\label{section:r-3-n-18}

The threefold $X$ can be obtained by blowing up $\mathbb{P}^3$ in disjoint union of a line and a conic.
This shows that $h^{1,2}(X)=0$.
A~toric Landau--Ginzburg model of this family is given by Minkowski polynomial \textnumero $211$, which is
$$
\frac{x}{y}+x+\frac{1}{y}+z+\frac{x}{z}+y+\frac{1}{x}+\frac{y}{z}.
$$
Thus, the pencil $\mathcal{S}$ is given by the equation
$$
x^2tz+x^2zy+t^2zx+z^2yx+x^2ty+y^2zx+t^2zy+y^2tx=\lambda xyzt.
$$

As usual, we suppose that $\lambda\ne\infty$. Then
\begin{equation}
\label{equation:3-18}
\begin{split}
H_{\{x\}}\cdot S_\lambda&=L_{\{x\},\{y\}}+L_{\{x\},\{z\}}+2L_{\{x\},\{t\}},\\
H_{\{y\}}\cdot S_\lambda&=L_{\{x\},\{y\}}+L_{\{y\},\{z\}}+L_{\{y\},\{t\}}+L_{\{y\},\{x,t\}},\\
H_{\{z\}}\cdot S_\lambda&=L_{\{x\},\{z\}}+L_{\{y\},\{z\}}+L_{\{z\},\{t\}}+L_{\{z\},\{x,y\}},\\
H_{\{t\}}\cdot S_\lambda&=L_{\{x\},\{t\}}+L_{\{y\},\{t\}}+L_{\{z\},\{t\}}+L_{\{t\},\{x,y,z\}}.
\end{split}
\end{equation}
Thus, the base locus of the pencil $\mathcal{S}$ consists of
the lines $L_{\{x\},\{y\}}$, $L_{\{x\},\{z\}}$, $L_{\{x\},\{t\}}$, $L_{\{y\},\{z\}}$,
$L_{\{y\},\{t\}}$, $L_{\{z\},\{t\}}$,  $L_{\{y\},\{x,t\}}$, $L_{\{z\},\{x,y\}}$, and $L_{\{t\},\{x,y,z\}}$.

For every $\lambda\in\mathbb{C}$, the surface $S_\lambda$ is irreducible, it has isolated singularities,
and its singular points contained in the base locus of the pencil~$\mathcal{S}$ can be described as follows:
\begin{itemize}\setlength{\itemindent}{2cm}
\item[$P_{\{y\},\{z\},\{t\}}$:] type $\mathbb{A}_1$;

\item[$P_{\{x\},\{z\},\{t\}}$:] type $\mathbb{A}_2$ with quadratic term $x(z+t)$;

\item[$P_{\{x\},\{y\},\{t\}}$:] type $\mathbb{A}_3$ with quadratic term $xy$;

\item[$P_{\{x\},\{y\},\{z\}}$:] type $\mathbb{A}_3$ with quadratic term $x(x+y)$ for $\lambda\neq -1$, type $\mathbb{A}_5$ for $\lambda=-1$;

\item[$P_{\{x\},\{t\},\{y,z\}}$:] type $\mathbb{A}_1$;

\item[$P_{\{z\},\{t\},\{x,y\}}$:] type $\mathbb{A}_1$ for $\lambda\neq -2$, type $\mathbb{A}_2$ for $\lambda=-2$.
\end{itemize}
Then $[\mathsf{f}^{-1}(\lambda)]=1$ for every $\lambda\in\mathbb{C}$ by Lemma~\ref{corollary:irreducible-fibers}.
This confirms \eqref{equation:main-1} in Main Theorem.

To verify \eqref{equation:main-2} in Main Theorem,
we may assume that $\lambda\ne-1$ and $\lambda\ne-2$.
In this case, the intersection matrix of the curves $L_{\{y\},\{t\}}$, $L_{\{y\},\{x,t\}}$, $L_{\{z\},\{t\}}$, $L_{\{z\},\{x,y\}}$, $L_{\{t\},\{x,y,z\}}$, and~$H_{\lambda}$
on the surface $S_\lambda$ is given by the following table:
\begin{center}\renewcommand\arraystretch{1.42}
\begin{tabular}{|c||c|c|c|c|c|c|}
\hline
 $\bullet$  & $L_{\{y\},\{t\}}$ & $L_{\{y\},\{x,t\}}$ & $L_{\{z\},\{t\}}$ & $L_{\{z\},\{x,y\}}$ & $L_{\{t\},\{x,y,z\}}$ &  $H_{\lambda}$ \\
\hline\hline
$L_{\{y\},\{t\}}$ & $-\frac{3}{4}$ & $\frac{3}{4}$ & $\frac{1}{2}$ & $0$ & $1$ & $1$ \\
\hline
$L_{\{y\},\{x,t\}}$ & $\frac{3}{4}$ & $-\frac{5}{4}$ & $0$ & $0$ & $0$ & $1$ \\
\hline
$L_{\{z\},\{t\}}$ & $\frac{1}{2}$ & $0$ & $-\frac{1}{3}$ & $\frac{1}{2}$ & $\frac{1}{2}$ & $1$ \\
\hline
$L_{\{z\},\{x,y\}}$ & $0$ & $0$ & $\frac{1}{2}$ & $-\frac{3}{4}$ & $\frac{1}{2}$ & $1$ \\
\hline
$L_{\{t\},\{x,y,z\}}$ & $1$ & $0$ & $\frac{1}{2}$ & $\frac{1}{2}$ & $-1$ & $1$ \\
\hline
 $H_{\lambda}$  & $1$ & $1$ & $1$ & $1$ & $1$ & $4$ \\
\hline
\end{tabular}
\end{center}
This matrix has rank $6$. On the other hand, it follows from \eqref{equation:3-18} that
\begin{multline*}
H_\lambda\sim L_{\{x\},\{y\}}+L_{\{x\},\{z\}}+2L_{\{x\},\{t\}}\sim L_{\{x\},\{y\}}+L_{\{y\},\{z\}}+L_{\{y\},\{t\}}+L_{\{y\},\{x,t\}}\sim \\
\sim L_{\{x\},\{z\}}+L_{\{y\},\{z\}}+L_{\{z\},\{t\}}+L_{\{z\},\{x,y\}}\sim L_{\{x\},\{t\}}+L_{\{y\},\{t\}}+L_{\{z\},\{t\}}+L_{\{t\},\{x,y,z\}}.
\end{multline*}
Thus, the intersection matrix of the lines
$L_{\{x\},\{y\}}$, $L_{\{x\},\{z\}}$, $L_{\{x\},\{t\}}$, $L_{\{y\},\{z\}}$,
$L_{\{y\},\{t\}}$, $L_{\{z\},\{t\}}$,  $L_{\{y\},\{x,t\}}$, $L_{\{z\},\{x,y\}}$, and $L_{\{t\},\{x,y,z\}}$
on the surface $S_\lambda$ has the same rank
as the intersection matrix of the curves $L_{\{y\},\{t\}}$, $L_{\{y\},\{x,t\}}$, $L_{\{z\},\{t\}}$, $L_{\{z\},\{x,y\}}$, $L_{\{t\},\{x,y,z\}}$, and~$H_{\lambda}$.
Moreover, we have  $\mathrm{rk}\,\mathrm{Pic}(\widetilde{S}_{\Bbbk})=\mathrm{rk}\,\mathrm{Pic}(S_{\Bbbk})+11$.
Thus, we see that \eqref{equation:main-2-simple} holds, so that \eqref{equation:main-2} in Main Theorem holds by Lemma~\ref{lemma:cokernel}.

\subsection{Family \textnumero $3.19$}
\label{section:r-3-n-19}

The threefold $X$ can be obtained by blowing up a smooth quadric hypersurface in $\mathbb{P}^3$ in two points, so that $h^{1,2}(X)=0$.
A~toric Landau--Ginzburg model of this family is given by Minkowski polynomial \textnumero $74$, which is
$$
\frac{z}{x}+\frac{1}{x}+y+z+x+\frac{1}{yz}+\frac{x}{yz}.
$$
The quartic pencil $\mathcal{S}$ is given by the following equations:
$$
z^2ty+t^2yz+y^2xz+z^2xy+x^2yz+t^3x+x^2t^2=\lambda xyzt.
$$

Suppose that $\lambda\ne\infty$. Then
\begin{equation}
\label{equation:3-19}
\begin{split}
H_{\{x\}}\cdot S_\lambda&=L_{\{x\},\{y\}}+L_{\{x\},\{z\}}+L_{\{x\},\{t\}}+L_{\{x\},\{z,t\}},\\
H_{\{y\}}\cdot S_\lambda&=L_{\{x\},\{y\}}+2L_{\{y\},\{t\}}+L_{\{y\},\{x,t\}},\\
H_{\{z\}}\cdot S_\lambda&=L_{\{x\},\{z\}}+2L_{\{z\},\{t\}}+L_{\{z\},\{x,t\}},\\
H_{\{t\}}\cdot S_\lambda&=L_{\{x\},\{t\}}+L_{\{y\},\{t\}}+L_{\{z\},\{t\}}+L_{\{t\},\{x,y,z\}}.
\end{split}
\end{equation}
Thus, the base locus of the pencil $\mathcal{S}$ consists of
the lines $L_{\{x\},\{y\}}$, $L_{\{x\},\{z\}}$, $L_{\{x\},\{t\}}$,
$L_{\{y\},\{t\}}$, $L_{\{z\},\{t\}}$, $L_{\{x\},\{z,t\}}$, $L_{\{y\},\{x,t\}}$,
$L_{\{z\},\{x,t\}}$, $L_{\{t\},\{x,y,z\}}$.

For every $\lambda\in\mathbb{C}$, the quartic surface $S_\lambda$ has isolated singularities, so that it is irreducible.
Moreover, the singular points of the surface $S_\lambda$ contained in the base locus of the pencil~$\mathcal{S}$ can be described as follows:
\begin{itemize}\setlength{\itemindent}{3cm}
\item[$P_{\{x\},\{y\},\{t\}}$:] type $\mathbb{A}_4$ with quadratic term $y(x+t)$;

\item[$P_{\{x\},\{z\},\{t\}}$:] type $\mathbb{A}_4$ with quadratic term $xz$;

\item[$P_{\{y\},\{z\},\{t\}}$:] type $\mathbb{A}_1$;

\item[$P_{\{y\},\{t\},\{x,z\}}$:] type $\mathbb{A}_1$;

\item[$P_{\{z\},\{t\},\{x,y\}}$:] type $\mathbb{A}_1$.
\end{itemize}
Then each fiber $\mathsf{f}^{-1}(\lambda)$ is irreducible by Lemma~\ref{corollary:irreducible-fibers}.
This confirms \eqref{equation:main-1} in Main Theorem.

Let us verify \eqref{equation:main-2} in Main Theorem.
It follows from \eqref{equation:3-19}  that
the intersection matrix of the lines
 $L_{\{x\},\{y\}}$, $L_{\{x\},\{z\}}$, $L_{\{x\},\{t\}}$,
$L_{\{y\},\{t\}}$, $L_{\{z\},\{t\}}$, $L_{\{x\},\{z,t\}}$, $L_{\{y\},\{x,t\}}$,
$L_{\{z\},\{x,t\}}$, $L_{\{t\},\{x,y,z\}}$ on the surface $S_\lambda$ has the same rank
as the intersection matrix of the curves $L_{\{x\},\{y\}}$, $L_{\{x\},\{z,t\}}$, $L_{\{y\},\{t\}}$, $L_{\{z\},\{t\}}$, $L_{\{t\},\{x,y,z\}}$, and $H_{\lambda}$.
The latter matrix is given by
\begin{center}\renewcommand\arraystretch{1.42}
\begin{tabular}{|c||c|c|c|c|c|c|}
\hline
 $\bullet$  & $L_{\{x\},\{y\}}$ & $L_{\{x\},\{z,t\}}$ & $L_{\{y\},\{t\}}$ & $L_{\{z\},\{t\}}$ & $L_{\{t\},\{x,y,z\}}$ &  $H_{\lambda}$ \\
\hline\hline
$L_{\{x\},\{y\}}$ & $-\frac{6}{5}$ & $1$ & $\frac{4}{5}$ & $0$ & $0$ & $1$ \\
\hline
$L_{\{x\},\{z,t\}}$ & $1$ & $-\frac{6}{5}$ & $0$ & $\frac{1}{5}$ & $0$ & $1$ \\
\hline
$L_{\{y\},\{t\}}$ & $\frac{4}{5}$ & $0$ & $-\frac{1}{5}$ & $\frac{1}{2}$ & $\frac{1}{2}$ & $1$ \\
\hline
$L_{\{z\},\{t\}}$ & $0$ & $\frac{1}{5}$ & $\frac{1}{2}$ & $-\frac{1}{5}$ & $\frac{1}{2}$ & $1$ \\
\hline
$L_{\{t\},\{x,y,z\}}$ & $0$ & $1$ & $\frac{1}{2}$ & $\frac{1}{2}$ & $-1$ & $1$ \\
\hline
 $H_{\lambda}$  & $1$ & $1$ & $1$ & $1$ & $1$ & $4$ \\
\hline
\end{tabular}
\end{center}
The rank of this matrix is $6$.
Moreover, we have $\mathrm{rk}\,\mathrm{Pic}(\widetilde{S}_{\Bbbk})=\mathrm{rk}\,\mathrm{Pic}(S_{\Bbbk})+11$.
Thus, we conclude that \eqref{equation:main-2-simple} holds.
Then \eqref{equation:main-2} in Main Theorem holds by Lemma~\ref{lemma:cokernel}.

\subsection{Family \textnumero $3.20$}
\label{section:r-3-n-20}

In this case, the threefold $X$ is a blow up of the smooth quadric threefold along a disjoint union of two lines,
so that $h^{1,2}(X)=0$.
A~toric Landau--Ginzburg model of this family is given by Minkowski polynomial \textnumero $79$, which is
$$
\frac{y}{x}+\frac{1}{x}+y+z+\frac{1}{y}+x+\frac{x}{yz}.
$$
The quartic pencil $\mathcal{S}$ is given by the following equation:
$$
y^2tz+t^2yz+y^2xz+z^2xy+t^2xz+x^2yz+x^2t^2=\lambda xyzt.
$$

As usual, we suppose that $\lambda\ne\infty$. Then
\begin{equation}
\label{equation:3-20}
\begin{split}
H_{\{x\}}\cdot S_\lambda&=L_{\{x\},\{y\}}+L_{\{x\},\{z\}}+L_{\{x\},\{t\}}+L_{\{x\},\{y,t\}},\\
H_{\{y\}}\cdot S_\lambda&=L_{\{x\},\{y\}}+2L_{\{y\},\{t\}}+L_{\{y\},\{x,z\}},\\
H_{\{z\}}\cdot S_\lambda&=2L_{\{x\},\{z\}}+2L_{\{z\},\{t\}},\\
H_{\{t\}}\cdot S_\lambda&=L_{\{x\},\{t\}}+L_{\{y\},\{t\}}+L_{\{z\},\{t\}}+L_{\{t\},\{x,y,z\}}.
\end{split}
\end{equation}
Thus, the base locus of the pencil $\mathcal{S}$ consists of
the lines $L_{\{x\},\{y\}}$, $L_{\{x\},\{z\}}$, $L_{\{x\},\{t\}}$, $L_{\{y\},\{t\}}$,
$L_{\{z\},\{t\}}$, $L_{\{x\},\{y,t\}}$, $L_{\{y\},\{x,z\}}$, and $L_{\{t\},\{x,y,z\}}$.

For every $\lambda\in\mathbb{C}$, the surface $S_\lambda$ has isolated singularities, so that it is irreducible.
The singular points of the surface~$S_\lambda$ contained in the base locus of the pencil~$\mathcal{S}$ can be described as follows:
\begin{itemize}\setlength{\itemindent}{3cm}
\item[$P_{\{x\},\{y\},\{z\}}$:] type $\mathbb{A}_1$;

\item[$P_{\{x\},\{y\},\{t\}}$:] type $\mathbb{A}_3$ with quadratic term $xy$;

\item[$P_{\{x\},\{z\},\{t\}}$:] type $\mathbb{A}_3$ with quadratic term $z(x+t)$;

\item[$P_{\{y\},\{z\},\{t\}}$:] type $\mathbb{A}_1$;

\item[$P_{\{x\},\{y\},\{z,t\}}$:] type $\mathbb{A}_1$ for $\lambda\neq-1$, type $\mathbb{A}_2$ for $\lambda=-1$;

\item[$P_{\{y\},\{t\},\{x,z\}}$:] type $\mathbb{A}_2$ with quadratic term
$$
y(x+y+z-\lambda t)
$$
for $\lambda\neq 0$, type $\mathbb{A}_3$ for $\lambda=0$;

\item[$P_{\{z\},\{t\},\{x,y\}}$:] type $\mathbb{A}_1$.
\end{itemize}
By Lemma~\ref{corollary:irreducible-fibers}, each fiber $\mathsf{f}^{-1}(\lambda)$ is irreducible.
This confirms \eqref{equation:main-1} in Main Theorem.

If $\lambda\ne 0$ and $\lambda\ne -1$, then the intersection matrix of the curves $L_{\{x\},\{y\}}$, $L_{\{x\},\{y,t\}}$, $L_{\{y\},\{x,z\}}$, $L_{\{t\},\{x,y,z\}}$, and $H_{\lambda}$
on the surface $S_\lambda$ is given by the following table:
\begin{center}\renewcommand\arraystretch{1.42}
\begin{tabular}{|c||c|c|c|c|c|}
\hline
 $\bullet$  & $L_{\{x\},\{y\}}$ & $L_{\{x\},\{y,t\}}$ & $L_{\{y\},\{x,z\}}$ & $L_{\{t\},\{x,y,z\}}$ &  $H_{\lambda}$ \\
\hline\hline
$L_{\{x\},\{y\}}$ & $0$ & $\frac{1}{2}$ & $\frac{1}{2}$ & $0$ & $1$ \\
\hline
$L_{\{x\},\{y,t\}}$ & $\frac{1}{2}$ & $-\frac{5}{4}$ & $0$ & $0$ & $1$ \\
\hline
$L_{\{y\},\{x,z\}}$ & $\frac{1}{2}$ & $0$ & $-\frac{5}{6}$ & $\frac{1}{3}$ & $1$ \\
\hline
$L_{\{t\},\{x,y,z\}}$ & $0$ & $0$ & $\frac{1}{3}$ & $-\frac{5}{6}$ & $1$ \\
\hline
 $H_{\lambda}$  & $1$ & $1$ & $1$ & $1$ & $4$ \\
\hline
\end{tabular}
\end{center}
The rank of this matrix is $5$.
On the other hand, it follows from \eqref{equation:3-20} that
\begin{multline*}
L_{\{x\},\{y\}}+L_{\{x\},\{z\}}+L_{\{x\},\{t\}}+L_{\{x\},\{y,t\}}\sim L_{\{x\},\{y\}}+2L_{\{y\},\{t\}}+L_{\{y\},\{x,z\}}\sim \\
\sim 2L_{\{x\},\{z\}}+2L_{\{z\},\{t\}}\sim L_{\{x\},\{t\}}+L_{\{y\},\{t\}}+L_{\{z\},\{t\}}+L_{\{t\},\{x,y,z\}}\sim H_\lambda
\end{multline*}
on the surface $S_\lambda$.
Thus, if $\lambda\ne 0$ and $\lambda\ne -1$, then the rank of the intersection matrix of the lines
$L_{\{x\},\{y\}}$, $L_{\{x\},\{z\}}$, $L_{\{x\},\{t\}}$, $L_{\{y\},\{t\}}$,
$L_{\{z\},\{t\}}$, $L_{\{x\},\{y,t\}}$, $L_{\{y\},\{x,z\}}$, and $L_{\{t\},\{x,y,z\}}$
on the surface $S_\lambda$ is also $5$.
Moreover, we have $\mathrm{rk}\,\mathrm{Pic}(\widetilde{S}_{\Bbbk})=\mathrm{rk}\,\mathrm{Pic}(S_{\Bbbk})+12$.
Thus, we see that \eqref{equation:main-2-simple} holds, so that \eqref{equation:main-2} in Main Theorem holds by Lemma~\ref{lemma:cokernel}.

\subsection{Family \textnumero $3.21$}
\label{section:r-3-n-21}

In this case, the threefold $X$ is  a blow up of $\mathbb{P}^1\times\mathbb{P}^2$ in a curve of bidegree $(2,1)$, so that $h^{1,2}(X)=0$.
A~toric Landau--Ginzburg model of this family is given by Minkowski polynomial \textnumero $213$, which is
$$
\frac{z}{y}+x+\frac{1}{y}+z+\frac{z}{xy}+\frac{1}{z}+y+\frac{1}{x}.
$$
The quartic pencil $\mathcal{S}$ is given by the equation
$$
z^2xt+x^2yz+t^2xz+z^2xy+t^2z^2+t^2xy+y^2xz+t^2yz=\lambda xyzt.
$$

Suppose that $\lambda\ne\infty$.
Let $\mathcal{C}$ be the conic in $\mathbb{P}^3$ that is given by $y=xz+xt+zt=0$.
Then
\begin{equation}
\label{equation:3-21}
\begin{split}
H_{\{x\}}\cdot S_\lambda&=L_{\{x\},\{z\}}+2L_{\{x\},\{t\}}+L_{\{x\},\{y,z\}},\\
H_{\{y\}}\cdot S_\lambda&=L_{\{y\},\{z\}}+L_{\{y\},\{t\}}+\mathcal{C},\\
H_{\{z\}}\cdot S_\lambda&=L_{\{x\},\{z\}}+L_{\{y\},\{z\}}+2L_{\{z\},\{t\}},\\
H_{\{t\}}\cdot S_\lambda&=L_{\{x\},\{t\}}+L_{\{y\},\{t\}}+L_{\{z\},\{t\}}+L_{\{t\},\{x,y,z\}}.
\end{split}
\end{equation}
Thus, the base locus of the pencil $\mathcal{S}$ consists of the curves $L_{\{x\},\{z\}}$ $L_{\{x\},\{t\}}$,
$L_{\{y\},\{z\}}$, $L_{\{y\},\{t\}}$, $L_{\{z\},\{t\}}$, $L_{\{x\},\{y,z\}}$, $L_{\{t\},\{x,y,z\}}$, and $\mathcal{C}$.

For every $\lambda\in\mathbb{C}$, the surface $S_\lambda$ has isolated singularities, so that it is irreducible.
Moreover, the singular points of the surface $S_\lambda$ contained in the base locus of the pencil~$\mathcal{S}$ can be described as follows:
\begin{itemize}\setlength{\itemindent}{1.5cm}

\item[$P_{\{x\},\{y\},\{t\}}$:] type $\mathbb{A}_1$;

\item[$P_{\{x\},\{y\},\{z\}}$:] type $\mathbb{A}_2$ with quadratic term $(x+z)(y+z)$ for $\lambda\neq-1$, type $\mathbb{A}_3$ for $\lambda=-1$;

\item[$P_{\{x\},\{z\},\{t\}}$:] type $\mathbb{A}_3$ with quadratic term $xz$;

\item[$P_{\{x\},\{t\},\{y,z\}}$:] type $\mathbb{A}_2$ with quadratic term
$$
x(x+y+z-t-\lambda t)
$$
for $\lambda\neq-1$, type $\mathbb{A}_3$ for $\lambda=-1$;

\item[$P_{\{y\},\{z\},\{t\}}$:] type $\mathbb{A}_3$ with quadratic term $yz$;

\item[$P_{\{z\},\{t\},\{x,y\}}$:] type $\mathbb{A}_1$.
\end{itemize}
By Lemma~\ref{corollary:irreducible-fibers}, each fiber $\mathsf{f}^{-1}(\lambda)$ is irreducible.
This confirms \eqref{equation:main-1} in Main Theorem.

Now let us show that \eqref{equation:main-2} in Main Theorem also holds in this case.
To do this, we may assume that $\lambda\ne-1$.
Then the intersection matrix of the curves $L_{\{x\},\{z\}}$, $L_{\{x\},\{y,z\}}$, $L_{\{y\},\{t\}}$, $L_{\{t\},\{x,y,z\}}$, and $H_{\lambda}$
on the surface $S_\lambda$ is given by
\begin{center}\renewcommand\arraystretch{1.42}
\begin{tabular}{|c||c|c|c|c|c|}
\hline
 $\bullet$  & $L_{\{x\},\{z\}}$ & $L_{\{x\},\{y,z\}}$ & $L_{\{y\},\{t\}}$ & $L_{\{t\},\{x,y,z\}}$ &  $H_{\lambda}$ \\
\hline\hline
$L_{\{x\},\{z\}}$ & $-\frac{1}{3}$ & $\frac{1}{3}$ & $0$ & $0$ & $1$ \\
\hline
$L_{\{x\},\{y,z\}}$ & $\frac{1}{3}$ & $-\frac{2}{3}$ & $0$ & $\frac{1}{3}$ & $1$ \\
\hline
$L_{\{y\},\{t\}}$ & $0$ & $0$ & $-\frac{3}{4}$ & $1$ & $1$ \\
\hline
$L_{\{t\},\{x,y,z\}}$ & $0$ & $\frac{1}{3}$ & $1$ & $-\frac{5}{6}$ & $1$ \\
\hline
 $H_{\lambda}$  & $1$ & $1$ & $1$ & $1$ & $4$ \\
\hline
\end{tabular}
\end{center}
The rank of this intersection matrix is $5$.
On the other hand, it follows from \eqref{equation:3-21} that
\begin{multline*}
H_{\lambda}\sim L_{\{x\},\{z\}}+2L_{\{x\},\{t\}}+L_{\{x\},\{y,z\}}\sim L_{\{y\},\{z\}}+L_{\{y\},\{t\}}+\mathcal{C}\sim\\
\sim L_{\{x\},\{z\}}+L_{\{y\},\{z\}}+2L_{\{z\},\{t\}}\sim L_{\{x\},\{t\}}+L_{\{y\},\{t\}}+L_{\{z\},\{t\}}+L_{\{t\},\{x,y,z\}}
\end{multline*}
on the surface $S_\lambda$.
Thus, the rank of the intersection matrix of the curves
$L_{\{x\},\{z\}}$ $L_{\{x\},\{t\}}$,
$L_{\{y\},\{z\}}$, $L_{\{y\},\{t\}}$, $L_{\{z\},\{t\}}$, $L_{\{x\},\{y,z\}}$, $L_{\{t\},\{x,y,z\}}$, and $\mathcal{C}$
on the surface $S_\lambda$ is also $5$.
Moreover, we have $\mathrm{rk}\,\mathrm{Pic}(\widetilde{S}_{\Bbbk})=\mathrm{rk}\,\mathrm{Pic}(S_{\Bbbk})+12$.
Thus, we see that \eqref{equation:main-2-simple} holds, so that \eqref{equation:main-2} in Main Theorem holds by Lemma~\ref{lemma:cokernel}.

\subsection{Family \textnumero $3.22$}
\label{section:r-3-n-22}

In this case, the threefold $X$ is  a blow up of $\mathbb{P}^1\times\mathbb{P}^2$ in a conic contained
in a fiber of the projection $\mathbb{P}^1\times\mathbb{P}^2\to\mathbb{P}^1$.
Thus, we have $h^{1,2}(X)=0$.
A~toric Landau--Ginzburg model of this family is given by Minkowski polynomial \textnumero $75$, which is
$$
\frac{z}{x}+\frac{1}{x}+y+z+\frac{1}{xyz}+x+\frac{1}{yz}.
$$
The quartic pencil $\mathcal{S}$ is given by
$$
z^2ty+t^2yz+y^2xz+z^2xy+t^4+x^2yz+t^3x=\lambda xyzt.
$$

Let $\mathcal{C}$ be a cubic curve in $\mathbb{P}^3$ that is given by $x=yz^2+yzt+t^3=0$.
Then $\mathcal{C}$ is singular at the point $P_{\{x\},\{z\},\{t\}}$.
Moreover, if $\lambda\ne\infty$, then
\begin{equation}
\label{equation:3-22}
\begin{split}
H_{\{x\}}\cdot S_\lambda&=L_{\{x\},\{t\}}+\mathcal{C},\\
H_{\{y\}}\cdot S_\lambda&=3L_{\{y\},\{t\}}+L_{\{y\},\{x,t\}},\\
H_{\{z\}}\cdot S_\lambda&=3L_{\{z\},\{t\}}+L_{\{z\},\{x,t\}},\\
H_{\{t\}}\cdot S_\lambda&=L_{\{x\},\{t\}}+L_{\{y\},\{t\}}+L_{\{z\},\{t\}}+L_{\{t\},\{x,y,z\}}.
\end{split}
\end{equation}
Therefore, the base locus of the pencil $\mathcal{S}$ consists of
the curves $L_{\{x\},\{t\}}$, $L_{\{y\},\{t\}}$, $L_{\{z\},\{t\}}$, $L_{\{y\},\{x,t\}}$, $L_{\{z\},\{x,t\}}$,
$L_{\{t\},\{x,y,z\}}$, and $\mathcal{C}$.

For every $\lambda\ne\infty$, the surface $S_\lambda$ has isolated singularities, which implies that $S_\lambda$ is irreducible.
In~this case, the singular points of the surface $S_\lambda$ contained in the base locus of the pencil~$\mathcal{S}$ can be described as follows:
\begin{itemize}\setlength{\itemindent}{2cm}
\item[$P_{\{x\},\{y\},\{t\}}$:] type $\mathbb{A}_4$ with quadratic term $y(x+t)$ for $\lambda\neq-2$, type $\mathbb{A}_6$ for $\lambda=-2$;

\item[$P_{\{x\},\{z\},\{t\}}$:] type $\mathbb{A}_3$ with quadratic term $xz$;

\item[$P_{\{y\},\{z\},\{t\}}$:] type $\mathbb{A}_2$ with quadratic term $yz$;

\item[$P_{\{y\},\{t\},\{x,z\}}$:] type $\mathbb{A}_2$ with quadratic term $y(x+y+z-t-\lambda t)$;

\item[$P_{\{z\},\{t\},\{x,y\}}$:] type $\mathbb{A}_2$ with quadratic term $z(x+y+z-\lambda t)$.
\end{itemize}
Thus, it follows from Lemma~\ref{corollary:irreducible-fibers} that $[\mathsf{f}^{-1}(\lambda)]=1$ for every $\lambda\in\mathbb{C}$.
This confirms \eqref{equation:main-1} in Main Theorem, since $h^{1,2}(X)=0$.

Let us verify \eqref{equation:main-2} in Main Theorem. If $\lambda\ne\infty$, then
\begin{multline*}
H_{\lambda}\sim L_{\{x\},\{t\}}+\mathcal{C}\sim 3L_{\{y\},\{t\}}+L_{\{y\},\{x,t\}}\sim \\
\sim 3L_{\{z\},\{t\}}+L_{\{z\},\{x,t\}}\sim L_{\{x\},\{t\}}+L_{\{y\},\{t\}}+L_{\{z\},\{t\}}+L_{\{t\},\{x,y,z\}}
\end{multline*}
on the surface $S_\lambda$. This follows from \eqref{equation:3-22}.
Thus, if $\lambda\ne\infty$, then the intersection matrix of
$L_{\{x\},\{t\}}$, $L_{\{y\},\{t\}}$, $L_{\{z\},\{t\}}$, $L_{\{y\},\{x,t\}}$, $L_{\{z\},\{x,t\}}$,
$L_{\{t\},\{x,y,z\}}$, and $\mathcal{C}$ on the surface $S_\lambda$
has the same rank as the intersection matrix of the curves $L_{\{y\},\{x,t\}}$, $L_{\{z\},\{x,t\}}$, $L_{\{t\},\{x,y,z\}}$, and $H_{\lambda}$.
If $\lambda\ne\infty$ and $\lambda\ne-2$, then the latter matrix is given by the following table:
\begin{center}\renewcommand\arraystretch{1.42}
\begin{tabular}{|c||c|c|c|c|}
\hline
 $\bullet$  & $L_{\{y\},\{x,t\}}$ & $L_{\{z\},\{x,t\}}$ & $L_{\{t\},\{x,y,z\}}$ &  $H_{\lambda}$ \\
\hline\hline
$L_{\{y\},\{x,t\}}$ & $-\frac{4}{5}$ & $1$ & $0$ & $1$ \\
\hline
$L_{\{z\},\{x,t\}}$ & $1$ & $-\frac{5}{4}$ & $0$ & $1$ \\
\hline
$L_{\{t\},\{x,y,z\}}$ & $0$ & $0$ & $-\frac{2}{3}$ & $1$ \\
\hline
 $H_{\lambda}$  & $1$ & $1$ & $1$ & $4$ \\
\hline
\end{tabular}
\end{center}
The rank of this matrix is $4$.
On the other hand, we have $\mathrm{rk}\,\mathrm{Pic}(\widetilde{S}_{\Bbbk})=\mathrm{rk}\,\mathrm{Pic}(S_{\Bbbk})+13$.
Thus, we see that \eqref{equation:main-2-simple} holds, so that \eqref{equation:main-2} in Main Theorem holds by Lemma~\ref{lemma:cokernel}.

\subsection{Family \textnumero $3.23$}
\label{section:r-3-n-23}

In this case, the threefold $X$ is  a blow up of $\mathbb{P}^3$ blown up at a point at the proper transform of a conic passing through this point.
A~toric Landau--Ginzburg model of this family is given by Minkowski polynomial \textnumero $76$, which is
$$
\frac{z}{x}+\frac{1}{x}+y+z+\frac{1}{xy}+x+\frac {1}{yz}.
$$
The pencil $\mathcal{S}$ is given by the following equation:
$$
z^2ty+t^2yz+y^2xz+z^2xy+t^3z+x^2yz+t^3x=\lambda xyzt.
$$

Suppose that $\lambda\ne\infty$.
Let $\mathcal{C}$ be the conic in $\mathbb{P}^3$ that is given by $x=yz+yt+t^2=0$.
Then
\begin{equation}
\label{equation:3-23}
\begin{split}
H_{\{x\}}\cdot S_\lambda&=L_{\{x\},\{z\}}+L_{\{x\},\{t\}}+\mathcal{C},\\
H_{\{y\}}\cdot S_\lambda&=3L_{\{y\},\{t\}}+L_{\{y\},\{x,z\}},\\
H_{\{z\}}\cdot S_\lambda&=L_{\{x\},\{z\}}+3L_{\{z\},\{t\}},\\
H_{\{t\}}\cdot S_\lambda&=L_{\{x\},\{t\}}+L_{\{y\},\{t\}}+L_{\{z\},\{t\}}+L_{\{t\},\{x,y,z\}}.
\end{split}
\end{equation}
Thus, the base locus of the pencil $\mathcal{S}$ consists of the curves $L_{\{x\},\{z\}}$, $L_{\{x\},\{t\}}$,
$L_{\{y\},\{t\}}$,  $L_{\{z\},\{t\}}$, $L_{\{y\},\{x,z\}}$, $L_{\{t\},\{x,y,z\}}$, and $\mathcal{C}$.

Observe that $S_\lambda$ has isolated singularities.
In particular, it is irreducible.
Moreover, its singular points contained in the base locus of the pencil $\mathcal{S}$ can be described as follows:
\begin{itemize}\setlength{\itemindent}{3cm}
\item[$P_{\{x\},\{y\},\{t\}}$:] type $\mathbb{A}_2$ with quadratic term $y(x+t)$;

\item[$P_{\{x\},\{z\},\{t\}}$:] type $\mathbb{A}_4$ with quadratic term $xz$;

\item[$P_{\{y\},\{z\},\{t\}}$:] type $\mathbb{A}_2$ with quadratic term $yz$;

\item[$P_{\{y\},\{t\},\{x,z\}}$:] type $\mathbb{A}_3$ with quadratic term
$$
y(x+y+z-t-\lambda t)
$$
for $\lambda\neq-1$, type $\mathbb{A}_4$ for $\lambda=-1$;

\item[$P_{\{z\},\{t\},\{x,y\}}$:] type $\mathbb{A}_2$ with quadratic term $z(x+y+z-\lambda t)$.
\end{itemize}
Thus, by Lemma~\ref{corollary:irreducible-fibers}, every fiber $\mathsf{f}^{-1}(\lambda)$ is irreducible.
This confirms \eqref{equation:main-1} in Main Theorem, since $h^{1,2}(X)=0$.

To check \eqref{equation:main-2} in Main Theorem, we may assume that $\lambda\ne-1$.
Then the intersection matrix of the curves $L_{\{x\},\{t\}}$, $L_{\{y\},\{x,z\}}$, $L_{\{t\},\{x,y,z\}}$, and $H_{\lambda}$
on the surface $S_\lambda$ is given by
\begin{center}\renewcommand\arraystretch{1.42}
\begin{tabular}{|c||c|c|c|c|}
\hline
 $\bullet$  & $L_{\{x\},\{t\}}$ & $L_{\{y\},\{x,z\}}$ & $L_{\{t\},\{x,y,z\}}$ &  $H_{\lambda}$ \\
\hline\hline
$L_{\{x\},\{t\}}$ & $-\frac{8}{15}$ & $0$ & $1$ & $1$ \\
\hline
$L_{\{y\},\{x,z\}}$ & $0$ & $-\frac{5}{4}$ & $\frac{1}{4}$ & $1$ \\
\hline
$L_{\{t\},\{x,y,z\}}$ & $1$ & $\frac{1}{4}$ & $-\frac{1}{2}$ & $1$ \\
\hline
 $H_{\lambda}$  & $1$ & $1$ & $1$ & $4$ \\
\hline
\end{tabular}
\end{center}
This matrix has rank $4$.
On the other hand, it follows from \eqref{equation:3-23} that
\begin{multline*}
H_\lambda\sim L_{\{x\},\{z\}}+L_{\{x\},\{t\}}+\mathcal{C}\sim 3L_{\{y\},\{t\}}+L_{\{y\},\{x,z\}}\sim \\
\sim L_{\{x\},\{z\}}+3L_{\{z\},\{t\}}\sim L_{\{x\},\{t\}}+L_{\{y\},\{t\}}+L_{\{z\},\{t\}}+L_{\{t\},\{x,y,z\}}
\end{multline*}
on the surface $S_\lambda$.
Hence, the rank of the intersection matrix of the curves
$L_{\{x\},\{z\}}$, $L_{\{x\},\{t\}}$,
$L_{\{y\},\{t\}}$,  $L_{\{z\},\{t\}}$, $L_{\{y\},\{x,z\}}$, $L_{\{t\},\{x,y,z\}}$, and $\mathcal{C}$
on the surface $S_\lambda$ is also $4$.
Using the description of the singular points of the surface $S_\lambda$, we see that
$\mathrm{rk}\,\mathrm{Pic}(\widetilde{S}_{\Bbbk})=\mathrm{rk}\,\mathrm{Pic}(S_{\Bbbk})+13$.
Thus, we see that \eqref{equation:main-2-simple} holds, so that \eqref{equation:main-2} in Main Theorem holds by Lemma~\ref{lemma:cokernel}.

\subsection{Family \textnumero $3.24$}
\label{section:r-3-n-24}

In this case, a toric Landau--Ginzburg model is given by Minkowski polynomial \textnumero $77$, which is
$$
x+y+z+\frac{y}{x}+\frac{1}{y}+\frac{1}{x}+\frac{1}{xyz}.
$$
Thus, the quartic pencil $\mathcal{S}$ is given by
$$
x^2yz+y^2xz+z^2xy+y^2tz+t^2xz+t^2yz+t^4=\lambda xyzt.
$$

Let $\mathcal{C}_1$ be the cubic curve in $\mathbb{P}^3$ that is given by $x=y^2z+yzt+t^3=0$.
Then $\mathcal{C}_1$ is singular at the point $P_{\{x\},\{y\},\{t\}}$, but its proper transform on $U$ is a smooth rational curve.
Let $\mathcal{C}_2$ be the conic in $\mathbb{P}^3$ that is given by $y=xz+t^2=0$.
If $\lambda\ne\infty$, then
\begin{equation}
\label{equation:3-24}
\begin{split}
H_{\{x\}}\cdot S_\lambda&=L_{\{x\},\{t\}}+\mathcal{C}_1, \\
H_{\{y\}}\cdot S_\lambda&=2L_{\{y\},\{t\}}+\mathcal{C}_2,\\
H_{\{z\}}\cdot S_\lambda&=4L_{\{z\},\{t\}},\\
H_{\{t\}}\cdot S_\lambda&=L_{\{x\},\{t\}}+L_{\{y\},\{t\}}+L_{\{z\},\{t\}}+L_{\{t\},\{x,y,z\}}.
\end{split}
\end{equation}
Thus, the base locus of the pencil $\mathcal{S}$ is a union of the curves
$L_{\{x\},\{t\}}$, $L_{\{y\},\{t\}}$, $L_{\{z\},\{t\}}$, $L_{\{t\},\{x,y,z\}}$, $\mathcal{C}_1$, and $\mathcal{C}_2$,

If $\lambda\ne\infty$, then the quartic surface $S_\lambda$ has isolated singularities, so that it is irreducible.
In~this case, its singular points contained in the base locus of the pencil $\mathcal{S}$ can be described as follows:
\begin{itemize}\setlength{\itemindent}{3cm}
\item[$P_{\{x\},\{y\},\{t\}}$:] type $\mathbb{A}_3$ with quadratic term $xy$;

\item[$P_{\{x\},\{z\},\{t\}}$:] type $\mathbb{A}_3$ with quadratic term $z(x+t)$;

\item[$P_{\{y\},\{z\},\{t\}}$:] type $\mathbb{A}_3$ with quadratic term $yz$;

\item[$P_{\{x\},\{t\},\{y,z\}}$:] type $\mathbb{A}_1$ for $\lambda\neq-\frac{5}{2}$, type $\mathbb{A}_2$ for $\lambda=-\frac{5}{2}$;

\item[$P_{\{y\},\{t\},\{x,z\}}$:] type $\mathbb{A}_1$;

\item[$P_{\{z\},\{t\},\{x,y\}}$:] type $\mathbb{A}_3$ with quadratic term $z(x+y+z-t-\lambda t)$.
\end{itemize}
Thus, by Lemma~\ref{corollary:irreducible-fibers}, every fiber $\mathsf{f}^{-1}(\lambda)$ is irreducible.
This confirms \eqref{equation:main-1} in Main Theorem, since $h^{1,2}(X)=0$ in this case.

If $\lambda\ne\infty$, then it follows from \eqref{equation:3-24} that
\begin{multline*}
L_{\{x\},\{t\}}+\mathcal{C}_1\sim 2L_{\{y\},\{t\}}+\mathcal{C}_2\sim 4L_{\{z\},\{t\}}\sim L_{\{x\},\{t\}}+L_{\{y\},\{t\}}+L_{\{z\},\{t\}}+L_{\{t\},\{x,y,z\}}\sim H_\lambda.
\end{multline*}
In this case, the intersection matrix of the curves
$L_{\{x\},\{t\}}$, $L_{\{y\},\{t\}}$, $L_{\{z\},\{t\}}$, $L_{\{t\},\{x,y,z\}}$, $\mathcal{C}_1$, and $\mathcal{C}_2$
on the surface $S_\lambda$ has the same rank as the intersection matrix of the curves $L_{\{x\},\{t\}}$, $L_{\{t\},\{x,y,z\}}$, and $H_{\lambda}$.
On the other hand, if $\lambda\ne\infty$ and $\lambda\ne-\frac{5}{2}$,
then the latter matrix is given by
\begin{center}\renewcommand\arraystretch{1.42}
\begin{tabular}{|c||c|c|c|}
\hline
 $\bullet$  & $L_{\{x\},\{t\}}$ & $L_{\{t\},\{x,y,z\}}$ &  $H_{\lambda}$ \\
\hline\hline
$L_{\{x\},\{t\}}$ & $0$ & $\frac{1}{2}$ & $1$ \\
\hline
$L_{\{t\},\{x,y,z\}}$ & $\frac{1}{2}$ & $-\frac{1}{4}$ & $1$ \\
\hline
 $H_{\lambda}$  & $1$ & $1$ & $4$ \\
\hline
\end{tabular}
\end{center}
The rank of this matrix is $3$.
Thus, if $\lambda\ne\infty$ and $\lambda\ne-\frac{5}{2}$,
then the rank of the intersection matrix of the curves
$L_{\{x\},\{t\}}$, $L_{\{y\},\{t\}}$, $L_{\{z\},\{t\}}$, $L_{\{t\},\{x,y,z\}}$, $\mathcal{C}_1$, and $\mathcal{C}_2$
on the surface $S_\lambda$ is also~$3$.
This implies \eqref{equation:main-2-simple},
because $\mathrm{rk}\,\mathrm{Pic}(\widetilde{S}_{\Bbbk})=\mathrm{rk}\,\mathrm{Pic}(S_{\Bbbk})+14$.
Then \eqref{equation:main-2} in Main Theorem holds by Lemma~\ref{lemma:cokernel}.

\subsection{Family \textnumero $3.25$}
\label{section:r-3-n-25}

In this case, the threefold $X$ is  a blow up of $\mathbb{P}^3$ in a disjoint union of two lines.
Thus, we have $h^{1,2}(X)=0$.
A~toric Landau--Ginzburg model of this family is given by Minkowski polynomial \textnumero $24$, which is
$$
x+y+z+\frac{x}{z}+\frac{1}{x}+\frac{1}{xy}.
$$
Hence, the quartic pencil $\mathcal{S}$ is given by the following equation:
$$
x^2yz+y^2xz+z^2xy+x^2ty+t^2yz+t^3z=\lambda xyzt.
$$

Suppose that $\lambda\ne\infty$. Then
\begin{equation}
\label{equation:3-25}
\begin{split}
H_{\{x\}}\cdot S_\lambda&=L_{\{x\},\{z\}}+2L_{\{x\},\{t\}}+L_{\{x\},\{y,t\}},\\
H_{\{y\}}\cdot S_\lambda&=L_{\{y\},\{z\}}+3L_{\{y\},\{t\}},\\
H_{\{z\}}\cdot S_\lambda&=2L_{\{x\},\{z\}}+L_{\{y\},\{z\}}+L_{\{z\},\{t\}},\\
H_{\{t\}}\cdot S_\lambda&=L_{\{x\},\{t\}}+L_{\{y\},\{t\}}+L_{\{z\},\{t\}}+L_{\{t\},\{x,y,z\}}.
\end{split}
\end{equation}
Thus, the base locus of the pencil $\mathcal{S}$ consists of
the lines $L_{\{x\},\{z\}}$, $L_{\{x\},\{t\}}$, $L_{\{y\},\{z\}}$, $L_{\{y\},\{t\}}$,
$L_{\{z\},\{t\}}$, $L_{\{x\},\{y,t\}}$, $L_{\{t\},\{x,y,z\}}$.

Observe that $S_\lambda$ has isolated singularities, so that it is irreducible.
Moreover, its singular points contained in the base locus of the pencil $\mathcal{S}$ can be described as follows:
\begin{itemize}\setlength{\itemindent}{3cm}
\item[$P_{\{x\},\{y\},\{t\}}$:] type $\mathbb{A}_2$ with quadratic term $xy$;

\item[$P_{\{x\},\{z\},\{t\}}$:] type $\mathbb{A}_4$ with quadratic term $xz$;

\item[$P_{\{y\},\{z\},\{t\}}$:] type $\mathbb{A}_3$ with quadratic term $y(z+t)$;

\item[$P_{\{x\},\{z\},\{y,t\}}$:] type $\mathbb{A}_1$;

\item[$P_{\{x\},\{t\},\{y,z\}}$:] type $\mathbb{A}_1$;

\item[$P_{\{y\},\{t\},\{x,z\}}$:] type $\mathbb{A}_2$  with quadratic term $y(x+y+z+(\lambda+1)t)$.
\end{itemize}
Therefore, by Lemma~\ref{corollary:irreducible-fibers}, every fiber $\mathsf{f}^{-1}(\lambda)$ is irreducible.
This confirms \eqref{equation:main-1} in Main Theorem, since $h^{1,2}(X)=0$.

Let us verify \eqref{equation:main-2} in Main Theorem. It follows from \eqref{equation:3-25} that
\begin{multline*}
H_\lambda\sim L_{\{x\},\{z\}}+2L_{\{x\},\{t\}}+L_{\{x\},\{y,t\}}\sim L_{\{y\},\{z\}}+3L_{\{y\},\{t\}}\sim \\
\sim 2L_{\{x\},\{z\}}+L_{\{y\},\{z\}}+L_{\{z\},\{t\}}\sim L_{\{x\},\{t\}}+L_{\{y\},\{t\}}+L_{\{z\},\{t\}}+L_{\{t\},\{x,y,z\}}
\end{multline*}
on the surface $S_\lambda$.
Thus, the  intersection matrix of the lines $L_{\{x\},\{z\}}$, $L_{\{x\},\{t\}}$, $L_{\{y\},\{z\}}$, $L_{\{y\},\{t\}}$,
$L_{\{z\},\{t\}}$, $L_{\{x\},\{y,t\}}$, $L_{\{t\},\{x,y,z\}}$ has the same rank as
the intersection matrix of the curves $L_{\{x\},\{y,t\}}$, $L_{\{y\},\{z\}}$, $L_{\{t\},\{x,y,z\}}$, and $H_{\lambda}$.
The latter matrix is given by
\begin{center}\renewcommand\arraystretch{1.42}
\begin{tabular}{|c||c|c|c|c|}
\hline
 $\bullet$  & $L_{\{x\},\{y,t\}}$ & $L_{\{z\},\{t\}}$ & $L_{\{t\},\{x,y,z\}}$ &  $H_{\lambda}$ \\
\hline\hline
$L_{\{x\},\{y,t\}}$ & $-\frac{5}{6}$ & $0$ & $0$ & $1$ \\
\hline
$L_{\{y\},\{t\}}$ & $0$ &  $-\frac{4}{3}$ & $\frac{1}{2}$ & $1$ \\
\hline
$L_{\{t\},\{x,y,z\}}$ & $0$ &  $\frac{1}{2}$ & $-\frac{1}{2}$ & $1$ \\
\hline
 $H_{\lambda}$  & $1$ & $1$ & $1$ & $4$ \\
\hline
\end{tabular}
\end{center}
This matrix has rank $4$. This gives \eqref{equation:main-2-simple}, since
$\mathrm{rk}\,\mathrm{Pic}(\widetilde{S}_{\Bbbk})=\mathrm{rk}\,\mathrm{Pic}(S_{\Bbbk})+13$.
Thus, we see that  \eqref{equation:main-2} in Main Theorem holds by Lemma~\ref{lemma:cokernel}.

\subsection{Family \textnumero $3.26$}
\label{section:r-3-n-26}

The threefold can be obtained from $\mathbb{P}^3$ by blowing up disjoint union of a point and a line, so that $h^{1,2}(X)=0$.
A~toric Landau--Ginzburg model of this family is given by Minkowski polynomial \textnumero $25$, which is
$$
\frac{y}{x}+\frac{1}{x}+y+z+x+\frac{1}{yz}.
$$
Then the pencil $\mathcal{S}$ is given by the following equation:
$$
y^2tz+t^2yz+y^2xz+z^2xy+x^2yz+t^3x=\lambda xyzt.
$$

Suppose that $\lambda\ne\infty$. Then
\begin{equation}
\label{equation:3-26}
\begin{split}
H_{\{x\}}\cdot S_\lambda&=L_{\{x\},\{y\}}+L_{\{x\},\{z\}}+L_{\{x\},\{t\}}+L_{\{x\},\{y,t\}},\\
H_{\{y\}}\cdot S_\lambda&=L_{\{x\},\{y\}}+3L_{\{y\},\{t\}},\\
H_{\{z\}}\cdot S_\lambda&=L_{\{x\},\{z\}}+3L_{\{z\},\{t\}},\\
H_{\{t\}}\cdot S_\lambda&=L_{\{x\},\{t\}}+L_{\{y\},\{t\}}+L_{\{z\},\{t\}}+L_{\{t\},\{x,y,z\}}.
\end{split}
\end{equation}
Thus, the base locus of the pencil $\mathcal{S}$ consists of
the lines $L_{\{x\},\{y\}}$, $L_{\{x\},\{z\}}$, $L_{\{x\},\{t\}}$,
$L_{\{y\},\{t\}}$, $L_{\{z\},\{t\}}$, $L_{\{x\},\{y,t\}}$, and $L_{\{t\},\{x,y,z\}}$.

The surface $S_\lambda$ has isolated singularities, so that it is irreducible.
Its singular points contained in the base locus of the pencil $\mathcal{S}$ can be described as follows:
\begin{itemize}\setlength{\itemindent}{3cm}
\item[$P_{\{x\},\{y\},\{t\}}$:] type $\mathbb{A}_4$ with quadratic term $xy$;

\item[$P_{\{x\},\{z\},\{t\}}$:] type $\mathbb{A}_3$ with quadratic term $z(x+t)$;

\item[$P_{\{y\},\{z\},\{t\}}$:] type $\mathbb{A}_2$ with quadratic term $yz$;

\item[$P_{\{y\},\{t\},\{x,z\}}$:] type $\mathbb{A}_2$ with quadratic term $  y(x+y-\lambda t)$;

\item[$P_{\{z\},\{t\},\{x,y\}}$:] type $\mathbb{A}_2$ with quadratic term $z(x+y+z-t-\lambda t)$.
\end{itemize}
By Lemma~\ref{corollary:irreducible-fibers}, every fiber $\mathsf{f}^{-1}(\lambda)$ is irreducible.
This confirms \eqref{equation:main-1} in Main Theorem.

To verify \eqref{equation:main-2} in Main Theorem, observe that
the intersection matrix of the curves $L_{\{x\},\{y,t\}}$, $L_{\{z\},\{t\}}$, $L_{\{t\},\{x,y,z\}}$, and $H_{\lambda}$ on the surface $S_\lambda$ is given by
the following matrix:
\begin{center}\renewcommand\arraystretch{1.42}
\begin{tabular}{|c||c|c|c|c|}
\hline
 $\bullet$  & $L_{\{x\},\{y,t\}}$ & $L_{\{z\},\{t\}}$ & $L_{\{t\},\{x,y,z\}}$ &  $H_{\lambda}$ \\
\hline\hline
$L_{\{x\},\{y,t\}}$ & $-\frac{5}{4}$ & $0$ & $0$ & $1$ \\
\hline
$L_{\{z\},\{t\}}$ & $0$ &  $\frac{1}{12}$ & $\frac{1}{3}$ & $1$ \\
\hline
$L_{\{t\},\{x,y,z\}}$ & $0$ &  $\frac{1}{3}$ & $-\frac{2}{3}$ & $1$ \\
\hline
 $H_{\lambda}$  & $1$ & $1$ & $1$ & $4$ \\
\hline
\end{tabular}
\end{center}
The rank of this matrix is $4$.
On the other hand, it follows from \eqref{equation:3-26} that
\begin{multline*}
L_{\{x\},\{y\}}+L_{\{x\},\{z\}}+L_{\{x\},\{t\}}+L_{\{x\},\{y,t\}}\sim L_{\{x\},\{y\}}+3L_{\{y\},\{t\}}\sim \\
\sim L_{\{x\},\{z\}}+3L_{\{z\},\{t\}}\sim L_{\{x\},\{t\}}+L_{\{y\},\{t\}}+L_{\{z\},\{t\}}+L_{\{t\},\{x,y,z\}}\sim H_\lambda.
\end{multline*}
Thus, the rank of the intersection matrix of the lines $L_{\{x\},\{y\}}$, $L_{\{x\},\{z\}}$, $L_{\{x\},\{t\}}$,
$L_{\{y\},\{t\}}$, $L_{\{z\},\{t\}}$, $L_{\{x\},\{y,t\}}$, and $L_{\{t\},\{x,y,z\}}$ is also $4$.
Moreover, the description of the singular points of the surface $S_\lambda$ easily gives
$\mathrm{rk}\,\mathrm{Pic}(\widetilde{S}_{\Bbbk})=\mathrm{rk}\,\mathrm{Pic}(S_{\Bbbk})+13$.
Thus, we see that \eqref{equation:main-2-simple} holds, so that \eqref{equation:main-2} in Main Theorem holds by Lemma~\ref{lemma:cokernel}.

\subsection{Family \textnumero $3.27$}
\label{section:r-3-n-27}

We already discussed this case in Example~\ref{example:r-3-n-27}, where we also described the pencil $\mathcal{S}$.
Suppose that $\lambda\ne\infty$. Then
\begin{equation}
\label{equation:3-27}
\begin{split}
H_{\{x\}}\cdot S_\lambda&=L_{\{x\},\{y\}}+L_{\{x\},\{z\}}+2L_{\{x\},\{t\}},\\
H_{\{y\}}\cdot S_\lambda&=L_{\{x\},\{y\}}+L_{\{y\},\{z\}}+2L_{\{y\},\{t\}},\\
H_{\{z\}}\cdot S_\lambda&=L_{\{x\},\{z\}}+L_{\{y\},\{z\}}+2L_{\{z\},\{t\}},\\
H_{\{t\}}\cdot S_\lambda&=L_{\{x\},\{t\}}+L_{\{y\},\{t\}}+L_{\{z\},\{t\}}+L_{\{t\},\{x,y,z\}}.
\end{split}
\end{equation}
Thus, the base locus of the pencil $\mathcal{S}$ consists of
the lines $L_{\{x\},\{y\}}$, $L_{\{x\},\{z\}}$, $L_{\{x\},\{t\}}$,
$L_{\{y\},\{z\}}$, $L_{\{y\},\{t\}}$,
$L_{\{z\},\{t\}}$, and $L_{\{t\},\{x,y,z\}}$.

The surface $S_\lambda$ has isolated singularities, so that it is irreducible.
Moreover, its singular points contained in the base locus of the pencil $\mathcal{S}$ can be described as follows:
\begin{itemize}\setlength{\itemindent}{3cm}
\item[$P_{\{x\},\{y\},\{z\}}$:] type $\mathbb{A}_1$;

\item[$P_{\{x\},\{y\},\{t\}}$:] type $\mathbb{A}_3$ with quadratic term $xy$;

\item[$P_{\{x\},\{z\},\{t\}}$:] type $\mathbb{A}_3$ with quadratic term $xz$;

\item[$P_{\{y\},\{z\},\{t\}}$:] type $\mathbb{A}_3$ with quadratic term $yz$;

\item[$P_{\{x\},\{t\},\{y,z\}}$:] type $\mathbb{A}_1$;

\item[$P_{\{y\},\{t\},\{x,z\}}$:] type $\mathbb{A}_1$;

\item[$P_{\{z\},\{t\},\{x,y\}}$:] type $\mathbb{A}_1$.
\end{itemize}
By Lemma~\ref{corollary:irreducible-fibers}, we have $[\mathsf{f}^{-1}(\lambda)]=1$.
This confirms \eqref{equation:main-1} in Main Theorem.

To prove \eqref{equation:main-2} in Main Theorem, observe that the intersection matrix of the curves $L_{\{x\},\{y\}}$, $L_{\{x\},\{z\}}$, $L_{\{y\},\{z\}}$, and $H_{\lambda}$
on the surface $S_\lambda$ is given by the following table:
\begin{center}\renewcommand\arraystretch{1.42}
\begin{tabular}{|c||c|c|c|c|}
\hline
 $\bullet$  & $L_{\{x\},\{y\}}$ & $L_{\{x\},\{z\}}$ & $L_{\{y\},\{z\}}$ & $H_\lambda$ \\
\hline\hline
 $L_{\{x\},\{y\}}$ &  $1$ & $\frac{1}{2}$ & $\frac{1}{2}$ & $1$ \\
\hline
 $L_{\{x\},\{z\}}$ &  $\frac{1}{2}$ & $1$ & $\frac{1}{2}$ & $1$ \\
\hline
 $L_{\{y\},\{z\}}$ &  $\frac{1}{2}$ & $\frac{1}{2}$ & $1$ & $1$ \\
\hline
 $H_\lambda$  & $1$ & $1$ & $1$ & $4$ \\
\hline
\end{tabular}
\end{center}
The determinant of this matrix is $\frac{5}{4}$.
On the other hand, it follows from \eqref{equation:3-27} that
\begin{multline*}
H_\lambda\sim L_{\{x\},\{y\}}+L_{\{x\},\{z\}}+2L_{\{x\},\{t\}}\sim L_{\{x\},\{y\}}+L_{\{y\},\{z\}}+2L_{\{y\},\{t\}}\sim\\
\sim L_{\{x\},\{z\}}+L_{\{y\},\{z\}}+2L_{\{z\},\{t\}}\sim L_{\{x\},\{t\}}+L_{\{y\},\{t\}}+L_{\{z\},\{t\}}+L_{\{t\},\{x,y,z\}}.
\end{multline*}
Thus, the rank of the intersection matrix of the lines the lines $L_{\{x\},\{y\}}$, $L_{\{x\},\{z\}}$, $L_{\{x\},\{t\}}$,
$L_{\{y\},\{z\}}$, $L_{\{y\},\{t\}}$, $L_{\{z\},\{t\}}$, and $L_{\{t\},\{x,y,z\}}$ is $4$.
As we have seen above, the description of the singular points of the surface $S_\lambda$ gives
$\mathrm{rk}\,\mathrm{Pic}(\widetilde{S}_{\Bbbk})=\mathrm{rk}\,\mathrm{Pic}(S_{\Bbbk})+13$,
so that \eqref{equation:main-2-simple} holds.
This gives \eqref{equation:main-2} in Main Theorem by Lemma~\ref{lemma:cokernel}.

\subsection{Family \textnumero $3.28$}
\label{section:r-3-n-28}

The threefold $X$ is $\mathbb{P}^1\times\mathbb{F}_1$, where $\mathbb{F}_1$ is a blow up of $\mathbb{P}^2$ in a point.
Thus, we have $h^{1,2}(X)=0$.
A~toric Landau--Ginzburg model of this family is given by Minkowski polynomial \textnumero $29$, which is
$$
x+y+z+\frac{x}{z}+\frac{1}{x}+\frac{1}{y}.
$$
Then the pencil $\mathcal{S}$ is given by
$$
x^2yz+y^2xz+z^2xy+x^2ty+t^2xz+t^2yz=\lambda xyzt.
$$

As usual, we suppose that $\lambda\ne\infty$. Then
\begin{equation}
\label{equation:3-28}
\begin{split}
H_{\{x\}}\cdot S_\lambda&=L_{\{x\},\{y\}}+L_{\{x\},\{z\}}+2L_{\{x\},\{t\}},\\
H_{\{y\}}\cdot S_\lambda&=L_{\{x\},\{y\}}+L_{\{y\},\{z\}}+2L_{\{y\},\{t\}},\\
H_{\{z\}}\cdot S_\lambda&=2L_{\{x\},\{z\}}+L_{\{y\},\{z\}}+L_{\{z\},\{t\}},\\
H_{\{t\}}\cdot S_\lambda&=L_{\{x\},\{t\}}+L_{\{y\},\{t\}}+L_{\{z\},\{t\}}+L_{\{t\},\{x,y,z\}}.
\end{split}
\end{equation}
Thus, the base locus of the pencil $\mathcal{S}$ consists of the lines
$L_{\{x\},\{y\}}$, $L_{\{x\},\{z\}}$, $L_{\{x\},\{t\}}$, $L_{\{y\},\{z\}}$, $L_{\{y\},\{t\}}$, $L_{\{z\},\{t\}}$, and $L_{\{t\},\{x,y,z\}}$.

Each surface $S_\lambda$ has isolated singularities.
In particular, it is irreducible.
Its singular points contained in the base locus of the pencil $\mathcal{S}$ can be described as follows:
\begin{itemize}\setlength{\itemindent}{3cm}
\item[$P_{\{y\},\{z\},\{t\}}$:] type $\mathbb{A}_2$ with quadratic term $y(z+t)$;

\item[$P_{\{x\},\{z\},\{t\}}$:] type $\mathbb{A}_4$ with quadratic term $xz$;

\item[$P_{\{x\},\{y\},\{t\}}$:] type $\mathbb{A}_3$ with quadratic term $xy$;

\item[$P_{\{x\},\{y\},\{z\}}$:] type $\mathbb{A}_2$ with quadratic term $x(x+y)$;

\item[$P_{\{x\},\{t\},\{y,z\}}$:] type $\mathbb{A}_1$;

\item[$P_{\{y\},\{t\},\{x,z\}}$:] type $\mathbb{A}_1$.
\end{itemize}
Thus, each fiber $\mathsf{f}^{-1}(\lambda)$ is irreducible by Lemma~\ref{corollary:irreducible-fibers}.
This confirms \eqref{equation:main-1} in Main Theorem.

To verify \eqref{equation:main-2} in Main Theorem,
observe that the intersection matrix of the curves $L_{\{y\},\{z\}}$, $L_{\{y\},\{t\}}$, $L_{\{t\},\{x,y,z\}}$, and $H_{\lambda}$
on the surface $S_\lambda$ is given by the following table:
\begin{center}\renewcommand\arraystretch{1.42}
\begin{tabular}{|c||c|c|c|c|}
\hline
 $\bullet$  & $L_{\{y\},\{z\}}$ & $L_{\{y\},\{t\}}$ & $L_{\{t\},\{x,y,z\}}$ &  $H_{\lambda}$ \\
\hline\hline
$L_{\{y\},\{z\}}$ & $-\frac{2}{3}$ & $\frac{2}{3}$ & $0$ & $1$ \\
\hline
$L_{\{y\},\{t\}}$ & $\frac{2}{3}$ &  $-\frac{1}{12}$ & $\frac{1}{2}$ & $1$ \\
\hline
$L_{\{t\},\{x,y,z\}}$ & $0$ &  $\frac{1}{2}$ & $-1$ & $1$ \\
\hline
 $H_{\lambda}$  & $1$ & $1$ & $1$ & $4$ \\
\hline
\end{tabular}
\end{center}
This matrix has rank $4$. Using \eqref{equation:3-28}, we see that
\begin{multline*}
H_\lambda\sim L_{\{x\},\{y\}}+L_{\{x\},\{z\}}+2L_{\{x\},\{t\}}\sim L_{\{x\},\{y\}}+L_{\{y\},\{z\}}+2L_{\{y\},\{t\}}\sim\\
2L_{\{x\},\{z\}}+L_{\{y\},\{z\}}+L_{\{z\},\{t\}}\sim L_{\{x\},\{t\}}+L_{\{y\},\{t\}}+L_{\{z\},\{t\}}+L_{\{t\},\{x,y,z\}}
\end{multline*}
on the surface $S_\lambda$.
Thus, the rank of the intersection matrix of the lines
$L_{\{x\},\{y\}}$, $L_{\{x\},\{z\}}$, $L_{\{x\},\{t\}}$, $L_{\{y\},\{z\}}$, $L_{\{y\},\{t\}}$, $L_{\{z\},\{t\}}$, and $L_{\{t\},\{x,y,z\}}$
on the surface $S_\lambda$ is also $4$.
On the other hand, we have $\mathrm{rk}\,\mathrm{Pic}(\widetilde{S}_{\Bbbk})=\mathrm{rk}\,\mathrm{Pic}(S_{\Bbbk})+13$.
Thus, we see that \eqref{equation:main-2-simple} holds, so that \eqref{equation:main-2} in Main Theorem holds by Lemma~\ref{lemma:cokernel}.

\subsection{Family \textnumero $3.29$}
\label{section:r-3-n-29}

In this case, we have $h^{1,2}(X)=0$.
A~toric Landau--Ginzburg model of this family is given by Minkowski polynomial \textnumero $26$, which is
$$
\frac{y}{x}+\frac{1}{x}+y+z+\frac{1}{xyz}+x.
$$
Hence, the pencil $\mathcal{S}$ is given by the equation
$$
y^2tz+t^2yz+y^2xz+z^2xy+t^4+x^2yz=\lambda xyzt.
$$

Let $\mathcal{C}$ be the cubic curve in $\mathbb{P}^3$ that is given by $x=y^2z+yzt+t^3=0$.
Then $\mathcal{C}$ is singular at the point $P_{\{x\},\{y\},\{t\}}$.
If $\lambda\ne\infty$, then
\begin{equation}
\label{equation:3-29}
\begin{split}
H_{\{x\}}\cdot S_\lambda&=L_{\{x\},\{t\}}+\mathcal{C},\\
H_{\{y\}}\cdot S_\lambda&=4L_{\{y\},\{t\}},\\
H_{\{z\}}\cdot S_\lambda&=4L_{\{z\},\{t\}},\\
H_{\{t\}}\cdot S_\lambda&=L_{\{x\},\{t\}}+L_{\{y\},\{t\}}+L_{\{z\},\{t\}}+L_{\{t\},\{x,y,z\}}.
\end{split}
\end{equation}
Thus, the base locus of the pencil $\mathcal{S}$ consists of the curves
$L_{\{x\},\{t\}}$, $L_{\{y\},\{t\}}$, $L_{\{z\},\{t\}}$, $L_{\{t\},\{x,y,z\}}$, and $\mathcal{C}$.

If $\lambda\ne\infty$, then $S_\lambda$ has isolated singularities, so that it is irreducible.
In this case, the singular points of the surface $S_\lambda$ contained in the base locus of the pencil $\mathcal{S}$ can be described as follows:
\begin{itemize}\setlength{\itemindent}{3cm}
\item[$P_{\{x\},\{y\},\{t\}}$:] type $\mathbb{A}_3$ with quadratic term $xy$,

\item[$P_{\{x\},\{z\},\{t\}}$:] type $\mathbb{A}_3$ with quadratic term $z(x+t)$,

\item[$P_{\{y\},\{z\},\{t\}}$:] type $\mathbb{A}_3$ with quadratic term $yz$,

\item[$P_{\{y\},\{t\},\{x,z\}}$:] type $\mathbb{A}_3$ with quadratic term $y(x+y+z-\lambda t)$,

\item[$P_{\{z\},\{t\},\{x,y\}}$:] type $\mathbb{A}_3$ with quadratic term $z(x+y+z-t-\lambda t)$.
\end{itemize}
By Lemma~\ref{corollary:irreducible-fibers}, we have $[\mathsf{f}^{-1}(\lambda)]$ for every $\lambda\in\mathbb{C}$.
This confirms \eqref{equation:main-1} in Main Theorem.

To verify \eqref{equation:main-2} in Main Theorem, observe that
$$
L_{\{x\},\{t\}}+\mathcal{C}\sim 4L_{\{y\},\{t\}}\sim 4L_{\{z\},\{t\}}\sim L_{\{x\},\{t\}}+L_{\{y\},\{t\}}+L_{\{z\},\{t\}}+L_{\{t\},\{x,y,z\}}.
$$
on the surface $S_\lambda$ with $\lambda\ne\infty$. This follows from \eqref{equation:3-29}.
Thus, the intersection matrix of the curves
$L_{\{x\},\{t\}}$, $L_{\{y\},\{t\}}$, $L_{\{z\},\{t\}}$, $L_{\{t\},\{x,y,z\}}$, and $\mathcal{C}$
on the surface $S_\lambda$ has the same rank as the intersection matrix of the lines $L_{\{x\},\{t\}}$ and $L_{\{y\},\{t\}}$.
On the other hand, the rank of the latter matrix is $2$, because have $L_{\{x\},\{t\}}^2=L_{\{x\},\{t\}}\cdot L_{\{y\},\{t\}}=\frac{1}{4}$
and $L_{\{x\},\{t\}}^2=\frac{1}{2}$.
Moreover, we have
$\mathrm{rk}\,\mathrm{Pic}(\widetilde{S}_{\Bbbk})=\mathrm{rk}\,\mathrm{Pic}(S_{\Bbbk})+15$.
This shows that \eqref{equation:main-2-simple} holds.
Then \eqref{equation:main-2} in Main Theorem also holds by Lemma~\ref{lemma:cokernel}.

\subsection{Family \textnumero $3.30$}
\label{section:r-3-n-30}

The threefold $X$ can be obtained from $\mathbb{P}^3$ blown up at a point by blowing up the proper transform of a line passing through this point.
This  shows that $h^{1,2}(X)=0$.
A~toric Landau--Ginzburg model of this family is given by Minkowski polynomial \textnumero $28$, which is
$$
x+y+z+\frac{y}{z}+\frac{x}{y}+\frac{1}{x}.
$$
In this case, the quartic pencil $\mathcal{S}$ is given by the equation
$$
x^2yz+y^2xz+z^2xy+y^2xt+x^2tz+t^2yz=\lambda xyzt.
$$

Suppose that $\lambda\ne\infty$. Then
\begin{equation}
\label{equation:3-30}
\begin{split}
H_{\{x\}}\cdot S_\lambda&=L_{\{x\},\{y\}}+L_{\{x\},\{z\}}+2L_{\{x\},\{t\}},\\
H_{\{y\}}\cdot S_\lambda&=2L_{\{x\},\{y\}}+L_{\{y\},\{z\}}+L_{\{y\},\{t\}},\\
H_{\{z\}}\cdot S_\lambda&=L_{\{x\},\{z\}}+2L_{\{y\},\{z\}}+L_{\{z\},\{t\}},\\
H_{\{t\}}\cdot S_\lambda&=L_{\{x\},\{t\}}+L_{\{y\},\{t\}}+L_{\{z\},\{t\}}+L_{\{t\},\{x,y,z\}}.
\end{split}
\end{equation}
Thus, the base locus of the pencil $\mathcal{S}$ consists of the lines
$L_{\{x\},\{y\}}$, $L_{\{x\},\{z\}}$, $L_{\{x\},\{t\}}$, $L_{\{y\},\{z\}}$, $L_{\{y\},\{t\}}$, $L_{\{z\},\{t\}}$, and $L_{\{t\},\{x,y,z\}}$.

Each surface $S_\lambda$ is irreducible and has isolated singularities.
Moreover, its singular points contained in the base locus of the pencil $\mathcal{S}$ can be described as follows:
\begin{itemize}\setlength{\itemindent}{3cm}
\item[$P_{\{x\},\{y\},\{z\}}$:] type $\mathbb{A}_4$ with quadratic term $yz$;

\item[$P_{\{x\},\{y\},\{t\}}$:] type $\mathbb{A}_4$ with quadratic term $xy$;

\item[$P_{\{x\},\{z\},\{t\}}$:] type $\mathbb{A}_2$ with quadratic term $x(z+t)$;

\item[$P_{\{y\},\{z\},\{t\}}$:] type $\mathbb{A}_2$ with quadratic term $z(y+t)$;

\item[$P_{\{x\},\{t\},\{y,z\}}$:] type $\mathbb{A}_1$.
\end{itemize}
By Lemma~\ref{corollary:irreducible-fibers}, each fiber $\mathsf{f}^{-1}(\lambda)$ is irreducible.
This confirms \eqref{equation:main-1} in Main Theorem.

To verify \eqref{equation:main-2} in Main Theorem, observe that
$\mathrm{rk}\,\mathrm{Pic}(\widetilde{S}_{\Bbbk})=\mathrm{rk}\,\mathrm{Pic}(S_{\Bbbk})+13$.
On the other hand, it follows from \eqref{equation:3-30} that
\begin{multline*}
H_\lambda\sim L_{\{x\},\{y\}}+L_{\{x\},\{z\}}+2L_{\{x\},\{t\}}\sim 2L_{\{x\},\{y\}}+L_{\{y\},\{z\}}+L_{\{y\},\{t\}}\sim \\
\sim L_{\{x\},\{z\}}+2L_{\{y\},\{z\}}+L_{\{z\},\{t\}}\sim L_{\{x\},\{t\}}+L_{\{y\},\{t\}}+L_{\{z\},\{t\}}+L_{\{t\},\{x,y,z\}}
\end{multline*}
on the surface $S_\lambda$.
Thus, the intersection matrix of the lines
$L_{\{x\},\{y\}}$, $L_{\{x\},\{z\}}$, $L_{\{x\},\{t\}}$, $L_{\{y\},\{z\}}$, $L_{\{y\},\{t\}}$, $L_{\{z\},\{t\}}$, and $L_{\{t\},\{x,y,z\}}$
has the same rank as the intersection matrix of the curves $L_{\{x\},\{y\}}$, $L_{\{x\},\{z\}}$, $L_{\{t\},\{x,y,z\}}$, and $H_{\lambda}$.
The latter matrix is given by
\begin{center}\renewcommand\arraystretch{1.42}
\begin{tabular}{|c||c|c|c|c|}
\hline
 $\bullet$  & $L_{\{x\},\{y\}}$ & $L_{\{x\},\{z\}}$ & $L_{\{t\},\{x,y,z\}}$ &  $H_{\lambda}$ \\
\hline\hline
$L_{\{x\},\{y\}}$ & $0$ & $\frac{1}{4}$ & $0$ & $1$ \\
\hline
$L_{\{x\},\{z\}}$ & $\frac{1}{4}$ &  $-\frac{7}{12}$ & $0$ & $1$ \\
\hline
$L_{\{t\},\{x,y,z\}}$ & $0$ &  $0$ & $-\frac{3}{2}$ & $1$ \\
\hline
 $H_{\lambda}$  & $1$ & $1$ & $1$ & $4$ \\
\hline
\end{tabular}
\end{center}
Its rank is $4$, so that \eqref{equation:main-2-simple} holds.
Then \eqref{equation:main-2} in Main Theorem holds by Lemma~\ref{lemma:cokernel}.

\subsection{Family \textnumero $3.31$}
\label{section:r-3-n-31}

The threefold $X$ can be obtained by blowing up irreducible quadric cone in $\mathbb{P}^4$ in its vertex.
This implies that $h^{1,2}(X)=0$.
A~toric Landau--Ginzburg model of this family is given by Minkowski polynomial \textnumero $27$, which is
$$
x+y+z+\frac{x}{z}+\frac{x}{y}+\frac{1}{x}.
$$
Then the pencil $\mathcal{S}$ is given by the following equation:
$$
t^2yz+tx^2y+tx^2z+x^2yz+xy^2z+xyz^2=\lambda xyzt.
$$

Suppose that $\lambda\ne\infty$. Then
\begin{equation}
\label{equation:3-31}
\begin{split}
H_{\{x\}}\cdot S_\lambda&=L_{\{x\},\{y\}}+L_{\{x\},\{z\}}+2L_{\{x\},\{t\}},\\
H_{\{y\}}\cdot S_\lambda&=2L_{\{x\},\{y\}}+L_{\{y\},\{z\}}+L_{\{y\},\{t\}},\\
H_{\{z\}}\cdot S_\lambda&=2L_{\{x\},\{z\}}+L_{\{y\},\{z\}}+L_{\{z\},\{t\}},\\
H_{\{t\}}\cdot S_\lambda&=L_{\{x\},\{t\}}+L_{\{y\},\{t\}}+L_{\{z\},\{t\}}+L_{\{t\},\{x,y,z\}}.
\end{split}
\end{equation}
Thus, the base locus of the pencil $\mathcal{S}$ consists of the lines
$L_{\{x\},\{y\}}$, $L_{\{x\},\{z\}}$, $L_{\{x\},\{t\}}$, $L_{\{y\},\{z\}}$, $L_{\{y\},\{t\}}$, $L_{\{z\},\{t\}}$, and $L_{\{t\},\{x,y,z\}}$.

Each surface $S_\lambda$ is irreducible, it has isolated singularities,
and its singular points contained in the base locus of the pencil $\mathcal{S}$ can be described as follows:
\begin{itemize}\setlength{\itemindent}{3cm}
\item[$P_{\{x\},\{y\},\{z\}}$:] type $\mathbb{A}_3$ with quadratic term $yz$;

\item[$P_{\{x\},\{y\},\{t\}}$:] type $\mathbb{A}_4$ with quadratic term $xy$;

\item[$P_{\{x\},\{z\},\{t\}}$:] type $\mathbb{A}_4$ with quadratic term $xz$;

\item[$P_{\{y\},\{z\},\{t\}}$:] type $\mathbb{A}_1$;

\item[{$P_{\{x\},\{t\},\{y,z\}}$}:] type $\mathbb{A}_1$.
\end{itemize}
Thus, by Lemma~\ref{corollary:irreducible-fibers}, every fiber $\mathsf{f}^{-1}(\lambda)$ is irreducible.
This confirms \eqref{equation:main-1} in Main Theorem, since $h^{1,2}(X)=0$.

To verify \eqref{equation:main-2} in Main Theorem, observe that $\mathrm{rk}\,\mathrm{Pic}(\widetilde{S}_{\Bbbk})=\mathrm{rk}\,\mathrm{Pic}(S_{\Bbbk})+13$.
Moreover, it follows from \eqref{equation:3-31} that
\begin{multline*}
H_\lambda\sim L_{\{x\},\{y\}}+L_{\{x\},\{z\}}+2L_{\{x\},\{t\}}\sim 2L_{\{x\},\{y\}}+L_{\{y\},\{z\}}+L_{\{y\},\{t\}}\sim \\
\sim 2L_{\{x\},\{z\}}+L_{\{y\},\{z\}}+L_{\{z\},\{t\}}\sim L_{\{x\},\{t\}}+L_{\{y\},\{t\}}+L_{\{z\},\{t\}}+L_{\{t\},\{x,y,z\}}
\end{multline*}
on the surface $S_\lambda$.
Thus, the intersection matrix of the lines
$L_{\{x\},\{y\}}$, $L_{\{x\},\{z\}}$, $L_{\{x\},\{t\}}$, $L_{\{y\},\{z\}}$, $L_{\{y\},\{t\}}$, $L_{\{z\},\{t\}}$, and $L_{\{t\},\{x,y,z\}}$
has the same rank as the intersection matrix of the curves $L_{\{x\},\{y\}}$, $L_{\{x\},\{z\}}$, $L_{\{y\},\{z\}}$, and $H_{\lambda}$.
The latter matrix can be computed as
\begin{center}\renewcommand\arraystretch{1.42}
\begin{tabular}{|c||c|c|c|c|}
\hline
 $\bullet$  & $L_{\{x\},\{y\}}$ & $L_{\{x\},\{z\}}$ & $L_{\{y\},\{z\}}$ & $H_{\lambda}$ \\
\hline\hline
 $L_{\{x\},\{y\}}$ &  $-\frac{1}{20}$ & $\frac{1}{4}$ & $\frac{1}{2}$ & $1$ \\
\hline
 $L_{\{x\},\{z\}}$  & $\frac{1}{4}$& $-\frac{1}{20}$& $\frac{1}{2}$ & $1$ \\
\hline
 $L_{\{y\},\{z\}}$  & $\frac{1}{2}$& $\frac{1}{2}$& $-\frac{1}{2}$& $1$ \\
\hline
 $H_{\lambda}$  & $1$ & $1$ & $1$ & $4$ \\
\hline
\end{tabular}
\end{center}
The determinant of this matrix is $-\frac{3}{25}$.
Thus, we see that \eqref{equation:main-2-simple} holds.
Then \eqref{equation:main-2} in Main Theorem holds by Lemma~\ref{lemma:cokernel}.

\section{Fano threefolds of Picard rank $4$}
\label{section:rank-4}

\subsection{Family \textnumero $4.1$}
\label{section:r-4-n-1}

The threefold $X$ is  a divisor of degree $(1,1,1,1)$ on $\mathbb{P}^1\times\mathbb{P}^1\times\mathbb{P}^1\times\mathbb{P}^1$.
In this case,  we have $h^{1,2}(X)=1$.
A~toric Landau--Ginzburg model of this family is given by Minkowski polynomial \textnumero $2354.1$, which is
$$
x+y+z+\frac{y}{z}+\frac{y}{x}+\frac{z}{y}+\frac{z}{x}+\frac{1}{z}+\frac{y}{xz}+\frac{1}{y}+\frac{3}{x}+\frac{z}{xy}+\frac{1}{xz}+\frac{1}{xy}.
$$
The quartic pencil $\mathcal{S}$ is given by the following equation:
\begin{multline*}
x^2yz+xy^2z+xyz^2+xy^2t+y^2zt+xz^2t+yz^2t+xyt^2+y^2t^2+xt^2z+\\+3yzt^2+z^2t^2+yt^3+zt^3=\lambda xyzt.
\end{multline*}
As usual, we assume that $\lambda\ne\infty$ (just for simplicity).

Let $\mathcal{C}$ be the conic in $\mathbb{P}^3$ given by $x=yz+yt+zt=0$.
Then
\begin{equation}
\label{equation:4-1}
\begin{split}
H_{\{x\}}\cdot S_\lambda&=L_{\{x\},\{t\}}+L_{\{x\},\{y,z,t\}}+\mathcal{C},\\
H_{\{y\}}\cdot S_\lambda&=L_{\{y\},\{z\}}+L_{\{y\},\{t\}}+L_{\{y\},\{x,t\}}+L_{\{y\},\{z,t\}},\\
H_{\{z\}}\cdot S_\lambda&=L_{\{y\},\{z\}}+L_{\{z\},\{t\}}+L_{\{z\},\{x,t\}}+L_{\{z\},\{y,t\}},\\
H_{\{t\}}\cdot S_\lambda&=L_{\{x\},\{t\}}+L_{\{y\},\{t\}}+L_{\{z\},\{t\}}+L_{\{t\},\{x,y,z\}}.
\end{split}
\end{equation}
Thus, the base locus of the pencil $\mathcal{S}$ consists of the curves
$L_{\{x\},\{t\}}$, $L_{\{y\},\{z\}}$, $L_{\{y\},\{t\}}$, $L_{\{z\},\{t\}}$,
$L_{\{y\},\{x,t\}}$, $L_{\{y\},\{z,t\}}$, $L_{\{z\},\{x,t\}}$, $L_{\{z\},\{y,t\}}$,
$L_{\{x\},\{y,z,t\}}$, $L_{\{t\},\{x,y,z\}}$, and $\mathcal{C}$.

Observe that $S_{-4}=H_{\{x,t\}}+\mathbf{S}$, where $\mathbf{S}$ is a cubic surface in $\mathbb{P}^3$ that is given by
$$
yt^2+zt^2+z^2t+y^2t+3yzt+y^2z+yz^2+xyz=0.
$$
On the other hand, if $\lambda\ne-4$, then $S_\lambda$ is irreducible and has isolated singularities.
Moreover, if $\lambda\ne-3$ and $\lambda\ne-4$, then singular points of the surface~$S_\lambda$ contained in the base locus of the pencil $\mathcal{S}$ can be described as follows:
\begin{itemize}\setlength{\itemindent}{3cm}
\item[$P_{\{x\},\{y\},\{t\}}$:] type $\mathbb{A}_2$ with quadratic term $(x+t)(y+t)$;

\item[$P_{\{x\},\{z\},\{t\}}$:] type $\mathbb{A}_2$ with quadratic term $(x+t)(z+t)$;

\item[$P_{\{y\},\{z\},\{t\}}$:] type $\mathbb{A}_3$ with quadratic term $yz$;

\item[$P_{\{x\},\{t\},\{y,z\}}$:] type $\mathbb{A}_1$ with quadratic term $(x+t)(x+y+z+t)-(\lambda+4)xt$;

\item[$P_{\{y\},\{z\},\{x,t\}}$:] type $\mathbb{A}_1$ with quadratic term $(x+t)(y+z)+(\lambda+4)yz$.
\end{itemize}
Furthermore, the surface $S_{-3}$ has the same type singularities at the points
$P_{\{x\},\{z\},\{t\}}$, $P_{\{x\},\{y\},\{t\}}$, $P_{\{y\},\{z\},\{t\}}$, $P_{\{x\},\{t\},\{y,z\}}$, and $P_{\{y\},\{z\},\{x,t\}}$.
In addition to this, the surface $S_{-3}$ is also singular at the points $[0:\xi_3:1:\xi_3^2]$ and $[0:\xi_3^2:1:\xi_3]$,
where $\xi_3$ is a primitive cube root of unity.
Both these points are singular points of the surface $S_{-3}$ of type $\mathbb A_1$.

For $\lambda\ne -4$, the surface $S_{\lambda}$ has du Val singularities at the base points of the pencil~$\mathcal{S}$.
Therefore, by Lemma~\ref{corollary:irreducible-fibers}, the fiber $\mathsf{f}^{-1}(\lambda)$ is irreducible for every $\lambda\ne -4$.
Moreover, the points $P_{\{x\},\{z\},\{t\}}$, $P_{\{x\},\{y\},\{t\}}$, $P_{\{y\},\{z\},\{t\}}$, $P_{\{x\},\{t\},\{y,z\}}$, and $P_{\{y\},\{z\},\{x,t\}}$
are good double points of the surface $S_{-4}$.
Furthermore, the surface $S_{-4}$ is smooth
at general points of the curves $L_{\{x\},\{t\}}$, $L_{\{y\},\{z\}}$, $L_{\{y\},\{t\}}$, $L_{\{z\},\{t\}}$,
$L_{\{y\},\{x,t\}}$, $L_{\{y\},\{z,t\}}$, $L_{\{z\},\{x,t\}}$, $L_{\{z\},\{y,t\}}$,
$L_{\{x\},\{y,z,t\}}$, $L_{\{t\},\{x,y,z\}}$, and $\mathcal{C}$.
Thus, we see that
$$
\big[\mathsf{f}^{-1}(-4)\big]=\big[S_{-4}\big]=2
$$
by \eqref{equation:equation:number-of-irredubicle-components-refined} and Lemmas~\ref{lemma:main} and \ref{lemma:normal-crossing}.
This confirms \eqref{equation:main-1} in Main Theorem.

To verify \eqref{equation:main-2} in Main Theorem, we may assume that $\lambda\ne -4$.
Then the intersection matrix of the curves $L_{\{x\},\{t\}}$, $L_{\{x\},\{y,z,t\}}$, $L_{\{y\},\{x,t\}}$, $L_{\{y\},\{z,t\}}$, $L_{\{z\},\{y,t\}}$, $L_{\{t\},\{x,y,z\}}$, and $H_{\lambda}$
on the surface $S_\lambda$ is given by
\begin{center}\renewcommand\arraystretch{1.42}
\begin{tabular}{|c||c|c|c|c|c|c|c|}
\hline
 $\bullet$  & $L_{\{x\},\{t\}}$ & $L_{\{x\},\{y,z,t\}}$ & $L_{\{y\},\{x,t\}}$ & $L_{\{y\},\{z,t\}}$ & $L_{\{z\},\{y,t\}}$ & $L_{\{t\},\{x,y,z\}}$ & $H_{\lambda}$ \\
\hline\hline
$L_{\{x\},\{t\}}$ & $-\frac{1}{6}$ & $\frac{1}{2}$ & $\frac{2}{3}$ & $0$ & $0$ & $\frac{1}{2}$ & $1$ \\
\hline
$L_{\{x\},\{y,z,t\}}$ & $\frac{1}{2}$ & $-\frac{3}{2}$ & $0$ &  $1$ & $1$ & $\frac{1}{2}$ & $1$ \\
\hline
$L_{\{y\},\{x,t\}}$ & $\frac{2}{3}$ & $0$ & $-\frac{5}{6}$ &  $1$ & $0$ & $0$ & $1$ \\
\hline
$L_{\{y\},\{z,t\}}$ & $0$ & $1$ & $1$ &  $-\frac{5}{4}$ & $\frac{1}{4}$ & $0$ & $1$ \\
\hline
$L_{\{z\},\{y,t\}}$ & $1$ & $1$ & $0$ &  $\frac{1}{4}$ & $-\frac{5}{4}$ & $0$ & $1$ \\
\hline
$L_{\{t\},\{x,y,z\}}$ & $\frac{1}{2}$ & $\frac{1}{2}$ & $0$ &  $0$ & $0$ & $-\frac{3}{2}$ & $1$ \\
\hline
 $H_{\lambda}$  & $1$ & $1$ & $1$ & $1$ & $1$ & $1$ & $4$ \\
\hline
\end{tabular}
\end{center}
This matrix has rank~$7$.
On the other hand, it follows from \eqref{equation:4-1} that
\begin{multline*}
H_\lambda\sim L_{\{x\},\{t\}}+L_{\{x\},\{y,z,t\}}+\mathcal{C}\sim L_{\{y\},\{z\}}+L_{\{y\},\{t\}}+L_{\{y\},\{x,t\}}+L_{\{y\},\{z,t\}}\sim\\
\sim L_{\{y\},\{z\}}+L_{\{z\},\{t\}}+L_{\{z\},\{x,t\}}+L_{\{z\},\{y,t\}}\sim L_{\{x\},\{t\}}+L_{\{y\},\{t\}}+L_{\{z\},\{t\}}+L_{\{t\},\{x,y,z\}}.
\end{multline*}
Moreover, we also have $2L_{\{x\},\{t\}}+L_{\{y\},\{x,t\}}+L_{\{z\},\{x,t\}}\sim H_\lambda$, because
$$
H_{\{x,t\}}\cdot S_\infty=2L_{\{x\},\{t\}}+L_{\{y\},\{x,t\}}+L_{\{z\},\{x,t\}}.
$$
Therefore, the intersection matrix of the curves
$L_{\{x\},\{t\}}$, $L_{\{y\},\{z\}}$, $L_{\{y\},\{t\}}$, $L_{\{z\},\{t\}}$,
$L_{\{y\},\{x,t\}}$, $L_{\{y\},\{z,t\}}$, $L_{\{z\},\{x,t\}}$, $L_{\{z\},\{y,t\}}$,
$L_{\{x\},\{y,z,t\}}$, $L_{\{t\},\{x,y,z\}}$, and $\mathcal{C}$ on the surface~$S_\lambda$
has the same rank as the intersection matrix of the curves $L_{\{x\},\{t\}}$, $L_{\{x\},\{y,z,t\}}$, $L_{\{y\},\{x,t\}}$, $L_{\{y\},\{z,t\}}$, $L_{\{z\},\{y,t\}}$, $L_{\{t\},\{x,y,z\}}$, and $H_{\lambda}$.
But $\mathrm{rk}\,\mathrm{Pic}(\widetilde{S}_{\Bbbk})=\mathrm{rk}\,\mathrm{Pic}(S_{\Bbbk})+9$.
Therefore, we conclude that \eqref{equation:main-2-simple} holds.
Then \eqref{equation:main-2} in Main Theorem holds by Lemma~\ref{lemma:cokernel}.

\subsection{Family \textnumero $4.2$}
\label{section:r-4-n-2}

In this case, the threefold $X$ is a blow up of the irreducible quadric cone in $\mathbb{P}^4$
in its vertex and a smooth elliptic curve that does not pass through the vertex.
This shows that $h^{1,2}(X)=1$.
A~toric Landau--Ginzburg model of this family is given by Minkowski polynomial \textnumero $663$, which is
$$
x+y+\frac{z}{y}+\frac{z}{x}+\frac{x}{y}+\frac{y}{x}+\frac{2}{x}+\frac{2}{y}+\frac{1}{yz}+\frac{1}{xz}.
$$
The quartic pencil $\mathcal{S}$ is given by the equation
$$
x^2yz+xy^2z+xz^2t+yz^2t+x^2zt+y^2zt+2xzt^2+2yzt^2+xt^3+yt^3=\lambda xyzt.
$$
For simplicity, we assume that $\lambda\ne\infty$.

If $\lambda\ne-2$, then $S_\lambda$ is irreducible and has isolated singularities.
On the other hand, we have $S_{-2}=H_{\{x,y\}}+\mathbf{S}$,
where $\mathbf{S}$ is an irreducible cubic surface that is given by the equation $xyz+xzt+t^3+z^2t+2zt^2+yzt=0$.

Let $\mathcal{C}_1$ be the conic in $\mathbb{P}^3$ that is given by $x=yz+(z+t)^2=0$,
and let $\mathcal{C}_2$ be the conic in $\mathbb{P}^3$ that is given by $y=xz+(z+t)^2=0$.
Then
\begin{equation}
\label{equation:4-2}
\begin{split}
H_{\{x\}}\cdot S_\lambda&=L_{\{x\},\{y\}}+L_{\{x\},\{t\}}+\mathcal{C}_1,\\
H_{\{y\}}\cdot S_\lambda&=L_{\{x\},\{y\}}+L_{\{y\},\{t\}}+\mathcal{C}_2,\\
H_{\{z\}}\cdot S_\lambda&=3L_{\{z\},\{t\}}+L_{\{z\},\{x,y\}},\\
H_{\{t\}}\cdot S_\lambda&=L_{\{x\},\{t\}}+L_{\{y\},\{t\}}+L_{\{z\},\{t\}}+L_{\{t\},\{x,y\}}.
\end{split}
\end{equation}
Thus, the base locus of the pencil $\mathcal{S}$ consists of the curves
$L_{\{x\},\{y\}}$, $L_{\{x\},\{t\}}$,  $L_{\{y\},\{t\}}$, $L_{\{z\},\{t\}}$, $L_{\{z\},\{x,y\}}$,
$L_{\{t\},\{x,y\}}$, $\mathcal{C}_1$, and $\mathcal{C}_2$.

If $\lambda\ne-2$, then singular points of the surface $S_\lambda$ contained in the base locus of the pencil $\mathcal{S}$ can be described as follows:
\begin{itemize}\setlength{\itemindent}{3cm}
\item[$P_{\{x\},\{y\},\{t\}}$:] type $\mathbb{A}_3$ with quadratic term $t(x+y)$;

\item[$P_{\{x\},\{z\},\{t\}}$:] type $\mathbb{A}_2$ with quadratic term $z(x+t)$;

\item[$P_{\{y\},\{z\},\{t\}}$:] type $\mathbb{A}_2$ with quadratic term $z(y+t)$;

\item[$P_{\{x\},\{y\},\{z,t\}}$:] type $\mathbb{A}_3$ with quadratic term $x^2+y^2+\lambda xy$;

\item[$P_{\{z\},\{t\},\{x,y\}}$:] type $\mathbb{A}_3$ with quadratic term $z(x+y-2t-\lambda t)$.
\end{itemize}
Therefore, by Lemma~\ref{corollary:irreducible-fibers}, the fiber $\mathsf{f}^{-1}(\lambda)$ is irreducible for every $\lambda\ne -2$.
Moreover, the points $P_{\{y\},\{z\},\{t\}}$, $P_{\{x\},\{z\},\{t\}}$, $P_{\{x\},\{y\},\{t\}}$,
$P_{\{x\},\{y\},\{z,t\}}$, and $P_{\{z\},\{t\},\{x,y\}}$ are good double points of the surface $S_{-2}$.
Furthermore, the surface $S_{-2}$ is smooth
at general points of the curves $L_{\{x\},\{y\}}$, $L_{\{x\},\{t\}}$,  $L_{\{y\},\{t\}}$, $L_{\{z\},\{t\}}$,
$L_{\{z\},\{x,y\}}$, $L_{\{t\},\{x,y\}}$, $\mathcal{C}_1$, and $\mathcal{C}_2$.
Thus, we see that $[\mathsf{f}^{-1}(-2)]=[S_{-2}]=2$
by \eqref{equation:equation:number-of-irredubicle-components-refined} and Lemmas~\ref{lemma:main} and \ref{lemma:normal-crossing}.
This confirms \eqref{equation:main-1} in Main Theorem.

To verify \eqref{equation:main-2} in Main Theorem, we may assyme that $\lambda\ne -2$.
By \eqref{equation:4-2}, we have
\begin{multline*}
H_\lambda\sim L_{\{x\},\{y\}}+L_{\{x\},\{t\}}+\mathcal{C}_1\sim L_{\{x\},\{y\}}+L_{\{y\},\{t\}}+\mathcal{C}_2\sim\\
\sim 3L_{\{z\},\{t\}}+L_{\{z\},\{x,y\}}\sim L_{\{x\},\{t\}}+L_{\{y\},\{t\}}+L_{\{z\},\{t\}}+L_{\{t\},\{x,y\}}
\end{multline*}
on the surface $S_\lambda$. Since $H_{\{x,y\}}\cdot S_\lambda=2L_{\{x\},\{y\}}+L_{\{z\},\{x,y\}}+L_{\{t\},\{x,y\}}$, we also have
$$
2L_{\{x\},\{y\}}+L_{\{z\},\{x,y\}}+L_{\{t\},\{x,y\}}\sim H_\lambda.
$$
Thus, the intersection matrix of the curves
$L_{\{x\},\{y\}}$, $L_{\{x\},\{t\}}$,  $L_{\{y\},\{t\}}$, $L_{\{z\},\{t\}}$, $L_{\{z\},\{x,y\}}$,
$L_{\{t\},\{x,y\}}$, $\mathcal{C}_1$, and $\mathcal{C}_2$  on the surface~$S_\lambda$
has the same rank as the intersection matrix
\begin{center}\renewcommand\arraystretch{1.42}
\begin{tabular}{|c||c|c|c|c|}
\hline
 $\bullet$  & $L_{\{z\},\{x,y\}}$ & $L_{\{z\},\{x,t\}}$ & $L_{\{t\},\{x,y\}}$ & $H_{\lambda}$ \\
\hline\hline
 $L_{\{z\},\{x,y\}}$ &  $-\frac{5}{4}$ & $1$ & $\frac{1}{4}$ & $1$ \\
\hline
 $L_{\{z\},\{x,t\}}$ &  $1$ & $-\frac{7}{12}$ & $\frac{1}{2}$ & $1$ \\
\hline
 $L_{\{t\},\{x,y\}}$ &  $\frac{1}{4}$ & $\frac{1}{2}$ & $-1$ & $1$ \\
\hline
 $H_{\lambda}$  & $1$ & $1$ & $1$ & $4$ \\
\hline
\end{tabular}
\end{center}
Its rank is $4$.
On the other hand, we have $\mathrm{rk}\,\mathrm{Pic}(\widetilde{S}_{\Bbbk})=\mathrm{rk}\,\mathrm{Pic}(S_{\Bbbk})+12$,
because the quadratic term of the defining equation of the surface $S_\lambda$
at $P_{\{x\},\{y\},\{z,t\}}$ is $x^2+y^2+\lambda xy$, which is irreducible over $\Bbbk$.
Then \eqref{equation:main-2} in Main Theorem holds by Lemma~\ref{lemma:cokernel}.

\subsection{Family \textnumero $4.3$}
\label{section:r-4-n-3}

In this case, the threefold $X$ is  a blow up of $\mathbb{P}^1\times\mathbb{P}^1\times\mathbb{P}^1$ at a smooth rational curve of tridegree $(1,1,2)$.
Thus, we have $h^{1,2}(X)=0$.
A~toric Landau--Ginzburg model of this family is given by Minkowski polynomial \textnumero $740$, which is
$$
x+y+z+\frac{y}{z}+\frac{y}{x}+\frac{z}{y}+\frac{z}{x}+\frac{1}{z}+\frac{1}{y}+\frac{1}{x}.
$$
The quartic pencil $\mathcal{S}$ is given by the following equation:
$$
x^2yz+y^2zx+z^2yx+y^2tx+y^2tz+z^2tx+z^2ty+t^2yx+t^2zx+t^2yz=\lambda xyzt.
$$
As usual, we suppose that $\lambda\ne\infty$.

The base locus of the pencil $\mathcal{S}$ consists of the lines
$L_{\{x\},\{y\}}$, $L_{\{x\},\{z\}}$, $L_{\{x\},\{t\}}$, $L_{\{y\},\{z\}}$, $L_{\{y\},\{t\}}$, $L_{\{z\},\{t\}}$,
$L_{\{x\},\{y,z,t\}}$, $L_{\{y\},\{z,t\}}$, $L_{\{z\},\{y,t\}}$, and $L_{\{t\},\{x,y,z\}}$, because
\begin{equation}
\label{equation:4-3}
\begin{split}
H_{\{x\}}\cdot S_\lambda&=L_{\{x\},\{y\}}+L_{\{x\},\{z\}}+L_{\{x\},\{t\}}+L_{\{x\},\{y,z,t\}},\\
H_{\{y\}}\cdot S_\lambda&=L_{\{x\},\{y\}}+L_{\{y\},\{z\}}+L_{\{y\},\{t\}}+L_{\{y\},\{z,t\}},\\
H_{\{z\}}\cdot S_\lambda&=L_{\{x\},\{z\}}+L_{\{y\},\{z\}}+L_{\{z\},\{t\}}+L_{\{z\},\{y,t\}},\\
H_{\{t\}}\cdot S_\lambda&=L_{\{x\},\{t\}}+L_{\{y\},\{t\}}+L_{\{z\},\{t\}}+L_{\{t\},\{x,y,z\}}.
\end{split}
\end{equation}

For every $\lambda\in\mathbb{C}$, the surface $S_\lambda$ is irreducible, it has isolated singularities,
and its singular points  contained in the base locus of the pencil $\mathcal{S}$ can be described as follows:
\begin{itemize}\setlength{\itemindent}{3cm}
\item[$P_{\{x\},\{y\},\{z\}}$:] type $\mathbb{A}_1$;

\item[$P_{\{x\},\{y\},\{t\}}$:] type $\mathbb{A}_1$;

\item[$P_{\{x\},\{z\},\{t\}}$:] type $\mathbb{A}_1$;

\item[$P_{\{y\},\{z\},\{t\}}$:] type $\mathbb{A}_3$ with quadratic term $yz$;

\item[$P_{\{x\},\{y\},\{z,t\}}$:] type $\mathbb{A}_1$ for $\lambda\neq -3$, type $\mathbb{A}_2$ for $\lambda=-3$;

\item[$P_{\{x\},\{z\},\{y,t\}}$:] type $\mathbb{A}_1$ for $\lambda\neq -3$, type $\mathbb{A}_2$ for $\lambda=-3$;

\item[$P_{\{x\},\{t\},\{y,z\}}$:] type $\mathbb{A}_1$ for $\lambda\neq -3$, type $\mathbb{A}_2$ for $\lambda=-3$.
\end{itemize}
Then $[\mathsf{f}^{-1}(\lambda)]=1$ for every $\lambda\in\mathbb{C}$ by Lemma~\ref{corollary:irreducible-fibers}.
This confirms \eqref{equation:main-1} in Main Theorem.

To verify \eqref{equation:main-2} in Main Theorem, observe that
$\mathrm{rk}\,\mathrm{Pic}(\widetilde{S}_{\Bbbk})=\mathrm{rk}\,\mathrm{Pic}(S_{\Bbbk})+9$.
On the other hand, it follows from \eqref{equation:4-3} that
the intersection matrix of the lines
$L_{\{x\},\{y\}}$, $L_{\{x\},\{z\}}$, $L_{\{x\},\{t\}}$, $L_{\{y\},\{z\}}$, $L_{\{y\},\{t\}}$, $L_{\{z\},\{t\}}$,
$L_{\{x\},\{y,z,t\}}$, $L_{\{y\},\{z,t\}}$, $L_{\{z\},\{y,t\}}$, and $L_{\{t\},\{x,y,z\}}$ on the surface $S_\lambda$
has the same rank as the intersection matrix of the curves $L_{\{x\},\{y\}}$, $L_{\{x\},\{z\}}$, $L_{\{x\},\{t\}}$, $L_{\{y\},\{z\}}$, $L_{\{y\},\{t\}}$, $L_{\{t\},\{x,y,z\}}$, and $H_{\lambda}$.
If $\lambda\ne -3$, the latter matrix is given by
\begin{center}\renewcommand\arraystretch{1.42}
\begin{tabular}{|c||c|c|c|c|c|c|c|}
\hline
 $\bullet$  & $L_{\{x\},\{y\}}$ & $L_{\{x\},\{z\}}$ & $L_{\{x\},\{t\}}$ & $L_{\{y\},\{z\}}$ & $L_{\{y\},\{t\}}$ & $L_{\{t\},\{x,y,z\}}$ & $H_{\lambda}$ \\
\hline\hline
$L_{\{x\},\{y\}}$ & $-\frac{1}{2}$ & $\frac{1}{2}$ & $\frac{1}{2}$ & $\frac{1}{2}$ & $\frac{1}{2}$ & $0$ & $1$ \\
\hline
$L_{\{x\},\{z\}}$ & $\frac{1}{2}$ & $-\frac{1}{2}$ & $\frac{1}{2}$ &  $\frac{1}{2}$ & $0$ & $0$ & $1$ \\
\hline
$L_{\{x\},\{t\}}$ & $\frac{1}{2}$ & $\frac{1}{2}$ & $-\frac{1}{2}$ &  $0$ & $\frac{1}{2}$ & $\frac{1}{2}$ & $1$ \\
\hline
$L_{\{y\},\{z\}}$ & $\frac{1}{2}$ & $\frac{1}{2}$ & $0$ &  $-1$ & $\frac{1}{2}$ & $0$ & $1$ \\
\hline
$L_{\{y\},\{t\}}$ & $\frac{1}{2}$ & $0$ & $\frac{1}{2}$ &  $\frac{1}{2}$ & $-\frac{1}{4}$ & $1$ & $1$ \\
\hline
$L_{\{t\},\{x,y,z\}}$ & $0$ & $0$ & $\frac{1}{2}$ &  $0$ & $1$ & $-\frac{3}{2}$ & $1$ \\
\hline
 $H_{\lambda}$  & $1$ & $1$ & $1$ & $1$ & $1$ & $1$ & $4$ \\
\hline
\end{tabular}
\end{center}
Its rank is $7$, so that \eqref{equation:main-2-simple} holds.
Then \eqref{equation:main-2} in Main Theorem holds by Lemma~\ref{lemma:cokernel}.

\subsection{Family \textnumero $4.4$}
\label{section:r-4-n-4}

In this case, we have $h^{1,2}(X)=0$.
A~toric Landau--Ginzburg model of this family is given by Minkowski polynomial \textnumero $426$, which is
$$
x+y+z+\frac{x}{z}+\frac{y}{z}+\frac{x}{y}+\frac{y}{x}+\frac{1}{y}+\frac{1}{x}.
$$
The quartic pencil $\mathcal{S}$ is given by
$$
x^2yz+y^2zx+z^2yx+x^2ty+y^2tx+x^2tz+y^2tz+t^2zx+t^2yz=\lambda xyzt.
$$

Suppose that $\lambda\ne\infty$.
Then
\begin{equation}
\label{equation:4-4}
\begin{split}
H_{\{x\}}\cdot S_\lambda&=L_{\{x\},\{y\}}+L_{\{x\},\{z\}}+L_{\{x\},\{t\}}+L_{\{x\},\{y,t\}},\\
H_{\{y\}}\cdot S_\lambda&=L_{\{x\},\{y\}}+L_{\{y\},\{z\}}+L_{\{y\},\{t\}}+L_{\{y\},\{x,t\}},\\
H_{\{z\}}\cdot S_\lambda&=L_{\{x\},\{z\}}+L_{\{y\},\{z\}}+L_{\{z\},\{t\}}+L_{\{z\},\{x,y\}},\\
H_{\{t\}}\cdot S_\lambda&=L_{\{x\},\{t\}}+L_{\{y\},\{t\}}+L_{\{z\},\{t\}}+L_{\{t\},\{x,y,z\}}.
\end{split}
\end{equation}

Each surface $S_\lambda$ is irreducible, it has isolated singularities,
and its singular points contained in the base locus of the pencil $\mathcal{S}$ can be described as follows:
\begin{itemize}\setlength{\itemindent}{3cm}
\item[$P_{\{x\},\{y\},\{t\}}$:] type $\mathbb{A}_3$ with quadratic term $xy$;

\item[$P_{\{x\},\{z\},\{t\}}$:] type $\mathbb{A}_1$;

\item[$P_{\{y\},\{z\},\{t\}}$:] type $\mathbb{A}_1$;

\item[$P_{\{x\},\{y\},\{z\}}$:] type $\mathbb{A}_3$ with quadratic term $yz$ for $\lambda\neq -2$, type $\mathbb{A}_5$ for $\lambda=-2$;

\item[$P_{\{z\},\{t\},\{x,y\}}$:] type $\mathbb{A}_1$ for $\lambda\neq -3$, type $\mathbb{A}_3$ for $\lambda=-3$.
\end{itemize}
Then each fiber $\mathsf{f}^{-1}(\lambda)$ is irreducible by Lemma~\ref{corollary:irreducible-fibers}.
This confirms \eqref{equation:main-1} in Main Theorem.

Let us prove \eqref{equation:main-2} in Main Theorem. We may assume that $\lambda\ne-2$ and $\lambda=-3$.
Then the intersection matrix of the curves $L_{\{x\},\{y,t\}}$, $L_{\{y\},\{x,t\}}$, $L_{\{y\},\{z\}}$, $L_{\{z\},\{t\}}$, $L_{\{z\},\{x,y\}}$, $L_{\{t\},\{x,y,z\}}$, and $H_{\lambda}$
on the surface $S_\lambda$ is given by
\begin{center}\renewcommand\arraystretch{1.42}
\begin{tabular}{|c||c|c|c|c|c|c|c|}
\hline
 $\bullet$  & $L_{\{x\},\{y,t\}}$ & $L_{\{y\},\{x,t\}}$ & $L_{\{y\},\{z\}}$ & $L_{\{z\},\{t\}}$ & $L_{\{z\},\{x,y\}}$ & $L_{\{t\},\{x,y,z\}}$ & $H_{\lambda}$ \\
\hline\hline
$L_{\{x\},\{y,t\}}$ & $-\frac{5}{4}$ & $\frac{1}{4}$ & $0$ & $0$ & $0$ & $0$ & $1$ \\
\hline
$L_{\{y\},\{x,t\}}$ & $\frac{1}{4}$ & $-\frac{5}{4}$ & $1$ &  $0$ & $0$ & $0$ & $1$ \\
\hline
$L_{\{y\},\{z\}}$ & $0$ & $1$ & $-\frac{1}{2}$ &  $\frac{1}{2}$ & $\frac{1}{2}$ & $0$ & $1$ \\
\hline
$L_{\{z\},\{t\}}$ & $0$ & $0$ & $\frac{1}{2}$ &  $-\frac{1}{2}$ & $\frac{1}{2}$ & $\frac{1}{2}$ & $1$ \\
\hline
$L_{\{z\},\{x,y\}}$ & $0$ & $0$ & $\frac{1}{2}$ &  $\frac{1}{2}$ & $-\frac{3}{4}$ & $\frac{1}{2}$ & $1$ \\
\hline
$L_{\{t\},\{x,y,z\}}$ & $0$ & $0$ & $0$ &  $\frac{1}{2}$ & $\frac{1}{2}$ & $-\frac{3}{2}$ & $1$ \\
\hline
 $H_{\lambda}$  & $1$ & $1$ & $1$ & $1$ & $1$ & $1$ & $4$ \\
\hline
\end{tabular}
\end{center}
This matrix has rank~$7$.
Thus, it follows from \eqref{equation:4-4} that
the rank of the intersection matrix of the lines
$L_{\{x\},\{y\}}$, $L_{\{x\},\{z\}}$, $L_{\{x\},\{t\}}$,
$L_{\{y\},\{z\}}$, $L_{\{z\},\{t\}}$, $L_{\{y\},\{t\}}$, $L_{\{x\},\{y,t\}}$, $L_{\{y\},\{x,t\}}$,
$L_{\{z\},\{x,y\}}$, and $L_{\{t\},\{x,y,z\}}$ is also $7$.
But  $\mathrm{rk}\,\mathrm{Pic}(\widetilde{S}_{\Bbbk})=\mathrm{rk}\,\mathrm{Pic}(S_{\Bbbk})+9$,
so that \eqref{equation:main-2-simple} holds.
Then \eqref{equation:main-2} in Main Theorem holds by Lemma~\ref{lemma:cokernel}.

\subsection{Family \textnumero $4.5$}
\label{section:r-4-n-5}

In this case, we have $h^{1,2}(X)=0$.
A~toric Landau--Ginzburg model of this family is given by Minkowski polynomial \textnumero $425$, is
$$
x+y+z+\frac{y}{z}+\frac{y}{x}+\frac{z}{y}+\frac{1}{z}+\frac{1}{y}+\frac{1}{x}.
$$
Then the pencil $\mathcal{S}$ is given by the equation
$$
x^2yz+y^2zx+z^2yx+y^2tx+y^2tz+z^2tx+t^2yx+t^2zx+t^2yz=\lambda xyzt.
$$

Suppose that $\lambda\ne\infty$. Then
\begin{equation}
\label{equation:4-5}
\begin{split}
H_{\{x\}}\cdot S_\lambda&=L_{\{x\},\{y\}}+L_{\{x\},\{z\}}+L_{\{x\},\{t\}}+L_{\{x\},\{y,t\}},\\
H_{\{y\}}\cdot S_\lambda&=L_{\{x\},\{y\}}+L_{\{y\},\{z\}}+L_{\{y\},\{t\}}+L_{\{y\},\{z,t\}},\\
H_{\{z\}}\cdot S_\lambda&=L_{\{x\},\{z\}}+L_{\{y\},\{z\}}+L_{\{z\},\{t\}}+L_{\{z\},\{y,t\}},\\
H_{\{t\}}\cdot S_\lambda&=L_{\{x\},\{t\}}+L_{\{y\},\{t\}}+L_{\{z\},\{t\}}+L_{\{t\},\{x,y,z\}}.
\end{split}
\end{equation}

Observe that $S_\lambda$ is irreducible,
it has isolated singularities, and its singular points contained in the base locus of the pencil $\mathcal{S}$ can be described as follows:
\begin{itemize}\setlength{\itemindent}{2cm}
\item[$P_{\{x\},\{y\},\{z\}}$:] type $\mathbb{A}_1$;

\item[$P_{\{x\},\{y\},\{t\}}$:] type $\mathbb{A}_3$ with quadratic term $x(y+t)$ for $\lambda\neq -2$, type $\mathbb{A}_4$ for $\lambda=-2$;

\item[$P_{\{x\},\{z\},\{t\}}$:] type $\mathbb{A}_1$;

\item[$P_{\{y\},\{z\},\{t\}}$:] type $\mathbb{A}_3$ with quadratic term $yz$;

\item[$P_{\{x\},\{z\},\{y,t\}}$:] type $\mathbb{A}_1$ for $\lambda\neq -2$, type $\mathbb{A}_2$ for $\lambda=-2$.
\end{itemize}
Then each fiber $\mathsf{f}^{-1}(\lambda)$ is irreducible by Lemma~\ref{corollary:irreducible-fibers}.
This confirms \eqref{equation:main-1} in Main Theorem.

It follows from \eqref{equation:4-5} that
the intersection matrix of the base curves of the pencil $\mathcal{S}$ on the surface $S_\lambda$
has the same rank as the intersection matrix of the curves $L_{\{x\},\{y\}}$, $L_{\{x\},\{z\}}$, $L_{\{x\},\{t\}}$, $L_{\{y\},\{z\}}$, $L_{\{y\},\{z,t\}}$, $L_{\{t\},\{x,y,z\}}$, and $H_{\lambda}$.
If~$\lambda\ne-2$, the latter matrix is given~by
\begin{center}\renewcommand\arraystretch{1.42}
\begin{tabular}{|c||c|c|c|c|c|c|c|}
\hline
 $\bullet$  & $L_{\{x\},\{y\}}$ & $L_{\{x\},\{z\}}$ & $L_{\{x\},\{t\}}$ & $L_{\{y\},\{z\}}$ & $L_{\{y\},\{z,t\}}$ & $L_{\{t\},\{x,y,z\}}$ & $H_{\lambda}$ \\
\hline\hline
$L_{\{x\},\{y\}}$ & $-\frac{3}{4}$ & $\frac{1}{2}$ & $\frac{3}{4}$ & $\frac{1}{2}$ & $1$ & $0$ & $1$ \\
\hline
$L_{\{x\},\{z\}}$ & $\frac{1}{2}$ & $-\frac{1}{2}$ & $\frac{1}{2}$ &  $\frac{1}{2}$ & $0$ & $0$ & $1$ \\
\hline
$L_{\{x\},\{t\}}$ & $\frac{3}{4}$ & $\frac{1}{2}$ & $-\frac{3}{4}$ &  $0$ & $0$ & $1$ & $1$ \\
\hline
$L_{\{y\},\{z\}}$ & $\frac{1}{2}$ & $\frac{1}{2}$ & $0$ &  $-\frac{1}{2}$ & $\frac{1}{2}$ & $0$ & $1$ \\
\hline
$L_{\{y\},\{z,t\}}$ & $1$ & $0$ & $0$ &  $\frac{1}{2}$ & $-\frac{5}{4}$ & $0$ & $1$ \\
\hline
$L_{\{t\},\{x,y,z\}}$ & $0$ & $0$ & $1$ &  $0$ & $0$ & $-2$ & $1$ \\
\hline
 $H_{\lambda}$  & $1$ & $1$ & $1$ & $1$ & $1$ & $1$ & $4$ \\
\hline
\end{tabular}
\end{center}
The determinant of this matrix is $\frac{39}{128}$.
However, we also have $\mathrm{rk}\,\mathrm{Pic}(\widetilde{S}_{\Bbbk})=\mathrm{rk}\,\mathrm{Pic}(S_{\Bbbk})+9$.
Therefore, we see that \eqref{equation:main-2} in Main Theorem holds by Lemma~\ref{lemma:cokernel}.

\subsection{Family \textnumero $4.6$}
\label{section:r-4-n-6}

In this case, the threefold $X$ is  a blow up of $\mathbb{P}^3$ in a disjoint union of three lines.
Thus, we have $h^{1,2}(X)=0$.
A~toric Landau--Ginzburg model of this family is given by Minkowski polynomial \textnumero $423$, which is
$$
x+y+z+\frac{z}{y}+\frac{1}{z}+\frac{1}{y}+\frac{1}{x}+\frac{1}{xz}+\frac{1}{xy}.
$$
The quartic pencil $\mathcal{S}$ is given by the following equation:
$$
x^2yz+y^2zx+z^2yx+z^2tx+t^2yx+t^2zx+t^2yz+t^3y+t^3z=\lambda xyzt.
$$
As usual, we assume that $\lambda\ne\infty$.

Let $\mathcal{C}_1$ be the conic in $\mathbb{P}^3$ that is given by $x=yz+yt+zt=0$,
and let $\mathcal{C}_2$ be the conic in $\mathbb{P}^3$ that is given by $y=xz+xt+t^2=0$.
Then
\begin{equation}
\label{equation:4-6}
\begin{split}
H_{\{x\}}\cdot S_\lambda&=2L_{\{x\},\{t\}}+\mathcal{C}_1,\\
H_{\{y\}}\cdot S_\lambda&=L_{\{y\},\{z\}}+L_{\{y\},\{t\}}+\mathcal{C}_2,\\
H_{\{z\}}\cdot S_\lambda&=L_{\{y\},\{z\}}+2L_{\{z\},\{t\}}+L_{\{z\},\{x,t\}},\\
H_{\{t\}}\cdot S_\lambda&=L_{\{x\},\{t\}}+L_{\{y\},\{t\}}+L_{\{z\},\{t\}}+L_{\{t\},\{x,y,z\}}.
\end{split}
\end{equation}

For every $\lambda$, the surface $S_\lambda$ is irreducible, it has isolated singularities,
and its singular points contained in the base locus of the pencil $\mathcal{S}$ can be described as follows:
\begin{itemize}\setlength{\itemindent}{2cm}
\item[$P_{\{x\},\{y\},\{t\}}$:] type $\mathbb{A}_3$ with quadratic term $x(y+t)$ for $\lambda\neq -3$, type $\mathbb A_5$ for $\lambda=-3$;

\item[$P_{\{x\},\{z\},\{t\}}$:] type $\mathbb{A}_2$ with quadratic term $xz$;

\item[$P_{\{y\},\{z\},\{t\}}$:] type $\mathbb{A}_3$ with quadratic term $yz$;

\item[$P_{\{x\},\{t\},\{y,z\}}$:] type $\mathbb{A}_1$;

\item[$P_{\{y\},\{z\},\{x,t\}}$:] type $\mathbb{A}_1$ for $\lambda\neq -3$, type $\mathbb A_3$ for $\lambda=-3$;

\item[$P_{\{z\},\{t\},\{x,y\}}$:] type $\mathbb{A}_1$.
\end{itemize}
By Lemma~\ref{corollary:irreducible-fibers}, each fiber $\mathsf{f}^{-1}(\lambda)$ is irreducible.
This confirms \eqref{equation:main-1} in Main Theorem.

To verify \eqref{equation:main-2} in Main Theorem, observe that
$\mathrm{rk}\,\mathrm{Pic}(\widetilde{S}_{\Bbbk})=\mathrm{rk}\,\mathrm{Pic}(S_{\Bbbk})+11$.
On the other hand, it follows from \eqref{equation:4-6} that the intersection matrix of the curves
$2L_{\{x\},\{t\}}$, $L_{\{y\},\{z\}}$, $L_{\{y\},\{t\}}$, $L_{\{z\},\{t\}}$, $L_{\{z\},\{x,t\}}$, $L_{\{t\},\{x,y,z\}}$, $\mathcal{C}_1$, and $\mathcal{C}_2$
has the same rank as the intersection matrix of the curves $L_{\{x\},\{t\}}$, $L_{\{y\},\{z\}}$, $L_{\{y\},\{t\}}$, $L_{\{z\},\{t\}}$, and $H_{\lambda}$.
If $\lambda\ne-3$, then the latter matrix is given~by
\begin{center}\renewcommand\arraystretch{1.42}
\begin{tabular}{|c||c|c|c|c|c|}
\hline
 $\bullet$  & $L_{\{x\},\{t\}}$ & $L_{\{y\},\{z\}}$ & $L_{\{y\},\{t\}}$ & $L_{\{z\},\{t\}}$ &  $H_{\lambda}$ \\
\hline\hline
$L_{\{x\},\{t\}}$ & $-\frac{1}{12}$ & $0$ & $\frac{1}{4}$ & $\frac{1}{3}$ & $1$ \\
\hline
$L_{\{y\},\{z\}}$ & $0$ &  $-\frac{1}{2}$ & $\frac{1}{2}$ & $\frac{1}{2}$ & $1$ \\
\hline
$L_{\{y\},\{t\}}$ & $\frac{1}{4}$ &  $\frac{1}{2}$ & $-\frac{1}{2}$ & $\frac{1}{4}$ & $1$ \\
\hline
$L_{\{z\},\{t\}}$ & $\frac{1}{3}$ &  $\frac{1}{2}$ & $\frac{1}{4}$ & $-\frac{1}{12}$ & $1$ \\
\hline
 $H_{\lambda}$  & $1$ & $1$ & $1$ & $1$ & $4$ \\
\hline
\end{tabular}
\end{center}
Its rank is $5$, so that \eqref{equation:main-2-simple} holds.
Then \eqref{equation:main-2} in Main Theorem holds by Lemma~\ref{lemma:cokernel}.

\subsection{Family \textnumero $4.7$}
\label{section:r-4-n-7}

In this case, the threefold $X$ can be obtained by blowing up a smooth hypersurface in $\mathbb{P}^2\times\mathbb{P}^2$ of bidegree $(1,1)$
in a disjoint union of two smooth rational curves.
This shows that $h^{1,2}(X)=0$.
A~toric Landau--Ginzburg model of this family is given by Minkowski polynomial \textnumero $215$, which is
$$
x+y+z+\frac{z}{y}+\frac{1}{z}+\frac{1}{y}+\frac{1}{x}+\frac{1}{xz}.
$$
The quartic pencil $\mathcal{S}$ is given by
$$
x^2yz+y^2zx+z^2yx+z^2tx+t^2yx+t^2zx+t^2yz+t^3y=\lambda xyzt.
$$

Suppose that $\lambda\ne\infty$. Then
\begin{equation}
\label{equation:4-7}
\begin{split}
H_{\{x\}}\cdot S_\lambda&=L_{\{x\},\{y\}}+2L_{\{x\},\{t\}}+L_{\{x\},\{z,t\}},\\
H_{\{y\}}\cdot S_\lambda&=L_{\{x\},\{y\}}+L_{\{y\},\{z\}}+L_{\{y\},\{t\}}+L_{\{y\},\{z,t\}},\\
H_{\{z\}}\cdot S_\lambda&=L_{\{y\},\{z\}}+2L_{\{z\},\{t\}}+L_{\{z\},\{x,t\}},\\
H_{\{t\}}\cdot S_\lambda&=L_{\{x\},\{t\}}+L_{\{y\},\{t\}}+L_{\{z\},\{t\}}+L_{\{t\},\{x,y,z\}}.
\end{split}
\end{equation}
Thus, the base locus of the pencil $\mathcal{S}$ consists of the lines
$L_{\{x\},\{y\}}$, $L_{\{x\},\{t\}}$, $L_{\{y\},\{z\}}$, $L_{\{y\},\{t\}}$, $L_{\{z\},\{t\}}$,
$L_{\{x\},\{z,t\}}$, $L_{\{y\},\{z,t\}}$, $L_{\{z\},\{x,t\}}$, and $L_{\{t\},\{x,y,z\}}$.

For every $\lambda\in\mathbb{C}$, the surface $S_\lambda$ has isolated singularities.
Thus, we conclude that every surface $S_\lambda$ is irreducible.
Moreover, the singular points of the surface~$S_\lambda$ contained in the base locus of the pencil $\mathcal{S}$ can be described as follows:
\begin{itemize}\setlength{\itemindent}{3cm}
\item[$P_{\{x\},\{y\},\{t\}}$:] type $\mathbb{A}_2$ with quadratic term $x(y+t)$;

\item[$P_{\{x\},\{z\},\{t\}}$:] type $\mathbb{A}_2$ with quadratic term $xz$;

\item[$P_{\{y\},\{z\},\{t\}}$:] type $\mathbb{A}_3$ with quadratic term $yz$;

\item[$P_{\{x\},\{y\},\{z,t\}}$:] type $\mathbb{A}_1$ for $\lambda\neq -2$, type $\mathbb A_2$ for $\lambda=-2$;

\item[$P_{\{x\},\{t\},\{y,z\}}$:] type $\mathbb{A}_1$;

\item[$P_{\{z\},\{t\},\{x,y\}}$:] type $\mathbb{A}_1$.
\end{itemize}
By Lemma~\ref{corollary:irreducible-fibers}, each fiber $\mathsf{f}^{-1}(\lambda)$ is irreducible.
This confirms \eqref{equation:main-1} in Main Theorem.

To verify \eqref{equation:main-2} in Main Theorem, observe first that
$\mathrm{rk}\,\mathrm{Pic}(\widetilde{S}_{\Bbbk})=\mathrm{rk}\,\mathrm{Pic}(S_{\Bbbk})+10$.
On the other hand, it follows from \eqref{equation:4-7} that the intersection matrix of the lines
$L_{\{x\},\{y\}}$, $L_{\{x\},\{t\}}$, $L_{\{y\},\{z\}}$, $L_{\{y\},\{t\}}$, $L_{\{z\},\{t\}}$,
$L_{\{x\},\{z,t\}}$, $L_{\{y\},\{z,t\}}$, $L_{\{z\},\{x,t\}}$, and $L_{\{t\},\{x,y,z\}}$
on the surface $S_\lambda$
has the same rank as the intersection matrix of the curves $L_{\{x\},\{t\}}$, $L_{\{x\},\{z,t\}}$, $L_{\{y\},\{z,t\}}$, $L_{\{z\},\{x,t\}}$, $L_{\{t\},\{x,y,z\}}$, and $H_{\lambda}$.
If $\lambda\ne-2$, then the latter matrix is given by
\begin{center}\renewcommand\arraystretch{1.42}
\begin{tabular}{|c||c|c|c|c|c|c|}
\hline
 $\bullet$  & $L_{\{x\},\{t\}}$ & $L_{\{x\},\{z,t\}}$ & $L_{\{y\},\{z,t\}}$ & $L_{\{z\},\{x,t\}}$ & $L_{\{t\},\{x,y,z\}}$ &  $H_{\lambda}$ \\
\hline\hline
$L_{\{x\},\{t\}}$ & $-\frac{1}{6}$ & $\frac{2}{3}$ & $0$ & $\frac{1}{3}$ & $\frac{1}{2}$ & $1$ \\
\hline
$L_{\{x\},\{z,t\}}$ & $\frac{2}{3}$ &  $-\frac{5}{6}$ & $\frac{1}{2}$ & $\frac{1}{3}$ & $\frac{1}{2}$ & $1$ \\
\hline
$L_{\{y\},\{z,t\}}$ & $0$ &  $\frac{1}{2}$ & $-\frac{5}{4}$ & $0$ & $0$ & $1$ \\
\hline
$L_{\{z\},\{x,t\}}$ & $\frac{1}{3}$ &  $\frac{1}{3}$ & $0$ & $-\frac{4}{3}$ & $0$ & $1$ \\
\hline
$L_{\{t\},\{x,y,z\}}$ & $\frac{1}{2}$ &  $\frac{1}{2}$ & $0$ & $0$ & $-1$ & $1$ \\
\hline
 $H_{\lambda}$  & $1$ & $1$ & $1$ & $1$ & $1$ & $4$ \\
\hline
\end{tabular}
\end{center}
Its rank is $6$, so that \eqref{equation:main-2-simple} holds.
Then \eqref{equation:main-2} in Main Theorem holds by Lemma~\ref{lemma:cokernel}.

\subsection{Family \textnumero $4.8$}
\label{section:r-4-n-8}

The threefold $X$ can be obtained by blowing up $\mathbb{P}^1\times\mathbb{P}^1\times\mathbb{P}^1$
along a smooth rational curve of tridegree $(1,1,0)$.
Then $h^{1,2}(X)=0$.
A~toric Landau--Ginzburg model of this family is given by Minkowski polynomial \textnumero $216$, which is
$$
x+y+z+\frac{z}{y}+\frac{z}{x}+\frac{1}{z}+\frac{1}{y}+\frac{1}{x}.
$$
The quartic pencil $\mathcal{S}$ is given by
$$
x^2yz+y^2zx+z^2yx+z^2tx+z^2ty+t^2yx+t^2zx+t^2yz=\lambda xyzt.
$$

Suppose that $\lambda\ne\infty$. Then
\begin{equation}
\label{equation:4-8}
\begin{split}
H_{\{x\}}\cdot S_\lambda&=L_{\{x\},\{y\}}+L_{\{x\},\{z\}}+L_{\{x\},\{t\}}+L_{\{x\},\{z,t\}},\\
H_{\{y\}}\cdot S_\lambda&=L_{\{x\},\{y\}}+L_{\{y\},\{z\}}+L_{\{y\},\{t\}}+L_{\{y\},\{z,t\}},\\
H_{\{z\}}\cdot S_\lambda&=L_{\{x\},\{z\}}+L_{\{y\},\{z\}}+2L_{\{z\},\{t\}},\\
H_{\{t\}}\cdot S_\lambda&=L_{\{x\},\{t\}}+L_{\{y\},\{t\}}+L_{\{z\},\{t\}}+L_{\{t\},\{x,y,z\}}.
\end{split}
\end{equation}
Thus, the base locus of the pencil $\mathcal{S}$ consists of the lines
$L_{\{x\},\{y\}}$, $L_{\{x\},\{z\}}$, $L_{\{x\},\{t\}}$, $L_{\{y\},\{z\}}$, $L_{\{y\},\{t\}}$, $L_{\{z\},\{t\}}$,
$L_{\{x\},\{z,t\}}$, $L_{\{y\},\{z,t\}}$, and $L_{\{t\},\{x,y,z\}}$.

For every $\lambda\in\mathbb{C}$, the surface $S_\lambda$ is irreducible, it has isolated singularities
and its singular points contained in the base locus of the pencil $\mathcal{S}$ can be described as follows:
\begin{itemize}\setlength{\itemindent}{3cm}
\item[$P_{\{x\},\{y\},\{z\}}$:] type $\mathbb{A}_1$;

\item[$P_{\{x\},\{y\},\{t\}}$:] type $\mathbb{A}_1$;

\item[$P_{\{x\},\{z\},\{t\}}$:] type $\mathbb{A}_3$ with quadratic term $xz$;

\item[$P_{\{y\},\{z\},\{t\}}$:] type $\mathbb{A}_3$ with quadratic term $yz$;

\item[$P_{\{x\},\{y\},\{z,t\}}$:] type $\mathbb{A}_1$ for $\lambda\neq -2$, type $\mathbb A_5$ for $\lambda=-2$;

\item[$P_{\{z\},\{t\},\{x,y\}}$:] type $\mathbb{A}_1$.
\end{itemize}
By Lemma~\ref{corollary:irreducible-fibers}, each fiber $\mathsf{f}^{-1}(\lambda)$ is irreducible.
This confirms \eqref{equation:main-1} in Main Theorem.

To verify \eqref{equation:main-2} in Main Theorem, observe that $\mathrm{rk}\,\mathrm{Pic}(\widetilde{S}_{\Bbbk})=\mathrm{rk}\,\mathrm{Pic}(S_{\Bbbk})+10$.
On the other hand, it follows from \eqref{equation:4-8} that the intersection matrix of the lines
$L_{\{x\},\{y\}}$, $L_{\{x\},\{z\}}$, $L_{\{x\},\{t\}}$, $L_{\{y\},\{z\}}$, $L_{\{y\},\{t\}}$, $L_{\{z\},\{t\}}$,
$L_{\{x\},\{z,t\}}$, $L_{\{y\},\{z,t\}}$, and $L_{\{t\},\{x,y,z\}}$ on the surface $S_\lambda$
has the same rank as the intersection matrix of the curves $L_{\{x\},\{y\}}$, $L_{\{x\},\{z\}}$, $L_{\{x\},\{z,t\}}$, $L_{\{y\},\{z,t\}}$, $L_{\{y\},\{t\}}$, and $H_{\lambda}$.
If~$\lambda\ne-2$, then the latter matrix is given by
\begin{center}\renewcommand\arraystretch{1.42}
\begin{tabular}{|c||c|c|c|c|c|c|}
\hline
 $\bullet$  & $L_{\{x\},\{y\}}$ & $L_{\{x\},\{z\}}$ & $L_{\{x\},\{z,t\}}$ & $L_{\{y\},\{z,t\}}$ & $L_{\{y\},\{t\}}$ &  $H_{\lambda}$ \\
\hline\hline
$L_{\{x\},\{y\}}$ & $-\frac{1}{2}$ & $\frac{1}{2}$ & $\frac{1}{2}$ & $\frac{1}{2}$ & $\frac{1}{2}$ & $1$ \\
\hline
$L_{\{x\},\{z\}}$ & $\frac{1}{2}$ &  $-\frac{1}{2}$ & $\frac{1}{2}$ & $0$ & $0$ & $1$ \\
\hline
$L_{\{x\},\{z,t\}}$ & $\frac{1}{2}$ &  $\frac{1}{2}$ & $-\frac{3}{4}$ & $\frac{1}{2}$ & $0$ & $1$ \\
\hline
$L_{\{y\},\{z,t\}}$ & $\frac{1}{2}$ &  $0$ & $\frac{1}{2}$ & $-\frac{3}{4}$ & $\frac{3}{4}$ & $1$ \\
\hline
$L_{\{y\},\{t\}}$ & $\frac{1}{2}$ &  $0$ & $0$ & $\frac{3}{4}$ & $-\frac{3}{4}$ & $1$ \\
\hline
 $H_{\lambda}$  & $1$ & $1$ & $1$ & $1$ & $1$ & $4$ \\
\hline
\end{tabular}
\end{center}
Its rank is $6$, so that \eqref{equation:main-2-simple} holds.
Then \eqref{equation:main-2} in Main Theorem holds by Lemma~\ref{lemma:cokernel}.

\subsection{Family \textnumero $4.9$}
\label{section:r-4-n-9}

In this case, we have $h^{1,2}(X)=0$.
A~toric Landau--Ginzburg model of this family is given by Minkowski polynomial \textnumero $81$, which is
$$
\frac{y}{x}+\frac{1}{x}+y+z+\frac{1}{y}+x+\frac{1}{yz}.
$$
The quartic pencil $\mathcal{S}$ is given by
$$
y^2tz+t^2yz+y^2zx+z^2yx+t^2zx+x^2yz+t^3x=\lambda xyzt.
$$

As usual, we suppose that $\lambda\ne\infty$. Then
\begin{equation}
\label{equation:4-9}
\begin{split}
H_{\{x\}}\cdot S_\lambda&=L_{\{x\},\{y\}}+L_{\{x\},\{z\}}+L_{\{x\},\{t\}}+L_{\{x\},\{y,t\}},\\
H_{\{y\}}\cdot S_\lambda&=L_{\{x\},\{y\}}+2L_{\{y\},\{t\}}+2L_{\{y\},\{z,t\}},\\
H_{\{z\}}\cdot S_\lambda&=L_{\{x\},\{z\}}+3L_{\{z\},\{t\}},\\
H_{\{t\}}\cdot S_\lambda&=L_{\{x\},\{t\}}+L_{\{y\},\{t\}}+L_{\{z\},\{t\}}+L_{\{t\},\{x,y,z\}}.
\end{split}
\end{equation}
Thus, we see that the base locus of the pencil $\mathcal{S}$ consists of the eight lines
$L_{\{x\},\{y\}}$, $L_{\{x\},\{z\}}$, $L_{\{x\},\{t\}}$, $L_{\{y\},\{t\}}$, $L_{\{z\},\{t\}}$,
$L_{\{x\},\{y,t\}}$, $L_{\{y\},\{z,t\}}$, and $L_{\{t\},\{x,y,z\}}$.

For every $\lambda\in\mathbb{C}$, the surface $S_\lambda$ is irreducible and has isolated singularities.
Moreover, the singular points of the surface $S_\lambda$ contained in the base locus of the pencil $\mathcal{S}$ can be described as follows:
\begin{itemize}\setlength{\itemindent}{3cm}
\item[$P_{\{x\},\{y\},\{t\}}$:] type $\mathbb{A}_3$ with quadratic term $xy$;

\item[$P_{\{x\},\{z\},\{t\}}$:] type $\mathbb{A}_3$ with quadratic term $z(x+t)$;

\item[$P_{\{y\},\{z\},\{t\}}$:] type $\mathbb{A}_2$ with quadratic term $yz$;

\item[$P_{\{y\},\{t\},\{x,z\}}$:] type $\mathbb{A}_1$;

\item[$P_{\{z\},\{t\},\{x,y\}}$:] type $\mathbb{A}_2$ with quadratic term $z(x+y+z-t-\lambda t)$.
\end{itemize}
Therefore, it follows from Lemma~\ref{corollary:irreducible-fibers} that the fiber $\mathsf{f}^{-1}(\lambda)$ is irreducible for every $\lambda\in\mathbb{C}$.
This confirms \eqref{equation:main-1} in Main Theorem in this case, since $h^{1,2}(X)=0$.

To verify \eqref{equation:main-2} in Main Theorem, observe that $\mathrm{rk}\,\mathrm{Pic}(\widetilde{S}_{\Bbbk})=\mathrm{rk}\,\mathrm{Pic}(S_{\Bbbk})+11$.
This immediately follows from the description of singular points of the surface $S_\lambda$ given above.
Note also that
\begin{multline*}
L_{\{x\},\{y\}}+L_{\{x\},\{z\}}+L_{\{x\},\{t\}}+L_{\{x\},\{y,t\}}\sim L_{\{x\},\{y\}}+2L_{\{y\},\{t\}}+2L_{\{y\},\{z,t\}}\sim \\
\sim L_{\{x\},\{z\}}+3L_{\{z\},\{t\}}\sim L_{\{x\},\{t\}}+L_{\{y\},\{t\}}+L_{\{z\},\{t\}}+L_{\{t\},\{x,y,z\}}\sim H_\lambda
\end{multline*}
on the surface $S_\lambda$. This follows from \eqref{equation:4-9}.
Using this, we see that the intersection matrix of the lines
$L_{\{x\},\{y\}}$, $L_{\{x\},\{z\}}$, $L_{\{x\},\{t\}}$, $L_{\{y\},\{t\}}$, $L_{\{z\},\{t\}}$,
$L_{\{x\},\{y,t\}}$, $L_{\{y\},\{z,t\}}$, and $L_{\{t\},\{x,y,z\}}$ on the surface $S_\lambda$
has the same rank as the intersection matrix of the curves $L_{\{x\},\{y,t\}}$, $L_{\{y\},\{t\}}$, $L_{\{y\},\{z,t\}}$, $L_{\{t\},\{x,y,z\}}$, and $H_{\lambda}$.
The later matrix is not hard to compute:
\begin{center}\renewcommand\arraystretch{1.42}
\begin{tabular}{|c||c|c|c|c|c|}
\hline
 $\bullet$  & $L_{\{x\},\{y,t\}}$ & $L_{\{y\},\{t\}}$ & $L_{\{y\},\{z,t\}}$ & $L_{\{t\},\{x,y,z\}}$ &  $H_{\lambda}$ \\
\hline\hline
$L_{\{x\},\{y,t\}}$ & $-\frac{5}{4}$ & $\frac{1}{3}$ & $0$ & $0$ & $1$ \\
\hline
$L_{\{y\},\{t\}}$ & $\frac{1}{3}$ &  $-\frac{1}{12}$ & $\frac{2}{3}$ & $\frac{1}{2}$ & $1$ \\
\hline
$L_{\{y\},\{z,t\}}$ & $0$ &  $\frac{2}{3}$ & $-\frac{4}{3}$ & $0$ & $1$ \\
\hline
$L_{\{t\},\{x,y,z\}}$ & $0$ &  $\frac{1}{2}$ & $0$ & $-\frac{5}{6}$ & $1$ \\
\hline
 $H_{\lambda}$  & $1$ & $1$ & $1$ & $1$ & $4$ \\
\hline
\end{tabular}
\end{center}
The rank of this matrix is $5$.
Thus, we conclude that \eqref{equation:main-2-simple} holds in this case,
so that \eqref{equation:main-2} in Main Theorem also holds by Lemma~\ref{lemma:cokernel}.

\subsection{Family \textnumero $4.10$}
\label{section:r-4-n-10}

In this case, we have $X\cong\mathbb{P}^1\times\mathbf{S}_7$, where $\mathbf{S}_7$ is a smooth del Pezzo surface of degree $7$.
This shows that $h^{1,2}(X)=0$.
A~toric Landau--Ginzburg model of this family is given by Minkowski polynomial \textnumero $84$, which is
$$
\frac{y}{x}+\frac{1}{x}+y+z+\frac{1}{z}+\frac{1}{y}+x.
$$
Thus, the quartic pencil $\mathcal{S}$ is given by the following equation:
$$
y^2tz+t^2zy+y^2xz+z^2xy+t^2xy+t^2xz+x^2zy=\lambda xyzt.
$$

Suppose that $\lambda\ne\infty$. Then
\begin{equation}
\label{equation:4-10}
\begin{split}
H_{\{x\}}\cdot S_\lambda&=L_{\{x\},\{y\}}+L_{\{x\},\{z\}}+L_{\{x\},\{t\}}+L_{\{x\},\{y,t\}},\\
H_{\{y\}}\cdot S_\lambda&=L_{\{x\},\{y\}}+L_{\{y\},\{z\}}+2L_{\{y\},\{t\}},\\
H_{\{z\}}\cdot S_\lambda&=L_{\{x\},\{z\}}+L_{\{y\},\{z\}}+2L_{\{z\},\{t\}},\\
H_{\{t\}}\cdot S_\lambda&=L_{\{x\},\{t\}}+L_{\{y\},\{t\}}+L_{\{z\},\{t\}}+L_{\{t\},\{x,y,z\}}.
\end{split}
\end{equation}
Thus, the base locus of the pencil $\mathcal{S}$ consists of the lines
$L_{\{x\},\{y\}}$, $L_{\{x\},\{z\}}$, $L_{\{x\},\{t\}}$,
$L_{\{y\},\{z\}}$, $L_{\{y\},\{t\}}$, $L_{\{z\},\{t\}}$, $L_{\{x\},\{y,t\}}$, and $L_{\{t\},\{x,y,z\}}$.

For every $\lambda\in\mathbb{C}$, the surface $S_\lambda$ is irreducible and has isolated singularities.
Moreover, the singular points of the surface $S_\lambda$ contained in the base locus of the pencil $\mathcal{S}$ can be described as follows:
\begin{itemize}\setlength{\itemindent}{3cm}
\item[$P_{\{x\},\{y\},\{z\}}$:] type $\mathbb{A}_1$;

\item[$P_{\{x\},\{y\},\{t\}}$:] type $\mathbb{A}_3$ with quadratic term $xy$;

\item[$P_{\{x\},\{z\},\{t\}}$:] type $\mathbb{A}_2$ with quadratic term $z(x+t)$;

\item[$P_{\{y\},\{z\},\{t\}}$:] type $\mathbb{A}_3$ with quadratic term $yz$;

\item[$P_{\{y\},\{t\},\{x,z\}}$:] type $\mathbb{A}_1$;

\item[$P_{\{z\},\{t\},\{x,y\}}$:] type $\mathbb{A}_1$.
\end{itemize}
Therefore, using Lemma~\ref{corollary:irreducible-fibers}, we see that the fiber $\mathsf{f}^{-1}(\lambda)$ is irreducible for every $\lambda\in\mathbb{C}$.
This confirms \eqref{equation:main-1} in Main Theorem in this case, because $h^{1,2}(X)=0$.

Let us prove \eqref{equation:main-2} in Main Theorem.
Observe that the intersection matrix of the curves $L_{\{x\},\{y\}}$, $L_{\{x\},\{z\}}$, $L_{\{x\},\{y,t\}}$, $L_{\{t\},\{x,y,z\}}$, and $H_{\lambda}$
on the surface $S_\lambda$ is given by
 \begin{center}\renewcommand\arraystretch{1.42}
\begin{tabular}{|c||c|c|c|c|c|}
\hline
 $\bullet$  & $L_{\{x\},\{y\}}$ & $L_{\{x\},\{z\}}$ & $L_{\{x\},\{y,t\}}$ & $L_{\{t\},\{x,y,z\}}$ &  $H_{\lambda}$ \\
\hline\hline
$L_{\{x\},\{y\}}$ & $-\frac{1}{2}$ & $\frac{1}{2}$ & $\frac{1}{2}$& $0$ & $1$ \\
\hline
$L_{\{x\},\{z\}}$ & $\frac{1}{2}$ &  $-\frac{5}{6}$ & $1$ & $0$ & $1$ \\
\hline
$L_{\{x\},\{y,t\}}$ & $\frac{1}{2}$ &  $1$ & $-\frac{5}{4}$ & $0$ & $1$ \\
\hline
$L_{\{t\},\{x,y,z\}}$ & $0$ &  $0$ & $0$ & $-1$ & $1$ \\
\hline
 $H_{\lambda}$  & $1$ & $1$ & $1$ & $1$ & $4$ \\
\hline
\end{tabular}
\end{center}
The rank of this matrix is $5$.
On the other hand, it follows from \eqref{equation:3-30} that
\begin{multline*}
H_\lambda\sim L_{\{x\},\{y\}}+L_{\{x\},\{z\}}+L_{\{x\},\{t\}}+L_{\{x\},\{y,t\}}\sim L_{\{x\},\{y\}}+L_{\{y\},\{z\}}+2L_{\{y\},\{t\}}\sim\\
\sim L_{\{x\},\{z\}}+L_{\{y\},\{z\}}+2L_{\{z\},\{t\}}\sim L_{\{x\},\{t\}}+L_{\{y\},\{t\}}+L_{\{z\},\{t\}}+L_{\{t\},\{x,y,z\}}
\end{multline*}
on the surface $S_\lambda$.
Therefore, the intersection matrix of the lines
$L_{\{x\},\{y\}}$, $L_{\{x\},\{z\}}$, $L_{\{x\},\{t\}}$,
$L_{\{y\},\{z\}}$, $L_{\{y\},\{t\}}$, $L_{\{z\},\{t\}}$, $L_{\{x\},\{y,t\}}$, and $L_{\{t\},\{x,y,z\}}$
has the same rank as the intersection matrix of the curves $L_{\{x\},\{y\}}$, $L_{\{x\},\{z\}}$, $L_{\{x\},\{y,t\}}$, $L_{\{t\},\{x,y,z\}}$, and $H_{\lambda}$.
On the other hand, we also have $\mathrm{rk}\,\mathrm{Pic}(\widetilde{S}_{\Bbbk})=\mathrm{rk}\,\mathrm{Pic}(S_{\Bbbk})+11$.
Thus, we see that \eqref{equation:main-2-simple} holds.
Then we use Lemma~\ref{lemma:cokernel} to conclude that \eqref{equation:main-2} in Main Theorem also holds in this case.

\subsection{Family \textnumero $4.11$}
\label{section:r-4-n-11}

In this case, we have $h^{1,2}(X)=0$.
A~toric Landau--Ginzburg model of this family is given by Minkowski polynomial \textnumero $82$, which is
$$
\frac{y}{x}+\frac{1}{x}+y+z+{\frac {1}{xz}}+\frac{1}{y}+x.
$$
Then the quartic pencil $\mathcal{S}$ is given by the following equation:
$$
y^2tz+t^2zy+y^2xz+z^2xy+t^3y+t^2xz+x^2zy=\lambda xyzt.
$$
As usual, we assume that $\lambda\ne\infty$.

Let $\mathcal{C}$ be the conic in $\mathbb{P}^3$ that is given by $x=yz+zt+t^2=0$.
Then
\begin{equation}
\label{equation:4-11}
\begin{split}
H_{\{x\}}\cdot S_\lambda&=L_{\{x\},\{y\}}+L_{\{x\},\{t\}}+\mathcal{C},\\
H_{\{y\}}\cdot S_\lambda&=L_{\{x\},\{y\}}+L_{\{y\},\{z\}}+2L_{\{y\},\{t\}},\\
H_{\{z\}}\cdot S_\lambda&=L_{\{y\},\{z\}}+3L_{\{z\},\{t\}},\\
H_{\{t\}}\cdot S_\lambda&=L_{\{x\},\{t\}}+L_{\{y\},\{t\}}+L_{\{z\},\{t\}}+L_{\{t\},\{x,y,z\}}.
\end{split}
\end{equation}

For each $\lambda$, the surface $S_\lambda$ is irreducible, it has isolated singularities,
and its singular points contained in the base locus of the pencil $\mathcal{S}$ can be described as follows:
\begin{itemize}\setlength{\itemindent}{3cm}
\item[$P_{\{x\},\{y\},\{t\}}$:] type $\mathbb{A}_3$ with quadratic term $xy$;

\item[$P_{\{x\},\{z\},\{t\}}$:] type $\mathbb{A}_2$ with quadratic term $z(x+t)$;

\item[$P_{\{y\},\{z\},\{t\}}$:] type $\mathbb{A}_4$ with quadratic term $yz$;

\item[$P_{\{y\},\{t\},\{x,z\}}$:] type $\mathbb{A}_1$;

\item[$P_{\{z\},\{t\},\{x,y\}}$:] type $\mathbb{A}_2$ with quadratic term $z(x+y+z-t-\lambda t)$.
\end{itemize}
By Lemma~\ref{corollary:irreducible-fibers}, the fiber $\mathsf{f}^{-1}(\lambda)$ is irreducible for every $\lambda\in\mathbb{C}$.
This confirms \eqref{equation:main-1} in Main Theorem, since $h^{1,2}(X)=0$.

The description of the singular points of the surface $S_\lambda$ gives
$\mathrm{rk}\,\mathrm{Pic}(\widetilde{S}_{\Bbbk})=\mathrm{rk}\,\mathrm{Pic}(S_{\Bbbk})+12$.
On the other hand, it follows from \eqref{equation:4-11} that
\begin{multline*}
H_\lambda\sim L_{\{x\},\{y\}}+L_{\{x\},\{t\}}+\mathcal{C}\sim L_{\{x\},\{y\}}+L_{\{y\},\{z\}}+2L_{\{y\},\{t\}}\sim \\
\sim L_{\{y\},\{z\}}+3L_{\{z\},\{t\}}\sim L_{\{x\},\{t\}}+L_{\{y\},\{t\}}+L_{\{z\},\{t\}}+L_{\{t\},\{x,y,z\}}
\end{multline*}
on the surface $S_\lambda$.
Therefore, the intersection matrix of the lines
$L_{\{x\},\{y\}}$, $L_{\{x\},\{t\}}$, $L_{\{y\},\{z\}}$, $L_{\{y\},\{t\}}$, $L_{\{z\},\{t\}}$, $L_{\{t\},\{x,y,z\}}$, and $\mathcal{C}$
has the same rank as the intersection matrix of the curves $L_{\{x\},\{y\}}$, $L_{\{x\},\{t\}}$, $L_{\{t\},\{x,y,z\}}$, and $H_{\lambda}$.
The latter matrix is given by
\begin{center}\renewcommand\arraystretch{1.42}
\begin{tabular}{|c||c|c|c|c|}
\hline
 $\bullet$  & $L_{\{x\},\{y\}}$ & $L_{\{x\},\{t\}}$ & $L_{\{t\},\{x,y,z\}}$ &  $H_{\lambda}$ \\
\hline\hline
$L_{\{x\},\{y\}}$ & $-1$ & $\frac{1}{2}$ & $0$ & $1$ \\
\hline
$L_{\{x\},\{t\}}$ & $\frac{1}{2}$ &  $-\frac{7}{12}$ & $1$ & $1$ \\
\hline
$L_{\{t\},\{x,y,z\}}$ & $0$ &  $1$ & $-\frac{5}{6}$ & $1$ \\
\hline
 $H_{\lambda}$  & $1$ & $1$ & $1$ & $4$ \\
\hline
\end{tabular}
\end{center}
The rank of this matrix is $4$.
Thus, we see that \eqref{equation:main-2-simple} holds in this case.
Then \eqref{equation:main-2} in Main Theorem also holds by Lemma~\ref{lemma:cokernel}.

\subsection{Family \textnumero $4.12$}
\label{section:r-4-n-12}

In this case, we have~$h^{1,2}(X)=0$.
A~toric Landau--Ginzburg model is given by Minkowski polynomial \textnumero $83$, which is
$$
\frac{y}{x}+\frac{y}{xz}+\frac{1}{x}+y+z+\frac{1}{y}+x.
$$
Then the quartic pencil $\mathcal{S}$ is given by the following equation:
$$
y^2tz+t^2y^2+t^2zy+y^2xz+z^2xy+t^2xz+x^2zy=\lambda xyzt.
$$
Here, for simplicity, we suppose that $\lambda\ne\infty$.

Let $\mathcal{C}$ be the conic in $\mathbb{P}^3$ that is given by $x=yz+yt+zt=0$.
Then
\begin{equation}
\label{equation:4-12}
\begin{split}
H_{\{x\}}\cdot S_\lambda&=L_{\{x\},\{y\}}+L_{\{x\},\{t\}}+\mathcal{C},\\
H_{\{y\}}\cdot S_\lambda&=L_{\{x\},\{y\}}+L_{\{y\},\{z\}}+2L_{\{y\},\{t\}},\\
H_{\{z\}}\cdot S_\lambda&=2L_{\{y\},\{z\}}+2L_{\{z\},\{t\}},\\
H_{\{t\}}\cdot S_\lambda&=L_{\{x\},\{t\}}+L_{\{y\},\{t\}}+L_{\{z\},\{t\}}+L_{\{t\},\{x,y,z\}}.
\end{split}
\end{equation}
This shows that the base locus of the pencil $\mathcal{S}$ consists of the curves
$L_{\{x\},\{y\}}$, $L_{\{x\},\{t\}}$, $L_{\{y\},\{z\}}$, $L_{\{y\},\{t\}}$, $L_{\{z\},\{t\}}$, $L_{\{t\},\{x,y,z\}}$, and $\mathcal{C}$.

Every surface $S_\lambda$ is irreducible, it has isolated singularities,
and its singular points contained in the base locus of the pencil $\mathcal{S}$ can be described as follows:
\begin{itemize}\setlength{\itemindent}{3cm}
\item[$P_{\{x\},\{y\},\{z\}}$:] type $\mathbb{A}_1$;

\item[$P_{\{x\},\{y\},\{t\}}$:] type $\mathbb{A}_3$ with quadratic term $xy$;

\item[$P_{\{x\},\{z\},\{t\}}$:] type $\mathbb{A}_1$;

\item[$P_{\{y\},\{z\},\{t\}}$:] type $\mathbb{A}_5$ with quadratic term $yz$;

\item[$P_{\{y\},\{t\},\{x,z\}}$:] type $\mathbb{A}_1$;

\item[$P_{\{z\},\{t\},\{x,y\}}$:] type $\mathbb{A}_1$.
\end{itemize}
By Lemma~\ref{corollary:irreducible-fibers}, every fiber $\mathsf{f}^{-1}(\lambda)$ is irreducible.
This confirms \eqref{equation:main-1} in Main Theorem.

One has $\mathrm{rk}\,\mathrm{Pic}(\widetilde{S}_{\Bbbk})=\mathrm{rk}\,\mathrm{Pic}(S_{\Bbbk})+12$.
On the other hand, it follows from \eqref{equation:4-12} that
\begin{multline*}
H_\lambda\sim L_{\{x\},\{y\}}+L_{\{x\},\{t\}}+\mathcal{C}\sim  L_{\{x\},\{y\}}+L_{\{y\},\{z\}}+2L_{\{y\},\{t\}}\sim\\
\sim 2L_{\{y\},\{z\}}+2L_{\{z\},\{t\}}\sim L_{\{x\},\{t\}}+L_{\{y\},\{t\}}+L_{\{z\},\{t\}}+L_{\{t\},\{x,y,z\}}
\end{multline*}
on the surface $S_\lambda$.
Thus, the intersection matrix of the lines
$L_{\{x\},\{y\}}$, $L_{\{x\},\{t\}}$, $L_{\{y\},\{z\}}$, $L_{\{y\},\{t\}}$, $L_{\{z\},\{t\}}$, $L_{\{t\},\{x,y,z\}}$, and $\mathcal{C}$
on the surface $S_\lambda$ has the same rank as the intersection matrix of the curves $L_{\{x\},\{y\}}$, $L_{\{x\},\{t\}}$, $L_{\{t\},\{x,y,z\}}$, and $H_{\lambda}$.
The latter matrix is given by
\begin{center}\renewcommand\arraystretch{1.42}
\begin{tabular}{|c||c|c|c|c|}
\hline
 $\bullet$  & $L_{\{x\},\{y\}}$ & $L_{\{x\},\{t\}}$ & $L_{\{t\},\{x,y,z\}}$ &  $H_{\lambda}$ \\
\hline\hline
$L_{\{x\},\{y\}}$ & $-\frac{1}{2}$ & $\frac{1}{2}$ & $0$ & $1$ \\
\hline
$L_{\{x\},\{t\}}$ & $\frac{1}{2}$ &  $-\frac{3}{4}$ & $1$ & $1$ \\
\hline
$L_{\{t\},\{x,y,z\}}$ & $0$ &  $1$ & $-1$ & $1$ \\
\hline
 $H_{\lambda}$  & $1$ & $1$ & $1$ & $4$ \\
\hline
\end{tabular}
\end{center}
Its rank is $4$, so that \eqref{equation:main-2-simple} holds.
Then \eqref{equation:main-2} in Main Theorem holds by Lemma~\ref{lemma:cokernel}.

\subsection{Family \textnumero $4.13$}
\label{section:r-4-n-13}

In this case, the threefold $X$ is  a blow up of $\mathbb{P}^1\times\mathbb{P}^1\times\mathbb{P}^1$ along a smooth rational curve of tridegree $(1,1,3)$.
Thus, we have $h^{1,2}(X)=0$.
This family is missed in \cite{IP}.
A~toric Landau--Ginzburg model of this family is given by Minkowski polynomial \textnumero $1080$, which is
$$
x+y+z+\frac{x}{y}+\frac{y}{x}+\frac{x}{yz}+\frac{1}{z}+\frac{2}{y}+\frac{2}{x}+\frac{1}{xy}+\frac{1}{yz}.
$$
The quartic pencil $\mathcal{S}$ is given by
$$
x^2yz+xy^2z+xyz^2+x^2zt+y^2zt+x^2t^2+xyt^2+2xzt^2+2yzt^2+xt^3+zt^3=\lambda xyzt.
$$
As usual, we suppose that $\lambda\ne\infty$.

Let $\mathcal{C}$ be the conic in $\mathbb{P}^3$ that is given by $y=xz+xt+tz=0$. Then
\begin{equation}
\label{equation:4-13}
\begin{split}
H_{\{x\}}\cdot S_\lambda&=L_{\{x\},\{z\}}+L_{\{x\},\{t\}}+2L_{\{x\},\{y,t\}},\\
H_{\{y\}}\cdot S_\lambda&=L_{\{y\},\{t\}}+L_{\{y\},\{x,t\}}+\mathcal{C},\\
H_{\{z\}}\cdot S_\lambda&=L_{\{x\},\{z\}}+2L_{\{z\},\{t\}}+L_{\{z\},\{x,y,t\}},\\
H_{\{t\}}\cdot S_\lambda&=L_{\{x\},\{t\}}+L_{\{y\},\{t\}}+L_{\{z\},\{t\}}+L_{\{t\},\{x,y,z\}}.
\end{split}
\end{equation}
Therefore, we conclude that the base locus of the pencil $\mathcal{S}$ consists of the curves
$L_{\{x\},\{z\}}$, $L_{\{x\},\{t\}}$, $L_{\{y\},\{t\}}$, $L_{\{z\},\{t\}}$, $L_{\{x\},\{y,t\}}$, $L_{\{y\},\{x,t\}}$,
$L_{\{z\},\{x,y,t\}}$, $L_{\{t\},\{x,y,z\}}$, and $\mathcal{C}$.

For every $\lambda\in\mathbb{C}$, the surface $S_\lambda$ is irreducible,
it has isolated singularities, and its singular points contained in the base locus of the pencil $\mathcal{S}$ can be described as follows:
\begin{itemize}\setlength{\itemindent}{3cm}
\item[$P_{\{x\},\{y\},\{t\}}$:] type $\mathbb{A}_2$ with quadratic term $xy$;

\item[$P_{\{x\},\{z\},\{t\}}$:] type $\mathbb{A}_2$ with quadratic term $z(x+t)$;

\item[$P_{\{y\},\{z\},\{t\}}$:] type $\mathbb{A}_1$;

\item[{$P_{\{x\},\{z\},\{y,t\}}$}:] type $\mathbb{A}_2$ with quadratic term
$$
x(x+y+t+3z+\lambda z)
$$
for $\lambda\neq -3$, type $\mathbb{A}_4$ for $\lambda=-3$;

\item[$P_{\{z\},\{t\},\{x,y\}}$:] type $\mathbb{A}_2$ with quadratic term
$$
z(x+y+z-2t-\lambda t)
$$
for $\lambda\neq -3$, type $\mathbb{A}_3$ for $\lambda=-3$;

\item[{$[0:1:-3-\lambda:-1]$}:] type $\mathbb{A}_1$ for $\lambda\neq -3$, type $\mathbb{A}_4$ for $\lambda=-3$.
\end{itemize}
Thus, using Lemma~\ref{corollary:irreducible-fibers}, we conclude that the fiber $\mathsf{f}^{-1}(\lambda)$ is irreducible for every $\lambda\in\mathbb{C}$.
This confirms \eqref{equation:main-1} in Main Theorem, since $h^{1,2}(X)=0$.

Using the description of the singular points of the surface $S_\lambda$ we gave above, we see that
$\mathrm{rk}\,\mathrm{Pic}(\widetilde{S}_{\Bbbk})=\mathrm{rk}\,\mathrm{Pic}(S_{\Bbbk})+10$.
On the other hand, it follows from \eqref{equation:4-13} that
\begin{multline*}
H_\lambda\sim L_{\{x\},\{z\}}+L_{\{x\},\{t\}}+2L_{\{x\},\{y,t\}}\sim L_{\{y\},\{t\}}+L_{\{y\},\{x,t\}}+\mathcal{C}\sim\\
\sim L_{\{x\},\{z\}}+2L_{\{z\},\{t\}}+L_{\{z\},\{x,y,t\}}\sim L_{\{x\},\{t\}}+L_{\{y\},\{t\}}+L_{\{z\},\{t\}}+L_{\{t\},\{x,y,z\}}
\end{multline*}
on the surface $S_\lambda$.
Hence, the intersection matrix of the curves
$L_{\{x\},\{z\}}$, $L_{\{x\},\{t\}}$, $L_{\{y\},\{t\}}$, $L_{\{z\},\{t\}}$, $L_{\{x\},\{y,t\}}$, $L_{\{y\},\{x,t\}}$,
$L_{\{z\},\{x,y,t\}}$, $L_{\{t\},\{x,y,z\}}$, and $\mathcal{C}$ on the surface $S_\lambda$
has the same rank as the intersection matrix of the curves $L_{\{x\},\{y,t\}}$, $L_{\{y\},\{x,t\}}$, $L_{\{z\},\{t\}}$, $L_{\{z\},\{x,y,t\}}$, $L_{\{t\},\{x,y,z\}}$, and $H_{\lambda}$.
Using Propositions~\ref{proposition:du-Val-intersection} and \ref{proposition:du-Val-self-intersection},
we see that the latter matrix is
\begin{center}\renewcommand\arraystretch{1.42}
\begin{tabular}{|c||c|c|c|c|c|c|}
\hline
 $\bullet$  & $L_{\{x\},\{y,t\}}$ & $L_{\{y\},\{x,t\}}$ & $L_{\{z\},\{t\}}$ & $L_{\{z\},\{x,y,t\}}$ & $L_{\{t\},\{x,y,z\}}$ &  $H_{\lambda}$ \\
\hline\hline
$L_{\{x\},\{y,t\}}$ & $-\frac{5}{6}$ & $\frac{1}{3}$ & $1$ & $\frac{1}{3}$ & $0$ & $1$ \\
\hline
$L_{\{y\},\{x,t\}}$ & $\frac{1}{3}$ &  $-\frac{4}{3}$ & $0$ & $1$ & $0$ & $1$ \\
\hline
$L_{\{z\},\{t\}}$ & $1$ &  $0$ & $-\frac{1}{6}$ & $\frac{1}{3}$ & $\frac{1}{3}$ & $1$ \\
\hline
$L_{\{z\},\{x,y,t\}}$ & $\frac{1}{3}$ &  $1$ & $\frac{1}{3}$ & $-\frac{2}{3}$ & $\frac{1}{3}$ & $1$ \\
\hline
$L_{\{t\},\{x,y,z\}}$ & $0$ &  $0$ & $\frac{1}{3}$ & $\frac{1}{3}$ & $-\frac{4}{3}$ & $1$ \\
\hline
 $H_{\lambda}$  & $1$ & $1$ & $1$ & $1$ & $1$ & $4$ \\
\hline
\end{tabular}
\end{center}
Observe that the rank of this matrix is $6$. Thus, we see that \eqref{equation:main-2-simple} holds.
Thus, it follows from Lemma~\ref{lemma:cokernel} that \eqref{equation:main-2} in Main Theorem also holds in this case.

\section{Fano threefolds of Picard rank $5$}
\label{section:rank-5}

\subsection{Family \textnumero $5.1$}
\label{section:r-5-n-1}

In this case, we have $h^{1,2}(X)=0$.
A~toric Landau--Ginzburg model of this family is given by Minkowski polynomial \textnumero $1082$, which is
$$
x+y+\frac{1}{z}+\frac{x}{y}+\frac{y}{x}+\frac{2}{y}+\frac{2}{x}+\frac{z}{y}+\frac{z}{x}+\frac{1}{xy}+\frac{z}{xy}.
$$
The quartic pencil $\mathcal{S}$ is given by
$$
x^{2}zy+y^{2}xz+xyt^2+x^{2}zt+y^{2}zt+2zxt^2+2zyt^2+z^{2}xt+z^{2}yt+zt^3+z^2t^2=\lambda xyzt.
$$

Suppose that $\lambda\ne\infty$. Then
\begin{equation}
\label{equation:5-1}
\begin{split}
H_{\{x\}}\cdot S_\lambda&=L_{\{x\},\{z\}}+L_{\{x\},\{t\}}+L_{\{x\},\{y,t\}}+L_{\{x\},\{y,z,t\}},\\
H_{\{y\}}\cdot S_\lambda&=L_{\{y\},\{z\}}+L_{\{y\},\{t\}}+L_{\{y\},\{x,t\}}+L_{\{y\},\{x,z,t\}},\\
H_{\{z\}}\cdot S_\lambda&=L_{\{x\},\{z\}}+L_{\{y\},\{z\}}+2L_{\{z\},\{t\}},\\
H_{\{t\}}\cdot S_\lambda&=L_{\{x\},\{t\}}+L_{\{y\},\{t\}}+L_{\{z\},\{t\}}+L_{\{t\},\{x,y\}}.
\end{split}
\end{equation}
Thus, the base locus of the pencil $\mathcal{S}$ consists of the lines
$L_{\{x\},\{z\}}$, $L_{\{x\},\{t\}}$, $L_{\{y\},\{z\}}$, $L_{\{y\},\{t\}}$, $L_{\{z\},\{t\}}$,
$L_{\{x\},\{y,t\}}$,  $L_{\{y\},\{x,t\}}$, $L_{\{x\},\{y,z,t\}}$, $L_{\{y\},\{x,z,t\}}$, and $L_{\{t\},\{x,y\}}$.

For every $\lambda\in\mathbb{C}$, the surface $S_\lambda$ has isolated singularities, so that it is irreducible.
The singular points of the surface $S_\lambda$ contained in the base locus of the pencil $\mathcal{S}$ can be described as follows:
\begin{itemize}\setlength{\itemindent}{1.5cm}
\item[$P_{\{y\},\{z\},\{t\}}$:] type $\mathbb{A}_2$ with quadratic term $z(y+t)$;

\item[$P_{\{x\},\{z\},\{t\}}$:] type $\mathbb{A}_2$ with quadratic term $z(x+t)$;

\item[$P_{\{x\},\{y\},\{t\}}$:] type $\mathbb{A}_3$ with quadratic term $t(x+y+t)$ for $\lambda\neq -3$, type $\mathbb{A}_5$ for $\lambda=-3$;

\item[$P_{\{z\},\{t\},\{x,y\}}$:] type $\mathbb{A}_1$.
\end{itemize}
Thus, it follows from Lemma~\ref{corollary:irreducible-fibers} that the fiber $\mathsf{f}^{-1}(\lambda)$ is irreducible for every $\lambda\in\mathbb{C}$.
This confirms \eqref{equation:main-1} in Main Theorem, because $h^{1,2}(X)=0$.

To verify \eqref{equation:main-2} in Main Theorem, observe that
$\mathrm{rk}\,\mathrm{Pic}(\widetilde{S}_{\Bbbk})=\mathrm{rk}\,\mathrm{Pic}(S_{\Bbbk})+8$.
This follows from the description of the singular points of the surface $S_\lambda$ for $\lambda\ne -3$.
On the other hand, it follows from \eqref{equation:5-1} that
\begin{multline*}
L_{\{x\},\{z\}}+L_{\{x\},\{t\}}+L_{\{x\},\{y,t\}}+L_{\{x\},\{y,z,t\}}\sim L_{\{y\},\{z\}}+L_{\{y\},\{t\}}+L_{\{y\},\{x,t\}}+L_{\{y\},\{x,z,t\}}\sim \\
\sim L_{\{x\},\{z\}}+L_{\{y\},\{z\}}+2L_{\{z\},\{t\}}\sim L_{\{x\},\{t\}}+L_{\{y\},\{t\}}+L_{\{z\},\{t\}}+L_{\{t\},\{x,y\}}\sim H_\lambda
\end{multline*}
on the surface $S_\lambda$.
Thus, the intersection matrix of the lines
$L_{\{x\},\{z\}}$, $L_{\{x\},\{t\}}$, $L_{\{y\},\{z\}}$, $L_{\{y\},\{t\}}$, $L_{\{z\},\{t\}}$,
$L_{\{x\},\{y,t\}}$,  $L_{\{y\},\{x,t\}}$, $L_{\{x\},\{y,z,t\}}$, $L_{\{y\},\{x,z,t\}}$, and $L_{\{t\},\{x,y\}}$
on the surface $S_\lambda$ has the same rank as the intersection matrix of the curves
$L_{\{x\},\{y,t\}}$, $L_{\{x\},\{y,z,t\}}$, $L_{\{y\},\{x,t\}}$, $L_{\{y\},\{x,z,t\}}$, $L_{\{z\},\{t\}}$, $L_{\{y\},\{t\}}$, and $H_{\lambda}$.
If $\lambda\ne-3$, then the latter matrix is given by
 \begin{center}\renewcommand\arraystretch{1.42}
\begin{tabular}{|c||c|c|c|c|c|c|c|}
\hline
 $\bullet$  & $L_{\{x\},\{y,t\}}$ & $L_{\{x\},\{y,z,t\}}$ & $L_{\{y\},\{x,t\}}$ & $L_{\{y\},\{x,z,t\}}$ & $L_{\{z\},\{t\}}$ & $L_{\{y\},\{t\}}$ & $H_{\lambda}$ \\
\hline\hline
$L_{\{x\},\{y,t\}}$ & $-\frac{5}{4}$ & $1$ & $\frac{3}{4}$ & $0$ & $0$ & $\frac{1}{4}$ & $1$ \\
\hline
$L_{\{x\},\{y,z,t\}}$ & $1$ & $-2$ & $0$ &  $1$ & $0$ & $0$ & $1$ \\
\hline
$L_{\{y\},\{x,t\}}$ & $\frac{3}{4}$ & $0$ & $-\frac{5}{4}$ &  $1$ & $0$ & $\frac{1}{4}$ & $1$ \\
\hline
$L_{\{y\},\{x,z,t\}}$ & $0$ & $1$ & $1$ &  $-2$ & $0$ & $1$ & $1$ \\
\hline
$L_{\{z\},\{t\}}$ & $0$ & $0$ & $0$ &  $0$ & $-\frac{1}{6}$ & $\frac{1}{3}$ & $1$ \\
\hline
$L_{\{y\},\{t\}}$ & $\frac{1}{4}$ & $0$ & $\frac{1}{4}$ &  $1$ & $\frac{1}{3}$ & $-\frac{7}{12}$ & $1$ \\
\hline
 $H_{\lambda}$  & $1$ & $1$ & $1$ & $1$ & $1$ & $1$ & $4$ \\
\hline
\end{tabular}
\end{center}
Its determinant is $\frac{7}{9}$. This shows that \eqref{equation:main-2-simple} holds.
Thus, we can use Lemma~\ref{lemma:cokernel} to conclude that \eqref{equation:main-2} in Main Theorem also holds in this case.

\subsection{Family \textnumero $5.2$}
\label{section:r-5-n-2}

In this case, we have $h^{1,2}(X)=0$.
A~toric Landau--Ginzburg model of this family is given by Minkowski polynomial \textnumero $219$, which is
$$
x+y+z+\frac{x}{z}+\frac{x}{y}+\frac{y}{x}+\frac{1}{y}+\frac{1}{x}.
$$
Thus, the quartic pencil $\mathcal{S}$ is given by the equation
$$
x^2zy+y^2xz+z^2xy+x^2ty+x^2tz+y^2tz+t^2xz+t^2zy=\lambda xyzt.
$$

Suppose that $\lambda\ne\infty$. Then
\begin{equation}
\label{equation:5-2}
\begin{split}
H_{\{x\}}\cdot S_\lambda&=L_{\{x\},\{y\}}+L_{\{x\},\{z\}}+L_{\{x\},\{t\}}+L_{\{x\},\{y,t\}},\\
H_{\{y\}}\cdot S_\lambda&=L_{\{x\},\{y\}}+L_{\{y\},\{z\}}+L_{\{y\},\{t\}}+L_{\{y\},\{x,t\}},\\
H_{\{z\}}\cdot S_\lambda&=2L_{\{x\},\{z\}}+L_{\{y\},\{z\}}+L_{\{z\},\{t\}},\\
H_{\{t\}}\cdot S_\lambda&=L_{\{x\},\{t\}}+L_{\{y\},\{t\}}+L_{\{z\},\{t\}}+L_{\{t\},\{x,y,z\}}.
\end{split}
\end{equation}

For every $\lambda$, the surface $S_\lambda$ is irreducible, it has isolated singularities,
and its singular points contained in the base locus of the pencil $\mathcal{S}$ can be described as follows:
\begin{itemize}\setlength{\itemindent}{3cm}
\item[$P_{\{x\},\{y\},\{z\}}$:] type $\mathbb{A}_2$ with quadratic term $z(x+y)$;

\item[$P_{\{x\},\{y\},\{t\}}$:] type $\mathbb{A}_3$ with quadratic term $xy$;

\item[$P_{\{x\},\{z\},\{t\}}$:] type $\mathbb{A}_2$ with quadratic term $z(x+t)$;

\item[$P_{\{y\},\{z\},\{t\}}$:] type $\mathbb{A}_1$;

\item[$P_{\{x\},\{t\},\{y,z\}}$:] type $\mathbb{A}_1$.
\end{itemize}
Therefore, using Lemma~\ref{corollary:irreducible-fibers}, we see that the fiber $\mathsf{f}^{-1}(\lambda)$ is irreducible for every $\lambda\in\mathbb{C}$.
This confirms \eqref{equation:main-1} in Main Theorem, since $h^{1,2}(X)=0$.

Using \eqref{equation:5-2}, we see that the intersection matrix of the lines
$L_{\{x\},\{y\}}$, $L_{\{x\},\{z\}}$, $L_{\{x\},\{t\}}$, $L_{\{y\},\{z\}}$, $L_{\{y\},\{t\}}$,
$L_{\{z\},\{t\}}$, $L_{\{x\},\{y,t\}}$, $L_{\{y\},\{x,t\}}$, and $L_{\{t\},\{x,y,z\}}$ on the surface $S_\lambda$
has the same rank as the intersection matrix of the curves $L_{\{x\},\{y\}}$, $L_{\{x\},\{z\}}$, $L_{\{y\},\{z\}}$, $L_{\{y\},\{t\}}$,  $L_{\{x\},\{t\}}$, and $H_{\lambda}$.
The latter matrix is given by
 \begin{center}\renewcommand\arraystretch{1.42}
\begin{tabular}{|c||c|c|c|c|c|c|}
\hline
 $\bullet$  & $L_{\{x\},\{y\}}$ & $L_{\{x\},\{z\}}$ & $L_{\{y\},\{z\}}$ & $L_{\{y\},\{t\}}$ & $L_{\{x\},\{t\}}$ & $H_{\lambda}$ \\
\hline\hline
$L_{\{x\},\{y\}}$ & $-\frac{1}{3}$ & $\frac{2}{3}$ & $\frac{2}{3}$ & $\frac{1}{2}$ & $\frac{1}{2}$ & $1$ \\
\hline
$L_{\{x\},\{z\}}$ & $\frac{2}{3}$ & $-\frac{1}{6}$ & $\frac{2}{3}$ & $0$ & $\frac{2}{3}$ & $1$ \\
\hline
$L_{\{y\},\{z\}}$ & $\frac{2}{3}$ & $\frac{2}{3}$ & $-\frac{3}{2}$ & $\frac{1}{2}$ & $0$ & $1$ \\
\hline
$L_{\{y\},\{t\}}$ & $\frac{1}{2}$ & $0$ & $\frac{1}{2}$ & $-\frac{3}{4}$ & $\frac{1}{4}$ & $1$ \\
\hline
$L_{\{x\},\{t\}}$ & $\frac{1}{2}$ & $\frac{2}{3}$ & $0$ & $\frac{1}{4}$ & $-\frac{1}{12}$ & $1$ \\
\hline
$H_{\lambda}$ & $1$ & $1$ & $1$ & $1$ & $1$ & $4$ \\
\hline
\end{tabular}
\end{center}
Note that this matrix has rank~$6$.
Moreover, using the description of the singular points of the surface $S_\lambda$,
we see that $\mathrm{rk}\,\mathrm{Pic}(\widetilde{S}_{\Bbbk})=\mathrm{rk}\,\mathrm{Pic}(S_{\Bbbk})+9$.
This shows that \eqref{equation:main-2-simple} holds in this case.
Then \eqref{equation:main-2} in Main Theorem holds by Lemma~\ref{lemma:cokernel}.

\subsection{Family \textnumero $5.3$}
\label{section:r-5-n-3}

In this case, we have $X\cong\mathbb{P}^1\times\mathbf{S}_6$, where $\mathbf{S}_6$ is a smooth del Pezzo surface of degree $6$.
This implies that $h^{1,2}(X)=0$.
A~toric Landau--Ginzburg model of this family is given by Minkowski polynomial \textnumero $218$, which is
$$
x+y+z+\frac{y}{z}+\frac{z}{y}+\frac{1}{z}+\frac{1}{y}+\frac{1}{x}.
$$
Therefore, the corresponding pencil $\mathcal{S}$ is given by the following equation:
$$
x^2zy+y^2xz+z^2xy+y^2xt+z^2xt+t^2xy+t^2xz+t^2zy=\lambda xyzt.
$$

Suppose that $\lambda\ne\infty$. Then
\begin{equation}
\label{equation:5-3}
\begin{split}
H_{\{x\}}\cdot S_\lambda=L_{\{x\},\{y\}}+L_{\{x\},\{z\}}+2L_{\{x\},\{t\}},\\
H_{\{y\}}\cdot S_\lambda=L_{\{x\},\{y\}}+L_{\{y\},\{z\}}+L_{\{y\},\{t\}}+L_{\{y\},\{z,t\}},\\
H_{\{z\}}\cdot S_\lambda=L_{\{x\},\{z\}}+L_{\{y\},\{z\}}+L_{\{z\},\{t\}}+L_{\{z\},\{y,t\}},\\
H_{\{t\}}\cdot S_\lambda=L_{\{x\},\{t\}}+L_{\{y\},\{t\}}+L_{\{z\},\{t\}}+L_{\{t\},\{x,y,z\}}.
\end{split}
\end{equation}
Thus, the base locus of the pencil $\mathcal{S}$ consists of the lines
$L_{\{x\},\{y\}}$, $L_{\{x\},\{z\}}$, $L_{\{x\},\{t\}}$,
$L_{\{y\},\{z\}}$, $L_{\{y\},\{t\}}$, $L_{\{z\},\{t\}}$,
$L_{\{y\},\{z,t\}}$, $L_{\{z\},\{y,t\}}$, and $L_{\{t\},\{x,y,z\}}$.

For every $\lambda$, the surface $S_\lambda$ is irreducible, it has isolated singularities,
and its singular points contained in the base locus of the pencil $\mathcal{S}$ can be described as follows:
\begin{itemize}\setlength{\itemindent}{3cm}
\item[$P_{\{x\},\{y\},\{z\}}$:] type $\mathbb{A}_1$;

\item[$P_{\{x\},\{z\},\{t\}}$:] type $\mathbb{A}_2$ with quadratic term $x(z+t)$;

\item[$P_{\{x\},\{y\},\{t\}}$:] type $\mathbb{A}_2$ with quadratic term $x(y+t)$;

\item[$P_{\{y\},\{z\},\{t\}}$:] type $\mathbb{A}_3$ with quadratic term $yz$;

\item[$P_{\{x\},\{t\},\{y,z\}}$:] type $\mathbb{A}_1$.
\end{itemize}
Thus, by Lemma~\ref{corollary:irreducible-fibers}, the fiber $\mathsf{f}^{-1}(\lambda)$ is irreducible for every $\lambda\in\mathbb{C}$.
This confirms \eqref{equation:main-1} in Main Theorem, since $h^{1,2}(X)=0$.

Now let us verify \eqref{equation:main-2} in Main Theorem. On the surface $S_\lambda$, we have
\begin{multline*}
H_\lambda\sim L_{\{x\},\{y\}}+L_{\{x\},\{z\}}+2L_{\{x\},\{t\}}\sim L_{\{x\},\{y\}}+L_{\{y\},\{z\}}+L_{\{y\},\{t\}}+L_{\{y\},\{z,t\}}\sim\\
L_{\{x\},\{z\}}+L_{\{y\},\{z\}}+L_{\{z\},\{t\}}+L_{\{z\},\{y,t\}}\sim L_{\{x\},\{t\}}+L_{\{y\},\{t\}}+L_{\{z\},\{t\}}+L_{\{t\},\{x,y,z\}}.
\end{multline*}
This follows from  \eqref{equation:5-3}.
Thus, the intersection matrix of the lines
$L_{\{x\},\{y\}}$, $L_{\{x\},\{z\}}$, $L_{\{x\},\{t\}}$, $L_{\{y\},\{z\}}$, $L_{\{y\},\{t\}}$, $L_{\{z\},\{t\}}$, $L_{\{y\},\{z,t\}}$, $L_{\{z\},\{y,t\}}$, and $L_{\{t\},\{x,y,z\}}$
has the same rank as the intersection matrix of the curves $L_{\{x\},\{y\}}$, $L_{\{x\},\{z\}}$, $L_{\{y\},\{z\}}$, $L_{\{y\},\{t\}}$,  $L_{\{z\},\{t\}}$, and $H_{\lambda}$.
The latter matrix is given by
\begin{center}\renewcommand\arraystretch{1.42}
\begin{tabular}{|c||c|c|c|c|c|c|}
\hline
 $\bullet$  & $L_{\{x\},\{y\}}$ & $L_{\{x\},\{z\}}$ & $L_{\{y\},\{z\}}$ & $L_{\{y\},\{t\}}$ & $L_{\{z\},\{t\}}$ & $H_{\lambda}$ \\
\hline\hline
$L_{\{x\},\{y\}}$ & $-\frac{2}{3}$ & $\frac{1}{2}$ & $\frac{1}{2}$ & $\frac{1}{3}$ & $0$ & $1$ \\
\hline
$L_{\{x\},\{z\}}$ & $\frac{1}{2}$ & $-\frac{2}{3}$ & $\frac{1}{2}$ & $0$ & $\frac{1}{3}$ & $1$ \\
\hline
$L_{\{y\},\{z\}}$ & $\frac{1}{2}$ & $\frac{1}{2}$ & $\frac{1}{2}$ & $\frac{1}{2}$ & $\frac{1}{2}$ & $1$ \\
\hline
$L_{\{y\},\{t\}}$ & $\frac{1}{3}$ & $0$ & $\frac{1}{2}$ & $-\frac{7}{12}$ & $\frac{1}{4}$ & $1$ \\
\hline
$L_{\{z\},\{t\}}$ & $0$ & $\frac{1}{3}$ & $\frac{1}{2}$ & $\frac{1}{4}$ & $-\frac{7}{12}$ & $1$ \\
\hline
$H_{\lambda}$ & $1$ & $1$ & $1$ & $1$ & $1$ & $4$ \\
\hline
\end{tabular}
\end{center}
Its rank is $6$.
On the other hand, it follows from the description of the singular points of the surface $S_\lambda$ that
$$
\mathrm{rk}\,\mathrm{Pic}(\widetilde{S}_{\Bbbk})=\mathrm{rk}\,\mathrm{Pic}(S_{\Bbbk})+9,
$$
so that \eqref{equation:main-2-simple} holds in this case.
Thus, by Lemma~\ref{lemma:cokernel}, we see that \eqref{equation:main-2} in Main Theorem also holds in this case.

\section{Fano threefolds of Picard rank $6$}
\label{section:rank-6}

\subsection{Family \textnumero $6.1$}
\label{section:r-6-n-1}

We have $X\cong\mathbb{P}^1\times\mathbf{S}_5$, where $\mathbf{S}_5$ is the smooth del Pezzo surface of degree $5$.
In particular, we have $h^{1,2}(X)=0$.
A~toric Landau--Ginzburg model of this family is given by Minkowski polynomial \textnumero $283$, which is
$$
x+\frac{1}{y}+z+\frac{1}{xy}+\frac{1}{z}+2y+\frac{3}{x}+\frac{3y}{x}+\frac{y^2}{x}.
$$
Thus, the corresponding pencil $\mathcal{S}$ is given by the equation
$$
x^2yz+xzt^2+xyz^2+zt^3+xyt^2+2xy^2z+3yzt^2+3y^2zt+y^3z=\lambda xyzt.
$$
For simplicity, we suppose that $\lambda\ne\infty$.

Let $\mathcal{C}$ be the conic in $\mathbb{P}^3$ that is given by $t=(x+y)^2+xz=0$. Then
\begin{equation}
\label{equation:6-1}
\begin{split}
H_{\{x\}}\cdot S_\lambda&=L_{\{x\},\{z\}}+3L_{\{x\},\{y,t\}},\\
H_{\{y\}}\cdot S_\lambda&=L_{\{y\},\{z\}}+2L_{\{y\},\{t\}}+L_{\{y\},\{x,t\}},\\
H_{\{z\}}\cdot S_\lambda&=L_{\{x\},\{z\}}+L_{\{y\},\{z\}}+2L_{\{z\},\{t\}},\\
H_{\{t\}}\cdot S_\lambda&=L_{\{y\},\{t\}}+L_{\{z\},\{t\}}+\mathcal{C}.
\end{split}
\end{equation}
Thus, the base locus of the pencil $\mathcal{S}$ consists of the curves
$L_{\{x\},\{z\}}$, $L_{\{x\},\{y,t\}}$, $L_{\{y\},\{z\}}$, $L_{\{y\},\{t\}}$, $L_{\{y\},\{x,t\}}$, $L_{\{z\},\{t\}}$, and $\mathcal{C}$.

For every $\lambda\in\mathbb{C}$, the surface $S_\lambda$ is irreducible and has isolated singularities.
Moreover, if $\lambda\ne-1$ and $\lambda\ne-5$, then
the singular points of the surface $S_\lambda$ contained in the base locus of the pencil $\mathcal{S}$ can be described as follows:
\begin{itemize}\setlength{\itemindent}{6cm}
\item[$P_{\{x\},\{y\},\{t\}}$:] type $\mathbb{A}_2$ with quadratic term $xy$;

\item[$P_{\{y\},\{z\},\{t\}}$:] type $\mathbb{A}_3$ with quadratic term $yz$;

\item[$P_{\{y\},\{z\},\{x,t\}}$:] type $\mathbb{A}_1$;

\item[$P_{\{z\},\{t\},\{x,y\}}$:] type $\mathbb{A}_3$ with quadratic term $z^2+t^2+(\lambda+3)tz$;

\item[{$[0:-2:\lambda+3\pm\sqrt{\lambda^2+6\lambda+5}:2]$}:] type $\mathbb{A}_2$.
\end{itemize}
Thus, in the notation of Subsection~\ref{subsection:scheme-step-6}, the set $\Sigma$
consists of the points $P_{\{y\},\{z\},\{t\}}$, $P_{\{x\},\{y\},\{t\}}$, $P_{\{y\},\{z\},\{x,t\}}$, and $P_{\{z\},\{t\},\{x,y\}}$.

The description of the singular points of the surface $S_\lambda$ also gives
\begin{equation}
\label{equation:6-1-Pic}
\mathrm{rk}\,\mathrm{Pic}(\widetilde{S}_{\Bbbk})=\mathrm{rk}\,\mathrm{Pic}(S_{\Bbbk})+10.
\end{equation}
Observe that the singular point $P_{\{z\},\{t\},\{x,y\}}$ contributes $\textbf{\textcircled{2}}$ to this formula.
Similarly, the singular points $[0:-2:\lambda+3\pm\sqrt{\lambda^2+6\lambda+5}:2]$ also contribute $\textbf{\textcircled{2}}$ to \eqref{equation:6-1-Pic}.

To verify \eqref{equation:main-1} in Main Theorem, observe that
the surface $S_{\lambda}$ has du Val singularities in base points of the pencil~$\mathcal{S}$ provided that $\lambda\ne-1$ and $\lambda\ne-5$.
Thus, by Lemma~\ref{corollary:irreducible-fibers}, the fiber $\mathsf{f}^{-1}(\lambda)$ is irreducible for every $\lambda\in\mathbb{C}$ such that $\lambda\ne-1$ and $\lambda\ne-5$.
Moreover we have

\begin{lemma}
\label{lemma:6-1-defect}
One has $[\mathsf{f}^{-1}(-1)]=[\mathsf{f}^{-1}(-5)]=1$.
\end{lemma}

\begin{proof}
It is enough to prove that $[\mathsf{f}^{-1}(-1)]=1$, since the proof is identical in the remaining case.
Observe that the points $P_{\{y\},\{z\},\{t\}}$, $P_{\{x\},\{y\},\{t\}}$, $P_{\{y\},\{z\},\{x,t\}}$.
are good double points of the surface $S_{-1}$.
Thus, it follows from
\eqref{equation:equation:number-of-irredubicle-components-refined} and
Lemmas~\ref{lemma:main} and \ref{lemma:normal-crossing} that
$$
\big[\mathsf{f}^{-1}(-1)\big]=\big[S_{-1}\big]+\mathbf{D}_{P_{\{z\},\{t\},\{x,y\}}}^{-1}=1+\mathbf{D}_{P_{\{z\},\{t\},\{x,y\}}}^{-1}.
$$

In the neighborhood of the point $P_{\{z\},\{t\},\{x,y\}}$ the morphism $\alpha$ in \eqref{equation:main-diagram} is just a blow up of this point.
Moreover, its exceptional surface that is mapped to $P_{\{z\},\{t\},\{x,y\}}$
does not contain base curves of the pencil $\widehat{\mathcal{S}}$,
because the quadratic term of the surface $S_\lambda$ at this point is $z^2+t^2+(\lambda+3)tz$.
Furthermore, the point $P_{\{z\},\{t\},\{x,y\}}$ is a double point of the surface $S_{-1}$.
In fact, the surface $S_{-1}$ has singularity of type $\mathbb{D}_4$ at  $P_{\{z\},\{t\},\{x,y\}}$.
We~see that $A_{P_{\{z\},\{t\},\{x,y\}}}^{-1}=0$.
Then $\mathbf{D}_{P_{\{z\},\{t\},\{x,y\}}}^{-1}=0$ by \eqref{equation:D-A-B},
so that $[\mathsf{f}^{-1}(-1)]=1$.
\end{proof}

Thus, we conclude that $\mathsf{f}^{-1}(\lambda)$ is irreducible for every $\lambda\in\mathbb{C}$.
This confirms \eqref{equation:main-1} in Main Theorem, since $h^{1,2}(X)=0$.
To verify \eqref{equation:main-2} in Main Theorem, observe that
\begin{multline*}
H_\lambda\sim L_{\{x\},\{z\}}+3L_{\{x\},\{y,t\}}\sim L_{\{y\},\{z\}}+2L_{\{y\},\{t\}}+L_{\{y\},\{x,t\}}\sim \\
\sim L_{\{x\},\{z\}}+L_{\{y\},\{z\}}+2L_{\{z\},\{t\}}\sim L_{\{y\},\{t\}}+L_{\{z\},\{t\}}+\mathcal{C}
\end{multline*}
on the surface $S_\lambda$. This follows from \eqref{equation:6-1}.
Thus, the intersection matrix of the curves
$L_{\{x\},\{z\}}$, $L_{\{x\},\{y,t\}}$, $L_{\{y\},\{z\}}$, $L_{\{y\},\{t\}}$, $L_{\{y\},\{x,t\}}$, $L_{\{z\},\{t\}}$, and $\mathcal{C}$
on the surface $S_\lambda$ has the same rank as the intersection matrix of the curves
$L_{\{x\},\{z\}}$, $L_{\{y\},\{z\}}$, $L_{\{y\},\{t\}}$, and $H_{\lambda}$.
If $\lambda\ne-1$ and $\lambda\ne-5$, then the latter matrix is given by
\begin{center}\renewcommand\arraystretch{1.42}
\begin{tabular}{|c||c|c|c|c|}
\hline
 $\bullet$  & $L_{\{x\},\{z\}}$ & $L_{\{y\},\{z\}}$ & $L_{\{y\},\{t\}}$ &  $H_{\lambda}$ \\
\hline\hline
$L_{\{x\},\{z\}}$ & $-2$ & $1$ & $0$ & $1$ \\
\hline
$L_{\{y\},\{z\}}$ & $1$ &  $-\frac{1}{2}$ & $\frac{1}{2}$ & $1$ \\
\hline
$L_{\{y\},\{t\}}$ & $0$ &  $\frac{1}{2}$ & $-\frac{7}{12}$ & $1$ \\
\hline
 $H_{\lambda}$  & $1$ & $1$ & $1$ & $4$ \\
\hline
\end{tabular}
\end{center}
The rank of this matrix is $4$.
Thus, using \eqref{equation:6-1-Pic}, we see  that \eqref{equation:main-2-simple} holds.
Then \eqref{equation:main-2} in Main Theorem holds by Lemma~\ref{lemma:cokernel}.

\section{Fano threefolds of Picard rank $7$}
\label{section:rank-7}

\subsection{Family \textnumero $7.1$}
\label{section:r-7-n-1}

In this case, we have $X\cong\mathbb{P}^1\times\mathbf{S}_4$, where $\mathbf{S}_4$ is a smooth del Pezzo surface of degree $4$.
This implies that $h^{1,2}(X)=0$.
A~toric Landau--Ginzburg model of this family is given by Minkowski polynomial \textnumero $505$, which is
$$
\frac{1}{x}+\frac{1}{y}+z+\frac{2y}{x}+\frac{2x}{y}+\frac{1}{z}+\frac{y^2}{x}+3y+3x+\frac{x^2}{y}.
$$
Hence, the corresponding pencil $\mathcal{S}$ is given by the following equation:
$$
t^2zy+t^2xz+z^2xy+2y^2tz+2x^2tz+t^2xy+y^3z+3y^2xz+3x^2zy+x^3z=\lambda xyzt.
$$
For simplicity, we suppose that $\lambda\ne\infty$.

Let $\mathcal{C}$ be the cubic curve in $\mathbb{P}^3$ that is given by $t=xyz+(x+y)^3=0$.
This curve is singular at the point $P_{\{x\},\{y\},\{t\}}$.
Moreover, we have
\begin{equation}
\label{equation:7-1}
\begin{split}
H_{\{x\}}\cdot S_\lambda&=L_{\{x\},\{y\}}+L_{\{x\},\{z\}}+2L_{\{x\},\{y,t\}},\\
H_{\{y\}}\cdot S_\lambda&=L_{\{x\},\{y\}}+L_{\{y\},\{z\}}+2L_{\{y\},\{x,t\}},\\
H_{\{z\}}\cdot S_\lambda&=L_{\{x\},\{z\}}+L_{\{y\},\{z\}}+2L_{\{z\},\{t\}},\\
H_{\{t\}}\cdot S_\lambda&=L_{\{z\},\{t\}}+\mathcal{C}.
\end{split}
\end{equation}
Thus, the base locus of the pencil $\mathcal{S}$ consists of the curves
$L_{\{x\},\{y\}}$, $L_{\{x\},\{z\}}$, $L_{\{y\},\{z\}}$, $L_{\{z\},\{t\}}$, $L_{\{x\},\{y,t\}}$, $L_{\{y\},\{x,t\}}$, and $\mathcal{C}$.

If $\lambda\ne-2$ and $\lambda\ne -6$, then $S_\lambda$ is irreducible and has isolated singularities.
In this case, the singular points of the surface $S_\lambda$ contained in the base locus of the pencil $\mathcal{S}$ can be described as follows:
\begin{itemize}\setlength{\itemindent}{6cm}
\item[$P_{\{x\},\{y\},\{z\}}$:] type $\mathbb{A}_1$;

\item[$P_{\{x\},\{y\},\{t\}}$:] type $\mathbb{A}_3$ with quadratic term $xy$;

\item[$P_{\{z\},\{t\},\{x,y\}}$:] type $\mathbb{A}_5$ with quadratic term $z^2+t^2-(\lambda+4)zt$;

\item[{$[0:-2:\lambda+4\pm\sqrt{\lambda^2+8\lambda+12}:2]$}:] type $\mathbb{A}_1$;

\item[{$[-2:0:\lambda+4\pm\sqrt{\lambda^2+8\lambda+12}:2]$}:] type $\mathbb{A}_1$.
\end{itemize}
Thus, the set $\Sigma$
consists of the points $P_{\{y\},\{z\},\{t\}}$, $P_{\{x\},\{y\},\{t\}}$, $P_{\{y\},\{z\},\{x,t\}}$, and $P_{\{z\},\{t\},\{x,y\}}$.

The description of the singular points of the surface $S_\lambda$ also gives
\begin{equation}
\label{equation:7-1-Pic}
\mathrm{rk}\,\mathrm{Pic}(\widetilde{S}_{\Bbbk})=\mathrm{rk}\,\mathrm{Pic}(S_{\Bbbk})+9.
\end{equation}
Note that the singular point $P_{\{z\},\{t\},\{x,y\}}$ contributes $\textbf{\textcircled{3}}$ to this formula.
Similarly, the singular points $[0:-2:\lambda+4\pm\sqrt{\lambda^2+8\lambda+12}:2]$ contribute $\textbf{\textcircled{1}}$ to this formula.
Likewise, the singular points $[-2:0:\lambda+4\pm\sqrt{\lambda^2+8\lambda+12}:2]$ also contribute $\textbf{\textcircled{1}}$ to \eqref{equation:7-1-Pic}.

To verify \eqref{equation:main-1} in Main Theorem, observe that
the surface $S_{\lambda}$ has du Val singularities at base points of the pencil~$\mathcal{S}$ provided that $\lambda\ne-2$ and $\lambda\ne-6$.
Thus, by Lemma~\ref{corollary:irreducible-fibers}, the fiber $\mathsf{f}^{-1}(\lambda)$ is irreducible for every $\lambda\in\mathbb{C}$ such that $\lambda\ne-2$ and $\lambda\ne-6$.
On the other hand, we have

\begin{lemma}
\label{lemma:7-1-defect}
One has $[\mathsf{f}^{-1}(-2)]=[\mathsf{f}^{-1}(-6)]=1$.
\end{lemma}

\begin{proof}
Observe that both surfaces $S_{-2}$ and $S_{-6}$ have non-isolated singularities.
Namely, the surface $S_{-2}$ is singular along the line $x+y+z=x+y+t=0$,
and $S_{-6}$ is singular along the line $x+y-z=x+y+t=0$.
However, both these surfaces are irreducible. This can be checked by analyzing their hyperplane sections.

It is enough to prove that $[\mathsf{f}^{-1}(-2)]=1$, since the proof is identical in the remaining case.
Observe that the points $P_{\{y\},\{z\},\{t\}}$, $P_{\{x\},\{y\},\{t\}}$, and $P_{\{y\},\{z\},\{x,t\}}$
are good double points of the surface $S_{-2}$.
Thus, it follows from
\eqref{equation:equation:number-of-irredubicle-components-refined},
Lemma~\ref{lemma:main} and Lemma~\ref{lemma:normal-crossing} that
$$
\big[\mathsf{f}^{-1}(-2)\big]=1+\mathbf{D}_{P_{\{z\},\{t\},\{x,y\}}}^{-2}.
$$
Arguing as in the proof of Lemma~\ref{lemma:6-1-defect}, we get
$\mathbf{D}_{P_{\{z\},\{t\},\{x,y\}}}^{-2}=0$, so that $[\mathsf{f}^{-1}(-2)]=1$.
\end{proof}

We see that $\mathsf{f}^{-1}(\lambda)$ is irreducible for every $\lambda\in\mathbb{C}$.
This confirms \eqref{equation:main-1} in Main Theorem.

Let us verify \eqref{equation:main-2} in Main Theorem.
It follows from \eqref{equation:7-1} that the intersection matrix of the curves
$L_{\{x\},\{y\}}$, $L_{\{x\},\{z\}}$, $L_{\{y\},\{z\}}$, $L_{\{z\},\{t\}}$, $L_{\{x\},\{y,t\}}$, $L_{\{y\},\{x,t\}}$, $\mathcal{C}$
on the surface~$S_\lambda$ has the same rank as the intersection matrix of the curves
$L_{\{x\},\{z\}}$, $L_{\{y\},\{z\}}$, $L_{\{y\},\{x,t\}}$, $H_{\lambda}$.
If~$\lambda\ne-2$ and $\lambda\ne-6$, then the latter matrix is given by
\begin{center}\renewcommand\arraystretch{1.42}
\begin{tabular}{|c||c|c|c|c|}
\hline
 $\bullet$  & $L_{\{x\},\{z\}}$ & $L_{\{y\},\{z\}}$ & $L_{\{y\},\{x,t\}}$ &  $H_{\lambda}$ \\
\hline\hline
$L_{\{x\},\{z\}}$ & $-\frac{3}{2}$ & $\frac{1}{2}$ & $0$ & $1$ \\
\hline
$L_{\{y\},\{z\}}$ & $\frac{1}{2}$ &  $-2$ & $1$ & $1$ \\
\hline
$L_{\{y\},\{x,t\}}$ & $0$ &  $1$ & $-\frac{5}{4}$ & $1$ \\
\hline
 $H_{\lambda}$  & $1$ & $1$ & $1$ & $4$ \\
\hline
\end{tabular}
\end{center}
Its rank is $4$, so that \eqref{equation:main-2} in Main Theorem holds by \eqref{equation:7-1-Pic} and Lemma~\ref{lemma:cokernel}.

\section{Fano threefolds of Picard rank $8$}
\label{section:rank-8}

\subsection{Family \textnumero $8.1$}
\label{section:r-8-n-1}
We discussed this case in Example~\ref{example:r-8-n1}, where we also described the pencil $\mathcal{S}$.
Let us use the notation of this example and we assume that $\lambda\ne\infty$. Then
\begin{equation}
\label{equation:8-1}
\begin{split}
H_{\{x\}}\cdot S_\lambda&=L_{\{x\},\{y\}}+L_{\{x\},\{z\}}+2L_{\{x\},\{t\}},\\
H_{\{y\}}\cdot S_\lambda&=L_{\{x\},\{y\}}+3L_{\{y\},\{t,z\}},\\
H_{\{z\}}\cdot S_\lambda&=L_{\{x\},\{z\}}+3L_{\{z\},\{t,y\}},\\
H_{\{t\}}\cdot S_\lambda&=L_{\{x\},\{t\}}+\mathcal{C}.
\end{split}
\end{equation}

If $\lambda\ne -4$ and $\lambda\ne -8$, then the surface $S_\lambda$ is irreducible and has isolated singularities.
In fact, in this case, we can say more:

\begin{lemma}
\label{lemma:r8-n1-singular-points}
Suppose that $\lambda\ne-4$ and $\lambda\ne-8$.
Then the singular points of the surface $S_\lambda$ contained in the base locus can be described as follows:
\begin{itemize}\setlength{\itemindent}{7cm}
\item[$P_{\{y\},\{z\},\{t\}}$:] type $\mathbb{A}_2$;

\item[$P_{\{x\},\{t\},\{y,z\}}$:] type $\mathbb{A}_5$;

\item[{$[\lambda+3\pm\sqrt{\lambda^2+12\lambda+32}:0:-2:2]$}:] type $\mathbb{A}_2$;

\item[{$[\lambda+3\pm\sqrt{\lambda^2+12\lambda+32}:-2:0:2]$}:] type $\mathbb{A}_2$.
\end{itemize}
\end{lemma}

\begin{proof}
Taking partial derivatives, we see that the singular points of the surface $S_\lambda$ contained in the base locus of the pencil $\mathcal{S}$
are those described in the assertion of the lemma.
To describe their types, we start with  $P_{\{y\},\{z\},\{t\}}$.
In the chart $x=1$, the surface $S_\lambda$ is given by
$$
yz+t^3+\text{higher order terms}=0,
$$
where we order monomials with respect to the weights $\mathrm{wt}(y)=3$, $\mathrm{wt}(z)=3$, $\mathrm{wt}(t)=2$.
This implies that $P_{\{y\},\{z\},\{t\}}$ is a singular point of type $\mathbb{A}_2$.

To describe the type of the singular point $P_{\{x\},\{t\},\{y,z\}}$,
we consider the chart $y=1$ and change coordinates as follows: $\bar{x}=x$, $\bar{z}=z+1$, and $\bar{t}=t$.
Then $S_\lambda$ is given by
$$
-\bar{x}^2+(\lambda+6)\bar{x}\bar{t}-\bar{t}^2+\text{higher order terms}=0.
$$
Now that
$$
-\bar{x}^2+(\lambda+6)\bar{x}\bar{t}-\bar{t}^2=-\Big(\bar{x}-\big(\lambda+3+\sqrt{\lambda^2+12\lambda+32}\big)\bar{t}\Big)\Big(\bar{x}-\big(\lambda+3-\sqrt{\lambda^2+12\lambda+32}\big)\bar{t}\Big),
$$
and this quadratic form has rank $2$, because $\lambda\ne -4$ and $\lambda\ne -8$.
Introducing new coordinates $\hat{z}=\bar{z}$, $\hat{y}=\frac{\bar{y}}{\bar{z}}$, $\hat{t}=\frac{\bar{t}}{\bar{z}}$,
we obtain the equation of the blow up of the surface $S_\lambda$ at $P_{\{x\},\{t\},\{y,z\}}$.
It is
$$
\hat{x}^2-(\lambda+6)\hat{t}\hat{x}+\hat{t}^2=\hat{x}^2\hat{z}-(\lambda+6)\hat{t}\hat{x}\hat{z}+\hat{z}^2\hat{x}+\hat{t}^2\hat{z}+3\hat{t}\hat{x}\hat{z}^2+\hat{t}^3\hat{x}\hat{z}^2+3\hat{t}^2\hat{x}\hat{z}^2.
$$
The two exceptional curves of the blow up are given by  $\hat{z}=\hat{t}=0$ and $\hat{z}=\hat{y}=0$.
They intersect at the point $(0,0,0)$, which is singular point of the obtained surface.
To blow up the latter point, we introduce new coordinates $\tilde{z}=\hat{z}$, $\tilde{y}=\frac{\hat{y}}{\hat{z}}$, $\tilde{t}=\frac{\hat{t}}{\hat{z}}$.
After dividing by~$\hat{z}^2$, we rewrite the latter equation as
$$
\tilde{x}^2-\tilde{z}\tilde{x}-(\lambda+6)\tilde{t}\tilde{x}=\tilde{x}^2\tilde{z}-\tilde{t}^2+\tilde{t}^2\tilde{z}-(\lambda+6)\tilde{t}\tilde{x}\tilde{z}+3\tilde{t}\tilde{x}\tilde{z}^2+3\tilde{t}^2\tilde{x}\tilde{z}^3+\tilde{t}^3\tilde{x}\tilde{z}^4.
$$
The quadratic form of this equation has rank $3$, so that this surface as an ordinary double point at $(0,0,0)$.
This implies that $P_{\{x\},\{t\},\{y,z\}}$ is a singular point of type $\mathbb{A}_5$.

Now we describe the type of the {floating} singular points.
We will only consider the singular point $[\lambda+3+\sqrt{\lambda^2+12\lambda+32}:0:-2:2]$,
because computations in the remaining cases are similar.
Let us introduce an auxiliary parameter $\mu\in\mathbb{C}$ such that  $\lambda=-\frac{4\mu^2-4\mu-1}{\mu(\mu-1)}$.
We assume that $\mu\ne 0$ and $\mu\ne 1$. Then
$$
[\lambda+3+\sqrt{\lambda^2+12\lambda+32}:0:-2:2]=[\mu-1:0:-\mu:\mu].
$$
Taking the chart $t=1$ and introducing new coordinates $\bar{x}=x+\frac{\mu-1}{\mu}$, $\bar{y}=y$, and $\bar{z}=z-1$,
we see that $S_\lambda$ is given by
$$
(2\mu-1)\bar{x}\bar{y}+(\mu-1)^2\bar{z}^3+\text{higher order terms}=0.
$$
Here, as above, we order monomials with respect to the weights $\mathrm{wt}(\bar{x})=3$, $\mathrm{wt}(\bar{y})=3$, and $\mathrm{wt}(\bar{z})=2$.
This implies that $[\mu-1:0:-\mu:\mu]$ is a singular point of type $\mathbb{A}_2$.
\end{proof}

Note that the singular locus of the surface $S_{-4}$ consists of the point $P_{\{x\},\{y\},\{z\}}$ and the line $\{x-t=y+z+t=0\}$.
Similarly, the singular locus of the surface $S_{-8}$ consists of the point $P_{\{x\},\{y\},\{z\}}$ and the line $\{x+t=y+z+t=0\}$.
Moreover, we have

\begin{lemma}
\label{lemma:r8-n1-irreducible}
Both surfaces $S_{-8}$ and $S_{-4}$ are irreducible.
\end{lemma}

\begin{proof}
It is enough to prove $S_{-4}$ is irreducible, because the remaining case can be handled in a similar way.
Let $\Pi$ be the plane $\{t=z\}$. Denote by $C_4$ the intersection $S_{-4}\cap\Pi$.
Then $C_4$ is the quartic curve in $\Pi\cong\mathbb{P}^2$ that it is given by
$$
x^2yz+xy^3+6xy^2z+10xyz^2+8xz^3+yz^3=0.
$$
This curve has exactly two singular points: $[1:0:0:0]$ and $[1:-2:1:1]$.
Moreover, the point $[1:0:0:0]$ is an ordinary double point of the curve $C_4$,
and the point $[1:-2:1:1]$ is an ordinary cusp of the curve $C_4$.
This implies that the curve $C_4$ is irreducible, so that the surface $S_{-4}$ is also irreducible.
\end{proof}

In Example~\ref{example:r-8-n1}, we proved that $[\mathsf{f}^{-1}(\lambda)]=1$ for every $\lambda\in\mathbb{C}$.
Thus, we conclude that \eqref{equation:main-1} in Main Theorem holds in this case.

Let us verify \eqref{equation:main-2} in Main Theorem. It follows from \eqref{equation:8-1} that
the intersection matrix of the curves
$L_{\{x\},\{y\}}$, $L_{\{x\},\{z\}}$, $L_{\{x\},\{t\}}$, $L_{\{y\},\{t,z\}}$, $L_{\{z\},\{t,y\}}$, and~$\mathcal{C}$
on the surface $S_\lambda$ has the same rank as the intersection matrix of the curves $L_{\{x\},\{y\}}$, $L_{\{x\},\{z\}}$, and $H_\lambda$.
If~$\lambda\ne-4$ and $\lambda\ne-8$, then  the latter matrix is given by
\begin{center}\renewcommand\arraystretch{1.42}
\begin{tabular}{|c||c|c|c|}
\hline
 $\bullet$  & $L_{\{x\},\{y\}}$ & $L_{\{x\},\{z\}}$ & $H_\lambda$ \\
\hline\hline
 $L_{\{x\},\{y\}}$ &  $-2$ & $1$ & $1$ \\
\hline
 $L_{\{x\},\{z\}}$ &  $1$ & $-2$ & $1$ \\
\hline
 $H_\lambda$  & $1$ & $1$ & $4$ \\
\hline
\end{tabular}
\end{center}
Its determinant is $18\ne 0$.
On the other hand, we have $\mathrm{rk}\,\mathrm{Pic}(\widetilde{S}_{\Bbbk})=\mathrm{rk}\,\mathrm{Pic}(S_{\Bbbk})+9$.
Thus, we see  that \eqref{equation:main-2-simple} holds.
Then \eqref{equation:main-2} in Main Theorem holds by Lemma~\ref{lemma:cokernel}.

\section{Fano threefolds of Picard rank $9$}
\label{section:rank-9}

\subsection{Family \textnumero $9.1$}
\label{section:r-9-n-1}

In this case, we have $X\cong\mathbb{P}^1\times\mathbf{S}_2$, where $\mathbf{S}_2$ is a smooth del Pezzo surface of degree $2$.
In particular, we have $h^{1,2}(X)=0$.
This case is somehow similar to the cases we treated in Subsections~\ref{section:r-2-n-2} and \ref{section:r-2-n-3}.
As in these two cases, this family does not have toric Landau--Ginzburg models with reflexive Newton polytope.
Let $\mathsf{p}$ be the Laurent polynomial
$$
\frac{(a+b+1)^4}{ab}+c+\frac{1}{c}.
$$
Then $\mathsf{p}$ gives the commutative diagram \eqref{equation:CCGK-compactification} by \cite[Proposition 16]{P16}.

Let $\gamma\colon\mathbb{C}^3\dasharrow\mathbb{C}^\ast\times\mathbb{C}^\ast\times\mathbb{C}^\ast$
be a birational transformation that is given by the change of coordinates
$$
\left\{\aligned
&a=xz,\\
&b=x-xz-1,\\
&c=\frac{z}{y}.\\
\endaligned
\right.
$$
Like in Subsection~\ref{section:r-2-n-2}, we can use $\gamma$ to expand
\eqref{equation:CCGK-compactification} to the commutative diagram \eqref{equation:diagram-2-2}.
The only difference is that now the pencil $\mathcal{S}$ is given by the equation
\begin{equation}
\label{equation:9-1-pencil}
x^3y=(\lambda yz-y^2-z^2)(xt-xz-t^2)
\end{equation}
where $\lambda\in\mathbb{C}\cup\{\infty\}$.
As in Subsection~\ref{section:r-2-n-2}, we will follow the scheme described in~Section~\ref{section:scheme}.
The only difference is that now $S_\lambda$ is the quartic surface given by the equation \eqref{equation:9-1-pencil}.

Let $\mathbf{Q}$ be the quadric in $\mathbb{P}^3$ given by  $xt-xz-t^2=0$. Then
$$
S_\infty=H_{\{y\}}+H_{\{z\}}+\mathbf{Q}.
$$
One the other hand, if $\lambda\ne\infty$, then $S_\lambda$ is irreducible and has isolated singularities.

Let $\mathcal{C}_1$ be the conic in $\mathbb{P}^3$ that is given by $y=xt-xz-t^2=0$,
and let $\mathcal{C}_2$ be the cubic curve in $\mathbb{P}^3$ that is given by $z=x^3+yt(x+t)=0$.
If $\lambda\ne\infty$, then
\begin{equation}
\label{equation:9-1}
\begin{split}
H_{\{y\}}\cdot S_\lambda&=2L_{\{y\},\{z\}}+\mathcal{C}_1,\\
H_{\{z\}}\cdot S_\lambda&=L_{\{y\},\{z\}}+\mathcal{C}_2,\\
\mathbf{Q}\cdot S_\lambda&=6L_{\{x\},\{t\}}+\mathcal{C}_1.
\end{split}
\end{equation}
Thus, the base locus of the pencil $\mathcal{S}$ consists of the curves
$L_{\{x\},\{t\}}$, $L_{\{y\},\{z\}}$, $\mathcal{C}_1$, and $\mathcal{C}_2$.

If $\lambda\ne\infty$,
then the singular points of the surface $S_\lambda$ contained in the base locus of the pencil $\mathcal{S}$ can be described as follows:
\begin{itemize}\setlength{\itemindent}{5cm}
\item[$P_{\{x\},\{z\},\{t\}}$:] type $\mathbb{A}_1$;
\item[$P_{\{x\},\{y\},\{z\}}$:] type $\mathbb{A}_5$ for $\lambda\ne\pm 2$, non-du Val for $\lambda=\pm 2$;
\item[{$[0:\lambda\pm\sqrt{\lambda^2-4}:2:0]$:}] type $\mathbb{A}_5$ for $\lambda\ne\pm 2$, non-du Val for $\lambda=\pm 2$.
\end{itemize}
Thus, it follows from
\eqref{equation:equation:number-of-irredubicle-components-refined-2}
and Lemma~\ref{lemma:normal-crossing} that the fiber $\mathsf{f}^{-1}(\lambda)$ is irreducible for every $\lambda\in\mathbb{C}$.
This confirms \eqref{equation:main-1} in Main Theorem.

To verify \eqref{equation:main-2} in Main Theorem, observe that $\mathrm{rk}\,\mathrm{Pic}(\widetilde{S}_{\Bbbk})=\mathrm{rk}\,\mathrm{Pic}(S_{\Bbbk})+9$.
Indeed, the minimal resolutions $\widetilde{S}_{\Bbbk}\to S_{\Bbbk}$ of the point $P_{\{x\},\{y\},\{z\}}$
is given by three consecutive blow ups that has three irreducible (over $\Bbbk$) exceptional curves.
Two of them are geometrically reducible, and one is geometrically irreducible.
Similarly, the minimal resolution $\widetilde{S}_{\Bbbk}\to S_{\Bbbk}$ of the point $[0:\lambda\pm\sqrt{\lambda^2-4}:2:0]$ has $5$ exceptional curves,
and the minimal resolution of the point $P_{\{x\},\{z\},\{t\}}$ has $1$ exceptional curve.

If $\lambda\ne\infty$, then it follows from \eqref{equation:9-1} that
$$
H_\lambda\sim 2L_{\{y\},\{z\}}+\mathcal{C}_1\sim L_{\{y\},\{z\}}+\mathcal{C}_2\sim_{\mathbb{Q}} 3L_{\{x\},\{t\}}+\frac{1}{2}\mathcal{C}_1
$$
on the surface $S_\lambda$.
Thus, if $\lambda\ne\infty$, then the intersection matrix of the curves $L_{\{x\},\{t\}}$ and $H_\lambda$
on the surface $S_\lambda$ has the same rank as the intersection matrix of the curves  $L_{\{x\},\{t\}}$, $L_{\{y\},\{z\}}$, $\mathcal{C}_1$,  $\mathcal{C}_2$, and $H_\lambda$.
On the other hand, if $\lambda\ne\infty$ and $\lambda\ne\pm 2$, the latter matrix is given by
\begin{center}\renewcommand\arraystretch{1.42}
\begin{tabular}{|c||c|c|}
\hline
 $\bullet$  & $L_{\{x\},\{t\}}$ &  $H_{\lambda}$ \\
\hline\hline
$L_{\{x\},\{t\}}$ & $-\frac{3}{2}$ & $1$ \\
\hline
 $H_{\lambda}$  & $1$ & $4$ \\
\hline
\end{tabular}
\end{center}
The rank of this matrix is $2$. Thus, we see that \eqref{equation:main-2-simple} holds in this case.
By Lemma~\ref{lemma:cokernel}, this confirms \eqref{equation:main-2} in Main Theorem.

\section{Fano threefolds of Picard rank $10$}
\label{section:rank-10}

\subsection{Family \textnumero $10.1$}
\label{section:r-10-n-1}

In this case, we have $X\cong\mathbb{P}^1\times\mathbf{S}_1$, where $\mathbf{S}_1$ is a smooth del Pezzo surface of degree $1$.
In particular, we have $h^{1,2}(X)=0$.
This case is very similar to the case we discussed in Subsection~\ref{section:r-2-n-1}.
As in that case, this family does not have toric Landau--Ginzburg models with reflexive Newton polytope.
However, there are Laurent polynomials with non-reflexive Newton polytopes that give the commutative diagram \eqref{equation:CCGK-compactification}.
One of them is the Laurent polynomial
$$
\frac{(x+y+1)^6}{xy^2}+z+\frac{1}{z},
$$
which we also denote by $\mathsf{p}$.

Let $\gamma\colon\mathbb{C}^3\dasharrow\mathbb{C}^\ast\times\mathbb{C}^\ast\times\mathbb{C}^\ast$
be a birational transformation that is given by the change of coordinates
$$
\left\{\aligned
&x=\frac{1}{b}-\frac{1}{b^2c}-1,\\
&y=\frac{1}{b^2c},\\
&z=y.\\
\endaligned
\right.
$$
Arguing as in Subsection~\ref{subsection:scheme-step-7}, we can expand \eqref{equation:CCGK-compactification} to the commutative diagram~\eqref{equation:diagram-2-1}.
The only difference is that now the pencil $\mathcal{S}$ is given by the equation
\begin{equation}
\label{equation:10-1-pencil}
xyc^3=(\lambda xy-x^2-y^2)(abc-b^2c-a^3),
\end{equation}
where $\lambda\in\mathbb{C}\cup\{\infty\}$.
Here $([x:y],[a:b:c])$ is a point in $\mathbb{P}^1\times\mathbb{P}^2$.

As in Subsection~\ref{section:r-2-n-1}, we will follow the scheme described in~Section~\ref{section:scheme},
and we will use assumptions and the notation introduced in that section.
The only difference is that $\mathbb{P}^3$ is now replaced by $\mathbb{P}^1\times\mathbb{P}^2$,
and $S_\lambda$ now is the surface in $\mathbb{P}^1\times\mathbb{P}^2$ that is given by \eqref{equation:10-1-pencil}.
As~in Subsection~\ref{section:r-2-n-1}, we will extend our handy notation in Subsection~\ref{subsection:notations} to bilinear sections of $\mathbb{P}^1\times\mathbb{P}^2$.

Let $\mathsf{S}$ be the surface in $\mathbb{P}^1\times\mathbb{P}^2$ given by $abc-b^2c-a^3=0$.
Then $\mathsf{S}$ is irreducible. Moreover, we have
$$
S_\infty=H_{\{x\}}+H_{\{y\}}+\mathsf{S}.
$$
On the other hand, if $\lambda\ne\infty$, then $S_\lambda$ is irreducible and has isolated singularities.

Let $\mathcal{C}_1$ be the curve in $\mathbb{P}^1\times\mathbb{P}^2$ that is given by $x=abc-b^2c-a^3=0$,
and let $\mathcal{C}_2$ be the curve in $\mathbb{P}^1\times\mathbb{P}^2$ that is given by $y=abc-b^2c-a^3=0$.
Then
\begin{equation}
\label{equation:10-1}
\begin{split}
H_{\{x\}}\cdot S_\lambda&=\mathcal{C}_1,\\
H_{\{y\}}\cdot S_\lambda&=\mathcal{C}_2,\\
\mathsf{S}\cdot S_\lambda&=\mathcal{C}_1+\mathcal{C}_2+9L_{\{a\},\{c\}}.
\end{split}
\end{equation}
Thus, the base locus of the pencil $\mathcal{S}$ consists of the curves $\mathcal{C}_1$, $\mathcal{C}_2$, and $L_{\{a\},\{c\}}$.

If $\lambda\ne\infty$, then the only singular points of the surface $S_\lambda$ contained in the base locus of the pencil $\mathcal{S}$
are the points
\begin{equation}
\label{equation:10-1-singular-points}
\Big(\big[\lambda\pm\sqrt{\lambda^2-4}:2\big],\big[0:1:0\big]\Big).
\end{equation}
If $\lambda\ne\pm 2$, then the surface $S_\lambda$ has singularity of type $\mathbb{A}_9$ at each of the points \eqref{equation:10-1-singular-points}.
If~$\lambda=\pm 2$, then \eqref{equation:10-1-singular-points} gives the points $([\pm 1:1],[0:1:0])$.
One can check that the surface $S_{\pm 2}$ has triple singularity at these points.

\begin{remark}
\label{remark:r-10-n-1-singular-curve}
There exists a commutative diagram
$$
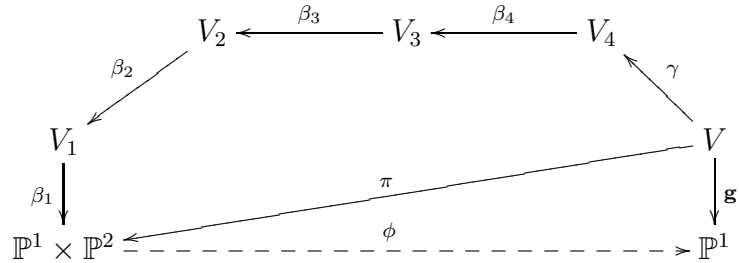
\xymatrix{
&&&&V_2\ar@{->}[ld]_{\beta_2}&&V_3\ar@{->}[ll]_{\beta_3}&&V_4\ar@{->}[ll]_{\beta_4}\\
&&&V_1\ar@{->}[d]_{\beta_1}&&&&&&V\ar@{->}[d]^{\mathbf{g}}\ar@{->}[lu]_{\gamma}\ar@{->}[lllllld]_{\pi}&&\\
&&&\mathbb{P}^1\times\mathbb{P}^2\ar@{-->}[rrrrrr]^{\phi}&&&&&&\mathbb{P}^1}
$$
where $\phi$ is a rational map that is given by the pencil $\mathcal{S}$,
the morphism~$\beta_1$ is the blow up of the curve $\mathcal{C}_1$,
the morphism~$\beta_2$ is the blow up of the proper transform of the curve $\mathcal{C}_2$,
the morphism~$\beta_3$ is the blow up of a curve that dominates the curve $\mathcal{C}_1$,
the morphism~$\beta_4$ is the blow up of a curve that dominates the curve $\mathcal{C}_2$,
and $\gamma$ is a birational morphism that is a composition of $9$ blow up of smooth curves
that dominate the curve $L_{\{a\},\{c\}}$.
Note that the curve $\mathcal{C}_1$ has a node at the point $P_{\{x\},\{a\},\{b\}}$.
Similarly, the curve $\mathcal{C}_1$ has a node at the point $P_{\{y\},\{a\},\{b\}}$.
Thus, both threefolds $V_1$ and $V_2$ are singular.
Moreover, the morphism $\beta_3$ blows up a nodal curve that is contained in the smooth locus of the threefold $V_2$.
Likewise, the morphism $\beta_4$ blows up a nodal curve that is contained in the smooth locus of the threefold $V_3$.
Thus, the threefold $V$ has four isolated ordinary double points.
However, they all are contained in the fiber $\mathbf{g}^{-1}(\infty)$,
which consists of the proper transforms on $V$ of the following surfaces: $H_{\{x\}}$, $H_{\{y\}}$, $\mathsf{S}$,
the exceptional surface of the morphism $\beta_1$, and the exceptional surface of the morphism $\beta_2$.
Thus, the singularities of the threefold $V$ are not important for the proof of Main Theorem in this case.
Note that
$$
-K_V\sim \mathbf{g}^{-1}\big(\infty\big).
$$
If we want to keep this condition and make $V$ smooth,
we must compose~$\pi$ with small resolution of singular points of the threefold $V$.
However, the resulting smooth threefold would not be projective (cf. the proof of \cite[Proposition~29]{P16}).
Indeed, by construction, the threefold $V$ is $\mathbb{Q}$-factorial, so that it does not admit projective small resolutions.
\end{remark}

Note that surfaces in the pencil $\mathcal{S}$ do not have fixed singular points, so that $\Sigma=\varnothing$.
Thus, using \eqref{equation:equation:number-of-irredubicle-components-refined}, we get $[\mathsf{f}^{-1}(\lambda)]=1$ for every $\lambda\in\mathbb{C}$.
This confirms \eqref{equation:main-1} in Main Theorem, since $h^{1,2}(X)=0$.

To verify \eqref{equation:main-2} in Main Theorem, observe that
$\mathrm{rk}\,\mathrm{Pic}(\widetilde{S}_{\Bbbk})=\mathrm{rk}\,\mathrm{Pic}(S_{\Bbbk})+9$.
One the other hand, if $\lambda\ne\infty$ and $\lambda\ne\pm 2$,
the rank of the intersection matrix of the curves  $\mathcal{C}_1$, $\mathcal{C}_2$, and $L_{\{a\},\{c\}}$ on the surface $S_\lambda$ is $1$.
This follows from \eqref{equation:10-1}.
Thus, we see that \eqref{equation:main-2-simple} holds in this case.
By Lemma~\ref{lemma:cokernel}, this confirms \eqref{equation:main-2} in Main Theorem.

\appendix

\section{Curves on singular surfaces}
\label{section:intersection}

Let $S$ be a normal surface, let $C$ and $Z$ be distinct irreducible curves in $S$.
For every point $P\in S$, one can define the intersection multiplicity $(C\cdot Z)_P\in\mathbb{Q}_{\geqslant 0}$ as in \cite{Sa84}.
As in the case when $S$ is smooth, one has
$$
C\cdot Z=\sum_{P\in C\cap Z}\Big(C\cdot Z\Big)_P.
$$
In this appendix, we present two (probably well-known to many experts) simple results that can be used to compute
the (local) intersection multiplicity $(C\cdot Z)_P$ and the (global) self-intersection $C^2$ in simple cases.
These results are Propositions~\ref{proposition:du-Val-intersection} and \ref{proposition:du-Val-self-intersection} below.

\subsection{Intersection multiplicity}
\label{subsection:intersection-multiplicity}

Fix a point $O\in C\cap Z$.
Let $\pi\colon\widetilde{S}\to S$ be the minimal resolution of singularity of the point $O$,
and let $G_1,\ldots,G_n$ be the exceptional curves of the birational morphism~$\pi$.
Denote by $\widetilde{C}$ and $\widetilde{Z}$ the proper transforms of the curves $C$ and $Z$ on the surface $\widetilde{S}$, respectively.
Following \cite{Sa84}, one can define $\pi^*(C)$ as
$$
\pi^*(C)=\widetilde{C}+\sum_{i=1}^{n}\mathbf{a}_iG_i
$$
for some positive rational numbers $\mathbf{a}_1,\ldots,\mathbf{a}_n$ such that
$$
\Big(\widetilde{C}+\sum_{i=1}^{n}\mathbf{a}_iG_i\Big)\cdot G_i=0.
$$
Similarly, we have
$$
\pi^*(Z)=\widetilde{Z}+\sum_{i=1}^{n}\mathbf{b}_iG_i
$$
for some positive rational numbers $\mathbf{b}_1,\ldots,\mathbf{b}_n$. We define
$$
C\cdot Z=\Big(\widetilde{C}+\sum_{i=1}^{n}\mathbf{a}_iG_i\Big)\cdot\Big(\widetilde{Z}+\sum_{i=1}^{n}\mathbf{b}_iG_i\Big)=\pi^*(C)\cdot\pi^*(Z)=\pi^*(C)\cdot\widetilde{Z}=\widetilde{C}\cdot\pi^*(Z).
$$
Let $\mathbf{G}=G_1\cup\cdots\cup G_n$. Then one can define $(C\cdot Z)_O$ as
\begin{equation}
\label{equation:intersection-multiplicity}
\Big(C\cdot Z\Big)_O=C\cdot Z-\widetilde{C}\cdot\widetilde{Z}+\sum_{P\in \widetilde{C}\cap\widetilde{Z}\cap \mathbf G}\Big(\widetilde{C}\cdot\widetilde{Z}\Big)_P.
\end{equation}

The main goal of this appendix is to prove the following two simple results.

\begin{proposition}
\label{proposition:du-Val-intersection}
Suppose that $O$ is a du Val singular point of the surface $S$, both curves $C$ and $Z$ are smooth at $O$,
and $C$ intersects $Z$ transversally at the point $O$.
Then the following assertions hold.
\begin{itemize}
\item[(i)] The point $O$ is a singular point of $S$ of type $\mathbb{A}_n$ or $\mathbb{D}_n$.

\item[(ii)] If $O$ is a singular point of type $\mathbb{A}_n$ and proper transforms of the curves $C$ and $Z$ on
the surface $\widetilde{S}$ intersect $k$-th and $r$-th exceptional curves in the chain of exceptional curves of the minimal resolution of~$O$,
then
$$
\Big(C\cdot Z\Big)_O=\left\{\aligned
&\frac{r(n+1-k)}{n+1}\ \text{for}\ r\leqslant k,\\
&\frac{k(n+1-r)}{n+1}\ \text{for}\ r>k.\\
\endaligned
\right.
$$

\item[(i)] If $O$ is of type $\mathbb D_n$, then $\Big(C\cdot Z\Big)_O=\frac{1}{2}$.
\end{itemize}
\end{proposition}

\begin{proposition}
\label{proposition:du-Val-self-intersection}
Suppose that $O$ is a du Val singular point of the surface $S$,
and the curve $C$ is smooth at the point $O$.
Then the following holds.
\begin{itemize}
\item[(i)]
The point $O$ is a singular point of the surface $S$ of type $\mathbb A_n$, $\mathbb D_n$, $\mathbb E_6$ or $\mathbb E_7$.

\item[(ii)] If $O$ is a singular point of type $\mathbb{A}_n$, and $\widetilde{C}$ intersects $k$-th exceptional curve in the chain of exceptional curves of the minimal resolution of~$O$, then
$$
C^2=\widetilde{C}^2+\frac{k(n+1-k)}{n+1}.
$$

\item[(iii)] If $O$ is a singular point of type $\mathbb{D}_n$, then $C^2=\widetilde{C}^2+1$ or $C^2=\widetilde{C}^2+\frac{n}{4}$.

\item[(iv)] If $O$ is a singular point of type $\mathbb{E}_6$, then $C^2=\widetilde{C}^2+\frac{4}{3}$.

\item[(iv)] If $O$ is a singular point of type $\mathbb{E}_7$, then $C^2=\widetilde{C}^2+\frac{3}{2}$.
\end{itemize}
\end{proposition}

The assertion of Propositions~\ref{proposition:du-Val-intersection} and \ref{proposition:du-Val-self-intersection}
follows from  Corollaries~\ref{corollary:intersection} and \ref{corollary:self-intersection}
and Lemmas~\ref{lemma:D4}, \ref{lemma:Dn}, \ref{lemma:E6}, \ref{lemma:E7}, and \ref{lemma:E8},
which we will prove below.

\subsection{Singular points of type $\mathbb{A}$}
\label{subsection:A}

In this subsection, we suppose that the surface $S$ has du Val singularity of type $\mathbb{A}_n$ at the point $O$, where $n\geqslant 1$.
Then we may assume that
$$
G_i\cdot G_j=\left\{\aligned
&-2\ \text{if}\ i=j,\\
&0\ \text{if}\ |i-j|>1,\\
&1\ \text{if}\ |i-j|=1.\\
\endaligned
\right.
$$
If the curve $C$ is smooth at $O$, then $\widetilde{C}$ is smooth along $\mathbf{G}$,
it intersects exactly one curve among $G_1,\ldots,G_n$, this intersection is transversal and consists of one point.
The same holds for $\widetilde{Z}$ in the case when $Z$ is smooth at $O$.
This is well-known (see \cite{Ar66}).

\begin{lemma}
\label{lemma:intersection}
Suppose that $C$ is smooth at $O$, and $\widetilde{C}\cap G_k\ne\varnothing$.
Then
$$
\mathbf{a}_i=\left\{\aligned
&\frac{i(n+1-k)}{n+1}\ \text{for}\ i\leqslant k,\\
&\frac{k(n+1-i)}{n+1}\ \text{for}\ i>k.\\
\endaligned
\right.
$$
In particular, one has $\mathbf{a}_k=\frac{k(n+1-k)}{n+1}$.
\end{lemma}

\begin{proof}
We may assume that $n\geqslant 2$, since the assertion is obvious for $n=1$.
Replacing $k$ by $n+1-l$, we may assume that $k\leqslant\frac{n+1}{2}$. Then
$$
0=\widetilde{C}\cdot G_n=-2\mathbf{a}_n+\mathbf{a}_{n-1}.
$$
If $k=1$, then $1=\widetilde{C}\cdot G_1=-2\mathbf{a}_1+\mathbf{a}_2$ and
$$
0=\widetilde{C}\cdot G_i=-2\mathbf{a}_i+\mathbf{a}_{i-1}+\mathbf{a}_{i+1}
$$
in the case when $n>i>1$. This gives $\mathbf{a}_i=\frac{n+1-i}{n+1}$ in this case.

Thus we may assume that $k\geqslant 2$, so that $n\geqslant 3$. Then $0=\widetilde{C}\cdot G_1=-2\mathbf{a}_1+\mathbf{a}_{2}$ and
$$
1=\widetilde{C}\cdot G_k=-2\mathbf{a}_k+\mathbf{a}_{k-1}+\mathbf{a}_{k+1}.
$$
For every $i\ne k$ such that $i\ne 1$ and $i\ne n-1$, we also have
$$
0=\widetilde{C}\cdot G_i=-2\mathbf{a}_i+\mathbf{a}_{i-1}+\mathbf{a}_{i+1}.
$$
Solving this system of equations, we obtain the required assertion.
\end{proof}

\begin{corollary}
\label{corollary:intersection}
Suppose that both $C$ and $Z$ are smooth at $O$. Suppose that $C$ intersects the curve $Z$ transversally at $O$.
Suppose also that $\widetilde{C}\cap G_k\ne\varnothing$ and $\widetilde{Z}\cap G_r\ne\varnothing$.
Then
$$
\Big(C\cdot Z\Big)_O=\left\{\aligned
&\frac{r(n+1-k)}{n+1}\ \text{for}\ r\leqslant k,\\
&\frac{k(n+1-r)}{n+1}\ \text{for}\ r>k.\\
\endaligned
\right.
$$
\end{corollary}

\begin{proof}
Since $C$ intersects $Z$ transversally at $O$, we have $\widetilde{C}\cap\widetilde{Z}\cap \mathbf{G}=\varnothing$.
But
$$
\widetilde{C}\cdot\widetilde{Z}=\Big(\pi^*(C)-\sum_{i=1}^{n}\mathbf{a}_iG_i\Big)\cdot\widetilde{Z}=C\cdot Z-\mathbf{a}_k.
$$
Thus, the required assertion follows from \eqref{equation:intersection-multiplicity} and Lemma~\ref{lemma:intersection}.
\end{proof}

\begin{corollary}
\label{corollary:self-intersection}
Suppose that $C$ is smooth at $O$, and $\widetilde{C}\cap G_k\ne\varnothing$. Then
$$
C^2=\widetilde{C}^2+\frac{k(n+1-k)}{n+1}.
$$
\end{corollary}

\begin{proof}
One has
$$
\widetilde{C}^2=\Big(\pi^*(C)-\sum_{i=1}^{n}\mathbf{a}_iG_i\Big)^2=C^2-\mathbf{a}_1=C^2-\frac{n}{n+1}
$$
by Lemma~\ref{lemma:intersection}.
\end{proof}

\begin{remark}
\label{remark:transversal}
Suppose that $n\geqslant 3$.
Then there exists a commutative diagram
$$
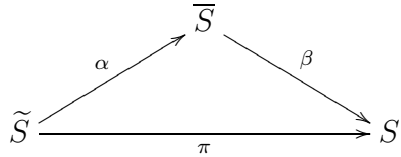
\xymatrix{
&&\overline{S}\ar@{->}[drr]^\beta&&\\
\widetilde{S}\ar@{->}[rru]^\alpha\ar@{->}[rrrr]_\pi &&&& S}
$$
such that $\beta$ is the blow up of the point $O$, and $\alpha$ is a birational morphism
that contracts the curves $G_2,\ldots,G_{n-1}$ to the singular point of type $\mathbb{A}_{n-2}$.
Denote by $\overline{G}_1$, $\overline{G}_n$, $\overline{C}$, and $\overline{Z}$
the proper transforms of the curves $G_1$, $G_n$, $\widetilde{C}$, and $\widetilde{Z}$
on the surface $\overline{S}$, respectively.
If $C$ and $Z$ are smooth at $O$, and the curve $C$ intersects $Z$ transversally at~$O$,
then the curves $\overline{C}$ and $\overline{Z}$ are smooth,
and at most one curve among $\overline{C}$ and $\overline{Z}$ passes through the intersection point $\overline{G}_1\cap\overline{G}_n$.
\end{remark}

\subsection{Singular points of type $\mathbb{D}$}
\label{subsection:Dn}

Now we suppose that the surface $S$ has du Val singularity of type $\mathbb{D}_n$ at the point $O$, where $n\geqslant 4$.
We start with the following.

\begin{lemma}
\label{lemma:D4}
Suppose that $n=4$, both $C$ and $Z$ are smooth at $O$, and $C$ intersects the curve $Z$ transversally at $O$.
Then $(C\cdot Z)_O=\frac{1}{2}$ and $C^2=\widetilde{C}^2+1$.
\end{lemma}

\begin{proof}
We may assume that
the intersection form of the curves $G_1$, $G_2$, $G_3$, $G_4$ is given~by
\begin{center}\renewcommand\arraystretch{1.42}
\begin{tabular}{|c||c|c|c|c|}
\hline
 $\bullet$  & $G_1$ & $G_2$& $G_3$ & $G_4$\\
\hline\hline
 $G_1$ &  $-2$ & $1$ & $1$ & $1$ \\
\hline
 $G_2$ &  $1$ & $-2$ & $0$ & $0$\\
\hline
 $G_3$ &  $1$ & $0$ & $-2$ & $0$\\
\hline
 $G_4$ &  $1$ & $0$ & $0$ & $-2$\\
\hline
\end{tabular}
\end{center}
Then $2G_1+G_2+G_3+G_4$ is the \emph{fundamental cycle} of the singular point $O$, see \cite{Ar66}.
This implies that
$$
\widetilde{C}\cdot\Big(2G_1+G_2+G_3+G_4\Big)=\mathrm{mult}_{O}\big(C\big)=1.
$$
Thus, we see that $\widetilde{C}\cap G_1=\varnothing$.
Hence, we may assume that $\widetilde{C}\cdot G_2=1$, which implies that $\widetilde{C}\cdot G_1=\widetilde{C}\cdot G_3=\widetilde{C}\cdot G_4=0$,
which gives
$$
\left\{\aligned
&1+\mathbf{a}_1-2\mathbf{a}_2=0,\\
&\mathbf{a}_2+\mathbf{a}_3+\mathbf{a}_4-2\mathbf{a}_1=0,\\
&\mathbf{a}_1-2\mathbf{a}_3=0,\\
&\mathbf{a}_1-2\mathbf{a}_4=0.\\
\endaligned
\right.
$$
Solving this system of equations, we see that $\mathbf{a}_1=1$, $\mathbf{a}_2=1$, $\mathbf{a}_3=\frac{1}{2}$, $\mathbf{a}_4=\frac{1}{2}$.
This implies that $C^2=\widetilde{C}^2+1$.
Note also that there exists a commutative diagram
$$
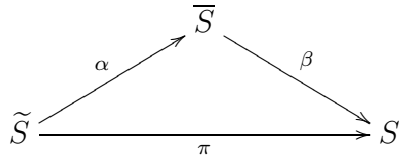
\xymatrix{
&&\overline{S}\ar@{->}[drr]^\beta&&\\
\widetilde{S}\ar@{->}[rru]^\alpha\ar@{->}[rrrr]_\pi &&&& S}
$$
such that $\beta$ is the blow up of the point $O$, and $\alpha$ is a birational morphism
that contracts the curves $G_2$, $G_3$, and $G_4$ to three ordinary double points of the surface $\overline{S}$.
Denote by $\overline{C}$ and $\overline{Z}$
the proper transforms of the curves $\widetilde{C}$ and $\widetilde{Z}$
on the surface $\overline{S}$, respectively.
If $C$ and $Z$ are smooth at $O$, and the curve $C$ intersects $Z$ transversally at~$O$,
then $(C\cdot Z)_O=\frac{1}{2}$,
because $\overline{C}\cap\overline{Z}\cap\alpha(G_1)=\varnothing$,
the curves $\overline{C}$ and $\overline{Z}$ are smooth along $\alpha(G_1)$,
and each of them contains a singular point of the surface $\overline{S}$ contained in $\alpha(G_1)$.
\end{proof}

Now we suppose that $n\geqslant 5$.
In this case, we may assume that the intersection form of the exceptional curves $G_1,\ldots,G_n$ is given by the following table:
\begin{center}\renewcommand\arraystretch{1.42}
\begin{tabular}{|c||c|c|c|c|c|c|c|c|}
\hline
 $\bullet$  & $G_1$ & $G_2$& $G_3$ & $G_4$ & $G_5$ & $\ldots$ & $G_{n-1}$ & $G_n$\\
\hline\hline
 $G_1$ &  $-2$ & $1$ & $1$ & $1$ & $0$ & $\ldots$ & $0$ & $0$\\
\hline
 $G_2$ &  $1$ & $-2$ & $0$ & $0$ & $0$ & $\ldots$ & $0$ & $0$\\
\hline
 $G_3$ &  $1$ & $0$ & $-2$ & $0$ & $0$ & $\ldots$ & $0$ & $0$\\
\hline
 $G_4$ &  $1$ & $0$ & $0$ & $-2$ & $1$ & $\ldots$ & $0$ & $0$\\
\hline
 $G_5$ &  $0$ & $0$ & $0$ & $1$ & $-2$ & $\ldots$ & $0$ & $0$ \\
\hline
 $\ldots$ & $\ldots$ & $\ldots$ & $\ldots$ & $\ldots$ & $\ldots$ & $\ddots$ & $\ldots$ & $\ldots$\\
\hline
 $G_{n-1}$ &  $0$ & $0$ & $0$ & $0$ & $0$ & $\ldots$ & $-2$ & $1$\\
\hline
 $G_n$ &  $0$ & $0$ & $0$ & $0$ & $0$ & $\ldots$ & $1$ & $-2$\\
\hline
\end{tabular}
\end{center}

\begin{lemma}
\label{lemma:Dn}
Suppose that $C$ and $Z$ are smooth at $O$, and $C$ intersects $Z$ transversally at the point $O$.
Then
$$
\Big(C\cdot Z\Big)_O=\frac{1}{2}.
$$
If $C\cap G_n\ne\varnothing$, then $C^2=\widetilde{C}^2+1$.
Otherwise, one has $C^2=\widetilde{C}^2+\frac{n}{4}$.
\end{lemma}

\begin{proof}
Recall from \cite{Ar66} that $2G_1+G_2+G_3+2G_4+\ldots+2G_{n-1}+G_n$ is the fundamental cycle of the singular point $O$.
Then
$$
\widetilde{C}\cdot\big(2G_1+G_2+G_3+2G_4+\ldots+2G_{n-1}+G_n\big)=\mathrm{mult}_{O}\big(C\big)=1.
$$
This shows that $\widetilde{C}\cdot G_1=\widetilde{C}\cdot G_4=\ldots=\widetilde{C}\cdot G_{n-1}=0$ and
$\widetilde{C}\cdot G_2+\widetilde{C}\cdot G_3+\widetilde{C}\cdot G_n=1$.
Hence, the curve $\widetilde C$ intersects exactly one of curves $G_2$, $G_3$ or $G_n$,
and it intersects this curve transversally at a single point.
Similarly, the same holds for the curve $Z$.

Let $\beta\colon\overline{S}\to S$ be the blow up the point $O$.
Then there exists the following commutative diagram:
$$
\xymatrix{
&&\overline{S}\ar@{->}[drr]^\beta&&\\
\widetilde{S}\ar@{->}[rru]^\alpha\ar@{->}[rrrr]_\pi &&&& S}
$$
where $\alpha$ is a birational morphism that contracts the curves $G_1,G_2,G_3,\ldots,G_{n-2}$, and $G_n$.
Thus, we see that $\alpha(G_{n-1})$ is the exceptional curve of the blow up $\beta$.
Note that $\alpha(G_n)$ is an isolated ordinary double point of the surface $\overline{S}$.
Similarly, we see that the surface $\overline{S}$ has a du Val singular point of type $\mathbb{D}_{n-2}$
at the point $\alpha(G_1)=\ldots=\alpha(G_{n-2})$.
Here, we assume that $\mathbb{D}_3=\mathbb{A}_3$.

Denote by $\overline{C}$ and $\overline{Z}$ the proper transforms on $\overline{S}$ of the curves $\widetilde{C}$ and $\widetilde{Z}$, respectively.
Since $\widetilde{C}$ and $\widetilde{Z}$ do not intersect the curve $G_{n-1}$,
each of the curves $\overline{C}$ and $\overline{Z}$ must pass through some singular point of the surface $\overline{S}$ contained in $\beta(G_{n-1})$.
Furthermore, we have
$\overline{C}\cap\overline{Z}=\varnothing$,
since the curve $C$ intersects the curve $Z$ transversally at $O$.
Thus, without loss of generality, one can assume $\widetilde{C}\cdot G_n=1$ and $\widetilde{Z}\cdot G_2=1$.
This gives us the following system of equations:
$$
\left\{\aligned
&2\mathbf{a}_1-\mathbf{a}_2-\mathbf{a}_3-\mathbf{a}_4=\widetilde{C}\cdot G_1=0,\\
&2\mathbf{a}_2-\mathbf{a}_1=\widetilde{C}\cdot G_2=0,\\
&2\mathbf{a}_3-\mathbf{a}_1=\widetilde{C}\cdot G_3=0,\\
&2\mathbf{a}_4-\mathbf{a}_1-\mathbf{a}_5=\widetilde{C}\cdot G_4=0,\\
&2\mathbf{a}_5-\mathbf{a}_4-\mathbf{a}_6=\widetilde{C}\cdot G_5=0,\\
&\ldots\\
&2\mathbf{a}_{n-1}-\mathbf{a}_{n-2}-\mathbf{a}_n=\widetilde{C}\cdot G_{n-1}=0,\\
&2\mathbf{a}_n-\mathbf{a}_{n-1}=\widetilde{C}\cdot G_n=1.
\endaligned
\right.
$$
Solving it, we obtain $\mathbf{a}_1=1$, $\mathbf{a}_2=\mathbf{a}_3=\frac{1}{2}$, $\mathbf{a}_4=\ldots=\mathbf{a}_n=1$.
In particular, we have
$$
\widetilde{C}\cdot\widetilde{Z}=\Big(\pi^*(C)-\sum_{i=1}^{n}\mathbf{a}_iG_i\Big)\cdot\widetilde{Z}=C\cdot Z-\mathbf{a}_2=C\cdot Z-\frac{1}{2}.
$$
Hence, we see that $(C\cdot Z)_O=\frac{1}{2}$. Likewise, we get $C^2=\widetilde{C}^2+1$.

Similarly, we have the following system of equations:
$$
\left\{\aligned
&2\mathbf{b}_1-\mathbf{b}_2-\mathbf{b}_3-\mathbf{b}_4=\widetilde{Z}\cdot G_1=0,\\
&2\mathbf{b}_2-\mathbf{b}_1=\widetilde{Z}\cdot G_2=1,\\
&2\mathbf{b}_3-\mathbf{b}_1=\widetilde{Z}\cdot G_3=0,\\
&2\mathbf{b}_4-\mathbf{b}_1-\mathbf{b}_5=\widetilde{Z}\cdot G_4=0,\\
&2\mathbf{b}_5-\mathbf{b}_4-\mathbf{b}_6=\widetilde{Z}\cdot G_5=0,\\
&\ldots\\
&2\mathbf{b}_{n-1}-\mathbf{b}_{n-2}-\mathbf{b}_n=\widetilde{Z}\cdot G_{n-1}=0,\\
&2\mathbf{b}_n-\mathbf{b}_{n-1}=\widetilde{Z}\cdot G_n=0.
\endaligned
\right.
$$
Solving it, we see that
$$
\mathbf{b}_1=\frac{n-2}{4},\ \mathbf{b}_2=\frac{n}{4},\ \mathbf{b}_3=\frac{n-2}{4},\ \mathbf{b}_4=\frac{n-3}{2},\ \mathbf{b}_5=\frac{n-4}{2},\ \ldots,\ \mathbf{b}_n=\frac{1}{2}.
$$
As above, this gives $Z^2=\widetilde{Z}^2+\frac{n}{4}$. This completes the proof of the lemma.
\end{proof}

\subsection{Singular points of type $\mathbb{E}$}
\label{subsection:E}

Now we consider the case when $S$ has du Val singularity of type $\mathbb{E}_6$, $\mathbb{E}_7$ or $\mathbb{E}_8$ at the point $O$.
We start with the following fact.

\begin{lemma}
\label{lemma:E6}
Suppose that $S$ has du Val singularity of type $\mathbb{E}_6$ at the point $O$,
and both curves $C$ and $Z$ are smooth at $O$.
Then $C$ is tangent to $Z$ at the point $O$, and
$$
C^2=\widetilde{C}^2+\frac{4}{3}.
$$
\end{lemma}

\begin{proof}
We have $n=6$. We may assume that the intersection form of the curves $G_1$, $G_2$, $G_3$, $G_4$, $G_5$, and $G_6$ is given by the following table:
\begin{center}\renewcommand\arraystretch{1.42}
\begin{tabular}{|c||c|c|c|c|c|c|}
\hline
 $\bullet$  & $G_1$ & $G_2$& $G_3$ & $G_4$ & $G_5$ & $G_6$\\
\hline\hline
 $G_1$ &  $-2$ & $1$ & $1$ & $1$ & $0$ & $0$\\
\hline
 $G_2$ &  $1$ & $-2$ & $0$ & $0$ & $0$ & $0$\\
\hline
 $G_3$ &  $1$ & $0$ & $-2$ & $0$ & $1$ & $0$\\
\hline
 $G_4$ &  $1$ & $0$ & $0$ & $-2$ & $0$ & $1$\\
\hline
 $G_5$ &  $0$ & $0$ & $1$ & $0$ & $-2$ & $0$\\
\hline
 $G_6$ &  $0$ & $0$ & $0$ & $1$ & $0$ & $-2$\\
\hline
\end{tabular}
\end{center}
Thus, the curve $G_1$ is the \emph{fork} curve.

Let $\beta\colon\overline{S}\to S$ be the blow up of the point $O$.
Then there exists a commutative diagram:
$$
\xymatrix{
&&\overline{S}\ar@{->}[drr]^\beta&&\\
\widetilde{S}\ar@{->}[rru]^\alpha\ar@{->}[rrrr]_\pi &&&& S,}
$$
where $\alpha$ is a contraction of the curves $G_1$, $G_3$, $G_4$, $G_5$, and $G_6$
We see that $\alpha(G_2)$ is the exceptional curve of the blow up $\beta$.
This curve contains one singular point of the surface $\overline{S}$.
Denote it by $P$. Then $P$ is the image of the curves $G_1$, $G_3$, $G_4$, $G_5$, and $G_6$.
Note that $\overline{S}$ has a du Val singular point of type $\mathbb{A}_5$ at the point $P$.

Let $\overline{C}$ and $\overline{Z}$ be the proper transforms on $\overline{S}$ of the curves $C$ and $Z$, respectively.
Then both $\overline{C}$ and $\overline{Z}$ are smooth along $\alpha(G_2)$.
We claim that $\overline{C}\cap\overline{Z}=P$.
Indeed, the fundamental cycle of the singular point $O$ is $G_5+G_6+2G_2+2G_3+2G_4+3G_1$.
Thus, the curve $\widetilde{C}$ does not intersect the curves $G_1$, $G_2$, $G_3$, and $G_4$.
Similarly, we see that the curve $\widetilde{Z}$ does not intersect the curves $G_1$, $G_2$, $G_3$, and $G_4$.
Hence, without loss of generality, we may assume that $\widetilde{C}\cap G_5\ne\varnothing$.
Then $\widetilde{C}\cap G_6=\varnothing$ and  $\widetilde{C}\cdot G_5=1$.
Similarly, we see that either $\widetilde{Z}\cap G_5\ne\varnothing$ or $\widetilde{Z}\cap G_6\ne\varnothing$.
In both cases, we have $\overline{C}\cap\overline{Z}=P$, so that the curve $C$ is tangent to $Z$ at the point $O$.

Since $\widetilde{C}\cdot G_5=1$ and $\widetilde{C}\cdot G_1=\widetilde{C}\cdot G_2=\widetilde{C}\cdot G_3=\widetilde{C}\cdot G_4=\widetilde{C}\cdot G_6$,
we get the following system of equations:
$$
\left\{\aligned
&2\mathbf{a}_1-\mathbf a_2-\mathbf a_3-\mathbf a_4=\widetilde{C}\cdot G_1=0,\\
&\mathbf{a}_2-\mathbf{a}_1=\widetilde{C}\cdot G_2=0,\\
&2\mathbf{a}_3-\mathbf{a}_1-\mathbf{a}_5=\widetilde{C}\cdot G_3=0,\\
&2\mathbf{a}_4-\mathbf{a}_1-\mathbf{a}_6=\widetilde{C}\cdot G_4=0,\\
&2\mathbf{a}_5-\mathbf{a}_3=\widetilde{C}\cdot G_5=1,\\
&2\mathbf{a}_6-\mathbf{a}_4=\widetilde{C}\cdot G_6=0.
\endaligned
\right.
$$
Solving it, we see that
$\mathbf{a}_1=2$, $\mathbf{a}_2=1$, $\mathbf{a}_3=\frac{5}{3}$, $\mathbf{a}_4=\frac{4}{3}$, $\mathbf{a}_5=\frac{4}{3}$, and $\mathbf{a}_6=\frac{2}{3}$.
Thus
$$
\widetilde{C}^2=\Big(\pi^*(C)-2G_1-G_2-\frac{5}{3}G_3-\frac{4}{3}G_4-\frac{4}{3}G_5-\frac{2}{3}G_6\Big)\cdot\widetilde{C}=C^2-\frac{4}{3},
$$
which gives
$C^2=\widetilde{C}^2+\frac{4}{3}$.
\end{proof}

\begin{lemma}
\label{lemma:E7}
Suppose that $S$ has du Val singularity of type $\mathbb{E}_7$ at the point $O$,
and both curves $C$ and $Z$ are smooth at $O$.
Then $C$ is tangent to $Z$ at the point $O$, and
$$
C^2=\widetilde{C}^2+\frac{3}{2}.
$$
\end{lemma}

\begin{proof}
We may assume that the intersection form of the curves $G_1$, $G_2$, $G_3$, $G_4$, $G_5$, $G_6$, and $G_7$ is given by the following table:
\begin{center}\renewcommand\arraystretch{1.42}
\begin{tabular}{|c||c|c|c|c|c|c|c|}
\hline
 $\bullet$  & $G_1$ & $G_2$& $G_3$ & $G_4$ & $G_5$ & $G_6$ & $G_7$\\
\hline\hline
 $G_1$ &  $-2$ & $1$ & $1$ & $1$ & $0$ & $0$ & $0$\\
\hline
 $G_2$ &  $1$ & $-2$ & $0$ & $0$ & $0$ & $0$ & $0$\\
\hline
 $G_3$ &  $1$ & $0$ & $-2$ & $0$ & $1$ & $0$ & $0$\\
\hline
 $G_4$ &  $1$ & $0$ & $0$ & $-2$ & $0$ & $1$ & $0$\\
\hline
 $G_5$ &  $0$ & $0$ & $1$ & $0$ & $-2$ & $0$ & $0$\\
\hline
 $G_6$ &  $0$ & $0$ & $0$ & $1$ & $0$ & $-2$ & $1$\\
\hline
 $G_7$ &  $0$ & $0$ & $0$ & $0$ & $0$ & $1$ & $-2$\\
\hline
\end{tabular}
\end{center}
Thus, the curve $G_1$ is the {fork} curve.

The fundamental cycle of the singular point $O$ is $2G_5+2G_6+2G_2+3G_3+3G_4+4G_1+G_7$.
This shows that $\widetilde{C}\cdot G_7=1$ and
$\widetilde{C}\cdot G_1=\widetilde{C}\cdot G_2=\widetilde{C}\cdot G_3=\widetilde{C}\cdot G_4=\widetilde{C}\cdot G_5=\widetilde{C}\cdot G_6=0$.
This gives us the following system of equations:
$$
\left\{\aligned
&2\mathbf{a}_1-\mathbf a_2-\mathbf a_3-\mathbf a_4=\widetilde{C}\cdot G_1=0,\\
&\mathbf{a}_2-\mathbf{a}_1=\widetilde{C}\cdot G_2=0,\\
&2\mathbf{a}_3-\mathbf{a}_1-\mathbf{a}_5=\widetilde{C}\cdot G_3=0,\\
&2\mathbf{a}_4-\mathbf{a}_1-\mathbf{a}_6=\widetilde{C}\cdot G_4=0,\\
&2\mathbf{a}_5-\mathbf{a}_3=\widetilde{C}\cdot G_5=0,\\
&2\mathbf{a}_6-\mathbf{a}_4-\mathbf{a}_7=\widetilde{C}\cdot G_6=0,\\
&2\mathbf{a}_7-\mathbf{a}_6=\widetilde{C}\cdot G_7=1.
\endaligned
\right.
$$
Then
$\mathbf{a}_1=3$, $\mathbf{a}_2=\frac{3}{2}$, $\mathbf{a}_3=2$, $\mathbf{a}_4=\frac{5}{2}$, $\mathbf{a}_5=1$, $\mathbf{a}_6=2$
and $\mathbf{a}_7=\frac{3}{2}$.
This gives $C^2=\widetilde{C}^2+\frac{3}{2}$.
Arguing as in the proof of Lemma~\ref{lemma:E6}, we see that $C$ tangents $Z$ at the point $O$.
\end{proof}

Finally, we conclude this appendix by proving the following.

\begin{lemma}
\label{lemma:E8}
If $S$ has du Val singularity of type $\mathbb{E}_8$ at $O$, then $C$ is singular at $O$.
\end{lemma}

\begin{proof}
This follows from the fact that coefficients at all exceptional curves of the minimal resolution of $O$
in the fundamental cycle are greater than $1$.
\end{proof}

\end{document}